\documentclass{amsart}
\usepackage{amssymb,amscd,epsfig}
\usepackage{xcolor}
\usepackage{tikz-cd}
\tikzset{math3D/.style= {x= {(-0.5cm,-0.5cm)}, z={(0cm,1cm)},y={(1cm,0cm)}}}
\usepackage[colorlinks=true, linkcolor=blue, citecolor=blue, urlcolor=blue, backref]{hyperref}

\newtheorem{proposition}{Proposition}[subsection]
\newtheorem{theorem}[proposition]{Theorem}
\newtheorem{lemma}[proposition]{Lemma}
\newtheorem{main lemma}[proposition]{Main Lemma}
\newtheorem{corollary}[proposition]{Corollary}
\newtheorem*{conjecture*}{Conjecture}
\newtheorem*{corollary*}{Corollary}
\newtheorem*{lemma*}{Lemma}
\newtheorem{question}[proposition]{Question}
\newtheorem*{question*}{Question}

\theoremstyle{definition}
\newtheorem{definition}[proposition]{Definition}

\theoremstyle{remark}
\newtheorem*{problem*}{Problem}
\newtheorem*{remark*}{Remark}
\newtheorem{remark}[proposition]{Remark}
\newtheorem*{remarks*}{Remarks}
\newtheorem{remarks}[proposition]{Remarks}
\newtheorem{example}[proposition]{Example}

\DeclareMathOperator{\Aut}{Aut}
\DeclareMathOperator{\Et}{Et}
\DeclareMathOperator{\Out}{Out}
\DeclareMathOperator{\SAut}{SAut}
\DeclareMathOperator{\Mor}{Mor}
\DeclareMathOperator{\Mat}{M}
\DeclareMathOperator{\Linear}{\mathcal L}
\DeclareMathOperator{\End}{End}
\DeclareMathOperator{\Dom}{Dom}
\DeclareMathOperator{\codim}{codim}
\DeclareMathOperator{\SL}{SL}
\DeclareMathOperator{\Sp}{Sp}
\DeclareMathOperator{\PSL}{PSL}
\DeclareMathOperator{\SLtwo}{{SL}_2}
\DeclareMathOperator{\GL}{GL}
\DeclareMathOperator{\PGL}{PGL}
\DeclareMathOperator{\Hom}{Hom}
\DeclareMathOperator{\Ker}{Ker}
\DeclareMathOperator{\Coker}{Coker}
\DeclareMathOperator{\Image}{Im}
\DeclareMathOperator{\Cent}{Cent}

\DeclareMathOperator{\id}{id}
\DeclareMathOperator{\tr}{tr}
\DeclareMathOperator{\Span}{span}
\DeclareMathOperator{\Spec}{Spec}
\DeclareMathOperator{\Tame}{Tame}
\DeclareMathOperator{\Lie}{Lie}
\DeclareMathOperator{\lgt}{\ell}
\DeclareMathOperator{\jac}{jac}
\DeclareMathOperator{\VEC}{Vec}
\DeclareMathOperator{\Der}{Der}
\DeclareMathOperator{\LNV}{\VEC^{\text{\it ln}}}
\DeclareMathOperator{\Akplus}{\mathfrak{Act}_{\kplus}}
\DeclareMathOperator{\Act}{\mathfrak{Act}}
\DeclareMathOperator{\res}{res}
\DeclareMathOperator{\lt}{lt}
\DeclareMathOperator{\ad}{ad}
\DeclareMathOperator{\Alg}{Alg}
\DeclareMathOperator{\Inj}{Inj}
\DeclareMathOperator{\Aff}{Aff}
\DeclareMathOperator{\SAff}{SAff}

\DeclareMathOperator{\Int}{Int}
\DeclareMathOperator{\Ad}{Ad}

\DeclareMathOperator{\M}{M}
\DeclareMathOperator{\NZ}{NZ}
\DeclareMathOperator{\XX}{\mathcal X}
\DeclareMathOperator{\IndVar}{Ind-Var}
\DeclareMathOperator{\Sets}{Sets}
\DeclareMathOperator{\Hilb}{Hilb}
\DeclareMathOperator{\pr}{pr}
\DeclareMathOperator{\Div}{Div}
\DeclareMathOperator{\Cplus}{\CC^{+}}
\DeclareMathOperator{\eval}{\eps}
\DeclareMathOperator{\Imm}{Imm}
\DeclareMathOperator{\Tr}{Tr}
\DeclareMathOperator{\sL}{\mathfrak{sl}}
\DeclareMathOperator{\AAA}{\mathcal A}
\DeclareMathOperator{\LLL}{\mathcal L}
\DeclareMathOperator{\VVV}{\mathcal V}
\DeclareMathOperator{\OOO}{\mathcal O}
\DeclareMathOperator{\NNN}{\mathcal N}
\DeclareMathOperator{\R}{\mathcal R}
\DeclareMathOperator{\JJJ}{\mathcal J}
\DeclareMathOperator{\Stab}{Stab}
\DeclareMathOperator{\Mod}{Mod}
\DeclareMathOperator{\rad}{rad}
\DeclareMathOperator{\tdeg}{tdeg}
\DeclareMathOperator{\Quot}{Quot}

\renewcommand{\AA}{{\mathbb A}}
\newcommand{\An}{{\AA^{n}}}
\newcommand{\Am}{{\AA^{m}}}
\newcommand{\C}{{\mathbb C}}
\newcommand{\CC}{\kk}
\newcommand{\Cn}{\CC^{n}}
\newcommand{\FF}{{\mathbb F}}
\newcommand{\NN}{{\mathbb N}}
\newcommand{\PP}{{\mathbb P}}
\newcommand{\QQ}{{\mathbb Q}}
\newcommand{\RR}{{\mathbb R}}
\newcommand{\ZZ}{{\mathbb Z}}
\newcommand{\ktwo}{\kk^{2}}
\newcommand{\GLn}{\GL(n)}
\newcommand{\SLn}{\SL_{n}}

\renewcommand{\aa}{{\mathfrak a}}

\renewcommand{\gg}{{\mathfrak g}}

\newcommand{\mm}{{\mathfrak m}}

\newcommand{\pp}{{\mathfrak p}}

\newcommand{\rr}{{\mathfrak r}}

\newcommand{\sltwo}{\sL_{2}}
\newcommand{\bsm}{\left[\begin{smallmatrix}}
\newcommand{\esm}{\end{smallmatrix}\right]}
\newcommand{\Cb}[1]{\overline{\CC(#1)}}
\newcommand{\dcup}{\,\dot\cup\,}
\newcommand{\kinfty}{\Bbbk^{\infty}}
\newcommand{\kn}{\Bbbk^{n}}

\newcommand{\BBB}{\mathcal B}
\newcommand{\CCC}{\mathcal C}

\newcommand{\FFF}{\mathcal F}
\newcommand{\GGG}{\mathcal{G}}
\newcommand{\HHH}{\mathcal H}
\newcommand{\KKK}{\mathcal K}
\newcommand{\MMM}{\mathcal M}
\newcommand{\PPP}{\mathcal P}
\newcommand{\TTT}{\mathcal T}
\newcommand{\UUU}{\mathcal U}
\newcommand{\WWW}{\mathcal W}
\newcommand{\YYY}{\mathcal Y}
\newcommand{\ZZZ}{\mathcal Z}
\newcommand{\RRR}{\mathcal R}
\newcommand{\SSS}{\mathcal S}
\newcommand{\XXX}{\mathcal X}
\newcommand{\III}{\mathcal I}

\newcommand{\Glf}{\GGG^{\text{\it lf}}}
\newcommand{\Gss}{\GGG^{\text{\it ss}}}
\newcommand{\Gu}{\GGG^{\text{\it u}}}

\renewcommand{\a}{\mathbf a}
\renewcommand{\b}{\mathbf b}
\renewcommand{\c}{\mathbf c}
\renewcommand{\d}{\mathbf d}
\newcommand{\f}{\mathbf f}
\newcommand{\g}{\mathbf g}
\newcommand{\h}{\mathbf h}
\newcommand{\n}{\mathbf n}

\newcommand{\s}{\mathbf s}
\renewcommand{\t}{\mathbf t}
\renewcommand{\u}{\mathbf u}
\newcommand{\vbold}{\mathbf v}
\newcommand{\wbold}{\mathbf w}
\newcommand{\x}{\mathbf x}

\newcommand{\z}{\mathbf z}
\newcommand{\ii}{\mathbf i}
\newcommand{\jj}{\mathbf j}
\newcommand{\balpha}{\boldsymbol\alpha}
\newcommand{\bbeta}{\boldsymbol\beta}
\newcommand{\bgamma}{\boldsymbol\gamma}

\newcommand{\bG}{\mathcal F}
\newcommand{\tbG}{\tilde\bG}
\newcommand{\into}{\hookrightarrow}

\newcommand{\quot}{/\!\!/}
\newcommand{\eps}{\varepsilon}
\renewcommand{\phi}{\varphi}
\newcommand{\Ft}{F^{t}}
\newcommand{\tG}{\GGG^{t}}
\newcommand{\Gz}{\GGG_{z}}
\newcommand{\tGz}{\GGG^{t}_{z}}
\newcommand{\simto}{\xrightarrow{\sim}}

\newcommand{\wc}[1]{\overline{#1}^{w}}
\newcommand{\name}[1]{\textsc{#1\/}}
\newcommand{\tT}{{\tilde T}}

\newcommand{\lgr}{\overline{\langle g\rangle}}

\newcommand{\Ju}{\JJJ^{u}}
\newcommand{\Cst}{{{\kk}^*}}
\newcommand{\kst}{{{\kk}^*}}
\newcommand{\Cxn}{\kk[x_{1},\ldots,x_{n}]}
\newcommand{\dx}{\frac{\partial}{\partial x}}
\newcommand{\dxi}{\frac{\partial}{\partial x_{i}}}
\newcommand{\dxj}{\frac{\partial}{\partial x_{j}}}
\newcommand{\dxk}{\frac{\partial}{\partial x_{k}}}
\newcommand{\dy}{\frac{\partial}{\partial y}}
\newcommand{\ddx}{\partial_x}
\newcommand{\ddxi}{\partial_{x_{i}}}

\newcommand{\ddy}{\partial_y}
\newcommand{\ab}[2]{\frac{\partial#1}{\partial#2}}
\newcommand{\hotimes}{{\widehat{\otimes}}}
\newcommand{\kplus}{{\kk^{+}}}
\newcommand{\alg}{\text{\it alg}}
\newcommand{\A}[1]{\AA^{#1}}
\newcommand{\EndA}[1]{\End(\AA^{#1})}
\newcommand{\AutA}[1]{\Aut(\AA^{#1})}
\newcommand{\Autgr}{\Aut_{\text{\it gr}}}
\newcommand{\SAutA}[1]{\SAut(\AA^{#1})}
\newcommand{\AutlfA}[1]{\Aut^{\text{\it lf}}(\AA^{#1})}
\newcommand{\Endlf}{\End^{\text{\it lf}}}
\newcommand{\Autlf}{\Aut^{\text{\it lf}}}

\newcommand{\AutUA}[1]{\Aut^{\text{\it u}}(\AA^{#1})}
\newcommand{\AutU}{\Aut^{\text{\it u}}}
\newcommand{\AutTrA}[1]{\Aut^{\text{\it tr}}(\AA^{#1})}
\newcommand{\TameA}[1]{\Tame(\AA^{#1})}
\newcommand{\J}[1]{{\JJJ}(#1)}
\newcommand{\SA}{\mathcal{SA}}
\newcommand{\SJ}{\mathcal{SJ}}
\newcommand{\Aone}{\AA^{1}}
\newcommand{\Atwo}{\AA^{2}}
\newcommand{\Atwod}{\dot\AA^{2}}
\newcommand{\Athree}{\AA^{3}}
\newcommand{\Ainfty}{\AA^{\infty}}
\newcommand{\iii}{\langle \ii_{1},\ii_{2},\ii_{3}\rangle}
\newcommand{\Ques}[2]{\ps\noindent{\bf Question~\ref{#1}. {\it\ignorespaces#2}}\par\smallskip}
\newcommand{\Autorb}{\Aut_{\text{{\it orb}}}}
\newcommand{\Autinv}{\Aut_{\text{{\it inv}}}}

\newcommand{\gn}{\gg^{\text{\it nil}}}
\newcommand{\gln}{\gg^{\text{\it ln}}}
\newcommand{\Lieln}{\Lie^{\text{\it ln}}}
\newcommand{\LLie}{\LLL_{\Lie}}
\newcommand{\kk}{\Bbbk}
\newcommand{\KK}{\mathbb K}
\newcommand{\tX}{{\tilde X}}
\newcommand{\tphi}{{\tilde\phi}}
\newcommand{\onto}{\twoheadrightarrow}
\newcommand{\surto}{\twoheadrightarrow}
\newcommand{\psmall}{\par\smallskip}
\newcommand{\ps}{\par\smallskip}
\newcommand{\pmed}{\par\medskip}
\newcommand{\bfit}[1]{\textbf{\textit{#1}}}
\newcommand{\Aoned}{{\dot\AA}^{1}}

\newcommand{\be}{\begin{enumerate}}
\newcommand{\ee}{\end{enumerate}}
\newcommand{\dis}{\displaystyle}

\newcommand{\idx}[1]{\index{#1}}
\newcommand{\itind}[1]{{\it #1}\idx{#1}}

\makeindex

\title[On the geometry of automorphism groups]{On the geometry of the \\automorphism groups of affine varieties}
\author{Jean-Philippe Furter and Hanspeter Kraft}

\address{Jean-Philippe Furter \newline
\indent Department of Mathematics, University of La Rochelle, \newline
\indent av. Cr\'epeau, F-17000 La Rochelle, France}
\email{jpfurter@univ-lr.fr}

\address{Hanspeter Kraft \newline
\indent Departement Mathematik und Informatik, Universit\"at Basel, \newline
\indent Spiegelgasse 1, CH-4051 Basel, Switzerland}
\email{hanspeter.kraft@unibas.ch}

\date{Version from August, 2018}

\thanks{The first author gratefully acknowledges the French National Research Agency Grant "BirPol" ANR-11-JS01-004-01 as well as the  R\'egion Poitou-Charentes. 
The second author was partially supported by the SNF (Schweizerischer Nationalfonds)}

\begin{document}
\begin{abstract}
This article is a survey on ind-varieties and ind-groups introduced by \name{Shafarevich} in 1965, with a special emphasis on automorphism groups of affine varieties and actions of ind-groups on ind-varieties. We give precise definitions and complete proofs, including several known results. The survey contains many examples and also some questions which came up during our work on the subject. 

Among the new results we show that for an affine variety $X$ the automorphism group $\Aut(X)$ is always locally closed in the ind-semigroup $\End(X)$ of all endomorphisms, and we give an example of a strict closed subgroup of a connected ind-group which has the same Lie algebra, based on the work of
\name{Shestakov-Umirbaev} on the existence of non-tame automorphisms of $\Athree$.
\end{abstract}

\maketitle
{\footnotesize
\tableofcontents}

\section*{Introduction}

In a lecture given in Rome in 1965 \name{Shafarevich} \cite{Sh1966On-some-infinite-d} introduced the concepts of an {\it infinite-dimensional algebraic variety} and an {\it infinite-dimensional algebraic group}\footnote{We will use the short notation {\it ind-variety} and {\it ind-group}.}, and he announced some striking results. He remarks that a number of interesting groups, like the automorphism group of affine $n$-space $\An$ or the group $\GL(\CC[t])$ (cf. the later paper \cite{Sh2004On-the-group-rm-GL}) have a natural structure of an ind-group.

About 15 years later, in \cite{Sh1981On-some-infinite-d} (with some corrections in \cite{Sh1995Letter-to-the-edit}), \name{Shafarevich} offers a detailed exposition of some material from that lecture. Among other things 
he shows that the tangent space $T_{e}\GGG$ at the unit element $e$ of an ind-group $\GGG$ has a natural structure of a Lie algebra, and  in his main results he proves a very strong connection between the ind-group $\GGG$ and its Lie algebra $\Lie\GGG$ in case of a connected ind-group $\GGG$, similar to what is known for algebraic groups.

\name{Kumar}'s book \cite{Ku2002Kac-Moody-groups-t} contains in Section~IV an introduction to ind-varieties and ind-groups, with complete and detailed proofs, and with \cite{Sh1981On-some-infinite-d} as a basic reference. The author also gives a proof of the main results of \name{Shafarevich}, but under additional assumptions. Let us also mention here the two papers  of \name{Kambayashi} \cite{Ka1996Pro-affine-algebra,Ka2003Some-basic-results} about this subject where the author  points out some minor flaws in the papers of \name{Shafarevich} and thus tries to use a different approach.

\par\smallskip
One of the starting points of our paper is an example which implies that the two main results of \cite{Sh1981On-some-infinite-d}, namely Theorems~1 and 2, are not correct. This example is based on \cite{ShUm2004The-tame-and-the-w} showing that not all automorphisms of affine $3$-space $\Athree$ are tame. The example is given with all details in Section~\ref{tame-is-closed.subsec}. As a consequence, it became unclear how a connected ind-group $\GGG$ is related to its Lie algebra $\Lie\GGG$. 

Another unclear point was the definition of the ind-group structure on $\Aut(\An)$, or more generally, on $\Aut(X)$ for any affine variety $X$. Clearly, $\Aut(X)$ is a subset of $\End(X)$, the set of all morphisms $X \to X$, and the latter is easily seen to have a natural structure of an  ind-semigroup. As a consequence, one can define the ind-structure on $\Aut(X)$ by identifying $\Aut(X)$ with the pairs $(\phi,\psi)$ of endomorphisms such that $\phi \circ \psi=\id_{X}=\psi \circ \phi$. But this definitions does not say anything about the embedding $\Aut(X) \into \End(X)$. 

This problem was solved by our second main discovery. We show that $\Aut(X)$ is closed in $\Dom(X)$, the dominant endomorphisms of $X$, and that $\Dom(X)$ is open in $\End(X)$, see Theorem~\ref{AutX-locally-closed-in-EndX.thm}. Moreover, we show that the tangent space $T_{\id}\End(X)$ embeds into $\VEC(X)$, the Lie algebra of all vector fields on $X$ which explains why $\Lie\Aut(X)$ can always be understood, in a natural way, as a Lie subalgebra of $\VEC(X)$.

Although the counterexample above shows that the relation between $\GGG$ and its Lie algebra $\Lie\GGG$ is not as strong as in the case of algebraic groups, we discovered that for a connected ind-group $\GGG$ two homomorphisms $\phi,\psi\colon \GGG \to \HHH$ of ind-groups are equal if and only if the differentials $d\phi_{e}$ and $d\psi_{e}$ are equal. It is well known that this has important consequences. Together with a careful study of the adjoint representation $\Ad\colon \GGG \to \GL(\Lie\GGG)$ we can show that a connected ind-group $\GGG$ is commutative if and only if its Lie algebra $\Lie\GGG$ is commutative.

\par\smallskip
Working on these problems we came up with many questions, problems, examples and counterexamples. And because of the controversial discussions of some of the results above we decided to include our new results into a survey on ind-groups, automorphism groups, ind-group actions, and the relations to its Lie algebra and to the vector fields, giving precise definitions, detailed proofs, many examples and extensive references, together with some questions which we could not answer. We hope that this survey is going to be useful.

\pmed
\section*{Outline and  Main Results}

\subsection{Ind-varieties}
Our base field $\kk$ is algebraically closed of characteristic zero.
For a variety $X$ we denote by \itind{$\OOO(X)$} the \itind{algebra of regular functions} on $X$, i.e. the global sections of the sheaf $\OOO_{X}$ of regular functions on $X$. 

We start with the definition of an \itind{ind-variety} and an \itind{admissible filtration} of it. An ind-variety $\VVV$ is given by a sequence of closed immersions  $\VVV_{1}\subseteq \VVV_{2} \subseteq \VVV_{3}\cdots$ of algebraic varieties such that $\VVV = \bigcup_{k}\VVV_{k}$. We say that another such filtration is admissible if it endows $\VVV$ with the same structure of an ind-variety (Section~\ref{ind.subsec}).
A typical example is the {\it infinite-dimensional affine space} $\AA^{\infty} := \varinjlim\AA^{n}$ (Example~\ref{A-infinity.exa}). 

An ind-variety carries a natural topology, the \itind{Zariski topology}, and we can define the \itind{dimension} $\dim\VVV$\idx{dimension@$\dim\VVV$} of an ind-variety $\VVV=\bigcup_{k}\VVV_{k}$ as $\dim\VVV := \sup_{k}\dim \VVV_{k}$ (Section~\ref{ind.subsec}).

\begin{theorem} 
Let $\VVV$ be an ind-variety.
\be
\item
$\VVV$ is connected if and only if there is an admissible filtration consisting of connected varieties (Proposition~\ref{connected.prop}).
\item 
$\VVV$ is curve-connected if and only if there is an admissible filtration consisting of irreducible varieties (Proposition~\ref{curve-connected.prop}).
\item
The connected components of $\VVV$ are open and closed, and the number of connected components is countable (Proposition~\ref{connected-components-are-open.prop}).
\ee
\end{theorem}
An ind-variety $\VVV$ is {\it affine} if every closed algebraic subset $X \subset \VVV$ is affine. Equivalently, there is an admissible filtration $\VVV = \bigcup_{k}\VVV_{k}$ with affine varieties $\VVV_{k}$ 
and all admissible filtrations $\VVV = \bigcup_{k}\VVV_{k}$ are such that the $\VVV_{k}$  are affine varieties (Section~\ref{affine-ind.subsec}).

\begin{theorem}
An ind-variety $\VVV$ is affine if and only if there exists a closed 
immersion $\VVV \into \AA^{\infty}$ (Theorem~\ref{embedding-into-Ainfty.thm}).\idx{affine ind-variety}
\end{theorem}

We then define \itind{morphisms} between ind-varieties (Section~\ref{ind.subsec}) and \itind{tangent spaces} of ind-varieties (Section~\ref{tangent-space.subsec}). Some results only hold for uncountable base fields $\kk$; the method of {\it base field extensions} $\KK/\kk$ where $\KK$ is algebraically closed will help out in certain cases (Section~\ref{base-field-extension.subsec}). 

\begin{proposition}
Every ind-variety $\VVV$ is defined over an algebraically closed countable base field $\kk$ (Proposition~\ref{defined-over-countable-field.prop}).
\end{proposition}

\ps
\subsection{Ind-groups}
An \itind{ind-group $\GGG$} is an ind-variety with a group structure such that multiplication and inverse are morphisms (Section~\ref{Liealgebra.subsec}). For an ind-group, the properties curve-connected, connected and irreducible are equivalent (Remark~\ref{irreducible-curve-connected.rem}).

\begin{proposition}
The connected component $\GGG^{\circ}$ of the neutral element of an ind-group $\GGG$ is an open and closed normal subgroup of countable index (Proposition~\ref{connected-component.prop}).
\end{proposition}
\idx{$\GL$@$\GGG^{\circ}$}

Examples are $\GL_{\infty}(\kk) = \varinjlim \GL_{n}(\kk)$\idx{$\GL_{\infty}(\kk)$}, the rational points $G(R)$\idx{$\GL$@$G(R)$} of a linear algebraic group $G$ over a finitely generated $\kk$-algebra $R$, and, as we will see below, the automorphism group \itind{$\Aut(X)$} of an affine variety $X$ (Examples~\ref{examples-of-ind-groups.exa}).

The \itind{tangent space} $T_{e}\GGG$ of an affine ind-group has the structure of a \itind{Lie algebra}, denoted by  \itind{$\Lie\GGG$}, and for every homomorphism $\phi\colon \GGG \to \HHH$ of ind-groups the differential $d\phi_{e}\colon \Lie\GGG \to \Lie\HHH$ is a homomorphism of Lie algebras (Section~\ref{Liealgebra.subsec}).\idx{$\Tr$@$T_{e}\GGG$}

Because of the problems concerning closed ind-subgroups mentioned in the introduction we discuss bijective morphisms of ind-varieties (Section~\ref{small-fibers.subsec}) and homomorphisms of ind-groups with ``small'' kernels (Section~\ref{small-kernels.subsec}). 

\begin{proposition}
Let $\GGG$ be an ind-group, $G$ a linear algebraic group, and $\phi\colon \GGG \to G$ a homomorphism of ind-groups.
\be
\item
If $\dim\Ker\phi<\infty$, then $\GGG^{\circ}$ is an algebraic group. In particular, $\dim\GGG<\infty$.
\item
If $\GGG$ is connected, then $\phi(\GGG) \subseteq G$ is a closed subgroup.
\item
If $\GGG$ is connected and $\phi$ surjective, then $d\phi_{e}\colon \Lie\GGG \to \Lie G$ is surjective, and $\Ker d\phi_{e}\supseteq \Lie\Ker\phi$.
\ee
(Proposition~\ref{hom-to-algebraic.prop})
\end{proposition}

We finish the first part by defining {\it families of endomorphisms and automorphisms} and giving  some important properties (Section~\ref{fam-morph.sec}) which will be used in the following sections. Here is an example.

\begin{proposition}
Let $X$ and $Y$ be varieties, and let $\Phi=(\Phi_{y})_{y\in Y}$ be a family of endomorphisms of $X$ parametrized by $Y$. If every $\Phi_{y}$ is an automorphism, then so is $\Phi$ (Proposition~\ref{fam-end-aut.prop}).
\end{proposition}

\ps
\subsection{Automorphism groups}
The main result is the following.

\begin{theorem}
Let $X$ be an affine variety. There exists a universal structure of an affine ind-group on $\Aut (X)$, and   $\Aut (X)$ is locally closed in $\End(X)$. More precisely, 
$$
\begin{CD}
\Aut(X) @>{\subseteq}>{\text{closed}}> \Dom(X) @>{\subseteq}>{\text{open}}> \End(X)
\end{CD}
$$\idx{$\Aut(X)$}\idx{$\End(X)$}\idx{$\Dom(X)$}
where $\Dom (X)$ denotes the ind-semigroup of dominant endomorphisms of $X$ (Theorems~\ref{AutX.thm} and \ref{AutX-locally-closed-in-EndX.thm}).
\end{theorem}
We have the following important relation between the Lie algebra of the ind-group $\Aut(X)$ and the Lie algebra $\VEC(X)$ of vector fields on $X$.\idx{$\VEC(X)$}

\begin{theorem} There is a canonical inclusion $\xi\colon T_{\id}\End(X) \into \VEC(X)$ which induces
injective antihomomorphism of Lie algebras $\xi\colon \Lie\Aut(X) \into \VEC(X)$ (Propositions~\ref{End(X)-and-Vec(X).prop} and \ref{Liealg-VF.prop}).
\end{theorem}

\ps
\subsection{Homomorphisms of groups}
If $G$ and $H$ are linear algebraic groups, then the set $\Hom(G,H)$ of algebraic group homomorphisms has a natural structure of an ind-variety (see Section~\ref{Hom.sec}). If $H = \GL(V)$ where $V$ is a finite-dimensional vector space of dimension $n$, then $\Hom(G,\GL(V))$ can be understood as the  representations of $G$ on $V$, or as the $G$-module structures on $V$. In this case, the $\GL(V)$-orbits on $\Hom(G,\GL(V))$ are the equivalence classes of $n$-dimensional representations. Here is our main result (Theorem~\ref{Main-Hom.prop}).

\begin{theorem}
Let $G, H$ be linear algebraic groups.
\be
\item 
The ind-variety $\Hom(G,H)$ is finite-dimensional.
\item 
If the radical of $G$ is unipotent, then $\Hom(G,H)$ is an affine variety.
\item 
If $G$ is reductive, then $\Hom(G,\GL(V))$ is a
countable union of closed $\GL(V)$-orbits, hence it is strongly smooth of dimension $\leq (\dim V)^{2}$.
\item
If $G^{\circ}$ is semisimple or if $G$ is finite, then $\Hom(G,\GL(V))$ is a finite union of closed $\GL(V)$-orbits and thus a smooth affine algebraic variety of dimension $\leq(\dim V)^{2}$.
\item
If $U$ is a unipotent group, then $\Hom(U,H)$ is an affine algebraic variety of dimension $\leq \dim U \cdot \dim H^{u}$.
\ee
\end{theorem}
The question whether $\Hom(G,H)$ is algebraic or not is answered by the following result (Proposition~\ref{Hom-algebraic.prop}).

\begin{proposition}
For a connected  linear algebraic group $G$ the following assertions are equivalent.
\be
\item[(i)]
The radical $\rad G$ is unipotent;
\item[(ii)]
$G$ is generated by unipotent elements;
\item[(iii)]
The character group of $G$ is trivial;
\item[(iv)]
$\Hom(G,H)$ is an affine variety for any linear algebraic group $H$.
\ee
\end{proposition}

\ps
\subsection{Ind-group actions and fixed points}
An {\it action of an ind-group $\GGG$ on an affine variety $X$} is the same as a homomorphism $\rho\colon \GGG \to \Aut(X)$ of ind-groups. It then follows that the differential $d\rho_{e}\colon \Lie\GGG \to \VEC(X)$, $A \mapsto \xi_{A}$, is a anti-homomorphism of Lie algebras.

\begin{proposition}
Let $\GGG$ be a connected ind-group acting on an affine variety $X$, and let $Y \subseteq X$ be a closed subvariety. Then  $Y$ is $\GGG$-stable if and only if $Y$ is $\Lie\GGG$-invariant, i.e., $Y$ is $\xi_{A}$-invariant for all $A \in \Lie\GGG$ (Proposition~\ref{G-stable-is-LieG-invariant.prop}).\idx{Lie@$\Lie\GGG$-invariant}\idx{G@$\GGG$-stable}
\end{proposition}

An {\it action of an ind-group $\GGG$ on an ind-variety $\VVV$} is a homomorphism
$\GGG \to \Aut(\VVV)$ such that the action map $\GGG\times \VVV \to \VVV$ is a morphism of ind-varieties (Section~\ref{group-actions-orbits.subsec}).
In the same way we define a \itind{representation} of an ind-group on a $\kk$-vector space $V$ of countable dimension (Section~\ref{reps-of-ind-groups.subsec}), namely as a homomorphism $\rho\colon \GGG \to \GL(V)$ such that the linear action $\GGG \times V \to V$ is a morphism. 
Note that $\GL(V)$ is not an ind-group if $V$ is not finite-dimensional, so that we cannot define a representation as a homomorphism $\GGG \to \GL(V)$ of ind-groups.

For a representation $\rho$  the differential $d\rho_{e}\colon \Lie\GGG \to \LLL(V)$ is well defined and is a homomorphism of Lie algebras. Here \itind{$\LLL(V)$} denotes the \itind{linear endomorphisms} 
of the $\kk$-vector space $V$. 

If $v_{0} \in \VVV$ is a fixed point, then every $ g \in\GGG$ induces a linear automorphism $d g_{v_{0}}$ of $T_{v_{0}}\VVV$  (Section~\ref{tangent-rep-fixed-points.subsec}) which defines the \itind{tangent representation} $\tau_{v_{0}}\colon \GGG \to \GL(T_{v_{0}}\VVV)$.

\begin{theorem}
Let the affine ind-group $\GGG$ act on the ind-variety $\VVV$, and assume that $v_{0}\in\VVV$ is a fixed point. Then the action of $\GGG$ on $T_{v_{0}}\VVV$ is a linear representation (Theorem~\ref{rep-in-fixed-point.thm}).
\end{theorem}

If  $\GGG$ acts on $\VVV$ and if $v \in \VVV$ we denote by $\mu_{v}\colon \GGG \to \VVV$ the \itind{orbit map} $ g \mapsto  g v$. The next lemma is crucial.

\begin{lemma}
Assume that $\GGG$ is connected. If $(d\mu_{v})_{e}\colon \Lie\GGG \to T_{v}\VVV$ is the zero map, then $\mu_{v}$ is constant, i.e. $v$ is a fixed point of $\GGG$ (Lemma~\ref{orbitmap.lem}).
\end{lemma}

There are a number of important consequences.

\begin{proposition}
Let $\phi,\psi \colon \GGG \to \HHH$ be two homomorphisms of ind-groups. If $\GGG$ is connected and 
$d\phi_e=d\psi_e$, then  $\phi=\psi$.
In particular, $\phi$ is trivial if and only if $d\phi_e$
is trivial (Proposition~\ref{phi-dphi.prop}).
\end{proposition}

\begin{corollary}
If $\GGG$ is connected, then the canonical homomorphism (of abstract groups) 
$\omega\colon\Aut(\GGG) \to \Aut(\Lie\GGG)$, $\phi\mapsto d\phi_{e}$, is injective (Corollary~\ref{AutG-into-AutLie.cor}).
\end{corollary}

\ps
\subsection{Adjoint representation}
Consider the action of the affine ind-group 
$\GGG$ on itself by conjugation: $ g \mapsto \Int g: h \mapsto  g\cdot h\cdot  g^{-1}$. This defines the \itind{adjoint representation} (Section~\ref{tangent-adjoint.subsec})
$$
\Ad \colon \GGG \to \GL(\Lie\GGG).
$$

\begin{corollary}
Let $\GGG$ be connected and let $h \in \GGG$. Then $h$ belongs to the center of $\GGG$ if and only if $\Ad(h)$ is trivial:
$$
Z(\GGG) = \Ker\Ad\colon\GGG \to \GL(\Lie\GGG)
$$
(Corollary~\ref{center.cor}).
\end{corollary}\idx{center $Z(\GGG)$}

Finally, one has the following important result.
\begin{proposition}
The differential $\ad \colon \Lie\GGG \to \End(\Lie\GGG)$ is given by
$$
\ad(A)(B) = [A,B] \text{ for } A,B \in \Lie\GGG
$$
(Proposition~\ref{adjoint.prop}).
\end{proposition}\idx{$\Ad$@$\ad$}

\begin{corollary}
Let $\GGG$ be a connected ind-group. Then $\GGG$ is commutative if and only if 
$\Lie\GGG$ is commutative (Corollary~\ref{G-commut-LieG.cor}).
\end{corollary}

\ps   
\subsection{Large subgroups and modifications of \texorpdfstring{$\kplus$}{kplus}-actions}
Let $X$ be an affine variety with an action
of a linear algebraic group $G$. Denote by $\rho \colon G \to \Aut (X)$ the corresponding homomorphism of ind-groups. 
A morphism $\alpha \colon X \to G$ is called {\it $G$-invariant}\idx{G@$G$-invariant}  if it satisfies $\alpha(g x) = \alpha(x)$ for $g \in G$ and $x \in X$. For every such a $G$-invariant morphism $\alpha$
we can define an automorphism $\rho_{\alpha}\in \Aut(X)$  in the following way (see Section~\ref{large-subgroups.subsec}): 
$$
\rho_{\alpha}(x):=\alpha(x) x \text{ \ for \ }x \in X.
$$ 
Note that the $G$-invariant morphisms $\alpha$ form a group, namely
$$
\Mor(X,G)^G = G ( \OOO (X) )^G = G ( \OOO (X)^G ).
$$
One shows the following (see Proposition~\ref{Large-subgroups.prop} and Proposition~\ref{tilde-rho-injective.prop}).
\begin{proposition} 
\be
\item
The map $\tilde\rho\colon G(\OOO(X)^{G}) \to \Aut(X)$, $\alpha\mapsto \rho_{\alpha}$, is a homomorphism of ind-groups.
\item
The ind-group $G(\OOO(X)^{G})$ has the same orbits in $X$ than $G$.
\item
The image of $\tilde\rho$ is contained in the  subgroup 
\begin{multline*}
\Aut_{\rho}(X):= \\ \{\phi\in\Aut(X)\mid \phi(Gx) = Gx \text{ for all } x \in X, \,\phi|_{Gx}=\rho(g_{x}) \text{ for some }g_{x}\in G\}.
\end{multline*}
\item
The subgroup $\Aut_{\rho}(X) \subseteq \Aut(X)$ is closed, and we have the following inclusions
\[
\Aut_{\rho}(X) \subseteq \Autorb(X) \subseteq \Autinv(X) \subseteq  \Aut(X).
\]
(see Definition~\ref{Some-subgroups-of-Aut(X).def})
\item
If $G$ acts faithfully on $X$ and if $\OOO(X)^{G}\neq\kk$,  then the image of $\tilde\rho$ is strictly larger than $\rho(G)$. 
\item
Assume that $\bigcap_{g \in G}G_{g x}=\{e\}$ for all $x$ from a dense subset of $X$, 
then $\tilde\rho$ is injective. This holds in particular for a faithful action of a reductive group $G$ on an irreducible $X$
(Lemma~\ref{NG-for-reductive.lem}).
\ee
\end{proposition}
For $G=\kplus$, we get an injection $\OOO(X)^{\kplus} \into \Aut(X)$ of the invariant ring considered as an additive ind-group. In particular, every invariant $f\in\OOO(X)^{\kplus}$ defines a $\kplus$-action $\rho_{f}\colon \kplus = \kk f \to \Aut(X)$, often called a {\it modification} of the action $\rho$ (Section~\ref{modification.subsec}).

\begin{lemma}
The modification $\rho_{f}$ commutes with $\rho$, and the $\rho_{f}$-orbits are contained in the $\rho$-orbits. If $X_{f} := \{ x \in X \mid f(x) \neq 0 \}$ is dense in $X$, then both actions have the same invariants. Moreover, the fixed point sets of the two actions are related by 
$$
X^{\rho_{f}}= X^{\rho}\cup\{f=0\}
$$
(Lemma~\ref{modification.lem}).
\end{lemma}

\ps
\subsection{Bijective homomorphisms}
An important result in the theory of algebraic groups in characteristic zero  is that bijective homomorphisms are isomorphisms. This does not carry over to ind-groups. 
\begin{proposition}
Denote by $\CC \langle x,y \rangle$ the free associative algebra in two generators.
The canonical map $\Aut(\CC \langle x,y \rangle) \to \Aut(\CC[x,y])$ is a bijective homomorphism of ind-groups, but not an isomorphism. More precisely, the induced homomorphism $\Lie\Aut(\CC\langle x,y\rangle) \to \Lie \Aut(\CC[x,y])$ is surjective with a nontrivial kernel (Proposition~\ref{diff-not-bijective.prop}).
\end{proposition}
\noindent
(This example is already mentioned in \cite[Section~11, last paragraph]{BeWi2000Automorphisms-and-}.)
If we make stronger assumptions for the target group, then this cannot happen.

\begin{proposition}
Let $\phi\colon \GGG \to \HHH$ be a bijective homomorphism of ind-groups. Assume that $\GGG$ is connected and $\HHH$ strongly smooth in $e$. Then $\phi$ is an isomorphism (Proposition~\ref{bijective-Hom.prop}).
\end{proposition}

\ps
\subsection{Nested ind-groups}
An ind-group $\GGG$ is \itind{nested} if $\GGG$ has an admissible filtration $\GGG = \bigcup_{k}\GGG_{k}$ consisting of closed algebraic subgroups $\GGG_{k}$.
The following result generalizes \cite[Remark 2.8]{KoPeZa2016On-Automorphism-Gr}, see Lemma~\ref{injective-homomorphism-into-nested.lem}.

\begin{lemma}
Let $\phi\colon \GGG \to \HHH$ be an injective homomorphism of ind-groups where $\GGG$ is connected and $\HHH$ nested. Then $\GGG$ is nested.
\end{lemma}

It is clear that any element of a nested ind-group is locally finite. We prove a partial converse of this (see Proposition~\ref{locally-finite-commutative-ind-group.prop}).

\begin{proposition} Assume that $\kk$ is uncountable.
Let $X$ be an affine variety, and  let $\GGG \subseteq \Aut(X)$ be a commutative closed connected subgroup. If every element of $\GGG$ is locally finite, then $\GGG$ is nested.
\end{proposition}

The following result is well known for commutative linear algebraic groups (Proposition~\ref{ss-unipotent-for-nested.prop}).
\begin{proposition}
Let $\GGG$ be a commutative nested ind-group.
\be
\item 
The subsets $\Gss$ of semisimple elements and $\Gu$ of unipotent elements of $\GGG$ are closed subgroups, and $\GGG = \Gss \times \Gu$.
\item 
$\Gu$ is a nested unipotent ind-group isomorphic to the additive group of a vector space of 
countable dimension.
\item 
$(\Gss)^{\circ}$ is a \itind{nested torus}, i.e. a finite dimensional torus or isomorphic to
$(\kk^{*})^{\infty}:=\varinjlim \,(\kst)^{k}$.
\item
There is a closed discrete subgroup $\FFF \subset \Gss$ such that $\Gss = \FFF\cdot (\Gss)^{\circ}$.
\ee
\end{proposition}

\ps
\subsection{Normalization}
Let $X$ be an irreducible affine variety and $\eta\colon \tX \to X$ its normalization\idx{normalization}. It is well known that every automorphism of $X$ lifts to an automorphism of $\tX$. Thus we have an injective  homomorphism of groups $\iota\colon\Aut(X) \to \Aut(\tX)$. Moreover, any action of an algebraic group $G$ on $X$ lifts to an action on $\tX$.
More generally, if $\phi\colon X \to X$ is a dominant morphism, then there is a unique ``lift'' $\tphi\colon\tX \to\tX$, i.e. a morphism $\tphi$ such that $\eta\circ\tphi = \phi\circ\eta$, and $\tphi$ is also dominant. This shows that we get an injective map $\iota\colon\Dom(X) \to \Dom(\tX)$. 

The following general result holds (see Proposition~\ref{normalization.prop}).

\begin{proposition}
The map $\iota\colon\Dom(X)\to\Dom(\tX)$ is a closed immersion of ind-semigroups, and $\iota\colon\Aut(X) \to \Aut(\tX)$ is a closed immersion of ind-groups.
\end{proposition}

\ps
\subsection{Automorphisms of \texorpdfstring{$\A{n}$}{An}}\label{Aut-An.subsec}
In the third part we study the automorphisms of affine
$n$-spaces $\An$.

The so-called \itind{locally finite} automorphisms turn out to play a central role.
An automorphism $\g$ of $\An$ is called locally finite when there exists a constant $C$ such that all iterates $\g^k$, $k \geq 1$, satisfy $\deg ( \g^k) \leq C$, see Definition~\ref{locfinite.def}. Equivalently, this means that $\g$ belongs to a linear algebraic group included into $\AutA{n}$.

The group $\AutA{n}$ contains two important closed subgroups, the group \itind{$\Aff(n)$} of \itind{affine automorphisms} and the \name{de Jonqui\`eres}\footnote{\name{Ernest Jean Philippe Fauque de Jonqui\`eres}, 1820-1901} 
subgroup \itind{$\JJJ(n)$} of {\it triangular automorphisms}\idx{triangular automorphism}, both consisting of locally finite automorphisms:\idx{$\Tame(\An)$}
\begin{align*}
\Aff(n)&:=\{\g\in\AutA{n}\mid \deg \g = 1\}= \GL(n)\ltimes \Tr (n), \\
\JJJ(n)&:=\{\g=(g_{1},\ldots,g_{n})\mid g_{i}\in\CC[x_{i},\ldots,x_{n}]\text{ for }i=1,\ldots, n\},
\end{align*}
where $\Tr(n) \simto (\CC^{+})^{n}$\idx{$\Tr(n)$} are the {\it translations}\idx{translation}. The subgroup generated by $\Aff(n)$ and $\JJJ(n)$ is called the group of {\it tame automorphisms}\idx{tame automorphism}:
$$
\TameA{n} :=\langle \Aff(n), \JJJ(n) \rangle.
$$
An element  $\g \in \AutA{n}$ is called \itind{triangularizable} if $\g$ is conjugate to an element of $\JJJ(n)$. We denote by $\AutTrA{n} \subseteq \AutlfA{n}$ the subset of triangularizable automorphisms.
The element $\g$ is called \itind{linearizable} if it is conjugate to an element of $\GL(n)$, and \itind{diagonalizable} if it is conjugate to an element of $D(n)\subseteq \GL(n)$, the diagonal matrices.

For $\g\in\AutA{n}$ we denote by $C(\g) :=\{\h^{-1}\cdot \g\cdot \h \mid \h\in\AutA{n}\}$ its conjugacy class in $\AutA{n}$, and by $\langle \g \rangle\subseteq \AutA{n}$ the subgroup generated by $\g$.
Recall that an algebraic group $D$ is called {\it diagonalizable}\idx{diagonalizable group} if it is isomorphic to a closed subgroup of a torus $(\Cst)^{m}$. A diagonalizable group has the form $F \times (\Cst)^{d}$ where $F$ is finite and commutative.

The \itind{jacobian determinant} defines a character $\jac\colon \Aut(\An) \to \kst$, and we set $\SAut(\An):=\Ker\jac$. For the Lie algebras we have the following result.

\begin{proposition}
The identification $\xi\colon \End(\An)=T_{\id}\End(\An) \simto \VEC(\An)$ induces the following anti-isomorphisms of Lie algebras
\be
\item $\Lie\Aut(\An) \simto \VEC^{c}(\An):=\{\delta \in \VEC (\An) \mid \Div\delta \in \kk\}$,
\item $\Lie \SAut(\An) \simto \VEC^{0}(\An):=\{\delta \in \VEC (\An) \mid \Div \delta  = 0 \}$,
\ee
where $\Div$ is the divergence of a vector field
(Proposition~\ref{LieAlgAutAn.prop}).
\end{proposition}

\ps
\subsection{Main results about \texorpdfstring{$\Aut(\An)$}{Aut(An}}
Here is a collection of results which will be given in the third part. Some are certainly known to the specialists. For the topological notion \itind{weakly closed}, \itind{weak closure} $\wc{C}$, and \itind{weakly constructible} we refer to Section~\ref{top.subsec}.

\begin{theorem}
\be
\item An automorphism $\g \in \AutA{n}$ is locally finite if and only if the closure $\overline{\langle \g \rangle}\subseteq \AutA{n}$ is an algebraic group. In this case $\overline{\langle \g \rangle}$ is isomorphic to $D$ or $D\times \kplus$ where $D$ is diagonalizable and $D/D^{\circ}$ cyclic (Section~\ref{Locally-finite.subsec}).
\item Every locally finite $\g \in \AutA{n}$ has a uniquely defined Jordan decomposition $\g = \g_{s}\cdot \g_{u}$ where $\g_{s}$ is semisimple, $\g_{u}$ is unipotent and both commute (Section~\ref{Locally-finite.subsec}).
\idx{Jordan decomposition}
\item If $\g$ is semisimple, then $\g$ has a fixed point (Proposition \ref{semisimple-auto-has-fixed-point.prop}).
\item The subset $ \AutlfA{n} \subseteq \AutA{n}$ of locally finite automorphisms is weakly closed (Proposition~\ref{Glf-weakly-closed.prop}).
\item If $\g\in \AutA{n}$ has a fixed point in $\A{n}$, then the weak closure $\wc{C(\g)}$ contains a linear automorphism (Proposition~\ref{Anclosure.prop}).
\item The weak closure $\wc{C(\g)}$ of the conjugacy class of a semisimple element $\g$ consists of semisimple elements (Proposition~\ref{fam.prop}(3)).
\item The conjugacy class  of a diagonalizable element $\g$ is weakly closed (Corollary~\ref{ss-weakly-closed.cor}).
\item The unipotent elements $\AutUA{n} \subseteq\AutA{n}$ form a weakly closed subset (Section~\ref{unipotent.subsec}).
\item The conjugacy classes of nontrivial unipotent triangular automorphisms $\u$ all have the same weak closure $\wc{C(\u)}$  and thus the same closure (Proposition~\ref{triangclosure.prop}). The weak closure $\wc{C(\u)}$ contains all triangularizable unipotent elements.
\item The group $\AutA{n}$ is connected (Proposition~\ref{k*-connected.prop}) and for $n \ge 2$
acts infinitely-transitively on $\An$ (Proposition~\ref{inftrans.prop}).
\ee
\end{theorem}

In case $n=2$ these results can be improved (Theorem~\ref{constr.thm} and \ref{henon.thm}).
\begin{theorem}
\be
\item Conjugacy classes in $\AutA{2}$ are weakly constructible.
\item An automorphism $\g\in \AutA{2}$ is semisimple if and only if its conjugacy class $C(\g)$ is closed.
\item The locally finite elements $\AutlfA{2} \subseteq \AutA{2}$ as well as the unipotent elements $\AutUA{2} \subseteq \AutA{2}$ form closed subsets.
\item For every locally finite $\g\in \AutA{2}$ we have $\g_{s}\in\wc{C(\g)}$.
\item For a nontrivial unipotent $\u\in\AutA{2}$ we get  $\wc{C(\u)} = \overline{C(\u)} = \AutUA{2}$.
\item The subgroup $\SAutA{2}\subset \AutA{2}$ is the only nontrivial closed connected normal subgroup of $\AutA{2}$. 
\ee
\end{theorem}

\ps
\subsection{Tame automorphisms of \texorpdfstring{$\Aut(\Athree)$}{Aut(A3)}}
In 2003, \name{Shestakov} and \name{Umirbaev} settled an old problem by showing that the Nagata automorphism $\n\in\AutA3$ (Section~\ref{shifted-lin.subsec}) is  not tame (\cite{ShUm2003The-Nagata-automor,ShUm2004The-tame-and-the-w}). \name{Edo} and \name{Poloni} showed in \cite{EdPo2015On-the-closure-of-} that the tame automorphisms $\TameA3\subseteq \AutA3$ do not form a closed subgroup. In fact, $\Tame(\Athree)$ is even not weekly closed in $\Aut(\Athree)$. We will prove the following result in the ``opposite'' direction (Section~\ref{tame-is-closed.subsec}).
Consider the closed subgroup $\GGG \subseteq \AutA{3}$ of those automorphisms of $\A{3}$ which leave the projection $\pr_{3}\colon \A{3} \to \A{1}$ invariant,
\[ 
\GGG:=\{\f=(f_{1},f_{2},f_{3})\in \AutA{3} \mid f_{3}=z\},
\]
and let $\tG:= \GGG \cap \TameA{3} \subseteq \GGG$ be the subgroup consisting of tame elements. We have the following result (see Theorem~\ref{closed-subgroup-with-same-Liealgebra.thm}).
\begin{theorem}
\be
\item $\tG \subseteq \GGG$ is a closed subgroup, and $\tG \neq \GGG$;
\item $\GGG$ is connected;
\item $\Lie \tG = \Lie \GGG$.
\ee
\end{theorem}

This theorem is in contrast to some results claimed by \name{Shafarevich} in \cite{Sh1981On-some-infinite-d,Sh1995Letter-to-the-edit}.

\newpage
\part{GENERALITIES ON IND-VARIETIES AND IND-GROUPS}
The first part of the survey is a general introduction to ind-varieties and ind-groups and some of their basic properties. 
We give several examples illustrating the concept of ind-varieties and ind-groups and also some unexpected behavior of these objects.

\section{Ind-Varieties and Morphisms}
\subsection{Ind-varieties, Zariski topology, and dimension}\label{ind.subsec}
The notion of an \itind{ind-variety} and an \itind{ind-group} goes back to \name{Shafarevich} 
 who called these
objects {\it infinite-dimensional algebraic varieties}\idx{infinite-dimensional algebraic variety} and {\it infinite-dimensional algebraic groups}\idx{infinite-dimensional algebraic group}, see \cite{Sh1966On-some-infinite-d,Sh1981On-some-infinite-d,Sh1995Letter-to-the-edit}). 

\begin{definition}\label{indvar.def}
An \itind{ind-variety} $\VVV$ is a set together with an ascending filtration $\VVV_{0}\subseteq \VVV_{1}\subseteq \VVV_{2}\subseteq \cdots\subseteq \VVV$ such that the following holds:
\be
\item $\VVV = \bigcup_{k \in \NN}\VVV_{k}$;
\item Each $\VVV_{k}$ has the structure of an algebraic variety;
\item For all $k \in \NN$, the inclusion  $\VVV_{k}\into \VVV_{k+1}$ is a closed immersion of algebraic varieties.
\ee
\end{definition}

A \itind{morphism} between ind-varieties $\VVV$ and $\WWW$  is a map $\phi\colon \VVV \to \WWW$  such that for any $k$ there is an $l$ such that $\phi(\VVV_{k}) \subseteq \WWW_{l}$ and that the induced map $\VVV_{k}\to \WWW_{l}$ is a morphism of varieties.
In the sequel, a morphism of ind-varieties will often be called an \itind{ind-morphism}, and we denote by \itind{$\Mor(\VVV,\WWW)$} the set of ind-morphisms $\phi\colon \VVV\to\WWW$.

An \itind{isomorphism} of ind-varieties is defined in the obvious way: it  is a bijective morphism $\phi\colon \VVV \to \WWW$ such that $\phi^{-1}\colon \WWW \to \VVV$ is also a morphism.

Two ind-variety structures $\VVV = \bigcup_{k \in \NN} \VVV_{k}$ and $\VVV = \bigcup_{k \in \NN} \VVV_{k}'$ on the same set $\VVV$ are called {\it equivalent\/}\idx{equivalent structures} if the identity map $\id \colon \VVV = \bigcup_{k \in \NN} \VVV_{k} \to \VVV = \bigcup_{k \in \NN} \VVV_{k}'$ is an isomorphism. This means that for any $k$ there is an $\ell$ such that $\VVV_k$ is closed in $\VVV'_\ell$, and for any  $m$ there exists an $n$ such that $\VVV'_m$ is closed in $\VVV_n$.

A  \itind{filtration} on an ind-variety $\VVV=\bigcup_{k}\VVV_{k}$ is an increasing sequence of closed subsets $A_{1}\subseteq A_{2}\subseteq A_{3}\cdots$ such that $\VVV = \bigcup_{\ell}A_{\ell}$. The filtration is called {\it admissible\/}\idx{admissible filtration}
if it defines an equivalent ind-variety structure on $\VVV$. We will freely move between admissible filtrations.

\begin{lemma} \label{good-filtrations-for-morphisms.lem}
Let $\phi \colon \VVV \to \WWW$ be an ind-morphism. Then, for any admissible filtration $ \VVV =\bigcup_{k \in \NN}\VVV_{k}$ of $\VVV$, there exists an admissible filtration $ \WWW =\bigcup_{k \in \NN} \WWW_{k}$ of $\WWW$ such that $\phi(\VVV_{k}) \subseteq \WWW_{k}$ for each $k$.
\end{lemma}

\begin{proof}
Let $ \WWW =\bigcup_{k \in \NN}\WWW'_{k}$ be any admissible filtration of $\WWW$. Since $\phi$ is a morphism, there is a strictly increasing sequence $(m_k)_{k}$ such that $\phi(\VVV_{k}) \subseteq \WWW'_{m_k}$ for each $k$. It is then enough to set $\WWW_k:= \WWW'_{m_k}$ for each $k$.
\end{proof}

\begin{definition} \label{The-Zariski-topology-of-an-ind-variety.def}
The \itind{Zariski topology} of an ind-variety $\VVV=\bigcup_{k}\VVV_{k}$ is defined by declaring a subset $U \subseteq \VVV$ to be \itind{open} if the intersection $U \cap\VVV_{k}$ is Zariski-open in $\VVV_{k}$ for all $k$. It is obvious that $A \subseteq \VVV$ is \itind{closed} if and only if  $A \cap\VVV_{k}$ is Zariski-closed in $\VVV_{k}$ for all $k$. It follows that a \itind{locally closed} subset $\WWW \subseteq \VVV$ has a natural structure of an ind-variety, given by the filtration $\WWW_{k}:=\WWW \cap \VVV_{k}$ which are locally closed subvarieties of $\VVV_{k}$. These subsets are called {\it ind-subvarieties}.\idx{ind-subvariety} Note that a subset $S \subseteq \VVV$ with the property that $S_{k}:=S \cap \VVV_{k}$ is locally closed in $\VVV_{k}$ for all $k$ is not necessarily locally closed, see  \cite[\S 2.3]{FuMa2010A-characterization}.

\begin{definition}\label{affine-algebraic.def}
\be
\item 
An ind-variety $\VVV$ is called {\it affine}\idx{affine ind-variety} if it admits an admissible filtration such that all $\VVV_{k}$ are affine. It follows that all admissible filtrations of $\VVV$ do have this property. 
\item 
A subset  $X\subseteq \VVV$ is called {\it algebraic}\idx{algebraic subset} if it is locally closed and contained in $\VVV_{k}$ for some $k$. Such an $X$ has a natural structure of an algebraic variety.
\item 
The {\it algebra of regular functions\/}\idx{regular function}  on $\VVV=\bigcup \VVV_{k}$ or the {\it coordinate ring}\idx{coordinate ring} is defined as\idx{$\OOO(\VVV)$} 
$$
\OOO(\VVV):= \Mor(\VVV,\Aone) = \varprojlim \OOO(\VVV_{k})
$$
We endow each algebra $\OOO(\VVV_{k})$ with the discrete topology and $\OOO(\VVV)=  \varprojlim \OOO(\VVV_{k})$ with the inverse limit topology, i.e. with the coarsest topology making all projections $\OOO(\VVV) \to  \OOO(\VVV_{k})$ continuous. For any morphism $\phi\colon \VVV \to \WWW$ the induced homomorphism $\phi^{*}\colon\OOO(\WWW) \to \OOO(\VVV)$ is continuous. Moreover, an affine ind-variety $\VVV$ is uniquely determined by the \itind{topological algebra $\OOO(\VVV)$}, up to isomorphisms.
\ee
\end{definition}

The following lemma can be found in \cite[Lemma~4.1.2]{Ku2002Kac-Moody-groups-t}).
\begin{lemma}\label{Kumar.lem}
Let $\phi\colon X \to \WWW$ be a continuous map where $X$ is an algebraic variety and $\WWW$ is an ind-variety. Then $\overline{\phi(X)} \subseteq \WWW$ is algebraic, i.e. there is an $\ell$ such that $\phi(X) \subseteq \WWW_{\ell}$.
\end{lemma}
\begin{proof}
Assume that $\phi(X)$ is not contained in any $\WWW_{\ell}$, i.e. that 
$\bigcup_{\ell \leq m} \phi^{-1}(\WWW_{\ell})\subsetneqq X$ for all $m\geq 1$. Then we can find an infinite sequence $m_{1}<m_{2}<\cdots$ of natural numbers and points $x_{i}\in X$, $i\in\NN$,  such that $\phi(x_{i})\in \WWW_{m_{i}}\setminus\WWW_{m_{i-1}}$. It follows that the infinite set $S:=\{\phi(x_{1}), \phi(x_{2}),\ldots\} \subseteq \WWW$  is discrete, because $S \cap \WWW_{\ell}$ is finite for every $\ell$. Hence $Z:=\phi^{-1}(S)\subseteq X$ is closed and $Z_{i}:=\phi^{-1}(\phi(x_{i}))\subseteq Z$ is open (and closed) in $Z$ for all $i$, and so  $Z\subseteq X$ has infinitely many connected components which is a contradiction since $Z$ is a variety.
\end{proof}

Next we define the dimension of an ind-variety and draw some easy consequences.
\begin{definition}  \label{dimension.def}
For an ind-variety $\VVV=\bigcup_{k}\VVV_{k}$ we define the {\it local dimension of $\VVV$ in $v\in\VVV$}\idx{local dimension} and the {\it dimension of $\VVV$}\idx{dimension} by
$$
\dim_{v}\VVV := \sup_{k}\dim_{v}\VVV_{k}, \quad \dim \VVV := \sup_{k}\dim\VVV_{k}
=\sup_{v}\dim_{v}\VVV.
$$
\end{definition}
We have $\dim\VVV \leq d <\infty$ if and only if every algebraic subvariety has dimension $\leq d$. In this case $\VVV$ is a countable union of closed algebraic subvarieties of dimension $\leq d$. In particular, $\dim\VVV = 0$ if and only if $\VVV$ is discrete. 

It is also clear that there are no injective morphisms of an ind-variety of infinite dimension into an algebraic variety.

\ps
\subsection{Closed immersions}
A morphism $\phi\colon \VVV \to \WWW$ is called an {\it immersion\/}\idx{immersion} if the image $\phi(\VVV) \subseteq \WWW$ is locally closed and $\phi$ induces an isomorphism $\VVV \simto \phi(\VVV)$ of ind-varieties. An immersion $\phi$ is called a \itind{closed immersion} (resp. an \itind{open immersion}) if $\phi(\VVV) \subseteq \WWW$ is closed (resp. open).
\end{definition}

Our definition of a closed immersion coincides with the one given in \name{Kumar}'s book 
\cite[Section~4.1.1]{Ku2002Kac-Moody-groups-t} as the following lemma shows.

\begin{lemma}   \label{Kumar-criterion.lem}
A morphism $\phi\colon \VVV=\bigcup_{k}\VVV_{k} \to \WWW=\bigcup_{k}\WWW_{k}$ of ind-varieties is a closed immersion if and only if the following holds:
\be
\item[(a)] 
For every $k$ there is an $\ell$ such that $\phi(\VVV_{k}) \subseteq \WWW_{\ell}$ and 
$\phi|_{\VVV_{k}}\colon \VVV_{k} \to \WWW_{l}$ is a closed immersion of varieties.
\item[(b)]
$\phi(\VVV) \subseteq \WWW$ is closed.
\item[(c)]
$\phi\colon \VVV \to \phi(\VVV)$ is a homeomorphism.
\ee
\end{lemma}
\begin{proof}
(1) Assume that $\phi$ is a closed immersion, i.e. that (i) $\phi(\VVV) \subseteq \WWW$ is closed, and (ii) $\phi\colon \VVV \to \phi(\VVV)$ is an isomorphism.
Clearly,   (i) implies  (b) and (ii) implies (c). It remains to prove (a). Given an index $k$ there is an $\ell$ such that $\phi(\VVV_{k})\subseteq \WWW_{\ell}$, because $\phi$ is an ind-morphism. Since $\VVV_k$
is closed in $\VVV$ it follows that $\phi(\VVV_k)$ is closed in $\phi(\VVV)$, and therefore in $\WWW$ and in $\WWW_\ell$). The isomorphism $\VVV \simto \phi(\VVV)$ induces an isomorphism $\VVV_k \simto \phi(\VVV_k)$, proving that $\phi$ induces a closed immersion  $\VVV_k \into\WWW_\ell$.
\psmall
(2) Now assume that (a), (b), (c) are satisfied. Since (b) implies (i)  it remains to prove (ii). Assertions (a) and (b) imply 
that $\phi$ induces a bijective morphism $\phi_{0} \colon \VVV \to \phi(\VVV)$. We have to show that $\psi:= \phi_{0}^{-1} \colon \phi(\VVV) \to \VVV$ is a morphism. By (a) and (b), the image $\phi(\VVV_{k}) \subseteq \WWW$ is closed and the induced map $\phi_{k}\colon \VVV_{k}\simto \phi(\VVV_{k})$ is an isomorphism. Hence, the inverse map $\phi_{k}^{-1}\colon \phi(\VVV_{k}) \simto \VVV_{k}$ is also an isomorphism.

Moreover, $\psi\colon\phi(\VVV) \to \VVV$ is a homeomorphism, and so Lemma~\ref{Kumar.lem} below implies that for every $\ell$ there is a $k$ such that $\psi(\phi(\VVV)\cap \WWW_{\ell}) \subseteq \VVV_{k}$, i.e. $\phi(\VVV)\cap\WWW_{\ell} \subseteq \phi(\VVV_{k})$. Hence $\phi_{k}^{-1}$ induces a morphism (a closed immersion) $\phi(\VVV)\cap\WWW_{\ell} \to \VVV_{k}$, showing that $\psi$ is a morphism.
\end{proof}

\begin{remark}
When the ground field $\kk$ is uncountable, we will show in Lemma~\ref{simplification-Kumar.lem} below that condition (c) is already implied by the conditions (a) and (b).
\end{remark}

The following result is well known for varieties.
\begin{lemma}  \label{closed-immersion.lem}
Let $\VVV$, $\WWW$ and $\UUU$ be ind-varieties, and let $\phi\colon \VVV \to \WWW$, $\psi\colon \WWW \to \UUU$ be ind-morphisms. If the composition $\VVV \xrightarrow{\phi} \WWW \xrightarrow{\psi} \UUU$ is a closed immersion, then so is $\phi\colon \VVV \to \WWW$. 
\end{lemma}

\begin{proof}
Set $\mu:=\psi\circ\phi\colon \VVV \to \UUU$. Replacing $\UUU$ by the closed ind-subvariety $\mu(\VVV)\subset \UUU$ and $\WWW$ by $\psi^{-1}(\mu(\VVV))$ we can assume that $\mu$ is an isomorphism. Clearly, $w\in\WWW$ belongs to the image of $\phi$ if and only if $w = \phi\circ\mu^{-1}\circ \psi(w)$. Setting $\delta\colon \WWW \to \WWW \times \WWW$, $\delta(w):=(w,\phi\circ\mu^{-1}\circ \psi(w))$ we see that 
$\phi(\VVV) = \delta^{-1}(\Delta)$ where $\Delta := \{(w,w)\mid w\in\WWW\} \subset \WWW \times \WWW$ is the diagonal. Thus $\phi(\VVV) \subset \WWW$ is a closed ind-subvariety.

The ind-morphism $\mu^{-1}\circ\psi\colon \WWW \to \VVV$ restricts to an ind-morphism $\bar\psi\colon \phi(\VVV) \to \VVV$ such that $\bar \psi\circ\phi = \id_{\VVV}$, hence $\phi\colon \VVV \simto \phi(\VVV)$ is an isomorphism.
\end{proof}

\ps
\subsection{The case of an uncountable base field \texorpdfstring{$\kk$}{k}}  \label{uncountable-base-field.subsec}
In this section we will prove that for an \itind{uncountable base field $\kk$} any filtration of an ind-variety by closed algebraic subsets is admissible (Theorem~\ref{closed-algebraic-filtration.thm}). Apart from the aesthetic point of view, this allows to simplify some statements, e.g.  the criterion for closed immersions given by \name{Kumar} in Lemma~\ref{Kumar-criterion.lem}, see Lemma~\ref{simplification-Kumar.lem} below. 

\begin{lemma}   \label{constructible.lem}
Let $\kk$ be uncountable.
Let $X$ be a variety and $C \subseteq X$ a constructible subset. Assume that there is a countable set  $\{C_{i}\mid i\in\NN\}$  of constructible subsets $C_{i}\subseteq C$ such that  $C = \bigcup_{i=1}^{\infty}C_{i}$. Then there is a finite subset $F \subseteq \NN$ such that $C = \bigcup_{i\in F} C_{i}$.
\end{lemma}

\begin{proof} Since every constructible set is the image of a morphism we can assume that $C$ is a variety. By assumption, $C \subseteq  \bigcup_{i}\overline{C_{i}}$. Since the base field is not countable, we get $C = \bigcup_{i\in F_{1}}\overline{C_{j}}$ for a finite subset $F_{1} \subseteq \NN$.  This implies that there is a subset $U \subseteq \bigcup_{i\in F} C_{j}$ which is open and dense in $C$. Replacing $C$  by $C':=C\setminus U$ and the  $C_{i}$ by $C_{i}':=C_{i}\setminus U$ we get $\dim C' < \dim C$ and $C' = \bigcup_{i} C_{i}'$. The claim follows by induction on $\dim C$.
\end{proof}

A first and easy application of this lemma is the following.
\begin{proposition}\label{surjective-morphisms-to-varieties.prop}
Assume that the ground field $\kk$ is uncountable. Let $\phi\colon \VVV \to X$ be a surjective morphism where $\VVV$ is an ind-variety and $X$ a variety. Then there is an algebraic subset $Y \subseteq \VVV$ such that $\phi(Y) = X$.
\end{proposition}
\begin{proof}
If $\VVV = \bigcup_{k}\VVV_{k}$, then $X = \bigcup_{k}\phi(\VVV_{k})$, and the claim follows from the lemma above.
\end{proof}

Here is a striking and useful consequence of Lemma~\ref{constructible.lem}.

\begin{theorem}   \label{closed-algebraic-filtration.thm}
Assume that $\kk$ is uncountable, and let $\VVV$ be an ind-variety. Then every filtration of $\VVV$ by closed algebraic subsets is admissible.
\end{theorem}

\begin{proof}
Let $\VVV = \bigcup_{k}\VVV_{k}$ and assume that $\VVV = \bigcup_{\ell}A_{\ell}$ where $A_{1}\subseteq A_{2}\subseteq A_{3}\cdots$ are closed algebraic subsets.
By definition, each $A_{k}$ is contained in some $\VVV_{\ell}$. On the other hand, $\VVV_{k} = \bigcup_{\ell} (\VVV_{k}\cap A_{\ell})$, and so $\VVV_{k}=A_{\ell}\cup\VVV_{k}$ by Lemma~\ref{constructible.lem}, i.e. $\VVV_{k}\subseteq A_{\ell}$.
\end{proof}

\begin{remark}
This result does not hold for a countable field $\kk$, since the affine line $\AA^1$ can be filtered by a sequence of finite sets.
\end{remark}

If $\kk$ is uncountable, it follows from the theorem above that condition (c) of Lemma~\ref{Kumar-criterion.lem} is already implied by the conditions (a) and (b):

\begin{lemma}   \label{simplification-Kumar.lem}
Assume that the ground field $\kk$ is uncountable. Then  a morphism $\phi\colon \VVV=\bigcup_{k}\VVV_{k} \to \WWW=\bigcup_{k}\WWW_{k}$ of ind-varieties is a closed immersion if and only if the following holds:
\be
\item[(a)] 
For every $k$ there is an $\ell$ such that $\phi(\VVV_{k}) \subseteq \WWW_{\ell}$ and 
$\phi|_{\VVV_{k}}\colon \VVV_{k} \to \WWW_{l}$ is a closed immersion of varieties.
\item[(b)]
$\phi(\VVV) \subseteq \WWW$ is closed.
\ee
\end{lemma}

\begin{proof}
It is enough to show that condition (c) of Lemma~\ref{Kumar-criterion.lem} is satisfied. Note that the two conditions above imply that for each $k$
\be
\item $\phi (\VVV_k)$ is closed in $\phi (\VVV)$ and in $\WWW$, and 
\item $\phi$ induces an isomorphism $\VVV_k \simto \phi (\VVV_k)$.
\ee
In order to prove that $\phi \colon \VVV \to \phi (\VVV)$ is a homeomorphism it is enough to show that $\phi$ is a closed map, i.e. that it sends  closed sets to  closed sets. Let $Z \subseteq \VVV$ be any closed subset. This means that $Z \cap \VVV_k$ is closed in $\VVV_k$ for each $k$.  The conditions (1)-(2) above imply that $\phi (Z \cap \VVV_k)= \phi (Z) \cap \phi (\VVV_k) $ is closed in $\phi (\VVV_k)$. Since the filtration of $\phi (\VVV)$ by the closed algebraic subsets $\phi (\VVV_k)$ is admissible, this means that $\phi (Z)$ is closed in $\phi (\VVV)$.
\end{proof}

\ps
\subsection{Examples of ind-varieties}  \label{examples-of-ind-varieties.subsec}

\begin{example}
Every {\it variety $X$} is an ind-variety in a canonical way where we take $X_{k}:=X$ for all $k$. On the other hand, it is clear that an ind-variety $\VVV=\bigcup \VVV_{k}$ is a variety if and only if there is a $k_0$ such that  $\VVV_{k} = \VVV_{k_0}$ for $k \geq k_0$.
\end{example}

\begin{example} \label{countable-vector-space.exa}
Any {\it $\kk$-vector space $V$ of countable dimension} is given the structure of an (affine) ind-variety by choosing an increasing sequence of finite-dimensional subspaces $V_{k}$ such that $V = \bigcup_{k} V_{k}$. Clearly, all filtrations by finite-dimensional subspaces are admissible. 
\end{example}

\begin{example}\label{subst-R-algebra.exa}
If $R$ is a commutative $\kk$-algebra of countable dimension, $\aa \subseteq R$ a subspace (e.g. an ideal), and $P\subseteq\kk[x_{1},\ldots,x_{n}]$ a set of polynomials, then the subset 
\[ 
\{(a_{1},\ldots,a_{n})\in R^{n}\mid f(a_{1},\ldots,a_{n})\in\aa \text{ for all }f\in P\}\subseteq R^{n}
\] 
is a closed ind-subvariety of $R^{n}$. A special case of this is the set $X(R)$ of $R$-rational points of an affine 
variety $X$. We will discuss this in Section~\ref{R-rational-points.subsec}, see Proposition~\ref{X(R).prop}.
\end{example}

\begin{example}\label{discrete-indvar.exa}
Every countable set $S$\idx{countable set} is an ind-variety in a canonical way: $S = \bigcup S_{k}$ where all $S_{k}$ are finite. Clearly, all these filtrations are admissible. These ind-varieties are called {\it discrete}\idx{discrete ind-variety}.
\end{example}

\begin{example} \label{countable-family-of-closed-algebraic-subsets.exa}
Let $(X_{i})_{i\in\NN}$ be a countable set of $\kk$-varieties. Then the disjoint union $\XXX:=\bigcup_{i\in\NN} X_{i}$ has a natural structure of an ind-variety where the filtration is given by the finite disjoint unions $\XXX_{k} = \bigcup_{i=1}^{k}X_{i} = X_{1}\dcup X_{2}\dcup\cdots\dcup X_{k}$. Other admissible filtrations can be obtained by writing $\NN$ as a union of finite subsets $N_{i}$ where $N_{i} \subseteq N_{i+1}$, and setting $\XXX_{k}:=\bigcup_{i\in N_{k}}X_{i}$. 

This ind-variety has a countable set of connected components, namely the connected components of the $X_{i}$. In particular, $\XXX$ is not an algebraic variety. Clearly, $\XXX$ is affine if and only if all $X_{i}$ are affine. Moreover, every $X_{i}\subset \XXX$ is open and closed in $\XXX$. Note that $\dim\XXX = \max_{i}\dim X_{i}$. 

If all the $X_{i}$ are closed subsets  of some variety $X$, then we obtain a natural ind-morphism $\phi\colon \XXX \to X$ in the obvious way. If the $X_{i}\subseteq X$ are disjoint, then $\phi$ is injective, but it is not a closed immersion, because $\XXX$ is not a variety. 

If all the $X_{i}$ are closed algebraic subsets  of some ind-variety $\VVV$, then again we get a natural ind-morphism $\phi\colon \XXX \to \VVV$. But even if the  $X_{i}\subseteq\VVV$ are disjoint, one cannot expect that $\phi$ is a closed immersion since, in general, the union $\bigcup_{i\in\NN}X_{i} \subset \VVV$ is not closed.
\end{example}

\begin{example}\label{A-infinity.exa}
The {\it infinite-dimensional affine space\/}\idx{infinite-dimensional affine space $\A{\infty}$} 
\[
\A{\infty}:=\{(x_{1},x_{2},\ldots) \mid  \text{there is a } k\geq 1\text{ such that }x_{i}=0 \text{ for } i \geq k \}
\]
is an ind-variety in a natural way, given by 
$\A{\infty}:= \bigcup_{n \geq 1} \A{n}$ 
where $\A{n}$ is the \itind{affine $n$-space} $\kk^{n}$ embedded into $\A{n+1}$ via the map $(a_{1},\ldots,a_{n})\mapsto (a_{1},\ldots,a_{n},0)$.
\end{example}\idx{$\AAA$@$\AA^{\infty}$}\idx{$\AAA$@$\An$}

\begin{example}\label{inductive-limit.exa}
If $X_{k}$, $k\in\NN$, are varieties and $\phi_{k}\colon X_{k} \to X_{k+1}$ closed immersions, then the \itind{inductive limit} $\varinjlim X_{k}$ is an ind-variety in a canonical way. If all $X_{k}$ are affine, then $\varinjlim X_{k}$ is affine.
\end{example}
\idx{$\LLL$@$\underset{\rightarrow}{\lim} X_{k}$}

\begin{proposition}  \label{indlimit-of-affine spaces.prop}
For any sequence of closed immersions $\phi_k \colon \A{k} \to \A{k+1}$, $k \geq 1$, we have 
$\varinjlim \A{k}  \simeq \Ainfty$.
\end{proposition}

\begin{proof}
The proof relies on the following statement proved by \name{Srinivas} in \cite[Theorem 2]{Sr1991On-the-embedding-d}. Assume that $\psi_1, \psi_2 \colon V \to \AA^n$ are  two closed embeddings of a smooth affine irreducible variety $V$ of dimension $d$  into an affine space $\AA^n$ of dimension $n > 2d+1$. {\it Then there exists an automorphism $\phi$ of $\AA^n$ such that $\psi_2 =  \phi \circ \psi_1$.}

Let $d_k$ be any sequence of non-negative integers satisfying $d_{k+1} > 2 d_k +1$.
For each $k \geq 0$, set $V_k:= \A{d_k}$ and let $\widetilde{\psi}_k$, $\tilde{\iota}_k \colon V_k \to V_{k+1}$ be the morphisms defined by the following compositions:
\[
\begin{CD}
\widetilde{\psi}_k \colon V_{k}=\A{d_k} @>{\psi_{d_k}}>>  \A{d_k+1} 
@>{\psi_{d_k + 1}}>>\cdots@>{\psi_{d_{k+1} - 1}}>>\A{d_{k+1}} = V_{k+1},
\end{CD}
\]
\[
\begin{CD}
\tilde{\iota}_k \colon V_{k}=\A{d_k} @>{\iota_{d_k}}>>\A{d_k+1} @>{\iota_{d_k + 1}}>>\cdots @>{\iota_{d_{k+1} - 1}}>> \A{d_{k+1}}=V_{k+1},
\end{CD}
\]
where $\iota_k \colon \A{k} \hookrightarrow \A{k+1}$ is the natural injection sending $(x_1,\ldots,x_k)$ to $(x_1,\ldots,x_k,0)$. Note that ${\dis \varinjlim_{\psi}\A{k} \simeq \varinjlim_{\tilde \psi}V_k}$ and that ${\dis  \A{\infty}= \varinjlim_{\iota}\A{k} \simeq \varinjlim_{\tilde \iota}V_k}$. Therefore, it is enough to construct a sequence of automorphisms $\alpha_k \colon V_k \simto V_k$ making the following diagram commutative:
\[\tag{$S_k$}
\begin{CD}
V_k @>{{\widetilde \psi}_k}>> V_{k+1} \\
@V{\simeq}V{\alpha_k}V  @V{\simeq}V{\alpha_{k+1}}V\\
V_k @>{{\tilde \iota}_k}>> V_{k+1}
\end{CD}
\]
Assume that $\alpha_0:=\id,\ldots,\alpha_k$ are constructed such that the squares $(S_0),\ldots,(S_{k-1})$ are commutative. Since ${\widetilde \psi}_k$ and ${\tilde \iota}_k \circ \alpha_k$ are two closed immersions of $V_{k}$ into $V_{k+1}$ and since $\dim V_{k+1}=d_{k+1}>2d_{k}+ 1 = 2\dim V_{k}+1$, the existence of $\alpha_{k+1}$ making $(S_{k})$ commutative follows from the above mentioned result of \name{Srinivas}.
\end{proof}

Generalizing Example~\ref{inductive-limit.exa} we get the following.

\begin{lemma}\label{lim-of-indvar.lem}
Let $(\MMM_k)_{k \in \NN}$ be a sequence of ind-varieties together with closed immersions $\iota_k \colon \MMM_k \to \MMM_{k+1}$ for $k\in \NN$. Then, the inductive limit $\MMM:=\varinjlim \MMM_k$ exists in the category of ind-varieties. If all $\MMM_{k}$ are affine, then $\MMM$ is affine.
\end{lemma}

\begin{proof}
We may assume that $\MMM_k$ is a closed subset of $\MMM_{k+1}$ for all $k$, and that the admissible filtrations $\MMM_k =\bigcup_{\ell}\MMM_{k,\ell}$ satisfy $\MMM_{k,\ell}\subseteq \MMM_{k+1,\ell}$ for each $k,\ell$ (Lemma~\ref{good-filtrations-for-morphisms.lem}). Then one easily checks that $\varinjlim \MMM_{k,k}$ is the inductive limit of the $(\MMM_k)_{k \in \NN}$.
\end{proof}

\ps
\subsection{Affine ind-varieties} \label{affine-ind.subsec}\idx{affine ind-variety}
Extending a classical result for varieties, we want to
show that every affine ind-variety $\VVV$ admits a closed immersion into $\AA^{\infty}$.

\begin{theorem}\label{embedding-into-Ainfty.thm}
An ind-variety $\VVV$ is affine if and only if it is isomorphic to a closed subvariety of $\AA^{\infty}$.
\end{theorem}

\begin{proof}
It is clear that a closed subvariety of $\AA^{\infty}$ is affine. Conversely, assume that the ind-variety $\VVV = \bigcup_{k \geq 0} \VVV_k$ is affine, i.e. that each $V_k$ is an affine variety. It is enough to prove that 
there exists a sequence of integers $0\leq n_0 \leq n_1 \leq \cdots $ and a sequence of closed immersions $\psi_k \colon \VVV_k \hookrightarrow \AA^{n_k}$ satisfying the two following two properties where  $\iota_k \colon \AA^{n_k} \to \AA^{n_{k+1}}$ denotes the natural inclusion sending $(x_1, \ldots, x_{n_k})$ to $(x_1, \ldots, x_{n_k}, 0, \ldots, 0)$:
\begin{enumerate}
\item [($P_k$)] 
The following diagram is commutative:
\[
\begin{CD}
\VVV_{k} @>{\psi_{k}}>{\subseteq}> \AA^{n_{k}} \\
@V{\subseteq}V{\iota_{k}}V @V{\subseteq}V{\iota_{k+1}}V \\
\VVV_{k+1} @>{\psi_{k+1}}>{\subseteq}> \AA^{n_{k+1}}
\end{CD}
\]
\item  [($Q_k$)] 
We have $\psi_{k+1} (\VVV_{k+1} ) \cap \AA^{n_k}  = \psi_k (\VVV_k)$.
\end{enumerate}
(Note that condition $Q_{k}$ is needed to make sure that the induced morphism $\psi\colon \VVV \to \AA^{\infty}$ is a closed immersion.)

Assume that the integers $n_k$ as well as the closed immersions $\psi_k \colon V_k \hookrightarrow \AA^{n_k}$  have been constructed for $k \leq \ell$ such that the properties ($P_k$) and ($Q_k$) are satisfied for $k \leq \ell-1$. By Lemma~\ref{extending-closed-immersion.lem} below there exists an integer $n_{l+1} \geq n_l$ and a closed immersion $\psi_{l+1} \colon \VVV_{l+1} \hookrightarrow \AA^{n_{l+1}}$ such that the properties ($P_\ell$) and ($Q_\ell$) are satisfied. By induction it follows that $\psi\colon \VVV \to \AA^{\infty}$ is well-defined and is a closed immersion.
\end{proof}

\begin{lemma} \label{extending-closed-immersion.lem}
Let $X$ be a closed subvariety of an affine variety $Y$ and  $\psi  \colon X \hookrightarrow  \AA^n$ a closed immersion. For any $m$ consider the natural embedding  $\AA^{n} \subseteq \AA^{n+m}$ given by $(x_{1},\ldots,x_{n})\mapsto (x_{1},\ldots,x_{n},0,\ldots,0)$. Then there exists an $m$ such that the composition $\psi\colon X \into \AA^{n}\subseteq \AA^{n+m}$  extends to a closed immersion $\tilde\psi\colon Y \into \AA^{n+m}$ with $\tilde\psi(Y) \cap \AA^{n}= \psi(X)$.
\[
\begin{CD}
X @>\psi>> \AA^{n} \\
@VV{\subseteq}V   @VV{\subseteq}V \\
Y @>\tilde\psi>> \AA^{n+m}
\end{CD}
\]
\end{lemma}

\begin{proof}
Let $f_1, \ldots, f_n \in \OOO (X)$ be the components of $\psi$, i.e. $\psi(x) = ( f_1(x), \ldots, f_n(x) )$ for any $x \in X$. We have a short exact sequence
\[ 
0 \to I(X) \to \OOO (Y) \xrightarrow{\res} \OOO (X) \to 0, 
\]
where $\res \colon \OOO (Y) \to \OOO (X)$ denotes the restriction map $f\mapsto f|_{X}$, and $I(X)$ is the ideal of regular functions on $Y$ vanishing on $X$. For each $i$ let $g_i \in \OOO (Y)$ be a regular function such that $\res g_i = f_i$. Since $\OOO (Y) = \kk [ g_1, \ldots, g_m] + I(X) $, we have $\OOO (Y) = \kk [g_1,\ldots, g_m, I(X) ]$, and so there exists a nonnegative integer $n$ and elements $g_{m+1}, \ldots, g_{m+n} \in I(X)$ satisfying the following two conditions:
\begin{enumerate}
\item The ideal $I(X)$ is generated by $\{g_{m+1}, \ldots, g_{m+n}\}$;
\item $\OOO (Y) =\kk [g_1, \ldots, g_{m+n} ]$.
\end{enumerate}
It follows from (2) that the map $\tilde\psi \colon Y \to \AA^{m+n}$ defined by $y \mapsto (g_1(y), \ldots, g_{m+n}( y) )$ is a closed immersion extending $\psi$. Furthermore,  
$\tilde\psi(y)$ belongs to $\AA^n$ if and only if $g_{m+1}(y) = \cdots = g_{m+n}(y) =0$, i.e. if and only if $y \in X$. The claim follows.
\end{proof}

\ps
\subsection{Connectedness and curve-connectedness}\label{connectedness-and-curve-connectedness.subsec}
We will use the following definition.

\begin{definition}\label{curve-connected.def}
An ind-variety $\VVV$ is called {\it curve-connected}\idx{curve-connected} if for any two points $a,b \in\VVV$ there is an irreducible algebraic curve $D$ and a morphism $D \to \VVV$ whose image contains  $a$ and $b$. Equivalently, there is a closed irreducible algebraic curve $D' \subseteq \VVV$ which contains $a$ and $b$.
\end{definition}
Clearly, a curve-connected ind-variety is connected. It is also known that an irreducible variety is curve-connected (see  \cite[Lemma on page 56]{Mu2008Abelian-varieties}). Here are the two main results of this section:

\begin{proposition}\label{connected.prop}
An ind-variety $\VVV$ is connected if and only if there exists an admissible  filtration $\VVV = \bigcup_{k}\VVV_{k}$ such that all $\VVV_{k}$ are connected.
\end{proposition}\idx{connected}

\begin{proposition}  \label{curve-connected.prop}
An ind-variety $\VVV$ is curve-connected if and only if there exists an admissible filtration $\VVV = \bigcup_k \VVV_{k}$ such that all $\VVV_{k}$ are irreducible.
\end{proposition}

\begin{proof}[Proof of Proposition~\ref{connected.prop}]
One direction is clear. For the other implication, let us start with an admissible filtration $\VVV = \bigcup_k \VVV_k$ of $\VVV$.
Choose any connected component $\VVV_1^0$ of $\VVV_1$. Then there is a unique connected component $\VVV_2^0$ of $\VVV_2$ which contains $\VVV_1^0$, and so on. In this way, we construct connected closed algebraic subsets
$\VVV_1^0 \subseteq \VVV_2^0 \subseteq  \cdots$.
Set $\VVV^{\circ}:=\bigcup_{k}\VVV_{k}^{\circ}$. Since $\VVV_{\ell}^{\circ}$ is a connected component of $\VVV_{\ell}$ it follows that for $k\leq \ell$ the intersection $\VVV_{\ell}^{\circ}\cap \VVV_{k}$ is a union of connected components of $\VVV_{k}$. Hence, $\VVV^{\circ}\cap \VVV_{k}$ is also a union of connected components of $\VVV_{k}$, and is
therefore open and closed in $\VVV_{k}$. Thus $\VVV^{\circ} = \VVV$, because $\VVV$ is connected. 

It remains to see that $\VVV=\bigcup_{k}\VVV_{k}^{\circ}$ is an admissible filtration.
We know that $V_k = \bigcup_{\ell  \geq k}  (V_\ell^0 \cap V_k)$. Since the intersection $V_\ell^0 \cap V_k$ is a finite union of connected components of $V_k$ it is clear that there exists an $\ell$ such that  $V_k = V_\ell^0 \cap V_k$, and the claim follows.
\end{proof}

\begin{proof}[Proof of Proposition~\ref{curve-connected.prop}]
One direction is clear. For the other implication we claim that for any $k$ there is an $\ell$ and an irreducible component $C$ of $\VVV_{\ell}$ such that $\VVV_{k}\subseteq C$. This implies, by induction, that there is an infinite sequence $\ell_{1}<\ell_{2}<\ell_{3}< \cdots$ and irreducible components $C_{i}\subseteq \VVV_{\ell_{i}}$ such that $C_{i}\supseteq \VVV_{\ell_{i-1}}$. It follows that $\bigcup_{i}C_{i} = \VVV$, finishing the proof of the proposition.

For the proof of the claim we can assume  that $\kk$ is uncountable. In fact, choose  an uncountable algebraically closed field $\KK \supseteq \kk$ and replace all $\VVV_{k}$ by $(\VVV_{k})_{\KK}$. If $\tilde C$ is an irreducible component of $(\VVV_{\ell})_{\KK}$ containing $(\VVV_{k})_{\KK}$, then $\tilde C = C_{\KK}$ for an irreducible component $C$ of $\VVV_{\ell}$, and $C$ contains $\VVV_{k}$.
(See Section~\ref{base-field-extension.subsec} for a detailed discussion of base field extensions; here we only use the obvious fact that if $X_1, \ldots, X_r$ are the irreducible components of a $\kk$-variety $X$, then $(X_1)_{\KK}, \ldots, (X_r)_{\KK}$ are the irreducible components of $X_{\KK}$.)

It now suffices to show that for any $k$ and any two irreducible components $Y_{1}$ and $Y_{2}$ of $\VVV_{k}$ there is an $\ell$ and an irreducible component of $\VVV_{\ell}$ containing $Y_{1}\cup Y_{2}$. Fix $b\in Y_{2}$ and choose for any $a\in Y_{1}$ an irreducible curve $D_{a}$ containing $a$ and $b$.
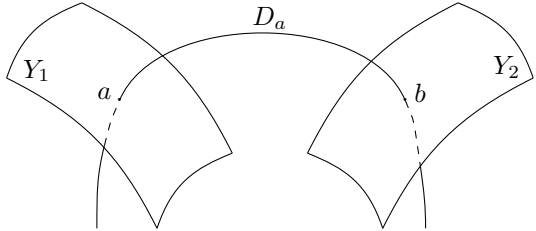
\[
\begin{tikzpicture}[scale=1]
\draw (-2,0) ..controls +(-0.5,1) and +(1, -0.5).. (-4,2);
\draw (-3, -1) ..controls +(-0.5,1) and +(1, -0.5).. (-5,1);
\draw (-3, -1) ..controls +(0.2, 0.6) and +(-0.5, -0.2).. (-2,0);
\draw (-5, 1) ..controls +(0.2, 0.6) and +(-0.5, -0.2).. (-4, 2);
\draw (-1,0) ..controls +(0.5,1) and +(-1, -0.5).. (1,2);
\draw (0, -1) ..controls +(0.5,1) and +(-1, -0.5).. (2,1);
\draw (0, -1) ..controls +(-0.2, 0.6) and +(0.5, -0.2).. (-1,0);
\draw (2, 1) ..controls +(-0.2, 0.6) and +(0.5, -0.2).. (1, 2); 
\draw (-3.5, 0.7) ..controls +(0.5, 1.2) and +(-0.5, 1.2).. (0.3, 0.7);
\draw [dashed] (-3.7, 0.1) ..controls +(0.1, 0.4) and +(-0.125, -0.3).. (-3.5, 0.7);
\draw (-3.8, -1) ..controls +(0, 0.5) and +(-0.1, -0.4).. (-3.7, 0.1);
\draw [dashed] ( 0.3, 0.7) ..controls +(0.125, -0.3) and +(-0.1, 0.7).. (0.5, -0.17);
\draw (0.57, -1) ..controls +(0, 0.3) and +(0.05, -0.35).. (0.5, -0.17); 
\node at (-4.6 , 1.1) {$Y_1$};
\node at (-3.5 , 0.7) {$\cdot$};
\node at (-3.7 , 0.8) {$a$};
\node at (0.3 , 0.7) {$\cdot$};
\node at (0.5 , 0.8) {$b$};
\node at (1.65 , 1.15) {$Y_2$};
\node at (-1.5, 1.8) {$D_a$};
\end{tikzpicture}
\]
This curve is contained in an irreducible component $Z$ of some $\VVV_{\ell}$. For such an irreducible component $Z$ define the subset $S(Z):=\{a\in Y_{1}\mid Z \supseteq D_{a}\}\subseteq Y_{1}$. We have $\bigcup_{Z} S(Z) = Y_{1}$, hence $\bigcup_{Z}\overline{S(Z)} = Y_{1}$.
Because the number of irreducible components of all the $\VVV_{\ell}$, $\ell \in \NN$, is countable, and the base field $\kk$ is uncountable, we can find finitely many $Z_{i}$ such that $\bigcup_{i}\overline{S(Z_{i})} = Y_{1}$ (see Lemma~\ref{constructible.lem}). Since $Y_{1}$ is irreducible, we have $Y_{1}=\overline{S(Z)}$ for some $Z$. Thus $S(Z)$ is dense in $Y_{1}$, hence $Z$ contains $Y_{1}$ and $b$, because $Z \supseteq S(Z)$. Call this irreducible component $Z_{b}$. Then $\bigcup_{b \in Y_{2}}Z_{b}\supseteq Y_{1}\cup Y_{2}$.
Now we repeat the same argument to show that  $Z_{b}\supseteq Y_{1}\cup Y_{2}$ for a suitable $Z_{b}$.
\end{proof}

\begin{remark}\label{irreducible.rem}
There is also the (topological) notion of an \itind{irreducible ind-variety}. It is easy to see that a curve-connected ind-variety is irreducible, but the other implication does not hold as shown by the following example taken from
\cite[Remark~4.3]{BlFu2013Topologies-and-str}. This example also implies that \cite[Proposition~2]{Sh1981On-some-infinite-d} is not correct.
\end{remark}

\begin{example}\label{irred-not-curve-connected.exa}
Consider the closed subvarieties 
$$
X_{k}:=V_{\Atwo}((x-1)\cdots(x-k)(y-1)\cdots (y-k))\subseteq \Atwo
$$ 
consisting of $k$ horizontal and $k$ vertical lines in $\Atwo$. {\it Then $\VVV:=\bigcup_{k} X_{k}$ is an irreducible ind-variety which is not curve-connected.}
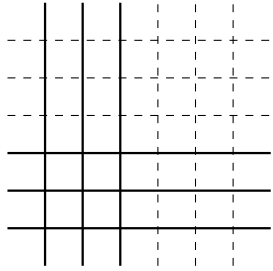
\[
\begin{tikzpicture}[scale = 0.5];
\def \l {0.3mm};
\draw[line width= \l ] (0,-1) -- (0,6);
\draw[line width= \l ] (1,-1) -- (1,6);
\draw[line width= \l ] (2,-1) -- (2,6);
\draw[line width= \l ] (-1,0) -- (6, 0);
\draw[line width= \l ] (-1,1) -- (6,1);
\draw[line width= \l ](-1,2) -- (6,2);
\draw [dashed] (3,-1) -- (3,6);
\draw [dashed] (4,-1) -- (4,6);
\draw [dashed] (5,-1) -- (5,6);
\draw [dashed] (-1,3) -- (6,3);
\draw [dashed] (-1,4) -- (6,4);
\draw [dashed] (-1,5) -- (6,5);
\end{tikzpicture}
\]
\begin{proof}
If a closed subvariety $\WWW \subseteq \VVV$ contains infinitely many lines, then $\WWW = \VVV$. In fact, assume that it contains infinitely many horizontal lines. Then $\WWW$ meets every vertical line in infinitely many points and thus $\WWW$ contains all vertical lines which implies,  with the same argument, that it also contains all horizontal lines. If $\VVV = \VVV' \cup \VVV''$ with two closed subset $\VVV', \VVV'' \subseteq \VVV$, then one of them contains infinitely many lines, hence equals $\VVV$. Thus $\VVV$ is irreducible, and it is clear that $\VVV$ is not curve-connected.
\end{proof}
\end{example}

The next result is a direct consequence of Proposition~\ref{curve-connected.prop}. The Example~\ref{irred-not-curve-connected.exa} above shows that the conclusion does not hold for irreducible ind-varieties.

\begin{corollary} \label{finite-dim-indvar.cor}
If an ind-variety $\VVV$ is curve-connected and satisfies $\dim \VVV < \infty$, then it is an algebraic variety.
\end{corollary}

\begin{remark}
If $X$ is a variety, and $ Y_1 \subseteq Y_2 \subseteq \cdots \subseteq X$ a countable increasing sequence of irreducible closed subsets, then $Y:= \bigcup_{n} Y_n$ 
is a closed subset of $X$. Indeed, there exists an integer $n_0\geq 1$ such that $\dim Y_n = \dim Y_{n_0}$ for $n \geq n_0$. Hence, we get $Y_n = Y_{n_0}$ for $n\geq n_{0}$, and so  $Y= Y_{n_0}$ is closed in $X$. By contrast, we now give an example of a countable increasing sequence   $Y_{1}\subseteq Y_{2}\subseteq \cdots \subseteq \VVV$ of irreducible closed algebraic subsets of an ind-variety $\VVV = \bigcup_{n} \VVV_n$ for which $Y:= \bigcup_{n} Y_n$ is not closed in $\VVV$.
\end{remark}

\begin{example}
Set $\VVV = \AA^{\infty} = \bigcup_{n} \VVV_n$, where $\VVV_n= \An$ for each $n$. Let $Y_n \subset \VVV_n = \An$ be the hypersurface defined by the equation $f_n = 0$, where $f_n \in \kk [x_1, \ldots, x_n]$ is defined inductively by $f_1 := x_1 -1$ and $f_n := x_n + (x_1-n) f_{n-1}$. E.g. $f_2= x_2+ (x_1-1)(x_1-2)$. The inclusion $Y_n \subset Y_{n+1}$ is obvious, and since there exists a polynomial $g_n$ such that $f_n =x_n + g_n (x_1, \ldots, x_{n-1} )$ the hypersurface $Y_n$ is isomorphic to $\AA^{n-1}$. The equality $f_n (x_1, 0, \ldots, 0) = (x_1-1) (x_1-2) \ldots (x_1 - n)$ implies that $Y_n \cap \AA^1 = \{ 1,2, \ldots, n \} \subset \AA^1$. Therefore,  $\bigcup_{n} Y_n \cap \Aone = \{ 1,2,3, \ldots  \}$,
proving that $Y:=\bigcup_{n}Y_{n}$ is not closed in $\VVV = \AA^{\infty}$.

By Proposition \ref{indlimit-of-affine spaces.prop}, we have $\varinjlim Y_k  \simeq \A{\infty}$. Therefore, this example also provides an injective morphism $\phi \colon \A{\infty} \into \Ainfty$ such that the following holds.
\be
\item
For any closed algebraic subset $Z \subset \Ainfty$, $\phi$ induces a closed immersion $Z \into \Ainfty$.
\item 
The image $\phi(\Ainfty) \subset \Ainfty$ is not closed, and so $\phi$ is not a closed immersion.
\ee
\end{example}

\ps
\subsection{Connected and irreducible components}\label{connected-comp.subsec}
For a general topological space $M$ the {\it connected components of $M$}\idx{connected component} are always closed, but not necessarily open. For an algebraic variety $X$, there are finitely many connected components, and they are open and closed. This carries over to ind-varieties in the following form. 

\begin{proposition} \label{connected-components-are-open.prop} 
The connected components of an ind-variety $\VVV$ are open and closed, and the number of connected components is countable.
\end{proposition}
If $M$ is a topological space and $p \in M$ we define $M^{(p)}\subseteq M$ to be the connected component of $M$ containing $p$.
\begin{proof}
Since $\VVV^{(x)}$ is connected and closed, it follows that $\VVV^{(x)} \cap \VVV_k$ is the union of connected components of $\VVV_{k}$, hence open and closed in $\VVV_{k}$,  and so $\VVV^{(x)} \subseteq \VVV$ is open and closed.

For every $k$ we can find finitely many points $\{x_{k,1},x_{k,2},\ldots,x_{k,r_{k}}\}$ representing the connected components of $\VVV_{k}$. It follows that every connected component of $\VVV$ is  of the form $\VVV^{(x_{k,j} )}$ for  some $k,j$. Hence, their number is countable.
\end{proof}

\begin{remark}
It is easy to see, using \name{Zorn}'s Lemma, that every irreducible subset of an ind-variety $\VVV$ is contained in a maximal irreducible subset, and that the maximal ones are closed. (This holds for every topological space.) Similarly, every curve-connected subset is contained in a maximal curve-connected subset, but we do not know if these are closed.

Since an ind-variety is a countable union of irreducible algebraic subvarieties, it is also a countable union of maximal curve-connected subsets.
\end{remark}

\begin{example}\label{non-countable-components.exa} In this example
we construct a closed ind-subvariety $\WWW\subseteq \AA^{\infty}$ with the following properties:
{\it \be
\item
The maximal irreducible subsets of $\WWW$ are closed, curve-connected, and isomorphic to $\AA^{\infty}$.
\item
The union of any number of maximal irreducible subsets is closed;
\item
The set of maximal irreducible subsets of $\WWW$ is not countable;
\item
Every point of  $\WWW$ is contained in uncountably many maximal irreducible subsets.
\ee}
Consider the free monoid $M= \langle a, b \rangle$ in the two letters $a$ and $b$,
\[ 
M := \{ 1 , a, b, a^2, ab, b a, b^2, a^3, a^2b, a b a, \ldots \},
\]
and the $\kk$-vector space $F$ whose (countable) basis is given by the elements of $M$. If $w$ is an element of $M$, its length $| w |$ is defined as the number of its letters. For example, $| 1| = 0$, $|a | = |b | = 1$ and $|a^2| = |a b| = | b a | = 2$. Denote by $\leq$ the {\it prefix partial order} on $M$, i.e.  $ u \leq v$ if and only if there exists an element $u' \in M$ such that $uu' =v$. For $n\in\NN$ set $F_n := \Span \{ w \in M \mid | w | \leq n \}\subseteq F$, and for $w \in M$ define $E_{w} := \Span \{ v \in M \mid v \leq w \}\subseteq F$. Note that  $E_{u} \subseteq E_{v}$ if and only if $u \leq v$. In particular, if $v = uu'$ and $|u| =n$, then $E_{v}\cap F_{n}=E_{u}$.

The inclusions between the $E_{w}$ are demonstrated graphically in the diagram below.
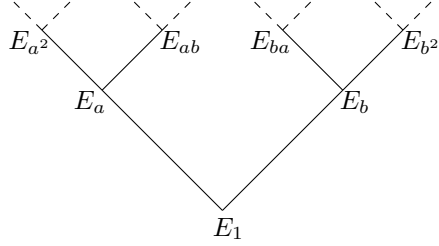
\[
\begin{tikzpicture}[scale=0.8]
\draw (0,0) -- (-2,2);
\draw (0,0) -- (2,2);
\draw (-2,2) -- (-3,3);
\draw (-2,2) -- (-1,3);
\draw [dashed] (-3,3) -- (-3.5,3.5);
\draw [dashed] (-3,3) -- (-2.5,3.5);
\draw [dashed] (-1,3) -- (-1.5,3.5);
\draw [dashed] (-1,3) -- (-0.5,3.5);
\draw (2,2) -- (1,3);
\draw (2,2) -- (3,3);
\draw [dashed] (1,3) -- (0.5,3.5);
\draw [dashed] (1,3) -- (1.5,3.5);
\draw [dashed] (3,3) -- (2.5,3.5);
\draw [dashed] (3,3) -- (3.5,3.5);
\node at (0.1 , -0.3) {$E_1$};
\node at ( -2.2 , 1.8) {$E_a$};
\node at ( -3.2 , 2.8) {$E_{a^2}$};
\node at ( -0.7 , 2.8) {$E_{ab}$};
\node at ( 2.2 , 1.8) {$E_b$};
\node at ( 0.8 , 2.8) {$E_{ba}$};
\node at ( 3.3 , 2.8) {$E_{b^2}$};
\end{tikzpicture}
\]
Define the closed subsets
\[ 
\WWW_n := \bigcup_{ | w | \leq n} E_w \subseteq F_{n}
\]
as the union of the finite-dimensional subspaces $E_{w}\subseteq F$, $|w|\leq n$. For example, $\WWW_0 = E_1 = F_0=  \Span (1)$,
\[ \WWW_1 = E_a \cup E_b =\Span (1,a) \cup \Span (1,b) \subseteq F_{1}=\Span(1,a,b), \quad \text{and}\]
\[
\begin{array}{ll}
\WWW_2  &= E_{a^2} \cup E_{ab} \cup E_{ba} \cup E_{b^2} \\
 & = \Span (1, a, a^2) \cup \Span (1,a, ab) \cup \Span (1,b,ba) \cup \Span (1,b, b^2).
 \end{array}
\]
By construction,  $\WWW_{n}\subseteq \WWW_{n+1}$, and $\WWW_{n+1}\cap F_{n}=\WWW_{n}$, as we have seen above. This implies that $\WWW:=\bigcup_{n} \WWW_{n} \subseteq F \simeq \AA^{\infty}$ is closed, because $\WWW \cap F_{n}=\WWW_{n}$.

\begin{proof}[Proof of the statements {\rm (1)-(4)}]
For each infinite word $x =x_1x_2 \ldots x_n \ldots$  in the letters $a$ and $b$, the following subset of $\WWW$
\[  
\WWW_{(x)} := \bigcup_{n \geq 0} E_{(x_1 x_2 \cdots x_n)} = \Span (1, \, x_1, \, x_1x_2, \ldots, x_1 \dots x_n, \ldots ) \subseteq \WWW
\]
is closed, curve-connected, and isomorphic to $\AA^{\infty}$. In fact, one has $\WWW_{(x)}\cap\WWW_{n} = E_{x_{1}x_{2}\cdots x_{n}}$, and  since the  $E_{x_{1}x_{2}\cdots x_{n}}$ are $\kk$-vector spaces, the rest is clear. We want to show that the $\WWW_{(x)}$ are the maximal irreducible subsets of $\WWW$ which implies (1), (3) and (4).

We first remark that if $I$ is any set of infinite words $x$, then the same argument as above shows that $\bigcup_{x\in I} \WWW_{(x)} \subseteq \WWW$ is closed, proving (2).

Let $C \subseteq \WWW$ be an irreducible subset. If $w \in M$ we set $\WWW_{(\geq w)}:=\bigcup_{x=wy} \WWW_{(x)}$. Clearly, $\WWW_{(\geq w)}= \WWW_{(\geq wa)}\cup\WWW_{(\geq wb)}$, and all these subsets are closed, by (2). If $C$ is not contained in a $\WWW_{(x)}$, then, by induction, there is a word $w$ such that $C \subseteq \WWW_{(\geq w)}$, but $C \nsubseteq \WWW_{(\geq wa)}$ and $C \nsubseteq \WWW_{(\geq wb)}$, and we have a contradiction.
\end{proof}
\end{example}

\ps
\subsection{Morphisms with small fibers}   \label{small-fibers.subsec}
A well-known result in the category  of varieties is that a bijective morphism  $\phi\colon X \to Y$ is an isomorphism if $X$ is irreducible and $Y$ is normal. This is a special case of the original form of \name{Zariski}'s  Main Theorem, see \cite[Chap.~III, \S9, p. 209]{Mu1999The-red-book-of-va}. (See also Lemma~\ref{Igusa.lem}.)

We can prove a similar statement for a bijective morphism $\phi\colon \VVV \to \WWW$ where $\VVV$ is connected, but only under a strong normality assumption for $\WWW$, namely that there exists an admissible
filtration consisting of normal varieties, see Proposition~\ref{bijective-morphism.prop} below. 

If the base field $\kk$ is countable, we have some very strange examples, e.g. take $\VVV$ to be $\kk$ considered as a discrete countable set, and take $\phi\colon \VVV \to \Aone$ to be the identity map.
Then, the inverse image of the algebraic set $\Aone$ by the bijective morphism $\phi$ is not algebraic!
However, for an uncountable base field $\kk$ such a behavior cannot occur as the following lemma shows.

\begin{lemma}  \label{bijective-morphisms.lem} Assume that $\kk$ is uncountable, and 
let $\phi\colon \VVV \to \WWW$ be a bijective ind-morphism. 
\be
\item \label{inverse-image-by-a-bijection-of-an-algebraic-set-is-algebraic}
For every algebraic subset $X \subseteq \WWW$ the inverse image $\phi^{-1}(X)$ is algebraic.
\item
If $\WWW = \bigcup_{k}\WWW_{k}$ is an admissible filtration of $\WWW$,
then $\VVV = \bigcup_{k}\phi^{-1}(\WWW_{k})$ is an admissible filtration of $\VVV$.
\ee
\end{lemma}

\begin{proof}
(1) Since $\phi$ is surjective, we have
$X = \bigcup_{k}\phi(\VVV_{k})\cap X$, and the subsets $\phi(\VVV_{k})\cap X\subseteq X$ are all constructible. By Lemma~\ref{constructible.lem} we have $X \subseteq \phi(\VVV_{k})$ for some $k$, hence $\phi^{-1}(X) \subseteq \VVV_{k}$,  and the claim follows.
\ps
(2)
By (1) the subsets $\phi^{-1}(\WWW_{k})$ are closed algebraic subsets, and the result follows from Theorem~\ref{closed-algebraic-filtration.thm}.
\end{proof}

Let $\VVV$ be an ind-variety, $X$ a variety, and let $\phi\colon \VVV \to X$ be a morphism. Even if the fibers are ``small''  this does not imply that $\VVV$ is a variety.
If the base field $\kk$ is countable, such an example was already given just before Lemma~\ref{bijective-morphisms.lem}. 
An example not assuming that $\kk$ is countable is the embedding of $\ZZ$ into $\Aone$. This is an injective morphism and thus has finite fibers, but $\ZZ$ is not a variety.
Even if $\VVV$ is connected we cannot conclude that $\VVV$ is a variety, as the following example shows.

\begin{example}
Let $\VVV = \bigcup_{k}\VVV_{k}$ where 
$$
\VVV_{k}:=V_{\Atwo}(y(x-1)(x-2)\cdots(x-k))\subseteq \Atwo,
$$ 
with the obvious immersions $\VVV_{k}\subseteq\VVV_{k+1}$. 
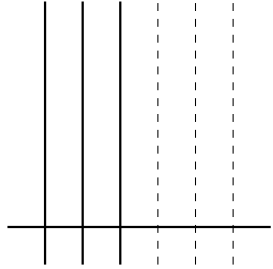
\[
\begin{tikzpicture}[scale = 0.5];
\def \l {0.3mm};
\draw[line width= \l ] (0,-1) -- (0,6);
\draw[line width= \l ] (1,-1) -- (1,6);
\draw[line width= \l ] (2,-1) -- (2,6);
\draw[line width= \l ] (-1,0) -- (6, 0);
\draw [dashed] (3,-1) -- (3,6);
\draw [dashed] (4,-1) -- (4,6);
\draw [dashed] (5,-1) -- (5,6);
\end{tikzpicture}
\]
Then the projection $\phi_{k}\colon\VVV_{k}\to \Aone$ onto the $x$-axis defines a morphism $\phi\colon \VVV \to \Aone$ whose fibers are either points or lines. However, $\VVV$ is not a variety.
\end{example}
Note that in this example $\VVV$ is connected, but not curve-connected. In fact, with this stronger assumption we get what we want.

\begin{proposition} \label{small-fibers-gives-variety.prop}
Let $\phi\colon \VVV \to X$ be a morphism where $X$ is a variety. 
Assume that there is an integer  $d \in \NN$ such that the fibers are ind-varieties  of dimension $\leq d$. Then $\VVV$ is finite-dimensional of dimension $\leq \dim X + d$. If $\VVV$ is curve-connected, then $\VVV$ is an algebraic variety.
\end{proposition}

\begin{proof}
(a) 
Assume that $\VVV$ is curve connected, and
choose an admissible filtration $\VVV=\bigcup_{k}\VVV_{k}$ where all $\VVV_{k}$ are irreducible (see Proposition~\ref{curve-connected.prop}). Since the fibers of $\VVV_{k}\to \overline{\phi(\VVV_{k})}$ have dimension $\leq d$, we see that $\dim\VVV_{k}\leq \dim  \overline{\phi(\VVV_{k})} + d\leq \dim X +d$. Thus, there is a $k_{0}$ such that $\VVV_{k}=\VVV_{k_{0}}$ for $k\geq k_{0}$, hence $\VVV = \VVV_{k_{0}}$ is an algebraic variety.
\ps
(b)
In general, $\VVV$ is a countable union of irreducible algebraic varieties, $\VVV = \bigcup_{k} Z_{k}$. By (a), $\dim Z_{k}\leq \dim X + d$ for all $k$. Hence $\dim \VVV \leq \dim X + d$.
\end{proof}

\begin{remark} 
The proof above shows that the assumptions for  $\VVV$ can be weakened in order to get that $\VVV$ is an algebraic variety. It suffices to assume that there is an admissible filtration  $\VVV=\bigcup_{k}\VVV_{k}$ and an integer $m$ such that the number of irreducible components of each $\VVV_{k}$ is $\leq m$.
\end{remark}

We finally prove the statement about bijective morphisms announced at the beginning.

\begin{proposition}\label{bijective-morphism.prop}
Assume that $\kk$ is uncountable. Let $\phi\colon \VVV \to \WWW$ be a bijective ind-morphism. Assume that $\VVV$ is connected and that
there exists an admissible filtration $\WWW=\bigcup_{k}\WWW_{k}$
such that all $\WWW_{k}$ are irreducible and normal. Then $\phi$ is an isomorphism.
\end{proposition}

\begin{proof}
For every $k$, set $\VVV_{k}:=\phi^{-1}(\WWW_{k})$.
By Lemma~\ref{bijective-morphisms.lem}(2) $\VVV=\bigcup_{k}\VVV_{k}$ is an admissible filtration.
The induced map $\phi_{k}\colon \VVV_{k} \to \WWW_{k}$ is bijective. By the following lemma, there is a uniquely defined connected  component $\VVV_{k}^{\circ}$ which is normal and such that $\phi$ induces an open immersion $\VVV_{k}^{\circ} \into \WWW_{k}$. Set $\VVV^{\circ}:=\bigcup_{k}\VVV_{k}^{\circ}$. Since $\VVV_{\ell}^{\circ}$ is a connected component of $\VVV_{\ell}$ it follows that for $k\leq \ell$ the intersection $\VVV_{\ell}^{\circ}\cap \VVV_{k}$ is a union of connected components of $\VVV_{k}$. Hence $\VVV^{\circ}\cap \VVV_{k}$ is also a union of connected components of $\VVV_{k}$, hence open and closed. Thus $\VVV^{\circ} = \VVV$, because $\VVV$ is connected.

Now we claim that for every $k$ there is an $\ell > k$ such that $\VVV_{\ell}^{\circ}\supseteq \VVV_{k}^{\circ}$. In fact, if this is not the case, define 
$\VVV':= \bigcup_{\ell>k}\VVV_{\ell}^{\circ}$. As above, $\VVV'$ is open and closed in $\VVV$, hence $\VVV' = \VVV$, contradicting the assumption. It follows that there is a sequence $k_{1}<k_{2}<\cdots$ such that $\VVV^{\circ}_{k_{j}} \subseteq \VVV^{\circ}_{k_{j+1}}$ for all $j\geq 1$, hence $\VVV=\bigcup_{j}\VVV^{\circ}_{k_{j}}$ is an admissible filtration (Theorem~\ref{closed-algebraic-filtration.thm}). Now the claim follows from Lemma~\ref{crit-iso.lem} below.
\end{proof}

\begin{lemma} Let $\phi\colon X \to Y$ be a bijective morphism of varieties where $Y$ is irreducible and normal. Then there is a connected component $X^{\circ}\subseteq X$ such that the induced map $\phi|_{X^{\circ}}\colon X^{\circ}\into Y$ is an open immersion.
\end{lemma}

\begin{proof}
There is a well-defined irreducible component $X_{0}$ of $X$ such that $\phi(X_{0})$ is dense in $Y$. Then \name{Zariski}'s Main Theorem in 
\name{Grothendieck}'s form (see \cite[Chap.~III, \S9, statement IV, p. 209]{Mu1999The-red-book-of-va})
implies that there exists an open immersion $X_{0}\into \tilde X$ where $\tilde X$ is an irreducible variety,
and a finite morphism $\tilde\phi\colon \tilde X \to Y$ extending $\phi|_{X_{0}}$. Since $\tilde\phi$ is birational and $Y$ is normal it follows that $\tilde\phi$ is an isomorphism, and so $\phi|_{X_{0}}\colon X_{0} \into Y$ is an open immersion. In particular, $X_{0}=\phi^{-1}(\phi(X_{0})) \subseteq X$ is open, hence a connected component of $X$.
\end{proof}

\begin{lemma}\label{crit-iso.lem} Assume that $\kk$ is uncountable.
Let $\phi\colon \VVV =\bigcup_{k}\VVV_{k} \to \WWW=\bigcup_{k}\WWW_{k}$ be a bijective ind-morphism where $\VVV$ is connected. If $\phi(\VVV_{k}) \subseteq \WWW_{k}$, and if the induced maps $\phi|_{\VVV_{k}}\colon \VVV_{k} \to \WWW_{k}$  are open immersions for all $k$, then $\phi$ is an isomorphism.
\end{lemma}

\begin{proof}
It follows from Lemma~\ref{bijective-morphisms.lem} that the closed subsets $\phi^{-1}(\WWW_{k})$, $k \geq 1$, form an admissible filtration of $\VVV$. Therefore, for each $k$, there is a $\ell$ such that 
$\phi^{-1}(\WWW_{k}) \subseteq \VVV_{\ell}$. 
From the inclusions $\phi(\VVV_{k})\subseteq \WWW_{k}\subseteq \phi(\VVV_{\ell})$ we see that $\WWW_{k}':= \phi(\VVV_{k})$ is closed in $\WWW_{k}$, because $\phi(\VVV_{k})$ is closed in $\phi(\VVV_{\ell})$, and that $\WWW_{k}\subseteq \WWW_{\ell}'$. This shows that  $\WWW = \bigcup_{k}\WWW_{k}'$ is an admissible filtration. Now it is clear that $\phi$ is an isomorphism.
\end{proof}

\ps
\subsection{Tangent spaces and smoothness}\label{tangent-smoothness.subsec}
\label{tangent-space.subsec}
For any ind-variety $\VVV = \bigcup_{k\in\NN}\VVV_{k}$ we define the \itind{tangent space} in $v\in \VVV$ in the obvious way. We have $v \in \VVV_{k}$ for $k \geq k_{0}$, and $T_{v}\VVV_{k}\subseteq T_{v}\VVV_{k+1}$ for $k\geq k_{0}$, and then define\idx{$\Tr$@$T_{v}\VVV$}
$$
T_{v}\VVV := \varinjlim_{k\geq k_{0}} T_{v}\VVV_{k}
$$
which is a $\kk$-vector space of countable dimension.
If $\VVV=V$ is a $\kk$-vector space of countable dimension (Example~\ref{countable-vector-space.exa}), then, for any $v\in V$, we have $T_{v}V = V$ in a canonical way. 

If $\VVV$ is affine, then every $\delta \in T_{v}\VVV$ defines a {\it continuous derivation} $\delta\colon \OOO(\VVV) \to \kk$
in $v$, i.e. a continuous linear map $\delta$ satisfying $\delta(fg)=f(v)\delta(g) + g(v)\delta(f)$ for all $f,g \in \OOO(\VVV)$. In fact, we have $\delta \in T_{v}\VVV_{k}$ for some $k\geq 1$, and then we set $\delta f := \delta (f|_{\VVV_{k}})$ for $f \in \OOO(\VVV)$. The following lemma is clear.\idx{continuous derivation}

\begin{lemma}\label{tangent-derivation.lem} If $\VVV$ is affine
there is a canonical isomorphism of the tangent space $T_{v}\VVV$ with the continuous derivations 
$\Der_{v}^{\text{\tiny\it cont}}(\OOO(V),\kk)$\idx{$\Der_{v}^{\text{\tiny\it cont}}(\OOO(V),\kk)$} in $v$.
\end{lemma}\idx{derivation, continuous}

It is also clear that for an open ind-subvariety $\WWW \subseteq \VVV$ we have $T_{w}\WWW = T_{w}\VVV$ for all $w\in\WWW$.
Moreover, any morphism $\phi\colon \VVV \to \WWW$ between ind-varieties induces a linear map $d\phi_{v}\colon T_{v}\VVV \to T_{\phi(v)}\WWW$ for every $v \in \VVV$, called the {\it differential of $\phi$ in $v$}\idx{differential}.

\begin{remark}
An ind-variety $\VVV$ is discrete if and only if all tangent spaces are trivial.\idx{discrete}
On the other hand, $\dim\VVV <\infty$ does not imply that the tangent spaces are finite-dimensional. In fact,  the union $\UUU$ of all the coordinate lines in $\A{\infty}$ is a closed ind-subvariety of dimension 1, but we have $\dim T_{0}\UUU = \infty$. A more interesting example will be given below in Example~\ref{closed-subvar-same-tangent-space.exa}.
\end{remark}

\begin{example}\label{trivial-differential.exa}
Let $\phi\colon \VVV \to \WWW$ be a surjective ind-morphism such that $d\phi_{v}$ is trivial in every point $v \in \VVV$. Then $\WWW$ is countable, hence discrete in case  $\kk$ is uncountable.  If $\VVV$ is connected, then $\WWW$ is a point.
\newline
(The assumption implies that every connected closed algebraic subset $X \subseteq \VVV$ is mapped to a point. Thus the first statement is clear, because $\VVV$ is a countable union of connected closed
algebraic subsets, namely of the connected components of the $\VVV_{k}$. For the second claim, we use the fact that 
there exists an admissible filtration $\VVV = \bigcup_k \VVV_k$ such that all $\VVV_{k}$ are connected, see Proposition~\ref{connected.prop}.)
\end{example}

\begin{definition}\label{smooth.def}
Let $\VVV$ be an ind-variety, and let $x\in\VVV$.
\be
\item
$\VVV$ is \itind{strongly smooth} in $x$ if there is an open neighborhood of $x$
which has an admissible filtration  consisting of connected smooth varieties.
\item
$\VVV$ is \itind{geometrically smooth} in $x$ if there is an admissible filtration $\bigcup_{k}\VVV_{k}$ such that $x\in\VVV_{k}$ is a smooth point for all $k$.
\ee
\end{definition}

There is a third concept of smoothness, called \itind{algebraic smoothness} which we will 
not discuss here. We refer to \cite[Section~4.3]{Ku2002Kac-Moody-groups-t} for a detailed investigation. It follows from the definition that a geometrically smooth point is also algebraically smooth.

\begin{example} \label{countable-union-of-varieties.exa}
We consider again the ind-variety $\XXX:=\bigcup_{i\in\NN} X_{i}$ from Example~\ref{countable-family-of-closed-algebraic-subsets.exa} where the $X_{i}$ are algebraic varieties.
For $x  \in X_{i} \subset \XXX$,
it follows immediately from the definitions that the next three points are equivalent:
\be
\item[(i)] $X_{i}$ is smooth in $x$;
\item[(ii)] $\XXX$ is strongly smooth in $x$;
\item[(iii)] $\XXX$ is geometrically smooth in $x$.
\ee
If all the $X_{i}$ are closed subvarieties of some variety $X$, then we obtain a natural ind-morphism $\phi\colon \XXX \to X$ in the obvious way.
If the $X_{i}\subseteq X$ are disjoint, then $\phi$ is injective, and the differential $d\phi_{x}\colon T_{x}\XXX \to T_{\phi(x)}X$ is injective for every $x\in\XXX$. Note that $\phi$ cannot be a closed immersion, since the image is not closed in case $\kk$ is uncountable.
\end{example}

Continuing our discussion about bijective morphisms in Section~\ref{small-fibers.subsec} we have the following result.

\begin{proposition}\label{connected-and-smooth-gives-iso.prop}
Let $\phi\colon \VVV \to \WWW$ be a bijective ind-morphism. Assume that $\VVV$ is curve-connected and $\WWW$ is strongly smooth in a  point $w \in \WWW$. Then there is an open neighborhood $\VVV'$ of $\phi^{-1}(w)\in\VVV$ such that $\phi|_{\VVV'}\colon \VVV' \into \WWW$ is an open immersion. If $\WWW$ is strongly smooth in every point, then $\phi$ is an isomorphism.
\end{proposition}

\begin{proof} Replacing $\WWW$ by an open neighborhood of $w$ and $\VVV$ by its inverse image, we
can assume that there exists an admissible filtration of $\WWW$ 
consisting of connected smooth varieties and that $\phi$ is bijective. Then, by Proposition~\ref{bijective-morphism.prop}, $\phi$ is an isomorphism.
\end{proof}

Next we give an example of a curve-connected and strongly smooth closed subvariety of $\A{\infty}$ with the same tangent space in $0$, but which is strictly included in  $\A{\infty}$.

\begin{example}\label{closed-subvar-same-tangent-space.exa}
Define inductively a sequence of polynomials $f_k \in \OOO(\A{k} )$ by $f_1:=x_1$ and $f_k:=x_k +(f_{k-1})^2$ for $k \geq 2$. For each $k$, let $\VVV_k$ be the (smooth) hypersurface of $\A{k}$ which is the zero set of $f_k$. Set $\VVV:= \bigcup_{k \geq 1} \VVV_k$. By construction, $\VVV\cap \A{k} = \VVV_{k}$ for all $k$, hence $\VVV \subsetneqq \Ainfty$ is a strict closed subset.

Note that $T_0\VVV_k =\A{k-1} \subseteq \A{k} =T_0\A{k}$, where $\A{k-1} \subseteq \A{k}$ is the hyperplane of equation $x_k=0$. Then, it is clear that $T_0\VVV= \bigcup_{k \geq 1} T_0\VVV_k = \Ainfty = T_0 \Ainfty$. Moreover, each $\VVV_{k}$ is isomorphic to $\A{k-1}$ and so $\VVV \simeq \Ainfty$, by Proposition~\ref{indlimit-of-affine spaces.prop} above. Thus we have constructed a strict closed immersion $\phi\colon\Ainfty \into \Ainfty$ such that $d\phi_{0}\colon T_{0}(\Ainfty) \simto T_{0}\Ainfty$ is an isomorphism.
\end{example}

\begin{question} \label{towards-a-criterion-for-a-bijective-morphism-to-be-an-isomorphism.ques}
Is it true that a bijective morphism $\phi\colon \VVV \to \WWW$ between ind-varieties is an isomorphism if the differential $d\phi_{v}$ is an isomorphism in every point $v\in\VVV$? Maybe one has to assume in addition that $\VVV$ is connected or even curve-connected.
\end{question}

\ps
\subsection{\texorpdfstring{$R$}{R}-rational points}\label{R-rational-points.subsec}
If $X$ is a $\kk$-variety and $R$ any $\kk$-algebra, then the {\it $R$-rational points of $X$\/}\idx{rational point} are defined by\idx{R-r@$R$-rational points}\idx{$\XX$@$X(R)$}
$$
X(R):=\Mor(\Spec R, X).
$$
If $X$ is affine, then $X(R) = \Alg_{\kk}(\OOO(X), R)$, the $\kk$-algebra homomorphisms from $\OOO(X)$ to $R$.  If $X$ is an algebraic semigroup or an algebraic group, then  $X(R)$ is a semigroup, resp. a group. It is clear that $X \mapsto X(R)$
is a functor, i.e. for any morphism $\phi\colon X \to Y$ we get a map $\phi(R)\colon X(R) \to Y(R)$ with the usual functorial properties.

If $S \subseteq R$ is a  $\kk$-subalgebra, then we have a natural inclusion $X(S) \subseteq X(R)$.
In particular, we have an inclusion $X = X(\kk) \subseteq X(R)$ for every $\kk$-algebra  $R$.

If $\KK/\kk$\idx{$\LLL$@$\KK/\kk$} is a field extension where $\KK$ is again algebraically closed, then $X(\KK)$ are the $\KK$-rational 
points\idx{K-r@$\KK$-rational points} of the variety $X_{\KK}:=\Spec \KK\times_{\Spec \kk} X$. We will often confuse $X(\KK)$ with $X_{\KK}$, and $\phi (\KK) \colon X( \KK) \to Y (\KK)$ with $\phi_{\KK} \colon X_{\KK} \to Y_{\KK}$ in case  $\phi \colon X \to Y$ is a morphism.

As an example, let $V$ be a $\kk$-vector space of countable dimension. Then $V_{\KK}=V(\KK) = \KK \otimes_{\kk}V$ in a canonical way.

It is clear that these constructions carry over to ind-varieties. If $\VVV = \bigcup_{k}\VVV_{k}$, then\idx{$\VVV(R)$}
$$
\VVV(R) := \bigcup_{k} \VVV_{k}(R)
$$
where we use the fact that for a closed subvariety $\VVV_{k} \subseteq \VVV_{k+1}$ we also get an inclusion $\VVV_{k}(R) \into \VVV_{k+1}(R)$ of the $R$-rational points. In case $\KK/\kk$ is a field extension where $\KK$ is algebraically closed we get a $\KK$-ind-variety $\VVV_\KK = \bigcup_{k}(\VVV_{k})_{\KK}$. Again, we identify $\VVV_{\KK}$\idx{$\VVV_{\KK}$} with its $\KK$-rational points: $\VVV_{\KK} = \VVV(\KK) = \bigcup_{k}\VVV_{k}(\KK)$.

\begin{proposition}  \label{X(R).prop}
Let $X$ be an affine $\kk$-variety and $R$ a (commutative) $\kk$-algebra of countable dimension. Then $X(R)$ has a natural structure of an affine ind-variety such that the following holds.
\be
\item
If $\phi\colon X \to Y$ is a morphism of affine $\kk$-varieties, then the induced map $\phi(R)\colon X(R) \to Y(R)$ is an ind-morphism. If $\phi$ is a closed immersion, then so is $\phi(R)$.
\item 
For every homomorphism $\rho\colon R \to S$ of $\kk$-algebras of countable dimension the induced map $X(\rho)\colon X(R) \to X(S)$ is an ind-morphism. If $\rho$ is injective, then $X(\rho)$ is a closed immersion.
\ee
\end{proposition}

\begin{proof}
(a) Let $X \subseteq \An$ be a closed subset and denote by $I(X) \subset \OOO(\An)=\kk[x_{1},\ldots,x_{n}]$ the ideal of $X$.
Then
\[
X(R) = \{(a_{1},\ldots,a_{n})\in R^{n}\mid f(a_{1},\ldots,a_{n})=0 \text{ for all }f\in I(X)\}\subseteq R^{n}.
\]
This is a closed subset of $R^{n}$ (cf. Example~\ref{subst-R-algebra.exa}),  and we thus obtain the structure of an affine ind-variety on $X(R)$ for every closed subset $X \subset \An$. 
\ps
(b)
Let $X \subseteq \An$ and $Y \subseteq \Am$ be closed subsets, and let $\phi\colon X \to Y$ be a morphism. Then there exists a morphism $\Phi\colon \An \to \Am$ which induces $\phi$, and we get the following commutative diagram.
$$
\begin{CD}
X(R) @>{\subseteq}>>  R^{n} \\
@VV{\phi(R)}V @VV{\Phi(R)}V \\
Y(R) @>{\subseteq}>> R^{m}
\end{CD}
$$
Since $\Phi(R)$ is given by polynomials with coefficients in $\kk$ it is an ind-morphism, and the same is true for $\phi(R)\colon X(R) \to Y(R)$.
\ps
(c)
Let $X$ be an affine variety and let  $\eta\colon X \into \An$, $\mu\colon X \into \Am$ be two closed immersions. Then we obtain an isomorphism $\phi\colon \eta(X) \simto \mu(X)$, $\phi(a):=\mu(\eta^{-1}(a))$. It follows from (b)  that $\phi(R)\colon \eta(X)(R) \simto \mu(X)(R)$ is an isomorphism of ind-varieties. Therefore, the ind-structure on the set $X(R)$ defined by a closed immersion $X \into \An$ does not depend on the immersion. Moreover, (b) also implies that for a morphism $\phi\colon X \to Y$ of affine varieties the induced map $\phi(R)\colon X(R) \to Y(R)$ is an ind-morphism, proving the first part of (1).
\ps
(d)
It follows from (b) and (c) that for a closed immersion $\eta\colon X \into \An$ the induced map $\phi(R) \colon X(R) \into R^{n}$ is a closed immersion. This implies the second part of (1), by composing a closed immersion $\phi\colon X \into Y$ with a closed immersion $Y \into \An$.
\ps
(e)
If $\rho\colon R \to S$ is a homomorphism of $\kk$-algebras, then the induced map $\rho^{n}\colon R^{n}\to S^{n}$ is $\kk$-linear, hence an ind-morphism. If $\rho$ is injective, then $\rho^{n}$ is a closed immersion. Applying this to a closed subset $X \subset \An$ we obtain (2).
\end{proof}

\begin{remark}
Another interesting case is the following. Let $X$ be an affine variety, $f\in\OOO(X)$ a nonzero regular function, and consider the principal open set $X_{f} \subset X$. We will show in Section~\ref{principal-open-sets.subsec} that the canonical morphism $X_{f}(R) \to X(R)$ is a locally closed immersion for any finitely generated $\kk$-algebra $R$.
\end{remark}

\ps
\subsection{Base field extension} \label{base-field-extension.subsec}
We will mainly use the construction above for field extensions\idx{base field extension $\KK/\kk$}  where $\KK$ is also algebraically closed. It will be applied to reduce some proofs to the case of an uncountable base field. (We have already used this method  in the proof of Proposition~\ref{curve-connected.prop}.)
\idx{$\LLL$@$\KK/\kk$}

In the following two lemmas we first recall and prove some basic properties in the case of varieties. After this, in Proposition~\ref{fieldextension.prop}, we will study the situation for ind-varieties. 

The group $\Gamma:=\Aut(\KK/\kk)$ of field automorphisms of $\KK$ fixing $\kk$ acts on $\KK$ and hence on $X_{\KK}=X(\KK)$ for every $\kk$-variety $X$. Since $\KK^{\Gamma}=\kk$, i.e. $\KK/\kk$ is a \itind{Galois-extension},
we get $X  = (X_{\KK})^{\Gamma}$. Moreover, if $\phi\colon X \to Y$ is a $\kk$-morphism, then $\phi_{\KK}\colon X_{\KK}\to Y_{\KK}$ is $\Gamma$-equivariant.

\begin{lemma}
\be
\item
Let $R$ be a finitely generated $\kk$-algebra, and let $I \subseteq R_{\KK}:=\KK\otimes R$ be an ideal. If $I$ is $\Gamma$-stable, the $I$ is defined over $\kk$, i.e. $I = \KK\otimes I^{\Gamma}$ and $I^{\Gamma}= I \cap R$.
\item 
Let $X$ be a variety and  $U \subseteq X_{\KK}$ a locally closed $\KK$-subvariety. If $U$ is stable under $\Gamma$, then $U$ is defined over $\kk$. More precisely, $U^{\Gamma}=U\cap X \subseteq X$ is a locally closed $\kk$-subvariety and $U = (U\cap X)_{\KK}$.
\ee
\end{lemma}\label{Gamma-stable.lem}

\begin{proof}
(1) We prove the following more general result. {\it If $V \subset \KK^{\infty}$ is a $\Gamma$-stable linear subspace, then $V = \KK V^{\Gamma}$.} Since $V$ is a union of $\Gamma$-stable finite-dimensional subspaces we can assume that $V$ is a finite-dimensional subspace of some $\KK^{m}$. Let $(v_{1},\ldots,v_{k})$ be a basis of $V$ in row echelon form where the first nonzero entry in each $v_{i}$ is 1. For every $\gamma \in \Gamma$ the element $\gamma v_{i}$ belongs to $V$  and thus can be expressed as a linear combination of the $v_{j}$'s. By the properties of the row echelon form we see that the basis vector $v_{j}$ for  $j\neq i$ cannot appear in this linear combination, and that the coefficient of $v_{i}$ must be 1. This shows that the basis is fixed by $\Gamma$, i.e. $v_{1},\ldots v_{k}\in (\KK^{m})^{\Gamma} = \kk^{m}$, and so $V = \KK V^{\Gamma}$.
\ps
(2) It suffices to consider the case of an affine variety $X$. It is also clear that it is enough to prove (2) for $\Gamma$-stable subsets of $X_{\KK}$ which are either closed or open. 

If $Z \subseteq X_{\KK}$ is closed and $\Gamma$-stable, then the ideal $I:=I (Z) \subseteq \OOO(X_{\KK}) = \KK\otimes \OOO(X)$ is $\Gamma$-stable. It follows from (1) that $I = \KK\otimes I^{\Gamma}$. This implies that $Z = (Z')_{\KK}$ where $Z' \subseteq X$ is the zero set of $I^{\Gamma}$. Since $I^{\Gamma}$ generates the ideal $I$ we get $Z' = Z \cap X = Z^{\Gamma}$. This proves the claim for a closed subvariety $Z \subseteq X_{\KK}$. 

If $U \subseteq X_{\KK}$ is open and $\Gamma$-stable, then $Z:=X_{\KK}\setminus U$ is closed and $\Gamma$-stable, hence $Z = (Z')_{\KK}$ where $Z':=Z\cap X$, as we have just seen. Setting $U':=X\setminus Z'$ we get $X_{\KK} = (U')_{\KK}\cup(Z')_{\KK}$  and $(U')_{\KK}\cap(Z')_{\KK} = \emptyset$. In fact, the second statement is clear, because the intersection is a $\Gamma$-stable closed subset of $(U')_{\KK}$, hence defined over $\kk$. For the first we remark that $\tilde Z:=X_{\KK}\setminus (U')_{\KK} \subseteq X_{\KK}$ is closed and $\Gamma$-stable, hence $\tilde Z = (\tilde Z \cap X)_{\KK}$.
Since $(U')_{\KK} \cap X = U'$ we get $\tilde Z \cap X = Z' = Z \cap X$, hence $\tilde Z = (\tilde Z \cap X)_{\KK} = (Z \cap X)_{\KK} =Z$, and the claim follows.
As a consequence, $U$ is defined over $\kk$ and $U = (U\cap X)_{\KK}$, showing that (2) also holds for open subsets $U \subset X_{\KK}$. 
\end{proof}

\begin{lemma}  \label{fieldextension.lem}
Let $X,Y$ be $\kk$-varieties.
\be
\item 
For $f\in\OOO(X_{\KK})$ we have that $f\in\OOO(X)$ if and only if $f(X)\subseteq \kk$.
\item 
If $\psi\colon X_{\KK}\to Y_{\KK}$ is a $\KK$-morphism such that $\psi(X) \subseteq Y$, then
the induced map $\phi \colon X \to Y$, $x \mapsto \psi(x)$, 
is a morphism and $\psi=\phi_{\KK}$. This holds in particular if $\psi$ is $\Gamma$-equivariant.
\item 
$X \subseteq X_{\KK}$ carries the induced topology, and $\overline{X}=X_{\KK}$.
\ee
Now let $\phi\colon X\to Y$ be a morphism and $\phi_{\KK}\colon X_{\KK}\to Y_{\KK}$ the extension.
\be\addtocounter{enumi}{3}
\item 
If $y\in Y$ and $y \notin \phi(X)$, then $y\notin \phi_{\KK}(X_{\KK})$. In particular, we have $\phi(X) = \phi_{\KK}(X_{\KK}) \cap Y$.
\item 
$\phi$ is injective (resp. surjective, resp. bijective) if and only if $\phi_{K}$ is injective (resp. surjective, resp. bijective).
\item  
$\phi$ is an isomorphism (resp. a closed immersion, resp. an open immersion) if and only if $\phi_{\KK}$ is an isomorphism (resp. a closed immersion, resp. an open immersion).
\ee
\end{lemma}

\begin{proof}
(1) We can assume that $X$ is affine. Then $\OOO(X_{\KK})=K \otimes_{\kk}\OOO(X)$, so that we can write $f = \sum_{i=1}^m a_{i} f_{i}$ where $a_{1}=1$, the $a_{i}\in K$ are linearly independent over $\kk$, and $f_{i}\in\OOO(X)$. 
Let us assume that the inclusion $f(X) \subseteq \kk$ holds. Then $f(x) = \sum_{i}a_{i}f_{i}(x) \in \kk$ for any $x \in X$, and so $f_i(x)=0$ for each $i \geq 2$. Thus, $f_i = 0$ for $i \geq 2$, and we get $f = f_{1} \in \OOO(X)$. The other implication is obvious.
\ps
(2) We can assume that $Y$ is affine and that $Y \subseteq \An$ is 
a closed subset. Then $Y_{\KK}\subseteq \AA^{n}_{\KK}$ and $\psi\colon X_{\KK}\to Y_{\KK}\subseteq \AA_{\KK}^{n}$ is given by $n$ regular functions $f_{1},\ldots,f_{n}\in\OOO(X_{\KK})$. Since $\phi(X)\subseteq Y \subseteq \An$ we see that $f_{i}(X) \subseteq\kk$, and the claim follows from (1).
\ps
(3) Again we can assume that $X$ is affine. If $f \in \OOO(X_{\KK})$, $f=\sum_{i=1}^{m}a_{i}\otimes f_{i}$ with $\kk$-linearly independent $a_{i}\in K$, then $f(x) = 0$ for some $x\in X$ if and only if $f_{1}(x)=\cdots=f_{m}(x)=0$. This shows that for the zero sets we get $\VVV_{X_{\KK}}(f)\cap X = \VVV_{X}(f_{1},\ldots,f_{m})$, proving the first claim. 

For the second, we can assume that  $\overline{X} \subsetneqq X_{\KK}$. Then there is a nonzero $f\in\OOO(X_{\KK})$ vanishing on $\overline{X}$. By (1), $f \in \OOO(X)$, and we end up with a contradiction.
\ps
(4) We can assume again that $X,Y$ are affine. Then $y\notin \phi(X)$ means that $\phi^{*}(\mm_{y})\subseteq \OOO(X)$ generates $\OOO(X)$, i.e. $1= \sum_{i}f_{i}\phi^{*}(h_{i})$ for some $h_{i}\in \mm_{y}$ and $f_{i}
\in \OOO(X)$. Hence, $\phi_{\KK}^{*}(\mm_{y})$
also generates $\OOO(X_{\KK})$, and the claim follows.
\ps
(5a) If $\phi_{\KK}$ is injective, then so is $\phi$, and  if $\phi_{\KK}$ is surjective, then (4) implies that $\phi$ is surjective.
\ps
(5b) Assume that $\phi$ is injective. We prove that $\phi_{\KK}$ is injective by induction on $\dim X$. The case $\dim X = 0$ is obvious.
If $\dim X > 0$, then there is an open dense set $U \subseteq X$ such that $\phi(U) \subseteq Y$ is locally closed and that $\phi$ induces an isomorphism $U \simto \phi(U)$. It follows that $\phi_{\KK}\colon X_{\KK} \to Y_{\KK}$ induces an isomorphism $U_{\KK} \simto \phi(U)_{\KK} \subseteq Y_{\KK}$. In particular, $\phi(U)_{\KK} = \phi_{\KK}(U_{\KK})$. Setting $Z:=X \setminus U$ we see that  $\dim Z < \dim X$, and so the morphism $Z_{\KK} \to Y_{\KK}$ induced by $\phi_{\KK}$ is injective, by induction. It remains to see that the images $\phi_{\KK}(Z_{\KK})$ and $\phi_{\KK}(U_{\KK})$ are disjoint. If not, there is a nonempty locally closed and dense subset $A \subseteq \phi_{\KK}(Z_{\KK})\cap \phi_{\KK}(U_{\KK})$. Since the intersection is stable under $\Gamma := \Aut(\KK/\kk)$ we can assume that $A$ is also $\Gamma$-stable. Then, by Lemma~\ref{Gamma-stable.lem}(2), $A$ is defined over $\kk$: $A = (A\cap X)_{\KK}$.

If $A_{1} \subset U_{\KK}$ and $A_{2}\subset Z_{\KK}$ denote the inverse images of $A$ under the morphisms $U_{\KK} \to Y_{\KK}$ and $Z_{\KK} \to Y_{\KK}$, then both varieties are defined over $\kk$. Moreover, $A_{1}\simto A$ is an isomorphism and $A_{2}\to A$ is bijective. It follows that  $A_{1}\cap X\subset U$ and $A_{2}\cap X\subset Z$ are both mapped bijectively onto $A\cap X$ under $\phi$. This contradicts the fact that $\phi$ is injective and that  $U$ and $Z=X \setminus U$ are disjoint.
\ps
(5c)
Finally, assume that $\phi$ is surjective. We prove that $\phi_{\KK}$ is surjective by induction on $\dim Y$. Again, the case $\dim Y = 0$ is obvious. If $\phi_{\KK}$ is not surjective, then there is a subset $U \subseteq C:= Y_{\KK}\setminus\phi_{\KK}(X_{\KK})$ which is open and dense in $\overline{C}$. Since $C$ is stable under $\Gamma$ we can assume that $U$ is also stable under $\Gamma$. Hence, by Lemma~\ref{Gamma-stable.lem}(1), $U$ is defined over $\kk$ and $U\cap X\neq\emptyset$. But $U\cap X \nsubseteq \phi(X)$, by construction, contradicting the assumption.

\ps
(6a) If $\phi$ is an isomorphism, then it is clear that $\phi_{\KK}$ is an isomorphism. If $\phi_{\KK}$ is an isomorphism, then the inverse $\phi_{\KK}^{-1}$ is $\Gamma$-equivariant, because $\phi_{\KK}$ is $\Gamma$-equivariant. Thus,
$\phi_{\KK}^{-1}$ is defined over $\kk$ and induces a morphism  $\psi\colon Y  \to X$. It follows that the compositions $\phi\circ\psi$ and $\psi\circ\phi$ are the identity, and thus $\phi$ is an isomorphism.
\ps
(6b) If $\phi$ is a closed immersion, i.e. $\phi(X) \subseteq Y$ is closed and $\phi$ induces an isomorphism $X \simto \phi(X)$, then $\phi(X)_{\KK} \subseteq Y_{\KK}$ is closed and equal to $\phi_{\KK}(X_{\KK})$, by (5). Now the claim follows from (6a).

The case of an open immersion $\phi$ follows with the same arguments.
\ps
(6c) If $\phi_{\KK}$ is a closed immersion, then $\phi_{\KK}(X_{\KK})\subseteq Y_{\KK}$ is closed and $\Gamma$-stable, hence defined over $\kk$. Since $\phi(X)=\phi_{\KK}(X_{\KK})\cap Y$, by (4), we see that $\phi(X) \subseteq Y$ is closed, by (3),  and $\phi_{\KK}(X_{\KK})= \phi(X)_{\KK}$. Moreover, the morphism $\psi\colon X \to \phi(X)$, $x \mapsto \phi(x)$, induces the isomorphism $\psi_{\KK}\colon X_{\KK} \simto \phi_{\KK}(X_{\KK})$, $x \mapsto \phi_{\KK}(x)$. Thus $\psi$ is an isomorphism, by (6a), and so $\phi$ is a closed immersion.

Again, the case of an open immersion follows with the same arguments.
\end{proof}

We now extend some of these results to ind-varieties $\VVV$. 
If $\phi\colon \VVV \to \WWW$ is a morphism and $\phi_{\KK}\colon \VVV_{\KK}\to \WWW_{\KK}$ its extension we get a commutative diagram
$$
\begin{CD}
\VVV @>\phi>> \WWW \\
@VV{\subseteq}V @VV{\subseteq}V \\
\VVV_{\KK} @>{\phi_{\KK}}>> \WWW_{\KK}
\end{CD}
$$
The group $\Gamma=\Aut(\KK/\kk)$ acts on $\VVV_{\KK}$, and $(\VVV_{\KK})^{\Gamma}=\VVV$. Moreover, the morphism $\phi_{\KK}\colon \VVV_{\KK} \to \WWW_{\KK}$ is $\Gamma$-equivariant.

\begin{proposition}\label{fieldextension.prop}
Let $\VVV$, $\WWW$ be ind-varieties, and let $\KK/\kk$ be a field extension where $\KK$ is also algebraically closed. 
\be
\item 
$\VVV \subseteq  \VVV_{\KK}$ carries the induced topology, 
and $\VVV$ is dense in $\VVV_{\KK}$.
\item If $\UUU \subseteq \VVV_{\KK}$ is locally closed and $\Gamma$-stable, then $\UUU$ is defined over $\kk$, i.e. $\UUU\cap \VVV = \UUU^{\Gamma}\subseteq \VVV$ is locally closed and $\UUU = (\UUU\cap \VVV)_{\KK}$.
\item 
If $\psi\colon \VVV_{\KK}\to \WWW_{\KK}$ is a morphism of $\KK$-ind-varieties such that $\psi(\VVV) \subseteq \WWW$, then
the induced map $\phi \colon \VVV \to \WWW$, $v \mapsto \psi(v)$, 
is a morphism of $\kk$-ind-varieties and $\psi=\phi_{\KK}$. This holds in particular if $\psi$ is $\Gamma$-equivariant.
\ee
Let $\phi\colon \VVV\to \WWW$ be a morphism of $\kk$-ind-varieties and $\phi_{\KK}\colon \VVV_{\KK}\to \WWW_{\KK}$ the extension.
\be\addtocounter{enumi}{3}
\item  
The morphism $\phi$ is an isomorphism (resp. a closed immersion, resp. an open immersion) if and only if $\phi_{\KK}$ is an isomorphism (resp. a closed immersion, resp. an open immersion).
\item
The morphism $\phi$ is injective if and only if $\phi_{\KK}$ is injective.
\item
If $\phi_{\KK}$ is surjective, then $\phi$ is surjective. The other implication holds if $\kk$ is uncountable.
\ee
\end{proposition}

\begin{proof}
(1) If $\VVV = \bigcup_{k}\VVV_{k}$ is an admissible filtration, then $\VVV_{\KK} = \bigcup_{k}(\VVV_{k})_{\KK}$ is an admissible filtration which is $\Gamma$-stable. If $T \subseteq \VVV_{\KK}$ is closed, then $T\cap (\VVV_{k})_{\KK}$ is closed in $(\VVV_{k})_{\KK}$ for all $k$, and so $T \cap \VVV_{k}$ is closed in $\VVV_{k}$, by Lemma~\ref{fieldextension.lem}(3). Thus $T \cap \VVV$ is closed in $\VVV$. If $S \subseteq \VVV$ is closed, then $S$ is a closed ind-subvariety of $\VVV$ and so $S_{\KK}$ is a closed ind-subvariety of $\VVV_{\KK}$. Thus $\VVV$ carries the induced topology of $\VVV_{\KK}$.
Since $(\VVV_{k})_{\KK} =\overline{\VVV_{k}} \subseteq \overline{\VVV}$ by Lemma~\ref{fieldextension.lem}(3), we finally get $\overline{\VVV} = \VVV_{\KK}$.
\ps
(2) If $\UUU \subset \VVV_{\KK}$ is locally closed, then $\UUU \cap \VVV$ is locally closed in $\VVV$ by (1).
Also, $\UUU \cap (\VVV_k)_{\KK}$ is locally closed in $(\VVV_k)_{\KK}$ and $\Gamma$-stable. By
Lemma~\ref{Gamma-stable.lem}(2), we see that $\UUU \cap \VVV_k$ is a locally closed $\kk$-subvariety of $\VVV_k$ (which also follows from the fact that $\UUU \cap \VVV$ is locally closed in $\VVV$) and that $\UUU \cap (\VVV_k)_{\KK} = (\UUU \cap \VVV_k)_{\KK}$.
This yields $(\UUU \cap \VVV)_{\KK}= \bigcup_k (\UUU \cap \VVV_k)_{\KK} = \bigcup_k  \UUU \cap (\VVV_k)_{\KK} = \UUU$ as we wanted.
\ps
(3) For any $k$ there is an $\ell$ such that $\psi((\VVV_{k})_{\KK}) \subseteq (\WWW_{\ell})_{\KK}$. Since $\VVV_{k}=(\VVV_{k})_{\KK}\cap\VVV$ and $\WWW_{\ell}=(\WWW_{\ell})_{\KK}\cap\WWW$ we get $\psi(\VVV_{k})\subseteq \WWW_{\ell}$. Hence, by Lemma~\ref{fieldextension.lem}(2), $\phi|_{\VVV_{k}}\colon \VVV_{k}\to \WWW_{\ell}$ is a morphism and  $\psi|_{(\VVV_{k})_{\KK}}=(\phi|_{\VVV_{k}})_{\KK}$. The claim follows.
\ps
(4) If $\phi_{\KK}$ is an isomorphism, then $\phi$ is bijective. Hence $\phi_{\KK}^{-1}$ sends $\WWW$ bijectively onto $\VVV$, and $\phi_{\KK}^{-1}|_{\WWW} = \phi^{-1}$. Now it follows from Lemma~\ref{fieldextension.lem}(2) that $\phi^{-1}$ is a morphism, and so $\phi$ is an isomorphism. The other implication is obvious. 

If $\phi_{\KK}$ is a closed immersion, then $\phi(\VVV) = \phi_{\KK}(\VVV_{\KK})\cap \WWW$ is closed in $\WWW$, and  $\phi(\VVV)_{\KK}=\phi(\VVV_{\KK})$, by (2). Thus we have a decomposition $\phi\colon \VVV \overset{\bar\phi}{\to} \phi(\VVV) \overset{\iota}{\into} \WWW$ into a bijective morphism $\bar\phi$ and a closed immersion $\iota$. By assumption,  ${\bar\phi}_{\KK}$ is an isomorphism, and the claim follows. Again, the other implication is obvious.

The case of an open immersion can be proved along the same lines.
\ps
(5)
This follows directly from the corresponding statements for varieties (see  Lemma~\ref{fieldextension.lem}(5)).
\ps
(6) If $\phi_{\KK}$ is surjective and $w \in \WWW$, then $w \in \phi_{\KK}((\VVV_{k})_{\KK})$ for some $k$. Since $\phi(\VVV_{k}) = \phi_{\KK}((\VVV_{k}) _\KK)\cap \WWW$, by Lemma~\ref{fieldextension.lem}(4), we get $w \in \phi(\VVV_{k})$, hence $\phi$ is surjective.

Now assume that $\kk$ is uncountable and that $\phi$ is surjective. For any $\ell$ the inverse image $\phi_{\KK}^{-1}((\WWW_{\ell})_{\KK})\subseteq \VVV_{\KK}$ is closed and $\Gamma$-stable, hence defined over $\kk$. Since  $\phi_{\KK}^{-1}((\WWW_{\ell})_{\KK})\cap \VVV = \phi^{-1}(\WWW_{\ell})$ we get $\phi_{\KK}^{-1}((\WWW_{\ell})_{\KK}) = \phi^{-1}(\WWW_{\ell})_{\KK}$, by (2). Moreover, the induced morphism $\phi^{-1}(\WWW_{\ell}) \to \WWW_{\ell}$ is surjective, so that we can find a closed algebraic subset $Y \subseteq \phi^{-1}(\WWW_{\ell})$ such that $\phi(Y)= \WWW_{k}$ (Proposition~\ref{surjective-morphisms-to-varieties.prop}). By Lemma~\ref{fieldextension.lem}(5) the induced morphism $Y_{\KK} \to (\WWW_{\ell})_{\KK}$ is surjective, and the claim follows.
\end{proof}

Note that the assumption of an uncountable field is necessary in statement (6) as shown by the bijective morphism $\VVV \to \Aone$ where $\VVV$ is $\kk$ considered as a discrete set.

\ps
\subsection{Field of definition}
A well-known and useful result from algebraic geometry says that every variety $X$ is defined over a field $K$ which is finitely generated over the prime field. It is easy to see that this does not hold for ind-varieties. However, we have the following result.

\begin{proposition}\label{defined-over-countable-field.prop}
Every ind-variety $\VVV$ is defined over a countable field\idx{field of definition}, i.e. there is a countable algebraically closed field $\kk_{0}\subseteq \kk$ and a $\kk_{0}$-ind-variety $\WWW$ such that   $\WWW_{\kk}$ is $\kk$-isomorphic to $\VVV$.
\end{proposition}
\begin{proof}
For the proof we first remark the following. If $\kk_{0}\subseteq \kk_{1}\subseteq \cdots \subseteq \kk$ are  countable subfields of $\kk$, then the union $\bigcup_{i=1}^{\infty}\kk_{i}\subseteq \kk$ is a countable field. 

Now let $\VVV = \bigcup_{k} \VVV_{k}$, and assume that we have constructed a countable algebraically closed field $\kk_{0}$, a $\kk_{0}$-variety $\WWW_{\ell}$ and closed subvarieties $\WWW_{1}\subseteq \WWW_{2}\subseteq\cdots\subseteq\WWW_{\ell}$, and an isomorphism $\phi_{\ell}\colon (\WWW_{\ell})_{\kk} \simto \VVV_{\ell}$ such that $\phi_{\ell}((\WWW_{i})_{\kk}) = \VVV_{i}$ for $i=1,\ldots,\ell-1$. We want to construct a countable algebraically closed field $\kk_{1}\subseteq \kk$ containing $\kk_{0}$, a $\kk_{1}$-variety $\WWW_{\ell+1}$, a closed inclusion $(\WWW_{\ell})_{\kk_{1}} \into \WWW_{\ell+1}$, and an isomorphism $\phi_{\ell+1}\colon(\WWW_{\ell+1})_{\kk}$ such that $\phi_{\ell+1}((\WWW_{i})_{\kk} = \VVV_{i}$ for $i=1,\ldots,\ell$. Then the claim follows by induction.
\[
\begin{CD}
(\WWW_{1})_{\kk} @>\subseteq>> (\WWW_{2})_{\kk} @>\subseteq>>\cdots @>\subseteq>> (\WWW_{\ell})_{\kk} @>\subseteq>> (\WWW_{\ell+1})_{\kk} \\
@V{\phi_{1}}V{\simeq}V @V{\phi_{2}}V{\simeq}V  & &  @V{\phi_{\ell}}V{\simeq}V @V{\phi_{\ell+1}}V{\simeq}V\\
\VVV_{1} @>\subseteq>> \VVV_{2} @>\subseteq>> \cdots @>\subseteq>> \VVV_{\ell} @>\subseteq>> \VVV_{\ell+1}
\end{CD}
\]

There is a finitely generated field extension $\kk'/\kk_{0}$ contained in $\kk$, a $\overline{\kk'}$-variety $X$ and an isomorphism $\phi\colon X_{\kk}\simto \VVV_{\ell+1}$. Moreover, the inclusion $\iota\colon\VVV_{\ell} \into\VVV_{\ell+1}$ is defined over a finitely generated field extension $\kk''/\overline{\kk'}$, i.e. we have an inclusion $\iota''\colon (\WWW_{\ell})_{\overline{\kk''}} \into X_{\overline{\kk''}}$ such that $\iota''_{\kk}=\iota$. Now set $\kk_{1}:=\overline{\kk''} \subseteq \kk$ and $\WWW_{\ell+1}:=X_{\kk_{1}}$. Then $\kk_{1}$ is countable, and there is an isomorphism  $\phi_{\ell+1}\colon (\WWW_{\ell+1})_{\kk}\simto \VVV_{\ell+1}$ such that $\phi_{\ell+1}((\WWW_{j})_{\kk}) = \VVV_{j}$ for $j = 1,\ldots,\ell$. 
\end{proof}

\ps
\subsection{More topological notions, and examples}\label{top.subsec}

Recall that a subset of a topological space is called \itind{constructible} if it is a finite union of locally closed subsets. This is an important notion in the Zariski topology of a variety, because images of constructible subsets under morphisms of varieties are constructible. Before giving definitions which turn out to be very useful in the study of automorphism groups, we begin with the following basic result.\\

\begin{lemma} \label{image-ind-var-in-var.lem}
Let $\VVV=\bigcup_{k}\VVV_{k}$ be an ind-variety where all $\VVV_{k}$ are
irreducible,
and let  $\phi\colon \VVV \to X$ be a morphism where $X$ is a variety. Then
there is a $k_{0}\geq 1$ such that the following holds:
\be
\item $\overline{\phi(\VVV)} = \overline{\phi(\VVV_{k_{0}})}\subseteq X$;
\item There exists a subset $U_{0} \subseteq \phi(\VVV_{k_{0}})$ such that $U_{0}\subseteq \overline{\phi(\VVV)}$ is open and dense.
\ee
\end{lemma}

\begin{proof}
(1) Choose $k_0\geq 1$ such that $\dim \overline{\phi(\VVV_{k})} = \dim \overline{\phi(\VVV_{k_0})}$ for $k \geq k_0$. This implies  $\overline{\phi(\VVV_{k})} = \overline{\phi(\VVV_{k_0})}$ and therefore $\overline{\phi(\VVV)} = \overline{\phi(\VVV_{k_{0}})}$.
\ps
(2) Since $\phi(\VVV_{k_{0}})\subset X$ is constructible, it contains a subset $U_{0}$ which is open and dense in 
$\overline{\phi(\VVV_{k_0})}   = \overline{\phi(\VVV)} $.
\end{proof}

\begin{definition}\label{constr.def}
Let $\VVV = \bigcup_{k>0} \VVV_{k}$ be an ind-variety and $S \subseteq \VVV$ a subset.
\be
\item $S$ is called {\it weakly constructible}\idx{weakly constructible} if $S\cap\VVV_{k}$ is a constructible subset of $\VVV_{k}$ for all $k$. 
\item $S$  is called {\it ind-constructible\/}\idx{ind-constructible} if $S$ is a countable union of locally closed ind-sub\-varieties. Equivalently, $S$ is a countable union of algebraic subsets. 
\item $S$ is called {\it weakly closed}\idx{weakly closed} if the following holds: For every algebraic subset $U\subseteq \VVV$  such that $U\subseteq S$ we have $\overline{U} \subseteq S$.
\item The {\it weak closure\/}\idx{weak closure $\wc{S}$} of $S$ is defined by  $\overline{S}^{w}:=\bigcup_{X\subseteq S}\overline{X}$ where $X$ runs through all algebraic subsets of $\VVV$ lying in $S$. If $S$ is ind-constructible, then the weak closure is ind-constructible and weakly closed (see Proposition~\ref{indconstr.prop} below).
\ee
\end{definition}

\begin{remarks}
\be
\item
If $S \subseteq \VVV$ is weakly closed, $X$ an algebraic variety and  $\phi\colon X \to \VVV$  a morphism such that $\phi(X) \subseteq S$, then $\overline{\phi(X)} \subseteq S$.
\item 
The image of a morphism $\phi\colon \VVV \to \WWW$ of ind-varieties is ind-constructible, because it is the union of the constructible subsets $\phi(\VVV_{k})$.
\item
A constructible subset of $\VVV$ is  weakly constructible, and a weakly constructible subset is ind-constructible, but
the other implications do not hold. E.g.  the subset $\ZZ \subseteq \CC$ is ind-constructible, because the inclusion $\ZZ \to \kk$ is an ind-morphism, but it is not weakly constructible. An example of a weakly constructible subset which is not constructible is given in the following example.
\item
A subset $S$ is closed if and only if it is weakly closed and weakly constructible. 
\item 
If $S$ is stable under an automorphism of $\VVV$, then so are $\overline{S}$ and  $\wc{S}$.
\ee
\end{remarks}\label{constr.rem}

\begin{example}
Here we give an example of a weakly constructible subset which is not constructible. Take the ind-variety $\VVV =\bigcup_{k}X_{k}$ defined in Example~\ref{irred-not-curve-connected.exa} where $X_{n}$ is the zeros set in $\Atwo$ of $(x-1)\cdots (x-n)(y-1)\cdots(y-n)$. Since $\VVV$ is irreducible every dense constructible subset contains a nonempty open set of $\VVV$. 

Let $S \subseteq \VVV$ be the countable union of the locally closed sets $L_k$ where $L_k$ is the vertical line $x= k$ from which the $k$ points $(k,1), \ldots, (k,k)$ are removed. Then we see that $\VVV_k \cap S = L_1 \cup \cdots \cup L_k$ is locally closed in $\VVV_k$ for each $k$, hence constructible, and so $S$ is weakly constructible. Since $S$ is dense in $\VVV$ and does not contain a nontrivial open set, it follows that $S$ is not constructible.
\end{example}

\begin{example}[\name{Immanuel Stampfli}] Define the following subsets of $\A{2}$: $L_{0}:=\A{1} \times \{0\}$ and $L_{m}:=\{m\}\times \A{1}$ for $m>0$. Then $S:=\bigcup_{i\geq 0} L_{i} \subseteq \A{2}$ has the structure of an ind-variety by setting $S_{k}:=\bigcup_{i=0}^{k}L_{i}$. The inclusion $S \to \A{2}$ is a morphism  and $S \subseteq \A{2}$ is weakly closed. On the other hand, the subset $U:=S\setminus L_{0}$ is open and dense in $S$, but the weak closure $\wc{U} = \bigcup_{i\geq 1}L_{i}$ is strictly contained in $S$. So even for open sets of ind-varieties the weak closure might be smaller than the closure. The example also shows that the image of a morphism need not be constructible or weakly constructible.
\end{example}

\begin{proposition}\label{indconstr.prop} 
Assume that $\kk$ is uncountable.
Then every countable union of closed algebraic subsets of an ind-variety $\VVV$ is weakly closed. In particular, if $S\subseteq\VVV$ is ind-constructible, $S =\bigcup_{j}S_{j}$ with algebraic subsets $S_{j}\subseteq \VVV$, then $\wc{S}=\bigcup_{j}\overline{S_{j}}$, and this set is weakly closed.
\end{proposition}
\begin{proof}
Let $S = \bigcup_{i\in I}X_{i}$ be a countable union of closed algebraic subsets $X_{i}\subseteq\VVV$, and let $Y \subseteq \VVV$ be a locally closed algebraic subset contained in $S$. It follows that  $Y = \bigcup_{i\in I}X_{i}\cap Y$. If $\kk$ is uncountable, then, by Lemma~\ref{constructible.lem}, there is a finite subset $F \subseteq I$ such that $Y = \bigcup_{i\in F}X_{i}\cap Y$. Hence, $Y \subseteq \bigcup_{i\in F}X_{i}$ and so $\overline{Y}\subseteq\bigcup_{i\in F}X_{i} \subseteq S$.
\end{proof}

In general, the weak closure needs not to be weakly closed as the following example shows.
\begin{example}
Consider the analytic hypersurface 
$$
S :=\{(x,y,z)\in \AA^3_{\C}  \mid y =  \exp (z) x\} \subseteq \AA^3_{\C}.
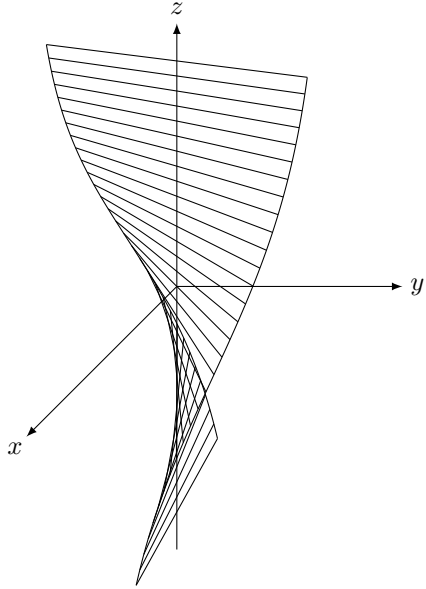
$$
\[\begin{tikzpicture}[scale=1]
\def \horizontalfactor {2}
\def \verticalfactor {2}
\draw[math3D][->,>=latex](0,0,0) -- (4,0,0);
\draw[math3D][->,>=latex](0,0,0) -- (0,3,0);
\draw[math3D][->,>=latex](0,0,-3.5) -- (0,0,3.5);
\draw[math3D] [domain={-1.5}:1.5, samples=80, smooth] (0,0,0) plot ({\horizontalfactor/sqrt(exp(2*\x)+1)},{\horizontalfactor*exp(\x)/sqrt(exp(2*\x)+1)},{\verticalfactor*\x});
\draw[math3D] [domain={-1.5}:1.5, samples=80, smooth] (0,0,0) plot ({-\horizontalfactor/sqrt(exp(2*\x)+1)},{-\horizontalfactor*exp(\x)/sqrt(exp(2*\x)+1)},{\verticalfactor*\x});
\foreach \x in {-1.5,-1.4,...,1.6} {
\draw[math3D] ({\horizontalfactor/sqrt(exp(2*\x)+1)},{\horizontalfactor*exp(\x)/sqrt(exp(2*\x)+1)},{\verticalfactor*\x})--({-\horizontalfactor/sqrt(exp(2*\x)+1)},{-\horizontalfactor*exp(\x)/sqrt(exp(2*\x)+1)},{\verticalfactor*\x});
}
\draw[math3D](4.3,0,0) node{$x$};
\draw[math3D](0,3.2,0) node{$y$};
\draw[math3D](0,0,3.7) node{$z$};
\end{tikzpicture}
\]
We claim that every irreducible algebraic subset of $S$ is contained  in one of the horizontal lines $H_a:= S \cap \{ (x,y,z) \mid z =a \}$, $a \in {\mathbb C}$, or in the vertical line $L:= S \cap \{ (x,y,z) \mid x =0 \}$. Note that $S$ is irreducible as an analytic variety, because the map $\phi \colon \AA^2_{\mathbb C} \to S$, $(x,z) \mapsto (x, \exp (z) x, z)$, is biholomorphic. Since $S$ is not algebraic it follows that every irreducible algebraic subset of $S$ is either a point or an irreducible algebraic curve $C$, and the closure 
$\overline{C}$ (for the $\C$-topology or for the Zariski topology) is still contained in $S$. 

In order to prove the claim above let us assume, by contradiction, that $\overline{C}$ is neither contained in the plane $(x=0)$ nor in a plane $(z=a)$ for some $a \in {\mathbb C}$. It follows that the map $\phi \colon \overline{C} \dasharrow \C$, $(x,y,z) \mapsto y/x=\exp z$, is rational and that the map $\psi \colon \overline{C} \to \C$, $(x,y,z) \mapsto z$, is dominant. Therefore, the image of $\psi$ is equal to $\C \setminus F$ where $F$ is a finite subset of $\C$. It follows that the rational maps $\phi, \psi \colon \overline{C} \dasharrow \C$ are algebraically independent over $\C$. This is impossible, because the field $\C ( \overline{C} )$ of rational functions on the curve $\overline{C}$ has transcendence degree one over $\C$, and the claim follows.

Now consider the open dense set $U:= S \setminus (H_0 \cup L)$ of $S$. The following assertions are straightforward consequences from the claim:
\be
\item The set $S$ is weakly closed in $\AA^3_{\C}$;
\item We have $\wc{U} = S \setminus H_0$;
\item We have $\wc{\wc{U}} = \wc{U}\cup\{0\}$;
\item We have $\wc{ \wc{\wc{U}} } = \wc{\wc{U}}$.
\ee
In particular, $\wc{U}$ is not weakly closed, but $\wc{\wc{U}}$ is.
\end{example}

Here is a useful  lemma whose easy proof is left to the reader.
\begin{lemma}\label{image-wconst.lem}
Let $\phi\colon \VVV \to \WWW$ be an ind-morphism. Then the image $\phi(\VVV)$ is ind-constructible. If we assume  that for every $k\in\NN$ there is an $m=m(k)$ such that $\phi(\VVV_{m})\supseteq  \phi(\VVV)\cap\WWW_{k}$, then the image $\phi(\VVV)$ is weakly constructible.
\end{lemma}

\pmed
\section{Ind-Groups, Lie Algebras, and  Representations}
\subsection{Ind-groups and Lie algebras}\label{Liealgebra.subsec}
The {\it product\/}\idx{product} of two ind-varieties $\VVV=\bigcup_{k}\VVV_{k}$ and $\WWW=\bigcup_{k}\WWW_{k}$ is defined in the obvious way: $\VVV \times \WWW:=\bigcup_{k} (\VVV_{k}\times \WWW_{k})$. This allows to define an {\it ind-group\/}\idx{ind-group} as an ind-variety $\GGG$ with a group structure such that multiplication $\GGG\times\GGG \to \GGG\colon (g,h)\mapsto g\cdot h$, and inverse $\GGG \to \GGG\colon g\mapsto g^{-1}$, are both morphisms.
It is clear that a closed subgroup $G$ of an ind-group $\GGG$ is an algebraic group if and only if $G$ is an algebraic subset of $\GGG$.

Recall that the \itind{Lie algebra} $\Lie G$\idx{$\Lie G$} of a linear algebraic group $G$ is defined as the Lie algebra of \itind{left invariant vector fields} on $G$. This means the following where we identify the vector fields on an affine variety $X$ with the derivations of $\OOO(X)$.\idx{derivation of $\OOO(X)$}
If $g \in G$ we denote by $\lambda_{g}\colon G \simto G$, $h\mapsto gh$, the \itind{left translation} on $G$, and by $\rho_{g}\colon G \simto G$, $h \mapsto hg$, the \itind{right translation}.
\idx{$\LLL$@$\lambda_{g}$}\idx{$\R$@$\rho_{g}$}

\begin{definition}  \label{LA-of-algebraic-group.def}
The Lie algebra $\Lie G$ of $G$ is the set of all derivations of the $\kk$-algebra $\OOO(G)$ that commute with the automorphisms $\lambda_{g}^{*}$ induced by the left translations $\lambda_{g}$, $g\in G$. The Lie bracket is defined by $[\delta_{1},\delta_{2}]:= \delta_{1}\circ\delta_{2} - \delta_{2}\circ\delta_{1}$.
\end{definition}

Let $T_e G$ be the tangent space of $G$ at the identity $e$, i.e. the vector space of all $\kk$-derivations from $\OOO(G)$ to the $\OOO(G)$-module $\OOO(G) / \mm_e = \kk$ where $\mm_e$ is the maximal ideal of $e$. Denote by $\eval_e \colon \OOO(G) \to \kk$ the evaluation homomorphism $f\mapsto f(e)$. For any vector field $\delta \in \Lie G$, the composition 
$\delta_{e}:=\eval_e \circ \delta$ is a tangent vector from $T_e G$, and we have the following crucial result (see for example \cite[Theorem~I.3.4]{Bo1991Linear-algebraic-g}).\idx{evaluation $\eval_{e}$}

\begin{proposition}  \label{derivations-and-tangent-space.prop}
The map $\Lie G \to T_e G$, $\delta \mapsto \delta_{e}:=\eval_e \circ \delta$, is an isomorphism of $\kk$-vector spaces.
\end{proposition}

This allows to identify $\Lie G$ with $T_e G$ and to carry over the structure of Lie algebra  to the tangent space $T_eG$.
\ps
The Lie algebra $\Lie\GGG$ of an affine ind-group $\GGG$ is defined analogously, using again the left-invariant vector fields on $\GGG$. The details, as for example the proof of Theorem~\ref{cont-derivations-and-tangent-space.thm} below, are given in \cite[Proposition 4.2.2]{Ku2002Kac-Moody-groups-t}.

\begin{definition} \label{Lie-algebra-of-an-affine-ind-group.def}
The Lie algebra $\Lie  \GGG$\idx{$\Lie\GGG$} of the affine ind-group $\GGG$ is the set of all continuous derivations of the $\kk$-algebra $\OOO(\GGG)$ that commute with automorphisms $\lambda_{g}^{*}$ induced by the left translations $\lambda_{g}$ of $\GGG$.
\end{definition}

The tangent space $T_e \GGG$ to $\GGG$ at the identity $e\in\GGG$ is the $\kk$-vector space of all continuous $\kk$-derivations $\delta\colon \OOO(\GGG) \to \OOO(\GGG)/\mm_e = \kk$.

\begin{theorem}[\protect{\cite[Proposition 4.2.2]{Ku2002Kac-Moody-groups-t}}] \label{cont-derivations-and-tangent-space.thm}
The linear map $\Lie  \GGG \to T_e \GGG$, $\delta \mapsto \delta_{e}:=\eval_e \circ \delta$, is an isomorphism of $\kk$-vector spaces.
\end{theorem}

As in the case of linear algebraic groups, this allows to identify $\Lie \GGG$ with $T_e \GGG$ and makes it possible to carry over the structure of Lie algebra to $T_e \GGG$.  

\begin{remark}\label{left-invariant-VF.rem}
The construction of the left-invariant vector field $\delta_{A}$ on $\GGG$ corresponding to the tangent vector $A \in T_{e}\GGG$ is given in the following way. 
\begin{equation*}
\begin{CD}
\delta_{A} \colon \OOO(\GGG) @>\mu^{*}>> \OOO(\GGG \times\GGG) = \OOO(\GGG)\hotimes\OOO(\GGG) @>{\id\hotimes A}>> \OOO(\GGG).
\end{CD}
\end{equation*}
Here we use that the coordinate ring of a product $\VVV \times \WWW$ of two affine ind-varieties is a completion $\OOO(\VVV)\hotimes\OOO(\WWW)$ of $\OOO(\VVV)\otimes\OOO(\WWW)$ (cf. \cite[IV.4.2]{Ku2002Kac-Moody-groups-t}).
\end{remark}

If $\phi\colon \GGG \to \HHH$ is a homomorphism of
affine ind-groups, then the differential 
$$
d\phi_{e}\colon \Lie\GGG \to \Lie\HHH
$$ 
is a homomorphism of Lie algebras. This follows from the definition above, see \cite[Proposition 4.2.2]{Ku2002Kac-Moody-groups-t}.

\ps
\subsection{The connected component of the identity}
For an ind-group $\GGG = \bigcup_{k}\GGG_{k}$ we denote by $\GGG^{\circ}$ the {\it connected component $\GGG^{(e)}$ containing the identity $e \in \GGG$}\idx{connected component of $e$}, see Section~\ref{connected-comp.subsec}.\idx{$\GL$@$\GGG^{\circ}$}

\begin{proposition}\label{connected-component.prop}
Let $\GGG = \bigcup_{k}\GGG_{k}$ be an ind-group.
\be
\item $\GGG^{\circ}=\bigcup_{k}\GGG_{k}^{(e)}$.
\item $\GGG^{\circ} \subseteq \GGG$ is a curve-connected open (and thus closed) normal subgroup of countable index. In particular, $\Lie\GGG=\Lie\GGG^{\circ}$.
\item We have  $\dim\GGG < \infty$ if and only if  $\GGG^{\circ}$ is an algebraic group.
\item We have  $\dim\GGG < \infty$ if and only if $\dim\Lie\GGG < \infty$.  
\ee
\end{proposition}
\begin{proof}
We know from Proposition~\ref{connected-components-are-open.prop} that $\GGG^{\circ}$ is open and closed in $\GGG$. 
\ps
(1) Set $\GGG':=\bigcup_{k}\GGG_{k}^{(e)} \subseteq \GGG$, the union of the connected components of $\GGG_{k}$ containing the identity $e$. Since $\GGG_{k}^{(e)}\subseteq \GGG^{\circ}$ we have $\GGG' \subseteq \GGG^{\circ}$. By Proposition~\ref{connected.prop}, there exists an admissible filtration $\GGG^{\circ} = \bigcup_k \GGG^{\circ}_{k}$ of $\GGG^{\circ}$ where all $\GGG^{\circ}_{k}$ are connected.
If $\GGG_{k}^{\circ}\subseteq \GGG_{\ell}$, then $\GGG_{k}^{\circ}\subseteq \GGG_{\ell}^{(e)}$, hence $\GGG^{\circ}\subseteq \GGG'$. 
\ps
(2) Clearly, $\GGG^{\circ}\cdot\GGG^{\circ}$, $(\GGG^{\circ})^{-1}$ and $g\GGG^{\circ}g^{-1}$ ($g \in \GGG$) are connected and contain $e$, and so $\GGG^{\circ}$ is a normal subgroup of $\GGG$. It is of countable index, because the number of connected components is countable. 

In order to prove that $\GGG^{\circ}$ is curve-connected it suffices to show that every closed connected algebraic subset $Y\subseteq \GGG^{\circ}$ is contained in an irreducible closed algebraic subset $X \subseteq \GGG^{\circ}$.  We prove this by induction on the number $n$ of irreducible components of $Y$. If $n\geq 2$ we can find two irreducible components
$C_{1},C_{2}$ of $Y$ which have a nonempty intersection. Choose a point $h \in C_{1}\cap C_{2}$ and consider the morphism $C_{1}\times C_{2}\to \GGG^{\circ}$, $(g_{1},g_{2})\mapsto g_{1}h^{-1}g_{2}$. Then the closure $D$ of the image is irreducible and contains $C_{1}$ and $C_{2}$. Hence, $Y':=D\cup Y$ has at most $n-1$ connected components, and the claim follows by induction.

\ps
(3) Assume that  $d:=\dim \GGG <\infty$.  Let $\GGG^{\circ}=\bigcup_{k}\GGG_{k}^{\circ}$ be an admissible filtration such that all $\GGG_{k}^{\circ}$ are irreducible (Proposition~\ref{curve-connected.prop}).  Then there is a $k_{0}$ such that $\dim\GGG_{k}^{\circ} = d= \dim\GGG_{k_{0}}^{\circ}$ for all $k\geq k_{0}$, hence $\GGG_{k}^{\circ}=\GGG_{k_{0}}^{\circ}$ for all $k\geq k_{0}$. 
It follows that $\GGG^{\circ}=\GGG_{k_{0}}^{\circ}$ is a connected algebraic group.

\ps
(4) If $\dim\GGG < \infty$, then $\GGG^{\circ}$ is an algebraic group, by (3). Since $\Lie\GGG = \Lie\GGG^{\circ}$ by (2) we see that $\Lie\GGG$ is finite-dimensional. On the other hand, $T_{g}\GGG \simeq T_{e}\GGG = \Lie \GGG$ for all $g \in \GGG$ and so $\dim T_{g}\GGG_{k} \leq \dim \Lie\GGG$ for all $k\geq 1$ and all $g \in \GGG_{k}$. Hence, $\dim \GGG_{k}\leq \dim\Lie\GGG$ for all $k$, i.e. $\dim\GGG \leq \dim\Lie\GGG$.     
\end{proof}

\begin{corollary}\label{subgroup-of-countable-index.cor}
Let $\kk$ be uncountable. If $\GGG$ is an ind-group and $\HHH \subseteq \GGG$ a closed subgroup of countable index, then $\GGG^{\circ} \subseteq \HHH$.
\end{corollary}\idx{countable index}

\begin{proof}
It is clear that $\HHH^{\circ}\subseteq \GGG^{\circ}$ and that $\HHH^{\circ}$ has countable index in $\GGG^{\circ}$, see Proposition~\ref{connected-component.prop}(2). Choose representatives $g_{i}\in \GGG^{\circ}$ ($i\in \NN$) of the left-cosets of $\HHH^{\circ}$ in $\GGG^{\circ}$ where $g_{1}=e$. If $\GGG^{\circ}=\bigcup_{k}\GGG_{k}$ with irreducible closed subsets $\GGG_{k}$ containing $e$, then $\GGG_{k}= \bigcup_{i=1}^{\infty}(\GGG_{k}\cap g_{i}\HHH^{\circ})$
is a countable union of disjoint closed subsets. Since $\kk$ is uncountable we get $\GGG_{k}= \GGG_{k}\cap g_{1}\HHH^{\circ} \subseteq \HHH^{\circ}$, hence $ \GGG^{\circ} \subseteq \HHH^{\circ}$, and so $ \GGG^{\circ} = \HHH^{\circ}$.
\end{proof}

\begin{remark}\label{irreducible-curve-connected.rem}
The proposition above shows that the following statements for an ind-group $\GGG$ are equivalent.  
\be
\item $\GGG$ is connected;
\item $\GGG$ is irreducible;
\item $\GGG$ is curve connected;
\item $\GGG = \GGG^{\circ}$.
\ee
\end{remark}

\ps
\subsection{Some examples of ind-groups}\label{examples-of-ind-groups.exa}
\be
\item
The most important examples for us are the automorphism groups $\Aut(X)$ of affine varieties $X$.  These groups will be studied in detail in Section~\ref{autgroup.sec}.

\item
If $V$ is a $\kk$-vector space of countable dimension, then $V^{+}:=(V,+)$ is a commutative connected ind-group.
Moreover this ind-group is \itind{nested} (see Example~\ref{nested.exa} below) since we may filter $V$ by a family of finite-dimensional subspaces.
If $W$ is another such $\kk$-vector space, then the homomorphisms (of ind-groups) $V \to W$ are the linear maps.
\idx{$\VEC$@$V^{+}$}

\item
Let $(G_{k},\iota_{k})_{k\in\NN}$ be a countable set of algebraic groups $G_{k}$ and injective homomorphisms $\iota_{k}\colon G_{k} \into G_{k+1}$. Then the limit $G:=\varinjlim G_{k}$ is a nested ind-group.
A typical example is $\GL_{\infty}(\kk):=\varinjlim \GL_{k}(\kk)\subsetneqq \GL(\CC^{\infty})$\idx{$\GL_{\infty}(\kk)$} which is the group of invertible matrices $(a_{ij})$ with the property that
$a_{ij}=0$ for large enough $i\neq j$, and $a_{ii}=1$ for large enough $i$.

\ps
More generally, for a $\kk$-vector space $V$ of countable infinite dimension we define
$$
\GL_{\infty}(V):=\{\phi\in\GL(V) \mid \Coker(\phi-\id) \text{ is finite-dimensional}\}.
$$
If $V_{1}\subset V_{2}\subset \cdots \subset V$ is a flag of finite-dimensional subspaces such that $V = \bigcup_{k}V_{k}$, then $\GL_{\infty}(V) = \varinjlim \GL(V_{k})$ where we use the obvious inclusions $\GL(V_{k}) \into \GL(V_{k+1})$. It follows that $\GL_{\infty}(V) \simeq \GL_{\infty}(\kk)$ as ind-groups.
\idx{$\GL_{\infty}(V)$}

\item
A countable group $G$ with the discrete topology is a \itind{discrete ind-group}. Moreover, an ind-group $\GGG$ is discrete if and only if its Lie algebra $\Lie\GGG$ is trivial. If $\kk$  is uncountable this is equivalent to the condition that $\GGG$ is countable.
\ee

\ps
\subsection{Smoothness of ind-groups}\label{smooth-groups.subsec}
It follows from the definitions of smoothness (see Definition~\ref{smooth.def}) that if an ind-groups is \itind{strongly smooth} in one point, then it is strongly smooth in every point, and the same holds for \itind{geometrically smooth}. Smoothness has very strong consequences, as the following corollary to Proposition~\ref{connected-and-smooth-gives-iso.prop} shows.

\begin{corollary}\label{bijective-hom.cor}
Let $\phi\colon \GGG \to \HHH$ be a bijective homomorphism of ind-groups. Assume that $\GGG$ is connected and $\HHH$ is strongly smooth in $e$. Then $\phi$ is an isomorphism.
\end{corollary}

In contrast to the case of algebraic groups, it is not true in general that an ind-group is strongly smooth. In fact, we will give an example of a bijective morphism of ind-groups which is not an isomorphism, see Proposition~\ref{diff-not-bijective.prop}. 
On the other hand, there is a class of groups which are strongly smooth. 

\begin{example}\label{nested.exa}
Assume that the ind-group $\GGG$ admits a filtration consisting of closed algebraic subgroups. Then $\GGG$ is strongly smooth in any of its points. 
\newline
Such ind-groups are called \itind {nested}, see \cite{KoPeZa2016On-Automorphism-Gr}. We will discuss these ind-groups in Section~\ref{nested.subsec}.
\end{example}

\ps
\subsection{\texorpdfstring{$R$}{R}-rational points of algebraic groups}\label{rational-points-of-groups.subsec}
Another type of interesting examples of ind-groups are the {\it groups $G(R)$ of $R$-rational points}\idx{R-r@$R$-rational points} of an algebraic group $G$ where $R$ is a commutative $\CC$-algebra of countable dimension.\idx{rational point}

\begin{proposition}  \label{G(R)-as-ind-group.prop}
Let $G$ be a linear algebraic group, and let $R$ be a commutative $\CC$-algebra of countable dimension, e.g. a finitely generated commutative $\kk$-algebra. Then the group $G(R)$ of $R$-rational points of $G$ is naturally an ind-group.
\end{proposition}\idx{$\GL$@$G(R)$}

\begin{proof}
We have seen in Proposition~\ref{X(R).prop} that for every affine variety $X$ and every commutative $\kk$-algebra $R$ the set of $R$-rational points $X(R)$ has a natural structure of an ind-variety. If $G$ is
a linear algebraic group, then $G(R)$ has a group structure, and the proposition also shows that the multiplication $G(R) \times G(R) \to G(R)$ is a morphism and the inverse $G(R) \simto G(R)$ is an isomorphism, hence the claim.
\end{proof}

\begin{remark}\label{G(R)-as-ind-group.rem}
If $X$ is an affine variety and $G$ a linear algebraic group, then we have a canonical identification 
$$
G(\OOO(X)) = G(X) := \Mor(X,G).
$$
\end{remark}

A special case is $R^{*}:=\GL_{1}(R)$, the {\it ind-group of invertible elements of $R$} \idx{invertible elements}. A detailed analysis of this ind-group will be given in Section~\ref{invertible-elements.sec}.\idx{$\R$@$R^{*}$}
   
\ps
Another special case is \itind{$\GL_{2}(\kk[t])$} which is worked out in \cite{Sh2004On-the-group-rm-GL}. It is shown there that, for the obvious filtration by degree, one has that $\GL_{2}(\kk[t])_{k}$ are exactly the singular points of $\GL_{2}(\kk[t])_{k+1}$.

The next result is an immediate consequence of Proposition~\ref{X(R).prop}.

\begin{proposition}
\be
\item 
Let $\phi\colon G \to H$ be a homomorphism of linear algebraic groups. Then, for any commutative $\CC$-algebra $R$ of countable dimension,  the induced homomorphism 
$$
\phi(R) \colon G(R) \to H(R)
$$ 
is a homomorphism of ind-groups. If $\phi$ is injective, then $\phi(R)$ is a closed immersion.
\item
Let $\rho\colon R \to S$ be a homomorphism of commutative $\CC$-algebras of countable dimension, and let $G$ be a linear algebraic group. Then the induced homomorphism $G(\rho)\colon G(R) \to G(S)$ is a homomorphism of ind-groups. If $\rho$ is injective, then $G(\rho)$ is a closed immersion.
\ee
\end{proposition}\label{Homs-of-G(R).prop}

For some applications one can reduce the study of the ind-group $G(R)$ to the case where $R$ is a {\it $\kk$-domain}, i.e. a commutative  $\kk$-algebra which is an integral domain.\idx{domain, $\kk$-domain}

\begin{proposition}   \label{reduction-to-domain.prop}
Let $G$ be a linear algebraic group and $R$ a finitely generated commutative $\kk$-algebra. There is a closed immersion of $G(R)$  into an ind-group of the form $H(S)$  where $H$ is a linear algebraic group and $S$ a finitely generated  $\kk$-domain.
\end{proposition}

The main ingredient of the proof is the following lemma which was indicated to us by \name{Claudio Procesi}.

\begin{lemma}  \label{embedding-into-a-matrix-algebra.lem}
Any finitely generated commutative $\kk$-algebra $R$ can be embedded into a matrix algebra $\M_n(S)$ where $S$ is a finitely generated $\kk$-domain.
\end{lemma}

\begin{proof}\idx{$\M_{n}(K)$}
It suffices to show that there is an embedding $\phi\colon R \into \M_{n}(K)$ where $K=\kk(x_{1},\ldots,x_{m})$ is a rational function field over $\kk$. In fact, if $R$ is generated by $a_{1},\ldots,a_{r}$ and if $S \subset K$ is the subalgebra generated by the coefficients of the matrices $\phi(a_{i})$, then $\phi(R) \subset \M_{n}(S)$.
\ps
The $\kk$-algebra $R$ embeds into the localization $R_{T}$ where $T \subset R$ is the set of nonzero divisors. Since $R_{T}$ is noetherian and of Krull-dimension $0$ it is a finite product of local artinian $\kk$-algebras: $R = R_{1}\times\cdots\times R_{m}$.  Thus it is enough to show that each $R_{i}$ can be embedded into an $\M_{n}(K)$ for a suitable rational function field $K$ over $\kk$.
\ps
So let $(R,\mm)$ be a local artinian $\kk$-algebra  whose residue field $L:=R/\mm$ is finitely generated over $\kk$.
We claim that $R$ contains a rational function field $K$ over $\kk$ such that $R$ is finite-dimensional over $K$. Then the regular representation of $R$ on itself defines an embedding $R \into \M_{n}(K)$, $n:=\dim_{K}R$, hence the claim.
\ps
It remains to prove the existence of a subfield $K \subset R$ such that $\dim_{K}R<\infty$. Choose  $c_{1},\ldots,c_{r}\in R$ such that the images $\bar c_{1},\ldots,\bar c_{r} \in R/\mm=L$ form transcendence basis of $L$. It follows that the field $K:=\kk(c_{1},\ldots,c_{r})$ is contained in $R$ and that $R/\mm$ is finite-dimensional over $K$. Since $\mm^{i}/\mm^{i+1}$ is a finitely generated $R$-module, it is also finite-dimensional over $K$, and sind $\mm^{d} = (0)$ for some $d\geq 1$ we finally get that $R$ is finite-dimensional over $K$.
\end{proof}

\begin{proof}[Proof of Proposition~\ref{reduction-to-domain.prop}]
By Lemma~\ref{embedding-into-a-matrix-algebra.lem}, there is an embedding $\phi \colon R \into \M_n(S)$  where $S$ is a finitely generated $\kk$-domain. Then $\phi$ induces an embedding 
$$
\tilde \phi \colon \M_n(R) \into \M_n ( \M_m(S) ) \simeq \M_{mn}(S)
$$
which is $\kk$-linear, hence closed. It follows that the subgroup
\[ 
H:= \{ h \in \GL_{mn} (S) \mid h, h^{-1} \in \Image\tilde \phi \}  
\]
is closed, and so $\tilde\phi$ induces an isomorphism $\GL_{n}(R) \simto H$. Now the claim follows since every linear algebraic group $G$ is isomorphic to a closed subgroup of some $\GL_{n}(\kk)$.
\end{proof}

Recall that an abstract group is said to be {\it linear} if it admits an embedding into a linear group $\GL_n (\KK)$ for some field $\KK$. Thus the above theorem implies the following result.\idx{linear group}

\begin{corollary}  \label{the-groups-G(R)-are-linear.cor}
Any ind-group of the form $G(R)$ where $G$ is a linear algebraic group and $R$ a finitely generated commutative $\kk$-algebra is linear.
\end{corollary}

\ps
\subsection{Representations of ind-groups}\label{reps-of-ind-groups.subsec}
Let $\GGG$ be an ind-group, and let $V$ be a $\kk$-vector space of countable dimension. We denote by $\LLL(V)$\idx{$\LLL(V)$} the $\kk$-vector space of \itind{linear endomorphisms}. Note that $\LLL(V)$ is not an ind-variety in case $V$ is not finite-dimensional.

\begin{definition}  \label{representation-of-an-ind-group.def}
A group homomorphism $\rho\colon\GGG \to \GL(V)$ is called a {\it representation of $\GGG$ on $V$\/}\idx{representation} if the induced map $\tilde\rho\colon\GGG \times V \to V$, $(g,v)\mapsto g v:= \rho(g)v$, is an ind-morphism. We also say that $V$ is a {\it $\GGG$-module}.\idx{G@$\GGG$-module}
\end{definition} 
If $\rho\colon\GGG \to \GL(V)$ is a representation we define a linear map $d\rho\colon \Lie\GGG \to \LLL(V)$ in the following way. For $v\in V$, let $\mu_{v}\colon \GGG \to V$  be the \itind{orbit map} $g \mapsto g v$. Then 
\[
d\rho(A)(v) := (d\mu_{v})_e (A) \text{ \ for \ } A \in \Lie\GGG.
\]
For $A \in \Lie\GGG$ and $v \in V$ we will write $A(v)$ for $d\rho(A)(v)$
The following result can be found in \cite[Lemma~4.2.4]{Ku2002Kac-Moody-groups-t}.

\begin{lemma}\label{rep-of-G-and-LieG.lem}
The map
$d \rho\colon \Lie\GGG \to \LLL (V)$,
$A \mapsto (d\mu_{v})_e (A)$, for  $A \in \Lie\GGG$ and $v \in V$ is a homomorphism of Lie algebras.
\end{lemma}
\begin{remark}\label{diff-for-linear-action.rem}
Denote by  $\tilde\rho\colon\GGG \times V \to V$ the corresponding linear action. It is easy to see that 
the differential $d\tilde\rho_{(e,v)}\colon \Lie\GGG \times T_{v}V \to T_{v}V$ has the following description:
$$
d\tilde\rho_{(e,v)}(A,w) = A(v) + w \text{ \ for }A \in \Lie \GGG, \ w \in T_{v}V=V.
$$
\end{remark}

\begin{proposition}\label{trivial-rep.prop}
Let $\rho\colon \GGG \to \GL(V)$ be a representation of a connected ind-group $\GGG$. Then $\rho$ is trivial if and only if $d\rho$ is trivial.
\end{proposition}
\begin{proof}
It is clear from the definition of $d\rho$ that for a trivial representation $\rho$ we have $d\rho = 0$. Now assume that $d\rho = 0$ and that $\GGG$ is connected. If $d\rho$ is trivial, then the differentials in $e\in\GGG$ of the orbit maps $\mu_{v}\colon \GGG \to V$ are all zero. Hence $(d\mu_{v})_{g}=0$ for all $g\in\GGG$. In fact, $\mu_{gv} = \mu_v \circ \rho_g$ where $\rho_{g}$ is the right translations, and so $(d \mu_{gv})_e = (d \mu_v)_g \circ (d \rho_g)_e$ from which the result follows.

We choose an admissible filtration $\GGG = \bigcup \GGG_{k}$ such that $\GGG_{k}$ is irreducible (Proposition~\ref{curve-connected.prop}) and that $e\in\GGG_{k}$ for all $k$. If $\rho$ is nontrivial we can find a $v\in V$ such that $\GGG v \neq \{v\}$. It follows that for a large enough $k$ the morphism $\mu\colon\GGG_{k}\to V$, $g\mapsto g v$, is non-constant. Since $\GGG_{k}$ is irreducible this implies that the differential of $\mu$ is non-constant on a dense open set of $\GGG_{k}$, contradicting what we said above.
\end{proof}

A fundamental result in the theory of algebraic groups is that every affine $G$-variety can be equivariantly embedded into a finite-dimensional $G$-module. This does not hold for ind-groups.

\begin{proposition}\label{no-embedding-into-representations.prop}
Assume that an ind-group $\GGG$ of infinite dimension (see Definition~\ref{dimension.def}) acts faithfully on an affine variety $X$. Then there is no embedding of $X$ as a closed $\GGG$-stable subvariety of a representation $V$ of $\GGG$.
\end{proposition}

\begin{proof}
Assume that $X$ embeds equivariantly into a representation $V$. Since $X$ is a variety the linear span $W:= \langle X \rangle \subseteq V$ is finite-dimensional and $\GGG$-stable. Therefore,  we obtain an injective homomorphism of ind-groups $\GGG \into \GL (W)$.
This is clearly  impossible, because $\dim \GGG = \infty$.
\end{proof}

Note that the automorphism group $\Aut (X)$ acts faithfully on the coordinate ring $\OOO (X)$. We will see later in Proposition~\ref{locally-finite-action-on-Vec(X).prop} that this is a representation of the ind-group $\Aut(X)$. But we don't even know if a general affine ind-group admits a nontrivial representation.

\begin{question} \label{faithful-representation.ques}
Does any affine ind-group admit a faithful (or only nontrivial) representation on a $\kk$-vector space of countable dimension ?
\end{question}

\ps
\subsection{Homomorphisms with small kernels}\label{small-kernels.subsec}
The first result is an immediate consequence of the fact that the image of a homomorphism of algebraic groups is closed.
\begin{proposition} \label{Hom-G-to-ind-group.prop}
Let $\GGG$ be an ind-group, $G$ an algebraic group, and $\phi\colon G \to \GGG$ a homomorphism of ind-groups. Then the image $\phi(G)$ is a closed algebraic subgroup.
\end{proposition}
\begin{proof}
If $\GGG = \bigcup_{k}\GGG_{k}$, then there is a $k$ such that $\phi(G) \subseteq \GGG_{k}$. It follows that  $\overline{\phi(G)} \subseteq \GGG_{k}$, and so $\overline{\phi(G)}$ is a closed algebraic subgroup and $\phi$ induces a homomorphism $G \to \overline{\phi(G)}$ of algebraic groups with a dense image. Hence $\phi(G) = \overline{\phi(G)}$.
\end{proof}

We now apply the results from Section~\ref{small-fibers.subsec} to the case of homomorphisms of ind-groups.
\begin{proposition}  \label{hom-to-algebraic.prop}
Let $\GGG$ be an ind-group, $G$ an algebraic group, and $\phi\colon \GGG \to G$ a homomorphism of ind-groups.
\be
\item \label{the-neutral-component-is-an-algebraic-group}
If $\dim\Ker\phi<\infty$, then $\GGG^{\circ}$ is an algebraic group. In particular, $\dim\GGG<\infty$.
\item
If $\GGG$ is connected, then $\phi(\GGG) \subseteq G$ is a closed subgroup.
Furthermore, there exists a closed irreducible algebraic subset $X$ of $\GGG$ such that $\phi (X) = \phi(\GGG)$.
\item
If $\GGG$ is connected and $\phi$ surjective, then $d\phi_{e}\colon \Lie\GGG \to \Lie G$ is surjective, and $\Ker d\phi_{e}\supseteq \Lie\Ker\phi$.
\ee
\end{proposition}

\begin{proof}
(1) This follows from Proposition \ref{small-fibers-gives-variety.prop}, because $\GGG^{\circ}$ is curve-connected, by  Proposition~\ref{connected-component.prop}
\par\smallskip
(2) We choose an admissible filtration $\GGG=\bigcup\GGG_{k}$ such that all $\GGG_{k}$ are irreducible. Then the subsets $\overline{\phi(\GGG_{k})} \subseteq G$ are irreducible and closed, and so there is a $k_{0}>0$ such that  
$\overline{\phi(\GGG_{k_{0}})} = \overline{\phi(\GGG_{k})}$ for all $k\geq k_{0}$. It follows that 
$H:=\overline{\phi(\GGG_{k_{0}})} \subseteq G$ is a closed subgroup and that $\phi(\GGG) \subseteq H$. 
The image $\phi(\GGG_{k_0})$ contains a subset $U$ which is open and dense in $H$. Since $H$ is irreducible, it is well-known that $U \cdot U = H$ (see for example \cite[Lemma 7.4, page 54]{Hu1975Linear-algebraic-g}). This yields $\phi(\GGG) = H$. Also, if $k_1$ is chosen such that $\GGG_{k_0} \cdot \GGG_{k_0} \subseteq \GGG_{k_1}$, we get $\phi ( \GGG_{k_1} ) = H$, so that we can take $X= \GGG_{k_1}$.

\ps
(3) 
Choose $X$ as in (2). This implies that the differential $d\phi_{g}$ is surjective for all $g$ in an open dense set of $X$. Hence, for the induced morphism $\phi \colon g^{-1}X \to G$ the differential $d\phi_{e}\colon T_{e}g^{-1}X \to \Lie G$ is surjective, proving the first claim. The second assertion is clear.
\end{proof}

The example $\ZZ \subseteq \kplus$ shows that the connectedness assumption in (2)  is necessary. If $\kk$ is uncountable, then (3) holds without assuming that $\GGG$ is connected. In fact, $\GGG \to G$ surjective then implies that $\GGG^{\circ}\to G^{\circ}$ is surjective.

\begin{question}   \label{Is-the-Kernel-of-the-differential-the-Lie-algebra-of-the-Kernel.ques}
Do we have that $\Ker d\phi_{e} = \Lie\Ker\phi$?
\end{question}

In case of an uncountable base field $\kk$ we have the following variation of the proposition above.
\begin{proposition}\label{inverse-image-of-algebraic.prop}
Assume that $\kk$ is uncountable.
Let $\phi\colon \GGG \to \HHH$ be a homomorphism of ind-groups whose kernel is an algebraic group. Let $H \subseteq \HHH$ be an algebraic subgroup and assume  that $H \cap \phi(\GGG)$ is closed in $H$. Then, the inverse image $\phi^{-1}(H)$ is an algebraic subgroup. In particular, this holds if $H \cap \phi(\GGG)$ is constructible or if $H \subseteq\phi(\GGG)$. 
\end{proposition}

\begin{proof}
We have $H \cap \phi(\GGG) = \bigcup_{k}H \cap\phi(\GGG_{k})$, and so $H \cap\phi(\GGG) = H \cap \phi(\GGG_{k})$ for some $k\geq 0$ by Lemma~\ref{constructible.lem}. This implies that $\phi^{-1}(H) = \Ker\phi \cdot( \phi^{-1}(H)\cap\GGG_{k}$), hence $\phi^{-1}(H) \subseteq \GGG_{m}$ for some $m\geq k$, proving that $\phi^{-1}(H)$ is an algebraic subgroup.
\end{proof}

\ps
\subsection{Factor groups and homogeneous spaces}\idx{factor group}
If $\HHH \subseteq \GGG$ is a closed subgroup of an ind-group $\GGG$ one might ask if the homogeneous space\idx{homogeneous space} $\GGG/\HHH$ has the structure of an ind-variety with the usual universal properties. If $\GGG$ is affine and $H=\HHH$ a reductive algebraic subgroup, then on easily shows that $\GGG$ has an admissible filtration $\GGG = \bigcup_{k}\GGG_{k}$ where all $\GGG_{k}$ are affine and stable under right multiplication with $H$. Since all $H$-orbits are isomorphic to $H$, the geometric quotient $\GGG_{k}/ H$ exists as an affine variety. It then follows that $\GGG/H = \varinjlim \GGG_{k} / H$ is an affine ind-variety with an action of $\GGG$ which has the usual universal properties of a homogeneous space.

However, one cannot expect such a result in general. In Section~\ref{tame-is-closed.subsec} we will construct a connected affine ind-group $\GGG$ and a closed connected subgroup $\HHH \subsetneqq \GGG$ which has the same Lie algebra. Let us show that  the homogeneous space $\GGG/\HHH$ cannot have a structure of an ind-variety with the usual properties of a homogeneous space. In fact, if  the projection $\pi\colon \GGG \to \GGG/\HHH$ is a $\GGG$-equivariant morphism, then $d\pi_{e}$ is trivial, because $\Ker d\pi_{e}\supseteq \Lie\HHH = \Lie\GGG$. It follows from the $\GGG$-equivariance that $d\pi_{g}$ is trivial in every point $g \in \GGG$ which implies that $\GGG/\HHH$ is a point (see Example~\ref{trivial-differential.exa}).

\pmed
\section{Families of Morphisms, of Endomorphisms, and of Automorphisms}\label{fam-morph.sec}
\ps
\subsection{Families of morphisms and the ind-variety of morphisms}\label{fam-morph.subsec}

\begin{definition} \label{fam-morphisms.def}
Let $\VVV,\WWW$ and $\ZZZ$ be  ind-varieties.  A {\it family of morphisms from $\VVV$ to $\WWW$ parametrized by $\ZZZ$}\idx{family of morphisms} is a morphism $\Phi \colon \ZZZ \times \VVV \to \WWW$ of ind-varieties. We use the notation $\Phi = (\Phi_{z})_{z \in \ZZZ}$ where $\Phi_{z} \colon \VVV \to \WWW$ denotes the morphism sending $v$ to $\Phi (z,v)$. 
\end{definition}

\begin{definition}  \label{universal-property-of-the-ind-variety-Mor(V,W).def}
We say that $\Mor(\VVV, \WWW)$ {\it admits a natural structure of ind-variety\/} if it can be endowed with the structure of an ind-variety in such a way that families of morphisms $\ZZZ \to \Mor ( \VVV, \WWW)$ correspond to ind-morphisms $\ZZZ \to \Mor ( \VVV, \WWW)$.

In other words, $\Mor ( \VVV, \WWW)$ admits a natural structure of ind-variety if and only if the contravariant functor 
\[
\Mor (\VVV, \WWW) \colon \IndVar \to \Sets, \ \ZZZ \mapsto \Mor(\ZZZ \times \VVV, \WWW)
\]
is representable.
\end{definition}\idx{$\Mor ( \VVV, \WWW)$}

The following results follow immediately from the definitions.

\begin{lemma}\label{tautological.lem}
Let $\XXX$, $\YYY$, $\ZZZ$, and $\TTT$ be ind-varieties, and assume that $\Mor ( \XXX, \YYY)$, $\Mor ( \YYY, \ZZZ)$, $\Mor ( \XXX, \ZZZ)$, $\Mor ( \YYY, \TTT)$ admit natural structures of ind-varieties. Then we have the following.
\be
\item \label{eval}
The evaluation map
\[  
\Mor(\XXX, \YYY) \times \XXX \to \YYY, \quad (f,x) \mapsto f(x) 
\]
is a morphism.
\item  \label{composition1}
The composition map
\[ 
\Mor(\XXX, \YYY) \times \Mor(\YYY, \ZZZ) \to \Mor (\XXX,\ZZZ), \quad (f,g) \mapsto g \circ f 
\]
is a morphism.
\item  \label{composition2}
For any $g \colon \ZZZ \to \TTT$ the composition map
\[  
\Mor(\YYY, \ZZZ)  \to \Mor(\YYY, \TTT), \quad f \mapsto g \circ f 
\]
is a morphism.
\item \label{composition3}
For any $f \colon \XXX \to \YYY$ the composition map
\[  
\Mor(\YYY, \ZZZ)  \to \Mor(\XXX, \ZZZ), \quad g \mapsto g \circ f 
\]
is a morphism.
\ee
\end{lemma}

Here is our first existence result.

\begin{lemma}   \label{ind-var-morphisms.lem}
If $X$, $Y$ are affine varieties, then $\Mor (X,Y)$ admits a natural structure of ind-variety. 
Moreover, $\Mor (X,Y)$ is an affine ind-variety.
\end{lemma}\idx{$\Mor (X,Y)$}

\begin{proof}
By Proposition~\ref{X(R).prop}, the set $\Mor(X,Y) = Y(\OOO(X))$ admits a structure of ind-variety. It is easy to see that this structure has the requested universal property.
\end{proof}

This result is easily extended to the case where $Y$ is an affine ind-variety. For the proof  we use the following result.

\begin{lemma}\label{Mor-for-closed-immersions.lem}
Let $X,Y,Z$ be affine varieties and let $\sigma\colon Y \into Z$ be a closed immersion. Then the induced map
$$
\sigma_{*}\colon \Mor(X,Y) \into \Mor(X,Z)
$$
is a closed immersion of ind-varieties.
\end{lemma}
\begin{proof}
By Lemma~\ref{tautological.lem}(\ref{composition2}) $\sigma_{*}$ is an ind-morphism.

Next we show that the image of $\sigma_{*}$, $\III:=\{\phi\in\Mor(X,Z)\mid \phi(X) \subseteq \sigma(Y)\}$, is closed. The evaluation maps $\eval_{x}\colon \Mor(X,Z)\to Z$, $\phi\mapsto \phi(x)$, are morphisms by Lemma~\ref{tautological.lem}(\ref{eval}), and so 
$\III = \bigcap_{x\in X} \eval_{x}^{-1}(\sigma(Y))$ is closed.

It remains to see that $\sigma_{*}$ is a closed immersion.
Again by Lemma~\ref{tautological.lem}(\ref{eval}), the map $\eps\colon \III\times X \to \sigma(Y)\simeq Y$, $(\varphi, x ) \mapsto \varphi (x)$ is also an ind-morphism,
as well as the corresponding map $\III \to \Mor(X,Y)$ (see Definition~\ref{universal-property-of-the-ind-variety-Mor(V,W).def}). It follows from the construction that this is the inverse of the ind-morphism $\Mor(X,Y) \to \III$ induced by $\sigma_{*}$.
\end{proof}

\begin{proposition} \label{indmor.prop}
If $X$ is an affine variety and $\YYY$ an affine ind-variety, then $\Mor (X,\YYY)$ admits a natural structure of ind-variety. Moreover, $\Mor (X,\YYY)$ is an affine ind-variety.
\end{proposition}\idx{$\Mor (X,\YYY)$}

\begin{proof}
Let $\YYY = \bigcup_{k\in\NN} \YYY_{k}$ be an admissible filtration of $\YYY$ by affine varieties. By Lemma~\ref{ind-var-morphisms.lem}, $\Mor (X, \YYY_{k})$ admits a universal structure of ind-variety, and, for each $k$, we have a natural closed immersion  $\Mor (X, \YYY_{k}) \to \Mor (X, \YYY_{k+1})$, by Lemma~\ref{Mor-for-closed-immersions.lem} above. It follows from Lemma~\ref{lim-of-indvar.lem} that  $\Mor(X,\YYY) = \varinjlim \Mor (X, \YYY_{k})$ is an ind-variety, and one easily checks that it has  the universal property we were looking for.
\end{proof}

We now provide examples showing that $\Mor (X,Y)$ does not necessarily admit a natural structure of ind-variety in the two following cases:
\be
\item $X$ is an affine variety and $Y$ a quasi-affine variety;
\item $X$ is an affine ind-variety and $Y$ an affine variety.
\ee

\begin{example}
Consider the quasi-affine variety $\Atwod:=\Atwo\setminus\{(0,0)\}$. We claim that $\Mor(\Aone,\Atwod)$ does not admit a natural structure of ind-variety. Assume that there is such a structure. Consider the natural inclusion 
$$
\iota\colon \Mor(\Aone,\Atwod) \into \Mor(\Aone,\Atwo)=\kk[x]^{2}
$$ 
and  the closed algebraic subset $Z:=\{(r,1+sx)\mid r,s \in\kk\} \subseteq \Mor(\Aone,\Atwo)$.
Then 
$$
Y:= Z \cap \Mor(\Aone,\Atwod) = \{(r,1+sx)\mid r,s\in\kk, r\neq 0 \text{ or } r=s=0\} \subseteq \Mor(\Aone,\Atwod)
$$
is closed, and it is an irreducible closed algebraic subset, because it is the image of the morphism $\Atwo \to \Mor(\Aone,\Atwod)$, $(a,b)\mapsto (a,1+abx)$. Since $Z$ is normal
and since the injective morphism $Y \to Z$ is dominant
it follows from \name{Zariski}'s Main Theorem in its original form (see \cite[Chap.~III, \S9, p. 209]{Mu1999The-red-book-of-va}) that $Y \to Z$ is an open immersion. This is obviously a contradiction since $Y$ is not open in $Z$.
\end{example}

\begin{example}
We claim that  $\Mor(\Ainfty, \AA^1 ) =\OOO(\Ainfty)$ does not admit a natural structure of ind-variety. On one hand $\Mor(\Ainfty, \AA^1 ) =\OOO(\Ainfty)$ is equal to the projective limit $\varprojlim \kk[x_{1},\ldots,x_{n}]$, which contains ${\dis \varprojlim_n ( \kk x_n) = \prod_n (\kk x_n) \simeq \kk ^{\NN} }$, hence its cardinality is at least equal to the cardinality of the power set $\PPP (\kk)$. On the other hand it is clear that the cardinality of any ind-variety is at most equal to the cardinality of $\kk$.
\end{example}

\begin{remark}
In the case where $X$ is a projective variety and $Y$ a quasi-projective variety, \name{Grothendieck} has shown that there exists a universal structure of scheme on $\Mor(X,Y)$, the scheme having in general countably many components \cite[4(c)]{Gr1995Techniques-de-cons}, see also \cite[Chapter 2]{De2001Higher-dimensional}. The scheme $\Mor(X,Y)$ is naturally realized as an open subscheme of the Hilbert scheme $\Hilb(X \times Y)$. For example,  $\Mor (\PP^1, \PP^n)$ is the disjoint union of the varieties $\Mor_d( \PP^1, \PP^n)$ of morphisms of degree $d$, and $\Mor_d( \PP^1, \PP^n)$ is an open set of the projective space $\PP ( S^d (\kk^2 )^{n+1})$ (see \cite[2.1, page 38]{De2001Higher-dimensional}).
\end{remark}

\begin{remark}\label{evaluation and composition.rem}
If $X$ is an affine variety, then $\End(X) = \Mor (X,X)$ is an \itind{ind-semigroup}, i.e. an ind-variety with a semigroup structure such that the multiplication is an ind-morphism.
\end{remark}

\begin{remark}   \label{morphisms-to-groups.rem}
If $G$ is a linear algebraic group and $X$ an affine variety, we have already seen in Section~\ref{rational-points-of-groups.subsec} that $\Mor(X,G)=G(\OOO(X))$ is an ind-group in a natural way.
\end{remark}

If $X,Y,Z$ are affine varieties, then every morphism $\phi\colon X\times Y \to Z$ can be regarded as a family of morphisms $Y \to Z$ parametrized by $X$, hence defines a morphism $X \to \Mor(Y,Z)$.

\begin{lemma}  \label{mor-mor.lem}
Let $X,Y,Z$ be affine varieties. Then the natural map
$$
\Mor(X \times Y,Z ) \to \Mor(X,\Mor(Y,Z)), \ \phi\mapsto (\phi(x,?))_{x\in X},
$$
is an isomorphism of ind-varieties.
\end{lemma}
\begin{proof}
(a) We first assume that $Z=\An$. Then $\Mor(X\times Y,Z) = \OOO(X\times Y)^{n}$ and 
$$
\Mor(X,\Mor(Y,Z)) = \Mor(X,\OOO(Y)^{n}) = \OOO(X)\otimes \OOO(Y)^{n}
$$
where we use that for a $\kk$-vector space $V$ of countable dimension we have $\Mor(X , V) = \OOO(X)\otimes V$. The claim follows in this case, because the map is a linear isomorphism of $\kk$-vector spaces of countable dimension.
\par\smallskip
(b) In general, we fix a  closed embedding $Z \subseteq \An$. Then $\Mor(X\times Y, Z) \subseteq \Mor(X\times Y, \An)$ and 
$\Mor(X, \Mor(Y, Z)) \subseteq \Mor(X, \Mor(Y,\An))$ are closed ind-subvarieties, and the claim follows from (a).
\end{proof}

Two other useful results in this context are the following.
\begin{lemma}  \label{mor-to-tangent.lem}
Let $X,Y$ be affine varieties, and let $x_{0}\in X$, $y_{0}\in Y$. Then the subset 
$$
\Mor_{0}(X,Y) := \{\phi\in\Mor(X,Y) \mid \phi(x_{0})=y_{0}\}\subseteq \Mor(X,Y)
$$
is closed, and the  map $\Mor_{0}(X,Y) \to \LLL(T_{x_{0}}X,T_{y_{0}}Y)$, $\phi\mapsto d\phi_{x_{0}}$, is
an ind-morphism.
\end{lemma}

\begin{proof}
One easily reduces to the case where $Y$ is a $\kk$-vector space $V$ and $y_{0} =0$. Then $\Mor(X,V) = \OOO(X)\otimes V$ and $\Mor_{0}(X,V) = \mm_{x_{0}}\otimes V$, hence  closed in $\Mor(X,Y)$, and the map $\phi\mapsto d\phi_{x_{0}}$ corresponds to $\mm_{x_{0}}\otimes V \to
\mm_{x_{0}}/\mm_{x_{0}}^{2}\otimes V = \LLL(T_{x_{0}}X,T_{0}V)$.
\end{proof}

\begin{lemma} \label{closed-immersion-Mor.lem}
Let $X,Y$ be affine varieties and $\WWW,\ZZZ$ affine ind-varieties.  
\be 
\item \label{closed-immersion-yields-closed-immersion-for-morphisms}
If $\sigma\colon \WWW \into \ZZZ$ is a closed immersion, then the induced map
$$
\sigma_{*}\colon \Mor(X,\WWW) \to \Mor(X,\ZZZ), \ \phi\mapsto \sigma\circ\phi,
$$
is a closed immersion of ind-varieties.
\item
If $\mu\colon Y \to X$ is a dominant morphism, then the induced map
$$
\mu^*\colon\Mor(X,\ZZZ) \into \Mor(Y,\ZZZ), \ \phi \mapsto \phi\circ\mu
$$
is a closed immersion of ind-varieties.\idx{closed immersion}
\ee
\end{lemma}

\begin{proof}
By Lemma~\ref{tautological.lem}(\ref{composition2}) $\sigma_{*}$ and $\mu^*$ are both ind-morphisms.

(1) The claim is true if $\WWW, \ZZZ$ are affine varieties, by Lemma~\ref{Mor-for-closed-immersions.lem}. Let $\ZZZ = \bigcup_{k}\ZZZ_{k}$
be an admissible filtration, and let $\WWW=\bigcup_{k}\WWW_{k}$ be the induced filtration, i.e. $\sigma(\WWW_{k}) = \ZZZ_{k} \cap\sigma(\WWW)$. For any $k$ we have a commutative diagram
$$
\begin{CD}
\Mor(X,\WWW_{k+1}) @>{\sigma_{*}}>> \Mor(X, \ZZZ_{k+1}) \\
@AA{\subseteq}A    @AA{\subseteq}A \\
\Mor(X,\WWW_{k}) @>{\sigma_{*}}>> \Mor(X, \ZZZ_{k}) \\
\end{CD}
$$
where all maps are closed immersions of ind-varieties. By construction we have
$\sigma_{*}(\Mor(X,\WWW_{k})) = \sigma_{*}(\Mor(X,\WWW))\cap 
\Mor(X,\ZZZ_k)$, and thus $\sigma_{*}$ is a closed immersion.

\ps
(2)
Let us first assume that $\ZZZ$ is an affine variety $Z$. 
Using a closed embedding $Z \subseteq V$ into a finite-dimensional $\kk$-vector space $V$ we obtain the following commutative diagram of ind-morphisms:
$$
\begin{CD}
\Mor(X,Z) @>{\subseteq}>> \Mor(X,V) @= \OOO(X)\otimes V\\
@VV{\mu^{*}}V @VV{\mu^{*}}V @VV{\mu^{*}\otimes\id}V \\
\Mor(Y,Z) @>{\subseteq}>> \Mor(Y,V) @= \OOO(Y)\otimes V\\
\end{CD}
$$
Since the induced map $\mu^{*}\otimes\id \colon \OOO(X)\otimes V \into \OOO(Y)\otimes V$ is an injective linear map, it is a closed immersion, and the claim follows in this case.

In general, we use an admissible filtration $\ZZZ =\bigcup_{k}\ZZZ_{k}$. Since $\mu\colon X \to Y$ is dominant, we clearly have
$\mu^{*}(\Mor(Y,\ZZZ_{k})) = \mu^{*}(\Mor(Y,\ZZZ))\cap\Mor(X,\ZZZ_{k})$ which implies that $\mu^{*}$ is a closed immersion of ind-varieties.
\end{proof}

The following result shows that the ind-semigroup\idx{ind-semigroup} $\End(X):=\Mor(X,X)$\idx{$\End(X)$} of all endomorphisms of $X$ determines $X$ up to isomorphisms.

\begin{proposition}\label{endX-gives-X.prop}
Let $X$ and $Y$ be affine varieties, and assume that there is an isomorphism $\psi\colon \End(X) \simto \End(Y)$ as ind-semigroups. Then $X \simeq Y$ as varieties. More precisely, the isomorphism $\psi$ is induced by a uniquely determined isomorphism $\phi\colon X \simto Y$.
\end{proposition}
\begin{proof} For $x\in X$ denote by $\gamma_{x} \in \End(X)$ the constant map with value $x$. Then the map $\iota_{X}\colon  X \to \End(X)$, $x\mapsto \gamma_{x}$, is a closed immersion. In fact, it is a morphism, and there is a retraction given by the morphism $\eval_{x_{0}} \colon \End(X) \to X$, $\phi \mapsto \phi(x_{0})$.

Now we remark that the closed subset $\iota_{X}(X)\subseteq \End(X)$ of constant maps is characterized by $\iota_{X}(X)=\{\phi\in\End(X) \mid \phi\circ\psi=\phi \text{ for all } \psi\in\End(X)\}$. This implies that the every isomorphism of ind-semigroups $\tau\colon \End(X) \simto \End(Y)$ defines a bijective morphism $\tau|_{\iota_{X}(X)}\colon \iota_{X}(X) \to \iota_{Y}(Y)$. The claim follows since the inverse map is given by $\tau^{-1}|_{\iota_{Y}(Y)}$.
\end{proof}
\begin{remark} 
A stronger result can be found in \cite{AnKr2014Varieties-characte}. 
\end{remark}

\ps
\subsection{Vector fields and endomorphisms}\label{Endo-VF.subsec}
As usual, a {\it vector field $\delta$\/}\idx{vector field} on an affine variety $X$ is a map $x \mapsto \delta_{x}\in T_{x}X$ with the property that for every $f \in \OOO(X)$ the function 
$$
\delta f \colon X \to \kk, \  x \mapsto \delta_{x}f
$$ 
is regular on $X$. It is easy to see that this gives a canonical identification of the vector fields $\VEC(X)$ on $X$ and the ($\kk$-linear) {\it derivations}\idx{derivation} $\Der_{\kk}(\OOO(X))$\idx{$\Der_{\kk}(\OOO(X))$} of $\OOO(X)$. Moreover, $\VEC(X)$\idx{$\VEC(X)$} is a module over $\OOO(X)$ by setting $(f\delta)_{x}:=f(x) \cdot\delta_{x}$. If $Y \subseteq X$ is a closed subvariety, we define the submodule of vector fields ``parallel'' to $Y$ by
$$
\VEC_{Y}(X):=\{\delta\in\VEC(X) \mid \delta_{y} \in T_{y}Y \text{ for all }y \in Y\} \subseteq \VEC(X).
$$
The following result is well known.\idx{$\VEC_{Y}(X)$}

\begin{proposition}  
\be
\item
$\VEC(X)$ is a finitely generated $\OOO(X)$-module.
\item
Let $f \in \OOO(X)$ be a nonzero element. The restriction map $\delta\mapsto \delta|_{X_{f}}$ 
induces an isomorphism $\OOO(X_{f})\otimes_{\OOO(X)}\VEC(X) \simto \VEC(X_{f})$.
\item
If $X$ is irreducible, then $\dim_{\CC(X)}\CC(X)\otimes_{\OOO(X)}\VEC(X) = \dim X$.
\item \label{closed-subvarieties-and-VF}
If $Y \subseteq X$ is a closed subvariety, then the linear map $\VEC_{Y}(X) \to \VEC(Y)$, $\delta\mapsto \delta|_{Y}$, is a homomorphism of $\OOO(X)$-modules. If $X$ is an affine space  $\AA^{n}$, then this map is surjective.
\ee
\end{proposition}\label{vector-fields.prop}

\begin{proof}
We set $R:=\OOO(X)$, so that $\Der_{\kk}(R)=\VEC(X)$, and we will only use that $R$ is a finitely generated $\kk$-algebra.
\ps
(1)
If $R = \kk [r_1, \ldots, r_n]$, then the map
\[ 
\Der_{\kk}(R) \to R^n, \quad \delta \mapsto (\delta(r_1), \ldots, \delta(r_n))
\]
is an injective homomorphism of $R$-modules.
\ps
(2)
Let $S \subseteq R$ a multiplicatively closed subset with $0\notin S$. Then we have a canonical isomorphism $R_{S}\otimes_{R}\Der_{\kk}(R) \simto \Der_{\kk}(R_{S})$.
\ps
(3) 
If $R$ is an integral domain with field of fraction $K$, then, by the previous statement, we get an isomorphism $K \otimes_{R}\Der_{\kk}(R) \simeq \Der_{\kk}(K)$. Now the claim follows from the following two well-known assertions.
\be
\item[(a)]  
If $L/K$ is an algebraic field extension containing $\kk$, then
\[  
\Der_{\kk} (L) \simeq  L \otimes_K \Der_{\kk}(K).
\]
\item[(b)] 
If $x_1, \ldots,x_n$ are algebraically independent over $\kk$, then 
\[ 
\Der_{\kk} ( \kk (x_1, \ldots, x_n ) ) = \bigoplus_{i=1}^n \kk (x_1, \ldots, x_n )  \partial_{x_i}.
\]
\ee
Indeed, the last two points imply that if $L$ is finitely generated extension of $\kk$, then the dimension of $\Der_{\kk}(L)$ over $\kk$ is equal to the transcendence degree of $L$ over $\kk$.
\ps
(4) 
The first assertion is obvious and the second follows from the fact that $\VEC(\AA^n)$ is a free $\OOO (\AA^n)$-module:
\[ 
\Der_{\kk} ( \kk [x_1, \ldots, x_n ] ) = \bigoplus_{i=1}^n \kk [x_1, \ldots, x_n ]  \partial_{x_i}. \qedhere 
\]
\end{proof}

\begin{corollary}\label{nonzero-VF.cor}
For every affine variety $X$ there exists a
vector field $\delta$ which is nonzero on a dense open set.
In particular, the map $\OOO(X) \to \VEC(X)$, $f\mapsto f\delta$, is injective.
\end{corollary}

\begin{proof} 
If $\delta \in \VEC(X)$ is a vector field, then the set $Z(\delta):=\{x\in X \mid \delta_{x}=0\} \subseteq X$ is closed, and $Z(\delta)=X$ if and only if $\delta=0$. Thus the first statement is clear in case $X$ is irreducible. 

Let $X = \bigcup_{i=1}^{m}X_{i}$ be the decomposition into irreducible components. There exist regular functions $f_{i}\in\OOO(X)$, $i=1,\ldots,m$, such that (a) $f_{i}|_{X_{i}}=0$, and (b) $f_{i}|_{X_{j}}\neq 0$ for $j\neq i$. Setting $f:=f_{1}f_{2}\cdots f_{m}$ we see that $X_{f}=\bigcup_{i} (X_{i})_{f}$ where the open sets $U_{i}:=(X_{i})_{f} \subseteq X_{i}$ are pairwise disjoint and nonempty. It follows that 
$$
\VEC(X_{f}) =  \bigoplus_{i=1}^{m}\VEC(U_{i}).
$$
Since $\VEC(X_{f}) \simeq \OOO(X_{f})\otimes_{\OOO(X)}\VEC(X)$, by 
Proposition~\ref{vector-fields.prop}(b) above, we can find a vector field $\delta \in \VEC(X)$ such that $\delta|_{U_{i}}\neq 0$ for all $i$, proving the first claim.

Now assume that $f\delta=0$ for some $f \in \OOO(X)$. Then $f(x)\delta_{x}=0$ for all $x \in X$. Hence $f=0$ because $\delta_{x}$ is nonzero on a dense open set. 
\end{proof}

\begin{remark}\label{Siebert.rem}
The following two important results are due to \name{Siebert} \cite{Si1996Lie-algebras-of-de}.
Let $X,Y$ be affine varieties.
\be
{\it
\item \cite[Proposition~1]{Si1996Lie-algebras-of-de}
The Lie algebra $\VEC(X)$ is simple if and only if $X$ is smooth.
\item \cite[Corollary~3]{Si1996Lie-algebras-of-de} 
If $\VEC(X)$ is isomorphic to $\VEC(Y)$ as Lie algebras, and if $X,Y$ are both normal, then $X \simeq Y$.
}
\ee
\end{remark}

We have seen in Lemma~\ref{ind-var-morphisms.lem} that $\End(X)$ is an ind-variety in a natural way.
For any $x\in X$ we have a morphism $\mu_{x}\colon\End(X) \to X$, $\phi\mapsto \phi(x)$, with differential $d\mu_{x}\colon T_{\id}\End(X) \to T_{x}X$. Thus, for any
$H\in T_{\id}\End(X)$,
we obtain an ``abstract'' vector field $\xi_{H}$ defined by $\xi_{H}(x):=d\mu_{x}(H)$.

\begin{proposition}  \label{End(X)-and-Vec(X).prop}
For every $H \in T_{\id}\End(X)$, 
$\xi_{H}$ is a vector field on $X$, and 
the linear map  $\xi\colon T_{\id}\End(X)\into\VEC(X)=\Der_{\kk}(\OOO(X))$, $H\mapsto \xi_{H}$, is an inclusion. If $X$ is a $\kk$-vector space, then $\xi$ is an isomorphism.
\end{proposition}
\begin{proof} 
(a) In case $X = \An$ we have $\End(\An) =  \kk[x_{1},\ldots,x_{n}]^{n}$, hence $T_{\id}\End(\An) =  \kk[x_{1},\ldots,x_{n}]^{n}$. If $H = (h_{1},\ldots,h_{n})\in T_{\id}\End(\An)$, then $\xi_{H}=\sum_{i=1}^{n}h_{i}\dxi$. In fact,  $\mu_{a}(\id+\eps H) = a + \eps H(a)$ hence $\xi_{H}(a)=H(a)$, i.e. $\xi_{H}= \sum_{i=1}^{n}h_{i}\dxi$.
\ps
(b) For the general case we choose a closed immersion $X\subseteq\Cn$ so that $\OOO(X) = \kk[x_{1},\ldots,x_{n}]/I(X)$. This defines a closed immersion of ind-varieties $\End(X) \subseteq \Mor(X,\Cn) = \OOO(X)^{n}$, hence an inclusion $T_{\id}\End(X) \subseteq \OOO(X)^{n}$. By definition, $(f_{1},\ldots,f_{n})\in\End(X)$ if and only if $F(f_{1},\ldots,f_{n})=0$ for all $F\in I(X)$. Therefore, we have for $H=(h_{1},\ldots,h_{n})\in T_{\id}(\End(X))$
$$
F(\bar x_{1}+\eps h_{1},\ldots,\bar x_{n}+\eps h_{n})=0 \text{ for all } F\in I(X),
$$
and so
$$
\sum_{i=1}^{n}h_{i}\frac{\partial F}{\partial x_{i}}(\bar x_{1},\ldots,\bar x_{n}) = 0 \text{ for all } F\in I(X).
$$
The latter means that $\sum_{i=1}^{n}h_{i}\frac{\partial}{\partial x_{i}}\colon \Cxn\to\OOO(X)$ induces a derivation $\xi_{H}$ of $\OOO(X)$ by $\xi_{H} \bar x_{i} = h_{i}$, i.e. a vector field given by $\xi_{H}(a)=(h_{1}(a),\ldots,h_{n}(a))\in T_{a}X \subseteq \Cn$. It is clear now that the linear map $H \mapsto\xi_{H}$ is injective.
\end{proof}
It is not true that $\xi\colon T_{\id}\End(X) \to \VEC(X)$ is always an isomorphism. In fact,  if $\End(X)$ is finite-dimensional (see Section~\ref{Small-End.subsec} for examples) and if $X$ has dimension at least one,
then $\xi$ cannot be an isomorphism,  because $\VEC(X)$ is a $\kk$-vector space of infinite dimension, by Proposition~\ref{vector-fields.prop}.

\begin{question} \label{Is-the-tangent-space-at-the-identity-of-End(X)-a-Lie-algebra.question}
Is it true that the image of $T_{\id}\End(X)$ in $\VEC(X)$ is a Lie subalgebra?
\end{question}

\ps
\subsection{Families of automorphisms}\label{fam-auto.subsec}
The study of {\it families of automorphisms}  plays a fundamental role in our paper. We give here some basic facts. 

\begin{definition}\label{family of automorphisms.def}
Let $X$ be a variety and $\YYY$ an ind-variety. A {\it family of automorphisms of $X$ parametrized by $\YYY$}\idx{family of automorphisms} is a $\YYY$-automorphism of $X \times \YYY$, i.e. an automorphism $\Phi$ of $X \times \YYY$ such that the projection $\pr \colon X \times \YYY \to \YYY$ is invariant: $\pr_{\YYY} \circ \Phi = \pr_{\YYY}$). We use the notation $\Phi = (\Phi_{y})_{y\in \YYY}$ where $\Phi_{y}$ is the induced automorphism of the fiber $\pr_{\YYY}^{-1}(y) = X\times\{y\}$ which we identify with $X$. In this way, the family $\Phi$ can be regarded as a map $\Phi\colon\YYY \to \Aut(X)$.

Similarly, for an algebraic group $G$, a {\it family of $G$-actions on $X$ parametrized by $\YYY$}\idx{family of $G$-actions} is a $G$-action $\Phi$ on $X\times \YYY$ such that the projection $\pr\colon X\times \YYY \to \YYY$ is $G$-invariant. Again we use the notation $\Phi = (\Phi_{y})_{y\in \YYY}$ where $\Phi_{y}$ is the $G$-action on the fiber $X\times\{y\}$ identified with $X$.
\end{definition}

In the definition of a family of automorphism we assumed that $\Phi$ itself is an automorphism. However, this is not necessary as the following proposition shows. Note that in the case where $X$ is affine, this proposition is a direct consequence of the statement claiming that any injective endomorphism of an affine variety is an isomorphism.

\begin{proposition}\label{fam-end-aut.prop}
Let $X$ and $Y$ be varieties where $X$ is irreducible, and let $\Phi=(\Phi_{y})_{y\in Y}$ be a family of endomorphisms of $X$. If every $\Phi_{y}$ is an automorphism, then so is $\Phi$.
\end{proposition}

\begin{proof}
We can clearly assume that $Y$ is irreducible. 

The assumption implies that the {\it tangent maps}\idx{tangent map}
\[ 
d\Phi_{(x,y)} \colon T_{(x,y)}(X\times Y)\to T_{\Phi(x,y)}(X\times Y)
\] 
are isomorphisms for all $(x,y)\in X\times Y$, because $d\Phi_{(x,y)} = ((d\Phi_{y})_{x},\id)$. Thus $\Phi$ induces an isomorphism $\Phi\colon (X\times Y)_{\text{smooth}} \simeq (X\times Y)_{\text{smooth}}$. In particular, $\Phi$ is birational.

If $\eta_{Y}\colon \tilde Y \to Y$ is the normalization of $Y$, then $\Phi$ induces, by base change, a morphism $\Phi'\colon X \times \tilde Y \to X \times \tilde Y$ over $\tilde Y$ which is also an isomorphism on the fibers $X \times \{\tilde y\}$.

If $\eta_{X}\colon \tilde X \to X$ is the normalization of $X$, then $\eta_{X} \times \id\colon \tilde X \times \tilde Y \to X \times \tilde Y$ is the normalization of $X\times \tilde Y$ and so $\Phi'$ induces a morphism $\tilde\Phi\colon \tilde X \times \tilde Y \to  \tilde X \times \tilde Y$ over $\tilde Y$, and $\tilde\Phi$ has again the property that it induces isomorphisms on the fibers $\tilde X \times \{\tilde y\}$. 
$$
\begin{tikzcd}
\tilde X \times \tilde Y 
\arrow[rr,"\tilde\Phi"]
\arrow[rd,"\pr_{\tilde Y}"']
\arrow[dd,"\eta_X\times\id"']
&&
\tilde X \times \tilde Y
\arrow[dd,"\eta_X\times\id"]
\arrow[ld,"\pr_{\tilde Y}"]
\\
& \tilde Y 
\arrow[dd,"\id" near start]
\\
X \times \tilde Y 
\arrow[rr,"\Phi'" near start,crossing over] 
\arrow[rd,"\pr_{\tilde Y}"']
\arrow[dd,"\id\times\eta_Y"']
&& X \times \tilde Y
\arrow[ld,"\pr_{\tilde Y}"]
\arrow[dd,"\id\times\eta_Y"]
\\
& \tilde Y 
\arrow[dd,"\eta_Y" near start] 
\\
X \times Y 
\arrow[rr,"\Phi" near start,crossing over]
\arrow[rd,"\pr_{Y}"']
&& X \times Y
\arrow[ld,"\pr_{Y}"]
\\
& Y &
\\
\end{tikzcd}
$$
Since $\tilde X \times \tilde Y$ is normal and $\tilde\Phi$ bijective, \name{Zariski}'s Main Theorem in its original form (see \cite[Chap.~III, \S9]{Mu1999The-red-book-of-va}) says that $\tilde\Phi$ is an isomorphism.

Next we claim that $\Phi$ is finite. For this we look at the square
$$
\begin{CD}
\tilde X \times \tilde Y @>{\tilde\Phi}>\simeq> \tilde X \times \tilde Y \\
@VV{\eta}V  @VV{\eta}V \\
X\times Y @>{\Phi}>> X \times Y
\end{CD}
$$
where the vertical map $\eta$ is the normalization.  Now \name{Zariski}'s Main Theorem in \name{Grothendieck}'s form (see \cite[Chap.~III, \S9]{Mu1999The-red-book-of-va}) says that there is an open immersion $\iota\colon X \times Y \into Z$ and a finite surjective morphism $\rho\colon Z \to X\times Y$ such that $\Phi=\rho\circ\iota$.
$$
\begin{tikzcd}
& \tilde X \times \tilde Y 
\arrow[r,"\tilde\Phi","\simeq"'] 
\arrow[ldd,"\eta'"'] 
\arrow[d,"\eta"]
& \tilde X \times \tilde Y
\arrow[d,"\eta"]\\
& X\times Y 
\arrow[r,"\Phi"] 
\arrow[ld,hook',"\iota"']
& X\times Y\\
Z \arrow[urr,"\rho"] \\
\end{tikzcd}
$$
Since $\Phi\circ\eta= \eta \circ \tilde \Phi$ is the normalization, this map factors through $\rho$, i.e. there exists a morphism $\eta' \colon \tilde X \times \tilde Y \to Z$ such that $ \Phi \circ \eta= \rho \circ \eta'$. Hence, $\eta'= \iota \circ \eta \colon \tilde X \times \tilde Y \to Z$ is finite and thus surjective, and so $X \times Y = Z$.

Now the claim follows from the next lemma. In fact, since the tangent maps are bijective, the fibers $\Phi^{-1}(x,y)$ are reduced.
\end{proof}

\begin{lemma} 
Let $\phi\colon U \to Z$ be a bijective finite morphism between irreducible varieties. If $\phi$ has reduced fibers, then $\phi$ is an isomorphism.
\end{lemma}
\begin{proof}
We can assume that $Z$ and hence $U$ are both affine. Consider the short exact sequence of finitely generated $\OOO(Z)$-modules
$$
\begin{CD}
0 @>>> \OOO(Z) @>\subseteq>> \OOO(U) @>>> M:=\OOO(U)/\OOO(Z) @>>> 0
\end{CD}
$$
If $\mm\subseteq\OOO(Z)$ is a maximal ideal, then $\mm\OOO(U) \subseteq \OOO(U)$ is a maximal ideal, by assumption, and so $M / \mm M = \OOO(U)/(\mm\OOO(U)+\OOO(Z))=(0)$. Hence $M=(0)$.
\end{proof}

\ps
\section{Commutative \texorpdfstring{$\kk$}{k}-Algebras and Invertible Elements}  \label{invertible-elements.sec}
\subsection{The embedding of \texorpdfstring{$R^{*}$}{R*} into \texorpdfstring{$R$}{R}}\label{R*-embedding.subsec}

Let $R$ be a finitely generated commutative $\kk$-algebra. 
We have seen in Proposition~\ref{G(R)-as-ind-group.prop} that the group $\GL_{1}(R)\simeq R^{*}$ of invertible elements of $R$ has a ``natural'' structure of an affine ind-group. This structure is obtained from the closed immersion $\GL_{1}(R) \into R \times R$, $r \mapsto (r,r^{-1})$. In particular, the first projection induces an injective morphism $p\colon \GL_{1}(R) \to R$ whose image is $R^{*} \subset R$. \idx{$\R$@$R^{*}$}\idx{$\GL_{1}(R)$}

We will now show that this image is locally closed and that the bijective morphism $p'\colon \GL_{1}(R) \to R^{*}$ is an isomorphism of ind-varieties. 

\begin{theorem} \label{R*-locally-closed-in-R.thm}
For a finitely generated commutative $\kk$-algebra $R$, the group $R^*$ of invertible elements is locally closed in $R$. More precisely, if $\NZ(R)$ denotes the set of nonzero divisors of $R$, then  $R^*$ is closed in $\NZ(R)$, and $\NZ (R)$ is open in $R$. In addition, the inverse $\iota\colon R^* \to R^*$, $r \mapsto r^{-1}$ is an isomorphism of ind-varieties.
In particular, $p'\colon \GL_{1}(R) \to R^{*}$ is an isomorphism of ind-groups.
\end{theorem}\idx{$\NZ(R)$}

In the following we write $R = \kk  [x_1,\ldots,x_n] / (q_1,\ldots, q_m)$ and set $R_{k}$ to be the image of 
$\kk[x_{1},\ldots,x_{n}]_{\leq k}$. Then $R = \bigcup_{k} R_{k}$ is a filtration by finite-dimensional 
subspaces, and one has $R_{k}R_{m}\subseteq R_{k+m}$.

\begin{lemma} \label{bound-of-inverse.lem}
For any $k\geq 1$ there is an $e=e(k)$ such that the following holds:
If $r\in R_{k}$ is invertible, then $r^{-1}\in R_{e}$.
\end{lemma} 

For the proof we will use the following result given in \cite[p. 49]{He1926Die-Frage-der-endl} (cf. \cite[p. 312]{MaMe1982The-complexity-of-} and \cite[\S 57]{Se1974Constructions-in-a}).

\begin{proposition} \label{degree-bound-for-ideal-membership.prop}
If $p,q_0,\ldots,q_m \in \kk [x_1,\ldots, x_n ]$ are of degree $\leq k$ and if $p \in (q_0,\ldots, q_m) $, then there exist $p_{0}, \ldots, p_m \in \kk [x_1,\ldots, x_n]$ such that
$$
p = \sum_{i=0}^{m} p_i \, q_i\quad\text{and}\quad \deg  p_i \leq k+ \left( (m+1)k \right)^{2^n} 
\text{ for all } i.
$$
\end{proposition}

\begin{proof}[Proof of Lemma~\ref{bound-of-inverse.lem}]
Let $q\in\kk[x_{1},\ldots,x_{n}]$ be a lift of $r\in R_{k}$ such that $\deg q \leq k$. If $r \in R$ is invertible, then 
$1 \in (q,q_{1},\ldots,q_{m})$. Thus, by Proposition~\ref{degree-bound-for-ideal-membership.prop} above, 
there exist $p,p_{1},\ldots,p_{m}$ of degree $\leq e:= k+ \left( (m+1)k \right)^{2^n}$ such that 
$1 = pq + \sum_{i\geq 1}p_{i}q_{i}$. Hence, $r^{-1}=\bar p \in R_{e}$.
\end{proof}

\begin{proof}[Proof of Theorem~\ref{R*-locally-closed-in-R.thm}]
It is clear that $\NZ(R)$ is open in $R$, because the set of zero divisors of $R$ is the union of the (finitely many) associated primes which is a closed subset.
\ps
Next we prove that $R^*$ is closed in $\NZ (R)$, i.e. that  $R^*_{k}:=R^{*}\cap R_{k}$ is closed in $\NZ(R)_{k}:=\NZ(R)\cap R_{k}$. Clearly, $r\in R_{k}$ is invertible if and only if the multiplication map $\lambda_r \colon R \to R$, $s \mapsto rs$, is bijective. This is equivalent to the condition that $1$ belongs to the image of $\lambda_{r}$. By Lemma~\ref{bound-of-inverse.lem} this means that  $1 \in \lambda_{r}(R_{e})$, $e = e(k)$. Denoting by $\LLL'(R_{e},R_{ke})$ the set of injective linear maps, we see from the following Lemma~\ref{linear-algebra.lem} that the set 
$\{h \in \LLL'(R_{e},R_{ke}) \mid 1\in h(R_{e})\}$ is closed in $\LLL'(R_{e},R_{ke})$. It follows that the set of invertible elements in $R_{k}$ is closed in the set of nonzero divisors of $R_{k}$.
\ps
It remains to see that the inverse map is an ind-morphism, i.e. that the maps $R_k \to \R_e$, $r \mapsto r^{-1}$, are morphisms. This is a direct consequence of Lemma~\ref{division-biss.lem} in the following section, applied to the multiplication map $\mu \colon R_k \times R_e \to R_{ke}$, $(r,s) \mapsto rs$.
\end{proof}

\begin{lemma} \label{linear-algebra.lem}
Let $V,W$ be finite-dimensional $\kk$-vector spaces, and let $w_{0} \in W$. Denote by $\LLL(V,W)$\idx{$\LLL(V,W)$} the linear maps from $V$ to $W$, and by $\LLL'(V,W)\subseteq \LLL(V,W)$ the subset of injective maps. Then the following holds.
\be
\item The subset $\LLL'(V,W) \subseteq \LLL (V,W)$ is open;
\item The subset $F := \{ h \in \LLL'(V,W) \mid w_{0}\in h(V) \}$ is closed in $\LLL'(V,W)$.
\ee
\end{lemma}

\begin{proof}
(1) This is well-known, because the set of maps of rank $\leq k$ is closed for each $k$.
\ps
(2) We can assume that $n:=\dim V < m:=\dim W$. Choosing bases of $V$ and $W$, the injective morphisms are given by the $m\times n$-matrices $A$ of rank $n$. The condition $w_{0} \in A(V)$ means that all $(n+1) \times (n+1)$ minors of the extended matrix $(A|w_{0})$ are zero, hence the claim.
\end{proof}

\begin{remark}
If we only assume that the algebra $R$ has countable dimension, then it no longer true that $R^*$ is locally closed in $R$. Set $R= \kk [x, (x+n)^{-1}, \; n \in \ZZ ]$ and consider the affine line $L:= \{ x+a, \; a \in \kk \} \subseteq R$. Then, $R^* \cap L = \{ x+a, \; a \in \ZZ \}$ is not locally closed in $L$, proving that $R^*$ is not locally closed in $R$.
\end{remark}

\ps
\subsection{Division is a morphism}\label{division.subsec}
The following ``Division Lemma''
was used in the proof of Theorem~\ref{R*-locally-closed-in-R.thm} above. 
We will need it again in Section~\ref{normalization.sec}, in the proof of Proposition~\ref{normalization.prop}.

\begin{lemma} \label{division-biss.lem}
Let $U$, $V$ and $W$ be finite-dimensional $\kk$-vector spaces, and let $\beta\colon U \times V \to W$ be a $\kk$-bilinear map. Let $X$ be a variety, and let $\mu\colon X \to U$, $\rho \colon X \to W$ be morphisms with the property that for all $x \in X$ there exists a unique element $\nu(x) \in V$ such that $\beta (\mu(x), \nu(x) ) = \rho(x)$.  Then the map $\nu \colon X \to V$ is a morphism.
\end{lemma}

\begin{proof}
Choose bases $u_1, \ldots,u_n$ of $U$, $v_1, \ldots, v_m$ of $V$, and $w_1, \ldots, w_\ell$ of $W$, and write $\mu(x) = \sum_i g_i(x) u_i$, $\nu(x) = \sum_j f_j(x) v_j$, and $\rho(x) = \sum_k h_k(x) w_k$. By assumption, the functions $g_{i}(x)$ and $h_{k}(x)$ are regular on $X$, and we have to show that the functions $f_{j}(x)$ are also regular.

The equation $\beta (\mu(x), \nu(x) ) = \rho(x)$ is equivalent to a linear system of equations
\[
 \tag{$S(x)$} 
 \sum_{j=1}^{m} a_{kj}(x) f_j(x) = h_k(x), \quad 1 \leq k \leq \ell, \label{system-S} 
\]
where the functions $a_{kj}(x)$ are linear in the $g_{i}(x)$, hence regular. For a fixed $x_0\in X$ the system $(S(x_0))$ has, by assumption, a unique solution $(f_1(x_0),\ldots,f_{m}(x_{0}))$. Therefore, the matrix $( a_{kj} (x_0) )_{kj}$ has an invertible $m \times m$ submatrix. Clearly, this matrix remains invertible for all $x$ in an open neighborhood of $x_0$. Now \name{Cramer}'s rule implies that the functions $f_j(x)$ are regular in this neighborhood, and the claim follows.
\end{proof}
For later use we give the following application of the Divison Lemma.

\begin{corollary}   \label{bilinear.cor}
Let $R$ be a finitely generated $\kk$-domain. Choose a filtration $R = \bigcup_{k}R_{k}$ by finite-dimensional subspaces. Let $X$ be a variety, and let $\mu\colon X \to R$, $\rho\colon X \to R$ be morphisms of ind-varieties such that the following holds: 
\be
\item[$(*)$] 
$\mu(x)\neq 0$ for all $x\in X$, and there is a $k\geq 1$ such that $\nu(x):=\frac{\rho(x)}{\mu(x)} \in R_{k}$ for all $x\in X$.
\ee
Then $\nu\colon X \to R$ is an ind-morphism.
\end{corollary}

\begin{proof}
The assumptions imply that there exist $m,\ell \geq 1$ such that the following holds:
{\it
\be
\item[(a)] $\mu(X) \subseteq R_{m}\setminus\{0\}$;
\item[(b)] $\rho(X) \subseteq R_{\ell}$ and $R_{m}\cdot R_{k} \subseteq R_{\ell}$;
\item[(c)] $\rho(X) \subseteq \mu(X) \cdot R_{k}$ (this follows from $(*)$).
\ee
}
Now set $U:=R_{m}$, $V:= R_{k}$ and $W:=R_{\ell}$, and take for $\beta\colon U \times V \to W$ the multiplication in $R$, using that $R_{m}\cdot R_{k} \subseteq R_{\ell}$, by (b).
Since $R$ is a domain it follows from (a) and (c)  that the property of Lemma~\ref{division-biss.lem} is satisfied for the map $\nu\colon X \to R_{k}$ defined by $\nu(x):=\frac{\rho(x)}{\mu(x)}$. Thus $\nu\colon X \to R$ is an ind-morphism, as claimed.
\end{proof}

\ps
\subsection{Principal open sets}\label{principal-open-sets.subsec}
If $X$ is an  affine variety and $f\in \OOO(X)$ a nonzero function, then one defines the \itind{principal open set} 
$X_{f} := \{x \in X \mid f(x)\neq 0\} \subseteq X$
which is again an affine variety. For any $\kk$-algebra $R$ of countable dimension we obtain an injective morphism 
$\iota \colon X_{f}(R) \to X(R)$ of ind-varieties (Proposition~\ref{X(R).prop}) whose image is 
$$
X'(R):=\{a \in X(R) \mid f(R)(a) \in R^{*}\} \subseteq X(R).
$$ 
This follows from the fact that the affine variety $X_{f}$ is defined by  
$$
X_{f} = \{(x,s)\mid s\cdot f(x)=1\} \subset X \times \kk.
$$
\begin{proposition}\label{special-open-set.prop}
The morphism $\iota\colon X_{f}(R) \to X(R)$ is a (locally closed) immersion, i.e. its image $X'(R):=\iota(X_{f}(R)) \subseteq X(R)$ is locally closed and $\iota$ induces an isomorphism  $X_{f}(R) \simto X'(R)$.
\end{proposition}

\begin{proof}
The function $f \colon X \to \kk=\Aone$ defines a morphism $f(R) \colon X(R) \to R$ of ind-varieties 
(Proposition~\ref{X(R).prop}). Hence the subset
$$
X'(R) = \{a \in X(R) \mid f(R)(a) \in R^{*}\} = f(R)^{-1}(R^{*})
$$
is locally closed, because $R^{*}\subset R$ is locally closed (Theorem~\ref{R*-locally-closed-in-R.thm}). 
By definition
$$
X_{f}(R) = \{(a,r)  \in X(R) \times R  \mid r \cdot f(R)(a) = 1\} \subset  X(R)\times R.
$$
Since the inverse  $R^{*}\to R^{*}$ is a morphism, by Theorem~\ref{R*-locally-closed-in-R.thm}, we see that the map
$g \colon X'(R) \to R$, $a \mapsto [ f(R) (a) ] ^{-1}$,  is also a morphism. Therefore, the map
$$
\kappa\colon X'(R) \to X(R) \times R, \quad a \mapsto (a,g(a)),
$$
is a morphism whose image is $X_{f}(R)$. Since $\kappa$ is the inverse of $\iota$ the claim follows.
\end{proof}

\begin{corollary}
The inclusion $\iota\colon \GL_{n}(R) \into \M_{n}(R)$ is a (locally closed) immersion of ind-varieties. In particular, the invertible $R$-matrices form a locally closed subset of $\M_{n}(R)$.
\end{corollary}

Here is another consequence of Theorem~\ref{R*-locally-closed-in-R.thm}.
\begin{corollary}\label{special-open-set.cor}
Let $X,Y$ be affine varieties, and let $Y_{f}\subseteq Y$ be a principal open set. Then $\Mor(X,Y_{f}) \subseteq \Mor(X,Y)$ is locally closed.
\end{corollary}
\begin{proof}
Consider the following map $\rho\colon\Mor(X,Y) \to \Mor(X,\kk)=\OOO(X)$, $\rho(\phi) := f\circ\phi = \phi^{*}(f)$. This is an ind-morphism (Remark~\ref{evaluation and composition.rem}), and $\Mor(X,Y_{f}) = \rho^{-1}(\OOO(X)^{*})$, hence the claim, by Theorem~\ref{R*-locally-closed-in-R.thm}.
\end{proof}

\begin{question} \label{does-Z-open-affine-in-Y-affine-imply-that-Mor(X,Z)-is-locally-closed-in-Mor(X,Y).ques}
Does this hold for an arbitrary open affine variety $Y' \subseteq Y$?
\end{question}

\ps
\subsection{The structure of the ind-group \texorpdfstring{$R^{*}$}{R*}}
For the coordinate ring $R=\OOO(X)$ of an irreducible affine variety $X$ it has been shown by
\name{Rosenlicht} in \cite[Lemma to Proposition~3]{Ro1957Some-rationality-q} that the group $R^{*}/\Cst$ is finitely generated and torsion free (cf. \cite[1.3~Proposition~3]{KnKrVu1989The-Picard-group-o}). The following result generalizes this. 

\begin{proposition}  \label{structure-of-R*.prop}
Let $R$ be a finitely generated commutative $\CC$-algebra, $\rr\subseteq R$ its nilradical, and let $d\geq 1$ be the number of connected components of the affine variety $\Spec ( R/ \rr)$. 
Then the following holds.
\be
\item
There is an isomorphism of ind-groups $(R^{*})^{\circ} \simeq (1+\rr)\times (\Cst)^{d}$. 
\item
The quotient $R^{*}/(R^{*})^{\circ}$ is a finitely generated free abelian group. In particular, $\dim R^{*}=\dim_{\CC}\rr + d$.
\item
The exponential map $\rr \to U:= 1 + \rr$ defines an isomorphism of ind-groups $( \rr, +) \simto (U, \times )$. In particular, the ind-group $(R^{*})^{\circ}$ is nested.
\ee
\end{proposition}\idx{$\R$@$\rr$}\idx{nilradical of $\R$}

\begin{proof}
(a) First assume that $R$ is reduced, i.e. $R = \OOO(X)$ for some affine variety $X$. Denote by $E \subseteq \OOO(X)$ the locally constant regular function. Clearly, $E$ is a product of $d$ copies of the field $\CC$. Let $X = \bigcup_{j}X_{j}$ be the decomposition into irreducible components. We have an injective homomorphism $\OOO(X) \into \prod_{j}\OOO(X_{j})$, and the projections $\OOO(X) \surto \OOO(X_{j})$ map $E$ surjectively onto $\CC$, the constant functions of $\OOO(X_{j})$. It follows that we get an injection $\OOO(X)^{*}/E^{*} \into \prod_{j}\OOO(X_{j})^{*}/\Cst$. 
In fact, if the image of $f$ in each $\OOO(X_{j})$ is constant, then $f$ is locally constant.
By \name{Rosenlicht}'s Theorem (cf. \cite[\S1, Proposition~1.3]{KnKrVu1989The-Picard-group-o}), the groups $\OOO(X_{j})/\Cst$ are finitely generated and torsion free, hence the  the claim follows in this case.
\ps
(b) In general, the nilradical $\rr$ is nilpotent, and so the inverse image $1+\rr$ of $1$ under the map $p\colon R \to R/\rr$ consists of invertible elements, hence $p^{-1}((R/\rr)^{*}) = R^{*}$. In particular, we have an exact sequence of ind-groups (see Proposition~\ref{Homs-of-G(R).prop}):
$$
\begin{CD}
1 @>>> 1+\rr  @>{\subseteq}>>  R^{*} @>p>> (R/\rr)^{*} @>>> 1
\end{CD}
$$
It follows that $1\to 1+\rr \to p^{-1}(E) \to E \to 1$ is also an exact sequence of ind-groups. We claim that this sequence splits, hence $p^{-1}(E)$ is connected and isomorphic to $(1+\rr) \times E^{*}$, and  $R^{*}/p^{-1}(E) \simto (R/\rr)^{*}/E$ which is a finitely generated free abelian group, by (a), proving (1) and (2).
\ps
In order  to prove the claim 
choose inverse images $\tilde e_{1},\ldots,\tilde e_{d} \in R$ of the standard $\kk$-basis of $E=\CC^{d} \subseteq R/\rr$, and set  $A:=\CC[\tilde e_{1},\ldots,\tilde e_{d}]\subseteq R$. Then $\rr_{A}:=\rr\cap A$ is the nilradical of $A$ and $A/\rr_{A} \simeq E$. It follows that $A$ is finite-dimensional, and so $A^{*} \subseteq R^{*}$ is an algebraic group. As above, we have an exact sequence $1 \to 1+\rr_{A} \to A^{*} \to E^{*} \to 1$ which clearly splits, because $1 + \rr_{A}$ is a unipotent group and $E^{*}$ is a torus.
\ps
Since $\rr^{m}=(0)$ for some $m\geq 1$, it is clear that the exponential map $\exp \colon \rr \to U$, $x \mapsto \exp (x) = \sum_{i \geq 0} \frac{x^i}{i!}$ and the logarithm $\log \colon U \to \rr$, $1+r \mapsto \log (1+r) = \sum_{i \geq 1} (-1)^{i-1}\frac{x^i}{i}$ are morphisms of ind-groups, and it is not difficult to see that each one is the inverse of the other, cf. Section~\ref{exponential-for-linear-algebraic-groups.subsec}).
\end{proof}

\ps
\subsection{The case  of non-commutative \texorpdfstring{$\kk$}{k}-algebras}\label{non-commutative-case.subsec}
If $R$ is an associative, but not necessarily commutative $\kk$-algebra of countable dimension, we also have a ``natural'' structure of an affine ind-group structure on the group $R^{*}$ of invertible elements, obtained in the following way (cf. Section~\ref{R*-embedding.subsec}).
The subset 
$$
A :=\{(r,s) \in R \times R \mid rs=1\} \subset R \times R
$$ 
is closed, and the first projection $\pr_{1}\colon R \times R \to R$ induces a bijection $A \simto R^{*}$ and thus endows $R^{*}$ with the structure of an ind-variety. It is easy to see that $R^{*}$ is an ind-group and that the injection $R^{*}\into R$ is a morphism of ind-semigroups.
(There will be more about this in Section~\ref{GeneralAlgebra.sec} where we study  ``general algebras'' $\R$.)

One could expect that the structure results for $R^{*}$ given in Proposition~\ref{structure-of-R*.prop} above carry over to the non-commutative case. But this is not true, and our example is based
on the counterexamples to the famous  \name{Burnside} Problem which we are going to recall first.

Let $F_{2}$ denote the free group in two generators, and let $P_{n}:=\langle z^{n}\mid z\in F_{2}\rangle$ be the subgroup generated by the $n$th powers of all elements from $F_{2}$. This is a normal subgroup, and the quotient group $B(n):=F_{2}/P_{n}$ is the so called \name{Burnside} group. By construction, every element of $B(n)$ has a finite order dividing $n$. It was shown by \name{S.V.~Ivanov} that  $B(n)$ is infinite for $n\geq 2^{48}$ \cite{Iv1992On-the-Burnside-pr}, and this bound was improved by \name{I.G.~Lysenok} who showed that $B(n)$ is infinite for $n\geq 8000$. 

Now we consider the group algebra $R(n):=\kk[B(n)]$ which is a quotient of the free associative
$\kk$-algebra in 2 generators. If we denote by $x,y$ the two generators of $B(n)$, then $R(n)$ is generated as a $\kk$-algebra by $x,y$, and $B(n) \subset R(n)^{*}$ is the set of monomials in $x,y$.
For every $z\in B(n)$ the commutative subalgebra $\kk[z] \subset R(n)$ is the group algebra of $\langle z \rangle$, hence it is isomorphic to a product of $m$ copies of $\kk$ where $m$ is the order of $z$. Therefore, $\kk[z]^{*} \subseteq R(n)^{*}$ is an $m$-dimensional torus, and in particular it is connected.

Now assume that $(R(n)^{*})^{\circ}$ is a nested ind-group. Then the subgroup generated by the two tori $k[x]^{*}$ and $k[y]^{*}$ is contained in a connected algebraic subgroup $G \subset (R(n)^{*})^{\circ}$. 
Since $G$ contains all monomials in $x,y$ we have $B(n)\subseteq G$. We claim that this implies that $B(n)$ is finite. In fact, the inclusion map $R(n)^{*} \into R(n)$ is an ind-morphism, hence $G$ is contained in an algebraic subset of $R(n)$. It follows that $B(n)$ belongs to a finite-dimensional linear subspace of $R(n)$ and thus must be finite, because $B(n)$ is a basis of $R(n)$. 

This establishes the following result.

\begin{proposition}  \label{A-case-where-(R*)-is-not-nested.prop}
Let $B(n) := B(2,n)$ be the \name{Burnside} group, and let $R(n):=\kk[B(n)]$ be its group algebra. If $n$ is large enough so that $B(n)$ is infinite (e.g. $n \geq 8000$), then the connected component $( R(n)^{*})^{\circ}$ does not have the structure of a nested ind-group.
\end{proposition}

We have seen above that every element from $B(n) \subset R(n)^{*}$ is semisimple and that  the group $R(n)^{*}$ contains a lot of tori. One might wonder whether all elements of $G$ are locally finite. The next proposition shows that this is not the case if $n$ is large enough.

\begin{proposition}\label{R*-not-locally-finite.prop}
We use the same notation as in Proposition~\ref{A-case-where-(R*)-is-not-nested.prop}.
\be
\item For any $\lambda \in\kk$ such that $\lambda^{n}\neq 1$ we have $(x- \lambda y)\in (R(n)^{*})^{\circ}$.
\item If $B(n)$ is infinite and $a \geq 2$ is an integer, then the element $(x+ay)\in R(n)$ is not locally finite.
\ee
\end{proposition}

\begin{proof}
(a)
The following holds in the algebra $\kk[z]:=\kk[Z]/(Z^{m}-1)$: {\it A polynomial $p(z) \in \kk[z]$ is invertible if and only if $p(\zeta)\neq 0$ for all $m$-th roots of unity $\zeta$.} In fact,
$\kk [z]$ is isomorphic to  $ \kk^m$, and the projections onto the factors are given by $p(z) \mapsto p(\zeta)$.
\ps
(b)
We have $(x- \lambda y) = x(1 - \lambda x^{-1}y)$, and $(1 - \lambda x^{-1}y) \in \kk[z]$ where $z:=x^{-1}y \in B(n)$. It follows from (a) that $(1 - \lambda z)$ is invertible if $\lambda$ is not a $n$-th root of unity. Moreover, the inverses of all $(1- \lambda z)$ belong to the finite-dimensional subspace $\kk[z] \subset R(n)$ which implies that the map 
$$
\mu\colon C:=\kk \setminus \{ \text{$n$th roots of unity} \} \to R(n)^{*}, \quad \lambda \mapsto (x - \lambda y),
$$ 
is a morphism. Since $\mu(0)=x \in (R(n)^{*})^{\circ}$ we find $\mu(C) \subset (R(n)^{*})^{\circ}$, proving (1).
\ps
(c)
If $a \geq 2$ is an integer, then $-a$ is not a $n$th root of unity, so that $(x+ ay)$ belongs to $(R(n)^{*})^{\circ}$ by (1). Denote by $B(n)_k$ the set of elements of $B(n)$ that might be expressed as monomials in $x,y$ of degree $k$. Since $(x+ay)^{k}$ is a linear combination of all monomials in $x,y$ of degree $k$ and since all coefficients are positive rational numbers there exist positive rational numbers $a_b$, $b \in B(n)_k$, such that
\[
(x+ay)^k = \sum_{b \in B(n)_k} a_b b.
\]
The infinite set $B(n)$ is the union of the finite sets $B(n)_k$, $k \geq 1$. This proves that there is no finite-dimensional subspace of $R(n)$ which contains all the powers $(x+ay)^{k}$, $k \geq 1$. Hence, the elements $(x+ay)$ are not locally finite.
\end{proof}

\ps
\subsection{Endomorphisms of commutative \texorpdfstring{$\kk$}{k}-algebras}\label{endo-k-algebra.subsec}
For a $\kk$-algebra $R$ we denote by $\End (R)$ the semigroup of $\kk$-algebra endomorphisms.

\begin{proposition}\label{TEnd(R)-into-Der(R).prop}
For every  finitely generated commutative $\kk$-algebra $R$ the set $\End (R)$ has a natural structure of an ind-semigroup.
Moreover,  there is a canonical embedding $\delta\colon T_{\id}\End (R) \into \Der_{\kk}(R)$.
\end{proposition}
\begin{proof}
(a) 
Assume that $R = \kk[x_{1},\ldots,x_{n}]/I$. Then 
$$
\End (R) = \{(r_{1},\ldots,r_{n})\in R^{n} \mid p(r_{1},\ldots,r_{n})=0 \text{ for all }p\in I\}\subseteq R^{n},
$$
and this is a closed subset of $R^{n}$. It is easy to see that the induced ind-variety structure has the property that families of algebra endomorphisms of $R$ parametrized by an ind-variety $\ZZZ$ correspond bijectively to morphisms $\ZZZ \to \End (R)$. Therefore, this structure does not depend on the presentation of $R$ as a quotient of a polynomial ring. In addition, the composition of homomorphism is a morphism, so that $\End (R)$ is an ind-semigroup.
\ps 
(b)
If $A = (a_{1},\ldots,a_{n}) \in T_{\id}\End (R) \subseteq R^{n}$, then 
$$
p(\bar x_{1}+\eps a_{1},\ldots,\bar x_{n}+\eps a_{n}) = 0 \text{ for all }p \in I,
$$
hence
$$
\sum_{i} a_{i}\frac{\partial p}{\partial x_{i}}(\bar x_{1},\ldots,\bar x_{n}) = 0 \text{ for all } p\in I.
$$
It follows that $\sum_{i} a_{i}\frac{\partial}{\partial x_{i}}\colon \kk[x_{1},\ldots,x_{n}] \to R$ induces a 
derivation $\delta_{A}\colon R \to R$ where $\delta_{A}(\bar x_{i}) = a_{i}$.
This shows that we obtain an
embedding $\delta\colon T_{\id}\End (R) \into \Der_{\kk}(R)$.
\end{proof}

\begin{remark}
This latter proposition will be generalized to a general algebra $\R$ in Section \ref{GeneralAlgebra.sec} (see Propositions~\ref{ind-var of algebra endomorphisms.prop} and \ref{Tangent-space-of-End(R)-and-Aut(R)-for-a-general-algebra-R.prop}).
\end{remark}

We can use this result to give a more algebraic description of the construction of the vector field $\xi_{H}$ associated to a tangent vector $H \in T_{\id}\End(X)$, see Proposition~\ref{End(X)-and-Vec(X).prop}. We have a canonical map $\tau\colon\End(X) \simto \End (\OOO(X))$, $\phi\mapsto \phi^{*}$, which is an anti-isomorphism of ind-semigroups. Now it is easy to see that the following diagram is commutative.
\[
\begin{CD}
T_{\id}\End(X) @>{\xi}>> \VEC(X) \\
@V{d_{\id}\tau}V{\simeq}V        @VV{\simeq}V  \\
T_{\id}\End (\OOO(X))  @>{\delta}>>  \Der_{\kk}(\OOO(X))
\end{CD}
\]
Therefore, if we identify $\End(X)$ with $\End (\OOO(X))$ via $\phi\mapsto\phi^{*}$, then the vector field $\xi_{H}$ corresponds to the derivation $\delta_{H}$. This also explains why $\xi$ is an anti-homomorphism of Lie algebras 
whereas $\delta$ is a homomorphism of Lie algebras.

\newpage
\part{AUTOMORPHISM GROUPS AND GROUP ACTIONS}

This part is devoted to the study of automorphism groups $\Aut(X)$ of affine varieties $X$. These groups are ind-groups in a natural way, and they are locally closed in $\End(X)$. We define  actions of ind-groups on affine varieties and representations of ind-groups. We will see that the Lie algebra of $\Aut(X)$ is naturally embedded into the Lie algebra $\VEC(X)$ of vector fields on $X$. We also show that the automorphism group $\Aut(\RRR)$ of a finitely generated general algebra $\RRR$ is an affine ind-group.  Finally, we give an example of a bijective homomorphism of connected ind-groups which is not an isomorphism.

\section{The Automorphism Group of an Affine Variety}\label{autgroup.sec}
\ps
\subsection{The ind-group of automorphisms}\label{autgroup.subsec}
In this section we show that for every affine variety $X$ the automorphism group\idx{automorphism group} $\Aut(X)$ has a natural structure of ind-variety.  We will also see that $\Aut(X)$\idx{$\Aut(X)$} is locally closed in $\End(X)$. Recall that a family of automorphisms of $X$ parametrized by an ind-variety $\YYY$ is an automorphism of $X \times \YYY$ over $\YYY$, see Definition~\ref{family of automorphisms.def}.

\begin{theorem} \label{AutX.thm}
Let $X$ be an affine variety.
\be
\item
There exists a  structure of an affine ind-group on $\Aut (X)$ such that families  of automorphisms of $X$ parametrized by an ind-variety $\YYY$ correspond to morphisms $\YYY \to \Aut (X)$ of ind-varieties.
\item\label{closed-embed}
The following map is a closed immersion:
\[ 
\iota\colon \Aut(X) \to \End(X)\times\End(X), \ \phi\mapsto(\phi,\phi^{-1}).
\]
\ee
\end{theorem}

\begin{proof}
(1)
Define the closed subset\idx{$\AAA (X)$}
\[ 
\AAA (X) :=
\{(\phi,\psi)\in \End(X) \times \End(X) \mid \phi\cdot \psi=\id=\psi\cdot \phi\} \subseteq \End(X) \times\End(X).
\]
which can be identified with $\Aut(X)$ via the first projection. It is clear that with this structure of a group the affine ind-variety $\AAA(X)$ becomes an ind-group which has the required universal property (see Lemma~\ref{ind-var-morphisms.lem}).
\ps
(2)
The previous construction shows that $\iota$ is a closed immersion with image $\AAA(X)$.
\end{proof}

Let us state and prove the following relative version of the theorem above.

\begin{lemma} \label{Aut_Y(X).lem}
For any morphism $\pi \colon X \to Y$ between affine varieties, the group
\[ \Aut_Y(X):= \{ \varphi \in \Aut (X) \mid \pi \circ \varphi = \pi \} \]
is closed in the ind-group $\Aut (X)$. In particular $\Aut_Y (X)$ is an ind-group and has the obvious universal property.
\end{lemma}

\begin{proof}
It follows from Lemma~\ref{tautological.lem}(\ref{composition2}) that the composition
\[
\Aut (X) \into \End (X) \to \Mor (X,Y), \  \phi \mapsto \pi \circ \phi ,
\]
is an ind-morphism. Since $\Aut_Y(X)$ is the preimage of $\{ \pi \}$ it is closed.
\end{proof}

If $\pi = \pr_Y \colon X \times Y \to Y$, then the families of automorphisms of $X$ parametrized by $Y$ forms the subgroup $\Aut_{Y}(X\times Y)$ of $\Aut(X \times Y)$. The theorem above tells us that there is a natural bijective map
$$
\Aut_{Y}(X\times Y) \simto \Mor(Y, \Aut(X)).
$$
\begin{proposition}\label{families-of-aut.prop}
The subgroup $\Aut_{Y}(X\times Y) \subset \Aut(X \times Y)$ is closed, and the natural map
$$
\Phi\colon \Aut_{Y}(X\times Y) \simto \Mor(Y, \Aut(X))
$$
is an isomorphism of ind-groups.
\end{proposition}
\begin{proof} 
The first statement is a special case of Lemma~\ref{Aut_Y(X).lem}.

By Lemma~\ref{mor-mor.lem} the canonical map
$\Mor(X \times Y,X) \to \Mor(Y,\End(X) )$
is an isomorphism of ind-varieties. In addition, there is an isomorphism $\End_{Y}(X\times Y) \simto \Mor(X\times Y,X)$ given by $\phi\mapsto \pr_{X}\circ \phi$, whose inverse is $\psi\mapsto (\psi,\id_{Y})$. Thus the canonical map
$\Psi\colon \End_{Y} (X\times Y) \to \Mor(Y,\End(X))$
is an isomorphism of ind-varieties.
From this we obtain a commutative diagram
$$
\begin{CD}
\End_{Y}(X\times Y)\times\End_{Y}(X\times Y) @>{\Psi}>{\simeq}>  \Mor(Y,\End(X)\times\End(X)) \\
@AA{\iota\colon\phi\mapsto (\phi,\phi^{-1})}A @AA{\Mor(Y,\iota)}A \\
\Aut_{Y}(X\times Y) @>{\Phi}>> \Mor(Y,\Aut(X))
\end{CD}
$$
where the vertical maps are closed immersions (Theorem~\ref{AutX.thm}(\ref{closed-embed})). Since $\Phi$  is bijective it is an isomorphism.
\end{proof}

\begin{example}
Let $X,Y$ be affine varieties. Assume that all morphisms $\phi\colon X \to Y$ are constant as well as all morphisms 
$\psi\colon Y \to X$. Then we have 
$$
\End(X\times Y) = \End(X) \times \End(Y)\ \text{ and  }\ \Aut(X\times Y) = \Aut(X) \times \Aut(Y).
$$
In fact, if $\Phi=(\phi_{1},\phi_{2}) \colon X\times Y \to X \times Y$ is an endomorphism, then $\phi_{1}\colon X\times Y \to X$ corresponds to a morphism $\tilde\phi_{1}\colon Y \to \End(X)$ of ind-varieties which is constant, because $y \mapsto \tilde\phi_{1}(x)(y) = \phi_{1}(x,y)$ is constant for every $x\in X$. Thus $\phi_{1}(x,y) = \phi_{1}(x)$, and similarly $\phi_{2}(x,y) = \phi_{2}(y)$.
\ps
An interesting example is the following (see \cite[Theorem~1.3]{LiReUrCharacterization-o}). Let $X$ be a torus, and let $C$ be an affine smooth curve such that $\Aut(C)$ is trivial and $\OOO(C)^{*}=\kst$. Then $\Aut(X \times C) = \Aut(X)$. 
\newline
(In fact, there are no non-constant morphisms $C \to X$, because $\OOO(X)$ is generated by invertible elements, and there are no non-constant morphisms $X \to C$, since otherwise $C$ is rational, hence isomorphic to $\Aone$ and thus $\Aut(C)$  is nontrivial.)
\end{example}

\begin{example}
Set $\Aoned:=\Aone\setminus\{0\}$.
For the group $\Aut(\Aone\times\Aoned)$ we have a split exact sequence
$$
\begin{CD}
1 @>>> \Aut_{\Aoned}(\Aone\times\Aoned) @>>> \Aut(\Aone\times\Aoned) @>{p}>> \Aut(\Aoned) @>>> 1.
\end{CD}
$$
The splitting is given by the obvious closed immersion $\Aut(\Aoned) \into \Aut(\Aone\times\Aoned)$. Moreover, $\Aut(\Aoned)  = \langle\tau\rangle \ltimes \kst$ where $\tau\colon x\mapsto x^{-1}$, and, by 
the proposition above,
$$
\Aut_{\Aoned}(\Aone\times\Aoned) \simeq \Mor(\Aoned,\Aut(\Aone)) \simeq \Aff(1)(\kk[t,t^{-1}])=\kk[t,t^{-1}]^{*}\ltimes\kk[t,t^{-1}]^{+},
$$
cf. Remark~\ref{G(R)-as-ind-group.rem}. 
In particular, $\Aut(\Aone\times\Aoned) = \Aut(\Aoned)\ltimes\Aut_{\Aoned}(\Aone\times\Aoned)$ is a semidirect product of closed ind-subgroups, and the connected component of $\Aut(\Aone\times\Aoned)$ has index 2.
\newline
(The only non-obvious fact is the existence of the ind-morphism $p$. This follows, because every automorphism of $\OOO(\Aone\times\Aoned) = \kk[s,t,t^{-1}]$ sends the invertible element $t$ to an invertible element which must be of the form $at$ or $at^{-1}$ for some $a \in\kst$, and thus it induces an automorphism of the subalgebra $\kk[t,t^{-1}]$.)
\end{example}

\begin{proposition} \label{Aut(X,Y).prop}
Let $X$ be an affine variety and $Y \subseteq X$ a closed subset. Then, the subgroup
\[
\Aut (X,Y) := \{ \phi \in \Aut (X) \mid \phi (Y) \subseteq Y \}
\]
is a closed subgroup of $\Aut (X)$, and the restriction map $\res_Y \colon \Aut (X,Y) \to \Aut (Y)$ is a homomorphism of ind-groups.
\end{proposition}

\begin{proof}
The evaluation maps $\eval_{y}\colon \Aut(X) \to Y$, $\phi \mapsto \phi(y)$, are morphisms by Lemma~\ref{tautological.lem}(\ref{eval}), and so 
$\Aut (X,Y) = \bigcap_{y \in Y} \eval_{y}^{-1}(Y)$ is closed in $\Aut (X)$. Now, the action map $\Aut (X,Y) \times Y \to Y$ is induced by the ind-morphism $\Aut(X) \times X \to X$, hence it is an ind-morphism, too, and so $\res_{Y}$ is a homomorphism of ind-groups.
\end{proof}

As an easy application we give a short proof of the the following result due to \name{Magid} \cite{Ma1978Separately-algebra}. 
\begin{proposition}
Assume that $\kk$ is uncountable, and let $X$ be an affine variety with a group structure such that 
left and  right multiplications with elements from $X$ are morphisms. Then $X$ is an algebraic group.
\end{proposition}
\begin{proof} First note that $X$ is smooth.
For $x\in X$ let $\lambda_{x},\rho_{x}\in\Aut(X)$ denote the left- and right-multiplication with $x$. Define
\[
G:=\{\phi\in\Aut(X) \mid \phi\circ\rho_{x}=\rho_{x}\circ\phi\text{ for all }x \in X\}.
\]
Then $G$ is a closed ind-subgroup of $\Aut(X)$ and thus acts faithfully on $X$. For $\phi\in G$ we have $\phi(yx) = \phi(y)x$, and so $\phi=\lambda_{\phi(e)}$ where $e\in X$ is the unit element of the group structure. It follows that the orbit map $\mu\colon G\to X$, $\phi\mapsto \phi(e)$, 
is a bijective ind-morphism. Hence $G$ is an algebraic group, by Lemma~\ref{bijective-morphisms.lem},
and $\mu$ is an isomorphism of varieties since $X$ is smooth. By construction, we have
$$
\mu(\phi\circ\psi) = \phi(\psi(e)) = \lambda_{\phi(e)}(\psi(e)) = \phi(e) \psi(e)=\mu(\phi)\mu(\psi)
$$ 
and so $\mu\colon G \to X$ is also an isomorphism of groups.
\end{proof}

\ps
\subsection{The embedding of \texorpdfstring{$\Aut(X)$}{Aut(X)} into \texorpdfstring{$\End(X)$}{End(X)}} \label{embedding.subsec}
For a finite-dimensional $\kk$-vector space $V$ the group $\GL(V)$ is open in $\LLL(V)$, the semigroup of linear endomorphisms. This implies that for any finite-dimensional $\CC$-algebra $A$ the group $\Aut_{\alg}(A)$ of $\CC$-algebra automorphisms is open in $\End_{\alg}(A)$, the semigroup of $\CC$-algebra endomorphisms, 
because $\End_{\alg}(A)$ is closed in $\LLL(A)$, and $\Aut_{\alg}(A) = \End_{\alg}(A) \cap \GL(A)$. 

In the case of $\Aut(X)$ where $X$ is an affine variety there are several problems. In the previous section we have given $\Aut(X)$ the structure of an ind-group using the affine ind-group $\AAA(X)\subset \End(X) \times \End(X)$ together with the ind-morphism  $p\colon \AAA(X) \to \End(X)$ induced by $\pr_{1}\colon\End(X) \times \End(X) \to \End(X)$ which defines a bijection $p'\colon \AAA(X) \to p(\AAA(X)) = \Aut(X)$.  
\idx{$\AAA (X)$}\idx{$\Aut(X)$}

At this point we have the following natural questions (cf. \cite[Remark 3.3.6, p. 132]{KaMi2005On-two-recent-view}).
\be
\item[$\bullet$]
{\it Is $\Aut(X) = p(\AAA(X)) \subset \End(X)$  a locally closed ind-subvariety?} 
\item[$\bullet$]
{\it And if so,  is $p'\colon \AAA(X) \to \Aut(X)$  an isomorphism of ind-varieties?}
\ee
We will answer affirmatively both questions in the following theorem. Recall that a morphism $\phi\colon X \to Y$ of varieties is {\it dominant\/}\idx{dominant morphism} if the image $\phi(X) \subseteq Y$ is dense. If $Y$ is affine this is equivalent to the condition that the comorphism $\phi^{*}\colon \OOO(Y) \to \OOO(X)$ is injective.

\begin{theorem} \label{AutX-locally-closed-in-EndX.thm}
For an affine variety $X$ the subset  $\Aut (X)\subseteq \End(X)$ is locally closed. More precisely, if $\Dom (X)$\idx{$\Dom(X)$} denotes the set of dominant endomorphisms of $X$, then  $\Aut (X)$ is closed in $\Dom (X)$ and $\Dom (X)$ is open in $\End(X)$.
Moreover, the morphism $p\colon \AAA(X) \to \End(X)$  induces an isomorphism $\AAA(X) \simto \Aut(X)$ of ind-varieties.
\end{theorem}

An interesting consequence is the following.
\begin{corollary}
Let $\Phi=(\Phi_{z})_{z\in Z}$ be a family of dominant endomorphisms of the affine variety $X$. Assume that the set $\{z\in Z\mid \Phi_{z}\in\Aut(X)\}$ is dense in $Z$.  Then $\Phi$ is an automorphism.
\end{corollary}

\begin{remark} If we assume in the corollary above that $\{z\in Z\mid \Phi_{z}\in\Aut(X)\}$ contains an open and dense set, then we can give a direct proof. First of all, we can assume that $Z$ is a smooth curve $C$. Since a suitable power $\Phi^{m}$ stabilizes the irreducible components of $C \times X$ we can also assume that $X$ is irreducible. 

Now the assumptions imply that $\Phi\colon C \times X \to C \times X$ is birational and almost surjective, i.e. $\codim_{C \times X}\overline{C \times X\setminus \Phi(C \times X)}\geq 2$. If $\eta\colon \tilde X \to X$ is the normalization, we get the commutative diagram
$$
\begin{CD}
C \times \tilde X @>\tilde\Phi>> C \times \tilde X \\
@VV{\id_{C}\times\eta}V   @VV{\id_{C}\times\eta}V \\
C \times X @>\Phi>> C \times X
\end{CD}
$$
which implies that $\tilde\Phi$ is also birational and almost surjective, hence $\tilde\Phi$ is an isomorphism, by  Lemma~\ref{Igusa.lem} below. From this we see that $\Phi$ is a finite morphism. If $\Phi$ were not an isomorphism, then we get an infinite sequence of finite $\OOO(C\times X)$-modules
$$
\OOO(C \times X) \subsetneqq \rho^{-1}(\OOO(C \times X)) \subsetneqq \rho^{-2}(\OOO(C \times X)) \subsetneqq \cdots 
\subsetneqq \OOO(C \times \tilde X)
$$
where $\rho:=(\tilde\Phi)^{*} \colon \OOO(C\times \tilde X) \simto\OOO(C\times \tilde X)$.
\end{remark}

The following result  was used in the remark above. It is due to \name{Igusa} who used it in his proof of 
\cite[Lemma~4]{Ig1973Geometry-of-absolu}.
Another proof can be found in \cite[II.3.4 Lemma and Bemerkung on page 106]{Kr1984Geometrische-Metho}.

\begin{lemma}  \label{Igusa.lem}
Let $X,Y$ be irreducible affine varieties, and let $\phi\colon X \to Y$ be a birational morphism. Assume that $Y$ is normal and that $\phi$ is almost surjective, i.e. $\codim_{Y}\overline{Y\setminus\phi(X)} \geq 2$. Then $\phi$ is an isomorphism.
\end{lemma}

We now fix a closed embedding $X \into \An$ and write $ \OOO(X) = \kk  [x_1,\ldots,x_n] / I$ where $I = (q_1,\ldots, q_m)$ is the vanishing ideal, and we define $\OOO(X)_{k}$ to be the image of $\kk[x_{1},\ldots,x_{n}]_{\leq k}$. This defines a filtration $\OOO(X) = \bigcup_{k\geq 0}\OOO(X)_{k}$ by finite-dimensional subspaces such that $\OOO(X)_{k}\cdot\OOO(X)_{\ell}\subseteq \OOO(X)_{k+\ell}$. We have a closed immersion
\[ 
\End (X) = \{ f=(f_1,\ldots,f_n) \in \OOO (X)^n \mid q(f) = 0 \text{ for all } q\in I\}\subseteq \OOO(X)^{n},
\]
and define
\[ 
\End (X) _{k} := \End (X) \cap \OOO(X)^{n}_{k}   \text{ \ and \ } \Dom (X) _{k}  := \Dom(X) \cap \End (X) _{k}.
\]
This allows to define the \itind{degree\/} of a regular function $f \in \OOO(X)$ and of an endomorphism $\phi\in\End(X)$:
\[
\deg f :=\min\{k\geq 0\mid f\in\OOO(X)_{k}\}, \quad 
\deg\phi:=\min\{k\geq 0\mid \phi\in\End(X)_{k}\}.
\]
Note that $\deg (fh) \leq \deg f + \deg h$ and $\deg \phi\circ\psi \leq \deg\phi \cdot \deg\psi$.

\ps
\subsection{Proof of Theorem~\ref{AutX-locally-closed-in-EndX.thm}}
The first step in the proof of Theorem~\ref{AutX-locally-closed-in-EndX.thm} relies on the following result from \name{Dub\'e} \cite{Dube1990The-structure-of-p} providing an upper bound for the degree of the polynomials in a reduced Gr\"obner basis of an ideal. Such bounds were already obtained by various authors, see e.g. \cite{MoMo1984Upper-and-lower-bo}. For generalities on  Gr\"obner bases\idx{Gr\"obner basis}, we refer to \cite{KrRo2008Computational-comm}. In particular,  the definition of a reduced Gr\"obner basis is given in \cite[Definition 2.4.12, page~115]{KrRo2008Computational-comm}.

\begin{proposition}[\name{Dub\'e}] \label{prop: degree bound for Grobner bases}
Let $J= ( h_1, \ldots, h_s ) \subseteq \kk [z_1,\ldots,z_m]$ be an ideal, and set $d:= \max_i \{\deg h_i\}$. If $\sigma$ is any term ordering and if $G= \{ g_1,\ldots,g_t \}$ is a reduced $\sigma$-Gr\"obner basis of $J$, then we have
\[
\max_i \{\deg g_i\}  \leq  2 \cdot \left( \frac{d^2}{2} + d \right)^{2^{m-1}}. 
\]
\end{proposition}

\begin{lemma} \label{lem:bounds for injectivity and surjectivity}
Let $X \subseteq \A{n}$ be a closed subvariety with vanishing ideal $ (q_1,\ldots, q_m)$. For $\phi \in \End (X)$ there is an  $e=e(\deg\phi)$ such that
\be
\item 
$\phi^* \colon \OOO (X) \to \OOO (X) $ is injective if and only if
the induced map $\phi^* \colon \OOO (X)_{e} \to \OOO (X) $ is injective.
\item 
If, in addition,  $\phi$ is an automorphism, then $\deg \phi^{-1} \leq e$.
\ee
\end{lemma}

\begin{proof}
(a) Set $d:=\deg\phi$, and let $f=(f_1,\ldots,f_n) \in \EndA{n}_{d}$ be a lift of $\phi\in(\OOO(X)_{d})^{n}$. Consider the polynomial ring  $Q:=\kk [x_1,\ldots,x_n,y_1,\ldots,y_n]$, and let $J \subseteq Q$ be the ideal generated by the following elements:
\[ 
x_i - f_i(y_1,\ldots,y_n) \text{ for } 1 \leq i \leq n \text{ \ and \ }  q_j(y_1,\ldots, y_n) \text{ for }1\leq j \leq m.
\]
Let $\sigma$ be an elimination ordering for $\{ y_1, \ldots,y_n\}$ and let $G=\{ g_1, \ldots,g_s \}$ be the reduced $\sigma$-Gr\"obner basis of $J$. Then, by Proposition~\ref{prop: degree bound for Grobner bases}, we have $\deg g_i \leq e:=2 \cdot \left( \frac{d^2}{2} + d \right)^{2^{m-1}}$. 

(b) Put $P:=\kk[x_{1},\ldots,x_{n}] \subseteq Q$, and let  $\Phi\colon \kk [x_1,\ldots,x_n]\to \OOO(X)$ be the composition of the projection $k[x_{1},\ldots,x_{n}]\to\OOO(X)$ with $\phi^{*}$, i.e. $\Phi(p)=p(p_{1},\ldots,p_{n})$. Then $\Ker\Phi=J \cap P$ by  \cite[Proposition 3.6.2]{KrRo2008Computational-comm}, and  $\widehat{G} :=G \cap P$ is the reduced  $\sigma$-Gr\"obner basis of $J \cap P$ by \cite[Theorem 3.4.5.c)]{KrRo2008Computational-comm}. Therefore, if $\Ker \phi^*$ is nonzero, it necessarily contains a nonzero element of degree $\leq e$.

(c) By \cite[Proposition 3.6.6.d)]{KrRo2008Computational-comm}, $G$ contains elements of the form $y_i - h_i(x_1,\ldots,x_n)$ for $1 \leq i \leq n$. Therefore, $h_i$ is the $\sigma$-normal form of $y_i$. By \cite[Proposition 3.6.6.b)]{KrRo2008Computational-comm}, we get $\bar{x}_i = h_i ( \bar{f}_1, \ldots, \bar{f}_n)$, so that $h=(h_1,\ldots,h_n)$ is a polynomial endomorphism of $\A{n}$ such that $h|_{X} = \phi^{-1}$.
In particular, $\deg \phi^{-1}\leq e$.
\end{proof}

The next lemma was already used in special form in Section~\ref{R*-embedding.subsec}, see Lemma~\ref{linear-algebra.lem}.
\begin{lemma} \label{lem: linear algebra}
Let $V,W$ be finite-dimensional $\kk$-vector spaces, and let $W' \subseteq W$ be a subspace. Denote by $\LLL(V,W)$\idx{$\LLL(V,W)$} the linear maps from $V$ to $W$, and by $\LLL'(V,W)\subseteq \LLL(V,W)$ the subset of injective maps.
\be
\item The subset $\LLL'(V,W) \subseteq \LLL (V,W)$ is open;
\item The subset $F_{W'} := \{ h \in \LLL'(V,W) \mid h(V)\supseteq W' \}$ is closed in $\LLL'(V,W)$.
\ee
\end{lemma}

\begin{proof}
(1) This is well-known, because the set of maps of rank $\leq k$ is closed for each $k$.
\par\smallskip
(2) Consider the morphism $\phi \colon \LLL'(V,W) \times V \to W$, $(h,v)\mapsto h(v)$, and take the inverse image $B:=\phi^{-1}(W')$. If $p \colon B \to \LLL'(V,W)$ is the morphism induced from the projection $\pr\colon \LLL'(V,W) \times V \to \LLL'(V,W)$, then we see that $p^{-1}(h) \simeq h^{-1}(W')$ for all $h \in \LLL'(V,W)$. Thus the fibers of $p$ are (isomorphic to) $\kk$-vector spaces. We may assume that $\dim W' \leq \dim U$ since otherwise $F_{W'} = \emptyset$. Then $F_{W'}$ is the set of points where the fiber has maximal dimension, namely $\dim W'$. It is again well-known that this set is closed.
\end{proof}

\begin{remark} \label{rem: a set of linear maps which is not locally closed}
Note that the set $C_{W'}:= \{ u \in \Linear (V,W) \mid W' \subseteq u(V) \}$ is in general not locally closed in $\Linear (V,W)$. Indeed, if $V=W=\kk^2$ and if $W'$ is any line of $W$, then $C_{W'}$ is a dense subset of $\Linear (V,W)$ which is not open.
\end{remark}

\begin{remark}
One might try to generalize Theorem~\ref{AutX-locally-closed-in-EndX.thm} and to prove that for any affine varieties $X$ and $Y$, the set \itind{$\Imm (X,Y)$} of closed immersions of $X$ into $Y$ is a locally closed subset of the ind-variety $\Mor (X,Y)$ of all morphisms from $X$ into $Y$. However, this is in general not true, as we now show by taking $X= \AA^1$ and $Y= \AA^2$. Set $V=  \Span (x,y) \subseteq \OOO (\AA^2) = \kk [x,y] $ and $W = \Span (x,x^2) \subseteq \OOO ( \AA^1) = \kk [x]$. The space $\Linear (V,W)$ can naturally be identified with a (closed) subset of $\Mor (\AA^1,\AA^2)$: For each element $u \in \Linear (V,W)$, consider the morphism of algebras $\OOO (\AA^2) \to \OOO (\AA^1)$ sending $x$ onto $u(x)$ and $y$ onto $u(y)$. If $\Imm (\AA^1,\AA^2)$ was locally closed in $\Mor (\AA^1, \AA^2)$, then its trace $\Imm (\AA^1,\AA^2) \cap \Linear (V,W)$ over $\Linear (V,W)$ should a fortiori be locally closed in $\Linear (V,W)$. However, setting $W' = \Span (x) \subseteq W$ and using the notations from Remark~\ref{rem: a set of linear maps which is not locally closed}, one could easily check that this trace is equal to the set $C_{W'}:= \{ u \in \Linear (V,W) \mid W' \subseteq u(V) \}$. The set $C_{W'}$ is not locally closed in $\Linear (V,W)$, again by Remark~\ref{rem: a set of linear maps which is not locally closed}.
\end{remark}

\begin{proof}[Proof of Theorem~\ref{AutX-locally-closed-in-EndX.thm}]
(a) We first show  that $\Dom (X)_{k}$ is open in $\End (X)_{k}$.
Let $\phi$ be any element of $\End (X)_{k}$. By Lemma~\ref{lem:bounds for injectivity and surjectivity}, $\phi$ is dominant if and only if  $\phi^* \colon \OOO (X)_{e} \to \OOO (X)_{ek}$ is injective, and this is an open condition, by Lemma~\ref{lem: linear algebra}\,(1).
\ps
(b) Now we prove that $\Aut (X)$ is closed in $\Dom (X)$, i.e. that  $\Aut (X)_{k}$ is closed in $\Dom (X)_{k}$ for all $k$. Clearly, $\phi\in\Dom_{k}$ is an automorphism if and only if  $\phi^* \colon \OOO (X) \to \OOO (X) $ is surjective. This is equivalent to $\bar{x}_i \in \Image (\phi^*)$ for each $i$. By Lemma~\ref{lem:bounds for injectivity and surjectivity}, this means that  $\overline{x}_i$ is in the image of the injective linear map $f^* \colon \OOO (X)_{e} \to \OOO (X)_{ek}$, and this is a closed condition, by  Lemma~\ref{lem: linear algebra}(2).
\ps
(c) Finally, we show that $p\colon \AAA(X) \to \End(X)$ is a (locally closed) immersion, i.e. that it
 induces an isomorphism $\AAA(X) \simto \Aut(X)$ onto its image. The family $\Aut (X)\times X \to X$ of automorphisms of $X$  corresponds, by Theorem~\ref{AutX.thm}, to a morphism $q \colon \Aut (X) \to \AAA(X)$. This morphism is obviously the inverse of the morphism $p' \colon  \AAA(X) \to \Aut (X)$ induced by $p$. This concludes the proof.
\end{proof}

\ps
\subsection{Automorphism groups of associative \texorpdfstring{$\kk$}{k}-algebras}\label{Aut(R).subsection}
The results in the previous sections about the automorphism group $\Aut(X)$ of an affine variety $X$ can be formulated in terms of the automorphism group $\Aut(\OOO(X))$ of the $\kk$-algebra $\OOO(X)$, using the anti-isomorphism $\Aut(X) \simto\Aut(\OOO(X))$, $\phi\mapsto \phi^{*}$. 
Some of them easily carry over to $\Aut(R)$ where $R$ is a finitely generated \itind{associative $\kk$-algebra}, or even to $\Aut(\R)$ where $\R$ is a finitely generated \itind
{general $\kk$-algebra}. The latter will be discussed in detail in Section~\ref{GeneralAlgebra.sec}. 

In the following $R$ is a finitely generated associative $\kk$-algebra. If $W \subseteq R$ is a finite-dimensional subspace generating $R$, then $\End(R)\subset \LLL(W,R)$ is a closed subset, hence an affine ind-variety. Moreover, the multiplication in $\End(R)$ is an ind-morphism, so that $\End(R)$ has a natural structure of an affine ind-semigroup (cf. Section~\ref{endo-k-algebra.subsec} where the case of a commutative $\kk$-algebra is discussed.)

\begin{proposition}\label{Aut-associative-algebra.prop}
The automorphism group $\Aut(R)$ of a finitely generated associative $\kk$-algebra $R$ has a natural structure of an affine ind-group.
\end{proposition}

\begin{proof} (Cf. Theorem~\ref{AutX.thm} and its proof.)
Consider the closed subset
$$
\AAA := \{(\phi,\psi) \in \End(R)\times\End(R)\mid \phi\circ \psi =\psi\circ\phi = \id_{R}\} \subset \End(R) \times \End(R)
$$
The first projection $\pr_{1}\colon \End(R) \times \End(R) \to \End(R)$ induces a bijective map  $\AAA \to \Aut(R)$, and thus gives $\Aut(R)$ the structure of an affine ind-variety. It is easy to see that with this structure $\Aut(R)$ is an affine ind-group with the usual universal properties.
\end{proof}

We have a canonical representation of $\Aut(R)$ on $R$  which defines a homomorphism of Lie algebras $\delta\colon \Lie\Aut(R) \to \LLL(R)$ where $\LLL(R)$ are the linear endomorphisms of the $\kk$-vector space $R$. (For this and the following see Lemma~\ref{rep-of-G-and-LieG.lem}.) The image $\delta_{A}\in\LLL(R)$ of $A\in\Lie\Aut(R)$ is defined by $\delta_{A}(a):=d\mu_{a}(A)$ where $\mu_{a}\colon \Aut(R) \to R$, $a \mapsto g(a)$, is the orbit map.
It is clear from the definition, that $\delta$ extends to a map $\tilde\delta\colon T_{\id} \End(R) \to \LLL(R)$ 
defined in the same way. \idx{$\Div$@$\delta:\Lie\Aut(R)\to\Der_{\kk}(R)$}

\begin{proposition}  \label{Tangent-space-of-End(R)-and-Aut(R).prop}
The homomorphism $\tilde\delta$ is injective, and its image is contained in the derivations of $R$,
\[
\tilde\delta \colon T_{\id}\End(R) \into \Der_{\kk}(R)\quad\text{and}\quad \delta\colon \Lie\Aut(R) \into \Der_{\kk}(R),
\]
where the second map is a homomorphism of Lie algebras.
\end{proposition}
\begin{proof} (Cf. Proposition~\ref{TEnd(R)-into-Der(R).prop} and its proof.)
For $a,b \in R$ the morphism $\mu_{a\cdot b}$ is the composition
\[
\begin{CD}
\End(R) @>(\mu_{a},\mu_{b})>> R \times R @>{m}>> R
\end{CD}
\]
where $m$ is the multiplication of $R$. From this we see that the differential $d\mu_{a\cdot b}$ is given by the composition
\[
\begin{CD}
T_{\id} \End(R) @>(d\mu_{a},d\mu_{b})>> T_{a}R \oplus T_{b}R @>{dm_{(a,b)}}>> R
\end{CD}
\]
Since $dm_{(a,b)}(x,y) = a\cdot y + x \cdot b$ we finally get $\delta_{A}(a\cdot b) = a\cdot\delta_{A}(b) + \delta_{A}(a)\cdot b$, showing that $\delta_{A}$ is a derivation of $R$.  

In order to see that $\tilde\delta$ is injective we choose a system  of generators $(a_{1},\ldots,a_{n})$ of $R$ to get a closed immersion $\End(R) \into R^{n}$ of ind-varieties.
We obtain an injection $T_{\id} \End(R) \into R^{n}$ which has the following description: $A \mapsto (\tilde\delta_{A}(a_{1}),\ldots,\tilde\delta_{A}(a_{n}))$. Since any derivation of $R$ is determined by the images of the generators $a_{1},\ldots,a_{n}$ we see that $\tilde\delta_{A}\neq 0$ if $A\neq 0$.
\end{proof}

\begin{question}   \label{Is-T(End(R))-a-Lie-algebra.question}
Is  $T_{\id}\End (R)$  a Lie subalgebra of  $\Der_{\kk} (R)$?
\end{question}

\begin{question} (Cf. Theorem~\ref{AutX-locally-closed-in-EndX.thm})\label{AutR-locally-closed.que}
Is $\Aut(R) \subseteq \End(R)$  locally closed when $R$ is an associative $\kk$-algebra?
\end{question}

\ps
\subsection{Small semigroups of endomorphisms}\label{Small-End.subsec}
We start with an example.
\begin{example}  \label{aut-torus.exa}
An endomorphism of $\Cst$ has the form $t \mapsto a\cdot t^{m}$ where $a\in\Cst$ and $m\in\ZZ$. This shows that the semi-group $\End(\Cst)$ is the semi-direct product 
$\ZZ^{\times} \ltimes \Cst$
where the (multiplicative) semigroup $\ZZ^{\times}$ acts on $\Cst$ by $m\, a := a^{m}$. In particular, $\dim\End(\Cst) = 1$.

This example easily generalizes to an {\it $n$-dimensional torus} $T = (\Cst)^{n}$\idx{torus $T\simeq(\kst)^{n}$} where one finds that $\End(T) = \Mat_{n}(\ZZ) \ltimes (\Cst)^{n}$.
Here, a matrix $M=(m_{ij})\in M_{n}(\ZZ)$ acts on $(\Cst)^{n}$ by $M(t_{1},\ldots,t_{n}) := (\ldots, \prod_{j}t_{j}^{m_{ij}},\ldots)$. Again, this ind-semigroup is finite-dimensional, $\dim\End(X)=n$, and $\Aut(X)$ is open in $\End(X)$. This is a special case of the following general result.
\end{example}

\begin{proposition}
Let $X \subseteq (\kst)^{m}$ be a closed subvariety. Then $\End(X)$ is finite-dimensional, and $\Aut(X)\subseteq \End(X)$ is open.
\end{proposition}
\begin{proof}
We will regard $\OOO(X)^{*}$ as a closed ind-subvariety of $\OOO(X) \times \OOO(X)$ via the embedding $f \mapsto (f,f^{-1})$. Choose invertible generators $f_{1},\ldots,f_{n}$ of $\OOO(X)$. They define a closed immersion 
$$
\End(X) \into \OOO(X)^{2n}, \ \phi\mapsto (\phi^{*}(f_{1}),\phi^{*}(f_{1}^{-1}),\ldots,\phi^{*}(f_{n}),\phi^{*}(f_{n}^{-1})).
$$ 
Then the image lies in $(\OOO(X)^{*})^{n}$ which is finite-dimensional by Proposition~\ref{structure-of-R*.prop} above. The last statement follows from the next lemma.
\end{proof}

\begin{lemma}\label{Aut-open-in-End.lem}
If $\End(X)$ is finite-dimensional, then $\Aut(X)$ is open in $\End(X)$.
\end{lemma}
\begin{proof} The connected component $M:=\End(X)^{(\id)}$ is an algebraic monoid which acts locally finitely on the coordinate ring $\OOO(X)$. Therefore, we can find an $M$-stable finite-dimensional subspace $W \subseteq \OOO(X)$ which generates $\OOO(X)$. We get an embedding of monoids $\rho\colon M\into \LLL(W)$ where $\LLL(W)$ are the linear endomorphisms of $W$. Now it is clear that $\Aut(X)\cap M =\rho^{-1}(\GL(W))$ is open in $M$, and so $\Aut(X)$ is open in $\End(X)$.
\end{proof}

\begin{example}
An affine variety is called \itind{endo-free} if every endomorphism $\neq \id_{X}$ is constant. In \cite{AnKr2014Varieties-characte} endo-free smooth affine varieties are constructed in arbitrary dimensions. If $X$ is endo-free, then $\id_{X} \in \End(X)$ is an isolated point and $\End(X) \simeq X \cup \{\id_{X}\}$ as ind-varieties (see Proposition~\ref{endX-gives-X.prop} and its proof). In particular, $\End(X)$ is an algebraic variety isomorphic to the disjoint union $X\cup \{*\}$.
\end{example}

\ps
\subsection{The ind-variety of group actions}\label{var-group-actions.subsec}
Let $G$ be a linear algebraic group,  and let $X$ be an affine variety. A \itind{group action} of $G$ on $X$ is given by a morphism $\rho\colon G \times X \to X$ satisfying the two conditions
\be
\item[(i)] For all $x\in X$ we have $\rho(e,x) = x$;
\item[(ii)] For all $g,h \in G$ and $x\in X$ we have $\rho(g,\rho(h,x)) = \rho(g h, x)$.
\ee
The following lemma shows that the space $\Act_{G}(X)$\idx{$\Act_{G}(X)$} of group actions of $G$ on $X$ has a natural structure of an affine ind-variety.

\begin{lemma}
\be
\item\label{Act1}
The subset 
$$
\Act_{G}(X):=\{\rho\colon G\times X \to X \mid \rho\text{ is a group action}\} \subseteq \Mor(G\times X ,X)
$$ 
is closed.
\item\label{Act2}
The subset $\Hom(G,\Aut(X)) \subseteq \Mor(G,\End(X))$ is closed.\idx{$\Hom(G,\Aut(X))$}
\ee
\end{lemma}\label{group-actions.lem}
\begin{proof}
(1)
We have to show that each of the two conditions (i) and (ii) above defines a closed subset. For the first one this is obvious. We have a morphism $\Psi\colon \Mor(G\times X, X) \to \Mor(X,X)$ given by $\Psi(\rho)=\rho\circ\iota$ where the morphism $\iota\colon X \to G\times X$ is the closed immersion  $x\mapsto(e,x)$. Then condition (i) is equivalent to $\Psi(\rho) = \id_{X}$.

For the second condition let $\mu\colon G \times G \to G$ denote the multiplication, and define the following two maps $\Phi_{1},\Phi_{2}\colon \Mor(G\times X, X) \to \Mor(G\times G \times X, X)$: 
$$
\Phi_{1}(\rho):=\rho\circ(\mu\times \id_{X}),\quad \Phi_{2}(\rho):=\rho\circ(\id_{G}\times \rho).
$$
It is again clear that both maps are morphisms of ind-varieties. Since condition (ii) is equivalent to $\Phi_{1}(\rho) = \Phi_{2}(\rho)$ the claim follows.
\ps
(2)
Let  $H$ be a group and $S$ a semigroup. Then it is easy to see that a map $\phi\colon H \to S$ induces a homomorphism $H\to S^{*}$ of groups if and only if one has
\be
\item[(i)] $\phi(ab) = \phi(a)\phi(b)$ for $a,b \in G$, and 
\item[(ii)] $\phi(e_{G}) = e_{S}$.
\ee
For $H = G$ and $S = \End(X)$ the two conditions (i) and (ii) define a closed subset of $\Mor(G,\End(X))$ which is equal to $\Hom(G,\Aut(X))$.
\end{proof}

It is clear that a $G$-action on $X$ defines a homomorphism $G \to \Aut(X)$. We will show now that this is a homomorphism of ind-groups, and that every homomorphism of ind-groups $G \to \Aut(X)$ defines an action of $G$ on $X$.

\begin{proposition}\label{group-action.prop} 
Let $G$ be a linear algebraic group and $X$ an affine variety.
\be
\item
There is a bijection between the $G$-actions on $X$ and the homomorphisms $G \to \Aut(X)$ of ind-groups.
\item
The set $\Hom(G, \Aut(X))$ of homomorphisms of ind-groups has a natural structure of an affine ind-variety, as a closed subset of the affine ind-variety $\Mor(G,\End(X))$.
\item 
The natural bijective map $\iota\colon\Act_{G}(X) \simto \Hom(G,\Aut(X))$ is an isomorphism of ind-varieties.
\ee
\end{proposition}

\begin{proof}
(1) An action of $G$ on $X$ defines a family of automorphisms of $X$ parametrized by $G$, hence a morphism $G \to \Aut(X)$, by Theorem~\ref{AutX.thm}, and this morphism is a homomorphism of groups. On the other hand, if $G \to \Aut(X)$ is a homomorphism of ind-groups, then the induced map $G \times X \to X$ is a morphism, hence an action of $G$ on $X$, again by Theorem~\ref{AutX.thm}.
\ps
(2) This is Lemma~\ref{group-actions.lem}(\ref{Act1}) above.
\ps
(3) By Lemma~\ref{mor-mor.lem} we have an isomorphism of ind-varieties 
$$
\iota\colon\Mor(G \times X ,X) \simto \Mor (G, \End(X))
$$
inducing a bijection $\iota\colon \Act_{G}(X) \to \Hom(G,\Aut(X))$. From Lemma~\ref{group-actions.lem}  we get that
$\Act_{G}(X)$ is closed in  $\Mor(G\times X,X)$ and that $\Hom(G, \Aut(X))$ is closed in $\Mor(G,\End(X))$, 
hence the map $\iota$ is an isomorphism of ind-varieties.
\end{proof}

\ps
\subsection{Base field extensions of families}
Let $S,X,Y$ be varieties and let $\Phi=(\Phi_{s})_{s\in S}$ be a family of morphisms $X \to Y$ parametrized by $S$. Let $\KK/\kk$ be a field extension where $\KK$ is algebraically closed.
Then $\Phi_{\KK}\colon S_{\KK}\times X_{\KK} \to Y_{\KK}$
is again a family of morphisms $X_{\KK} \to Y_{\KK}$. 
In case $X$ and $Y$ are both affine varieties, then these two families correspond to the morphisms $\psi\colon S \to \Mor(X,Y)$ and $\Psi\colon S_{\KK} \to \Mor(X_{\KK},Y_{\KK})$. It follows from the universal property of $\Mor(X,Y)$ that there is a unique morphism $\mu\colon\Mor(X,Y)_{\KK} \to \Mor(X_{\KK},Y_{\KK})$ such that the following diagram commutes.
\[\tag{$*$}
\begin{CD}
S @>\subseteq>> S_{\KK} @= S_{\KK}\\
@VV{\psi}V  @VV{\psi_{\KK}}V  @VV{\Psi}V \\
\Mor(X,Y)  @>\subseteq>> \Mor(X,Y)_{\KK} @>\mu>> \Mor(X_{\KK},Y_{\KK})\\
\end{CD}
\]
It is clear now that we get similar morphisms $\mu\colon\End(X)_{\KK}\to\End(X_{\KK})$ and $\mu\colon\Aut(X)_{\KK}\to\Aut(X_{\KK})$ which are compatible with the semigroup structure, resp. the group structure.

\begin{proposition}\label{field-extensions-for-morphisms.prop}
The canonical morphisms\idx{base field extension of families}
\begin{gather*}
\mu\colon\Mor(X,Y)_{\KK} \simto \Mor(X_{\KK},Y_{\KK}), \quad  \mu\colon\End(X)_{\KK} \simto \End(X_{\KK}), \\
  \mu\colon \Aut(X)_{\KK} \simto \Aut(X_{\KK})
\end{gather*}
are all isomorphisms.
\end{proposition}
\begin{proof}
(1) We first look at a family of morphisms $X \to \An$, $\Phi\colon S \times X \to \An$. If $\Phi=(f_{1},\ldots,f_{n})$ where $f_{i}\in\OOO(S \times X)$ and if we identify $\Mor(X,\An) = \OOO(X)^{n}$, then the map $\phi\colon S \to \OOO(X)^{n}$ is given by $s\mapsto (f_{1}(s,?),\ldots,f_{n}(s,?))$, and  $\phi_{\KK}\colon S_{\KK}\to\OOO(X)_{\KK}$ is given by the same formula. On the other hand, $\Phi_{\KK}\colon S_{\KK}\times X_{\KK}\to \AA_{\KK}^{n}$ is also given by $(f_{1},\ldots,f_{n})$ and so $\phi'\colon S_{\KK}\to \OOO(X_{\KK})^{n}$ has the form $s \mapsto (f_{1}(s,?),\ldots,f_{n}(s,?))$.
Thus we get the following commutative diagram
$$
\begin{CD}
S @>\subseteq>> S_{\KK} @= S_{\KK}\\
@VV{\phi}V  @VV{\phi_{\KK}}V  @VV{\phi'}V \\
\Mor(X,\An)  @>\subseteq>> \Mor(X,\An)_{\KK} @>\mu>> \Mor(X_{\KK},\AA^{n}_{\KK})\\
@V{=}V{\Phi_{X}}V @V{=}V{(\Phi_{X})_{\KK}}V @V{=}V{\Phi_{X_{\KK}}}V\\
\OOO(X)^{n} @>\subseteq>> \OOO(X)_{\KK}^{n} @>\nu>\simeq> \OOO(X_{\KK})^{n}
\end{CD}
$$
where $\nu$ is given by the canonical isomorphism $\OOO(X)_{\KK}=K\otimes_{\kk}\OOO(X)\simto\OOO(X_{\KK})$. Thus $\mu$ is an isomorphism in this case.

(2) For $\Mor(X,Y)$ we can assume $Y\subseteq \An$ where the vanishing ideal $I(Y) = (h_{1},\ldots,h_{n})$. Then $\Mor(X,Y) \subseteq \Mor(X,\An)=\OOO(X)^{n}$ is closed and also defined by the equation $h_{1},\ldots,h_{n}$ applied to $n$-tuples of functions of $\OOO(X)^{n}$. It follows that the closed subset $\Mor(X,Y)_{\KK}\subseteq\OOO(X)^{n}_{\KK}$ is defined by the same equations, and we have the following commutative diagram: 
$$
\begin{CD}
\Mor(X,Y)  @>\subseteq>> \Mor(X,Y)_{\KK} @>\mu>> \Mor(X_{\KK},Y_{\KK})\\
@V{\subseteq}V{\Phi_{X}}V @V{\subseteq}V{(\Phi_{X})_{\KK}}V @V{\subseteq}V{\Phi_{X_{\KK}}}V\\
\OOO(X)^{n}  @>\subseteq>> \OOO(X)^{n}_{\KK} @>\mu>\simeq> \OOO(X_{\KK})^{n}\\
\end{CD}
$$
Since the vanishing ideal $I(Y_{\KK})$ is also generated by $h_{1},\ldots, h_{n}$ we finally get the closed subset $\Mor(X_{\KK},Y_{\KK}) \subseteq \OOO(X_{\KK})^{n}$ is defined by the same equations, so that the closed immersion $\mu$ is an isomorphism.

(3) It remains to prove the assertion for $\Aut(X)$ which we consider as a closed ind-subvariety of $\End(X)\times\End(X)$. We have the following diagram
$$
\begin{CD}
\Aut(X)  @>\subseteq>> \Aut(X)_{\KK} @>\mu>> \Aut(X_{\KK})\\
@VV{\subseteq}V @VV{\subseteq}V @VV{\subseteq}V\\
\End(X)\times\End(X)  @>\subseteq>> \End(X)_{\KK}\times\End(X)_{\KK} 
@>\mu>\simeq> \End(X_{\KK})\times\End(X_{\KK})\\
\end{CD}
$$
which implies that $\mu$ is a closed immersion. Since the defining equations for the embeddings 
$\Aut(X)_{\KK}\subseteq \End(X)_{\KK}\times\End(X)_{\KK}$ and 
$\Aut(X_{\KK})\subseteq \End(X_{\KK})\times\End(X_{\KK})$ are the same, $\mu$ is surjective, and we are done.
\end{proof}

\pmed
\section{Algebraic Group Actions and Vector Fields}\label{VectorFields.sec}
\ps
In this section we discuss the important relation between group actions on varieties and vector fields in the case of linear algebraic groups.  The generalization to  ind-groups will be given in the next Section~\ref{ind-group-actions.sec}.

\ps
\subsection{The linear case}\label{linear-case.subsec}
Let $V$ be a finite-dimensional $\kk$-vector space. We identify $V$ with the tangent space $T_{p}V$\idx{tangent space} in an arbitrary point $p \in V$ by associating to $v\in V$ the \itind{directional derivative}  $\partial_{v,p}$ 
{\it in direction $v$ in the point $p$}:\idx{$\Div$@$\partial_{v,p}$}
$$
\partial_{v,p} f :=\frac{f(p+tv)-f(p)}{t}\Big|_{t=0}
$$
Choosing coordinates, we get for $v=(v_{1},\ldots,v_{n})$ 
$$
\partial_{v,p}=\sum_{i=1}^{n}v_{i}\dxi\Big|_{x=p}.
$$ 
This shows that a {\it vector field $\delta\in\VEC(V)$}\idx{vector field} (Section~\ref{Endo-VF.subsec}) is the same as a morphism $\phi\colon V \to V$, i.e. an element from $\End(V)$. Choosing a basis of $V$, a morphism $\phi\colon \An \to \An$ is given by the images $f_{i}:=\phi^*(x_{i}) \in \kk[x_{1},\ldots,x_{n}]$,
and we write $\phi=(f_{1},\ldots,f_{n})$. Then the corresponding vector field $\xi_{\phi}$ is
$$
\xi_{\phi} = \sum_{i=1}^{n}f_{i}\ddxi \text{ \ where we write }\ddxi \text{ for }\dxi.
$$ 
The {\it constant vector fields $\partial_{v}$}\idx{constant vector field} ($v \in V$), corresponding to the constant maps $x \mapsto v$,  are the directional derivatives $(\partial_{v}f)(x) =\partial_{v,x}f = \frac{f(x+tv)-f(x)}{t}\Big|_{t=0}$ defined above. 
The {\it linear vector fields}, corresponding to the linear endomorphisms $\LLL(V) \subseteq \End(V)$, also play an important role, because $\LLL(V)=\Lie\GL(V)$.  We denote by $\xi_{A} \in\VEC(V)=\End(V)$ the vector field corresponding to $A\in\Lie\GL(V)=\LLL(V)$. In coordinates, 
we find for $A=(a_{ij})\in \M_{n}(\kk)$\idx{linear vector field}
$$
\xi_{A}= \sum_{i=1}^{n}(a_{i1}x_{1}+\cdots + a_{in}x_{n})\ddxi.
$$
The vector fields also form a Lie algebra, and a simple calculation show that 
$$
\xi_{[A,B]} = [\xi_{B},\xi_{A}] \text{  for  }A,B \in \Lie\GL(V)=\LLL(V),
$$
i.e. the inclusion $\Lie \GL(V) \into \VEC(V)$ is an \itind{anti-homomorphism} of Lie algebras.

We have already seen in Section~\ref{reps-of-ind-groups.subsec}  that the vector field $\xi_{A}$ has the following description.
For $v \in V$ consider the \itind{orbit map} $\mu_{v} \colon \GL (V)  \to V$,
$g \mapsto gv$, and its differential
$d\mu_{v}\colon \Lie \GL (V) \to T_{v}V = V$.
Then
$$
d\mu_{v}(A) = (\xi_{A})_{v}  \text{ for all $A \in \Lie \GL (V)$ and $v \in V$}.
$$

\begin{lemma}  \label{linear-VF.lem}
Let $\lambda \in \VEC(V)$ be a linear vector field, i.e. $\lambda \in \LLL(V) \subseteq \VEC(V)$, and let $\delta \in\VEC(V)$ be an arbitrary
vector field.
Then
$$
\lambda(\delta_{0}) = -[\lambda,\delta]_{0}.
$$
(Note that $\delta_{0}\in T_{0}V = V$, and so $\lambda(\delta_{0})$ makes sense.)
\end{lemma}

\begin{proof} Let $(v_{1},\ldots,v_{n})$ be a basis of $V$ and $(x_{1},\ldots,x_{n})$ the dual basis. Write $\delta = \sum_{j}f_{j}\dxj$, 
so that $\delta_{0}=(f_{1}(0),\ldots,f_{n}(0))\in T_{0}V=V$, and denote by $A = (a_{ij})_{i,j}$ the $n\times n$ matrix corresponding to $\lambda$. Then, as a vector field, $\lambda = \sum_{i,j}a_{ij}x_{j}\dxi$, and we find
$$
(\lambda \circ\delta)(x_{k})= \lambda(f_{k}) = \sum_{i,j} a_{ij} x_{j} \ab{f_{k}}{x_{i}} \text{ \ and \ }
(\delta\circ\lambda)(x_{k}) = \delta(\sum_{j}a_{kj}x_{j})  = \sum_{j} a_{kj}f_{j},
$$
hence
\[
[\lambda,\delta]_{0} = - \sum_{j,k}  a_{kj}f_{j}(0)\dxk\Big|_{x=0} = - A \cdot \delta_{0}. \qedhere
\]
\end{proof}
\ps
If $N \in \Lie\GL(V)=\LLL(V)$ is a nilpotent element, then the \itind{exponential map}
$$
\exp N := \sum_{k=0}^{\infty}\frac{1}{k!} N^{k}\in\LLL(V)
$$  
is well-defined, and it is a \itind{unipotent automorphism} of $V$. In order to see this one uses a basis of $V$ such that $N$ is in Jordan normal form. More generally, for every nilpotent $N \in \LLL(V)$ the map
\[\tag{$*$}
\lambda_{N}\colon \kplus \to \GL(V), \quad s \mapsto \exp(sN),
\]
is a homomorphism of algebraic groups such that $d\lambda_{N}(1) = N$.
The following lemma is well known.

\begin{lemma}  \label{exp-basics.lem}
Let $V$ be a finite-dimensional $\kk$-vector space.
\be
\item
The subset $\NNN(V) \subseteq \LLL(V)$ of {\it nilpotent endomorphisms} is closed, as well as
the subset $\UUU(V)\subseteq \GL(V)$ of {\it unipotent elements}.\idx{$\NNN(V)$}\idx{$\Tr(n)$@$\UUU(V)$}
\idx{nilpotent endomorphism}\idx{unipotent element}
\item
The {\it exponential map}\idx{$\End$@$\exp$}
$$
\exp\colon \NNN(V) \simto \UUU(V), \ N \mapsto \exp N := \sum_{k=0}^{\infty}\frac{1}{k!} N^{k},
$$
is an isomorphism of varieties which is equivariant with respect to conjugation with elements from $\GL(V)$.
\item
If $\lambda\colon \kplus \to \GL(V)$ is a homomorphism of algebraic groups and $N:=d\lambda_{0}(1) \in \Lie\GL(V) = \LLL(V)$, then $N$ is nilpotent and $\lambda = \lambda_{N}$.
\item  \label{exponential-formula}
If $f \in \OOO(V)$, then
\[\tag{$**$}
f(\lambda_{N}(s)v) = \sum_{k=0}^{\infty}\frac{s^{k}}{k!} (\xi_{N}^{k}f)(v).
\]
\ee
\end{lemma}
\begin{proof}[Outline of Proof]
(1) This is easy  and well known.
\ps
(2) The map $\exp$ is a morphism of varieties and commutes with the conjugation by elements from $\GL(V)$. The inverse map is given by\idx{$\M$@$\log$}
$$
\log(u) = \sum_{k=1}^{\infty}\frac{(-1)^{k-1}}{k}(u-\id)^{k}
$$
\ps
(3) This follows from the fact that a homomorphism of connected algebraic groups is determined by the induced homomorphism of the Lie algebras.
\ps
(4)
Denote the function on the right hand side by $h(f,s)$. Then an easy calculation shows that
\[ 
\frac{d}{ds}h(f,s) = h(\xi_{N}f,s)\quad\text{and}\quad 
\frac{d}{ds}f(\lambda_{N}(s)v)=(\xi_{N}f)(\lambda_{N}(s)v).
\]
This implies that if the identity $(**)$ holds for $\xi_{N}f$, then it holds for $f$, because the two functions take the same value for $s=0$. Thus the claim follows by induction on the degree of $f$.
\end{proof}

\ps
\subsection{Exponential map for linear algebraic groups}  \label{exponential-for-linear-algebraic-groups.subsec}
In this section, we generalize Lemma~\ref{exp-basics.lem} above to the case of a linear algebraic group $G$.
We denote by $G^{u} \subset G$ the subset of unipotent elements, and by $\gn \subset \gg:=\Lie G$ the subset of nilpotent elements. Recall that $N\in\gg$ is \itind{nilpotent}, if for one faithful representation $\rho\colon G \into \GL(V)$ the image of $N$ under the differential $d\rho\colon \Lie G \to \LLL(V)$ is nilpotent. It then follows that this holds for any representation of $G$.

\begin{lemma}\label{lambda.lem}   
For any $N \in\gn$ there is unique homomorphism $\lambda_{N}\colon \kplus \to G$ such that $d\lambda_{N}(1) = N$. Moreover, the image of $\lambda_{N}$ is in $G^{u}$
\end{lemma}

\begin{proof}
We embed $G$ as a closed subgroup into some $\GL_{m}$, so that $\gg := \Lie G \subseteq \M_{m}$. If $N \in \gn$, then $N$ is a nilpotent matrix, and the homomorphism $\lambda_{N}\colon \kplus \to \GL_{m}$ defined in $(*)$ above has the property that $d\lambda_{N}(1) = N \in \gg$. Hence, the image of $\lambda_{N}$ is in $G$, and thus in $G \cap \UUU = G^{u}$. The uniqueness is clear.
\end{proof}

\begin{proposition}\label{exponential-for-linear-algebraic-groups.prop}
Let $G$ be a linear algebraic group and $\gg:=\Lie G$ its Lie algebra. The subsets $G^u\subset G$ and $\gn\subset \gg$ are closed, and there exists a $G$-equivariant isomorphism\idx{$\GL$@$\gn$}\idx{$\GL$@$G^u$}\idx{$\End$@$\exp_{G}$}
$$
\exp_{G}\colon \gn \simto G^u
$$
which is uniquely defined by the following property: If $\lambda\colon \kplus \to G$ is a homomorphism and $d\lambda\colon \kk \to \gg$ its differential, then $\exp(d\lambda(1)) = \lambda(1)$. Moreover, $d\exp(A)=A$ for all $A \in T_{0}\gn\subseteq \gg$.
\end{proposition}

\begin{proof}
We can assume that $G$ is a closed subgroup of a suitable $\GL(V)$, so that $\gg \subset \LLL(V)$.
It then follows that $G^u=\UUU(V) \cap G$ and $\gn= \NNN(V)\cap \gg$, showing that the two subsets are closed. We claim that the isomorphism $\exp \colon \NNN(V) \simto \UUU(V)$ from Lemma~\ref{exp-basics.lem}(2) induces an isomorphism $\exp_{G}\colon \gn\simto G^u$. This follows if we show that $\exp(\gn) = G^u$. 

If $N \in \gn$, then, by Lemma~\ref{lambda.lem}, we get a homomorphism $\lambda_{N}\colon \kplus\to G$, hence $\exp_{G}(1) = \lambda_{N}(1) \in G$.  This shows that $\exp_{G}(\gn) \subseteq G^u$.

If $u \in G^u \subset \UUU(V)$, then, by Lemma~\ref{exp-basics.lem}(2), $u = \exp(N)$ for a nilpotent element $N \in \NNN(V)$. It follows that the homomorphism $\lambda_{N}\colon \kplus \to \GL(V)$ has image in $G$, because $\lambda_{N}(1) = u$, and so $N = d\lambda_{N}(1) \in \gg_{u}$. This shows that $\exp_{G}(\gn)\supseteq G^u$. Hence, $\exp_{G}\colon \gn\simto G^u$ is an isomorphism, and the $G$-equivariance of $\exp_{G}$ follows from the $\GL(V)$-equivariance of $\exp$.

The remaining statements follow immediately from Lemma~\ref{exp-basics.lem}.
\end{proof}

\begin{example}\label{unipotent-exp.exa}
If $U$ is a unipotent group, then $U^{u} = U$ and $(\Lie U)^{\text{\it nil}}=\Lie U$. Hence the exponential map gives an isomorphism $\exp_{U} \colon \Lie U \simto U$. This shows that the underlying variety of a unipotent group is an affine space, cf. Theorem~\ref{unipotent-groups.thm}. Moreover, if $U$ is commutative, then $\exp_{U}\colon \Lie U^{+} \simto U$ is an isomorphism of algebraic groups.
\end{example}

The given property of $\exp_{G}$ implies the following ``functoriality'' of the exponential map. 

\begin{lemma}\label{exp-hom.lem}
Let $\rho\colon G \to H$ be a homomorphism of linear algebraic groups. Then the diagram
$$
\begin{CD}
(\Lie G)^{\text{\it nil}}  @>{d\rho}>>  (\Lie H)^{\text{\it nil}}\\
@V{\simeq}V{\exp_{G}}V   @V{\simeq}V{\exp_{H}}V \\
G^{u}@>{\rho}>> H^{u}
\end{CD}
$$
commutes, i.e. $\exp_{H}(d\rho(N)) = \rho(\exp_{G}(N))$ for any nilpotent $N \in \Lie G$.
\end{lemma}
\begin{proof}
Let $\lambda\colon\kplus\to G$ be a homomorphism of algebraic groups. Then, by the definition of $\exp_{G}$ we have $\exp_{G}(d\lambda(1)) = \lambda(1)$, see Proposition~\ref{exponential-for-linear-algebraic-groups.prop}. Similarly, for  $\tilde\lambda := \rho\circ\lambda$, we get $\exp_{H}(d\tilde\lambda(1)) = \tilde\lambda(1)$. Hence
$$
\exp_{H}(d\tilde\lambda(1)) = \exp_{H}(d\rho(\lambda(1))) = \tilde\lambda(1) = \rho(\lambda(1)) = \rho(\exp_{G}(\lambda(1))),
$$
and so $\exp_{H}(d\rho(N)) = \rho(\exp_{G}(N)$ for all nilpotent $N \in \Lie G$.
\end{proof}

There is another way to understand the given relation between the unipotent elements from the group $G$ and the nilpotent elements from the Lie algebra $\gg$.
For this we consider the set $\Hom(\kplus,G)$ of group homomorphisms $\lambda\colon \kplus \to G$, and the two maps
\begin{gather*}
\eps_{G}\colon  \Hom(\kplus,G) \to G, \quad \lambda \mapsto \lambda(1),\\
\nu_{G}\colon \Hom(\kplus,G) \to \gg, \quad \lambda\mapsto d\lambda_{0}(1).
\end{gather*}\idx{$\End$@$\eps_{G}$}\idx{$\NNN$@$\nu_{G}$}
Both maps are $G$-equivariant with respect to the obvious $G$-actions induced by the conjugation action of $G$ on $G$.

\begin{proposition} 
\be
\item \label{Hom(k,G)-closed-hence-affine-variety}
$\Hom(\kplus,G) \subset \Mor(\kk,G) = G(\kk[s])$ is a closed algebraic subset, hence an affine variety.
\item
The maps $\eps_{G}$ and $\nu_{G}$ induce $G$-equivariant isomorphisms
\[ 
\tilde\eps_{G}\colon  \Hom(\kplus,G) \simto G^{u} \quad\text{and}\quad \tilde\nu_{G}\colon \Hom(\kplus,G) \simto \gn.
\]
\item
The composition $\tilde\eps_{G} \circ ( \tilde\nu_{G}) ^{-1}\colon \gn \simto \GGG^{u}$ is equal to the exponential map $\exp_{G}$.
\ee
\[
\begin{tikzcd}
\Hom(\kplus,G)
\arrow[r,"\tilde\eps_G","\simeq"']
\arrow[d,"\tilde\nu_G"',"\simeq"]
& G^{u} \\
\gn \arrow[ru,"\exp_G"'] \\
\end{tikzcd}
\]
\end{proposition}\label{Hom-kplus.prop}

\begin{proof}
(1) It is easy to see that $\Hom(\kplus,G)$ is closed in $\Mor(\kk,G)$. Now we embed $G$ into $\GL_{m}$ in such a way that $G$ is closed in $\M_{m}$.  By Lemma~\ref{closed-immersion-Mor.lem}
we have the following closed immersions of ind-varieties:
$$
\Hom(\kplus,G) \subseteq \Mor(\kk, G) \subseteq \Mor(\kk,\M_{m}) = \M_{m}(\kk[s]).
$$
Since every homomorphism $\lambda\colon \kplus \to G$ has the form $\lambda_{N}(s) = \sum_{j=0}^{m-1}\frac{s^{j}}{j!} N^{j}$ for a suitable nilpotent matrix $N \in \M_{m}$ it follows that $\Hom(\kplus,G)$ lies in $\M_{m}(\kk[s]_{<m})$ which is a finite-dimensional subspace of $\M_{m}(\kk[s])$. Thus $\Hom(\kplus,G)$ is a closed algebraic subset of $\M_{m}(\kk[s])$.
\ps
(3) 
This follows immediately from the definitions of $\tilde \eps_{G}$ and $\tilde \nu_{G}$.
\ps
(2) It is clear that both maps $\tilde\eps_G$ and $\tilde\nu_G$
are $G$-equivariant morphisms. By definition, the inverse of $\tilde \nu_{G}$ is given by $N \mapsto \lambda_{N}$ which is a morphism since it is the restriction of the morphism $\M_{m} \to \M_{m}(\kk[s])$, $A \mapsto \sum_{j=0}^{m-1}\frac{s^{j}}{j!}A^{j}$.
Hence $\tilde \nu_{G}$ is an isomorphism, and $\tilde \eps_G$ is an isomorphism, because the composition $\tilde \eps_{G} \circ (\tilde \nu_{G})^{-1}$ is an isomorphism, by (3).
\end{proof}

The next result is due to \name{Kostant}, see \cite[Theorem~16 on page~392]{Ko1963Lie-group-represen}.
\begin{theorem}\label{kostant.thm}
For a reductive group $G$ with Lie algebra $\gg = \Lie G$ the nilpotent cone $\gn\subset\gg$ is a normal complete intersection.
\end{theorem}

The exponential map from Proposition~\ref{exponential-for-linear-algebraic-groups.prop} then gives the next result.
\begin{corollary} \label{kostant.cor}
For a reductive group $G$ the unipotent elements $G_u \subset G$ form a closed normal subvariety.
\end{corollary}

\ps
\subsection{Vector fields and invariant subvarieties}\label{VF.subsec}
Let $X$ be  an affine variety. As already mentioned in Section~\ref{Endo-VF.subsec}, we will always identify  the vector fields $\VEC(X)$ with the derivations $\Der_{\kk}(\OOO(X))$ of $\OOO(X)$.

\begin{definition}\label{delta-invariance.defn}
Let $X$ be an affine variety, and let $\delta \in \VEC(X)$ be a vector field on $X$.  
\be
\item A locally closed subset  $Y\subseteq X$ is called {\it $\delta$-invariant}\idx{del@$\delta$-invariant},  if $\delta(y)\in T_{y}Y$ for all $y \in Y$. In this case one also says that  $\delta$ is \itind{parallel to $Y$}.\idx{invariant closed subset}
\item A linear subspace $U \subseteq \OOO(X)$ is called {\it $\delta$-invariant} if $\delta(U) \subseteq U$.
\ee
\end{definition}
It is easy to see that $Y$ is $\delta$-invariant if and only if the closure $\overline{Y}$ is $\delta$-invariant, and this is equivalent to the condition that the ideal $I(Y) \subseteq \OOO(X)$ is $\delta$-invariant.

\begin{proposition}\label{xi-invariant.prop}
Let $X$ be an affine variety and $\delta \in \VEC(X)$ a vector field.
\be
\item If $Y_{i}\subseteq X$ are $\delta$-invariant closed subvarieties, then $\bigcap_{i}Y_{i}$ is  $\delta$-invariant.
\item If $Y \subseteq X$ is $\delta$-invariant, then every irreducible component of $Y$ is $\delta$-invariant.
\ee
\end{proposition}
This follows from the next more general result.
\begin{lemma} 
Let $R$ be a finitely generated $\kk$-algebra, and let $\delta \in \Der_{\kk}(R)$ be a derivation.
\be
\item If $I \subseteq R$ is a $\delta$-invariant ideal, then the radical $\sqrt{I} \subseteq R$ is also $\delta$-invariant.
\item If $I \subseteq R$ is a $\delta$-invariant ideal, then the minimal primes containing $I$ are also $\delta$-invariant.
\ee
    \end{lemma}\idx{delta@$\delta$-invariant ideal}
\begin{proof}
(1) Replacing $R$ by $R/I$ it suffices to show that if $f^{n}=0$, then $(\delta f)^{m}=0$ for some $m>0$. For this we can assume, by induction, that $n=2$. Then we find 
$$
0 = \delta ( f^{2} ) = 2f \cdot \delta f,
$$
hence
$$
0 = \delta f \cdot (\delta^{2}f^{2}) = 2 \delta f \cdot ((\delta f)^{2} + f\cdot \delta^{2}f) = 2 (\delta f)^{3},
$$
and so $\delta f \in \sqrt{I}$.

(2) By (1) we can assume that $I = \sqrt{I} = (0)$, hence $(0) = \pp_{1}\cap\ldots\cap \pp_{k}$ where the $\pp_{i}$ are the minimal primes of $\OOO(X)$. For every $i$ choose an element $p_{i}\in\bigcap_{j\neq i}\pp_{j}\setminus\pp_{i}$. Then $\pp_{i}=\{p\in\OOO(X)\mid p_{i}\cdot p = 0\}$, and the same holds for every power of $p_{i}$. For every $p \in \pp_{i}$ we find
$$
0 = p_{i}\cdot \delta(e_{i}\cdot p) = p_{i}\cdot(e_{i}\cdot \delta p + \delta p_{i}\cdot p) = p_{i}^{2}\cdot \delta p,
$$
hence $\delta p \in \pp_{i}$.
\end{proof}
A fundamental result in this context is the following strong relation between $G$-stability and $\Lie G$-invariance.

We will prove this in the more general setting of an action of a connected ind-group on an affine variety (Proposition~\ref{G-stable-is-LieG-invariant.prop}).

\begin{proposition}\label{stable-is-invariant-algGroups.prop}
Let $G$ be a connected linear algebraic group acting on an affine variety $X$, and let $Y\subseteq X$ be a closed subvariety. Then $Y$ is $G$-stable if and only if $Y$ is $\xi_{A}$-invariant for all $A \in \Lie G$.
\end{proposition}\idx{xi@$\xi_{A}$-invariant}

\begin{remark}\label{Seidenberg.rem}
\name{Seidenberg} has shown that the singular locus $X_{\text{\it sing}} \subseteq X$ is invariant under all vector fields $\delta \in \VEC(X)$, see \cite{Se1967Differential-ideal}. Moreover, he has proved
that if a strict closed subvariety $Y \subsetneqq X$ is invariant under all vector fields, then $Y \subseteq X_{\text{\it sing}}$.
\end{remark}

\ps
\subsection{Locally finite representations on vector fields}\label{LF-Jordan-decomposition.subsec}
Every automorphism $\phi$ of an affine variety $X$ induces a linear automorphism  of the regular functions $\OOO(X)$ and a linear automorphism of the vector fields $\VEC(X)$, both denoted by $\phi$. For $f \in \OOO(X)$ we have $\phi(f)\colon x \mapsto f(\phi^{-1}(x))$, i.e. we take for $\phi\colon \OOO(X) \to \OOO(X)$ the inverse of the comorphism $\phi^{*}\colon \OOO(X) \to \OOO(X)$. 

For a vector field $\delta\in\VEC(X)$ there are two ways to describe the action of $\phi$. If we consider $\delta$ as a derivation of $\OOO(X)$, then $\phi(\delta):=(\phi^{*})^{-1}\circ\delta\circ\phi^{*}$. If we see $\delta = (\delta_{x})_{x\in X}\colon X \to TX$ as a section of the \itind{tangent bundle} $TX$, then we have 
$$
\phi(\delta)_{\phi(x)}= d\phi_{x}(\delta_{x})\quad\text{or}\quad
\phi(\delta)_{x}= d\phi_{\phi^{-1}(x)} ( \delta_{\phi^{-1}(x)}):
$$

\begin{center}
\begin{tikzcd}
TX  \arrow[d, xshift=-0.2em, "p"']  \arrow[r, "d\phi", "\simeq"'] & TX  \arrow[d, xshift=-0.2em, "p"']  \\
X  \arrow[r,"\phi", "\simeq"'] \arrow[u, xshift=0.2em, "\delta"'] & X  \arrow[u, xshift=0.2em, "\phi(\delta)"']
\end{tikzcd}
\end{center}
\ps\noindent
For $f\in\OOO(X)$ and $\delta\in\VEC(X)$ we get $\phi(f\delta) = \phi(f)\phi(\delta)$ and $\phi(\delta f)=\phi(\delta)\phi(f)$.

\begin{example}\label{image-of-ddx.exa}
Let $\phi=(f_{1},\ldots,f_{n})\colon \An \to \An$ be an automorphism. Then we find
$$
\phi(\dxj)   = \sum_{i}\frac{\partial{f_{i}}}{\partial x_{j}}(\phi^{-1}(x)) \dxi. 
$$
In fact,
\begin{align*}
\phi(\dxj)  (x_{i})  & =\phi(\dxj)   \phi ( \phi^{-1} (x_i) ) =  \phi(\dxj \phi^{-1}(x_{i})) \\
   & = \phi (\frac{\partial{f_{i}}}{\partial x_{j}}) = (\frac{\partial{f_{i}}}{\partial x_{j}}) ( \phi^{-1}(x) ) .
\end{align*}
\end{example}

If a  group $H$ acts on $X$ by algebraic automorphisms, then we obtain in this way linear  representations of $H$ on the coordinate ring $\OOO(X)$ and on the vector fields $\VEC(X)$. These representations will play an important role in the following.

\begin{definition}\label{locally-finite.def}
A linear endomorphism $\phi\in\LLL(V)$ of a $\kk$-vector space $V$ is called \itind{locally finite} if the linear span $\langle \phi^{j}(v)\mid j\in\NN\rangle$ is finite-dimensional for all $v\in V$. It is called \itind{semisimple} if it is locally finite and if the action on every finite-dimensional $\phi$-stable subspace is semisimple. It is called \itind{locally nilpotent} if for any $v \in V$ there is an $m\in\NN$ such that $\phi^{m}(v)=0$. 
\end{definition}

Recall that every locally finite endomorphism $\phi$ has a uniquely defined \itind{Jordan decomposition} $\phi = \phi_{s} + \phi_{n}$ where $\phi_{s}$ is semisimple, $\phi_{n}$ is nilpotent and $\phi_{s}\circ\phi_{n}= \phi_{n}\circ\phi_{s}$.

\begin{definition}\label{ratrional-rep.def}
A subset $S \subseteq \LLL(V)$ is called {\it locally finite}\idx{locally finite subset} if every element $v \in V$ is contained in a finite-dimensional $S$-stable subspace. A representation $\rho$ of a linear algebraic group $G$ on $V$ is called {\it locally finite and rational}\idx{locally finite and rational representation} if $\rho(G)$ is locally finite and the induced representation of $G$ on any $G$-stable finite-dimensional subspace is rational.
\end{definition}

\begin{proposition}\label{locally-finite-group-actions.prop}
Let $G$ be a linear algebraic group, and let $X$ be an affine $G$-variety. Then the representations of $G$ on the coordinate ring $\OOO(X)$ and on the vector fields $\VEC(X)$ are locally finite and rational.
\end{proposition}

\begin{proof}
This is well-known for the regular representation of $G$ on $\OOO(X)$ given by $gf(x):= f(g^{-1}x)$. If we identify $\VEC(X)$ with the derivations $\Der_{\kk}(\OOO(X))$ considered as a subspace of $\LLL(\OOO(X))$, the linear endomorphisms of $\OOO(X)$,
then the $G$-action is the obvious one:
$$
g(\delta) = g\circ\delta\circ g^{-1}, \text{ \ or \ }
g(\delta)(f) := g(\delta(g^{-1}f)) \text{ for }g \in G, f \in\OOO(X).
$$
Choose a finite-dimensional $G$-stable subspace $U \subseteq \OOO(X)$ which generates $\OOO(X)$. Then we have a $G$-equivariant embedding $\Der_{\kk}(\OOO(X)) \subseteq \LLL(U, \OOO(X))$ where the representation on $\LLL(U, \OOO(X))$ is given by $\phi\mapsto g\circ\phi\circ g^{-1}$. This action is locally finite and rational since it is locally finite and rational on $\OOO(X)$.
\end{proof}

\begin{remark}\label{canonical-iso-VF.rem}
Let $V$ be a $G$-module. Let $v_{1},\ldots,v_{n}\in V$ be a basis and denote by $x_{1},\ldots,x_{n}\in\OOO(V)$ the dual basis. We have a canonical $G$-equivariant linear isomorphism of $\OOO(V)$-modules
$$
\VEC(V) \simto V\otimes \OOO(V), \quad \delta\mapsto \sum_{i}v_{i}\otimes \delta(x_{i})
$$
which does not depend on the choice of a basis. The inverse isomorphism is induced by $v\otimes f\mapsto f\partial_{v}$. 
\end{remark}

There is the following converse of Proposition~\ref{locally-finite-group-actions.prop} above.
\begin{proposition}\label{algebraic-subgroups.prop}
Let $H$ be an abstract group acting on the affine variety $X$ by automorphisms, i.e. via a homomorphism $\rho\colon H \to \Aut(X)$. Assume that the induce representation of $H$ on $\OOO(X)$ is locally finite. Then the closure $\overline{\rho(H)} \subseteq \Aut(X)$ is a linear algebraic group. The same conclusion holds if the induced representation on $\VEC(X)$ is locally finite.
\end{proposition}
\begin{proof} The first part is easy, and the proof is left to the reader. As for the second, we have to show that if the representation on $\VEC(X)$ is locally finite, then  so is the representation  on $\OOO(X)$. 

Let $\delta_{0}\in\VEC(X)$ be vector field which is nonzero on a dense open set (Corollary~\ref{nonzero-VF.cor}), and let $W:=\langle h\delta_{0}\mid h\in H\rangle$ be the finite-dimensional linear span of all $h\delta_{0}$. 
Assume that $U:=\langle hf_{0} \mid h \in H\rangle$ is not finite-dimensional for some $f_{0}\in\OOO(X)$. Define $V:=\langle h(f_{0}\delta)\mid h\in H, \delta\in W\rangle$.
This is again finite-dimensional. Since $hf_{0}\delta_{0} = h(f_{0}h^{-1}\delta_{0})$ we see that $V$ contains $U \delta_{0}$. But this space is infinite-dimensional, because the map $f\mapsto f\delta_{0}$ is injective (Corollary~\ref{nonzero-VF.cor}).
\end{proof}

\ps
\subsection{Algebraic group actions and vector fields}\label{group-action-VF.subsec}
Let $G$ be a linear algebraic group acting on an affine variety $X$. Generalizing the construction from Section~\ref{linear-case.subsec}, we obtain a canonical anti-homomorphism of Lie algebras \idx{$\xi : \Lie G \to \VEC(X)$}
$$
\xi\colon\Lie G \to \VEC(X), \quad A\mapsto \xi_{A}
$$ 
where the vector field \itind{$\xi_{A}$} is defined in the following way (see Proposition~\ref{End(X)-and-Vec(X).prop}): Consider the {\it orbit map} $\mu_{x}\colon G \to X$, $g\mapsto gx$, and set
$$
(\xi_{A})_{x} := (d\mu_{x})_{e}(A).
$$\idx{orbit map $\mu_{x}$}
In order to see that this is indeed an ``algebraic'' vector field on $X$, we note that the action is given by a homomorphism $\rho\colon G \to \Aut(X)\subseteq\End(X)$ which defines a homomorphism of Lie algebras $d\rho\colon \Lie G \to \Lie\Aut(X) \subseteq T_{\id}\End(X)$. It then follows from the construction that $\xi_{A}$ is the image of $d\rho (A) \in T_{\id}\End(X)$ under the linear map $\xi\colon T_{\id}\End(X) \to \VEC(X)$ defined in Section~\ref{Endo-VF.subsec}, see Proposition~\ref{End(X)-and-Vec(X).prop}. 

We could also use a $G$-equivariant closed immersion $X \subseteq V$ into a $G$-module $V$. This defines a homomorphism $\rho\colon G \to \GL(V)$ and thus a Lie algebra homomorphism $d\rho\colon \Lie G \to \Lie\GL(V) = \LLL(V)$. By the considerations in Section~\ref{linear-case.subsec} we obtain a vector field $\delta:=\xi_{d\rho(A)} \in \VEC(V)$ for every $A\in\Lie G$. It is not difficult to see that $X$ is $\delta$-invariant and that $\delta|_{X}=\xi_{A}$.

\ps
On the other hand, we have a locally finite and rational representation of $G$ on the coordinate ring $\OOO(X)$ and on the vector fields $\VEC(X)=\Der_{\kk}(\OOO(X))$ (Proposition~\ref{locally-finite-group-actions.prop}). This linear representation of $G$ on the vector fields induces a linear action of $\Lie G$ on the $\VEC(X)$  which relates to the construction above in a well-known manner.

\begin{proposition}
Let $G$ be a linear algebraic group and $X$ an affine $G$-variety. Then the locally finite and rational representation of $G$ on the vector fields $\VEC(X)$ induces a linear action of $\Lie G$ on $\VEC(X)$, $(A,\delta) \mapsto A\delta$, which is given by
\[
A\delta = [\delta,\xi_{A}]=-\ad \xi_{A}(\delta) \text{ \ for } A \in \Lie G \text{ and } \delta \in \VEC(X).
\]
\end{proposition}
\begin{proof}
It suffices to consider the case of a linear action of $G$ on a $\kk$-vector space $V$. Then $\VEC(V) = \Mor(V,V) = \OOO(V)\otimes V$ where $f\otimes v$ corresponds to the derivation $(f\otimes v)(h)=f\frac{\partial h}{\partial v}$, and the Lie bracket is given by $[f\otimes v,h\otimes w] = f\frac{\partial h}{\partial v} \otimes w - h \frac{\partial f}{\partial w} \otimes v$. Moreover, for $A\in \Lie G$ the vector field $\xi_{A}$ belongs to $V^{*}\otimes V \subseteq \OOO(V)\otimes V$, $\xi_{A}=\sum_{i} \ell_{i}\otimes w_{i}$, and so 
$$
[f\otimes v , \xi_{A}] = \sum_{i} \left(f\frac{\partial \ell_{i}}{\partial v}\otimes w_{i}-\ell_{i}\frac{\partial f}{\partial w_{i}}\otimes v\right) =
f\otimes \left(\sum_{i}\frac{\partial \ell_{i}}{\partial v}w_{i}\right) - \left(\sum_{i}\ell_{i}\frac{\partial f}{\partial w_{i}}\right)\otimes v.
$$
On the other hand, for $A \in \Lie G$ we have $A(f\otimes v)=Af \otimes v + f\otimes Av$. 
\end{proof}
Now let us go one step further and assume that there is a fixed point $x_{0}\in X$. Then we get the \itind{tangent representation} 
$\tau=\tau_{x_{0}} \colon G \to \GL(T_{x}X)$ and its differential $d\tau\colon \Lie G \to \End(T_{x}X)$ which defines a linear action of
$\Lie G$ on $T_{x}X$, denoted by $(A,v) \mapsto A(v):=d\tau(A)(v)$.

\begin{corollary}\label{tangent-rep-and-VF.cor}
Assume that the $G$-action on $X$ has a fixed point $x_{0}$. 
Then,  for  $A \in \Lie G$ and $\delta\in\VEC(X)$, we have
$$
A\delta_{x_{0}}=d\tau(A)(\delta_{x_{0}}) = -[\xi_{A},\delta]_{x_{0}}.
$$
\end{corollary}
\begin{proof}
(1) If $X$ is a $\kk$-vector space $V$ with a linear action of $G$, given by $\rho\colon G \to \GL(V)$, then the tangent representation 
$\tau_{0}$ is equal to $\rho$, and  $d\tau(A) = \xi_{A}$ (see Section~\ref{linear-case.subsec}). Hence the claim follows from 
the proposition above. 
\ps
(2)
In general, we can assume that 
$X$ is a closed $G$-stable subset of a $G$-module $V$, $X\subseteq V$, and that $x_{0}= 0 \in V$. Then,
for every $A \in \Lie G$, the subvariety $X$ is invariant under $\xi_{A,V}\in \VEC(V)$, and the restriction of $\xi_{A,V}$ 
to $\VEC(X)$ is $\xi_{A}$, by construction. Moreover, $T_{x_{0}}\subset T_{0}V = V$ is stable under $\rho\colon G \to \GL(V)$, and
$d\tau(A) = d\rho(A)|_{T_{x_{0}}}$. Now the claim follows from the linear case (1).
\end{proof}

\pmed
\section{Ind-Group Actions and Vector Fields}\label{ind-group-actions.sec}
We now extend the results of the previous section to actions of  ind-groups. We also define locally finite, semisimple and unipotent elements of ind-groups and  obtain bijections between unipotent automorphisms of an affine variety $X$, locally nilpotent vector fields on $X$, and $\kplus$-actions on $X$.
\ps
\subsection{Orbits of ind-groups}\label{group-actions-orbits.subsec}

\begin{definition}\label{ind-group-actions.def}
An action of an ind-group $\GGG$ on an ind-variety $\VVV$ is a homomorphism $\rho\colon \GGG \to \Aut(\VVV)$ such that the action map $\GGG \times \VVV \to \VVV$, $(g,x)\mapsto g   x :=\rho(g)x$,  is an ind-morphism.
\end{definition}

The following proposition generalizes a well-known result for algebraic groups actions on varieties, namely that the orbits are always open in their closure.

\begin{proposition}\label{ind-group-action.prop}
Let $\GGG = \bigcup_{k}\GGG_{k}$ be a connected ind-group acting on a variety $X$. Then, the following assertions are satisfied.
\be
\item For all $x \in X$, the $\GGG$-orbit\idx{orbit} $\GGG x$ is open in its closure.\idx{G-@$\GGG$-orbit}
\item There is an $\ell\geq 1$ such that $\GGG  x = \GGG_{\ell}  x$ for all $x\in X$.
\item If the stabilizer  $\GGG_{x}=\Stab_{\GGG}x$ is an algebraic group for some $x\in X$, then $\GGG$ is an algebraic group.
\ee
\end{proposition}
Again, the connectedness of $\GGG$ is necessary, as we will see in Section~\ref{braid.subsec} where we construct a discrete ind-group acting faithfully on an affine variety. 
\begin{proof}
(1)--(2) 
Consider the morphism
$$
\Phi\colon  \GGG\times X \to X \times X \  \text{ given by } \ (g,x)\mapsto (g  x,x),
$$
and denote by $P = \Phi(\GGG \times X) \subset X\times X$ its image. Note that $\GGG \times \GGG$ acts naturally on $X \times X$ and that $P$ is $\GGG \times \GGG$ stable.
If $p\colon P \to X$ is the map induced by the second projection, then
$p^{-1}(x) = \GGG  x \times \{x\}$. We will show that there is an open dense set $U \subset X$ with the following properties:
\be
\item [(i)] $U$ is $\GGG$-stable;
\item [ (ii)] $P_{U}:=\Phi(\GGG \times U) \subset U \times U$ is open in $\overline{P}$;
\item [ (iii)] There is an $\ell\geq 1$ such that $P_{U} = \Phi(\GGG_{\ell}\times U)$.
\ee
Let us show first that this implies (1) and (2). Since $\dim (X \setminus U) < \dim X$, using induction on $\dim X$, it is enough to prove (1) and (2) when $x \in U$. As above, it follows from the construction that the second projection  induces a surjective morphism $p_{U}\colon P_{U} \to U$ such that $p_{U}^{-1}(x) = \GGG  x \times \{x\}$. Since $P_{U}$ is open in $\overline{P}$ by (ii), we see that $P_U \cap ( X \times \{ x\}) = \GGG x \times \{ x \}$ is open in $\overline{P} \cap ( X \times \{ x \} )$, hence open in $\overline{\GGG x} \times \{x \}$. We have proved (1). Now (iii) implies that $\GGG  x \times \{x\}= \Phi(\GGG\times \{x\}) = \Phi(\GGG_{\ell}\times \{x\}) = \GGG_{\ell}  x \times\{x\}$ and we have proved (2). 

\ps
It remains to construct the open set $U \subset X$ with the properties (i)--(iii). By Lemma~\ref{image-ind-var-in-var.lem} there is a  subset $V \subset P$ which is open and dense in $\overline{P}$. Replacing $V$ by $(\GGG \times \GGG)  V$ we can assume that $V$ is $(\GGG \times \GGG)$-stable. It follows that  the image $U :=\pr_{2}(V) \subset X$ is an open dense subset and that the fibers of $V  \to U$ are of the form $\GGG  x\times\{x\}$, hence $V= \Phi(\GGG\times U)$. Using again Lemma~\ref{image-ind-var-in-var.lem} there is an $\ell \geq 1$ such that  $V =\Phi(\GGG_{\ell}\times U )$, i.e., we have $\GGG  x = \GGG_{\ell}  x$ for all $x \in U$.
\ps
(3) The claim follows from Proposition~\ref{small-fibers-gives-variety.prop} applied to the orbit map  $\GGG \to X$, $g \mapsto g x$.
\end{proof}

Next we give two examples of actions of an ind-group $\GGG$ on an ind-variety $\VVV$ with orbits which are not locally closed.

\begin{example}\label{non-locally-closed-orbit.exa1}
Consider the natural action of $\SLtwo (\kk [x])$ on $\kk[x] ^2$ given by left multiplication. Then, the orbit $\CCC$ of $(1,0)$ is dense, but not open. Indeed, ${\CCC}$ is equal to the set of pairs $(a,b) \in \kk[x]^2$ such that $a$ and $b$ are coprime. Let us check that $\overline{\CCC}= \kk [x]^2$. Each element $(a,b)$ of $\kk[x]^2$ belongs to the orbit of an element of the form $(p,0)$, where $p \in \kk [x]$ is the gcd of $a$ and $b$. It is therefore enough to check that $(p,0)$ belongs to $\overline{\CCC}$. If $\varepsilon \in \kk^*$, then $(p, \varepsilon)$ belongs to $\CCC$, hence  $(p,0)=\lim_{\varepsilon \to 0} (p, \varepsilon) \in \overline{\CCC}$. 

For $\varepsilon \in \kk^*$ we have $(1+ \varepsilon x, \varepsilon ( 1 + \varepsilon x) ) \in \kk[x]^{2} \setminus {\CCC}$. However, its limit for $\varepsilon \to 0$ is equal to $(1,0)$ which does not belong to $\kk[x]^{2} \setminus {\CCC}$. Therefore, $\overline{\CCC} \setminus {\CCC}$ is not closed, and so ${\CCC}$ is not open.
\end{example}

\begin{example}\label{non-locally-closed-orbi.exa2} An element $v \in \kk[x,y]$ is called a
\itind{variable} or a \itind{coordinate} if there exists a  $w \in \kk[x,y]$ such that $\kk[v,w]=\kk[x,y]$. It is clear that the subset $C \subseteq \kk[x,y]$ of variables coincides with the orbit of $x$ under $\Aut(\Atwo)$. By \cite{Fu2002On-the-length-of-p}, we have
\[ 
\overline{\CCC} =\{ p(v) \mid p \in \kk [t], \, v \in \VVV \},
\]
see Corollary~\ref{closure-of-variables.cor}. We claim that $\overline{\CCC} \setminus{\CCC}$ is not closed. In fact, for $\varepsilon \in \kk^*$ the polynomial $x + \varepsilon x^2$ belongs to $\overline{\CCC} \setminus {\CCC}$ (it is not a variable since it is not irreducible), but its limit $\lim_{\varepsilon \to 0} (x+ \varepsilon x^2) = x$ is a variable. This proves the claim and shows that the orbit ${\CCC}$ is not locally closed.
\end{example}

\ps
\subsection{Actions of ind-groups and vector fields}  \label{action-ind-groups.subsec}
Our previous considerations relating algebraic group actions and vector fields carry over to an action of an affine ind-group $\GGG$ on an ind-variety $\VVV$. We first have to define vector fields on $\VVV$. We will only need this for affine ind-varieties, and so we restrict to this case. 

\begin{definition}\label{VF.def}
A \itind{vector field} $\delta$ on an affine ind-variety $\VVV = \bigcup_{k}\VVV_{k}$ is a collection
$\delta=(\delta_{v})_{v \in \VVV}$ of tangent vectors $\delta_{v}\in T_{v}\VVV$
with the property that for every $k\geq 1$ there is an $\ell \geq k$ such that the following holds:
\be
\item[(i)] $\delta_{v}\in T_{v} \VVV_{\ell}$ for all $v \in \VVV_{k}$;
\item[(ii)] For every $f \in \OOO(\VVV_{\ell})$ the function
$\delta f\colon \VVV_k \to \kk$, $v \mapsto \delta_v(f)$, is regular on $\VVV_k$.
\ee
We denote by $\VEC(\VVV)$ the $\kk$-vector space of vector fields on $\VVV$.\idx{$\VEC(\VVV)$}
\end{definition}
If $X$ is an affine variety and $Y \subseteq X$ a closed subvariety, we define $\VEC(X,Y)$ to be the collections $\delta=(\delta_{y})_{y\in Y}$ where $\delta_{y}\in T_{y}X$ such that, for every $f\in\OOO(X)$, the function $\delta f\colon y \mapsto \delta_{y}f$ is regular on $Y$. It is easy to see that $\VEC(X,Y) = \Der_{\kk}(\OOO(X),\OOO(Y)$) where $\OOO(Y)$ is considered as an $\OOO(X)$-module via the restriction map $f\mapsto f|_{Y}$. With this definition, condition (ii) above can also be formulated as 
\be
\item[(ii)\textprime] $\delta|_{\VVV_{k}} := (\delta_{v})_{v \in \VVV_{k}} \in \VEC(\VVV_{\ell},\VVV_{k})$.
\ee

\begin{lemma}
If $\delta\in \VEC(\VVV)$ and $f \in \OOO(\VVV)$, then the function $\delta f \colon v \mapsto \delta_{v}f$, is regular on $\VVV$. Moreover,
$\delta \colon \OOO(\VVV) \to \OOO(\VVV)$ is a continuous derivation, and we obtain a canonical isomorphism $\VEC(\VVV) \simto \Der_{\kk}^{\text{\tiny\it cont}}(\OOO(\VVV) )$. In particular, $\VEC(\VVV)$ has the structure of a Lie algebra.
\end{lemma}\idx{$\Der_{\kk}^{\text{\it cont}}(\OOO(\VVV) )$}
\begin{proof}
By our definition, we have $(\delta f)|_{\VVV_{k}} \in \OOO(\VVV_{k})$ for each $k$, i.e. $\delta f \in \OOO(\VVV)$.
In order to see that the vector fields  are exactly the {\it continuous derivations}\idx{continuous derivation} we just recall that a derivation $\delta \colon \OOO(\VVV) \to \OOO(\VVV)$ is continuous if for every $k \geq 1$ there is an $\ell \geq k$ such that the composition $\res_{k}\circ\delta$ factors through $\res_{\ell}$:
$$
\begin{CD}
\OOO(\VVV) @>{\delta}>> \OOO(\VVV) \\
@VV{\res_{\ell}}V   @VV{\res_{k}}V  \\
\OOO(\VVV_{\ell}) @>{\delta_{k\ell}}>> \OOO(\VVV_{k})
\end{CD}
$$
The composition $\res_{k}\circ\delta$ corresponds to $\delta|_{\VVV_{k}}$, and the factorization of $\res_{k}\circ\delta$ in the form $\delta_{k\ell}\circ \res_{\ell}$ means that $\delta|_{\VVV_{k}}\in \VEC(\VVV_{\ell},\VVV_{k})$.
Hence we have a canonical identification $\VEC(\VVV) = \Der_{\kk}^{\text{\tiny\it cont}}(\OOO(\VVV))$.

Finally, it is clear that the bracket $[\delta,\mu] := \delta\circ\mu - \mu\circ \delta$ of two continuous derivations $\delta$ and $\mu$ of $\OOO(\VVV)$ is again a continuous derivation.
\end{proof}

If  $\WWW \subset \VVV$ is a closed ind-subvariety, we define $\VEC ( \VVV, \WWW)$ in the obvious way, and we get  $\VEC(\VVV, \WWW) = \Der_{\kk}^{\text{\tiny\it cont}}(\OOO(\VVV), \OOO (\WWW) )$ where $\OOO(\WWW)$ is considered as an $\OOO( \VVV)$-module via the restriction map $f \mapsto f|_{\WWW}$. 
Note that $\VEC (\VVV)  = \VEC(\VVV, \VVV)$. From this definition we see that 
$$
\VEC(\VVV, \VVV_{k}) = \bigcup_{\ell\geq k}\VEC(\VVV_{\ell},\VVV_{k}) 
\text{ for all $k$, and so } \VEC(\VVV) = \varprojlim_{k}\VEC(\VVV,\VVV_{k}).
$$
The case of a $\kk$-vector space of countable dimension
$V$ is easy. Here we have $\VEC(V) = \Mor(V,V)$. If $\WWW \subset V$ is a closed ind-subvariety, then there is a canonical surjective linear map $\VEC_{\WWW}(V) \to \VEC(\WWW)$ where
$$
\VEC_{\WWW}(V) := \{\delta=(\delta_{v})_{v\in V}\in \VEC(V) \mid \delta_{w} \in T_{w}\WWW \text{ for all }w \in \WWW\}.
$$
Now consider an action of an ind-group $\GGG$ on the ind-variety $\VVV$. Recall that this means that  
the action map $\eta\colon\GGG\times\VVV \to \VVV$ is a morphism (Definition~\ref{ind-group-actions.def}). As in the algebraic case (see Section~\ref{group-action-VF.subsec}) we make the following definition.\idx{$\xi_{A}$}

\begin{definition} \label{xi_A.def}
Every $A\in\Lie\GGG$ defines a vector field $\xi_{A}$ on $\VVV$,
\[
\xi_{A}(x) := (d\mu_{x})_{e}A \text{ \ for } x\in \VVV
\]
where $\mu_{x}\colon \GGG \to \VVV$ is the \itind{orbit map} $g\mapsto gx$.
\end{definition}

If $\VVV$ is affine, we have the following  description of the corresponding (continuous) derivation $\xi_{A}$ of $\OOO(\VVV)$ (cf. Remark~\ref{left-invariant-VF.rem}),
\[
\begin{CD}
\xi_{A}\colon \OOO(\VVV) @>\eta^{*}>> \OOO(\GGG) \hotimes \OOO(\VVV) @>{A\hotimes \id}>> \OOO(\VVV),
\end{CD}
\]
where we use again that the element $A \in\Lie\GGG = T_e\GGG$ may be regarded as a continuous derivation $A \colon \OOO (\GGG) \to \kk$ in  $e \in \GGG$ (see Section~\ref{tangent-space.subsec}).
Now we get the following result (see Section~\ref{group-action-VF.subsec}). \idx{$\xi : \Lie\GGG \to \VEC(\VVV)$}

\begin{proposition}\label{Liealg-VF.prop}
The map $\xi\colon \Lie\GGG \to \VEC(\VVV)$,  $A \mapsto \xi_{A}$, is an anti-homomor\-phism of Lie algebras. If $X$ is an affine variety, then  $\xi\colon \Lie\Aut(X) \to \VEC(X)$ is injective.
\end{proposition}
\begin{proof}

It follows from the description of the left-invariant vector field $\delta_{A}$ (Section~\ref{Liealgebra.subsec}) that the vector field corresponding to the right-action of $\GGG$ on itself, $(g,h)\mapsto hg^{-1}$, is equal to $-\delta_{A}$.  Now consider the isomorphism
$$
\phi\colon \GGG\times \VVV \simto \GGG \times \VVV, \ (g,x) \mapsto (g,gx),
$$
which is $\GGG$-equivariant with respect to the actions $g(h,x):=(hg^{-1},gx)$ on the first space and $g(h,x)=(hg^{-1},x)$ on the second. Take a tangent vector $A \in T_{e}\GGG$ and denote by 
$\zeta^{(i)}_{A}$ the corresponding vector fields on the two spaces. Clearly, $d\phi (\zeta^{(1)}_{A})= \zeta^{(2)}_{A}$. Moreover, $(\zeta^{(1)}_{A})_{(g,x)}= ((-\delta_{A})_{g},(\xi_{A})_{x})$ and 
$(\zeta^{(2)}_{A})_{(g,x)} = ((-\delta_{A})_{g},0)$. It follows that $\zeta^{(2)}_{[A,B]}=-[\zeta^{(2)}_{A},\zeta^{(2)}_{B}]$, hence  $\zeta^{(1)}_{[A,B]}=-[\zeta^{(1)}_{A},\zeta^{(1)}_{B}]$, because $d\phi$ is an isomorphism of the Lie algebras of vector fields. As a consequence,
\begin{align*}
(-\delta_{[A,B]},\xi_{[A,B]}) & =\zeta^{1}_{[A,B]}=-[\zeta^{1}_{A},\zeta^{1}_{B}]  = -[(-\delta_{A},\xi_{A}),(-\delta_{B},\xi_{B})]\\ 
& = (-[\delta_{A},\delta_{B}],-[\xi_{A},\xi_{B}]),
\end{align*}
and the claim follows.
\par\smallskip
For the second claim we use Proposition~\ref{End(X)-and-Vec(X).prop} which shows that $T_{e}\End(X) \into \VEC(X)$ is an embedding. By Theorem~\ref{AutX-locally-closed-in-EndX.thm} we know that $\Aut(X) \subseteq \End(X)$ is locally closed. Thus the map $A \mapsto \xi_{A}$ is injective.
\end{proof}

\begin{remark}\label{formula.rem}
If $\phi\colon \GGG \times \VVV \to \VVV$ is the action map, then the differential $d\phi$ is given by
\[
d\phi_{(e,x)} \colon \Lie\GGG \oplus T_{x}(\VVV) \to T_{x}\VVV, \ \ 
(A,v) \mapsto (\xi_{A})_{x} + v,
\]
see Lemma~\ref{diff-for-linear-action.rem} where the case of a linear action on a $\kk$-vector space is discussed.
\end{remark}
The next result shows another strong connection between connected ind-groups and their Lie algebras. For algebraic groups, this is well known, see Proposition~\ref{stable-is-invariant-algGroups.prop}.

\begin{proposition}\label{G-stable-is-LieG-invariant.prop}
Let $\GGG$ be a connected ind-group acting on an affine variety $X$, and let $Y \subseteq X$ be a closed subvariety. Then  $Y$ is $\GGG$-stable if and only if $Y$ is $\Lie\GGG$-invariant, i.e., $Y$ is $\xi_{A}$-invariant for all $A \in \Lie\GGG$.
\end{proposition}
\begin{proof}
One direction is obvious: If $Y$ is $\GGG$-stable, then $\xi_{A}(y) \in T_{y}Y$ for all $A \in \Lie \GGG$. So let us assume that $Y$ is $\xi_{A}$-invariant for all $A \in \Lie\GGG$. By Proposition~\ref{xi-invariant.prop}, we can assume that $Y$ is irreducible. Let $\GGG = \bigcup_{k}\GGG_{k}$ where all $\GGG_{k}$ are irreducible (Proposition~\ref{curve-connected.prop}). Then $\overline{\GGG Y} = \overline{\GGG_{k}Y}$ for some $k$, and so the morphism $\phi\colon \GGG_{k} \times Y \to \overline{\GGG Y}$ is dominant. It follows that there is an open dense set $U \subseteq \GGG_{k}\times Y$ such that $d\phi_{u}\colon T_{u}(\GGG_{k}\times Y) \to T_{\phi(u)}\overline{\GGG Y}$ is surjective for all $u \in U$. 

Denote by $Y' \subseteq Y$ the image of $U$ under the projection $\GGG_{k}\times Y \to Y$. For any $y \in Y'$ there is a $g \in\GGG_{k}$ such that $(g,y) \in U$. Consider the diagram
$$
\begin{CD}
\GGG_{k} \times Y @>\phi>> \overline{\GGG Y} \\
@V{\simeq}V{g^{-1}\times\id}V @V{\simeq}V{g^{-1}}V \\
g^{-1}\GGG_{k} \times Y @>\phi>> \overline{\GGG Y}
\end{CD}
$$
It follows that $d\phi_{(e,y)}\colon T_{e}(g^{-1}\GGG_{k})\oplus T_{y}Y \to T_{y} \,\overline{\GGG Y}$ is surjective. On the other hand, we have
$d\phi_{(e,y)}(A,v) = \xi_{A}(y) + v$ (Remark~\ref{formula.rem}),
hence $T_{y}\overline{\GGG Y} = T_{y}Y$ for all $y \in Y'$. Thus $\dim \overline{\GGG Y} = \dim Y$ and so $\GGG Y = Y$.
\end{proof}  

Now consider the case of a representation $\rho\colon \GGG \to \GL(V)$ of $\GGG$ on a $\kk$-vector space $V$ of countable dimension (Definition~\ref{representation-of-an-ind-group.def}). Then the differential $d\rho_{e}\colon \Lie \GGG \to \LLL(V)$ is a Lie-algebra homomorphism, and the corresponding linear action of $\Lie \GGG$ on $V$ is given by $A(v) = (\xi_{A})_{v}$, see Lemma~\ref{rep-of-G-and-LieG.lem}. 
This implies that a subspace $W \subseteq V$ is $\Lie\GGG$-invariant if and only if it is stable under the linear action of $\Lie\GGG$ on $V$ given by $d\rho_{e}$. This proves the following.

\begin{corollary}\label{LieG-stable-is-G-stable.cor}
Let $\rho\colon \GGG \to \GL(V)$ be a representation of a connected  ind-group $\GGG$. Then a subspace $W\subseteq V$ is 
$\GGG$-stable if and only if it is stable under $\Lie\GGG$.
\end{corollary}

\ps
\subsection{Linear representation on \texorpdfstring{$\OOO(X)$}{O(X)} and  \texorpdfstring{$\VEC(X)$}{Vec(X)}}
If an ind-group $\GGG$ acts on an affine variety $X$ we get a linear action of $\GGG$ on the coordinate ring $\OOO(X)$ and on the vector fields $\VEC(X) = \Der_{\kk}(\OOO(X) )$,  defined in the usual way (see Section~\ref{LF-Jordan-decomposition.subsec}):
\begin{gather*}
g f (x) := f(g^{-1}x) \text{ for } f\in\OOO(X), x\in X, \\ 
g(\delta)(f) := g(\delta(g^{-1}f)) \text{ for } \delta\in \VEC(X), f\in\OOO(X). 
\end{gather*}
Note that $\VEC(X)$ is a $\kk$-vector space of countable dimension, because it is a finitely generated $\OOO(X)$-module (Proposition~\ref{vector-fields.prop}).

\begin{proposition} \label{locally-finite-action-on-Vec(X).prop}
The linear actions of $\GGG$ on $\OOO(X)$ and $\VEC(X)$ are representations, i.e., the induced maps $\GGG \times \OOO(X) \to \OOO(X)$ and $\GGG \times \VEC(X) \to \VEC(X)$ are morphisms of ind-varieties.
\end{proposition}

Note that the kernels of the $\GGG$-actions on $X$ and on $\OOO (X)$ are the same. In particular, the automorphism group $\Aut (X)$ acts faithfully on $\OOO (X)$. 

\begin{question} \label{Aut(X)-acts-faithfully-on-vector-fields.ques}
Is it true that $\Aut (X)$ acts faithfully on $\VEC (X)$?
\end{question}

\begin{proof}[Proof of Proposition~\ref{locally-finite-action-on-Vec(X).prop}]
(a) Fix a filtration $\GGG = \bigcup \GGG_{k}$. The action of $\GGG$ on $X$ defines morphisms $\rho_{k}\colon \GGG_{k}\times X \to X$ for every $k$. If $\OOO(X)=\bigcup_{i}\OOO(X)_{i}$ is a filtration by finite-dimensional subspaces, then, for every $i$, there is a $j$ such that $\rho_{k}^{*}(\OOO(X)_{i}) \subseteq \OOO(\GGG_{k})\otimes \OOO(X)_{j}$. This means that for $f\in\OOO(X)_{i}$, $g \in \GGG_{k}$ and $\rho_{k}^{*}(f) = \sum_{\ell}h_{\ell}\otimes f_{\ell}$ we have
\[
g^{-1}f = \sum_{\ell}h_{\ell}(g)f_{\ell} \in \OOO(X)_{j}.
\]
This shows that the linear action $(g,f)\mapsto g^{-1}f$ is a morphism $\GGG_{k}\times \OOO(X)_{i} \to \OOO(X)_{j}$ of affine varieties.
\par\smallskip
(b) Now we choose a filtration $\OOO(X) = \bigcup_{i}\OOO(X)_{i}$ such that that $\OOO(X)_{1}$ generates the algebra $\OOO(X)$. Then the linear maps $\Der_{\kk}(\OOO(X)) \to \Hom(\OOO(X)_{i},\OOO(X))$, $\delta\mapsto \delta|_{\OOO(X)_{i}}$, are all injective. We define the filtration $\Der_{\kk}(\OOO(X)) = \bigcup_{\ell}\Der_{\kk}(\OOO(X))_{\ell}$ by setting $\Der_{\kk}(\OOO(X))_{\ell}:=\{\delta\in\Der_{\kk}(\OOO(X)) \mid \delta(\OOO(X)_{1})\subseteq \OOO(X)_{\ell}\}$.

Given $k$ and $\ell$ we can find integers $p,q,r$ such that 
\be
\item
$\GGG_{k}^{-1}(\OOO(X)_{1}) \subseteq \OOO(X)_{p}$,
\item 
$\Der_{\kk}(\OOO(X))_{\ell}(\OOO(X)_{p}) \subseteq \OOO(X)_{q}$,
\item
$\GGG_{k}(\OOO(X)_{q}) \subseteq \OOO(X)_{r}$.
\ee
Denote by $\rho'\colon \GGG_{k} \to \Hom(\OOO(X)_{1},\OOO(X)_{p})$ and $\rho\colon \GGG_{k}\to \Hom(\OOO(X)_{q},\OOO(X)_{r})$ the corresponding maps to (1) and (2) which are morphisms, by (a). Then we get the following commutative diagram:
{\tiny
$$
\begin{CD}
\GGG_{k} \times \Der_{\kk}(\OOO(X))_{\ell} @>{(g,\delta)\mapsto}>{(\rho'(g), \delta|_{\OOO(X)_{p}},\rho(g))}> \Hom(\OOO(X)_{1},\OOO(X)_{p}) \times \Hom(\OOO(X)_{p},\OOO(X)_{q}) 
\times \Hom(\OOO(X)_{q},\OOO(X)_{r})\\
@VV{(g,\delta)\mapsto g(\delta)}V           @V{(\alpha,\beta,\gamma) \mapsto}V{\gamma\circ\beta\circ\alpha}V\\
\Der_{\kk}(\OOO(X))_{r} @>{\delta\mapsto \delta|_{\OOO(X)_{1}}}>> \Hom(\OOO(X)_{1},\OOO(X)_{r})
\end{CD}
$$
}
The upper horizontal map and the right vertical map are both morphisms, and the lower horizontal map is a linear inclusion. Hence the left vertical map is also a morphism, as claimed.
\end{proof}

As a consequence, using Lemma~\ref{rep-of-G-and-LieG.lem}, we obtain linear actions of the Lie algebra $\Lie\GGG$ on $\OOO(X)$ and on $\VEC(X)$. Denoting these actions by $(A,f) \mapsto Af$ and $(A,\delta)\mapsto A\delta$ for $A \in \Lie\GGG$, $f \in \OOO(X)$ and $\delta \in \VEC(X)$, one  finds
\begin{equation} \label{formulas-for-A}
Af = -\xi_{A}f \text{ \ and \ } A\delta = [\delta,\xi_{A}] = -\ad\xi_{A}(\delta)
\end{equation}
where $\xi_{A}$ is discussed in Section~\ref{action-ind-groups.subsec}
(see Definition~\ref{xi_A.def}).

\ps
\subsection{Fixed points and tangent representations}\label{tangent-rep-fixed-points.subsec}
The following is well known for actions of algebraic groups; its generalization to actions of ind-groups on varieties is straightforward, because of our finiteness results in the previous section.
\begin{proposition}
Let $\GGG$ be an ind-group acting on an affine variety $X$. If $x \in X$ is a fixed point, then we obtain a linear representation of $\GGG$ on the tangent space $T_{x}X$, called the tangent representation\idx{tangent representation} in $x$.
\end{proposition}
\begin{proof}
By Proposition~\ref{locally-finite-action-on-Vec(X).prop} the action of $\GGG$ on $\OOO(X)$ is a linear representation. By assumption, the maximal ideal $\mm_{x} \subseteq \OOO(X)$ is $\GGG$-stable, as well as all the powers $\mm_{x}^{m}$. Thus we obtain a linear representation of $\GGG$ on $\mm_{x}/\mm_{x}^{2}$ and on its dual $T_{x}X$.
\end{proof}

\begin{remark}
Under the assumptions of the proposition above, the linear representation of $\GGG$ on  $\OOO(X)$ stabilizes the infinite flag 
$$
\OOO(X) \supseteq \mm_{x}\supseteq \mm_{x}^{2}\supseteq \mm_{x}^{3}\supseteq \cdots
$$
and induces finite-dimensional linear representations on the various quotients $\mm_{x}^{m}/\mm_{x}^{\ell}$. Assume now that $X$ is irreducible and that the action of $\GGG$ on $X$ faithful, and denote by $\KKK$ the kernel of the representation of $\GGG$ on $T_{x}X$. Then 
\be
{\it
\item The representation of every reductive subgroup $G \subseteq \GGG$ on $T_{x}X$ is faithful.
\item The image of $\KKK$ in $\GL(\OOO(X)/\mm_{x}^{m})$ is unipotent for all $m$. In particular, every locally finite element of $\KKK$ is unipotent.
\item The intersection of the kernels of the representations of $\GGG$ on $\OOO(X)/\mm_{x}^{m}$ for $m\geq 1$, is trivial.
}
\ee
The proofs of (2) and (3) are easy and left to the reader. Point (1) is classical, and here is a proof.\end{remark}

\begin{lemma}  \label{faithful-on-Tx.lem}
Let $G$ be a reductive group acting faithfully on an irreducible variety $X$. If $x_{0}\in X$ is a fixed point, then the tangent representation $G \to \GL (T_{x_{0}}X)$, $g\mapsto dg_{x_{0}}$,  is faithful.
\end{lemma}

\begin{proof}
The local ring $\OOO_{X,x_{0}}$ is stable under $G$, as well as its maximal ideal  $\mm:= \mm_{x_{0}}$ and all its powers $\mm^{j}$. 
For any $j\geq 2$ we have an exact sequence of $G$-modules
$$
0 \to \mm^{2}/\mm^{j} \to \mm/\mm^{j} \to \mm/\mm^{2}\simeq(T_{x_{0}}X)^{\vee} \to 0
$$
which splits, because $G$ is reductive. If $g$ acts trivially on $\mm/\mm^{2}$, then it acts trivially on $\mm^{2}/\mm^{3}$, because the action of $G$ on $\OOO_{X,x_{0}}$ is by algebra automorphisms. Therefore, $g$ acts trivially on $\mm/\mm^{3}$, by the $G$-splitting of the exact sequence above. By induction, we see that $g$ acts trivially on $\mm/\mm^{j}$ for all $j$. This implies that $gf-f \in \mm^{j}$ for any $f \in\mm$ and all $j$. But $\bigcap_{j} \mm^{j}=(0)$, and so $gf=f$. As a consequence, $g$ acts trivially on the local ring, hence trivially on the rational functions $\kk(X)$, and thus $g=e$, because the action of $G$ on $X$ is faithful.
\end{proof}

The proposition above generalizes to actions of $\GGG$ on ind-varieties. In fact, if $v_{0} \in \VVV$ is a fixed point, then every $ g\in\GGG$ induces a linear map $d g_{v}\colon T_{v}\VVV \to T_{v}\VVV$ and thus a linear action of $\GGG$ on $T_{v}\VVV$, denoted by $( g,w)\mapsto  g w$. We will see in the next theorem that this is indeed a representation 
$$
\tau_{v}\colon \GGG \to \GL(T_{v}\VVV)
$$ 
and thus induces  a linear representation of the Lie algebra (Lemma~\ref{rep-of-G-and-LieG.lem})
$$
d\tau_{v}\colon \Lie\GGG \to \LLL(T_{v}\VVV),
$$
hence an action of $\Lie\GGG$ on $T_{v}\VVV$ which we denote by $(A,w)\mapsto Aw:=d\tau_{}$.

\begin{theorem}\label{rep-in-fixed-point.thm}
Let the affine ind-group $\GGG$ act on the ind-variety $\VVV$, and assume that $v_{0}\in\VVV$ is a fixed point. Then the action of $\GGG$ on $T_{v_{0}}\VVV$ is a linear representation. 
\end{theorem}

\begin{proof}
For every $k\geq 1$ there is an $\ell$ such that $\g(\VVV_{k})\subseteq \VVV_{\ell}$ for all $g \in \GGG_{k}$. Hence, we get a morphism $\GGG_{k}\to \Mor(\VVV_{k},\VVV_{\ell})$, and the image is contained in $\Mor_{0}(\VVV_{k},\VVV_{\ell}):=\{\phi\in\Mor(\VVV_{k},\VVV_{\ell})\mid \phi(v_{0})=v_{0}\}$. By Lemma~\ref{mor-to-tangent.lem}, this subset is closed and the map 
$\Mor_{0}(\VVV_{k},\VVV_{\ell}) \to \LLL(T_{v_{0}}\VVV_{k},T_{v_{0}}\VVV_{\ell})$ is a morphism. Thus the composition 
$\GGG_{k}\to \LLL(T_{v_{0}}\VVV_{k},T_{v_{0}}\VVV_{\ell})$ is a morphism, as well as $\GGG_{k}\times T_{v_{0}}\VVV_{k} \to T_{v_{0}}\VVV_{\ell}$, and the claim follows.
\end{proof}

Every automorphism $\phi$ of  $\VVV$ defines a  linear automorphism of the vector fields $\VEC(\VVV)$, also denoted by $\phi$, which is given in the following way (see Proposition~\ref{locally-finite-action-on-Vec(X).prop}). If $\delta = (\delta_{x})_{x\in\VVV}$, then 
\[\tag{$*$}
\phi(\delta)_{\phi(x)} := d\phi_{x}(\delta_{x}).
\]
Equivalently, considering $\delta$ as a continuous derivation of $\OOO(\VVV)$ in case $\VVV$ is affine, we have $\phi(\delta) = (\phi^{*})^{-1}\circ\delta\circ\phi^{*}$. 
If $v_{0} \in \VVV$ is a fixed point, then the formula $(*)$ above shows that the action of $\GGG$ on the tangent space $T_{v_{0}}\VVV$ is induced by the action on the vector fields.

\begin{proposition} 
Consider an action of an affine ind-group $\GGG$ on an ind-variety $\VVV$, and assume that $v_{0}\in\VVV$ is a fixed point.
\be
\item For  $g \in \GGG$ and $\delta\in\VEC(\VVV)$ we have $g(\delta)_{v_{0}} = g (\delta_{v_{0}})$.
\item For  $g \in \GGG$ and  $A \in \Lie\GGG$ we have 
$g(\xi_{A}) = \xi_{\Ad(g)A}$.
\ee
\end{proposition}
\begin{proof}
(1) This is just the formula $(*)$.
\ps
(2) This follows from the two diagrams
$$
\begin{CD}
\GGG @>{\mu_{x}}>> \VVV    &\qquad\qquad&    \Lie\GGG @>{d\mu_{x}}>> T_{x}\VVV\\
@VV{\Int g}V        @VV{g}V      @VV{\Ad g }V   @VV{dg_{x}}V  \\
\GGG @>{\mu_{g x}}>> \VVV   &\qquad\qquad&    \Lie\GGG  @>{d\mu_{g x}}>> T_{g x}\VVV\
\end{CD}
$$
where the right diagram follows from the left diagram and implies that $g(\xi_{A})_{g x} = dg_{x}(\xi_{A})_{x} = (\xi_{\Ad(g)A})_{g x}$.
\end{proof}

Let again the affine ind-group $\GGG$ act on the ind-variety $\VVV$. If  $v\in\VVV$ is a fixed point, then the \itind{orbit map} $\mu=\mu_{v}\colon \GGG \to \VVV$, $g \mapsto g v$, is constant and its differential\idx{differential of orbit map} $(d\mu_{v})_{e}\colon \Lie\GGG \to T_{v}\VVV$ is trivial. Equivalently, all vector fields $\xi_{A}$, $A \in \Lie\GGG$, vanish at $v$. The converse holds for connected ind-groups.

\begin{lemma}\label{orbitmap.lem}
Assume that $\GGG$ is connected, and let $v \in \VVV$. If $(d\mu_{v})_{e}$ is the zero map, then $\mu_{v}$ is constant, i.e. $v$ is a fixed point of $\GGG$.
\end{lemma}

\begin{proof}
We have $\GGG = \bigcup \GGG_{k}$ where we can assume that all $\GGG_{k}$ are irreducible and contain $e$. If $\mu:=\mu_{v}$ is not the constant map, then there exists a $k>0$ such that $\dim\overline{\mu(\GGG_{k})}\geq 1$. It follows that there is $g\in\GGG_{k}$ such that the differential $d\mu_{g}\colon T_{g}\GGG_{k} \to T_{gv}\VVV$ is not the zero map. Denote by $\lambda_{h}\colon \GGG \simto \GGG$ and $\lambda_{h}\colon \VVV \simto \VVV$ the left multiplication with $h\in\GGG$. Then the commutative diagram
$$
\begin{CD}
T_{g}\GGG_{k} @>d\mu_{g}>> T_{gv} \VVV \\
@V{(d\lambda_{g^{-1}})_{g}}VV  @V{\simeq}V{(d\lambda_{g^{-1}})_{gv}}V \\
T_{e}\GGG_{\ell} @>d\mu_{e}>> T_{v} \VVV 
\end{CD}
$$
shows that $d\mu_{e}$ is not the zero map.
\end{proof}

This lemma has a number of interesting applications.  
\begin{proposition}\label{phi-dphi.prop}
Let $\phi,\psi \colon \GGG \to \HHH$ be two homomorphisms of ind-groups. If $\GGG$ is connected and $d\phi=d\psi$, then  $\phi=\psi$.
In particular, $\phi$ is trivial if and only if $d\phi$ is trivial.
\end{proposition}

\begin{proof}
Consider the action of $\GGG$ on $\HHH$ given by $g*h:=\psi(g) h \phi(g)^{-1}$, and take the orbit map $\mu\colon \GGG \to \HHH$ in $e\in\HHH$ which is given by $g\mapsto \psi(g)\phi(g)^{-1}$. Then $d\mu_{e}= d\psi_{e} - d\phi_{e}= 0$, and so, by Lemma~\ref{orbitmap.lem} above, we get $\psi(g)\phi(g)^{-1}= e$ for all $g\in\GGG$, hence the claim.
\end{proof}

\begin{corollary}\label{AutG-into-AutLie.cor}
If $\GGG$ is connected, then the canonical homomorphism (of abstract groups) $\omega\colon\Aut(\GGG) \to \Aut(\Lie\GGG)$ is injective.
\end{corollary}

\ps
\subsection{The adjoint representation}\label{tangent-adjoint.subsec}
A special case of the tangent representation in a fixed point (Theorem~\ref{rep-in-fixed-point.thm}) is the \itind{adjoint representation}
$$
\Ad\colon \GGG \to \GL(\Lie \GGG).
$$\idx{$\Ad$}
It is obtained from the action of $\GGG$ on itself by \itind{inner automorphism}:
$$ 
\Int\colon \GGG \to \Aut(\GGG),\quad \Int(g)\colon  h \mapsto g\cdot h\cdot g^{-1}.
$$\idx{$\Int$}
A first application is the following result which follows immediately from Proposition~\ref{phi-dphi.prop}.

\begin{corollary}\label{center.cor}
Let $\GGG$ be connected and let $h \in \GGG$. Then $h$ belongs to the center $Z(\GGG)$ of $\GGG$ if and only if $\Ad(h)$ is trivial:
$$
Z(\GGG) = \Ker\Ad.
$$
\end{corollary}\idx{center $Z(\GGG)$}

The adjoint representation induces a representation of $\Lie\GGG$, denoted by 
$$
\ad\colon \Lie\GGG \to \LLL(\Lie \GGG).
$$
\idx{$\Ad$@$\ad$}
\begin{proposition}\label{adjoint.prop}
For $A,B \in \Lie \GGG$ we have 
$$
\ad(A)(B) = [A,B].
$$
\end{proposition}

As a consequence we get the following result.
\begin{corollary}\label{G-commut-LieG.cor}
Let $\GGG$ be a connected ind-group. Then $\GGG$ is commutative if and only if $\Lie\GGG$ is commutative.
\end{corollary}
\begin{proof}
(1) The group $\GGG$ is commutative if and only if $\Int \colon \GGG \to \Aut(\GGG)$ is trivial which implies that $\Ad\colon \GGG \to \GL(\Lie\GGG)$ is trivial, hence $\ad\colon \Lie\GGG \to \LLL(\Lie\GGG)$ is trivial. By Proposition~\ref{adjoint.prop}, the latter is equivalent to $\Lie\GGG$ being commutative. This proves one implication.
\ps
(2) Now assume that $\Lie\GGG$ is commutative, hence $\ad\colon \Lie\GGG \to \LLL(\Lie\GGG)$ is trivial, by Proposition~\ref{adjoint.prop}.
Since $\GGG$ is connected, this implies that $\Ad\colon \GGG \to \LLL(\Lie \GGG)$ is trivial, by Proposition~\ref{phi-dphi.prop}. Since $\Ad(g)$ is the differential of the homomorphism $\Int(g)\colon \GGG \to \GGG$ if follows again from Proposition~\ref{phi-dphi.prop} that $\Int(g)$ is trivial for all $g \in \GGG$, hence $\GGG$ is commutative.
\end{proof}

The proof of Proposition~\ref{adjoint.prop} needs some preparation. It will be given at the end of this Section~\ref{tangent-adjoint.subsec}. We know that every affine ind-variety $\VVV$ admits a closed immersion into a $\kk$-vector space $V$ of countable dimension (Theorem~\ref{embedding-into-Ainfty.thm}). We will see below in Proposition~\ref{G-ind-module.prop} that an affine ind-variety $\VVV$ with an action of an algebraic group $G$ is isomorphic to a closed $G$-stable ind-subvariety of a $G$-module $V$ of countable dimension. 

On the other hand, we cannot expect that an ind-variety $\VVV$ together with an action of an ind-group $\GGG$ admits a closed immersion into a representation of $\GGG$. This is even not true for an affine variety $X$ with a $\GGG$-action, see Proposition~\ref{no-embedding-into-representations.prop}.
However, we can proof the following result which will be sufficient for our purposes.

\begin{lemma}\label{linearization.lem}
Let $\GGG = \bigcup_{k}\GGG_{k}$ be an affine ind-group acting on an affine ind-variety $\VVV$ by $\phi\colon \GGG \times \VVV \to \VVV$. For every $k\geq 0$ we can find a closed immersion $\iota\colon \VVV \into V$ where $V$ is a $\kk$-vector space of countable dimension, and a morphism $\Phi\colon \GGG_{k}\to \LLL(V)$ such that the following diagram commutes
$$
\begin{CD}
\GGG_{k} \times V @>{(g,v)\mapsto\Phi(g,v)}>> V \\
@AA{\id\times\iota}A   @AA{\iota}A \\
\GGG_{k} \times \VVV  @>{(g,v)\mapsto\phi(g,v)}>>  \VVV
\end{CD}
$$
\end{lemma}

\begin{proof}
(1) Given $k\geq 1$, we find for any $m\geq 1$ an $n_{0}\geq 1$ such that $\phi(\GGG_{k}\times\VVV_{m}) \subseteq \VVV_{n_{0}}$:
$$
\phi \colon \GGG_{k}\times \VVV_{m} \to \VVV_{n_{0}}.
$$
For the comorphism $\phi^{*}\colon \OOO(\VVV_{n_{0}}) \to \OOO(\GGG_{k})\otimes \OOO(\VVV_{m})$ we can find finite-dimensional sub-vector spaces
$L\subset \OOO(\VVV_{n_{0}})$ and $M \subset \OOO(V_{m})$ which generate the algebra and such that 
$\phi^{*}(L) \subset \OOO(\GGG_{k})\otimes M$. Setting $V_{n_{0}}:=L^{\vee}$ and $V_{m}:=M^{\vee}$ we get the  
following commutative diagram:
$$
\begin{CD}
\GGG_{k}\times V_{m} @>{\Phi_{m}}>>  V_{n_{0}} \\
@AA{\subseteq}A @AA{\subseteq}A \\
\GGG_{k}\times \VVV_{m} @>>{\phi_{m}}> \VVV_{n_{0}}
\end{CD}
$$
where for all $g\in\GGG_{k}$, the map $V_{m} \to V_{n}$, $v \mapsto \Phi_{m}(g,v)$, is linear, i.e. $\Phi_{m}$ defines a morphism $\tilde\Phi_{}\colon \GGG_{k} \to \LLL(V_{m},V_{n})$.
\ps
(2) Assume that $\phi(\GGG_{k}\times \VVV_{m+1})\subseteq \VVV_{n_{1}}$. Then we can find finite-dimensional $\kk$-vector spaces $L_{1}\subset \OOO(\VVV_{n_{1}})$ and $M_{1}\subset \OOO(\VVV_{m+1})$ with the following properties.
\be
\item[(a)] $L_{1}$ generates $\OOO(\VVV_{n_{1}})$ and $M_{1}$ generates $\OOO(\VVV_{m+1})$;
\item[(b)] $L_{1}\cap \Ker p_{n}$ generates the kernel of $p_{n}\colon \OOO(\VVV_{n_{1}})\to\OOO(\VVV_{n})$, and 
$M_{1}\cap \Ker q_{m}$ generates the kernel of $q_{m}\colon \OOO(\VVV_{m+1})\to\OOO(\VVV_{m})$;
\item[(c)] $\phi(L_{1}) \subset \OOO(\GGG_{k})\otimes M_{1}$;
\item[(d)] $L_{1}$ maps surjectively onto $L$, and $M_{1}$ maps surjectively onto $M$.
\ee
Setting $V_{n_{1}}:=M_{1}^{\vee}$ and $V_{n_{1}}:=L_{1}^{\vee}$ we get from (1) the following commutative diagram
\begin{center}
\begin{tikzcd}
\GGG_{k}\times V_{m} \arrow[ddd,hook]\arrow[rrr, "\Phi_{m}"]  &&& V_{n} \arrow[ddd,hook]  \\
& \GGG_{k}\times \VVV_{m} \arrow[r, "\phi_{m}"] \arrow[d,hook]\arrow[lu,hook] & \VVV_{n} \arrow[d,hook]\arrow[ru,hook] & \\
& \GGG_{k}\times \VVV_{m+1} \arrow[ld,hook] \arrow[r, "\phi_{m+1}"] & \VVV_{n_{1}}\arrow[rd,hook] & \\
\GGG_{k}\times V_{m+1} \arrow[rrr, "\Phi_{m+1}"]  &&& V_{n_{1}}
\end{tikzcd}
\end{center}
where $\VVV_{n}=\VVV_{n_{1}}\cap V_{n}$ and $\VVV_{m}=\VVV_{m+1}\cap V_{m}$, and $V_{m}\subset V_{m+1}$ and $V_{n} \subset V_{n_{1}}$ are linear subspaces. It follows that $\GGG_{k}\times \VVV = \bigcup_{m}(\GGG_{k}\times \VVV_{m})$ embeds into $\GGG_{k}\times V$ as a closed ind-subvariety where $V := \varinjlim_{m} V_{m}\simeq \AA^{\infty}$ is a $\kk$-vector space of countable dimension.  The same holds for  $\varinjlim_{i}V_{n_{i}}$, and we can identify $\varinjlim_{i}V_{n_{i}}$ with $V$. The claim follows.
\end{proof}

The lemma immediately implies the following result.

\begin{proposition}\label{G-ind-module.prop}
Let $G$ be an algebraic group acting on an affine ind-variety $\VVV$. Then $\VVV$ is $G$-isomorphic to a 
$G$-stable closed ind-subvariety of a $G$-module $V$.
\end{proposition} 

Now we generalize Corollary~\ref{tangent-rep-and-VF.cor} to the action of affine ind-group $\GGG$ on an affine ind-varieties $\VVV$. Assume that there is a fixed point $v_{0}\in \VVV$, and denote by $\tau_{v_{0}}\colon \GGG \to \GL(T_{v_{0}}\VVV)$ the tangent representation. The corresponding linear action of $\Lie\GGG$ on $T_{v_{0}}$ is given by $(A,w) \mapsto Aw:= d\tau_{v_{0}}(A)(w)$.  As before, $\xi_{A}\in\VEC(\VVV)$ denotes the vector field associated to $A \in \Lie\GGG$.

\begin{proposition}\label{rep-in-fixed-point.prop}
For $A \in \Lie\GGG$ and $\delta \in \VEC(\VVV)$ we have  
$$
A \delta_{v_{0}} = -[\xi_{A},\delta]_{v_{0}}.
$$
\end{proposition}
\begin{proof}
Choose a $k\geq 1$ such that $A \in T_{e}\GGG_{k}$. By Lemma~\ref{linearization.lem} there is a closed immersion $\VVV \into V$ into a $\kk$-vector space $V$ of countable dimension and a ``linear action'' $\tilde\Phi\colon \GGG_{k}\to \LLL(V)$ such that the following diagram commutes where $\Phi(g,v) := \tilde\Phi(g)(v)$:
$$
\begin{CD}
\GGG_{k} \times V @>{\Phi}>> V \\
@A{\subseteq}A{\id\times\iota}A   @A{\subseteq}A{\iota}A \\
\GGG_{k} \times \VVV  @>{\phi}>>  \VVV
\end{CD}
$$
The vector field $\xi_{A}\in\VEC(\VVV)$  is given in the following way. For $v \in \VVV$ we define the orbit map $\mu_{v}\colon \GGG_{k} \to \VVV$, $g\mapsto g v = \phi(g,v)$, and get $(\xi_{A})_{v}= d\mu_{v}(A)$. This shows that we also get a linear vector field $\tilde\xi_{A}\in \LLL(V) \subset \VEC(V)$ by the same construction, namely $\tilde\xi_{A}=d\tilde\Phi_{e}(A)$. Clearly, $\VVV$ is invariant under $\tilde\xi_{A}$ and $\xi_{A} = \tilde\xi_{A}\big|_{\VVV}$. 

Now let $v_{0}\in\VVV$ be a fixed point. We can assume that $v_{0}=0 \in V$.
If $\tilde\delta\in\VEC(V) = \Mor(V,V)$, then we get $A\tilde\delta_{0} = -[\tilde\xi_{A},\delta]_{0}$, because $\tilde\xi_{A}$ is linear, see Lemma~\ref{linear-VF.lem}. If $\VVV$ is invariant under $\tilde\delta$ and $\delta:=\tilde\delta\big|_{\VVV}$, then we get $A\delta_{0} = [\xi_{A},\delta]_{0}$. Now the claim follows because every vector field on $\VVV$ is a restriction from a vector field of $V$.
\end{proof}

Finally, let us recall the basic property of the Lie algebra structure of $T_{e}\GGG$, cf. Section~\ref{Liealgebra.subsec}. For this consider the action of $\GGG$ on itself by \itind{left multiplication}: $(g, h)\mapsto g\cdot h$, and denote by $\lambda_{A}\in \VEC(\GGG)$ the vector field corresponding to $A \in T_{e}\GGG$ with respect to this action. Similarly, we consider the action by \itind{right multiplication}, $(g, h)\mapsto h\cdot g$, and denote by $\rho_{A} \in \VEC(\GGG)$ the vector field corresponding to $A \in T_{e}\GGG$. Then we have the following result.
Note that the vector fields $\lambda_{A}$ are right-invariant and the $\rho_{B}$ are left-invariant (cf. Section~\ref{Liealgebra.subsec}).

\begin{lemma}\label{actions-on-G.lem} Let $A,B \in T_{e}\GGG$.
\be
\item $(\lambda_{A})_{e}=A=(\rho_{A})_{e}$;
\item $[\rho_{A},\rho_{B}] = \rho_{[A,B]}$;
\item $[\lambda_{A},\lambda_{B}] = \lambda_{[B,A]}$;
\item $[\lambda_{A},\rho_{B}]=0$;
\item For the action of $\GGG$ on itself by conjugation, we have $\xi_{A}=\lambda_{A}-\rho_{A}$.
\ee
\end{lemma}
\begin{proof}
(1) follows from the definition, and for (2) and (3) see Section~\ref{Liealgebra.subsec} where we 
discuss the left-invariant vector fields $\rho_{A}$.
\ps
(4) This is clear, because the two actions commute.
\ps
(5) Consider the action of $\GGG \times \GGG$ on $\GGG$ given by $(( g_{1}, g_{2}), h)\mapsto  g_{1}\cdot h\cdot g_{2}$. Then the we find $\xi_{(A,0)}=\lambda_{A}$, $\xi_{(0,B)}= \rho_{B}$, and so $\xi_{(A,B)}=\lambda_{A}+\rho_{B}$. The claim follows by embedding $\GGG \into \GGG \times \GGG$, $ g\mapsto ( g, g^{-1})$.
\end{proof}
\begin{proof}[Proof of Proposition~\ref{adjoint.prop}]
For the action of $\GGG$ on $\GGG$ by conjugation Proposition~\ref{rep-in-fixed-point.prop} implies that 
$$
\ad (A)(\delta_{e}) = -[\xi_{A},\delta]_{e} \text{ for }\delta\in\VEC(\GGG) \text{ and }A \in \Lie\GGG.
$$
Since $\xi_{A}=\lambda_{A}-\rho_{A}$ and $\rho_{B}$ commutes with all $\lambda_{C}$, we find
$$
A(B) = A(\lambda_{B})_{e} = -[\xi_{A},\lambda_{B}]_{e} = -[\lambda_{A},\lambda_{B}]_{e} =
 -(\lambda_{[B,A]})_{e}=(\lambda_{[A,B]})_{e} = [A,B].
$$
The claim follows.
\end{proof}

\begin{remark}
Let us point out here that the results above are well known for algebraic groups. However, the proofs for ind-groups needed some new ideas since the classical proofs do not carry over. E.g. we cannot prove that the center of the Lie algebra of $\GGG$ is equal to the Lie algebra of the center of $\GGG$, and it is even not true that the Lie algebra of a strict closed subgroup of a connected ind-group $\GGG$ is strictly contained in $\Lie\GGG$, see Theorem~\ref{closed-subgroup-with-same-Liealgebra.thm}.
\end{remark}

\ps
\subsection{Integration of locally finite vector fields}\label{Integration-of-VF.subsec}
Recall that we have a canonical embedding $\xi\colon \Lie\Aut(X) \into \VEC(X)$, see Proposition~\ref{Liealg-VF.prop}. A vector field $\delta$ is called \itind{locally finite} if $\delta$ considered as a linear endomorphism $\delta\in\LLL(\OOO(X))$ is locally finite (see Definition~\ref{locally-finite.def}).

\begin{proposition}\label{Integration-of-VF.prop}
Let $\delta \in \VEC(X)$ be a locally finite vector field, and let $\delta = \delta_{s}+\delta_{n}$ be its additive Jordan decomposition in $\LLL(\OOO(X))$. Then 
\be
\item 
$\delta_{s},\delta_{n}$ are both locally finite vector fields. 
\item 
There is a unique minimal torus $T \subseteq \Aut(X)$ such that $\delta_{s}\in \xi(\Lie T)$.
\item 
If $\delta_{n}\neq 0$, then
there is a unique 1-dimensional unipotent subgroup $U \subseteq \Aut(X)$ such $\delta_{n}\in \xi(\Lie U)$.
\item 
$T$ and $U$ commute.
\ee
\end{proposition}
\begin{proof}
(1) Let $R := \OOO(X)$ and consider $\delta$ as a locally finite derivation $D$ of $R$.
We have a direct sum decomposition $R = \bigoplus_{\lambda\in\kk} R_{\lambda}$ where 
\[
R_{\lambda}:=\{ r \in R \mid (D-\lambda)^{m}r = 0 \text{ for some } m \in \NN\}.
\]
We claim that $R_{\lambda}\cdot R_{\mu} \subseteq R_{\lambda+\mu}$. In fact, one easily checks the following equality 
for $\lambda,\mu \in \kk$ and $a,b \in R$:
\[
(D-(\lambda+\mu)) (a\cdot b) = (D - \lambda)(a)\cdot b + a \cdot (D - \mu)(b).
\]
By induction, this implies
\[
(D-(\lambda + \mu))^n(ab) = \sum_{k=0}^n \binom{n}{k} (D-\lambda)^k(a) \cdot  (D-\mu)^{n-k}(b),
\]
and the claim follows. Since $D_s(a) = \lambda a$ for any $a \in R_{\lambda}$, we get
\[ 
D_s (ab) =D_s(a)b + aD_s(b) \text{ for all } (a,b) \in R_{\lambda}\times R_{\mu}
\]
proving that $D_s$ is a derivation commuting with $D$. Hence, $D_{n}=D-D_{s}$ is also a derivation, as claimed.
\par\smallskip
(2) Since $R$ is finitely generated, the set of weights $\Lambda:=\{\lambda\in\kk\mid R_{\lambda}\neq 0\}\subseteq \kk$ generate a finitely generated free subgroup $\ZZ\Lambda=M:=\bigoplus_{i=1}^{m}\ZZ\mu_{i}\subseteq \kk$. This defines an action of the torus $T:=(\Cst)^{m}$ on $R$ in the following way: For $\lambda = \sum_{i}n_{i}\mu_{i}\in\Lambda$, $r \in R_{\lambda}$ and $t=(t_{1},\ldots,t_{m})\in T$ we put
$$
t r := t_{1}^{n_{1}}\cdots t_{m}^{n_{m}}\cdot r.
$$
Note that this action is faithful, since $M$ is generated by $\Lambda$.
Every $\alpha = (\alpha_{1},\ldots,\alpha_{m})\in\Lie T$ defines a derivation $D(\alpha)$ of $R$, namely
$$
D(\alpha) r := \sum_{i}n_{i}\alpha_{i} \cdot r \text{ \ for \ }r \in R_{\lambda},  \lambda= \sum_{i}n_{i}\mu_{i}.
$$
 In particular, $D_{s} = D(\mu_{1},\ldots,\mu_{m})$, hence $D_{s}$ belongs to $\xi(\Lie T)$. It also follows that $T$ is minimal. In fact, for a strict subtorus $T'\subsetneqq T$ the elements of $\Lie T' \subseteq \Lie T$ satisfy a linear equation with integral coefficients, and so $(\mu_{1},\ldots,\mu_{m})$ is not contained in $\Lie T'$. If $D_{s}\in\xi(\Lie T_{1})$ for some other torus $T_{1}\subseteq \Aut(R)$, then $D_{s}\in \xi(\Lie(T\cap T_{1})^{\circ})$, hence $T_{1}\supseteq T$.
\par\smallskip
(3) We define the exponential $\exp\colon \kplus \to \Aut(R)$ by
$$
\exp(s) r:=\sum_{m\geq 0} \frac{1}{m!} (D_{n})^{m}r.
$$
It is well-defined, because the sum is finite, and we obtain an action of $\kplus$ on $R$. The differential is given by $(d\exp)_{0}r = D_{n}r$. Hence, denoting by $U$ the image of $\kplus$ in $\Aut(R)$ we get $\xi(\Lie U) = \kk D_{s}$. The minimality is clear, and the uniqueness follows as in (2).
\par\smallskip
(4) By construction, the subspaces $R_{\lambda}\subseteq R$ are stable under $T$ and $U$. Since $T$ acts by scalar multiplications, this action commutes with the linear action of $U$, and the claim follows.
\end{proof}

\begin{remark} \label{loc-nil-in-LieAut(X).rem}
Let us call a tangent vector $N\in\Lie\Aut(X)$ \itind{locally nilpotent} if the corresponding vector field 
$\xi_N \in \VEC (X)$ (see Proposition~\ref{End(X)-and-Vec(X).prop}) is locally nilpotent.
Then it follows from Proposition~\ref{Integration-of-VF.prop}(3) that there is a well-defined $\kplus$-action on $X$, $\lambda_{N}\colon \kplus \to \Aut(X)$, such that $d\lambda_{N}(1) = N$. In particular, every locally nilpotent vector field on $X$ corresponds to a locally nilpotent tangent vector in $\Lie\Aut(X)$.

It also follows that for every representation $\rho\colon \Aut(X) \to \GL(V)$ on a $\kk$-vector space $V$ of countable dimension the image of a locally nilpotent tangent vector under $d\rho$ is a locally nilpotent linear endomorphism of $V$, see Section~\ref{reps-of-ind-groups.subsec}.
\end{remark}
\begin{question}\label{locally-nilpotent.ques}
Assume that $\rho\colon \Aut(X) \to \GL(V)$ is a representation such that
$d\rho\colon \Lie \Aut (X) \to \End(V)$ is injective
(where $V$ is again a vector space of countable dimension).
Is it true that if $d\rho(N) \in \End(V)$ is locally nilpotent, then $N \in \Lie\Aut(X)$ is locally nilpotent?

Assume that the adjoint representation $\ad\colon \Lie\Aut(X) \to \End(\Lie\Aut(X))$ is faithful. Is it true that $N\in\Lie\Aut(X)$ is locally nilpotent if and only if $\ad N$ is locally nilpotent?
\end{question}

The following consequence of the proposition above is clear.

\begin{corollary} \label{locally-finite-VF-in-Liealg.cor}
For every locally finite vector field $\delta\in\VEC(X)$ there exists a unique minimal connected and commutative algebraic subgroup $A \subseteq \Aut(X)$ such that $\delta \in \xi(\Lie A)$.
\end{corollary}

\begin{remark}\label{JD-general-algebra.rem}
The proof above shows that all the statements hold more generally for a locally finite derivation $\delta$ of a general algebra $\R$. We will use this in Section~\ref{GeneralAlgebra.sec} where we discuss the automorphism group of a general algebra $\R$.
\end{remark}

The following result is due to \name{Cohen-Draisma} \cite[Theorem~1]{CoDr2003From-Lie-algebras-}. It generalizes 
Corollary~\ref{locally-finite-VF-in-Liealg.cor} above.

\begin{theorem}\label{Cohen-Draisma.thm}
Let $L \subseteq \VEC(X)$ be a finite-dimensional Lie subalgebra. Assume that $L$ is locally finite as a subset of $\LLL(\OOO(X))$. Then there is an algebraic subgroup $G \subseteq \Aut(X)$ such that $L \subseteq \xi(\Lie G)$.
\end{theorem}

\ps
\subsection{Intersection of closed subgroups}

Another well-known result is that for two closed subgroups $H_{1},H_{2} \subseteq G$ of a linear algebraic group $G$ we always have $\Lie (H_{1}\cap H_{2}) = \Lie H_{1}\cap \Lie H_{2}$. We do not know if this holds also for closed subgroups of an ind-group. However, we have the following result.

\begin{proposition}\label{Lie-H-cap-G.prop}
Let $H_{1},H_{2}\subseteq \GGG$ be algebraic subgroups of an ind-group $\GGG$. Then we have 
$$
\Lie (H_{1}\cap H_{2}) = \Lie H_{1}\cap \Lie H_{2}.
$$
\end{proposition}

\begin{proof}
The schematic intersection $X \cap_{\text{\it schematic}} Y$ of two closed subvarieties $X,Y \subseteq Z$  is the fiber product $X \times_{Z}Y$, and $T_{z} (X \times_{Z}Y) = T_{z}X \cap T_{z}Y\subseteq T_{z}Z$ for $z \in X\cap Y$. Therefore, if the schematic intersection is reduced in $z$, then $T_{z} (X \cap Y) = T_{z}X \cap T_{z}Y$. But the schematic intersection of two algebraic groups is an algebraic group scheme over $\CC$, hence smooth by \name{Cartier}'s Theorem \cite[II, \S6, 1.1]{DeGa1970Groupes-algebrique}, and the claim follows.
\end{proof}

\begin{remark}
It is easy to see that $\Lie \HHH \cap \Lie H = \Lie (\HHH \cap H)$ if $\HHH \subseteq \GGG$ is a closed nested ind-subgroup and $H \subset \GGG$ an algebraic subgroup. On the other hand this does not hold for a general ind-group $\HHH_{1}$ as we will see later in Remark~\ref{Lie-H-cap-G.rem}.
\end{remark}

\pmed
\section{The ind-variety of group homomorphisms}\label{Hom.sec}
\subsection{Reductive and semisimple groups}\label{notation.subsec}
Let us recall some basic notion from the theory of algebraic groups. We refer to the text books 
\cite{Bo1991Linear-algebraic-g,DeGa1970Groupes-algebrique,Hu1978Introduction-to-Li,Hu1975Linear-algebraic-g,Kr1984Geometrische-Metho,Pr2007Lie-groups,Sp1989Aktionen-reduktive} for more details and further reading.

The \itind{radical}  of a linear algebraic group $G$ is the maximal connected solvable normal subgroup of $G$, see \cite[11.21]{Bo1991Linear-algebraic-g}. We denote it by $\rad G$. The maximal unipotent subgroup of 
$\rad G$ is the \itind{unipotent radical} and will be denoted by $\rad_{u}G$.\idx{$\rad G$}\idx{$\rad_{u}G$}
The group $G$ is called \itind{reductive} if the unipotent radical $\rad_{u}G$ is trivial. In this case the radical $\rad G$
is a torus and
coincides with the identity component of the center of $G$
(cf. \cite[11.21]{Bo1991Linear-algebraic-g}).
In characteristic zero reductive groups are \itind{linearly reductive} which means that the rational representation are \itind{completely reducible}, see \cite[II.3.5]{Kr1984Geometrische-Metho}.

We have recalled
in Section~\ref{VectorFields.sec} that for a connected group $G$ there is a strong connection between $G$ and $\Lie G$, see \cite[II.2.5]{Kr1984Geometrische-Metho}. E.g. if $\rho\colon G \to \GL(V)$ is a finite-dimensional representation of $G$ and $d\rho\colon \Lie G \to \End(V)$ the corresponding representation of $\Lie G$, then a subspace $W \subseteq V$ is $G$-stable if and only if it is $\Lie G$-stable. Moreover, a linear automorphism of $V$ is $G$-equivariant if and only if it is  $\Lie G$-equivariant: $\GL_{G}(V) = \GL_{\Lie G}(V)$.
\idx{$\GL_{G}(V)$}\idx{$\GL_{\Lie G}(V)$}

A connected group $G$ is called  \itind{semisimple} if $\rad G$ is trivial. In this case the center is a finite group and $G = (G,G)$ (see \cite[Proposition~14.2]{Bo1991Linear-algebraic-g}). In particular, the character group of a semisimple group is trivial. It also follows that for a connected reductive $G$ the derived group $(G,G)$ is semisimple and $G = Z(G)^{\circ}(G,G)$. Moreover, still assuming that $G$ is connected reductive, the group $G / (G,G)$ is a torus and this torus is trivial if and only if $G$ is semisimple.

The semisimple group $G$ is called \itind{simply connected} if the fundamental weights (with respect to a Borel subgroup $B$ and a maximal torus $T \subset B$) are characters of $T$. This implies that every representation of the Lie algebra $\Lie G$ is induced by a representation of $G$. For example, $\SL_{n}$ is simply connected whereas $\PSL_{n}$ is not.

\ps
\subsection{Representations and homomorphisms}
We start with the following easy lemma.\idx{$\Hom(G,\HHH)$}

\begin{lemma} \label{Hom-in-Mor.lem}
Let $G$ be a linear algebraic group and $\HHH$ an affine ind-group. Then the set $\Hom(G,\HHH)$ of homomorphisms of ind-groups   is a closed subset of $\Mor(G,\HHH)$, so that $\Hom(G,\HHH)$ has a natural structure of an affine ind-variety. Moreover, the actions of $G$ and $\HHH$ by conjugation on $\Mor(G,\HHH)$ and on $\Hom(G, \HHH)$ are regular.
\end{lemma}
\begin{proof}
It is easy to see that the map
$$
\Phi\colon \Mor(G,\HHH) \to \Mor(G\times G, \HHH), \ \Phi(\phi)(a,b):=\phi(ab)\phi(b)^{-1}\phi(a)^{-1},
$$ 
is a morphism of ind-varieties. If $\gamma_{e}\in\Mor(G\times G,\HHH)$ denotes the constant map $\gamma_{e}(a,b) = e$, then $\Hom(G,\HHH) = \Phi^{-1}(\gamma_{e})$, and so $\Hom(G,\HHH)$ is closed in $\Mor(G,\HHH)$. Since the action of $G$ on $G$ by conjugation is regular, as well as the action of $\HHH$ on $\HHH$, it follows that also the induced action on $\Mor(G,\HHH)$ and hence on $\Hom(G,\HHH)$ is regular.
\end{proof}

\begin{remark}\label{closed-subgroup-Hom.rem}
If $\HHH \into \GGG$ is a closed immersion of ind-groups, then it follows from 
Proposition~\ref{closed-immersion-Mor.lem}(\ref{closed-immersion-yields-closed-immersion-for-morphisms}) that the induced map $\Hom(G,\HHH) \into \Hom(G,\GGG)$ is a closed immersion.
\end{remark}

If $H = \GL(V)$ where $V$ is a finite-dimensional $\kk$-vector space of dimension $n$, then  $\Hom(G,\GL(V))$ is the set of representations of $G$ on $V$, and the $\GL(V)$-orbits are exactly the equivalence classes of $n$-dimensional representations. We use the notation 
$\Mod_{G}(V)$ for $\Hom(G,\GL(V))$ when we consider a representation $\rho\colon G \to \GL(V)$ as a $G$-module structure on $V$.  If $M \in \Mod_{G}(V)$, then $C_{M} \subseteq \Mod_{G} (V)$ denotes the {\it orbit of $M$} under $\GL(V)$, i.e. the {\it equivalence class of $G$-modules} in $\Mod_{G}(V)$ isomorphic to $M$.\idx{$\Mod_{G}(V)$}\idx{$C_{M}$}

Note that the stabilizer of $M \in \Mod_{G}(V)$ in $\GL(V)$ is equal to the group of $G$-equivariant automorphisms, $\Stab_{\GL(V)} M = \GL_{G}(M)$, hence the orbit $C_{M}$ is isomorphic to $\GL(V)/\GL_{G}(M)$.
As usual a $G$-module $M \in \Mod_{G}(V)$ is called {\it semisimple} if and only if the associated $G$-representation is \itind{completely reducible}.\idx{semisimple $G$-module}

\ps
\subsection{The lemma of \texorpdfstring{\name{Artin}}{Artin}}
The following lemma goes back to \name{Michael Artin}, see \cite[section~12]{Ar1969On-Azumaya-algebra}. It appears in several contexts, e.g. in representations of finite-dimensional algebras, representations of quivers.

\begin{lemma}\label{Artin.lem}
A $G$-module $M\in \Mod_{G}(V)$ is semisimple if and only if its equivalence class $C_{M} \subset \Mod_{G}(V)$ is closed. In particular, for every $M \in \Mod_{G}(V)$ the closure $\overline{C_{M}}$ contains a semisimple $G$-module.
\end{lemma}
\begin{proof}[Sketch of Proof]
(1) Let $\lambda\colon\kst \to \GL(V)$ be a 1-PSG (one-parameter subgroup), i.e. a homomorphism of algebraic groups, and assume that the limit $N:= \lim_{t\to 0}\lambda(t) M$ exists in $\Mod_{G}(V)$. 
For any $r \in\ZZ$ denote by
\[ 
V_{r}:=\{v\in V\mid \lambda(t)v = t^{r}\cdot v \text{ for all }t\in\kst\}  \subseteq V=M
\]
the weight spaces of $\lambda$ of weight $r$, and set
$M_{k} :=\bigoplus_{j \geq k}V_{j} \subseteq M$.
Then, the following assertion is left as an exercise for the reader:
\ps
$(*)$ {\it The $M_{k}$  are $G$-submodules of $M$, and the limit $N:=\lim_{t\to 0}\lambda(t) M$ is isomorphic to 
the associated graded module $\bigoplus_{k}M_{k}/M_{k+1}$.}
\ps
(2)
Now assume that the orbit $C_{M}$ of $M$ is not closed. Then, by the \name{Hilbert-Mumford}-Criterion, there is a 1-PSG $\lambda$ such that $\lim_{t\to0} \lambda (t) M = N$  exists and belongs to the closed orbit in $\overline{C_{M}}$. By $(*)$, $N$ is the associated graded $G$-module with respect to a suitable filtration of $M$. This implies that $M$ is not semisimple, because the associated graded module of any filtration of a semisimple module $M$ is isomorphic to $M$, hence belongs to $C_{M}$. 

On the other hand, if $M$ is not semisimple, then there exists a filtration of $M$ such that the associated graded module $N$ is not isomorphic to $M$. It is easy to see that the filtration is induced by a suitable 1-PSG $\lambda$, hence
$\lim_{t\to0}  \lambda (t) M \in C_{N} \subsetneqq \overline{C_{M}}$, and so $C_{M}$ is no closed.

The last statement is clear, since every orbit closure contains a closed orbit.
\end{proof}

\ps
\subsection{The ind-variety \texorpdfstring{$\Hom(G,H)$}{Hom(G,H)} is finite-dimensional}
Here is a first main result about the ind-variety \itind{$\Hom(G,H)$}.
\begin{proposition} \label{Hom(G,H).prop}
Let $G, H$ be linear algebraic groups.
\be
\item 
The ind-variety $\Hom(G,H)$ is finite-dimensional.
\item 
If $G$ is reductive, then $\Hom(G,\GL(V))$ is a
countable union of closed $\GL(V)$-orbits, hence strongly smooth of dimension $\leq (\dim V)^{2}$.
\item \label{Hom(G,GL(V))-for-G-semisimple}
If $G^{\circ}$ is semisimple or if $G$ finite, then $\Hom(G,H)$ is an affine algebraic variety. Moreover, $\Hom(G,\GL(V))$ is a finite union of closed $\GL(V)$-orbits and thus a smooth affine algebraic variety of dimension $\leq(\dim V)^{2}$.
\item \label{Hom(U,H)a}
If $U$ is a unipotent group, then $\Hom(U,H)$ is an affine algebraic variety of dimension $\leq \dim U \cdot \dim H^{u}$.
\ee
\end{proposition}
\begin{proof}
(1)
Using a closed embedding $H \subseteq \GL_{n}$ it suffices to prove that $\Hom(G,\GL_{n})$ is finite-dimensional (Remark~\ref{closed-subgroup-Hom.rem}). There is a finite subset $F \subset G$ such that $G = \overline{\langle F \rangle}$. Hence, we get an injective morphism of ind-varieties
$$
\phi\colon \Hom(G,\GL_{n}) \to (\GL_{n})^{F}, \quad \lambda \mapsto (\lambda(g)\mid g\in F).
$$
Now the claim follows from Proposition~\ref{small-fibers-gives-variety.prop}.
\ps
(2)
Since the representations of a reductive group are completely reducible the $\GL(V)$-orbits in $\Hom(G,\GL(V))$ are closed, by Lemma~\ref{Artin.lem}. Moreover, the number of equivalence classes of $n$-dimensional representations of $G$ are countable, hence $\Hom(G,\GL(V))$ is a countable union of closed $\GL(V)$-orbits. For the last claim we refer to Example~\ref{countable-union-of-varieties.exa}.
\ps
(3) 
Since there  are only finitely many equivalence classes of $n$-dimensional representations of $G$, the claim follows from (2).
\ps
(4) 
We can assume that $H$ is a closed subgroup of $\GL_{n}$ (Remark~\ref{closed-subgroup-Hom.rem}). Then we have embeddings
$$
\Hom(U,H) \subseteq \Hom(U,\GL_{n}) \subseteq \Mor(U,\GL_{n}) \subseteq \Mor(U,\M_{n}) = \OOO(U)\otimes \M_{n}.
$$
where the first two are closed immersions, and the last is a locally closed immersion, by Corollary~\ref{special-open-set.cor}.
The first claim of (4)
follows if we show that the image of $\Hom(U,H)$ in $\OOO(U)\otimes \M_{n}$ is contained in a finite dimensional subspace.

If $\lambda\in\Hom(U,H)$ we denote by $\tilde\lambda$ the image in $\OOO(U)\otimes \M_{n}$. Using the exponential isomorphism $\exp_{U}\colon\Lie U \simto U$ (Example~\ref{unipotent-exp.exa}) we see that every linear map $\ell\colon \kk \to \Lie U$ defines a homomorphism $\lambda_{\ell}:=\exp_{U}\circ\, \ell\colon \kplus \to U$ of algebraic groups.
If $\rho\in\Hom(U,H)$, then the image of $\rho\circ\lambda_{\ell} \in \Hom( \kplus ,H)$
in $\kk[s]\otimes \M_{n}$ is equal to $(\lambda_{\ell}^{*}\otimes\id)(\tilde \rho)$:
$$
\begin{CD}
\Hom(U,H) @>{\subseteq}>>  \Mor(U,\M_{n}) @= \OOO(U)\otimes \M_{n} @>{\exp_{U}^{*}\otimes\id}>{\simeq}> 
\OOO(\Lie U)\otimes \M_{n}\\
@VV{\rho\mapsto\rho\circ\lambda_{\ell}}V    @VV{\rho\mapsto\rho\circ\lambda_{\ell}}V   
@VV{\lambda_{\ell}^{*}\otimes\id}V    @VV{\ell^{*}\otimes\id}V\\
\Hom(\kplus,H) @>{\subseteq}>> \Mor(\kk,\M_{n}) @= \kk[s]\otimes \M_{n} @= \kk[s]\otimes \M_{n}
\end{CD}
$$
We know that the image of $\Hom(\kplus,H)$ in $\kk[s]\otimes\M_{n}$ is contained in the finite dimensional subspace $\kk[s]_{<n}\otimes \M_{n}$, see Proposition~\ref{Hom-kplus.prop}. If we denote by $S$ the image of $\Hom(U,H)$ in $\OOO(\Lie U)\otimes \M_{n}$, it follows that $\ell^{*}(S) \subseteq \kk[s]_{<n}\otimes \M_{n}$ for all linear maps $\ell\colon\kk \to \Lie U$. This clearly implies that $S$ is contained in a finite-dimensional subspace of $\OOO(\Lie U)\otimes \M_{n}$.

It remains to prove the dimension estimate. We can find $d:= \dim U$ elements $u_{1},\ldots,u_{d}\in U$ such that $\overline{\langle u_{1},\ldots,u_{d}\rangle} = U$. Then we obtain an injective morphism
$\Hom(U,H) \into (H^{\text{\it u}})^{d}$, $\lambda \mapsto (\lambda(u_{1}),\ldots,\lambda(u_{d}))$, and the claim follows.
\end{proof}

\begin{example}
If $G=T$ is a torus of dimension $d \geq 1$,
then $\Hom(T,\GL_{n})$ is finite-dimensional, with infinitely many connected components. In fact, it is a countable, but not finite union of closed $\GL_{n}$-orbits and therefore of dimension $\leq n^{2}-n$. In particular, $\Hom(T,\GL_{n})$ is not an algebraic variety.
\end{example}

\begin{example}\label{F-dot-T.exa}
Consider the semidirect product $N_{n}:=S_{n}\ltimes (\kst)^{n}$ with the obvious action of $S_{n}$ on $(\kst)^{n}$.
Then again $\Hom(N_{n},\GL(V))$ is not algebraic as soon as $\dim V \geq n$. In fact, for every $k \in \ZZ$
we have a homomorphism $p_{k}\colon N_{n}\to N_{n}$ given
by $(\sigma,t) \mapsto (\sigma,t^{k})$. Thus, starting with a faithful representation $\rho\colon N_{n}\to\GL(V)$ we obtain a countable but not finite set of representations $ \rho \circ p_k$, $k \in \NN$, which are nonequivalent since they have different kernels.

This example generalizes immediately to the case of a semidirect product $N:=F\ltimes T$ of a finite group $F$ with a torus $T$. Again, the maps $p_{k}\colon F\ltimes T \to F\ltimes T$, $(f,t)\mapsto (f,t^{k})$ are homomorphisms of algebraic groups for any $k \in\ZZ$. Hence, $\Hom(N,\GL_{n})$ is not algebraic if $n$ is large enough.
\end{example}

\begin{example}
If $U$ is a unipotent group, then the only semisimple $U$-modules are the trivial ones. It follows that $\Hom(U,\GL(V))$ contains a single closed orbit, namely $O = \{\rho_{0}\}$
where $\rho_{0}$ is the trivial representation, hence the unique fixed point of $\GL(V)$.
\end{example}

We set $\Autgr(G):=\Hom(G,G)$, the group of regular automorphisms of the linear algebraic group $G$. We have a homomorphism of ind-groups $\Int\colon G \to \Autgr(G)$ sending $g \in G$ 
to the inner automorphism $h\mapsto ghg^{-1}$.\idx{$\Autgr(G)$}

\begin{example}\label{Aut.exa}
\be
\item
$\Autgr(\SLtwo) = \Int(\SLtwo) \simeq \PSL_{2}$.
\item
For $n>2$ the involution $\tau\colon\SL_{n}\simto\SL_{n}$, $A\mapsto (A^{t})^{-1}$, is not inner, and we have 
$\Autgr (\SL_{n}) = \langle\tau\rangle \ltimes \Int(\SL_{n})$.
\item \label{torus}
For $T =( \kst)^{n}$ we have $\Aut_{gr}(T) = \GL_{n}(\ZZ)$, a discrete ind-group.
\item 
If $U$ is a commutative unipotent group, hence isomorphic to the additive group $\Lie U^{+}$ by the exponential map (see Example~\ref{unipotent-exp.exa}), then $\Autgr(U) \simeq \GL(\Lie U)$.
\ee
\end{example}

\ps
\subsection{Relation with the Lie algebras}
For any homomorphism $\rho\colon G \to H$ of linear algebraic groups we get a homomorphism of Lie algebras $d\rho\colon \Lie G \to \Lie H$. This defines a map
$$
L_{G,H}\colon \Hom(G,H) \into \LLie(\Lie G, \Lie H), \ \rho\mapsto d\rho,
$$
where $\LLie(\Lie G,\Lie H) \subseteq \LLL(\Lie G,\Lie H)$ denotes the closed subvariety of Lie algebra homomorphisms.
Note that $\LLie(\Lie G,\Lie H)$ is an affine variety even though $\Hom(G,H)$ might not be a variety.

\begin{proposition}\label{Hom-and-Lie.prop}
Let $G,H$ be linear algebraic groups.
\be
\item 
$L_{G,H} \colon \Hom(G,H) \to \LLie(\Lie G, \Lie H)$ is an ind-morphism. If $G$ is connected, then $L_{G,H}$ is injective.
\item 
If $G$ is a connected semisimple group, then $L_{G,H}$ is a closed immersion of affine varieties. If $G$ is simply connected, then $L_{G,H}$ is an isomorphism.
\item 
If $U$ is unipotent, then $L_{U,H}$ is a closed immersion of affine varieties.
\ee
\end{proposition}

\begin{proof}
(1)
The first part follows from Lemma~\ref{mor-to-tangent.lem}, and the second one is well known since a homomorphism $\rho\colon G \to H$ for a connected group $G$ is determined by $d\rho\colon \Lie G \to \Lie H$, see Section~\ref{notation.subsec}.
\ps
(2)
We first consider the case $H = \GL(V)$.
We will write $L_G$ instead of $L_{G, \GL (V) }$.
Then $\Hom(G,\GL(V))$ is an affine variety consisting of finitely many closed $\GL(V)$-orbits, by Proposition~\ref{Hom(G,H).prop}(\ref{Hom(G,GL(V))-for-G-semisimple}). Under $L_{G}$, each orbit is mapped isomorphically onto its image. In fact, we have seen above that the orbit $C_{\rho}$ of a representation $\rho\colon G \to \GL(V)$ is isomorphic to $\GL(V)/\GL_{G}(V)$. As remarked in Section~\ref{notation.subsec} we have $\GL_{G}(V) = \GL_{\Lie G}(V)$, and so the orbit $C_{\rho}$ is isomorphic to the orbit $C_{d\rho}$ of the representation $d\rho\colon \Lie G \to \End(V)$. Thus the map $L_{G}\colon \Hom(G,\GL(V)) \to \LLie(\Lie G, \End(V))$ is a closed immersion. 

If $G$ is simply connected, then every representation of $\Lie G$ is induced by a representation of $G$ (see Section~\ref{notation.subsec}), and so $L_{G}$ is an isomorphism.

In general, we can choose a closed embedding
$H \subseteq \GL(V)$, so that $\Lie H \subseteq \End(V)$, and we get a commutative diagram
$$
\begin{CD}
\Hom(G,H) @>{L_{G,H}}>> \LLie(\Lie G, \Lie H) \\
@VV{\subseteq}V   @VV{\subseteq}V  \\
\Hom(G,\GL(V)) @>{L_{G}}>{\subseteq}>  \LLie(\Lie G, \End(V))
\end{CD}
$$
where the lower horizontal map $L_{G}$ is a closed immersion. Thus $L_{G,H}$ is a closed immersion
(see Lemma~\ref{closed-immersion.lem}).

If  $G$ is connected and
$L_{G}$ is an isomorphism, then $L_{G,H}$ is also an isomorphism, since it is surjective. In fact, if $\rho\colon G \to \GL(V)$ is a homomorphism such that $d\rho\colon \Lie G \to \End(V)$ has its image in $\Lie H$, then $\rho(G) \subset H$, see Section~\ref{notation.subsec}.
\ps
(3)
By Lemma~\ref{exp-hom.lem} we get, for every homomorphism $\rho\colon U \to \GL(V)$, a commutative diagram
$$
\begin{CD}
(\Lie U)  @>{d\rho}>>  \NNN(V)\\
@V{\simeq}V{\exp_{U}}V   @V{\simeq}V{\exp}V \\
U@>{\rho}>>    \UUU (V)
\end{CD}
$$
where $\UUU (V) \subseteq \GL(V)$ are the unipotent elements  and $\NNN(V) \subseteq \End(V)$ the nilpotent elements (notation from Lemma~\ref{exp-basics.lem}).
This allows to define an ind-morphism $E_{U}\colon \Mor(\Lie U, \NNN(V)) \to \Mor(U,\UUU (V))$ by $\phi\mapsto \exp\circ\phi\circ\exp_{U}^{-1}$, with the property that $E_{U}(d\rho) = \rho$ for any homomorphism $\rho\colon U \to \GL(V)$.

The ind-morphism $L_U \colon \Hom(U, \GL(V)) \into \LLie(\Lie U, \End (V) )$ has values in $\Mor(\Lie U,\NNN(V))$, and thus induces an ind-mor\-phism ${\tilde L_{U}} \colon \Hom(U,\GL(V)) \to \Mor(\Lie U,\NNN(V))$ which makes the following diagram commutative
$$
\begin{CD}
\Hom(U,\GL(V)) @>{\tilde L_{U}}>> \Mor(\Lie U,\NNN(V)) \\
@VV{\id}V @VV{E_{U}}V \\
\Hom(U,\GL(V)) @>{\subseteq}>> \Mor(U, \UUU (V))
\end{CD}
$$
Since $\Hom(U,\GL(V)) \subseteq \Mor(U, \UUU (V))$ is closed it follows that $\tilde L_{U}$ is a closed immersion of ind-varieties (see Lemma~\ref{closed-immersion.lem}).  This proves that $L_U$ is a closed immersion as well.

The general case is obtained by embedding $H$ into a $\GL(V)$.
\end{proof}

\begin{remark}
In the proof above we did not use that $\Hom(U,H)$ is an affine variety. So we got a new proof for this statement, cf. Proposition~\ref{Hom(G,H).prop}(\ref{Hom(U,H)a}).
\end{remark}

\begin{remark}
If $T$ is a torus of dimension $d \geq 1$,
then $L_{T}\colon \Hom(T,\GL(V)) \into \LLie(\Lie T,\End(V))$ is an injective ind-morphism whose image is a countable,
but not finite union of $\GL(V)$-orbits in an affine $\GL(V)$-variety. Thus $L_{T}$ cannot be a closed immersion.
\end{remark}

\ps
\subsection{A generalization}
We start with the following result.
\begin{proposition}  \label{restriction-to-subgroups.prop}
Let $G, L$ be linear algebraic groups, and let $H,K \subseteq G$ be two closed connected subgroups generating $G$.
Then the image of the ind-morphism
$$
\Delta \colon \Hom(G,L) \to \Hom(H,L) \times \Hom(K,L), \ \rho\mapsto (\rho|_{H},\rho|_{K}),
$$
is closed.
\end{proposition}
\begin{proof}
There is an $n>1$ such that the multiplication map 
$$
\nu\colon \underbrace{H\times K \times H \times \cdots\times K}_{2n\text{ factors}} \to G
$$ 
is surjective. Now consider the multiplication map for $4n$ factors
$$
\mu\colon \underbrace{H\times K \times H \times \cdots\times K}_{4n\text{ factors}} \to G.
$$
The morphism $\mu$ defines an equivalence relation $\sim$ on the product $H\times K \times H \times \cdots\times K$ of $4n$ factors  in the usual way: $a \sim b$ if and only if $\mu(a) = \mu(b)$.

If $\phi\colon H \to L$ and $\psi\colon H \to L$ are two homomorphisms, then we obtain a morphism $[\phi,\psi]\colon 
\underbrace{H\times K \times H \times \cdots\times K}_{4n\text{ factors}} \to L$ defined by:
\[
[\phi,\psi](h_{1},k_{1},h_{2},\ldots,h_{2n},k_{2n}) = \phi(h_{1}) \psi(k_{1}) \phi(h_{2}) \cdots \phi(h_{2n})\psi(k_{2n} ).
\]
It is clear that this induces an ind-morphism
$$
[\ ,\ ]\colon \Hom(H,L) \times \Hom(K,L) \to \Mor(\underbrace{H\times K \times H \times \cdots\times K}_{4n\text{ factors}},L)
$$
Now the equivalence relation $\sim$ defines a closed subset $R \subseteq \Hom(H,L)\times\Hom(K,L)$ in the following way:
$$
R := \{(\phi,\psi) \in \Hom(H,L)\times\Hom(K,L) \mid  [\phi,\psi](a) = [\phi,\psi](b) \text{ for all }a\sim b\}.
$$
We claim that $R$ is the image of $\Delta$. Clearly, $\Delta(\Hom(G,L)) \subseteq R$. 
On the other hand, it follows from the construction that $(\phi,\psi) \in R$ if and only if the morphism $[\phi,\psi] \colon H\times K \times H \times \cdots\times K \to L$, considered as a map, factors through $G$ and induces a homomorphism of groups $\overline{[\phi,\psi]}\colon G \to L$. The first statement is clear. As for the second, let $g_1,g_2 \in G$. They may be written in the form $g_1= h_1 k_1 \cdots h_n k_n$, $g_2 = h'_1k'_1 \cdots h'_n k'_n$, and we get
\begin{eqnarray*}
\overline{[\phi,\psi]} (g_1 g_2)  & = & \overline{[\phi,\psi]}  (  h_1k_1 \cdots h_nk_n h'_1k'_1 \cdots h'_n k'_n )\\
         & = & \phi(h_1) \psi(k_1)  \cdots \phi(h_n) \psi(k_n) \phi(h'_1) \psi(k'_1)  \cdots \phi(h'_n) \psi(k'_n) \\
         & = & \overline{[\phi,\psi]}(g_1) \overline{[\phi,\psi]}(g_2). 
\end{eqnarray*}

The next lemma implies that $\overline{[\phi,\psi]}$ is a morphism, hence a homomorphism of algebraic groups, and the claim follows.
\end{proof}

\begin{lemma} Let $X,Y,Z$ be irreducible
affine varieties where $Y$ is normal. Let $\phi\colon X \to Y$ be a surjective morphism and $f\colon Y \to Z$ an arbitrary map. If the composition $\psi:=f\circ\phi\colon X \to Z$ is a morphism, then so is $f$.
\end{lemma}
\begin{proof}
Denoting by $\Gamma_{f}$ the graph of the map $f$  and by $\Gamma_{\psi}$  the graph of the morphism $\psi$ we get the following commutative diagram:
$$
\begin{CD}
X @>{\simeq}>> \Gamma_{\psi} @>{\subseteq}>> X \times Z \\
@VV{\phi}V @VV{\text{\tiny surjective}}V @VV{\phi \times\id}V \\
Y @>>{\text{\tiny bijective}}> \Gamma_{f} @>{\subseteq}>> Y \times Z \\
\end{CD}
$$
It follows that $\Gamma_{f}$ is the image of the irreducible
variety $\Gamma_{\psi}$ under the morphism $\phi\times\id$, hence it is a constructible  subset of $Y\times Z$. The projection $\pr_{Y}\colon Y \times Z \to Y$ induces a surjective morphism $p\colon \overline{\Gamma_{f}} \to Y$ which is injective on an open dense set, hence it is birational. Since $\overline{\Gamma_{f}}$ is irreducible and 
$Y$ is normal this implies that $p$ is an isomorphism, by Lemma \ref{Igusa.lem}. Hence
$\Gamma_{f}=\overline{\Gamma_{f}}$ and $f = \pr_{Z}\circ p^{-1}$ is a morphism.
\end{proof}

\begin{corollary} \label{G-generated-by-m-subgroups.cor} 
Let $G,H$ be linear algebraic groups, and let $G_{1},G_{2},\ldots,G_{m}\subseteq G$ be closed connected subgroups generating $G$. If all $\Hom(G_{i},H)$ are affine varieties, then so is $\Hom(G,H)$.
\end{corollary}

\begin{proof}
(a) We begin with the case $m=2$. We have an injective ind-morphism $\Delta \colon \Hom(G,H) \to \Hom(G_1,H) \times \Hom(G_2, H)$ with a closed image, by Proposition~\ref{restriction-to-subgroups.prop}. 
If $\kk$ is uncountable, the result now follows from Lemma~\ref{bijective-morphisms.lem}. 

In general, consider a field extension $\KK / \kk$ where $\KK$ is algebraically closed and uncountable. Since we have $\Hom(G_{\KK},H_{\KK}) = \Hom(G,H)_{\KK}$ in a canonical way (cf. Proposition~\ref{field-extensions-for-morphisms.prop}), we also get the conclusion in this case.
\ps
(b)  The general case easily follows by induction. Indeed, a subgroup generated by any collection of closed connected subgroups of a linear algebraic group is closed and connected (see e.g. \cite[Proposition~7.5]{Hu1975Linear-algebraic-g}).
\end{proof}

In order to extend the results above to non-connected groups $G$ we need the following.

\begin{lemma} \label{Hom(G,H)-and-Hom(G0,H).lem}
Let $G,H$ be linear algebraic groups. 
\be
\item \label{Hom1.item}
If $\Hom(G^{\circ}, H)$ is an affine variety, then so is $\Hom(G,H)$.
\item \label{Hom2.item}
If $G^{\circ}$ is a nontrivial torus, then $\Hom(G,\GL_{n})$ is not algebraic for large $n$.
\ee
\end{lemma}
\begin{proof}
(1) There exists a finite subgroup $F \subseteq G$ such that $G = F G^{\circ}$, see  \cite[Lemme~5.11]{BoSe1964Theoremes-de-finit}
(cf. \cite[Theorem~1.1]{Br2015On-extensions-of-a}).
Hence the induced homomorphism of algebraic groups $p\colon F\ltimes G^{\circ} \to G$ is surjective. We claim that the image of the ind-morphism
$$
\Delta\colon \Hom(G,H) \to \Hom(F,H) \times \Hom(G^{\circ},H), \ \rho\mapsto (\rho|_{F},\rho|_{G^{\circ}}),
$$
is closed.
Consider the subset $R\subseteq \Hom(F,H)\times\Hom(G^{\circ},H)$ of pairs $(\phi,\psi)$ defined by the following two conditions:
\be
\item[(i)] $\psi(fgf^{-1}) = \phi(f)\psi(g)\phi(f)^{-1}$ for all $f \in F$, $g \in G^{\circ}$;
\item[(ii)]  $\phi(h) = \psi(h)$ for all $h \in F\cap G^{\circ}$.
\ee
It is easy to see that each condition defines a closed subset, and so $R$ is closed. Moreover, (i) implies that $\phi(F)$ normalizes $\psi(G^{\circ})$ and that we get a homomorphism of algebraic groups $[\phi,\psi] \colon F\ltimes G^{\circ} \to \phi(F)\ltimes \psi(G^{\circ})$, and from (ii) we obtain that 
$[\phi,\psi]$ factors through $p\colon F\ltimes G^{\circ} \to G$, hence defines a homomorphism of groups $\overline{[\phi,\psi]}\colon G \to H$. This homomorphism being regular on $G^{\circ}$, a classical argument proves us that it is regular on $G$.
Thus $R$ is the image of the injective ind-morphism $\Delta$.

If $\Hom(G^{\circ},H)$ is affine, then $R$ is affine. Since $\Delta$ induces a bijective ind-morphism $\Hom(G,H)\to R$, it follows again from Lemma~\ref{bijective-morphisms.lem} that $\Hom(G,H)$ is affine in case $\kk$ is uncountable. The general case is obtained from this by base field extension $\KK/\kk$ as in 
the proof of Corollary~\ref{G-generated-by-m-subgroups.cor} above.
\ps
(2)
If $G^{\circ}$ is a torus, then $N:=F \cap G^{\circ}$ is normal in $G$, since it is normal in $F$ and in $G^{\circ}$, and the quotient $G/N$ is the semidirect product $F/N\ltimes G^{\circ}/N$. Now the claim follows from Example~\ref{F-dot-T.exa}.
\end{proof}

\begin{remark} 
It is not true that if $\Hom(G,H)$ is affine, then $\Hom(G^{\circ},H)$ is affine. As an example take the semidirect product $G:=\ZZ/2\ltimes \kst$ where $\ZZ/2$ acts by $t \mapsto t^{-1}$. Then $\Hom(G,\kst)$ has two elements, hence is affine, whereas $\Hom(G^{\circ},\kst) \simeq \ZZ$.
However, we have the following partial converse of the lemma above.
\end{remark}

\begin{lemma}
Let $G$ be a linear algebraic group. If $\Hom(G,\GL_{n})$ is an affine variety for large $n$, then $\Hom(G^{\circ},H)$ is an affine variety for every linear algebraic group $H$.
\end{lemma}

\begin{proof}
By the following proposition we have to show that the radical $\rad G$ is unipotent, i.e. that the connected reductive group
$G^{\circ}/\rad_{u}G$ is semisimple. If not, then the quotient $H:=G/(\rad_{u}G\;(G^{\circ},G^{\circ}))$ has the property that $H^{\circ}=G^{\circ}/(\rad_{u}G\;(G^{\circ},G^{\circ}))$ is a nontrivial torus. By Lemma~\ref{Hom(G,H)-and-Hom(G0,H).lem}(\ref{Hom2.item}) this implies that the closed subset $\Hom(H,\GL_{n})\subseteq \Hom(G,\GL_{n})$ is not affine for large $n$, and so $\Hom(G,\GL_{n})$ is neither, contradicting the assumption. 
\end{proof}

The following proposition characterizes the linear algebraic groups $G$ with the property that all $\Hom(G,H)$ are affine varieties.

\begin{proposition}\label{Hom-algebraic.prop}
For a connected  linear algebraic group $G$ the following assertions are equivalent.
\be
\item[(i)]
The radical $\rad G$ is unipotent;
\item[(ii)]
$G$ is generated by unipotent elements;
\item[(iii)]
The character group of $G$ is trivial;
\item[(iv)]
$\Hom(G,H)$ is an affine variety for any linear algebraic group $H$.
\ee
\end{proposition}

\begin{proof}
(i) $\Rightarrow$ (ii): 
If the radical is unipotent, then $\bar G:=G/\rad_{u}G$ is semisimple, hence generated by unipotent elements, because a semisimple group is generated by the root subgroups, see \cite[Theorem~9.4.1]{Sp1998Linear-algebraic-g}. It follows that $G$ is generated by unipotent elements since every fiber of $G \to \bar G$ of a unipotent element contains unipotent elements.
\ps
(ii) $\Rightarrow$ (iii):
If $\chi\colon G \to \kst$ is a character, then all unipotent elements belong to $\ker\chi$. Hence $\chi$ is trivial.
\ps
(iii) $\Rightarrow$ (i):
The group $\bar G := G /\rad_{u} G$ is reductive. Since $\bar G / (\bar G,\bar G)$ is a torus we get $\bar G = (\bar G,\bar G)$, and thus $\bar G$ is semisimple, see \cite[Proposition~14.2]{Bo1991Linear-algebraic-g}. Hence $\rad G = \rad_{u}G$ as claimed.
\ps
(ii) $\Rightarrow$ (iv):
If $G$ is generated by unipotent elements, then it is generated by a finite set of closed subgroups isomorphic to $\kplus$. Since $\Hom(\kplus,H)$ is an affine variety (Proposition~\ref{Hom-kplus.prop}(\ref{Hom(k,G)-closed-hence-affine-variety}) or Proposition~\ref{Hom(G,H).prop}(\ref{Hom(U,H)a})) the claim follows from Corollary~\ref{G-generated-by-m-subgroups.cor} and  Lemma~\ref{Hom(G,H)-and-Hom(G0,H).lem}.
\ps
(iv) $\Rightarrow$ (iii):
By base change we can assume that $\kk$ is uncountable.
Let $\chi\colon G \to \kst$ be a nontrivial character. It is clear, that the image of $\Hom(\kst,\kst) \into \Hom(G,\kst)$, $\rho\mapsto \rho\circ\chi$, is closed since it consists of the homomorphisms $\phi\colon G \to\kst$ such that $\ker\phi\supseteq\ker\chi$. But $\Hom(\kst,\kst)\simeq \NN$ is discrete, and such a set cannot be closed in an affine variety.
\end{proof}

Let us collect the results from this section in the following theorem.
\begin{theorem}\label{Main-Hom.prop}
Let $G, H$ be linear algebraic groups.
\be
\item 
The ind-variety $\Hom(G,H)$ is finite-dimensional.
\item 
If the radical of $G$ is unipotent, then $\Hom(G,H)$ is an affine variety.
\item 
If $G$ is reductive, then $\Hom(G,\GL(V))$ is a
countable union of closed $\GL(V)$-orbits, hence it is strongly smooth of dimension $\leq (\dim V)^{2}$.
\item
If $G^{\circ}$ is semisimple or if $G$ is finite, then $\Hom(G,\GL(V))$ is a finite union of closed $\GL(V)$-orbits and thus a smooth affine algebraic variety of dimension $\leq(\dim V)^{2}$.
\item
If $U$ is a unipotent group, then $\Hom(U,H)$ is an affine algebraic variety of dimension $\leq \dim U \cdot \dim H^{u}$.
\ee
\end{theorem}

\begin{question}\label{Delta.ques}
Let $G,L$ be linear algebraic groups, and let $H,K \subset G$ be closed subgroups which generate $G$.
Is it true that the canonical map 
$$
\Delta\colon\Hom(G,L) \to \Hom(H,L) \times \Hom(K,L)
$$
is a closed immersion of ind-varieties?
\end{question}

\pmed
\section{Locally Finite Elements of Ind-Groups}\label{locally-finite.sec}
\subsection{Locally finite elements of automorphism groups}\label{Locally-finite.subsec}
Let $ g\in\End(X)$ where $X$ is an affine variety, and assume that there is a $k>0$ such that $ g^{m}\in\End(X)_{k}$ for all $m\geq 0$. 
Then the closure $\overline{\{ g^{m}\mid m\in\NN\}} \subseteq \End(X)$ is a closed affine algebraic semigroup contained in $\End(X)_{k}$.

\begin{lemma}\label{deginvers.lem}
Let $ g\in\Aut(X)$ and assume that there is a $k>0$ such that $ g^{m}\in\End(X)_{k}$ for all $m\in\NN$. Then $ g^{n}\in\Aut(X)_{k}$ for all $n\in\ZZ$, and $\overline{\langle  g^{n}\mid n\in \NN\rangle} \subseteq \Aut(X)$ is a closed linear algebraic subgroup.
\end{lemma}
\begin{proof} 
Set $M:=\overline{\{ g^{m}\mid m\in\NN\}} \subseteq\End(X)_{k}$. This is a closed affine algebraic semigroup, and we have $ g M \subseteq M$. Since left multiplication with $ g$ defines an isomorphism $\End(X) \simto\End(X)$, Lemma~\ref{well-known.lem} below implies that  $ g M = M$. Hence $ g^{-1}\in M$, and the claim follows.
\end{proof}

\begin{lemma} \label{well-known.lem}
Let $\phi$ be an automorphism of an ind-variety $\VVV$, and let $Z \subseteq \VVV$ be a closed algebraic subset such that $\phi(Z) \subseteq Z$. Then we have $\phi(Z) = Z$.
\end{lemma}

\begin{proof}
The decreasing sequence of closed subsets $Z \supseteq \phi(Z) \supseteq \phi^2(Z)  \supseteq \cdots$ has to become stationary. Therefore, there exists  a $k\geq 0$ such that $\phi^{k+1} (Z) = \phi^{k}(Z)$. Applying $\phi^{-k}$ to this equality we get $\phi(Z) = Z$.
\end{proof}

\begin{definition}   \label{locfinite.def}
An endomorphism $\phi \in \End(X)$ is called \itind{locally finite}, resp. \itind{semisimple}, resp. \itind{locally nilpotent}, if the linear endomorphism $\phi^{*}$ of $\OOO(X)$ is  locally finite, resp.  semisimple, resp.  locally nilpotent (see Definition~\ref{locally-finite.def}).
A locally finite automorphism $\phi$ is called \itind{unipotent} if the linear endomorphism $(\phi^{*}-\id)$ of $\OOO(X)$ is locally nilpotent.

We denote by $\Endlf(X) \subseteq \End(X)$ the subset of {\it locally finite endomorphisms}\idx{locally finite endomorphism} and by $\Autlf(X) \subseteq \Aut(X)$ the subset of {\it locally finite automorphisms}\idx{locally finite automorphism}.\idx{$\Endlf(X)$}\idx{$\Autlf(X)$}
\end{definition}

\begin{lemma}\label{equivLF.lem}
For $ g\in\Aut(X)$ the following statements are equivalent:
\be
\item[(i)] $ g$ is locally finite;
\item[(ii)] There is a $k>0$ such that $ g^{m}\in \Aut(X)_{k}$ for all $m\in\NN$;
\item[(iii)] The closure $\lgr$ in $\Aut(X)$ of the subgroup generated by $ g$ is a linear algebraic group.
\ee
\end{lemma}
\begin{proof}
We can assume that the filtration of $\Aut(X)$ is obtained in the following way. We start with a filtration of $\OOO(X) = \bigcup_{k\geq 1}\OOO(X)_{k}$ by finite-dimensional subspaces such that $\OOO(X)_{1}$ generates $\OOO(X)$.
Then $\Aut(X)_{k}:=\{ g\in\Aut(X) \mid  g^{*}(\OOO(X)_{1})\subseteq\OOO(X)_{k}\}$ defines a filtration of $\Aut(X)$ by closed algebraic subsets.
\par\smallskip
(i)$\Rightarrow$(ii): 
If $ g$ is locally finite, then the linear span $\langle ( g^{*})^{m}(\OOO(X)_{1})\mid m\in\NN\rangle$ is finite-dimensional, hence contained in $\OOO(X)_{k}$ for some $k\geq 1$. But this means that $ g^{m}\in\Aut(X)_{k}$ for all $m\in\NN$.
\par\smallskip
(ii)$\Rightarrow$(iii): 
As in the proof of the previous lemma
consider the closed algebraic semigroup $M:=\overline{\{ g^{m}\mid m\in\NN\}} \subseteq\EndA{n}_{k}$.
It follows that the invertible elements $M^{*}$ of $M$ form a principal open set, hence a linear algebraic group which contains the group $\langle  g\rangle$ generated by $ g$ as a dense subset. 
\par\smallskip
(iii)$\Rightarrow$(i): 
This is clear since the action of an algebraic group on $X$ induces a locally finite and rational representation on $\OOO(X)$.
\end{proof}

Recall that a non-trivial connected linear algebraic group always contains a copy of the additive group $\kplus$ or a copy of the multiplicative group $\kst$. Hence, the following question seems natural.

\begin{question}\label{additive-or-multiplicative-group-in-Aut(X).ques}
Is it true that a nondiscrete automorphism group $\Aut(X)$ of an affine variety $X$ always contains a copy of the additive group $\kplus$ or a copy of the multiplicative group $\kst$? Equivalently, does it always contain locally finite elements of infinite order?
\end{question}

\ps
\subsection{Locally finite elements and Jordan decomposition in ind-groups}\label{locally-finite-ind-groups.subsec}
We now use the results above to define {\it locally finite elements}\idx{locally finite element} of an arbitrary affine ind-group $\GGG$, and to get the  \itind{Jordan decomposition} of such elements.

\begin{definition} 
Let $\GGG$ be an affine ind-group and $ g\in\GGG$. Then $ g$ is called \itind{locally finite} if the closure $\lgr$ of the subgroup generated by $ g$ is a linear algebraic group, i.e. there is an integer $k>0$ such that $\langle  g \rangle \subseteq \GGG_{k}$.
A locally finite element $ g\in\GGG$ is called \itind{semisimple} if $\lgr$ is a diagonalizable group\idx{diagonalizable group}, and \itind{unipotent} if $\lgr$ is a unipotent group.

We use the notation $\Glf, \Gss, \Gu\subseteq \GGG$ for the subsets of locally finite, semisimple and unipotent elements of $\GGG$.
\end{definition}\idx{$\GL$@$\Glf$}\idx{$\GL$@$\Gss$}\idx{$\GL$@$\Gu$}

\begin{example}\label{filter-by-alg-groups.exa}\idx{nested ind-group}
Let $\GGG$ be a nested ind-group, i.e. $\GGG$ admits a filtration consisting of closed algebraic subgroups (see Example~\ref{nested.exa}).
Then every element $ g \in \GGG$ is locally finite. We will prove a partial converse of this statement in Proposition~\ref{locally-finite-commutative-ind-group.prop}.
\end{example}

\begin{proposition}\label{Glf-weakly-closed.prop} Assume that  $\kk$ is uncountable.
Then the  subset $\Glf$ of  $\GGG$ is weakly closed. 
\end{proposition}

\begin{proof}
For $k,\ell \in \NN$ define the closed subsets $G_{k,\ell}:=\{ g\in\GGG\mid  g^{j}\in\GGG_{k} \text{ for }|j|\leq \ell\} \subseteq\GGG_{k}$. 
Then
$$
G_{k}:=\bigcap_{\ell} G_{k,\ell} = \{ g\in\GGG\mid \overline{\langle g\rangle} \subseteq \GGG_{k}\}
$$
is closed in $\GGG_{k}$ and $\Glf = \bigcup_{k} G_{k}$, showing that $\Glf$ is weakly closed (Proposition~\ref{indconstr.prop}).
\end{proof}

\begin{question} \label{locally-finite-elements-and-generation.ques}
If $\GGG$ is a connected ind-group, do we have $ \GGG = \langle  \GGG^{\text{\it lf}}   \rangle$, 
or at least $ \GGG = \overline{ \langle  \GGG^{\text{\it lf}}\rangle}$?
\end{question}

\ps
For a locally finite  $ g\in\GGG$ the structure of the commutative linear algebraic group $H:=\lgr$ is well-known: $H = H_{s}\times H_{u}$ where $H_{s}$ is the closed subgroup of the semisimple elements of $H$ and $H_{u}$ the closed subgroup of the unipotent elements. Moreover,  $H_{s}$ is a diagonalizable group isomorphic to $F\times ( \Cst)^{k}$ where $F$ is a finite cyclic group, and $H_{u}$ is either trivial or isomorphic to $\kplus$. Accordingly, we have a canonical decomposition $ g =  g_{s}\cdot  g_{u}$ where $ g_{s}$ is semisimple, $ g_{u}$ is unipotent, and $ g_{s}\cdot g_{u}=  g_{u}\cdot  g_{s}$. This decomposition is called the \itind{Jordan decomposition} of $ g$.

\begin{lemma}\label{JD.lem}
Let $ g\in\GGG$ be locally finite and $ g= g_{s}\cdot  g_{u}$ its Jordan decomposition.
\be
\item If $\phi\colon \GGG \to \HHH$ is a homomorphism of ind-groups, then $\phi( g)$ is locally finite and $\phi( g) = 
\phi( g_{s})\cdot \phi( g_{u})$ is its Jordan decomposition.
\item If $\HHH \subseteq \GGG$ is a closed subgroup and $ g\in\HHH$, then $ g_{s},  g_{u} \in \HHH$.
\ee
\end{lemma}
\begin{proof}
(1) is a consequence of the fact that the image  $\phi(G)\subseteq \HHH$ of a closed algebraic subgroup $G \subseteq \GGG$ is a closed algebraic subgroup of $\HHH$, and (2) is obvious.
\end{proof}

\begin{example} 
If $ g \in \Aut(X)$ is locally finite and $a\in X$ a fixed point of $ g$, then $a$ is fixed by $ g_{s}$ and $ g_{u}$, because it is a fixed point of the closed subgroup $\lgr$. On the other hand, the automorphism $ g := (x+y^{2}-1, -y)$ of $\A{2}$ is locally finite and fixed point free whereas  
both  $ g_{s}$ and $ g_{u}$  have fixed points. In fact, the Jordan decomposition is given by $ g = (x,-y)\cdot(x+y^{2}-1,y)$.
\end{example}

\ps
\subsection{Elements of finite order}
The first part of the next proposition is a well-known result from algebraic group theory. As for the second we could not find a reference.

\begin{proposition} 
Let $G$ be a linear algebraic group.
\be
\item If every element of $G$ has finite order, then $G$ is finite.
\item Let $F \subseteq G$ be a subgroup such that  every element of $F$ has finite order. Then $F$ is countable.
\ee
\end{proposition}

\begin{proof}
(1) Let $B \subseteq G$ be a Borel subgroup, i.e. a maximal connected solvable subgroup. Then $B$ must be trivial, because a nontrivial connected solvable group contains  either a $\kst$ or a $\kplus$. Since $G/B$ is a projective variety (\cite[Theorem~11.1]{Bo1991Linear-algebraic-g}) we get that $G$ is finite.

\ps
(2) The \name{Tits} alternative \cite[Theorem~1]{Ti1972Free-subgroups-in-} implies that $F$ contains a solvable subgroup of finite index. So we can assume that $F$ is solvable. Let $F = F^{(0)} \supset F^{(1)} \supset F^{(2)} \supset \cdots \supset F^{(m)}=\{e\}$
be the derived series, i.e., every subgroup is the commutator subgroup of the previous one. Taking the closures
$$
\overline{F}  \supseteq \overline{F^{(1)}} \supseteq \overline{F^{(2)}} \supseteq \cdots \supseteq \overline{F^{(m)}}=\{e\}
$$
we obtain a subnormal series with abelian factors. If $F$ is uncountable, then there is an $i<m$ such that $F_{i}:=F\cap \overline{F^{(i)}}$ is uncountable, but $F_{i+1}:=F \cap \overline{F^{(i+1)}}$ is countable. It follows that the image of $F_{i}$ in 
$H:=\overline{F^{(i)}}/\overline{F^{(i+1)}}$
is an uncountable abelian subgroup of the abelian linear group $H$. This is impossible, because
$H^{\circ} \simeq (\kst)^{p}\times (\kplus)^{q}$,  and the set of elements of finite order  in a torus is countable.
\end{proof}
The obvious generalization would be to show that an ind-group consisting of elements of finite order is discrete. 

\begin{question}\label{all-elements-of-finite-order-implies-discrete.ques} 
Is it true that every ind-group $\GGG$ consisting of elements of finite order is discrete?
More generally, is it true that a subgroup  $F \subseteq \GGG$ consisting of elements of finite order is countable?
\end{question}

Note that the second question has a negative answer for subgroups of $\GL(V)$ when $V$ is the $\kk$-vector space $\kinfty$. Let $(e_n)_{n \geq 1}$ be a basis of $V$. For each subset $S \subseteq \NN$ define the element $f_S \in \GL (V)$ by 
$$
f (e_n) = \begin{cases}  e_n & \text{ if }n \notin S,\\   -e_n & \text{ if } n \in S.
\end{cases}
$$ 
Then, $F:= \{ f_S \mid S \subseteq \NN \} \subset \GL(V)$ is a subgroup isomorphic to $(\ZZ /2 \ZZ)^{\NN}$ which is not countable. But note that the group $\GL(V)$ is not an ind-group.

\ps
Concerning the question above we have the following partial result.

\begin{lemma}\label{subgroups-with-finite-order-elements.lem}
Let $F \subseteq \Aut(X)$ be a commutative subgroup consisting of elements of finite order. Then $F$ is countable.
\end{lemma}

\begin{proof}
We have a faithful linear action of $F$ on the coordinate ring $\OOO(X)$ by $\kk$-algebra automorphisms. Since $F$ is commutative and consists of elements of finite order the representation of $F$ on $\OOO(X)$ can be diagonalized, i.e. there is a decomposition $\OOO(X) = \bigoplus_{\chi} \OOO(X)_{\chi}$ where $\chi$ runs through the characters of $F$ and 
\[
\OOO(X)_{\chi}:=\{f \in \OOO(X) \mid g f = \chi(g)\cdot f \text{ for all }g \in F\}.
\]
Since  $\OOO(X)$ is finitely generated as a $\kk$-algebra we can find a finite set of characters $\chi_{1}, \chi_{2},\ldots,\chi_{n}$ such that $V:=\bigoplus_{i=1}^{n}\OOO(X)_{\chi_{i}}$ generates $\OOO(X)$. Clearly, $g\in F$ acts trivially on $\OOO(X)$ if and only if $\chi_{1}(g)= \cdots = \chi_{n}(g) = 1$. Since the action of $F$ is faithful it follows that the homomorphism 
$$
F \to \kst^{n}, \ g \mapsto (\chi_{1}(g),\ldots,\chi_{n}(g)),
$$ 
is injective. Its image is in $\mu_{\infty}^{n}$ where $\mu_{\infty}:=\{\zeta\in\kst\mid \zeta^{m}=1 \text{ for some }m\geq 1\}$.
Since $\mu_{\infty}$ is countable the claim follows.
\end{proof}

\ps
\subsection{Nested ind-groups}\label{nested.subsec}
Recall that an ind-group $\GGG$ is \itind{nested} if $\GGG$ has an admissible filtration $\GGG = \bigcup_{k}\GGG_{k}$ consisting of closed algebraic subgroups $\GGG_{k}$.
The following result can be found in \cite[Remark 2.8]{KoPeZa2016On-Automorphism-Gr}.

\begin{lemma} \label{closed-subgroup-of-nested.lem}
A closed subgroup of a nested ind-group is nested.
\end{lemma}

We now prove the following analogous result.

\begin{lemma} \label{injective-homomorphism-into-nested.lem}
Let $\phi\colon \GGG \to \HHH$ be an injective homomorphism of ind-groups where $\GGG$ is connected and $\HHH$ nested. Then $\GGG$ is nested.
\end{lemma}

\begin{proof} 
We have $\HHH = \bigcup_{k}\HHH_{k}$ where the $\HHH_{k}$ are linear algebraic groups. Define the closed subgroups $\GGG_{k}:=\phi^{-1}(\HHH_{k}) \subseteq \GGG$. By Proposition~\ref{hom-to-algebraic.prop}(\ref{the-neutral-component-is-an-algebraic-group}) the connected components $\GGG_{k}^{\circ}$ are
algebraic groups.

We claim that $\bigcup_{k}\GGG_{k}^{\circ} =\GGG$. For $g \in \GGG$ there is an irreducible closed curve $C \subset \GGG$ such that $e,g\in C$. Since $C$ is algebraic the image $\phi(C)$ is contained in $\HHH_{k}$ for some $k\geq 1$. Hence $C \subset \GGG_{k} = \phi^{-1}(\HHH_{k})$. It follows that $C \subset \GGG_{k}^{\circ}$, and so $g \in \GGG_{k}^{\circ}$.
This shows that $\GGG=\bigcup_{k}\GGG_{k}^{\circ}$ is a nested ind-group. 
\end{proof}

With the same notation as in the lemma above we remark that if $\phi(\GGG)$ is closed, then $\phi$ is a closed immersion. However, we do not know whether this is always the case.
\ps
We have seen in Example~\ref{filter-by-alg-groups.exa} that any element of a nested ind-group is locally finite.
One might wonder if the converse holds: If all elements of an ind-group are locally finite, does it follow that the ind-group is nested? Before proving a partial result in this direction, let us mention that for a discrete ind-group this question is equivalent to the classical Burnside problem posed by \name{William Burnside} in 1902. Indeed, an element of a discrete ind-group is locally finite if and only if it has finite order. Therefore, the question is equivalent to the question if a finitely generated group in which every element has finite order must necessarily be a finite group. As one knows, the answer is negative. Thus the previous question should be asked only for connected ind-groups.\idx{Burnside problem}

\begin{question}\label{locally-finite-implies-nested.question}
If $\GGG$ is a connected ind-group whose elements are all locally finite, does
it follow that $\GGG$ is nested?
\end{question}

Here is a partial result in this direction.  

\begin{proposition}\label{locally-finite-commutative-ind-group.prop} Assume that $\kk$ is uncountable.
Let $X$ be an affine variety, and  let $\GGG \subseteq \Aut(X)$ be a commutative closed connected subgroup. If every element of $\GGG$ is locally finite, then $\GGG$ is nested.
\end{proposition}

\begin{proof}
For any $ g\in\GGG$ define the closed connected commutative algebraic subgroup $G_{ g}:=\overline{\langle  g^{m}\mid m\in\ZZ\rangle}^{\circ}\subseteq \GGG$ and its Lie algebra $L_{ g}:=\Lie G_{ g}\subset \Lie\GGG$. Set $L:=\sum_{ g}L_{ g} \subseteq\Lie\GGG$. Since $\Lie\GGG$ has countable dimension, we can find a countable set $\{ g_{i}\in\GGG\mid i\in\NN\}$ such that $L = \sum_{i\in\NN} L_{ g_{i}}$. For $n\in\NN$ define the subgroups 
$\GGG_{n}:=G_{ g_{1}}\cdot G_{ g_{2}}\cdots G_{ g_{n}}\subset \GGG$. By construction, these are closed connected commutative algebraic subgroups, and $\Lie \GGG_{n}=L_{n} :=\sum_{i=1}^{n} L_{ g_{i}}$. We claim that
$G:=\bigcup_{n} \GGG_{n}$ is equal to $\GGG$.

Denote by $F \subset \GGG$ the subgroup of elements of finite order. For any element  $ g \in \GGG$ there is an $n\in\NN$ such that $L_{ g}\subseteq L_{n}=\Lie \GGG_{n}$. It then follows from Corollary~\ref{locally-finite-VF-in-Liealg.cor} that $G_{ g} \subset \GGG_{n}$, hence $ g^{m}\in \GGG_{n}\subseteq G$ for some $m \geq 1$. Since the group $\GGG_n$ is divisible, i.e. all power maps $g \mapsto g^{k}$ are surjective, there exists an element $h \in \GGG_n$ such that $ g^m = h ^m$. Therefore, the element $ g' :=  g h^{-1}$ has finite order and is such that $ g' G =  g G$. This shows that $\GGG = F \cdot G$.

By Lemma~\ref{subgroups-with-finite-order-elements.lem}, $F$ is countable, so that it can be written as an increasing union  $F = \bigcup_{n}F_{n}$ of finite subgroups. 
Now, the closed algebraic subgroups $H_{n}:= F_{n}\cdot \GGG_{n} \subset \GGG$ satisfy $\GGG = \bigcup_{n}H_{n}$, and by Theorem~\ref{closed-algebraic-filtration.thm} they form an admissible filtration.
Since $\GGG_n$ has finite index in the connected group $H_n^{\circ}$, we have $H_n^{\circ} =\GGG_n$.

It remains to see that $\GGG= \bigcup_n \GGG_n$. Since $\GGG$ is curve-connected we can find, for every $ g \in \GGG$, an irreducible curve $C$ connecting $ g$ and $e$. Then there is an $n \geq 1$ such that $C \subset H_{n}$ and so $ g \in  H_{n}^{\circ} = \GGG_{n}$. Hence, $\GGG = \bigcup_{n} \GGG_{n}$.  
\end{proof}

\begin{remark} \label{nested.rem}
Let $\GGG$ be an ind-group and consider the following statements:
\be
\item $\GGG$ is nested;
\item For any finite subset $\{g_1, \ldots,g_n\}\subset \GGG$, the group $\overline{ \langle g_1, \ldots,g_n \rangle }$ is algebraic;
\item Any element of $\GGG$ is locally finite.
\ee
Then we have $(1) \Rightarrow (2) \Rightarrow (3)$,
but we have seen above that  $(3) \Rightarrow (2)$ does not hold in general. Moreover, it is unclear if the implication $(2) \Rightarrow (1)$ is true.
\end{remark}

\begin{example}There are interesting examples of \itind{nested discrete ind-groups}.
\be
\item The direct sum $F^{(\infty)}$ of countably many copies of a finite group $F$.
\item The symmetric group $S_{\infty} := \varinjlim S_{n}$.
\item The groups $\QQ^{+}/\ZZ = \varinjlim \frac{1}{n}\ZZ/\ZZ$ and $\ZZ_{p}^{+}/\ZZ =  \varinjlim \frac{1}{p^{k}}\ZZ/\ZZ$.
\ee
\end{example}

\ps
If $G$ is a commutative linear algebraic group, then the subsets $G^{\text{\it ss}}$ and $G^{\text{\it u}}$ of semisimple resp. unipotent elements are closed subgroups and $G = G^{\text{\it ss}}\times G^{\text{\it u}}$. Moreover, $G^{\text{\it u}}$ is a unipotent group isomorphic to the additive group of a finite-dimensional vector space, and $G^{\text{\it ss}}$ is a diagonalizable group which can be written in the form $G^{\text{\it ss}} \simeq F \times T$ where $F$ is a finite group and $T:=(G^{\text{\it ss}})^{\circ}$  a torus. These results carry over to commutative nested ind-groups, except for the last statement.

\begin{proposition}\label{ss-unipotent-for-nested.prop}
Let $\GGG$ be a commutative nested ind-group.
\be
\item 
The subsets $\Gss$ and $\Gu$ of $\GGG$ are closed subgroups, and $\GGG = \Gss \times \Gu$.
\item 
$\Gu$ is a nested unipotent ind-group isomorphic to the additive group of a vector space of 
countable dimension.
\item 
$(\Gss)^{\circ}$ is a \itind{nested torus}, i.e. a finite dimensional torus or isomorphic to
$(\kk^{*})^{\infty}:=\varinjlim \,(\kst)^{k}$.
\item
There is a closed discrete subgroup $\FFF \subset \Gss$ such that $\Gss = \FFF\cdot (\Gss)^{\circ}$.
\ee
\end{proposition}
\begin{proof}
(1), (2) and (3) are easy consequences of the finite dimensional linear case, and the proofs are left to the reader.
\ps
(4) Let $D$ be a diagonalizable group and $D'\subseteq D$ a closed subgroup. Let $F' \subset D'$ be a finite subgroups such that $D' = F' \cdot (D')^{\circ}$. Then there is a finite subgroup $F \subseteq D$ such that $D = F \cdot D^{\circ}$ and $F \supseteq F'$. Now the claim follows easily.
\end{proof}

\begin{question} Let $\GGG$ be a commutative nested ind-group.
Does there exist a closed discrete subgroup $\FFF \subset \Gss$ such that $\Gss \simeq \FFF \times (\Gss)^{\circ}$?
\end{question}

\ps
\subsection{Embeddings into \texorpdfstring{$\GL_{\infty}$}{}}\label{embeddings-into-GLinfty.subsec}
A typical case of a nested ind-group is $\GL_{\infty} = \varinjlim \GL_{n}$. Hence all closed subgroups of $\GL_{\infty}$ are nested ind-groups as well, by Lemma~\ref{closed-subgroup-of-nested.lem}. One might wonder whether a nested ind-group always admits an injective homomorphism into the group $\GL_{\infty}$. 
We show now that this is not case. 

Let $\AAA$ be the semidirect product $\kst\ltimes(\kinfty)^{+}$ where the action of $\kst$ on $(\kinfty)^{+}$ is given as follows:
$$
t (a_{1},a_{2},\ldots,a_{i},\ldots ) t^{-1} := (t\cdot a_{1},t^{2}\cdot a_{2},\ldots,t^{i}\cdot a_{i},\ldots).
$$
Note that $\AAA$ is the union of the closed linear algebraic subgroups $\AAA_{n}:=\kst\ltimes(\kn)^{+} \subset \AAA$, hence it is nested.

\begin{proposition} \label{no-injective-hom-of-A-into-GL_infty.prop}
There does not exist an injective homomorphism of ind-groups $\iota \colon \AAA \hookrightarrow \GL_{\infty}$.
\end{proposition} 

For $m\geq 1$ define $B_{m}:=\kst\ltimes \kplus$ where $\kst$ acts by $t s t^{-1} = t^{m}\cdot s$. We will use the following well known result, see e.g.  \cite[Proposition~27.2]{Hu1975Linear-algebraic-g}. 

\begin{lemma} \label{weights-of-Bm.lem}
Let $V$ be a finite-dimensional representation of $B_{m}$. If $s \in \kplus$, and if $V_{r}:=\{v\in V \mid t v = t^{r}\cdot v\}  \subseteq V$ is the weight space of weight $r \in\ZZ$, then $s(V_{r}) \subseteq \bigoplus_{i\geq 0}V_{r+im}$.
\end{lemma}

\begin{proof}[Proof of Proposition~\ref{no-injective-hom-of-A-into-GL_infty.prop}]
(1) It follows from the lemma above that a faithful representation $V$ of $B_{m}$ contains at least two different nonzero weight spaces $V_{r} \neq V_{r'}$ such that the difference $r - r'$ is divisible by $m$. In fact, if $s(V_{r}) = V_{r}$, then the actions of $\kst$ and of $\kplus$ on $V_{r}$ commute, hence the action of $\kplus$ on $V_{r}$ is trivial.

\ps
(2) Now assume that we have an injective homomorphism $\iota\colon \AAA \into \GL_{\infty}$. Then there is a $d>0$ such that $\iota(\kst) \subseteq \GL_{d}$, and for every $m\geq 1$ there exists an $r_{m}\geq 1$ and a commutative diagram
\[\tag{$*$}
\begin{CD}
\AAA_{m} @>{\iota_{m}}>{\subseteq}> \GL_{d+r_{m}} \\
@AA{\subseteq}A   @AA{\subseteq}A \\
\kst @>{\iota_{1}}>{\subseteq}>  \GL_{d}
\end{CD}
\]
Since $B_{m'}$ is a closed subgroup of $\AAA_{m}$ for all $m'\leq m$ this implies, by (1), that the $\AAA_{m}$-module $V:=\kk^{d+r_{m}}$ contains nonzero weight spaces $V_{r},V_{r'}$ such that $r-r'$ is divisible by $m'$ for all $m'\leq m$. This is not possible when $m$ is large enough,
since the diagram $(*)$ above shows that the nonzero weights of $\kst$ all belong to $\kk^{d}$.
\end{proof}
\begin{remark}
It is not difficult to see that the ind-group $\AAA$ occurs as a closed subgroup in the \name{de Jonqui\`ere}-subgroup $\JJJ(n) \subseteq \Aut(\An)$ for all $n\geq 2$, see Proposition~\ref{no-inbedding-of-Jn-into-GLinfty.prop}. As a consequence, $\JJJ(n)$ does not admit an injective homomorphism 
$\JJJ(n) \into \GL_{\infty}$.
\end{remark}

\pmed
\section{Automorphisms of \texorpdfstring{$G$}{G}-Varieties}
\subsection{\texorpdfstring{$G$}{G}-varieties and affine quotients}\label{quotient.subsec}
Let $G$ be a linear algebraic group and $X$ be an affine $G$-variety.
We say that {\it affine quotient $X\quot G$ exists}\idx{$\XX$@$X\quot G$} if the algebra of invariants $\OOO(X)^{G}$\idx{$\OOO(X)^{G}$} is finitely generated. In this case we define $X\quot G:=\Spec\OOO(X)^{G}$ and denote by $\pi_{X}\colon X \to X\quot G$ the canonical morphism defined by the inclusion $\OOO(X)^{G}\subset \OOO(X)$.
\idx{affine quotient $X\quot G$}
If the affine quotient exists, then it has the universal property that {\it every $G$-invariant morphism $\phi\colon X \to Z$ where $Z$ is affine factors uniquely through  $\pi_{X}$}:

\begin{center}
\begin{tikzcd}
X  \arrow[d, "\pi_{X}"']  \arrow[r, "\phi"] & Z   \\
X\quot G  \arrow[ru,"\bar\phi"']
\end{tikzcd}
\end{center}
Such a factorization needs not to exist if we do not assume that $Z$ is affine. In fact, there exist affine quotients $\pi_{X}\colon X \to X\quot G$ which are not surjective. Setting $Z:=X\quot G \setminus \{y_{0}\}$ where $y_{0}\notin \pi_{X}(X)$ the induced map $\phi\colon X \to Z$ is invariant, but does not factor through $\pi_{X}$. Such an example can be found in \cite[Example~4.10, page 231]{FeRi2005Actions-and-invari}. They consider the linear representation of $\kplus$ on $\M_{2}(\kk)$ 
where $s \in \kplus$ acts on $\M_{2}(\kk)$
by left multiplication with $\left[\begin{smallmatrix} 1& s\\0&1\end{smallmatrix}\right]$ and they even show that this action does not admit a {\it categorical quotient\/} at all. For more examples of this form see \cite{KrDu2015Invariants-and-Sep}.

\begin{lemma}
Let $X$ be an affine $G$-variety and assume that the affine quotient
$\pi_X \colon X \to X\quot G$ exists. Then,
for any affine variety $Z$, the canonical map
$$
\delta_{Z}\colon \Mor(X\quot G, Z) \to \Mor(X,Z)^{G}, \ \alpha \mapsto \alpha\circ\pi_{X},
$$
is an isomorphism of ind-varieties.
\end{lemma}
\begin{proof}
If $Z$ is a $\kk$-vector space, $Z=V$, then we have in a canonical way
$$
\Mor(X,V)^{G} \simeq (\OOO(X)\otimes V)^{G} = \OOO(X)^{G}\otimes V \simeq \Mor(X\quot G,V)
$$
showing that $\delta_{V}\colon \Mor(X\quot G,V) \to \Mor(X,V)^{G}$ is an isomorphism. If we embed $Z \subseteq V$ into a $\kk$-vector space, then $\Mor(X\quot G,Z) \subseteq \Mor(X\quot G,V)$ is closed as well as $\Mor(B,Z) \subseteq \Mor(B,V)$,
and the map $\delta_{Z}$ is induced by $\delta_{V}$, hence is an isomorphism.
\end{proof}

\begin{definition} \label{Some-subgroups-of-Aut(X).def}
Let $X$ be an affine $G$-variety where $G$ is a linear algebraic group. We denote by $\Aut_{G}(X) \subseteq \Aut(X)$ the subgroup of {\it $G$-equivariant automorphisms}, by $\Autorb(X) \subset \Aut(X)$ the subgroup of {\it orbit preserving automorphisms}, and by $\Autinv(X) \subset \Aut(X)$ the subgroup of {\it automorphisms preserving the invariant functions\/} $\OOO(X)^{G}$:
\begin{eqnarray*}
\Aut_{G}(X) &:=& \{\phi\in\Aut(X) \mid \phi(gx) = g\phi(x) \text{ for all } x \in X \text{ and } g \in G\}, \\
\Autorb(X) &:=& \{\phi\in\Aut(X) \mid \phi(Gx) = Gx \text{ for all } x \in X\}, \\
\Autinv(X) &:=& \{\phi\in\Aut(X) \mid \phi^{*}(f) = f \text{ for all } f \in \OOO(X)^{G}\}.
\end{eqnarray*}\idx{$\Autorb(X)$}\idx{$\Autinv(X)$}
\end{definition}

\begin{proposition}
\be
\item 
An automorphism $\phi$ of $X$ is orbit preserving if and only if it preserves the closures of the orbits.
\item
The subgroup $\Autorb(X)$ is contained in $\Autinv(X)$, and both are closed subgroups of $\Aut(X)$:
$$
\Autorb(X) \subseteq \Autinv(X) \subseteq \Aut(X).
$$
\item
The subgroup $\Aut_{G}(X) \subset \Aut(X)$ is closed.
\ee
\end{proposition}\label{Autorb-is-closed.prop}

\begin{proof}
(1) It is clear that an orbit preserving automorphism preserves the closures of the orbits. So assume that $\phi$ preserves the closures of the orbits. Then every $G$-stable closed subset is stable under $\phi$. If $O = Gx$ is an orbit, then the complement $\bar O \setminus O$ is closed and $G$-stable, hence stable under $\phi$, and the claim follows.
\ps
(2) Using (1) we get $\Autorb (X) = \bigcap_{x \in X} \Aut (X, \overline{G x} )$ which is closed by Proposition~\ref{Aut(X,Y).prop}.

For the subgroup $\Autinv(X)$ we use the fact that the subset 
$$
\{\phi\in\Aut(X) \mid f\circ\phi = f\} \subset \Aut(X)
$$ 
is closed for any $f \in \OOO(X)$.
\ps
(3) For $g \in G$ define $\gamma_g\colon \Aut (X) \to \Aut (X)$ by $\phi \mapsto \phi \circ \rho (g) \circ \phi^{-1} \circ \rho (g) ^{-1}$ where $\rho \colon G \to \Aut(X)$ is the action. Then we have
\[
\Aut_G (X) = \bigcap_{g \in G} (\gamma_g)^{-1} ( \id ),
\]
and the claim follows.
\end{proof}
\begin{remark} 
Consider the representation of $\kst$ on $\kk^{2}$ given by scalar multiplication. Then there are no invariants, and we get
$$
\Autorb(\kk^{2}) = \kst  \subsetneqq \Autinv(\kk^{2}) = \Aut(\kk^{2}).
$$
In fact, for any $\phi\in\Autorb(\kk^{2})$ we have $\phi^{*}(x) = \alpha x$ and $\phi^{*}(y)=\beta y$ for some $\alpha,\beta \in \kst$. Since $\phi$ preserves all lines through the origin $0$, we must have $\alpha = \beta$.
\end{remark}
\begin{remark}
If the affine quotient $\pi_{X}\colon X \to B:= X\quot G$
exists, then 
$$
\Autinv(X) = \Aut_{\pi_{X}}(X):=\{\phi\in\Aut(X)\mid \pi_{X}\circ\phi = \phi\}.
$$\idx{$\Aut_{B}(X)$}
\end{remark}

\ps
\subsection{Principal \texorpdfstring{$G$}{G}-bundles}\label{G-bundle.subsec}
Recall that a \itind{principal $G$-bundle} is a $G$-variety $X$ together with a morphism $\pi\colon X \to B$ which is locally trivial in the \'etale topology with fiber $G$, i.e. for every $b\in B$ there is an \'etale morphism $\phi\colon U \to B$ containing $b$ in its image such that the fiber product $X\times_{B}U$ is $G$-isomorphic to $G \times U$:
$$
\begin{CD}
G \times U @>{\simeq}>> X \times_{B} U @>{\phi'}>>  X \\
@VV{\pr_{U}}V  @VV{\pi'}V  @VV{\pi}V \\
U @= U @>{\phi}>> B
\end{CD}
$$
It follows that $\pi$ is an affine and smooth morphism and that $\pi\colon X \to B$ is a \itind{geometric quotient}. This means that for every affine open subset $U \subseteq B$ the induced morphism $\pi^{-1}(U) \to U$ is an affine quotient, i.e. $\OOO(U)=\OOO(\pi^{-1}(U))^{G}$, and that the fibers of $\pi$ are the $G$-orbits. 

Since $\pi$ is affine it follows that $X$ is affine in case $Y$ is affine, but the reverse implication is not correct. In fact, \name{Winkelmann} gave an example of ${\C}^{+}$-action on $\C^{5}$ which defines a principal bundle
$\alpha\colon \C^{5} \to Y$ where $Y$ is a smooth quasi-affine variety which is not affine.
This example can be found in \cite[Section~2]{Wi1990On-free-holomorphi}, and the fact that the quotient $\alpha$ is a principal
${\C}^{+}$-bundle follows from \cite[Proposition (v) in Section~8]{Wi1990On-free-holomorphi}.
\ps
A principal $G$-bundle $\pi\colon X \to B$ is {\it trivial}, i.e. $G$-isomorphic
to $\pr_{B}\colon G \times B \to B$ if and only if $\pi$ has a section $\sigma\colon B \to X$.\idx{trivial principal $G$-bundle}
In fact, the trivialization is given by $G \times B \simto X$, $(g,b)\mapsto g\sigma(b)$. It follows that for a principal $G$-bundle $\pi\colon X \to B$ the pull-back bundle $\pr_{2}\colon X \times_{B} X \to X$ is trivial, because it has the section $\sigma\colon x \mapsto (x,x)$. Hence we have the following commutative diagram:
\[\tag{$*$}
\begin{CD}
G\times X @>{\simeq}>{(g,x)\mapsto (gx,x)}> X\times_{B}X @>{\pr_{1}}>> X \\
@VV{\pr_{X}}V @VV{\pr_{2}}V @VV{\pi}V \\
X @= X @>{\pi}>> B
\end{CD}
\]

\begin{example}\label{GmodH.exa}
Important examples of principal bundles appear in the following way. Let $G$ be a linear algebraic group and $H \subset G$ a closed subgroup. Then $H$ acts from both sides on $G$, $(h,g) \mapsto hg$ and $(h,g)\mapsto gh^{-1}$. Therefore, we obtain two principal $H$-bundles, $G \to H\backslash G$, $g\mapsto Hg$, and $G \to G/H$, $g \mapsto gH$. These bundles are $H$-isomorphic where the isomorphism is induced by $G \simto G$, $g \mapsto g^{-1}$.

It is easy to see that the $H$-bundle $G \to H\backslash G$ is trivial, if and only if there exists an $H$-equivariant map $\tau\colon G \to H$, 
i.e. we have $\tau (hg) = h \tau (g)$ for $g \in G$, $h \in H$.
In fact, the triviality means that we have an $H$-equivariant isomorphism $G \simto H \times (H\backslash G)$
which must be of the form $g \mapsto (\tau(g),Hg)$ where $\tau$ is $H$-equivariant.
\end{example}

Let us mention the following important results about principal $G$-bundles. A linear algebraic group $G$ is called {\it special} if every principal $G$-bundle is locally trivial in the Zariski topology.\idx{special group}

\begin{proposition}
\be
\item
For a unipotent group $U$ every principal $U$-bundle over an affine variety is trivial. In particular, $U$ is special.
\item
The tori $T$, the groups $\GL_{n}$, $\SL_{n}$, $\Sp_{n}$, and products of them are special.
\item
Connected solvable groups are special.
\ee
\end{proposition}\label{principal-G-bundle.prop}
For the proofs we refer to \cite{Se1958Espaces-fibres-alg}, cf. \cite[III.2]{KrSc1992Reductive-group-ac}.

\ps
\subsection{Local sections for group actions}\label{local-sections.subsec}
Let $G$ be a linear algebraic group acting on an affine variety $X$.
\begin{definition} A {\it $G$-section\/}\idx{G-s@$G$-section} of the $G$-action on $X$ is a $G$-equivariant morphism 
$\sigma\colon X \to G$ where $G$ acts by left-multiplication on itself,
i.e. a morphism $\sigma$ satisfying $\sigma (gx) = g\sigma (x)$ for all $g \in G$ and $x \in X$.

A {\it local $G$-section\/}\idx{local $G$-section} is a $G$-section on a nonempty $G$-stable  open subset of $X$.
\end{definition}

Let $X := G \times Y$ where $Y$ is an arbitrary affine variety and where $G$ acts by left-multiplication on $G$. Then the projection 
$\sigma:=\pr_{G}\colon X \to G$ is a $G$-section and the projection $\pi:=\pr_{Y}\colon X \to Y$ is the affine quotient, i.e. $\pi^{*}$ induces an isomorphism $\OOO(Y)\simto \OOO(X)^{G}$. This is the general picture as the next proposition shows.

\begin{proposition}\label{G-section.prop}
The following assertions for a $G$-action on $X$ are equivalent.
\be
\item[(i)] 
The $G$-action has a $G$-section.
\item[(ii)]
There is a closed subset $Y\subseteq X$ such that the action map
$G \times X \to X$, $(g,x) \mapsto gx$, induces an isomorphism
$\rho \colon G \times Y  \simto  X$.
\ee
It then follows that $\pi:=\pr_{Y}\circ\rho^{-1}\colon X \to Y$ is the affine quotient of the $G$-action.
\end{proposition}
Another equivalent way to express this is the following: {\it The affine quotient $X \quot G$ exists, i.e. the invariant ring $\OOO(X)^{G}$ is finitely generated, and the quotient morphism $\pi_X \colon X \to X\quot G$
is a trivial principal $G$-bundle\/}
(see Section~\ref{G-bundle.subsec}).
\begin{proof}[Proof of Proposition~\ref{G-section.prop}]
(i)\,$\Rightarrow$\,(ii): Let $\sigma \colon X \to G$ be a section of the action. Set $Y:=\sigma^{-1}(e)$. Then the morphism $G \times Y \to X$, $(g,y)\mapsto gy$ has an inverse, namely  $x \mapsto (\sigma(x),\sigma(x)^{-1}x)$.
\ps
(ii)\,$\Rightarrow$\,(i): If $X = G\times Y$, then the projection $\pr_{G}\colon G \times Y \to G$ is a $G$-section.
\ps
The last statement is clear.
\end{proof}
\noindent
The situation is displayed in the following diagram
\ps
\begin{center}
\begin{tikzcd}
X  \arrow[d, xshift=-0.2em, "\pi"']  \arrow[r, "\sigma"] & G   \\
X\quot G  \arrow[ru,"e"'] \arrow[u, xshift=0.2em, "s"']  
\end{tikzcd}
\end{center}
where $s$ is the inverse map of the isomorphism $\sigma^{-1}(e) \simto X\quot G$ induced by $\pi$.
\ps
In case of the additive group $\kplus$, there is the following criterion for the existence of a section.
\begin{lemma}\label{sections-via-vector-fields.lem}
Consider a $\kplus$-action on $X$, and let $\delta\in\VEC(X)$ be the corresponding vector field. Then a regular function $f\colon X \to \kplus$ is a $\kplus$-section if and only if $\delta(f)=1$.
\end{lemma}
\begin{proof}
We are going to use  Lemma~\ref{exp-for-functions.lem} which will be proved in the following section. This lemma asserts that for $s\in\kplus$ and $f\in\OOO(X)$ we have
$$
f(sx)= \sum_{n=0}^{\infty} s^{n }\cdot \delta^{n}(f)(x).
$$
Thus $f$ is $\kplus$-equivariant, i.e. $f(sx) = f(x)+s$ for all $s \in\kplus$, if and only if $\delta(f)=1$.
\end{proof}

\begin{example} Let $\u \in \Aut(\An)$ be a nontrivial translation, $\u\colon x \mapsto x+u$, and  let $\rho\colon\kplus \to \Aut(\An)$ be the $\kplus$-action generated by $\u$, i.e. $\rho(s)(x) = x + su$.
It is clear that every linear map $\sigma\colon \An \to \kplus$ such that $\sigma(u)=1$ is a section.
\end{example}

\begin{proposition}\label{local-sections-for-kplus.prop}
For a nontrivial $\kplus$-action on an affine variety $X$ there always exists a \itind{local section}. More precisely, there is an invariant $f \in\OOO(X)^{\kplus}$ such that the $\kplus$-action on $X_{f}$ has a section. 
\end{proposition}

\begin{proof}
Let $\delta\in\VEC(X)$ be the vector field corresponding to the $\kplus$-action. Since the action is nontrivial, we can find an $h \in \OOO(X)$ such that $f:=\delta h \neq 0$, but $\delta f= \delta^{2}h=0$. Then $f \in \OOO(X)^{\kplus}$, and 
$$
\delta\left(\frac{h}{f}\right) = \frac{\delta h}{f} = 1.
$$ 
Thus $\frac{h}{f}\colon X_{f}\to \kplus$ is a section, by Lemma~\ref{sections-via-vector-fields.lem} above. 
\end{proof}
Local sections also exist for faithful actions of tori. The proof relies on the following classical result.

\begin{proposition} \label{trivial-generic-stabilizer.prop}
Let $G$ be an algebraic group which is either diagonalizable or finite or $\kplus$. Then the generic stabilizer of a faithful action on an irreducible affine variety $X$ is trivial, i.e. there exists an open dense set $U \subseteq X$ such that the stabilizer $G_x$ is trivial for every $x \in U$.
\end{proposition}

\begin{proof}
(1) 
Assume first that $G$ is diagonalizable and that $X$ is a $G$-module $V$.
Decompose $V$ into weight spaces: $V = V_{\lambda_{1}}\oplus V_{\lambda_{2}}\oplus\cdots\oplus V_{\lambda_{r}}$, where $\lambda_{1},\ldots,\lambda_{r}$ are characters of $G$ and $V_{\lambda_{i}}:=\{v \in V \mid g v = \lambda_{i}(g)\cdot v \text{ for }g \in G \}$. Define
$$
\PPP:=\{\Lambda \subseteq \{\lambda_{1},\ldots,\lambda_{r}\} \mid \textstyle{\bigcap_{\lambda\in\Lambda}}\Ker\lambda\neq \{e\}\},
$$
and set $V_{\Lambda}:=\bigoplus_{\lambda\in\Lambda}V_{\lambda}$ for $\Lambda\in\PPP$.
Then
$\{v\in V\mid G_{v}\neq\{e\}\} = \bigcup_{\Lambda\in\PPP}V_{\Lambda}$, hence
$$
U:= V \setminus \textstyle{\bigcup_{\Lambda\in\PPP}} V_{\Lambda} = \{v\in V \mid G_{v}\text{ is trivial}\}.
$$
It follows that $U\neq \emptyset$, because otherwise $V = V_{\Lambda}$ for some $\Lambda \in \PPP$, and so the action is not faithful.

For the general case, we embed $X$ into a $G$-module $V$ and check that $U \cap X \neq \emptyset$. Otherwise, we would have $X \subseteq V_{\Lambda}$ for some $\Lambda\in\PPP$, because $X$ is irreducible, contradicting the faithfulness of the action.
\ps
(2) Now assume that $G$ is finite.
For a $G$-module $V$ we get 
$$
U:= V \setminus  \textstyle{\bigcup_{g \neq e}}\Ker (\id -g) = \{v\in V \mid G_{v}\text{ is trivial}\}.
$$
The same argument as in (1) shows that $U\neq \emptyset$, and that for an embedding of $X$ into a $G$-module $V$ we have $X \cap U \neq \emptyset$.
\ps
(3) The case $G = \kplus$ is clear since $\kplus$ does not contain a nontrivial closed subgroup.
\end{proof}

\begin{remark}
The proposition does not hold if $X$ is reducible. Take $G := S_{n}$, the symmetric group in $n$ letters. Then the standard action of $S_{n}$ on $\{1,2,\ldots,n\}$ is faithful, but the stabilizer of each point is isomorphic to $S_{n-1}$.

Another example is the following. Let $\ktwo$ be the standard representation of $T:=(\kst)^{2}$,
and consider $X:=\VVV(xy) \subset \ktwo$, the union of the two coordinate lines. Again, the action of $T$ on $X$ is faithful, but the stabilizer of any point of $X$ is nontrivial.

The proposition does not hold for simple algebraic groups $G$ and neither for $\GL_{n}$ ($n \geq 2$),
because all these groups admit faithful representations $V$ of dimension $\dim V< \dim G$.
\end{remark}
\begin{corollary}\label{local-sections-for-tori.prop}
Let $T$ be a torus. For every faithful action of $T$ on an irreducible affine variety $X$ there exists a local section.
\end{corollary}

\begin{proof}
Let  $U \subseteq X$ be an open dense subset given in the proposition above. We can find a semi-invariant $f \in \OOO(X)$ such that $X_{f} \subset U$ and $X_{f}$ is smooth. It then follows from \name{Luna}'s Slice Theorem \cite[Corollaire~1]{Lu1973Slices-etales} that the quotient morphism $X_{f} \to X_{f}\quot T$ is a principal $T$-bundle, hence locally trivial in the Zariski topology (see Proposition~\ref{principal-G-bundle.prop}). The claim follows.
\end{proof}

\ps
\subsection{Large subgroups of \texorpdfstring{$\Aut(X)$}{Aut(X)}} \label{large-subgroups.subsec}
Let $G$ be a linear algebraic group acting on an affine variety $X$. Denote by $\rho\colon G \to \Aut(X)$ the corresponding homomorphism of ind-groups. If $\alpha\colon X \to G$ is a $G$-invariant morphism, i.e. $\alpha(gx) = \alpha(x)$ for all $x\in X$ and $g \in G$, then we can define an automorphism $\rho_{\alpha}\in\Aut(X)$ in the following way:
\[
\rho_{\alpha}(x):= \rho(\alpha(x))(x) = \alpha(x) x.
\]
By definition, $\rho_{\alpha}\colon X \to X$ is the composition of the two morphisms 
$$
\begin{CD}
X @>{x\mapsto(\alpha(x),x)}>> G\times X @>{(g,x)\mapsto gx}>> X
\end{CD}
$$
hence $\rho_{\alpha}\in\End(X)$, and $\rho_{\alpha}$ is an automorphism, because $\rho_{\alpha}\circ\rho_{\alpha^{-1}}=\id$ where $\alpha^{-1}(x):=\alpha(x)^{-1}$. Note that the $G$-invariant morphisms $\alpha\colon X \to G$ are the elements of the ind-group  $G(\OOO(X)^{G})$, see Remark~\ref{G(R)-as-ind-group.rem}. 

We denote by $\tilde\rho$ the corresponding map $G(\OOO(X)^{G}) \to \Aut(X)$, $\alpha\mapsto \rho_{\alpha}$. If $\alpha$ is the constant map $\alpha(x)=g$, then $\rho_{\alpha}(x) = gx$. This means that $\tilde\rho|_{G} = \rho$ where we identify $G$ with $G(\kk) \subset G(\OOO(X)^{G})$.

In case of a representation of  $G$ on a $\kk$-vector space $V$ this construction is given in \cite[D\'efinition~2.3]{Fu2008Sur-les-automorphi} where one also finds the examples below.

\begin{example}\label{rho-for-rep.exa}
Let us make the construction of $\rho_{\alpha}$ more explicit for an action of $G$ on the affine space $\An$. For $v = (a_{1},\ldots,a_{n})\in \An$ we have
$$
g v = g(a_{1},\ldots,a_{n}) = (f_{1}(g,a_{1},\ldots,a_{n}),\ldots,f_{n}(g,a_{1},\ldots,a_{n}))
$$
where $f_{i}\in \OOO(G)[x_{1},\ldots,x_{n}]$. If $\alpha\colon \kk^{n} \to G$ is $G$-invariant, then we get
$$
\rho_{\alpha}(v) = (f_{1}(\alpha(a_{1},\ldots,a_{n}),a_{1},\ldots,a_{n}),\ldots,f_{n}(\alpha(a_{1},\ldots,a_{n}),a_{1},\ldots,a_{n})).
$$
In particular, if the action is a matrix representation $G \to \GL_{n}$, $g\mapsto (a_{ij}(g))_{i,j}$, then 
$$
\rho_{\alpha}(a_{1},\ldots,a_{n}) =
\begin{bmatrix} a_{11}(\alpha(a_{1},\ldots,a_{n})) & \cdots &   a_{1n}(\alpha(a_{1},\ldots,a_{n}))  \\
\vdots && \vdots\\
a_{n1}(\alpha(a_{1},\ldots,a_{n})) & \cdots &   a_{nn}(\alpha(a_{1},\ldots,a_{n}))  
\end{bmatrix}
\begin{bmatrix}
a_{1}\\ \vdots \\ a_{n}
\end{bmatrix}.
$$
\end{example}

\begin{example}\label{Nagata.exa}
Consider the adjoint representation of $\SLtwo$ on $\sltwo:=\Lie\SLtwo$ which consists of the traceless matrices
$\left[ \begin{smallmatrix}
x & \phantom{-}y \\
z & -x
\end{smallmatrix} \right]$. Hence, $\OOO(\sltwo) = \kk[x,y,z]$, and we have
$I:=\OOO(\sltwo)^{\SLtwo} = \kk[q]$ where $q= - \det \left[ \begin{smallmatrix}
x & \phantom{-}y \\
z & -x
\end{smallmatrix} \right] = x^{2}+yz$. For any $f \in I$ we have 
$u_{f}:= \left[ \begin{smallmatrix}
1 & f \\
0 & 1
\end{smallmatrix} \right] \in \SL_{2}(I)$ and 
\[
\begin{bmatrix} 1 & f \\ 0 & 1\end{bmatrix}
\begin{bmatrix} x & \phantom{-}y \\ z & -x\end{bmatrix}
\begin{bmatrix} 1 & f \\ 0 & 1\end{bmatrix}^{-1}
=
\begin{bmatrix} x+fz & y -2fx -f^{2}z \\ z & -(x+fz)\end{bmatrix}.
\]
Hence, the induced automorphism of $\sltwo$ is given by 
$$
\rho_{u_{f}}=(x+fz, y-2fx-f^{2}z, z).
$$
If we take $f = q$ we get the famous \name{Nagata}-automorphism, up to interchanging the variables $x$ and $y$, see Example~\ref{Bass.exa}.
\end{example}

\begin{example}\label{Anick.exa}
Here we consider the action of $\SLtwo$ by left multiplication on the $2\times 2$-matrices $\M_{2}=\M_{2}(\kk)$. If we write such a matrix in the form $\left[\begin{smallmatrix} x&y\\z&w\end{smallmatrix}\right]$ we get $\OOO(\M_{2}) = \kk[x,y,z,w]$, and $I:=\OOO(\M_{2})^{\SLtwo} = \kk[q]$ where $q = \det = xw-yz$.
For $f \in I$ we have $u_{f}:= \left[ \begin{smallmatrix}
1 & f \\
0 & 1
\end{smallmatrix} \right] \in \SL_{2}(I)$,
and
\[
\begin{bmatrix} 1 & f \\ 0 & 1\end{bmatrix}
\begin{bmatrix} x & y \\ z & w\end{bmatrix}
=
\begin{bmatrix} x+fz & y +fw \\ z & w\end{bmatrix}.
\]
Hence, the induced automorphism of $\M_{2}$ is given by 
$$
\rho_{u_{f}}=(x+fz, y+fw, z,w).
$$
If we take $f = q$ we get the automorphism of  \name{Anick}, see \cite[p.~343]{An1983Limits-of-tame-aut}.
\end{example}

\begin{proposition} \label{Large-subgroups.prop}
Let $G$ be a linear algebraic group acting on an affine variety $X$,
where the action is given by the homomorphism $\rho\colon G \to \Aut(X)$.
\be
\item
The map $\tilde\rho\colon G(\OOO(X)^{G}) \to \Aut(X)$ is a homomorphism of ind-groups.  
\item
The ind-group $G(\OOO(X)^{G})$ has the same orbits in $X$ than $G$.
\item  \label{Image-is-contained-in-Aut_B(X)}
The image of $\tilde\rho$ is contained in the  subgroup 
\begin{multline*}
\Aut_{\rho}(X):= \\ \{\phi\in\Aut(X)\mid \phi(Gx) = Gx \text{ for all } x \in X, \,\phi|_{Gx}=\rho(g_{x}) \text{ for some }g_{x}\in G\}.
\end{multline*}\idx{$\Aut_{\rho}(X)$}
\item
The subgroup $\Aut_{\rho}(X) \subseteq \Aut(X)$ is closed, and we have the following inclusions
\[
\Aut_{\rho}(X) \subseteq \Autorb(X) \subseteq \Autinv(X) \subseteq  \Aut(X).
\]
(see Definition~\ref{Some-subgroups-of-Aut(X).def})
\item
If $G$ acts faithfully on $X$ and if $\OOO(X)^{G}\neq\kk$,  then the image of $\tilde\rho$ is strictly larger than $\rho(G)$. 
\ee
\end{proposition}

\begin{proof}
(1) The map $G(\OOO(X)^{G})\times X \to X$ is induced by the morphism
$$
\begin{CD}
\Mor(X,G) \times X @>{(\alpha,x)\mapsto (\alpha(x),x)}>> 
G\times X @>{(g,x)\mapsto gx}>> X.
\end{CD}
$$
Since $G(\OOO(X)^{G} ) \subseteq \Mor(X,G)$ is closed, we see that $\tilde \rho$  is an ind-morphism.
If $\alpha,\beta\in G(\OOO(X)^{G})$, then
\begin{eqnarray*}
(\rho_{\alpha}\circ\rho_{\beta})(x) &=& \rho_{\alpha}(\rho_{\beta}(x)) = 
\rho_{\alpha}(\beta(x)x) = \alpha(\beta(x)x)(\beta(x)x)\\
 &=& \alpha(x)(\beta(x)x) = (\alpha(x)\beta(x))x = \rho_{\alpha\beta}(x),
\end{eqnarray*}
and so $\tilde\rho$ is a group homomorphism.
\ps
(2) Since $\rho_{\alpha}(x)=\alpha(x)x \in Gx$ for all $\alpha$ we see that the $G(\OOO(X)^{G})$-orbits are contained in the $G$-orbits, 
and they are equal,  because $G\subseteq G(\OOO(X)^{G})$.
\ps
(3) By (2), every $\rho_{\alpha}$ preserves the $G$-orbits. Since $\rho_{\alpha}(gx) = \rho(\alpha(x))(gx)$, by definition, we get $\rho_{\alpha}|_{Gx} = \rho(\alpha(x))|_{Gx}$, and the claim follows.
\ps
(4) We already know from Proposition~\ref{Autorb-is-closed.prop} that $\Autorb(X)$ is closed in $\Aut (X)$. Hence, let us prove that $\Aut_{\rho} (X)$ is closed in $\Autorb(X)$.

For $x \in X$, the $G$-action on $\overline{Gx}$ corresponds to a homomorphism of ind-groups $\rho_x \colon G \to \Aut ( \overline{Gx} )$, and the image $\rho_x(G) \subseteq \Aut ( \overline{Gx} )$ is closed by Proposition~\ref{Hom-G-to-ind-group.prop}. The map $\gamma_x \colon \Autorb (X) \to \Aut ( \overline{Gx} )$, $\phi \mapsto \phi |_{\overline{Gx}}$,  is the restriction of the homomorphism of ind-groups $\Aut (X, \overline{Gx} ) \to \Aut ( \overline{Gx} )$ (see Proposition~\ref{Aut(X,Y).prop}), hence $\gamma_x$ is a homomorphism of ind-groups, too, and so $\Aut_{\rho}(X) = \bigcap_{x \in X} (\gamma_x)^{-1} ( \rho_x (G) ) $ is closed in $\Autorb(X)$.
\ps
(5) Assume that $\rho_{\alpha}$ is the action $x \mapsto gx$ for some $g\in G$. Then $\alpha(x) \in g G_{x}$ for all $x\in X$. Since $\bigcap_{x}G_{x} = \{e\}$ we get $\alpha(x) = g$ for all $x \in X$. Hence $\alpha$ is the constant map with value $g \in G$. If this holds for all $G$-invariant morphisms $X \to G$, then $\OOO(X)^{G}=\kk$.
\end{proof}

\begin{remark}\label{tilde-rho-for-commutative-groups.rem}
If $G$ is commutative, then the condition $\phi|_{Gx} = \rho(g_{x})$ for some $g_{x}\in G$ is equivalent to the condition that $\phi|_{Gx}$ is $G$-equivariant, and thus we have $\Aut_{\rho}(X) = \Autorb(X) \cap \Aut_{G}(X)$.
\newline
In fact, the first statement is an immediate consequence of a well-known result from group theory saying that the $G$-equivariant maps $G/H \to G/H$ are the right-multiplications with elements from the normalizer $N_{G}(H)$. In particular, if $G$ is commutative, the $G$-equivariant maps $G/H \to G/H$ are exactly the left-multiplications with elements from $G$.
\end{remark}

\begin{question}\label{image-of-rho.ques}
Is $\Aut_{\rho}(X)$ the image of $\tilde\rho$? And is the image of $\tilde\rho$ closed in $\Aut(X)$?
\end{question} 

The next proposition addresses the injectivity of the map $\tilde\rho$.
\begin{proposition}\label{tilde-rho-injective.prop}
Assume that $G$ acts faithfully on an affine variety $X$. Define the closed subgroup $N_{G}(x) := \bigcap_{g\in G}G_{gx} \subseteq G$ for any $x \in X$.
\be
\item
If $N_{G}(x)=\{e\}$ for all $x$ from a dense subset of $X$, then $\tilde\rho$ is injective. 
\item
If $X$ is irreducible this holds in the following two cases.
\be
\item[(i)]
The generic stabilizer of the $G$-action is trivial. 
\item[(ii)]
The group $G$ is reductive.
\ee
\ee
\end{proposition}

\begin{proof}
(1)
If $\rho_{\alpha} = \id$, then $\alpha(x)\in G_{x}$ for all $x \in X$.
Since $\alpha$ is constant on every orbit we get $\alpha(x) \in \bigcap_{g\in G}G_{gx}$. Hence, by assumption, $\alpha(x) = e$ for all $x$ from a dense  set, and so $\alpha = e$.
\par
(2)
Under the assumptions in (i) or (ii) we show that $N_{G}(x) = \{e\}$ holds on a nonempty open set of $X$. In case (i) this is clear, and in case (ii) this follows from the next lemma.
\end{proof}

\begin{lemma}\label{NG-for-reductive.lem}
Let $G$ be a reductive group acting faithfully on an irreducible affine variety $X$. For $x \in X$ define $N_{G}(x) := \bigcap_{g\in G}G_{gx}$.
Then there is an open dense subset $X' \subseteq X$ such that $N_{G}(x) = \{e\}$ for all $x \in X'$.
\end{lemma}
\begin{proof}
(1) 
We first remark that $N_{G}(x) \subset G$ is a closed normal subgroup, because it is the intersection of all conjugates to $G_{x}$. It is also clear that we can assume that $G$ is connected, because a finite group acting faithfully has trivial generic stabilizer.
\ps
(2)
If $H \subseteq G$ is a closed subgroup, then $N_{H}(x)\supseteq N_{G}(x)\cap H$.
\ps
(3) 
Let $D \subseteq G$ be a diagonalizable group. 
By Proposition~\ref{trivial-generic-stabilizer.prop}, $N_{D}(x)$ is trivial on a dense open set.
\ps
(4)
Let $H \subseteq G$ be a closed simple subgroup. If $x$ is not a fixed point of $H$, then $N_{H}(x)$ is a finite subgroup
of the center  $Z(H)$ of $H$. The (finite) center $Z(H)$ acts faithfully on $X$, hence there is an open dense subset 
$X' \subset X$ where $Z(H)$ has a trivial stabilizer. It follows that $N_{H}(x) = \{e\}$ for all $x \in X'$.
\ps
(5)
The connected reductive group $G$ is a product of the form $G = Z \cdot H_{1} \cdots H_{m}$ where $Z$ is the center of $G$ and the $H_{i}$ are simple closed normal subgroups. 
By (3) and (5) we can find a dense open set $X' \subseteq X$ such that $N_{Z}(x) = N_{H_{1}}(x) = \cdots =N_{H_{m}}(x) = \{e\}$ for all $x \in X'$. Let $g \in N_{G}(x)$. Then, for every $i$ and every $h \in H_{i}$ we have $hgh^{-1}g^{-1} \in N_{G}(x) \cap H_{i} \subseteq N_{H_{i}}(x) = \{e\}$. It follows that $g \in Z \cap N_{G}(x) \subseteq N_{Z}(x) = \{e\}$.
\end{proof}

\begin{example}
Consider the faithful action of $U:= (\kk^2,+)$ on $X:= \AA^2$ given by $(a,b)\colon (x,y) \mapsto (x+ay+b, y)$. Then, the induced action of $U ( \OOO (X)^U )$ on $X$ is not faithful. In particular, $\tilde\rho$ is not injective.
\newline
In fact, $\OOO(X)^{G}=\kk[y]$, hence $G(\OOO(X)^{G}) = (\kk[y]^{+})^{2}$, and this group  acts as $(a(y),b(y))\colon (x,y) \mapsto (x + a(y)y + b(y), y)$. We therefore see that  the elements $(a(y),-ya(y)) \in (\kk[y]^{+})^{2}$ act trivially on $X$.
\end{example}

\begin{question}\label{rho-closed-immersion.ques}
If $\tilde\rho$ is injective, is it a closed immersion? 
\end{question}

In some special cases, e.g. if $G$ is a torus, we can show that $\Aut_{\rho}(X)$ is exactly the image of the homomorphism $\tilde\rho$.

\begin{proposition}\label{commutative-modification.prop}
Let $A$ be a connected commutative linear algebraic group, and let $X$ be an irreducible affine $A$-variety. Assume that the $A$-action on $X$ has a local section. Then   $\tilde\rho$ induces a bijection
$$
A(\OOO(X)^{A}) \simto \Aut_{\rho}(X) = \Autorb(X) \cap \Aut_{G}(X).
$$
\end{proposition}

\begin{proof}
We have already seen in Remark~\ref{tilde-rho-for-commutative-groups.rem} that $\Aut_{\rho}(X) = \Autorb(X) \cap \Aut_{G}(X)$. Moreover, $\tilde\rho$ is injective, because the generic stabilizer of the $A$-action is trivial (Proposition~\ref{Large-subgroups.prop}(4)).

By assumption, there is an $A$-stable open dense subset $U \subseteq X$ and an $A$-equivariant morphism $\sigma\colon U \to A$ which induces the commutative diagram
\ps
\begin{center}
\begin{tikzcd}
U  \arrow[d, xshift=-0.2em, "\pi"']  \arrow[r, "\sigma"] & A   \\
U\quot A  \arrow[ru,"e"'] \arrow[u, xshift=0.2em, "s"']  
\end{tikzcd}
\end{center}
where $s$ is the inverse map of the isomorphism $\sigma^{-1}(e) \simto U\quot A$ induced by $\pi$. Since $A$ is commutative we can assume that $U = X_{f}$ for an $A$-semi-invariant $f \in \OOO(X)$. In fact, if $\aa \subset \OOO(X)$ is the ideal of functions vanishing on $X \setminus U$, then $\aa$ is stable under $A$. Since $A$ is commutative and
connected, it is a product of a torus and a commutative unipotent group, and it is well-known that every representation of such a group contains a nontrivial semi-invariant.

Now let $\phi \in \Aut(X)$ be such that $\phi(Ax)=Ax$ for all $x \in X$. It follows that $\phi(U)=U$. For $u\in U$ define 
$\alpha(u):=\sigma(\phi(s(\pi(u))))$. We claim that $\phi = \rho_{\alpha}$. In fact, for $u \in \sigma^{-1}(e)$ we have $\phi(u) \in Au$, hence 
$\phi(u) = \sigma(\phi(u))$, and $\alpha(u) = \sigma(\phi(u))$, 
because $s(\pi(u))=u$. 
Thus $\rho_{\alpha}(u) = \phi(u)$ for $u \in \sigma^{-1}(e)$. Since both maps are $A$-equivariant, we get $\rho_{\alpha}(u) = \phi(u)$ for all $u \in U$. 
Since the automorphism $\rho_{\alpha}$ extends to $X$, it follows that $\alpha\colon U\to A$ extends to $X$, hence the claim.
\end{proof}

In case $X$ is a principal $G$-bundle (see \ref{G-bundle.subsec}) we have a more precise result.

\begin{proposition}  \label{closed-immersion.prop}
Let $G$ be a linear algebraic group, and let $\pi \colon X \to B$ be a principal $G$-bundle where $B$ (and hence $X$) is affine.
Then $\tilde\rho\colon G(B) \to \Aut(X)$ induces an isomorphism $G(B) \simto \Aut_{\rho}(X)$ of ind-groups.
\end{proposition}

\begin{proof}
(1) Since $\pi\colon X \to B$ is a principal $G$-bundle we have the commutative diagram (see diagram $(*)$ in Section~\ref{G-bundle.subsec})
\[\tag{$**$}
\begin{CD}
G\times X @>{\simeq}>{(g,x)\mapsto (gx,x)}> X\times_{B}X @>{\pr_{1}}>> X \\
@VV{\pr_{X}}V @VV{\pr_{2}}V @VV{\pi}V \\
X @= X @>{\pi}>> B
\end{CD}
\]
It follows that every automorphism $\phi\in\Aut_{\rho}(X)$ induces an automorphism $\tilde\phi$ of $G \times X$ over $X$. Since $\phi|_{\pi^{-1}(b)}$ is the left-multiplication with some $g \in G$ the same holds for $\tilde\phi|_{G\times\{x\}}$, and thus $\tilde\phi$ is given by a morphism $\tilde\alpha\colon X \to G$. In fact, $\tilde\alpha$ is the composition 
$$
\begin{CD}
\tilde\alpha \colon X @>{\simeq}>> \{e\}\times X @>{\tilde\phi|_{\{e\}\times X}}>> G \times X @>{\pr_{G}}>> G
\end{CD}
$$
Since $\tilde\phi$ is induced by an automorphism of $X$ over $B$ it follows that $\tilde\alpha$ is constant along the fibers of $\pi$, hence induces a morphism $\alpha\colon B \to G$.  It is now easy to see that $\phi=\rho_{\alpha}$, hence the image of $\tilde\rho$ is equal to $\Aut_{\rho}(X)$.
\ps

(2) 
If $X = G\times B$ is the trivial bundle, then $\Aut_{\rho}(G\times B) \subset \Aut_{B}(G \times B)$ is closed and $\tilde\rho\colon G(B) \to \Aut_{\rho}(G\times B)$ is an isomorphism. In fact, there is an isomorphism $\Aut_{B}(G\times B) \simto \Mor(B,\Aut(G))$ (Proposition~\ref{families-of-aut.prop}), and
$\Aut_{\rho}(X)$ is the inverse image of the closed subgroup $\Mor(B,G) = G(B)$.
\ps
For the general case we look again at the diagram $(**)$ above which induces the following commutative diagram 
$$
\begin{CD}
\Aut_{\pi}(X) @>>> \Aut_{X}(G \times X) \\
@AAA   @AA{\text{closed immersion}}A\\
\Aut_{\rho}(X) @>>> \Aut_{\rho}(G \times X) \\
@AA{\text{bijective}}A   @AA{\simeq}A \\
G(B) @>{\text{closed immersion}}>> G(X)
\end{CD}
$$
where all maps are morphisms. Since $G(B) \into G(X)$ is a closed immersion, by Proposition~\ref{Homs-of-G(R).prop}(2), the composition $G(B) \to \Aut_{X}(G \times X)$ is a closed immersion. By Lemma~\ref{closed-immersion.lem} this implies that  $G(B) \to \Aut_{\rho}(X)$ is a closed immersion, hence an isomorphism.
\end{proof}

\ps
\subsection{Some applications}\label{applications.subsec}
\begin{proposition}
Let $X$ be an irreducible affine variety, and
let $U \subseteq \Aut(X)$ be a commutative unipotent subgroup of dimension $n$. Assume that the centralizer of $U$ in $\Aut(X)$ is $U$ itself. Then the orbit maps $\rho_{x}\colon U \to X$, $u\mapsto ux$, are isomorphisms for any $x \in X$. In particular, $X$ is isomorphic to $\An$.
\end{proposition}
\begin{proof}
Consider the homomorphism $\tilde\rho\colon U(\OOO(X)^{U}) \to \Aut(X)$. The image $\Image\tilde \rho$ is an abelian subgroup containing $U$. Since $U$ is its own centralizer, we get $\Image\tilde\rho =U$, hence $\OOO(X)^{U}=\kk$, by Proposition~\ref{Large-subgroups.prop}(5).
By Theorem~\ref{unipotent-groups.thm}(\ref{relation-between-regular-and-rational-invariants}), we also get $\kk (X)^{U}=\kk$. Now, it follows from Theorem~\ref{unipotent-groups.thm}(\ref{homogeneous}) that $X$ is a homogeneous space under $U$, and so $X$ is an affine space by Theorem~\ref{unipotent-groups.thm}(\ref{affine-space}). Since $U$ is commutative, the stabilizer $U_{x}$ acts trivially on the orbit $Ux=X$, hence $U_{x}=\{e\}$, and so $U \simeq X$. This proves that $X$ has dimension $n$, and since we already know that $X$ is an affine space, the conclusion follows.
\end{proof}

Here is another nice consequence of our previous results.

\begin{proposition}
Any ind-group of the form $G(R)$ where $G$ is a linear algebraic group and $R$ a finitely generated commutative $\kk$-algebra is isomorphic to a closed subgroup of  $\Aut (X)$ 
for some affine variety $X$.
\end{proposition}

\begin{proof}
By Proposition~\ref{reduction-to-domain.prop}, we may assume that $R$ is a $\kk$-domain. Let $B := \Spec (R)$ be the affine variety with coordinate ring $R$. Then, by Proposition~\ref{closed-immersion.prop},  we have a closed immersion  $G(R) = \Mor (B, G) \into \Aut ( G \times B)$ of ind-groups.
\end{proof}

\begin{remark}
The proposition above shows that any group of the form $G (R)$ where $G$ is a linear algebraic group and $R$ a finitely generated algebra admits a closed immersion into an automorphism group $\Aut (X)$ where $X$ is an irreducible affine variety. One might wonder if the converse holds. This it not the case. Indeed, by a result of \name{Yves de Cornulier} the group $\Aut ( \AA^2)$ is not linear (see Proposition~\ref{the-group-Aut(A2)-is-not-linear.prop}) whereas Corollary~\ref{the-groups-G(R)-are-linear.cor} shows that the group $G(R)$ is always linear.
\end{remark}

Let us finally mention that the automorphism groups of affine varieties do not contain arbitrary large tori.
\begin{proposition}\label{tori-in-AutX.prop}
Let $X$ be an affine variety, and let $T \subset \Aut(X)$ be a closed torus. Then $\dim T \leq \dim X$. Moreover, if $X$ is normal, then all closed tori $T \subseteq \Aut(X)$ of dimension $\dim T = \dim X$ are conjugate.
\end{proposition}

\begin{proof} If a torus $T$ acts faithfully on an affine variety $X$, then there is a point $x \in X$ with trivial stabilizer (Proposition~\ref{trivial-generic-stabilizer.prop}). Hence $\dim X \geq \dim Tx = \dim T$.
\ps
If $X$ is normal and if there exists a torus $T \subset \Aut(X)$ of dimension $\dim T = \dim X$, then $X$ is an affine toric variety. A result due to \name{Demushkin} \cite{De1982Combinatorial-inva} says that in this case all maximal tori in $\Aut(X)$ are conjugate (cf. \cite{Be2003Lifting-of-morphis}).
\end{proof}

\pmed
\section{Unipotent Elements and Unipotent Group Actions}
\ps
\subsection{Unipotent group actions and fixed points}  
\label{unipotent-fixpoints.subsec}
Unipotent groups and their actions on  affine varieties have some very special properties which we recall now.

\begin{theorem} \label{unipotent-groups.thm}
Let $U$ be a unipotent group.
\be
\item
The group $U$ is nilpotent. More precisely, there is a composition series 
$$
U=U_{0} \supseteq U_{1}\supseteq\cdots\supseteq U_{m}=\{1\}
$$ 
such that all $U_i$ are closed,
each factors $U_{i}/U_{i+1}$ is central in $U/U_{i+1}$,  and is isomorphic to $\kplus$. In particular, the center of $U$ is nontrivial in case $U$ is nontrivial.
\item \label{affine-space}
Every closed subgroup $H \subseteq U$ is unipotent, and the homogeneous space $U/H$ is isomorphic to an affine space $\AA^{m}$.
\item
If $U$ is commutative, then it
has a canonical structure of a $\kk$-vector space, given by the exponential map $\exp_{U}\colon \Lie U \simto U$, and $\Aut_{\text{\it gr}}(U) = \GL(U)$.
\item \label{power-map}
For every integer $m \geq 1$ the power map $u \mapsto u^m$ is bijective.
\ee
Now let $X$ be an affine variety with an action of $U$.
\be
\setcounter{enumi}{4}
\item 
The $U$-orbits in $X$ are closed.
\item 
The invariant functions $\OOO(X)^{U}$ separate general orbits, i.e. there is a non\-empty open set $X' \subseteq X$ such that for any two points $x,y \in X'$ with different orbits $Ux\neq Uy$ there is an invariant $f \in\OOO(X)^{U}$ such that $f(x)\neq f(y)$. 
In particular, $\tdeg\kk(X)^{U}= \dim X - \max_{x\in X}\{\dim Ux\}$.
\item \label{relation-between-regular-and-rational-invariants}
If $X$ is irreducible, then $\kk(X)^{U}=\Quot(\OOO(X)^{U})$, the field of fractions of the algebra of invariants.
\item \label{homogeneous}
If $X$ is irreducible and $\kk(X)^{U}=\kk$, then $X$ is a single orbit under $U$.
\item 
If $X$ is factorial, i.e. $\OOO(X)$ is factorial ring, then $\OOO(X)^{U}$ is also a factorial ring. 
\ee
\end{theorem}
\begin{proof}[Outline of Proofs]
(1) See (\cite[IV, \S2, Proposition~2.5]{DeGa1970Groupes-algebrique} or \cite[III.1.1, Satz~2c]{Kr1984Geometrische-Metho}.
\ps
(2) See (\cite[IV, \S4, Corollaire~3.16]{DeGa1970Groupes-algebrique}).
\ps
(3) See \cite[IV, \S2, Proposition~4.1(iii)]{DeGa1970Groupes-algebrique}.
\ps
(4) This is clear for a commutative unipotent group, by (3). In general, it follows by induction from the exact sequence $1 \to Z \to U \to U/Z \to 1$ where $Z \subset U$ is the center of $U$.
\ps
(5) See \cite[Proposition~I.4.10, page~88]{Bo1991Linear-algebraic-g} or \cite[III.1.2, Satz~4]{Kr1984Geometrische-Metho}.
\ps
(6) This is a general result for algebraic group actions on varieties, due to \name{Rosenlicht}, see \cite[IV.2.2]{Sp1989Aktionen-reduktive}. 
\ps
(7) It is easy to see that the claim holds for any action of a connected algebraic groups which admits a local section. Now choose a normal subgroups $U'\subseteq U$ of dimension 1, i.e. $U'\simeq \kplus$. Then $\kk(X)^{U'}=\kk(\OOO(X)^{U'})$. Since the action of $U$ on $\OOO(X)$ locally finite and rational we can find a finitely generate $\kk$-subalgebra $R\subseteq \OOO(X)^{U'}$  which is stable under $U$ and such that $\Quot(R) = \Quot(\OOO(X)^{U'})$. Setting $\bar U:=U/U'$ we can assume, by induction on $\dim U$, that $\Quot(R)^{\bar U}=\Quot(R^{\bar U})$, and the claim follows (cf. \cite[IV.2.3]{Sp1989Aktionen-reduktive}).
\ps
(8) It follows from (7) that $\OOO(X)^{U}=\kk$. This implies, by (6), that $X$ contains a dense orbit. Now the claim is a consequence of (5).
\ps
(9) If $f\in\OOO(X)^{U}$ is an invariant, and $f=f_{1}^{m_{1}}\cdots f_{s}^{m_{s}}$ the decomposition into irreducible factors, then $uf_{i} = r_{i}f_{i}$ for every $u\in U$ where $r_{i}$ is a unit of $R:=\OOO(X)$. This means that the orbit $Uf_{i}$ is a subset of $R^{*}f_{i}$. Since $U$ is connected, we get $Uf_{i}\in\kst f_{i}$, hence $uf_{i}=f_{i}$, because there are no nontrivial morphisms $U \to \kst$.
\end{proof}

\begin{proposition} \label{generalunipotent.prop}
Let $U$ be a unipotent group of dimension $n$ acting freely on a factorial variety $X$ of dimension $n + 1$. 
Then $X$ is $U$-isomorphic to $U \times C$ where $C$ is a smooth rational curve. If $\OOO(X)^{*}= \Cst$, then $X \simeq \AA^{n+1}$.
\end{proposition}

\begin{proof}
The $U$-orbits $O=Ux \subseteq X$ are irreducible hypersurfaces, hence each $O$ is the zero set of an irreducible invariant function. Since the invariant ring $\OOO(X)^{U}$ is of dimension 1, it is finitely generated. It follows that the \itind{affine quotient} $X\quot U:=\Spec\OOO(X)^{U}$ 
(see Section~\ref{quotient.subsec}) is a smooth factorial curve, hence an open set $C \subseteq\Aone$. Moreover, the fibers of the quotient morphism $\pi\colon X \to C$ are the orbits. It also follows that $\pi$ is flat, and since the fibers are reduced and smooth the map $\pi$ is even smooth. This implies that there exist \'etale sections of $\pi$ which means that $\pi$ is a principal $U$-bundle. The first claim follows since every principal $U$-bundle over an affine variety is trivial, and the second claim is clear.
\end{proof}

\begin{remark}
If we assume that the action is only generically free, then one shows along the same lines that the affine quotient is a smooth rational curve $C$ and that the quotient morphism $\pi\colon X \to C$ is flat and surjective with general fiber isomorphic to $U$. In the special case where $X = \AA^{n+1}$ and $U$ is commutative
the {\it Commuting Derivations Conjectures\/} of \name{Maubach} states that  in this case $\pi$ is a variable, i.e. $\pi$ is a trivial $\An$-bundle. For $n=1$ this is clear, by the famous Theorem of \name{Abhyankar-Moh-Suzuki}, see Proposition~\ref{AMS.prop}. In \cite{Ma2003The-commuting-deri} the conjecture is proved for $n=2$ as a special case of a difficult result of \name{Kaliman} \cite{Ka2002Polynomials-with-g}.
\end{remark} 

It is well-known that a nontrivial unipotent automorphism of an irreducible affine variety has no isolated fixed points. In fact, one easily reduces to the one-dimensional case where the claim is obvious.
More generally, we have the following result.

\begin{proposition}\label{fixedpoints.prop}
Let $U$ be a unipotent group acting nontrivially on an irreducible affine variety $X$. Set $d:=\min\{\dim Ux\mid x\in X \setminus X^{U}\}>0$. Then every irreducible component of $X^{U}$ has at least dimension $d$. In particular, $X$ has no isolated fixed points.
\end{proposition}

\begin{proof} Let $C \subseteq X^{U}$ be an irreducible component of minimal dimension. By \cite[Theorem~2.3.15]{DeKe2002Computational-inva} there exists a $U$-invariant separating morphism $\phi\colon X \to Y$. This means that $Y$ is affine, $\phi$ is dominant, and for any pair $x,x'\in X$ we have $\phi(x)\neq \phi(x')$ if and only if there is an invariant $f\in\OOO(X)^{U}$ such that $f(x)\neq f(x')$. It follows that $\kk(Y)=\kk(X)^{U}$ which implies that the generic fiber of $\phi$ is an orbit. If $C$ is an irreducible component of a fiber of $\phi$ we are done. Otherwise, we choose an irreducible component $X_{1}$ of a fiber of $\phi$ which meets $C$. It suffices to prove the claim for $X_{1}$. In fact, if  $C_{1}$ is an irreducible component of $X_{1}^{U}$ of minimal dimension, then 
$$
\dim C \geq \dim C_{1}\geq d_{1}:=\min\{\dim Ux \mid x\in X_{1}\setminus X_{1}^{U}\} \geq d.
$$
By induction, we end up with the situation where every irreducible component of $X^{U}$ of minimal dimension is also an irreducible component of a fiber of $\phi$, and we are done.
\end{proof}

\ps
\subsection{Unipotent elements and \texorpdfstring{$\kplus$}{k+}-actions} \label{Cplus.subsec}
It is well known that the $\kplus$-actions on an affine variety $X$ are in bijection with the {\it locally nilpotent vector fields}\idx{locally nilpotent vector field} on $X$, i.e. with the {\it locally nilpotent derivations}\idx{locally nilpotent derivation} of $\OOO(X)$. Let us shortly recall how this bijection is obtained.

If $\lambda\colon \kplus \to \Aut(X)$ is a $\kplus$-action, then the {\it corresponding vector field\/} $\delta=\delta_{\lambda}\in\VEC(X)$ is defined in the following way (cf. Section~\ref{group-action-VF.subsec}):
$$
(\delta_{\lambda})_{x} =(d\mu_{x})_{0}(1) \text{ where }\mu_{x}\colon \kplus \to X
\text{ is the {\it orbit map} }  s \mapsto sx:=\rho(s)x.
$$\idx{orbit map}
(Note that we have a canonical identification $\Lie \kplus = T_{0}\kplus = \kk$, so that the element $1\in\Lie\kplus$ is well defined. In our previous notation from Section~\ref{group-action-VF.subsec} we have $\delta_{\lambda}=\xi_{1}$.) 

Equivalently, we have $\delta_{\lambda} = \xi(d\lambda(1))$
where $\xi\colon \Lie\Aut(X) \to \VEC(X)$ is the canonical anti-homomorphism of Lie algebras (see Proposition~\ref{End(X)-and-Vec(X).prop}).

\begin{lemma}\label{exp-for-functions.lem}
Let  $\lambda$ be a $\kplus$-action  on the affine variety $X$, and let $\delta_{\lambda} \in \VEC(X)$ be the corresponding vector field. Then $\delta_{\lambda}$ is locally nilpotent, and  for $f\in\OOO(X)$ and $s\in\kplus$ one gets
\[\tag{$*$}
f(sx) = \sum_{n \geq 0} \frac{s^{n}}{n!}(\delta_{\lambda}^{n}f)(x), \text{ i.e. }s^{-1}f = \exp(s\delta_{\lambda})f.
\]
Conversely, for every locally nilpotent vector field $\delta\in\VEC(X)$ there is a $\kplus$-action on $X$ such that the corresponding vector field is $\delta$.
\end{lemma}
\begin{proof}
We choose a closed $\kplus$-equivariant embedding $X \subseteq V$ into a $\kplus$-module $V$. Then the action is given by a homomorphism $\lambda\colon\kplus \to \GL(V)$ which is of the form
$\lambda(s) = \exp(sN)$ where $N:=d\lambda(1)\in\LLL(V)$ is a  nilpotent endomorphism (see Lemma~\ref{exp-basics.lem}). 
If $\xi_{N}\in\VEC(V)$  denotes  the corresponding locally nilpotent vector field on $V$, then $X$ is $\xi_{N}$ invariant, and $\delta_{\lambda} = \xi_{N}|_{X}$. Hence $\delta_{\lambda}$ is also locally nilpotent, and the formula $(*)$ follows from Lemma~\ref{exp-basics.lem}(\ref{exponential-formula}).

For the last statement, let $\delta$ be a locally nilpotent vector field on $X$. Then the formula $(*)$ defines a locally nilpotent
$\kplus$-action on the coordinate ring $\OOO(X)$, hence a $\kplus$-action on $X$.
\end{proof}

Thus we obtain a bijection between $\kplus$-actions on $X$ and locally nilpotent vector fields on $X$
$$
\begin{CD}
\nu_{X}\colon \Akplus(X) @>{\text{\tiny bijective}}>> \LNV(X) \subseteq \VEC(X), \quad \lambda\mapsto \delta_{\lambda}=\xi(d\lambda(1)),
\end{CD}
$$\idx{$\Akplus(X)$}
where $\LNV(X)$ are the locally nilpotent vector fields on $X$.
\pmed
There is also bijection between unipotent elements $\u\in\Aut(X)$ and $\kplus$-actions. In fact, if $\u\neq \id$, then there is a well-defined isomorphism 
$\lambda\colon\kplus\simto\overline{\langle \u \rangle}\subseteq \Aut(X)$ such that $\lambda(1)=\u$. On the other hand, we associate to a nontrivial $\kplus$-action $\lambda\colon\kplus\to \Aut(X)$ the unipotent element $\lambda(1)$, and thus obtain a map
$$
\begin{CD}
\eps_{X}\colon \Akplus(X) @>\text{\tiny bijective}>> \AutU(X) \subseteq \Aut(X), \quad \lambda\mapsto \lambda(1),
\end{CD}
$$
where $\AutU(X)$ are  the unipotent elements of $\Aut(X)$.
We know that $\Akplus(X)$ is an affine ind-variety (Lemma~\ref{group-actions.lem} and Proposition~\ref{group-action.prop}), and the construction above implies that the two maps $\nu_{X}\colon \Akplus(X) \to \VEC(X)$ and $\eps_{X} \colon \Akplus(X) \to \Aut(X)$ are morphisms of ind-varieties.

It is easy to see that  $\LNV(X)\subseteq \VEC(X)$ is weakly closed.
In fact, fixing generators $f_{1},\ldots,f_{n}\in\OOO(X)$ of the $\kk$-algebra $\OOO(X)$ we see that $\LNV(X)_{k}:=\{\delta \in \VEC(X)\mid \delta^{k}f_{i}=0 \text{ for }i=1,\ldots,n\}$ is closed in $\VEC(X)$, hence $\LNV(X) = \bigcup_{k}\LNV(X)_{k}$ is a union of closed algebraic subsets. Hence, $\LNV(X) \subset \VEC(X)$ is weakly closed in case $\kk$ is uncountable, by Proposition~\ref{indconstr.prop}. For a field extension $\KK/\kk$ we have $\LNV(X(\kk) )= \VEC(X(\kk)) \cap \LNV(X(\KK))$, and so $\LNV(X(\kk))$ is weakly closed in $\VEC(X(\kk))$ if this holds over $\KK$.

In a similar way we get the same statement for $\AutU(X)$.

\begin{lemma} \label{unipotent-weakly-closed.lem}
Let $X$ be an affine variety. Then the subset $\AutU(X)$ of unipotent elements of $\Aut(X)$ is weakly closed.
\end{lemma}

\begin{proof} As above we can assume that $\kk$ is uncountable.
Choose a finite-dimensional generating subspace $W \subseteq \OOO(X)$ with $1\in W$, and define the filtration $\OOO(X)=\bigcup_{k}\OOO(X)_{k}$ by $\OOO(X)_{k}:=\langle w_{1}\cdots w_{k}\mid w_{i}\in W\rangle$. Then we get a filtration of $\Aut(X)$ by the closed subsets $\Aut(X)_{k}:=\{g\in\Aut(X) \mid g^{*}(W) \subseteq \OOO(X)_{k}\}$. For $k,\ell \in \NN$ define
$$
A_{k,\ell}:= \{g\in\Aut(X)_{k} \mid (g^{*}-\id)^{\ell}(W) = (0)\}.
$$
Note that for $g\in\Aut(X)_{k}$ we get $(g^{*})^{m}(W) \subseteq \OOO(X)_{k^{m}}$. It follows that the
map 
$$
\Aut(X)_{k} \to \Hom(W,\OOO(X)_{k^{\ell}}), \ g\mapsto (g^{*}-\id)^{\ell}|_{W}
$$ 
is a morphism of varieties, and so $A_{k,\ell}$ is closed in $\Aut(X)_{k}$. 
By definition, $\Aut(X)^{\text{\it u}} = \bigcup_{k} A_{k,k}$, and the claim follows.
\end{proof}

Let us summarize the results of the considerations above. We assume for simplicity that $\kk$ is uncountable.

\begin{proposition}  \label{bijections-k+actions-nilpotentVF-unipotent-autos.prop}
We use the notation from above.
\be
\item
The map $\nu_{X}\colon \Akplus(X) \to \VEC(X)$, $\lambda\mapsto\xi(d\lambda_{0}(1))$, is an injective ind-morphism and defines a bijection $\tilde\nu_{X}\colon \Akplus(X) \to \LNV(X)$ where $\LNV(X)$ is weakly closed in $\VEC(X)$.
\item
The map $\eps_{X} \colon \Akplus(X) \to \Aut(X)$, $\lambda\mapsto \lambda(1)$, is an injective ind-morphism and defines a bijection 
$\tilde\eps\colon \Akplus(X) \to  \AutU(X)$ where $\AutU(X)$ is weakly closed in $\Aut(X)$.
\item 
For every closed algebraic subset $Y \subseteq \VEC(X)$ contained in $\LNV(X)$ the map $\tilde\nu_{X}^{-1}|_{Y}\colon Y \to \Akplus(X)$ is a morphism.
\item 
For every closed algebraic subset $Z \subseteq \Aut(X)$ contained in $\AutU(X)$ the map $\tilde\eps_{X}^{-1}|_{Z}\colon Z \to \Akplus(X)$ is a morphism.
\ee
\end{proposition}

\begin{proof}
(3) 
The assumption implies that $Y$ is contained in some  $\LNV(X)_{k}$, so that we may assume that $Y= \LNV(X)_{k}$.  For $s \in \kk$ and $\delta \in Y$ the following formula defines an automorphism of $\OOO (X)$:
\[
\phi_{s,\delta} := \sum_{j \geq 0}  \frac{s^{j}}{j!}  \delta^j.
\]
In fact, this a priori infinite sum becomes finite when applied to an element of $\OOO (X)$. It follows that the map $\Phi\colon\kk \times Y \to \Aut(\OOO(X))$, $(s,\delta) \mapsto \phi_{s,\delta}$,
is an ind-morphism. Moreover, we have $\phi_{s+t,\delta} = \phi_{s,\delta}\circ\phi_{t,\delta}$. 

If we denote by $\lambda_{\delta}(s)$ the automorphism of $X$ corresponding to $\phi_{s,\delta}$,
then the map $\lambda_{\delta}\colon\kplus \to \Aut(X)$, $s\mapsto \lambda_{\delta}(s)$, is a homomorphism of ind-groups, i.e. $\lambda_{\delta}\in\Akplus(X)$, and we obtain a family of $\kplus$-actions on $X$ parametrized by $Y$, i.e. a morphism $Y \to \Akplus(X)$. This is the inverse of $\tilde\nu_{X}$.
\ps
(4) 
As before we may assume that $Z= A_{k,k}$.  For $s \in \kk$ and $g \in Z$ the following formula defines 
a linear endomorphism of $\OOO (X)$:
\[
\psi_{s,g} := \sum_{j \geq 0}  \binom{s}{j}  (g^* - \id)^j.
\]
In fact, this a priori infinite sum becomes finite when applied to an element of $\OOO (X)$. 
As a consequence, the map
$$
\Psi\colon \kk \times Z \times \OOO(X) \to \OOO(X), \ (s,g,f)\mapsto \psi_{s,g}(f),
$$
is an ind-morphism. We claim that $\psi_{s,g}$ is an algebra endomorphism. This is clear for $s \in \ZZ$, because 
$\psi_{n,g}=(g^{*})^{n}$ for $n \in \ZZ$. Now consider the following two ind-morphisms 
\begin{gather*}
\Psi_{1},\Psi_{2}\colon \kk\times Z \times \OOO(X) \times \OOO(X) \to \OOO(X):\\
\Psi_{1}(s,g,f_{1},f_{2})=\Psi(s,g,f_{1}\cdot f_{2}), \quad \Psi_{2}(s,g,f_{1},f_{2}) = \Psi(s,g,f_{1})\cdot\Psi(s,g,f_{2}).
\end{gather*}
Since they coincide on the dense subset $\ZZ \times Z \times \OOO(X) \times \OOO(X)$, they are equal, proving the claim.
It follows that we have an ind-morphism, again denoted by $\Psi$,
$$
\Psi\colon \kk \times Z \to \End(\OOO(X)), \quad (s,g) \mapsto \psi_{s,g}.
$$
Next we claim that $\psi_{s+t,g} = \psi_{s,g}\circ\psi_{t,g}$. Again, this holds
for $s,t \in \ZZ$, and a similar density argument as before proves the claim. As a consequence, each $\psi_{s,g}$ is invertible and thus defines a family of automorphisms $\phi_{s,g} \in \Aut(X)$ parametrised by $\kk \times Z$.
Setting $\lambda_{g}(s):=\phi_{s,g}$
we get a homomorphism $\lambda_{g}\colon \kplus\to\Aut(X)$ of ind-groups, and thus a family of $\kplus$-actions $Z \to \Akplus(X)$, $g\mapsto\lambda_{g}$.
\end{proof}

\begin{corollary} If $\LNV(X) \subseteq \VEC(X)$ is closed, then $\nu_{X}\colon \Akplus(X) \simto \LNV(X)$ is an isomorphism of ind-varieties. 

If $\AutU(X) \subseteq \Aut(X)$ is closed, then $\eps_{X}\colon \Akplus(X) \simto \AutU(X)$ is an isomorphism of ind-varieties. 
\end{corollary}

\begin{corollary}\label{exp-for-Aut(X).cor}
The composition $\eps_{X}\circ (\nu_{X})^{-1}$ defines a bijective map
$$
\begin{CD}
\exp_{X}\colon \LNV(X) @>\text{\tiny bijective}>> \AutU(X)
\end{CD}
$$
with the following property. If $\lambda$ is a $\kplus$-action on $X$ and $\xi_{\lambda}$ the corresponding locally nilpotent vector field, then $\exp_{X}(\xi_{\lambda}) = \lambda(1)$. 

Moreover, for every algebraic subset $Y \subset  \VEC(X)$ contained in $\LNV(X)$ the induced map $\exp_{X}\colon Y \to  \GGG$ is a morphism.
\end{corollary}

\begin{question} \label{Closedness-of-the-set-of-unipotent-automorphisms.ques} 
Do the unipotent elements $\AutU(X)$ form a closed subset of $\Aut(X)$? Is $\LNV(X)$ closed in $\VEC(X)$?
\end{question}

\begin{question} \label{Closedness-of-the-set-of-unipotent-endomorphisms.ques}
More generally, is $\End^{\text{\it ln}} (X)$ closed in  $\End(X)$?
\end{question}

\ps
\subsection{The exponential map for an affine ind-group}  \label{exponential-for-affine-ind-groups.subsec}
The exponential map has been defined for a linear algebraic group in Section~\ref{exponential-for-linear-algebraic-groups.subsec} and for an ind-group of the form $\Aut (X)$ in Section~\ref{Cplus.subsec}
(see Proposition~\ref{exponential-for-linear-algebraic-groups.prop} and Corollary~\ref{exp-for-Aut(X).cor}). In this section, we generalize in some sense these definitions to any affine ind-group. We begin by recalling some results from Sections~\ref{exponential-for-linear-algebraic-groups.subsec} and \ref{Integration-of-VF.subsec}.

For the Lie algebra $\gg:=\Lie G$ of a linear algebraic group $G$ an element $N\in \gg$ is called \itind{nilpotent} if there is a faithful representation $\rho\colon G \into \GL_{n}$ such that the image of $N$ under $d\rho\colon \gg \into \M_{n}$ is a nilpotent matrix. It then follows that this holds for every representation of $G$. Moreover, one shows that for every nilpotent $N\in\gg$ there is a well-defined homomorphism $\lambda_{N}\colon \kplus \to G$ such that $d\lambda_{N}(1) = N$, see Section~\ref{exponential-for-linear-algebraic-groups.subsec}. Since we do not know if a general ind-group $\GGG$ admits a faithful representation we cannot carry this over to define \itind{locally nilpotent} elements in $\Lie\GGG$.

On the other hand, if $\GGG = \Aut(X)$ for an affine variety $X$, we can define locally nilpotent elements in $\Lie\Aut(X)$ by using the embedding $\Lie\Aut(X) \into \VEC(X)$, since we know what a locally nilpotent vector field is, see Section~\ref{Integration-of-VF.subsec}.   We showed that every such $N$ can be integrated, i.e. there is an action of $\kplus$ on $X$, $\lambda\colon\kplus \to \Aut(X)$, such that $d\lambda(1) = N$, see Proposition~\ref{Integration-of-VF.prop}(3). On the other hand, this does not hold for closed ind-subgroups $\GGG \subseteq \Aut(X)$ as we will see in Section~\ref{construction.subsec}, cf. Remark~\ref{Lie-H-cap-G.rem}.

\begin{definition}\label{locally-nilpotent-general.def}
Let $\GGG$ be an affine ind-group. An element $A \in \Lie\GGG$ is called \itind{locally nilpotent} if there is a closed unipotent algebraic subgroup $U \subseteq \GGG$ such that $A \in \Lie U$. Set $\Lieln\GGG$ to denote the set of locally nilpotent elements.
\end{definition}
The following proposition is clear.
\begin{proposition}
\be
\item
For any $N \in \Lieln\GGG$ there is a uniquely defined  homomorphism $\lambda_{N}\colon \kplus \to \GGG$ such that $d\lambda_{N}(1)=N$.
\item 
If $\phi\colon \GGG\to \HHH$ is a homomorphism of ind-groups, then $d\phi(\Lieln\GGG) \subseteq \Lieln\HHH$.
\item If $\GGG = \Aut(X)$ where $X$ is an affine variety, then $N \in \Lie\GGG$ is locally nilpotent if and only if the corresponding vector field
$\xi_N$ is locally nilpotent.
\item
For every representation $\rho\colon \GGG \to \GL(V)$ the image of $\Lieln\GGG$ under $d\rho$ is contained in $\End^{\text{\it ln}}(V)$.
\ee
\end{proposition}\label{Lie-G-ln.prop}

Let $\GGG$ be an affine ind-group with Lie algebra $\gg:=\Lie\GGG$. Denote by $\GGG^{u}\subseteq \GGG$ the subset of unipotent elements and by  $\gln\subseteq \gg$ the subset of locally nilpotent elements. We know that the set of homomorphisms $\Hom(\kplus,\GGG)$ is a closed subset of $\Mor(\kplus,\GGG)$ and thus has a natural structure of an affine ind-variety. The evaluation in $1\in\kplus$ defines a map\idx{$\End$@$\eps_{\GGG}$}
$$
\eps_{\GGG}\colon \Hom(\kplus,\GGG) \to \GGG, \quad \eps_{\GGG}(\lambda):=\lambda(1),
$$
and similarly, using the differential of $\lambda\colon\kplus \to \GGG$, we get a map\idx{$\M$@$\nu_{\GGG}$}
$$
\nu_{\GGG}\colon \Hom(\kplus,\GGG) \to \gg, \quad \nu_{\GGG}(\lambda):=d\lambda_{0}(1),
$$
cf. Section~\ref{exponential-for-linear-algebraic-groups.subsec}.
Note that $\GGG$ acts on $\GGG$ by conjugation, on $\gg$ by the adjoint representation, and on $\Hom(\kplus,\GGG)$ via conjugation on $\GGG$.
\begin{lemma}
\be
\item 
The maps $\eps_{\GGG}$ and $\nu_{\GGG}$ are injective ind-morphisms.
\item
The morphisms
$\eps_{\GGG}$ and $\nu_{\GGG}$ are $\GGG$-equivariant.
\item 
The morphism $\eps_{\GGG}$ induces a bijection $ \tilde \eps_{\GGG} \colon \Hom(\kplus,\GGG) \to \GGG^{u}$.
\item
The morphism $\nu_{\GGG}$ induces a bijection $\tilde \nu_{\GGG} \colon \Hom(\kplus,\GGG) \to \gln$.
\ee
\end{lemma}
\begin{proof}
(1) This  follows from  Lemma~\ref{tautological.lem}(1) that the map $\eps_{\GGG}$ is a morphism, and from Lemma~\ref{mor-to-tangent.lem} that   $\Hom(\kplus,\GGG) \to \LLL(\kk,\gg)$, $\lambda\mapsto d\lambda_{0}$, is also a morphism. Thus $\eps_{\GGG}$ and $\nu_{\GGG}$ are both morphisms of ind-varieties.

If $\lambda,\mu\in\Hom(\kplus,\GGG)$ such that $\lambda(1) = \mu(1)$, then $\lambda(n) = \lambda(1)^{n}=\mu(1)^{n}=\mu(n)$ for all $n\in \ZZ$, hence $\lambda=\mu$, showing that $\eps_{\GGG}$ is injective.
The injectivity of $\nu_{\GGG}$ follows from Proposition~\ref{phi-dphi.prop}.
\ps
(2) The equivariance for $\eps_{\GGG}$ is clear, and for $\nu_{\GGG}$ is follows from the definition of the adjoint representation.
\ps
(3) It is clear that the image of $\eps_{\GGG}$ is in $\GGG^{u}$. If $u \in \GGG^{u}$, $u\neq e$, then $\overline{\langle u \rangle}\subseteq \GGG$ is isomorphic to $\kplus$, and we can choose the isomorphism $\lambda\colon \kplus\simto \overline{\langle u \rangle}$ in such a way that $\lambda(1)=u$. 
\ps
(4) This follows from Proposition~\ref{Lie-G-ln.prop}(1).  
\end{proof}

\begin{definition}
The {\it exponential} and {\it logarithm} maps are defined by\idx{$\exp_{\GGG}$}\idx{$\log_{\GGG}$}
\[ 
\exp_{\GGG} := \eps_{\GGG} \circ  (\tilde \nu_{\GGG})^{-1} \colon \gln \to \GGG \ \text{ and } \ 
\log_{\GGG}  := \nu_{\GGG} \circ  (\tilde \eps_{\GGG})^{-1} \colon  \GGG^{u} \to \gg. 
\]
\end{definition}

\begin{question}
\be
\item Is $\eps_{\GGG}$ a closed immersion?
\item Is $\nu_{\GGG}$ a closed immersion?
\item Is it true that for any morphism $\phi \colon Y \to \GGG$ with image in $\GGG^{u}$ the composition
$\log_{\GGG} \circ \phi$ is a morphism?
\ee
\end{question}\label{eps-and-nu.ques}

\begin{remark}
Note that $\eps_{\GGG}$ is a closed immersion if and only if the two following assertions are satisfied:
\be
\item[(i)] $\GGG^u$ is closed in $\GGG$;
\item[(ii)] $ \tilde \eps_{\GGG} \colon \Hom(\kplus,\GGG) \to \GGG^{u}$ is an isomorphism.
\ee
Analogously $\nu_{\GGG}$ is a closed immersion if and only if the two following assertions are satisfied:
\be
\item[(i)] $\gln$ is closed in $\gg$;
\item[(ii)] $ \tilde \nu_{\GGG} \colon \Hom(\kplus,\GGG) \to \gln$ is an isomorphism.
\ee
In the case where $\eps_{\GGG}$ and $\nu_{\GGG}$ are both closed immersions, it is clear that $\exp_{\GGG}$ and $ \log_{\GGG}$ are morphism of ind-varieties.
\end{remark}

\ps
\subsection{Modifications of \texorpdfstring{$\kplus$}{k+}-actions}  \label{modification.subsec}

Let us discuss here the important special case of the construction of $\tilde\rho\colon G(\OOO(X)^{G}) \to \Aut(X)$ given in Section~\ref{large-subgroups.subsec} for the group $G=\kplus$. In this case, we can identify the ind-group $\kplus(\OOO(X)^{\kplus})$ with the $\kk$-vector space $\OOO(X)^{\kplus}$ considered as an additive ind-group. Denote by $\delta_{\rho}\in\VEC(X)$ the corresponding locally nilpotent vector field, see Section~\ref{Cplus.subsec}.

\begin{proposition}\label{modification.prop} Consider a nontrivial action of $\kplus$ on the affine variety $X$
given by a homomorphism of ind-groups $\rho \colon \kplus \to \Aut (X)$. Then
the homomorphism $\tilde\rho\colon \OOO(X)^{\kplus} \to \Aut(X)$ induces an isomorphism of ind-groups
$\OOO(X)^{\kplus}\simto \Aut_{\rho}(X)$,
and the differential $d\tilde\rho\colon \OOO(X)^{\kplus} \to \VEC(X)$ is given by $f \mapsto f\delta_{\rho}$.
\end{proposition}
\begin{proof}
(1) 
We know from Proposition~\ref{commutative-modification.prop} that the image of $\tilde\rho$ is $\Aut_{\rho}(X)$, since every nontrivial $\kplus$-action has a local section (Proposition~\ref{local-sections-for-kplus.prop}). Thus it suffices to show that $\tilde\rho$ is a closed immersion.
\ps
(2) 
First consider the case where $X = V$ is a $\kplus$-module. If we choose a basis, then $\rho(s) = (f_{1}(s,x),\ldots,f_{n}(s,x))$ where the polynomials $f_{i}\in \kk[s,x_{1},\ldots,x_{n}]$ are homogeneous and linear in the $x_{i}$. For any $h \in \kk[x_{1},\ldots,x_{n}]$ we define
$$
\tilde\rho(h) := (f_{1}(h,x_{1},\ldots,x_{n}), \ldots,f_{n}(h,x_{1},\ldots,x_{n}))\in \End(V),
$$
extending the homomorphism $\tilde\rho\colon \kk[x_{1},\ldots,x_{n}]^{\kplus} \to \Aut(V)$,
see Example~\ref{rho-for-rep.exa}.

Choosing a suitable basis we can assume that one of the coordinate functions of $\rho(s)$ has the form $f_{i}(s,x) = x_{i}+ sx_{i+1}$. Therefore, the map $\kk[x_{1},\ldots,x_{n}]\to \End(V)$, $h \mapsto \tilde\rho(h)$, is a closed immersion, because the composition with the projection onto the $i$th coordinate is a closed immersion. 
Hence, $\tilde\rho\colon \kk[x_{1},\ldots,x_{n}]^{\kplus} \to \Aut(V)$ is a closed immersion.

Moreover, the differential of the morphism $\tilde\rho\colon \kk[x_{1},\ldots,x_{n}]\to \End(V)$ in the origin $0\in\kk[x_{1},\ldots,x_{n}]$ is the linear map $h \mapsto h\delta_{\rho}$, because $\delta_{\rho}=\frac{\partial \rho(s)}{\partial s}\big|_{s=0}$. This proves the claim for the case of a $\kplus$-module.
\ps
(3)
In general, we choose a $\kplus$-stable closed embedding $X \subseteq V$ into a $\kplus$-module $V$. Then we have inclusions
$\Aut(X) \subseteq \End(X) \subseteq \Mor(X,V)$
where the first is locally closed (Theorem~\ref{AutX-locally-closed-in-EndX.thm}) and the second is closed (Proposition~\ref{closed-immersion-Mor.lem}). Moreover, the restriction $\phi \mapsto \phi|_{X}$ is a  surjective linear map  $\End(V) \onto \Mor(X,V)$, and we get the following diagram
$$
\begin{CD}
\Aut(X) @>{\subseteq}>> \End(X) @>{\subseteq}>> \Mor(X,V) @<{\text{linear}}<< \End(V) \\
@AA{\tilde\rho}A @AA{\tilde\rho}A @AA{\tilde\rho}A @AA{\tilde\rho}A \\
\OOO(X)^{\kplus} @= \OOO(X)^{\kplus} @>{\subseteq}>> \OOO(X) @<<< \OOO(V)
\end{CD}
$$
which implies that $\tilde\rho\colon\OOO(X)^{\kplus}\to \Aut(X)$ is a closed immersion, with differential $f \mapsto f\delta_{\rho}$.
\end{proof}

If $f \in \OOO(X)^{\kplus}$, then the $\kplus$-action 
$$
\rho_{f} \colon \kplus \to \Aut(X), \quad \rho_{f}(s) := \tilde\rho(sf)
$$ 
is often called a \itind{modification} of the action $\rho$ (cf. \cite{ArFlKa2013Flexible-varieties}). Explicitly,
$$
 \rho_{f}(s)(x) = \rho(f(x) s)(x).
 $$
 By abuse of notation, we also say that the unipotent automorphism $\u':=\rho_{f}(1)$ is a modification of $\u:=\rho(1)$.
 \ps
The following lemma is clear.
\begin{lemma}\label{modification.lem}
The modification $\rho_{f}$ commutes with $\rho$, and the $\rho_{f}$-orbits are contained in the $\rho$-orbits. If $X_{f}$ is dense in $X$, then both actions have the same invariants. Moreover, the fixed point sets of the two actions are related by 
$$
X^{\rho_{f}}= X^{\rho}\cup\{f=0\}.
$$
\end{lemma}

Here is a first interesting application.
\begin{proposition}\label{factorial.prop}
Let $X$ be a factorial affine variety with a nontrivial $\kplus$-action $\mu$. 
Then $\mu$ is a modification of a $\kplus$-action $\tilde\mu$ such that $\codim_{X}X^{\tilde\mu}\geq 2$.
\end{proposition}
\begin{proof}
Let $Y \subseteq X^{\mu}$ be an irreducible component of codimension 1. Then $Y =\{h=0\}$ where $h\in\OOO(X)^{\mu}$ is irreducible. Denote by $\delta_{\mu}$ the locally nilpotent vector field corresponding to the action $\mu$. Then, for all $f\in\OOO(X)$, $\delta_{\mu}(f)$ vanishes on $Y$, and so $\delta_{\mu}(\OOO(X))\subseteq h\OOO(X)$. Thus $\delta_{\mu} = h\delta_{\mu'}$ for a suitable locally nilpotent vector field $\mu'$. The claim follows by induction.
\end{proof}
\begin{corollary}\label{Cplussurface.cor}
Let $Y$ be an affine factorial surface with a nontrivial $\kplus$-action. Then $Y \simeq \kk\times C$ where $C$ is a rational smooth curve. In particular, $Y \simeq \kk^{2}$ in case $\OOO(Y)^{*}=\Cst$.
\end{corollary}
\begin{proof} By the proposition above we can assume that the action is fixed-point free, because a $\kplus$-action on an affine variety never has isolated fixed points (see Proposition~\ref{fixedpoints.prop}). Now the claim follows from Proposition~\ref{generalunipotent.prop}.
\end{proof}

\ps
\subsection{Deformations of \texorpdfstring{$\kplus$}{k+}-actions}\label{deform.subsec}
Let $X$ be an affine variety, and let $\Akplus(X)$ be the affine ind-variety of $\kplus$-actions on $X$, see Section \ref{var-group-actions.subsec}. We have an action of $\Aut(X)$ on $\Akplus(X)$ defined in the obvious way by
$$
 g(\mu)(s) :=  g\cdot\mu(s)\cdot g^{-1} \text{ for } g\in\Aut(X) \text{ and }\mu\colon\kplus \to \Aut(X).
$$
This action has the usual properties, e.g. if $X^{\mu}$ denote the fixed points of $\mu$, then $X^{ g(\mu)}=  g( X^{\mu})$, and if $O$ is a $\mu$-orbit, then $ g(O)$  is a  $g(\mu)$-orbit. With respect to the linear representation of $\Aut(X)$ on the vector fields $\VEC(X)$ we get the following: If $\delta_{\mu}\in\VEC(X)$ is the locally nilpotent vector field corresponding to the $\kplus$-action $\mu$, then 
$\delta_{ g(\mu)}=  g \delta_{\mu}$.

There is also a natural $\Cst$-action on $\Akplus(X)$ by ``rescaling''\idx{rescaling}: $\mu\mapsto\mu_{t}$ where $\mu_{t}(s) := \mu(ts)$. It corresponds to the scalar multiplication on $\VEC(X)$: $\delta_{\mu_{t}}=t\delta_{\mu}$ for $t\in\kk$. Clearly, this action commutes with the action of $\Aut(X)$, so that we finally get an action of 
$\Aut(X) \times \Cst$ on $\Akplus(X)$. 

\begin{definition}  \label{normCplus.def}
Let $G$ be an algebraic group acting on $X$ and let $\mu$ be a $\kplus$-action on $X$. We say that $G$ \it{normalizes} $\mu$ if the image of $G$ in $\Aut(X)$ normalizes the image of $\kplus$. It then follows that  there is a character $\chi\colon G \to \kst$ such that $(g(\mu))(s) = \mu(\chi(g)s)$, i.e.
$$
g(\mu(s)x) = \mu(\chi(g)s) gx \text{  for } s\in\kplus, g\in G \text{ and }x \in X.
$$ 
The character $\chi$ is called the \itind{character of $\mu$}.
\end{definition}

If $G$ normalizes $\mu$, then $G$ permutes the $\kplus$-orbits and induces an action on the $\kplus$-invariants $\OOO(X)^{\mu}$.
Moreover, $\mu$ is normalized by $G$ if and only if the line $\kk\delta_{\mu}\subseteq\VEC(X)$ is $G$-stable. In this case the action of $G$ on the line is given by the character $\chi$ of $\mu$. This shows that the actions of $\kplus$ on $X$ normalized by $G$ correspond to the $G$-stable lines in the locally nilpotent vector fields $\LNV(X)$.

Now let $T$ be a torus acting on $X$ and consider the representation of $\tT:=T\times \Cst$ on $\VEC(X)$. If $\mu$ is a $\kplus$-action on $X$ we denote by $O_{\mu}\subseteq\Akplus(X)$ the $\tT$-orbit of $\mu$. 

\begin{proposition}\label{deformCplus.prop} 
The closure of $O_{\mu}$ in $\Akplus(X)$ contains at least $d=\dim O_{\mu}$ nontrivial $\kplus$-actions normalized by $T$, with different characters. 
\end{proposition}

\begin{proof}
As seen above we can work in $\VEC(X)$ and replace $O_{\mu}$ by the $\tT$-orbit $O_{\delta_{\mu}}$ of $\delta_{\mu}$. Since $\tT$ contains the scalar multiplications, the only closed orbit of $\tT$ is $\{0\}$. Hence the cone $\CCC\subseteq X(\tT)_{\RR}$ generated by the weights supporting $\delta_{\mu}$ is salient and of real dimension $d$. Therefore, the orbit closure $\overline{O_{\delta_{\mu}}}$ contains at least $d$ different one-dimensional $\tT$-orbits  $O_{1},O_{2}\ldots$,  corresponding to the 1-dimensional edges of $\CCC$. Clearly, the closures $L_{j}:=\overline{O_{j}}$ are $\tT$-stable lines whose characters in $T$ are all different. Hence the claim.
\end{proof}

\begin{example} Let $X = Y \times \kk$ and consider the $\Cst$-action given by $t(x,z)=(x,tz)$. Denote by $\mu_{0}$ the free $\kplus$-action $s(y,z)= (y,z+s)$ for $s \in\kplus$.
\newline
{\it We claim that there are only two types of (nontrivial) $\kplus$-actions on $X$ which are normalized by $\Cst$, namely
\be
\item The actions commuting with $\Cst$, i.e. those with trivial character. These are the $\kplus$-actions which stabilize all the lines $Y\times\{z\}$. 
\item The actions with character $\chi=\id$. These are the modifications of the action $\mu_{0}$ (see \ref{modification.subsec}).  
\ee}
\noindent
\begin{proof} 
Let $\mu$ be a nontrivial $\kplus$-action normalized by $\Cst$. The $\kplus$-orbit $O_{(y,0)}$ is stable under $\Cst$, because $(y,0)$ is a fixed point of $\Cst$. It follows that either $Y\times \{0\}$ is fixed under $\mu$ or there is a $y_{0} \in Y$ such that $O_{(y_{0},0)}=\{y_{0}\}\times \kk$. In the first case, $Y\times(\kk\setminus\{0\})$ is stable under $\kplus$, and so all $Y\times\{z\}$ are stable under $\mu$, hence we are in case (a). 

In the second case, the action $\mu$ on the line $\{y_{0}\}\times \kk$ is of the from $(y_{0},z)\mapsto (y_{0}, z + ts)$ where $t\in\kst$. Hence $\mu$ has character $\chi=\id$, and we are in case (b). It follows that all lines $\{y\}\times \kk$ are stable under $\mu$ which implies that $\mu$ has the form $(y,z)\mapsto (y,z+p(y) s)$ where $p\in\OOO(Y)$, i.e. $\mu = (\mu_{0})_{p}$ (see \ref{modification.subsec}).
\end{proof}
\end{example}
As a consequence, we get the following two results due to \name{Crachiola} and \name{Makar-Limanov} \cite{CrMa2008An-algebraic-proof}.

\begin{proposition} \label{CM.prop}
Let $Y$ be an irreducible affine variety. If $Y$ does not admit a nontrivial $\kplus$-action, then every $\kplus$-action on $Y \times \kk$ is a modification of the standard action $\mu_{0}$. 
\end{proposition}

As a consequence, we get the following ``cancellation result''.

\begin{corollary}\label{cancellation.cor}
For any variety $Y$ such that $Y \times \Aone \simeq \Athree$, we get $Y \simeq \Atwo$.
\end{corollary}

\begin{proof}[Proposition~\ref{CM.prop}] 
We use the action of $T:=\Cst$ on $Y\times\kk$ and the induced action of $\tT:=\Cst\times\Cst$ on $\Akplus(Y\times\kk)$ as discussed above. 

Let $\mu$ be a nontrivial $\kplus$-action on $Y\times \kk$ which is not a modification of $\mu_{0}$. Then, by the example above, $\mu$ is not normalized by $\Cst$. This implies that the $\tT$-orbit $O_{\mu}\subseteq \Akplus(Y\times \kk)$ has dimension $2$. Hence, by Proposition~\ref{deformCplus.prop}, the closure $\overline{O_{\mu}}$ contains two $\kplus$-actions with different characters. Again by the example above, one of the characters must be trivial which contradicts the assumption. 
\end{proof}

\begin{proof}[Proof of Corollary~\ref{cancellation.cor}]
Since we have a closed immersion $Y \into \Aone\times Y \simeq \Athree$ and a surjective projection 
$\Aone\times Y \onto Y$ we see that $Y$ is affine, irreducible and that $\OOO(Y)^{*}=\Cst$. Since $\OOO(\Aone\times Y) \simeq \OOO(Y)[x]$ is factorial, it follows that $\OOO(Y)$ is factorial.
Moreover, not every $\kplus$-action on $Y\times\Aone\simeq \Athree$ is a modification of $\mu_{0}$, and so $Y$ admits a nontrivial $\kplus$-action by the previous proposition. Now the claim follows from Corollary~\ref{Cplussurface.cor}.
\end{proof}

\pmed
\section{Normalization}\label{normalization.sec}
\ps
\subsection{Lifting dominant morphisms}
Let $X$ be an irreducible affine variety and $\eta\colon \tX \to X$ its normalization\idx{normalization}. It is well known that every automorphism of $X$ lifts to an automorphism of $\tX$. Thus we have an injective  homomorphism of groups $\iota\colon\Aut(X) \to \Aut(\tX)$. One can also show that an action of an algebraic group $G$ on $X$ lifts to an action on $\tX$. This means that for every algebraic subgroup $G \subseteq \Aut(X)$ the image $\iota(G) \subseteq \Aut(\tX)$ is closed and $\iota$ induces an isomorphism $G \simto\iota(G)$ of algebraic groups.

More generally, if $\phi\colon X \to X$ is a dominant morphism, then there is a unique ``lift'' $\tphi\colon\tX \to\tX$, i.e. a morphism $\tphi$ such that $\eta\circ\tphi = \phi\circ\eta$, and $\tphi$ is also dominant. In fact, $\phi^{*}\colon\OOO(X) \to \OOO(X)$ is injective and thus induces an injections $\phi^{*}\colon \kk(X) \into \kk(X)$ of the field of rational functions. If an element $r\in\kk(X)$ satisfies an integral equation over $\OOO(X)$, then the same holds for $\phi^{*}(r)$, and so $\phi^{*}(\OOO(\tX)) \subseteq\OOO(\tX)$. 
This shows that we get an injective map $\iota\colon\Dom(X) \to \Dom(\tX)$. 

Here is our next result.
\begin{proposition}\label{normalization.prop}
The map $\iota\colon\Dom(X)\to\Dom(\tX)$ is a closed immersion of ind-semigroups, and $\iota\colon\Aut(X) \to \Aut(\tX)$ is a closed immersion of ind-groups.
\end{proposition}

For the proof we need the ``Division Lemma'' from Section~\ref{division.subsec} in the form given in Corollary~\ref{bilinear.cor}.

\ps
\subsection{Proof of Proposition~\ref{normalization.prop}} \label{proof.subsec}
Since $\Aut(X)$ is closed in $\Dom(X)$, by Theorem~\ref{AutX-locally-closed-in-EndX.thm}, it suffices to prove the first assertion.
\ps
(a) 
We assume first that the base field $\kk$ is uncountable. We proceed as in Section~\ref{embedding.subsec}. We fix two closed embeddings $X \into \An$ and $\tilde X \into \Am$, or equivalently generators $f_{1},\ldots,f_{n}$ of $R:=\OOO(X)$, and generators $h_{1},\ldots,h_{m}$ of $\tilde R:=\OOO( \tilde X)$.

Writing $R = \kk  [x_1,\ldots,x_n] / I$ where $I $ is the vanishing ideal of $X \into \An$, we define $R_{k}$ to be the image of the space $\kk[x_{1},\ldots,x_{n}]_{\leq k}$ of polynomials of degree at most $k$. This defines a filtration $R = \bigcup_{k\geq 0}R_{k}$ by finite-dimensional subspaces. We have a closed immersion
\[ 
\End (X) = \{ f=(f_1,\ldots,f_n) \in  R^n \mid q(f) = 0 \text{ for all } q \in I\}\subseteq R^{n}
\]
and an open immersion $\Dom (X) \subseteq \End (X)$ by Theorem~\ref{AutX-locally-closed-in-EndX.thm}. 
Admissible filtrations of $\End (X)$ and $\Dom (X)$ are defined in the following way:
\[ 
\End (X) _{k} := \End (X) \cap (R_{k})^n   \text{ \ and \ } \Dom (X) _{k}  := \Dom(X) \cap \End (X) _{k}.
\]
The semigroup $\End (X)$, resp. $\Dom (X)$, is naturally anti-isomorphic to the semigroup $\End (R)$ of endomorphisms of (the $\kk$-algebra) $R$, resp. the semigroup $\Inj (R)$ of injective endomorphisms of $R$. In the sequel, we will use $\End (R)$, $\Inj (R)$ rather than $\End (X)$, $\Dom (X)$.

Writing $\tilde R = \kk  [x_1,\ldots,x_m] / J$, we get analogously  a filtration $\tilde R = \bigcup_{\ell \geq 0} \tilde R_{\ell}$, and the inclusions $\Dom (\tilde X) \subseteq \End (\tilde X) \subseteq (\tilde R) ^m$. We will also use $\End (\tilde R)$, $\Inj ( \tilde R)$ rather than $\End ( \tilde X)$, $\Dom ( \tilde X)$.

An admissible filtration of $\Inj(R) $ is obtained in the usual way:
$$
\Inj(R)_{k}:=\{\phi\in\Inj(R) \mid \phi(R_{1})\subseteq R_{k}\},
$$
and similarly for $\Inj(\tilde R)$.
\ps
(b) 
Define $\Inj_{R}(\tilde R):=\{\phi\in\Inj(\tilde R) \mid \phi(R) \subseteq R\}\subseteq \Inj(\tilde R)$. This is clearly a closed sub-semigroup, and the restriction map $\res\colon \Inj_{R}(\tilde R) \to \Inj(R)$ is a bijective homomorphism of ind-semigroups. This implies, by Lemma~\ref{bijective-morphisms.lem}, that for every $k$ there is an $\ell$ such that $\res^{-1}(\Inj(R)_{k})\subseteq \Inj_{R}(\tilde R)_{k} \subseteq \Inj(\tilde R)_{\ell}$, i.e. $\iota(\Inj(R)_{k}) \subseteq \Inj(\tilde R)_{\ell}$. (Here we used that $\kk$ is uncountable!)
\ps
(c) 
Now we write the generators $h_{j}$ of $\tilde R$ as fractions $h_{j}=\frac{p_{j}(f_{1},\ldots,f_{n})}{q_{j}(f_{1},\ldots,f_{n})}$, with polynomials $p_{j},q_{j}$. We obtain morphisms $\rho_{j},\mu_{j}\colon \Inj(R) \to \tilde R$, defined by $\rho_{j}(\phi):=p_{j}(\phi(f_{1}),\ldots,\phi(f_{n}))$,  $\mu_{j}(\phi):=q_{j}(\phi(f_{1}),\ldots,\phi(f_{n}))$. Set  $Y  := \Inj(R)_{k}$. Then $\iota( Y ) \subseteq \Inj(\tilde R)_{\ell}$, hence 
$$
\frac{\rho_{j}(\phi)}{\mu_{j}(\phi)} = \frac{p_{j}(\phi(f_{1}),\ldots,\phi(f_{n}))}{q_{j}(\phi(f_{1}),\ldots,\phi(f_{n}))}
=\phi(h_{j}) \in \tilde R_{\ell} \text{ for }\phi \in Y.
$$
It follows from Corollary~\ref{bilinear.cor} that the maps $\phi\mapsto \phi(h_{j})$ are morphisms, and so $\iota \colon Y \to \Inj(\tilde R)$ is a morphism since $\Inj(\tilde R)\subseteq \tilde R^{m}$ is open.
\ps
(d) 
To get the result for a general field $\kk$, we use a base field extension $\KK/\kk$ with an uncountable algebraically closed field $\KK$. Then the claim follows from the above using Proposition~\ref{field-extensions-for-morphisms.prop} and Proposition~\ref{fieldextension.prop} (2) and (3).
\qed

\ps
\subsection{The endomorphisms of a cusp}\idx{cusp}
As an example we study the plane curve $C :=V(y^{2}-x^{3}) \subseteq \Atwo$, called \itind{Neile's parabola}, with the isolated singularity in $0$ and with normalization $\eta\colon \Aone \to C$, $s\mapsto (s^{2},s^{3})$. This implies the following description of the coordinate ring of $C$:
$$
\OOO(C) = \kk[s^{2},s^{3}] \subseteq \kk[s]
$$
where we identify $\bar x$ with $s^{2}$ and $\bar y$ with $s^3$. 
It is clear from Proposition~\ref{normalization.prop} that every endomorphism of $C$ lifts to an endomorphism of the normalization $\Aone$ since the non-dominant endomorphisms are the constant maps $\gamma_{c} \colon C \to C$ with value $c \in C$.

\begin{proposition}
\be
\item
Let $p=\sum_{i}a_{i}s^{i} \in \kk[s]=\OOO(\Aone)$. Then the map $p\colon\Aone\to\Aone$ induces an endomorphism $\bar p\colon C \to C$ if and only if $p\in M:=\{p(s)\in\kk[s] \mid a_{0}a_{1}=0$\}. Hence 
$$
\End(C) = \{(p(s)^{2},p(s)^{3})\mid p(s) \in M\} = \{(a,b)\in \OOO(C)^{2} \mid a^{3}=b^{2}\} \subseteq \OOO(C)^{2}
$$
is a closed ind-subvariety. 
\item
The corresponding map 
$$
\psi\colon M=s\CC[s] \cup (\CC\oplus s^{2}\CC[s]) \to \End(C), \ p\mapsto \bar p,
$$ 
is a bijective ind-morphism inducing an isomorphism $M\setminus\{0\} \simto \End(C)\setminus\{\gamma_{0}\}$. 
\item
The ind-morphism $\psi$ maps the constant polynomials $\kk$ bijectively onto the constant maps 
$\{\gamma_{c}\mid c\in C\}$, and  the induced morphism
$\psi|_{\kk}\colon \kk \to \psi(\kk)\simeq C$ is the normalization. In particular, $\psi$ is not an isomorphism.
\item
The ind-morphism $\psi$ induces an isomorphism $\psi'\colon M \setminus\kk \simto \Dom(C)$.
\item
The map $\iota\colon\Aut(C) \into \Aut(\Aone)$ corresponds to the closed immersion $\Cst \into \Aff_{1}=\kst\ltimes\kplus$, and $\iota\colon \Dom(C) \into \Dom(\Aone)$ corresponds to the closed immersion $M\setminus\kk \into \kk[s]\setminus\kk$
\item
If we identify $\Aut(C)$ with $\Cst$ by setting $t(x,y):=(t^{2}x,t^{3}y)$, then
$\psi$ is $\Cst$-equivariant where the action on $M$ is by scalar multiplication and on $\End(C)$ by left-multiplication. 
\item
For the vector fields we get 
\[
\VEC(C) = \OOO(C) (2x\dx + 3y\dy)|_{C}+\OOO(C)(2y\dx + 3x^{2}\dy)|_{C},
\]
and the linear  map 
$$
\xi \colon  T_{\id}\End(C) \simto \VEC(C),
$$
described in Proposition~\ref{End(X)-and-Vec(X).prop},
is an isomorphism.
\ee
\end{proposition}
The situation is illustrated in the following diagram:
$$
\begin{CD}
\kk @>{\subseteq}>> M:=\{p\in\kk[s] \mid a_{0}a_{1}=0\} = (s\kk[s]) \cup (\kk\oplus s^{2}\kk[s])\\
@VV{\text{bijective}}V @V{\psi\colon}V{p\mapsto(p^{2},p^{3})}V\\
C @>{\into}>> \End(C) = \{(a,b)\in \OOO(C)^{2} \mid a^{3} = b^{2}\} \subseteq \OOO(C)^{2}
\end{CD}
$$

\begin{proof}
(1) The first part is clear. If $a,b \in \kk[s]$ satisfy $a^{3}=b^{2}$, then $a = p^{2}$ and $b = p^{3}$ for a suitable $p\in\kk[s]$. Now it is easy to see that if $a,b \in \kk[s^{2},s^{3}]$, then $p\in M$.
\par\smallskip
(2) It is obvious that $\psi$ is a bijective ind-morphism. To show that $\psi$ is an isomorphism on the complement of the origin we use Lemma~\ref{division-biss.lem} above which implies that the map $\End(C)\setminus\{\gamma_{0}\} \to M$, $(a,b)\mapsto \frac{b}{a}$, is a morphism.
\par\smallskip
(3) is clear and (4) follows from (2) since the non-constant morphisms are dominant. 
\par\smallskip
(5) and (6) are clear.
\par\smallskip
(7) 
Recall that $\VEC_{C}(\Atwo):=\{\delta\in\VEC(\Atwo) \mid \delta_{c} \in T_{c}C \text{ for all }c \in C\}$ and that the restriction map $\VEC_{C}(\Atwo) \to \VEC (C)$, $\delta \mapsto \delta|_{C}$ is surjective, by Proposition~\ref{vector-fields.prop}(\ref{closed-subvarieties-and-VF}).
We have $\VEC_{C}(\Atwo) = \{f\dx + h\dy\mid -3fx^{2}+2hy = 0\}$. It follows that the two restrictions $\bar f(s):=f(s^{2},s^{3}), \bar h(s):=h(s^{2},s^{3}) \in \kk[s^{2},s^{3}]=\OOO(C)$ satisfy the equation $3\bar f s^{4} = 2\bar h s^{3}$, hence $\bar h = \frac{3}{2}s \bar f$. In particular, $\bar f(0)=0$, and so $\bar f \in \mm=\OOO(C)s^{2}\oplus \kk s^{3}$. For $\bar f = 2s^{2}$ we obtain the vector field $(2x\dx+3y\dy)|_{C}$, and for  $\bar f = 2s^{3}$ we find $(2y\dx + 3x^{2}\dy)|_{C}$. Hence $\VEC(C) = \OOO(C)(2x\dx+3y\dy)|_{C} \oplus \kk(2y\dx + 3x^{2}\dy)|_{C}$, as claimed.

It remains to calculate the linear map $\xi \colon T_{\id}\End(C) \to \VEC(C)$.
We have seen that the bijective morphism $\Psi\colon M \to \End(C)$ induces an isomorphism $M\setminus\{0\} \simto \End(C)\setminus\{\gamma_{0}\}$,
and sends $s \in M$ to $\id\in\End(C)$, hence $T_{s}M \simeq T_{\id}\End(C)$. Since $s$ belongs only to the irreducible component $s\kk[s]$ of $M$, we have $T_{s}M = T_{s} s\kk[s] = s\kk[s]$, 
and we will identify $T_{\id}\End(C)$ with $s \kk[s]$. An easy computation (see Section~\ref{Endo-VF.subsec} and in particular Proposition~\ref{End(X)-and-Vec(X).prop}) shows that if $q \in T_{\id}\End(C) = s \kk[s]$, then we have $\xi_q = (2sq\dx + 3s^{2}q\dy)|_{C} \in \VEC(C)$.
Choosing $q =  fs$ where $f \in \kk[s^{2},s^{3}]$ we get $\xi_{q}=f(2x\dx + 3y\dy)|_{C}$, and for $q = s^{2}$ we find $\xi_{q}=(2y\dx + 3x^{2}\dy)|_{C}$, showing that $\xi$ is surjective.
\end{proof}

\begin{example}
If $r,s\geq 2$ are coprime, then  the zero set $C:=V(y^{r}-x^{s})\subseteq \Atwo$ is a \itind{rational cuspidal curve} with a unique singularity. It is shown in \cite[Proposition~8.2]{Kr2017Automorphism-group} that for a finite family of pairwise nonisomorphic cuspidal curves $C_{1},\ldots,C_{m}$ the automorphism group 
$\Aut(C_{1}\times C_{2}\times\cdots\times C_{m})$ is an $m$-dimensional torus.
\end{example}

\pmed
\section{The Automorphism Group of a General Algebra} 
\label{GeneralAlgebra.sec}
\ps
\subsection{Endomorphisms of a general algebras}\label{GeneralAlgebra.subsec}
Recall that a \itind{general  $\kk$-algebra} $\R$ is  a  $\kk$-vector space $\R$\idx{$\R$} endowed with a bilinear map $\R \times \R \to \R$, $(a,b) \mapsto a\cdot b$. A priori, there are  no additional conditions like associativity, commutativity or the existence of a neutral element. If $\R$ is finite-dimensional, then $\Aut (\R)$ is a linear algebraic group and $\Lie ( \Aut (\R) ) = \Der_{\kk} (\R)$, see Proposition~\ref{Lie algebra of the automorphism group of a fd general algebra.prop}. Let us mention that by a result of \name{Gordeev-Popov} \cite{GP2003Automorphism-groups} any linear algebraic group is isomorphic to $\Aut (\R)$ for some finite-dimensional general algebra $\R$.

We will now deal with the more general case where $\R$ is just assumed to be finitely generated. This implies that $\R$ has countable dimension and thus has a canonical structure of an ind-variety such that the multiplication is a morphism. We will see that $\Aut (\R)$ is naturally an ind-group whose Lie algebra $\Lie ( \Aut (\R) )$ embeds into  $\Der_{\kk} (\R)$. For this, we will first endow $\End (\R)$ with the structure of an ind-semigroup. 

\begin{definition} \label{family of algebra endomorphisms.def}
Let $\R$ be a finitely-generated general algebra and $\YYY$ an ind-variety. A {\it family of algebra endomorphisms of $\R$ parametrized by $\YYY$}\idx{family of algebra endomorphisms} is a morphism $\Phi \colon \R \times \YYY \to \R$ such that for any $y \in \YYY$ the induced map $\Phi_y \colon \R \to \R$, $a \mapsto \Phi (a,y)$ is an algebra endomorphism of $\R$. We use the notation $\Phi = (\Phi_{y})_{y \in \YYY}$, so that a family $\Phi$ of endomorphisms can be regarded as a map $\Phi\colon \YYY \to \End (\R)$.
\end{definition}

\begin{proposition} \label{ind-var of algebra endomorphisms.prop}
Let $\R$ be a finitely-generated general $\kk$-algebra. There exists a universal structure of ind-variety on $\End (\R)$ such that families of algebra endomorphisms of $\R$ parametrized by $\YYY$ correspond to morphisms $\YYY \to \End(\R)$. Moreover, $\End(\R)$ is an ind-semigroup, i.e. the multiplication is a morphism.
\end{proposition}

\begin{proof}
Let ${\FFF}(x_1,\ldots,x_n)$ be the free general algebra on $x_1, \ldots, x_n$. This means that given any general algebra $\SSS$ and any elements $b_1,\ldots,b_n$ in $\SSS$ , there exists a unique homomorphism of algebras ${\FFF}(x_1,\ldots,x_n) \to \SSS$ sending $x_i$ to $b_i$ for $1 \leq i \leq n$. Recall that a basis of ${\mathcal F}(x_1,\ldots,x_n)$ over $\kk$ is given by all well-formed expressions in the elements $x_1,\ldots,x_n$ where a well-formed expression is recursively defined by the following: each element $x_i$ is a well-formed expression, and if $t,t'$ are well-formed expressions, then so is $(t\cdot t')$. By definition, the product of two well-formed expressions $t$ and $t'$ is the well-formed expression $(t\cdot t')$. The elements of ${\FFF}(x_1,\ldots,x_n)$ are the general polynomials, i.e. the formal linear combinations of well-formed expressions in the elements $x_1,\ldots,x_n$ with coefficients in  $\kk$.

Now assume that $\R$ is generated by $a_{1},\ldots,a_{n}\in \R$.
We consider the surjective homomorphism of algebras $\phi \colon {\FFF}(x_1,\ldots,x_r) \to \R$ sending $x_i$ to $a_i$ for $1 \leq i \leq n$. Its kernel $\Ker \phi$ consists of the general polynomials $p(x_1,\ldots,x_n) \in {\FFF}(x_1,\ldots,x_n)$ which are sent to zero by $\phi$.
We can identify $\End(\R)$ with the following subset $E \subseteq \An$ by sending $\phi\in\End(\R)$ to $(\phi(a_{1}),\ldots,\phi(a_{n}))$:
\[ 
\End(\R) \simto E:=\{ f =(f_1, \ldots,f_n) \in \R^n \mid p (f) = 0 \text{ for all }p \in \Ker \phi  \}\subseteq \R^n.
\]
It is not difficult to see that  $E$ is closed in $\R^n$. One has only to check  that for a general polynomial $p \in {\FFF}(x_1,\ldots,x_n)$, the map $\R^n \to \R$ sending $(a_1,\ldots,a_n)$ to $p(a_1, \ldots, a_n)$ is a morphism. We leave the details to the reader. 

Now it is easy to see that with this structure of an ind-variety, $\End(\R)$ has the requested universal property, and that the multiplication is a morphism.
\end{proof}

\begin{remark}\label{ind-var-of-Hom.rem}
The proposition above generalized to $\Hom(\RRR,\SSS)$ where $\RRR$ is a finitely generated general $\CC$-algebra and $\SSS$ is a general algebra of countable dimension. The details are left to the reader.
\end{remark}

\begin{remark} \label{unitary algebras.rem}
In the case of \itind{unitary\/} general algebras, we must add a unique well-formed expression of length zero, usually denoted by $1$, with the property $1\cdot a = a \cdot 1 = a$ for all $a\in \R$.
\end{remark}

\ps
\subsection{Automorphisms of a general algebra}
We start with the obvious generalizations of Definition~\ref{family of automorphisms.def} and Theorem~\ref{AutX.thm} to general algebras.

\begin{definition} \label{family of algebra automorphisms.def}
Let $\R$ be a finitely generated general algebra and $\YYY$ be an ind-variety. A {\it family of algebra automorphisms of $\R$ parametrized by $\YYY$}\idx{family of algebra automorphisms} is an automorphism $\Phi \colon \R \times \YYY \to \R \times \YYY$ over $\YYY$ such that for any $y \in \YYY$ the induced map $\Phi_y \colon \R \to \R$, $(a,y) \mapsto \Phi (a,y)$ is an algebra automorphism of $\R$. We use the notation $\Phi = (\Phi_{y})_{y \in \YYY}$, so that a family $\Phi$ of automorphisms can be regarded as a map $\Phi\colon \YYY \to \Aut (\R)$.
\end{definition}

\begin{proposition}  \label{ind-var of algebra automorphisms.prop}
Let $\R$ be a finitely generated general algebra. There exists a universal structure of ind-variety on $\Aut (\R)$ such that families of algebra automorphisms of $\R$ parametrized by $\YYY$ correspond to morphisms $\YYY \to \Aut(\R)$. With this structure, $\Aut(\R)$ is an ind-group.
\end{proposition}

\begin{proof}
Define the closed subset
\[ 
\{(\f, \g)\in \End (\R) \times \End (\R) \mid \f\cdot  \g= \id= \g\cdot \f\} \subseteq \End (\R) \times \End (\R). 
\] 
which we identified with $\Aut (\R)$ via the first projection.
It is easy to see, using Proposition~\ref{ind-var of algebra endomorphisms.prop}  that with this structure of an ind-variety the group $\Aut (\R)$ is an ind-group with the required universal property.
\end{proof}

\ps
\subsection{The Lie algebra of \texorpdfstring{$\Aut(\R)$}{Aut(R)}}\label{Lie-algebra-of-Aut(R)-for-R-a-general-algebra.subsection}
For a general finitely generated algebra $\R$ we have a canonical representation of $\Aut(\R)$ on $\R$  which defines a homomorphism of Lie algebras $\xi\colon \Lie\Aut(\R) \to \LLL(\R)$ where $\LLL(\R)$ are the linear endomorphisms of the $\kk$-vector space $\R$, see Lemma~\ref{rep-of-G-and-LieG.lem}. The image $\xi_{A}\in\LLL(\R)$ of $A\in\Lie\Aut(\R)$ is defined by $\xi_{A}(a):=d\mu_{a}(A)$ where $\mu_{a}\colon \Aut(\R) \to \R$, $a \mapsto g(a)$, is the evaluation at $a$.
It is clear from the definition, that $\xi$ extends to a map $\tilde\xi\colon T_{\id} \End(\R) \to \LLL(\R)$ defined in the same way.

\begin{proposition}  \label{Tangent-space-of-End(R)-and-Aut(R)-for-a-general-algebra-R.prop}
The homomorphism $\tilde\xi$ is injective, and its image is contained in the derivations of $\R$,
\[
\tilde\xi \colon T_{\id}\End(\R) \into \Der_{\kk}(\R)\quad\text{and}\quad \xi\colon \Lie\Aut(\R) \into \Der_{\kk}(\R),
\]
where the second map is a homomorphism of Lie algebras.
\end{proposition}
\begin{proof}
For $a,b \in\R$ the morphism $\mu_{a\cdot b}$ is the composition
\[
\begin{CD}
\End(\R) @>(\mu_{a},\mu_{b})>> \R \times \R @>{m}>> \R
\end{CD}
\]
where $m$ is the multiplication of $\R$. From this we see that the differential $d\mu_{a\cdot b}$ is given by the composition
\[
\begin{CD}
T_{\id} \End(\R) @>(d\mu_{a},d\mu_{b})>> T_{a}\R \oplus T_{b}\R @>{dm_{(a,b)}}>> \R
\end{CD}
\]
Since $dm_{(a,b)}(x,y) = a\cdot y + x \cdot b$ we finally get $\xi_{A}(a\cdot b) = a\cdot\xi_{A}(b) + \xi_{A}(a)\cdot b$ showing that $\xi_{A}$ is a derivation of $\R$.  

In order to see that $\tilde\xi$ is injective we have to recall the definition of the ind-variety structure of $\End(\R)$ as a closed subset $E \subseteq \R^{n}$ using a system of generators $(a_{1},\ldots,a_{n})$ of $\R$ (see proof of Proposition~\ref{ind-var of algebra endomorphisms.prop}). It follows that we have an injection $T_{\id} \End(\R) \into \R^{n}$ which has the following description: $A \mapsto (\tilde\xi_{A}(a_{1}),\ldots,\tilde\xi_{A}(a_{n}))$. Since any derivation of $\R$ is determined by the images of the generators $a_{1},\ldots,a_{n}$ we see that $\tilde\xi_{A}\neq 0$ if $A\neq 0$.
\end{proof}

\begin{question}   \label{Is-T(End(RR))-a-Lie-algebra.question}
Is  $T_{\id}\End (\R)$  a Lie subalgebra of  $\Der_{\kk} \R$ for a finitely generated general algebra $\R$?
\end{question}

\ps
\subsection{Locally finite derivations}
Recall that a linear endomorphism $\phi$ of a $\kk$-vector space $V$ is called \itind{locally finite} if the linear span $\langle \phi^{j}(v)\mid j\in\NN\rangle$ is finite-dimensional for all $v\in V$. It is called \itind{semisimple} if it is locally finite and if the action on any finite-dimensional $\phi$-stable subspace is semisimple. It is called \itind{locally nilpotent} if for any $v \in V$ there is an $m\in\NN$ such that $\phi^{m}(v)=0$. It is well-known and easy to prove that every locally finite endomorphism $\phi$ of $V$ admits a Jordan decomposition: $\phi = \phi_{s}+\phi_{n}$ where $\phi_{s}$ is semisimple, $\phi_{n}$ is locally nilpotent and $\phi_{s}\circ\phi_{n}=\phi_{n}\circ \phi_{s}$, and this decomposition is uniquely determined by these properties.

We have already mentioned in Remark~\ref{Integration-of-VF.subsec} that the proof of Proposition~\ref{Integration-of-VF.prop} carries over to a general 
$\kk$-algebra $\R$, giving the following result.

\begin{proposition}\label{locally-finite-derivations.prop}
Let $\R$ be a finitely generated general $\kk$-algebra. Then the Lie algebra $\Lie \Aut (\R) $ contains all locally finite  derivations of $\R$.
More precisely, for any locally finite derivation $D$ there is a unique minimal commutative and connected algebraic subgroup $A \subseteq \Aut(\R)$ such that $D \in \Lie A$. 
\end{proposition}

If  $\R$ is finite-dimensional, then every derivation is locally finite, hence we have the following result.

\begin{corollary} \label{Lie algebra of the automorphism group of a fd general algebra.prop}
Let $\R$ be a finite-dimensional general algebra. Then $\Aut(\R)$ is an algebraic group, and we have
\[ 
\Lie \Aut (\R)   =\Der_{\kk} (\R).
\]
\end{corollary}

\begin{remark}
If $\kk = {\mathbb C}$ we have a very simple direct proof using the exponential map.
Indeed, if $\R$ is a finite general ${\mathbb C}$-algebra and $D$ any derivation of $\R$, we can define a linear endomorphism $\exp (D) \colon \R \to \R$ by $(\exp D) (a) = \sum_n \frac{1}{n!} D^na$. Using the   \name{Leibniz} formula $D^n(ab) = \sum_{k=0}^n \binom{n}{k} (D^ka) (D^{n-k}b)$, it is clear that $\exp D$ belongs to $\Aut \R$. 
Since the derivative at $0$ of the map $\AA^1 \to \Aut \R$, $t \mapsto \exp (t D)$, equals $D$ the claim follows.
\end{remark}

\begin{question}\label{generation-by-locally-finite-derivations.ques}
Is $\Lie \Aut ( \R)$ generated by the locally finite derivations for any finitely generated general algebra $ \R$?
\end{question}

\pmed
\section{Bijective homomorphism of ind-groups}\label{bijective-hom.sec}
It is well known that any topological group admits at most one structure of a real Lie group. In fact, if a topological group is locally Euclidean, i.e. admits charts, then it admits a unique structure of real Lie group, see e.g. \cite{Ka1971Lie-algebras-and-l}.

If $\kk$ is countable, then $(\kk_{\text{discr}},+) \to\kk^{+}$ is a bijective homomorphism of ind-groups, but not an isomorphism. So at least for countable base fields the same group can have different structures as an ind-group. We will show below that there is a much more interesting example which is defined for any algebraically closed field $\kk$. Before doing so let us describe a situation where we can give a positive answer.

\ps
\subsection{The case of a strongly smooth ind-group}
If we make very strong assumptions for the target group it is possible to show that a bijective homomorphism is indeed an isomorphism. The following result is an immediate consequence of Proposition~\ref{bijective-morphism.prop}.

\begin{proposition}  \label{bijective-Hom.prop}
Let $\phi\colon \GGG \to \HHH$ be a bijective homomorphism of ind-groups. Assume that $\GGG$ is connected and $\HHH$ strongly smooth in $e$. Then $\phi$ is an isomorphism.  
\end{proposition}
Note that it would be enough to assume that $\HHH$ admits a filtration with normal varieties (see Proposition~\ref{bijective-morphism.prop}). In view of the next result about the bijective morphism $\alpha\colon \Aut(\kk \langle x,y \rangle) \to \Aut(\kk [x,y])$ which is not an isomorphism (Proposition~\ref{diff-not-bijective.prop}) we have the following consequence.

\begin{corollary}  \label{Aut(A2)-is-not-strongly-smooth.cor}
The ind-group $\Aut(\Atwo)$ is not strongly smooth in the identity $\id \in \Aut(\Atwo)$.
More generally, $\Aut(\Atwo)$ does not admit an admissible filtration consisting of affine normal varieties.
\end{corollary}

\ps
\subsection{The free associative algebra}
Denote by $\kk  \langle x,y \rangle$ the free associative algebra in two generators. The abelianization $\alpha\colon \kk\langle x,y \rangle \to \kk[x,y]$ induces a homomorphism, also denoted by $\alpha$,
$$
\alpha \colon \Aut(\kk\langle x,y \rangle) \to \Aut(\kk[x,y]), \ (f,h)\mapsto (\alpha(f),\alpha(h)).
$$ 
This homomorphism is bijective (see \cite{Ma1970The-automorphisms-}, \cite{Cz1971Automorphisms-of-a} \cite{Cz1972Automorphisms-of-a}, \cite[Th.9.3, p. 355]{Co1985Free-rings-and-the}, but not an isomorphism of ind-groups.

\begin{proposition} \label{diff-not-bijective.prop}
The homomorphism
\[  
\alpha\colon \Aut(\kk \langle x,y \rangle) \to \Aut(\kk [x,y])
\]
is a bijective homomorphism of ind-groups which is not an isomorphism. More precisely, the differential 
$d\alpha\colon \Lie\Aut(\kk \langle x,y \rangle) \to \Lie\Aut(\kk [x,y])$ is surjective with a nontrivial kernel.
\end{proposition}

\begin{remark} 
This result was already pointed out by \name{Yuri Berest} and \name{George Wilson} in \cite[Section~11, last paragraph]{BeWi2000Automorphisms-and-}.
\end{remark}
\begin{proof}
It is clear that $\alpha$ is an ind-morphism. In order to see that the differential $d\alpha$ is not injective consider the following (inner) derivation of $\kk \langle x,y \rangle$:
\[ 
\ad_{[x,y]}=(2 xyx -x^2y-yx^2) \ddx + (xy^2 + y^2x -2 yxy) \ddy.
\]
An easy computation shows that
\[ 
\ad_{[x,y]} =\frac{3}{2} [ x^2 \ddy, y^2 \ddx ] - \frac{1}{2} [[x^2 \ddy, y \ddx ], [ x \ddy , y^2 \ddx ]], 
\]
so that $\ad_{[x,y]}$ belongs to $\Lie  \Aut ( \kk \langle x,y \rangle )$. However, this derivation is clearly in the kernel of  $ d\alpha\colon\Lie  \Aut ( \kk \langle x,y \rangle )  \to \Lie \Aut ( \kk [ x,y ] )$. 
\end{proof}

\ps
\subsection{A geometric argument}
We want to show that the bijective homomorphism of ind-groups  $\alpha\colon \Aut(\CC \langle x,y \rangle )  \to  \Aut ( \CC [ x,y ] )$ is not an isomorphism by constructing a rational curve in $\mu(t)\colon \CC \to \Aut ( \CC \langle x,y \rangle )$, i.e. a family of automorphisms, which is not sent isomorphically onto its image in $\Aut ( \CC [ x,y ] )$.

If $\rho\colon \CC \to \End(\Atwo)$ is a curve, we can write $\rho(t) = \sum_{i\geq 0} t^{i}\rho_{i}$ with $\rho_{i}\in \End(\Atwo)$ and $\rho_{i}=0$ for large $i$. For example, if $\delta = f\ddx+h\ddy$ is locally nilpotent, $\delta^{m}x=\delta^{m}y=0$, then $\exp(t\delta)_{0}=\id = (x,y)$, $\exp(t\delta)_{1}=(f,h)$ and $\exp(t\delta)_{i}=0$ for $i\geq m$.

If  $\exp(t\delta)_{i_{0}}$ is the lowest nonzero term different from the first term $(x,y)$, then we get for the product
$$
\exp(t\delta)\cdot\exp(t\mu) = (x,y) + t\exp(t\mu)_{1}+\cdots+t^{i_{0}}(\exp(t\delta)_{i_{0}}+\exp(t\mu)_{i_{0}}) + \cdots
$$
i.e. the first $i_{0}$ terms of the product are the same than those of $\exp(t\mu)$, and the term of degree $i_{0}$ is the sum of the two terms of $\exp(t\delta)_{i_{0}}$ and $\exp(t\mu)_{i_{0}}$.

Let us calculate the commutator $\gamma(t):=(\exp(tx^{2}\ddy),\exp(t y^{2}\ddx)$ in $\Aut(\CC\langle x,y \rangle)$. One finds 
$$
\gamma(t) = (x,y) + t^{2}(-x^{2}y -yx^{2}, xy^{2}+y^{2}x) + t^{3} \gamma_{3} + \cdots + t^{15}\gamma_{15}
$$
where $\gamma_{3}= (x^{4} + x y^{3} + y x y^{2} +  y^{2}xy + y^{3}x, -x^{3}y - x^{2}y x - x y x^{2} -   yx^{3} - y^{4})$.
By what we said above, it is clear that the linear term $\gamma_{1} = 0$. The idea is now to construct another curve
$$
\sigma\colon \CC \to \Aut(\CC\langle x,y \rangle)
$$
as a product of certain $\exp(t\delta)$
such that $\sigma_{1}= (x^{2}y + yx^{2} + xyx, -xy^{2} - y^{2}x - yxy)$.
For the abelianization we then get  $\alpha(\gamma_{2}) = (-2x^{2}y, 2xy^{2})$ and $\alpha(\sigma_{1})= (3x^{2}y, -3xy^{2})$, hence $\alpha(3\gamma_{2}+ 2\sigma_{1})=0$, but $3\gamma_{2}+ 2\sigma_{1}= (-x^{2}y - y x^{2} + 2xyx, xy^{2}+ y^{2}x - 2yxy) \neq 0$. It follows that
$$
\mu(t):= \gamma(t)^{3}\sigma(t^{2})^{2} = (x,y) + t^{2} (3\gamma_{2}+2\sigma_{1}) + t^{3} \rho_{3} + \cdots
$$
defines a cuspidal curve $C:=\mu(\CC) \subseteq\Aut(\CC\langle x,y \rangle)$ isomorphic to \name{Neile}'s parabola, because $\mu_{2}=3\gamma_{2}+2\sigma_{1}\neq 0$ and $\mu_{3}=3\gamma_{3}\neq 0$. For the image in $\Aut(\CC[x,y])$ we find 
$$
\alpha(\mu(t)) = (x,y) + t^{3}(x^{4}+xy^{3},-y^{4}-3x^{3}y) + \cdots,
$$
and therefore the morphism $C \to \alpha(C)$  is not an isomorphism. 
It remains to construct the curve $\sigma\colon \CC \to \Aut( \kk \langle x,y \rangle)$.

For $b \in \kk$, set $\tau^{(b)}:= (x+by)^3 (b \ddx -  \ddy) \in \Der_{\kk}(\kk\langle x,y \rangle)$. This is a locally nilpotent derivation, 
$(\tau^{(b)})^{2}x = (\tau^{(b)})^{2}y = 0$, and so 
$$
\tau^{(b)}(t):=\exp(t\tau^{(b)})=(x + tb(x+by)^{3}, y - t(x+by)^{3})
$$ 
is a family of automorphisms of $\kk\langle x,y \rangle$. An easy calculation shows that 
\[
\begin{split}
\tau^{(b)}_{1} &= (0, - x^{3}) + b(x^{3},-x^{2}y-xyx-yx^{2}) \\
&+ b^{2}(x^{2}y + yx^{2} +xyx, - xy^{2}  - y^{2}x - yxy)\\
&+b^{3}(xy^{2} + yxy + y^{2}x, - y^{3}) + b^{4}(y^{3},0).
\end{split}
\]
If $\omega$ is  a primitive third root of unity, then we get
$$
\tau^{(1)}_{1} + \omega\tau^{(\omega)}_{1}+ \omega^{2}\tau_{1}^{(\omega^{2})} = 3(x^{2}y + yx^{2} + xyx, - xy^{2}  - y^{2}x - yxy),
$$
because $1 +\omega + \omega^{2}=0$. Thus the product
$$
\sigma(t):=\exp(\frac{t}{3}\tau^{(1)}) \cdot \exp(\frac{t}{3}\omega \tau^{(\omega)}) \cdot \exp(\frac{t}{3}\omega^{2} \tau^{(\omega^{2})})
$$
is a rational curve in $\Aut(k\langle x,y \rangle)$ with $\sigma_{1}=(x^{2}y + yx^{2} + xyx, - xy^{2}  - y^{2}x - yxy)$.

\begin{question}\label{bijective+isomorphism-on-Lie-algebras-imply-isomorphism.ques}
Let $\phi\colon\GGG \to \HHH$ be a bijective homomorphism of ind-groups, and assume that $\Lie\phi\colon \Lie\GGG \to \Lie\HHH$ is an isomorphism. Does this imply that $\phi$ is an isomorphism?
\end{question}

\newpage
\part{AUTOMORPHISMS OF AFFINE  \texorpdfstring{$n$}{n}-SPACE}

In this part we study the automorphism group of affine $n$-space $\An$. If $n=1$, then $\Aut(\Aone)$ is an algebraic group, namely the semi-direct product of $\kst$ with $\kplus$. In dimension 2, the group $\Aut(\Atwo)$ is an \itind{amalgamated product} of the affine group $\Aff(2)$ and the \name{de Jonqui\`eres} subgroup $\JJJ(2)$, an old theorem of 
\name{Jung}, \name{van der Kulk} and \name{Nagata}.
This structure has some important consequences. E.g. every algebraic subgroup of $\Aut(\Atwo)$ is conjugate to a subgroup of $\Aff(2)$ or of $\JJJ(2)$.

We will first give some general results about $\Aut(\An)$ which hold in any dimension $n\geq 2$, e.g. the connectedness, the infinite transitivity
of the natural action on $\AA^n$,
the structure of the \name{de Jonqui\`eres} subgroup $\JJJ(n)$, the approximation property by tame automorphisms
and a proof that all automorphisms of $\Aut(\An)$ are inner, the description of the Lie algebra $\Lie\Aut(\An)$, and finally a discussion of the locally finite, the semisimple and the unipotent elements of $\Aut(\An)$.

The second part is about $\Aut(\Atwo)$. For this group, thanks to the amalgamated product structure, we can classify the unipotent elements and the semisimple elements, and we show that an element is semisimple if and only if its conjugacy class is closed, a result due to \name{Furter-Maubach} \cite{FuMa2010A-characterization}.

The last section is about $\Aut(\Athree)$. It is known that this group is not tame.
For example, the \name{Nagata} automorphism is not in the subgroup generated by $\Aff(3)$ and $\JJJ(3)$ (\name{Shestakov-Umirbaev} \cite{ShUm2003The-Nagata-automor}). We will use this in the construction of a pair of closed connected subgroups $\GGG^{t} \subset \GGG \subset \Aut(\Athree)$ which have the same Lie algebras,
but are not equal. We will also describe a family of surfaces in $\Athree$ whose automorphism groups are discrete and infinite.

\pmed
\section{Generalities about \texorpdfstring{$\Aut(\AA^n)$}{Aut(An)}} 
\label{generalities-on-Aut(An).sec}

\subsection{The group of affine transformations}
An \itind{affine transformation} $\phi$ of a finite-dimensional $\kk$-vector space $V$ has the form
$$
\phi(v) = gv + w \ \text{ where } \ g \in \GL(V) \  \text{ and } \ w \in V.
$$
Equivalently, $\phi = \t_{w}\circ g$ where $\t_{w}$ denotes the translation $v \mapsto v+w$. We will write $\phi=(w,g)$.

We denote by $\Aff(V)$\idx{$\Aff(V)$} the {\it algebraic group
of affine transformations}, and by $\Tr(V) \subset \Aff(V)$\idx{$\Tr$} the closed normal subgroup of {\it translations}\idx{group of translations} which we identify with $V^{+}$.
It follows that  $\Aff(V) = \GL(V) \ltimes \Tr(V)$ is a semi-direct product where the action of $\GL(V)$ on $\Tr(V) = V^{+}$ is the obvious one:
$g\circ \t_{w} \circ g^{-1} = \t_{gw}$
Choosing a basis of $V$ we get a canonical closed immersion
$$
\Aff(n) \into \GL(n+1), \quad (a,A)\mapsto \begin{bmatrix} 1 & 0  \\ a & A \end{bmatrix}.
$$\idx{$\Aff(n)$}

The first part of the following result is easy and is left to the reader. For the second we use Corollary~\ref{kostant.cor} which says, as a consequence of \name{Kostant}'s Theorem~\ref{kostant.thm}, that  the set of unipotent elements $G^{u}$ in any reductive group $G$ forms a closed normal subvariety.

\begin{lemma} 
\be
\item The affine transformation $\phi=(w,g)$ has a fixed point if and only if $w \in (g-\id)(V)$.
\item The unipotent elements $\Aff^{u}(V) = \GL^{u}(V) \cdot \Tr(V)$ form a closed normal subvariety of $\Aff(V)$.
\ee
\end{lemma}

\ps
\subsection{Structure of the \texorpdfstring{de Jonqui\`eres}{de Jonquieres} subgroup} \label{triang.subsec}
The  \name{de Jonqui\`eres} subgroup $\JJJ(n)\subseteq \AutA{n}$
\idx{$\JJJ(n)$}  consisting of {\it triangular automorphisms}\idx{triangular automorphism} was defined in Section~\ref{Aut-An.subsec}:\idx{de Jon@\name{de Jonqui\`eres} subgroup}
$$
\JJJ(n):=\{\g=(g_{1},\ldots,g_{n})\mid g_{i}\in\CC[x_{i},\ldots,x_{n}]\text{ for }i=1,\ldots, n\}.
$$
For  $\g= (g_{1},g_{2},\ldots,g_{n})\in\JJJ(n)$ we have
$g_{i}=a_{i}x_{i}+p_{i}(x_{i+1},\ldots,x_{n})$ where $a_{i}\in\kst$, for $i=1,\ldots,n$. Conversely, every endomorphism of $\An$ of this form is an automorphism. 

Let $D(n)\subseteq \GL(n) \subseteq \AutA{n}$ denote the {\it diagonal automorphisms}\idx{diagonal automorphism}, $D(n)\simeq (\Cst)^{n}$.\idx{$\End$@$D(n)$}
There is a canonical split exact sequence of ind-groups\idx{$\JJJ^{u}(n)$}
\[
1  \to \JJJ^{u}(n) \to \JJJ(n) \overset{d}{\to} D(n) \to 1 \tag{$*$}
\]
where $d(\g) := (a_{1}x_{1},\ldots,a_{n}x_{n})$, is the diagonal part of $\g$. In particular, $\JJJ(n)$ is a semidirect product 
$\JJJ(n) = D(n) \ltimes \Ju(n)$. The kernel $\JJJ^{u}(n)$ is the subgroup of unipotent elements of $\JJJ(n)$. In fact, the kernel of the projection 
$\JJJ^{u}(n) \to \JJJ^{u}(n-1)$, $(g_{1},g_{2},\ldots,g_{n})\mapsto (g_{2},\ldots,g_{n})$, is commutative and consists of unipotent elements, and the claim follows by induction.

For a subgroup $G \subseteq \AutA{n}$ and an element $\h\in\AutA{n}$ we define the {\it $G$-conjugacy class of $\h$\/} by $C_{G}(\h):=\{\g\cdot\h\cdot\g^{-1} \mid \g\in G\}$. If $G$ is algebraic, then $C_{G}(\h)$ is a locally closed algebraic subset of $\Aut(\An)$. We shortly write $C(\h)$ for $C_{\AutA{n}}(\h)$.\idx{G@$G$-conjugacy class $C_{G}(\h)$}\idx{conjugacy class $C_{G}(\h)$}

\begin{proposition}\label{triang.prop}
Let $\g = (g_{1},g_{2},\ldots,g_{n}) \in \JJJ(n)$ be a triangular automorphism where $g_{i}=a_{i}x_{i}+ p_{i}(x_{i+1},\ldots,x_{n})$.
\be
\item \label{1-semisimple}
If $\g$ is semisimple, then there is an $\h\in\JJJ(n)$ of degree $\deg\h\leq  \prod_{i=1}^{n} \deg g_{i}$ such that $\h\cdot\g\cdot\h^{-1}= d(\g) = (a_{1}x_{1},\ldots,a_{n}x_{n})$.
\item \label{2-unipotent}
$\g$ is unipotent if and only if $a_{1}=a_{2}=\cdots=a_{n}=1$.
\item \label{3-triangular}
The closure $\overline{C_{D(n)}(\g)}$ contains $(a_{1}x_{1},\ldots,a_{n}x_{n})$ which is  $\JJJ(n)$-conjugate to $\g_{s}$.
\ee
\end{proposition}

\begin{corollary}\label{ssinclosure.cor}
If $\g \in \AutA{n}$ is triangularizable\idx{triangularizable}, i.e. conjugate to an element of $\JJJ(n)$, then $\g_{s}\in \wc{C(\g)}\subseteq\overline{C(\g)}$.
\end{corollary}

\begin{proof}
(\ref{1-semisimple}) Assume that  $\g$ is of the form 
$\g=(g_{1},\ldots,g_{k},a_{k+1}x_{k+1},\ldots,a_{n}x_{n})$
where $g_{j}=a_{j}x_{j}+p_{j}(x_{j+1},\ldots,x_{n})$ for $j=1,\ldots,k$. This means that $\g^{*}x_{\ell}=a_{\ell}x_{\ell}$ for $\ell>k$. Moreover, $\Span\{(\g^{*})^{m}x_{k}\mid m\in\ZZ\} = \kk x_{k}\oplus M$ where $M$ is the span of $\{(\g^{*})^{m}p_{k}\mid m\in \ZZ\}$, hence a finite-dimensional $\g^{*}$-stable subspace of $\kk[x_{k+1},\ldots,x_{n}]$ of polynomials of degree $\leq \deg p_{k}$. Since $\g^{*}$ is semisimple, we can find an eigenvector  of the form $x_{k}+q_{k}$ with $q_{k}\in M$. Replacing $x_{k}$ by this eigenvector we get $(\g^{*})x_{k}= a_{k}x_{k}$, and the claim follows by induction.
\ps
(\ref{2-unipotent}) We have seen above that $\JJJ^{u}(n) = \Ker d$ is the subgroup of unipotent elements.
\ps
(\ref{3-triangular}) Define inductively $m_{n}=1$ and $m_{i}:=\deg g_{i} \cdot m_{i+1} + 1$ for $i=n-1,n-2,\ldots,1$, and put $\lambda(t):= (t^{m_{1}}x_{1},t^{m_{2}}x_{2},\ldots,t^{m_{n}}x_{n})\in D(n)$. Then a simple calculation shows that $\lim_{t\to 0} \lambda(t)^{-1}\cdot \g\cdot \lambda(t)= (a_{1}x_{1}, a_{2}x_{2},\ldots,a_{n}x_{n})$, proving the first claim. From the exact sequence $(*)$ and (2) we get $d(\g) = d(\g_{s})$, and by (1) we know that $\g_{s}$ is $\JJJ(n)$-conjugate to $d(\g_{s})$, proving the second claim of  (\ref{3-triangular}).
\end{proof}

\begin{remark}\label{ffunipotent.rem}
For a unipotent $\g\in\JJJ^{u}(n)$, $\g = (x_{1}+p_{1}, x_{2}+p_{2},\ldots,x_{n}+ p_{n})$, the fixed point set $(\A{n})^{\g}$ is given by $p_{1}=p_{2}= \cdots =p_{n}=0$. It was shown by \name{Snow} that for $n\leq 3$ every fixed point free triangular automorphism of $\A{n}$ is conjugate to a translation \cite{Sn1989Unipotent-actions-}. In fact, by Proposition~\ref{freeunipotent.prop} below due to \name{Kaliman} \cite{Ka2004Free-C-actions-on-} this holds for every fixed point free unipotent automorphism of $\AA^{3}$. On the other hand, \name{Winkelmann} \cite[Section~5, Lemma~8]{Wi1990On-free-holomorphi} gave an example of a fixed point free triangular unipotent automorphism of $\AA^{4}$ which is not conjugate to a translation. In fact, the orbit space $\C^{4}/\C^{+}$ is not Hausdorff.
\end{remark}

\begin{example} Let $\g \in  \JJJ(n)$ be a triangular automorphism with diagonal part $d(\g)=(a_{1}x_{1},\ldots,a_{n}x_{n})$. Assume that $a_{1},\ldots,a_{n}$ are multiplicatively independent, i.e. they do not satisfy a nontrivial relation of the form $a_{1}^{k_{1}}a_{2}^{k_{2}}\cdots a_{n}^{k_{n}}=1$. Then $\g$ is semisimple and linearizable. In fact, $\g_{s}$ is $\JJJ(n)$-conjugate to $d(\g)$ by Proposition~\ref{triang.prop}(\ref{3-triangular}), and so $\overline{\langle \g_{s}\rangle}$ is a maximal torus conjugate to $D(n)$ which commutes with $\g_{u}$. Now the claim follows, because the centralizer of $D(n)$ in $\AutA{n}$ is $D(n)$.
\end{example}

We have seen above that $\g=(g_{1},\ldots,g_{n})$ belongs to $\JJJ(n)$ if and only if $g_{i}= a_{i}x_{i}+p_{i}(x_{i+1},\ldots,x_{n})$ for all $i$ where $a_{i}\in\kst$. This shows that the \name{de Jonqui\`eres} subgroup $\JJJ(n)\subseteq\AutA{n}$ is, as an ind-variety, isomorphic to 
\[ 
(\Cst)^{n}\times (\kk\oplus\kk[x_{n}] \oplus \kk[x_{n-1},x_{n}] \oplus \cdots \oplus \kk[x_{2},\ldots,x_{n}]).
\]
In addition, $\JJJ(n)$ is a nested ind-group (Example~\ref{nested.exa}) as we see from the following result.

\begin{proposition}  \label{jonq.prop}
For all $d\geq 0$ we have $\langle \JJJ(n)_{d}\rangle \subseteq \JJJ(n)_{d^{n-1}}$. In particular, $\overline{\langle \JJJ(n)_{d}\rangle} \subseteq \AutA{n}$ is a closed algebraic subgroup and $\JJJ(n) = \bigcup_{d}\overline{\langle \JJJ(n)_{d}\rangle}$.
\end{proposition}

\begin{proof}
Let $d>1$ be an integer. Define a degree function $\deg_{d}$ on $\kk [ x_1, \ldots, x_n ]$ by setting $\deg_{d} x_{i}:= d^{n-i}$. Set
\[ M_{d}:=\{\g=(g_{1},\ldots,g_{n})\in\JJJ(n) \mid \deg_{d} g_{i} \leq d^{n-i}\}.\]
Note that we have $\JJJ(n)_d \subseteq M_d \subseteq \JJJ(n)_{d^{n-1}}$ and that 
$M_{d}$  is a closed algebraic subset of $\JJJ(n)$. It is easy to see that $M_{d} \cdot M_{d}\subseteq  M_{d}$. Therefore, $M_{d}$ is a closed algebraic subgroup of $\JJJ(n)$ (use Lemma~\ref{well-known.lem}), and the claim follows.
\end{proof}

Finally, let us mention the following results about the conjugacy classes in $\Ju(n)$.
\begin{proposition}\label{conj-class-Jn.prop}
For $c \in \kk$, let $\t_{c}\in\Ju(n)$ denote the translation
$$
\t_{c}:=(x_{1},\ldots,x_{n-1},x_{n}+c)
$$
\be
\item $C_{\JJJ(n)} (\t_{1}) = \{\u=(u_{1},\ldots,u_{n})\in\Ju(n) \mid u_{n}= x_{n}+c \text{ where }c\in\kst\}$, and this conjugacy class is open and dense in $\Ju(n)$.
\item For $c\neq 0$ we get $C_{\Ju(n)} (\t_{c}) = \{\u=(u_{1},\ldots,u_{n})\in\Ju(n) \mid u_{n}= x_{n}+c\}$, and this conjugacy class is closed in $\Ju(n)$.
\ee
In particular,  the weak closure $\wc{C(\t_{1})}$ contains $\Ju(n)$.
\end{proposition}

\begin{proof} 
(1) Let $\mu=\mu_{\u}\colon \kplus \to \Aut(\An)$ be the homomorphism associated to $\u$, i.e. $\mu(1)=\u$. Then the last coordinate of $\mu(s)$ is $x_{n}+sc$. Define $\phi\colon \An \to \An$ by
$$
\phi(a_{1},\ldots,a_{n}):= \mu(a_{n})(a_{1},\ldots,a_{n-1},0).
$$ 
This is an automorphism, with inverse 
$\phi^{-1}(b_{1},\ldots,b_{n}):=\mu(\frac{-b_{n}}{c})(b_{1},\ldots,b_{n-1},0)$, and the construction shows that $\phi\in\JJJ(n)$.  

We have the formal identity (as polynomials in $x_{1},\ldots, x_{n},s,z$)
$$
\mu(x_{n}+s  )(x_{1},\ldots,x_{n-1},z) = \mu(s)(\mu(x_{n})(x_{1},\ldots,x_{n-1},z)).
$$
Setting $z=0$ we get $\phi(a_{1},\ldots,a_{n}+s) = \mu(s)(\phi(a_{1},\ldots,a_{n}))$, hence 
$$
\phi^{-1}\circ \mu(s) \circ \phi = \t_{s}\quad \text{and so}\quad \phi^{-1}\circ \u \circ \phi = \t_{1}.
$$
It is clear that the last coordinate of any $\JJJ(n)$-conjugate of $\t_{1}$ has the form $x_{n}+ c$, $c\neq 0$, and the claim follows.
\ps
(2)
We have seen above that $\phi^{-1}\circ \mu(s) \circ \phi = \t_{s}$ for some $\phi \in\JJJ(n)$. We can write $\phi = \phi_{0} \circ d$ where $\phi_{0}\in\Ju(n)$ and $d = (d_{1}x_{1},\ldots,d_{n}x_{n})$, $d_{i}\in \kst$. It then follows that $\phi_{0}^{-1}\circ \mu(s) \circ \phi_{0} = 
d\circ \t_{s}\circ d^{-1}=\t_{d_{n}s}$. On the other, conjugation with an element from $\Ju(n)$ does not change the last coordinate, hence $d_{n}=c$ and $\phi_{0}^{-1}\circ \u \circ \phi_{0} = \t_{c}$. The remaining claims are clear.
\end{proof}

\ps
\subsection{The degree formula}\label{degree-formula.subsec}
The \itind{degree} of an automorphism $\f=(f_1,\ldots,f_n)$ of $\An$ is defined as
$$
\deg \f :=\max_i \{\deg f_i\}.
$$
We want to give a short proof of the following \itind{degree formula} (see \cite[Corollary~1.4]{BaCoWr1982The-Jacobian-conje}):
\[
\deg \f^{-1} \leq (\deg \f)^{n-1}.
\]
We start with an easy lemma. Recall that the {\it degree of a closed subvariety $X \subset \An$}\idx{degree of subvariety} of dimension $d$ is defined as 
$$
\deg X := \#(X \cap A) 
$$
where $A$ is an affine subspace of $\An$ of codimension $d$ in general position, see \cite[Section~8]{DeKr1997Constructive-invar}. Note that for a hyperplane $H \subseteq \An$ we get $\deg \f^{-1}(H)\leq \deg \f$ with equality for a generic $H$.
\begin{lemma} 
Let $H_{1},\ldots,H_{r} \subseteq \An$ be hypersurfaces,  $\deg H_{i}= d_{i}$. Put $X:=\bigcap_{i} H_{i}$. If $\codim_{\An}X = r$, then 
$\deg X\leq d_{1}d_{2}\cdots d_{r}$. In particular, if $r = n$ and $X$ is finite, then $|X| \leq d_{1}\cdots d_{n}$.
\end{lemma}
\begin{proof}
By induction, it suffices to proof the following statement.

 {\it Let $X \subseteq \An$ be closed and equidimensional of dimension $s$, and let $H \subseteq \An$ be a hypersurface such that $\dim ( H \cap X ) < s$. Then $\deg (H \cap X) \leq \deg X \deg H$.}
 
 We first remark that this statement is clear if $X$ is a curve.
For a generic affine subspace $A \subseteq \An$ of codimension $\dim X - 1$ we get
 \be
 \item $\deg (A \cap H) = \deg H$;
 \item $\deg (A \cap X) = \deg X$;
 \item $A \cap X \cap H$ is finite of cardinality equal to $\deg (X \cap H)$.
 \ee
 Now $X':=X \cap A \subseteq A$ is a curve, $H':=H \cap A \subseteq A$ is a hypersurface, and so 
 $$
 \deg (X\cap  H) = |X' \cap H'| \leq \deg X' \deg H' = \deg X \deg H,
 $$
proving the statement above.
\end{proof}
Now we can prove the degree formula above. Let $E \subseteq \An$ be a hyperplane so that $\deg \f(E) = \deg \f^{-1}$.  Choose generic hyperplanes $E_{1},\ldots, E_{n-1} \subseteq \An$ such that the intersection $E_{1}\cap\cdots\cap E_{n-1}\cap \f(E)$ is finite of cardinality equal to $\deg \f(E)$. Then, by the lemma above,
$$
\deg \f^{-1} = |\f^{-1}(E_{1})\cap\cdots\cap \f^{-1}(E_{n-1})\cap E| \leq \deg \f^{-1}(E_{1})\cdots\deg \f^{-1}(E_{n-1}) \leq (\deg\f)^{n-1}.
$$

\ps
\subsection{Connectedness}
Let $C$ be an irreducible curve and $\VVV$ an ind-variety. A morphism $\phi \colon C \to \VVV$ is said to {\it connect $v_1, v_2 \in \VVV$} if $v_{1},v_{2}\in \phi(C)$.

\begin{proposition}  \label{k*-connected.prop}
For  $\g,\h\in\AutA{n}$ where $\deg \g\geq \deg \h$ there is a morphism $\phi\colon\Cst \to \AutA{n}$ connecting $\g$ and $\h$ such that $\deg\phi(t) \leq \deg \g\,(\deg\h)^{n}$ for all $t\in\Cst$. 
\end{proposition}

\begin{proof}
Let $\g \in\AutA{n}$. Set $\g(0)=a$. Then $\g_{0}:=\t_{-a}\cdot\g=(g_{1},\ldots,g_{n})$ fixes the origin where $\t_{b}$ denotes the translation $x\mapsto x+b$. Define 
$$
\g(t):= (t \id)^{-1} \cdot(\t_{t^{2}a} \cdot \g_{0})\cdot (t \id)
$$
for $t\in\Cst$. Then $\g(1)=\g$ and $\lim_{t\to 0}\g(t)=d_{0}\g_{0}=(\ell_{1},\ldots,\ell_{n})$ where $\ell_{i}$ is the linear part of $g_{i}$. This shows that there is a morphism $\mu\colon \kk \to \AutA{n}$ such that $\mu(1)=\g$ and $\mu(0)=A\in\GL(n)$. Note that $\mu(\kk)\subseteq \AutA{n}_{k}$ where $k=\deg\g$. It is well-known that for every $A\in\GL(n)$ there is a morphism $\nu\colon \Cst \to \GL(n)$ such that $\nu(1)=E_{n}$ and $\nu(-1)=A$. Then $\phi(t):=\mu(\frac{t+1}{2})\cdot \nu(t)^{-1}$ defines a morphism $\varphi \colon \kk^* \to \Aut (\AA^n)$ of degree $\leq \deg\g$ and  we have $\varphi (1) = \g$, $\varphi (-1) = E_n$.
Now replace $\g$ with $\h^{-1}\cdot \g$ and multiply $\phi(t)$ with $\h$ to find a morphism $\phi\colon\Cst \to \AutA{n}$ of degree $\leq \deg (\h^{-1}\cdot\g)\deg\h \leq \deg\g(\deg\h)^{n}$ connecting $\g$ with $\h$.
\end{proof}

The next result is an immediate consequence of the proposition above. It can be found in \cite[Lemma~4]{Sh1981On-some-infinite-d}.
\begin{corollary}
The group $\AutA{n}$ is curve-connected.
\end{corollary}

\begin{remark}
If $\phi \colon \AA^1 \to \Aut (\AA^n)$ connects $\g, \h \in \Aut (\AA^n)$, then we necessarily have $\jac (g) = \jac (h)$, because the morphism $\jac \circ \phi \colon \AA^1 \to \kk^*$ has to be constant. Therefore, Proposition~\ref{k*-connected.prop} is ``optimal'' in the sense that in general two elements of $\Aut (\AA^n)$ cannot be connected by a morphism $\phi \colon \AA^1 \to \Aut (\AA^n)$.
\end{remark}

\ps
\subsection{Infinite transitivity of \texorpdfstring{$\Aut(\An)$}{Aut(An)}}
Recall that an action of group $G$ on a space $X$ is called \itind{infinitely transitive} if it is $n$-transitive for all $n \geq 1$. Equivalently, for every finite subset $F\subseteq X$ the pointwise stabilizer $G_{F}$ of $F$ is transitive on the complement $X\setminus F$. It is known that $\AutA{n}$ acts infinitely transitively on $\A{n}$ for $n\geq 2$; this can be found in  \cite{KaZa1999Affine-modificatio} together with generalizations to other affine varieties. It is also a consequence of a more general result which was proved by \name{Arzhantsev-Flenner-Kaliman et al}, see \cite{ArFlKa2013Flexible-varieties}. 

If $\rho$ is a $\kplus$-action on $\An$ and $f \in \OOO(\An)^{\kplus}$ an invariant, then we have defined the {\it modification $\rho_{f}$ of $\rho$} in Section~\ref{modification.subsec}: $\rho_{f}(s)(x) := \rho(f(x)s)(x)$.

\begin{example}\label{translation.exa}
Let $\mu$ be the $\kplus$-action on $\A{n}$ by translation, with translation vector $v \neq 0$, i.e. $\mu(s)(a) = a+sv$. Let $A\subseteq \A{n}$ be a subset and $b\in\A{n}\setminus\overline{\kplus A}$. Then there is a modification $\mu'$ of $\mu$ with the following properties:
\be
\item $A$ is fixed by $\mu'$;
\item The orbit of $b$ under $\mu'$ is the line $b+\kk v$.
\ee
In fact, the affine quotient by the $\kplus$-action $\mu$ is given by a linear map $p\colon \A{n} \to \AA^{n-1}$ with kernel $\kk v$, and the fibers of $p$ are the orbits. By assumption, $p(b)\notin \overline{p(A)}$, and so there is a $\mu$-invariant $f$ which vanishes on $A$ and is nonzero on $b$. Then the modification $\mu':= \mu_{f}$ has the required properties.
\end{example}

\begin{proposition}  \label{inftrans.prop}
Let $n \geq 2$, and 
let $\UUU \subseteq \AutA{n}$ be the subgroup generated by the modifications of the translations. Then $\UUU$ acts infinitely transitively on $\A{n}$.
\end{proposition}

\begin{proof}
(a) Let $F \subseteq \A{n}$ be a finite subset and let $b,b'\in\A{n}\setminus F$. If the line $\ell:=\overline{bb'}=b+\kk v$ does not meet $F$, then Example~\ref{translation.exa} above shows that there is a modification of the translation $\t_{v}$ which  fixes $F$ and maps $b$ to $b'$. 
\ps
(b) If the line $\ell$ meets $F$, then we can choose a third point $c\in\A{n}$ such that the two lines $\overline{bc}$ and $\overline{b'c}$ do not meet $F$, and the claim follows from (a).
\end{proof}
Recall that the \itind{special} (or \itind{unimodular}) automorphism group is defined as 
$$
\SAut(\An):= \{\phi\in\Aut(\An) \mid \jac(\phi)=1\} = \Ker(\jac\colon \Aut(\An) \to \kst),
$$
see Section~\ref{Aut-An.subsec}.
\begin{remark}  \label{roots.rem}
The \itind{root groups} $\alpha_{ij}(s)=(x_{1},\ldots,x_{i}+sx_{j},\ldots,x_{n})\subseteq \GL(n)$, $i\neq j$, are modifications of the translations $\mu(s)=(x_{1},\ldots,x_{i}+s,\ldots,x_{n})$. This implies that the group $\UUU$ of the proposition above contains $\SL(n)$ and the unipotent triangular subgroup $\Ju(n)$, hence all tame automorphisms with jacobian determinant equal to 1. Thus, for $n=2$ we have $\UUU = \SAutA{2}$.\idx{$\alpha_{ij}$}
\end{remark}

\begin{question} \label{SAut(An).ques} 
Do we have $\UUU = \SAutA{n}$ for all $n\geq 2$, or at least $\overline{\UUU} = \SAutA{n}$?
\end{question}

\ps
\subsection{Approximation property}

Denote by $\mm:=(x_{1},\ldots,x_{n})\subseteq \kk[x_{1},\ldots,x_{n}]$ the homogeneous maximal ideal. 
For $\g,\h \in\EndA{n}$, $\g=(g_{1},\ldots,g_{n})$ and $\h=(h_{1},\ldots,h_{n})$, we define
$$
\g \equiv \h \pmod{\mm^{d}} \ :\iff \   g_{i}-h_{i}\in\mm^{d} \text{ for all }i,
$$ 
i.e., if the homogeneous terms of  $\g$ and $\h$ of degree $<d$ coincide. If $\vbold  \equiv \wbold \pmod{\mm^{d}}$ and $\vbold (0) = 0$, then one easily sees that $\g \cdot \vbold \equiv \h \cdot \wbold \pmod{\mm^{d}}$.
Define 
$$
\AutA{n}^{(d)}:=\{\g=(g_{1},\ldots,g_{n})\in\AutA{n} \mid g_{i}\in x_{i}+\mm^{d+1}\}.
$$
Thus $\AutA{n}^{(0)}=\{\g\in\AutA{n}\mid\g(0) = 0\}$ and $\AutA{n}^{(1)}=\{\g\in\AutA{n}\mid \g\equiv\id\pmod{\mm^{2}}\}$.
Then, for $\g\in\End(\A{n})$ and $\h \in \AutA{n}$, we have 
$$
\g \equiv \h \pmod{\mm^{d+1}} \quad \iff \quad \g\cdot\h^{-1}\in\AutA{n}^{(d)}.
$$
Recall that the \itind{tame automorphism group} of $\AA^n$ is defined as the subgroup generated by the {\it affine transformations} and the {\it triangular automorphisms} (see Section~\ref{Aut-An.subsec}):\idx{$\Tame(\An)$}
$$
\Tame(\An) := \langle\Aff(n),\JJJ(n)\rangle.  
$$ 
The following ``approximation result'' is due to \name{Anick} \cite{An1983Limits-of-tame-aut}.\idx{approximation}

\begin{proposition} \label{Anick.prop}
Let $\f \in\AutA{n}$. For any $d\in\NN$ there is a tame automorphism $\h^{(d)}\in\AutA{n}$ such that $\f \equiv \h^{(d)} \pmod{\mm^{d+1}}$.
\end{proposition}

\begin{proof}
For $d=0$, it is enough to take $\h^{(0)}$ equal to the affine part of $\f$. By induction we can assume that $\g\equiv\id \pmod{\mm^{d}}$, $d>1$. We have to show that there is a tame automorphism $\h$ such that 
$\g\equiv \h \pmod{\mm^{d+1}}$. If $\f\in\End(\A{n})$ and $\f\equiv\id \pmod{\mm^{d}}$ we denote by $\f_{d} \in (\kk[x_{1},\ldots,x_{n}]_{d})^{n}$ the homogeneous part of $\f$ of degree $d$, i.e. $\f\equiv\id+\f_{d}\pmod{\mm^{d+1}}$. 
We will identify $(\kk[x_{1},\ldots,x_{n}]_{d})^{n}$ with the $\GL(n)$-module $V:=\kk^{n}\otimes \kk[x_{1},\ldots,x_{n}]_{d}$, with the obvious action given by $g(v\otimes p) := gv \otimes g^* p$.  

If $\f,\tilde\f$ are both  $\equiv \id \pmod { \mm^{d} }$ and $t \in \kk^*$, we have
\[ (\f\cdot\tilde\f)_{d}= \f_{d}+\tilde\f_{d} \quad \text{and} \quad \left( (t \id)^{-1} \cdot \f \cdot (t \id) \right)_d =t^{d-1} \f_d.\]
Moreover, for $g \in\GL(n)$, we have $(g\cdot \f \cdot g^{-1})_{d} = g\,\f_{d}$. This shows that the homogeneous parts $\f_{d}$ of degree $d$ of the automorphisms $\f\equiv\id\pmod{\mm^{d}}$ form a $\GL(n)$-submodule $V'$ of $V$, and those corresponding to tame automorphisms form a submodule $V_{t}\subseteq V'$. We have to show that these two submodules coincide.

If an endomorphism $\f\equiv\id\pmod{\mm^{d}}$, has a constant jacobian determinant $\jac (\f)\neq 0$, 
then $\jac (\f)=1$, and the homogeneous part of degree $d$ of $\sum_{i} \frac{\partial f_{i}}{\partial x_{i}}$  vanishes. Therefore, 
$\f_{d}$ belongs to the kernel of the map $p_{d}\colon V \to \kk[x_{1},\ldots,x_{n}]_{d-1}$ given by $v\otimes f \mapsto \partial_{v}f$. This map is a $\GL(n)$-homomorphism, and its kernel is irreducible, by \name{Pieri}' formula (see Lemma~\ref{ker-Div.lem} below). Thus, $V_{t}=V'
=\Ker p_{d}$.
\end{proof}

This result has some interesting consequences. E.g. it was used by \name{Belov-Kanel} and \name{Yu}  to show that every automorphism of the ind-group $\AutA{n}$ is inner (see \cite{BeYu2012Lifting-of-the-aut}, cf. \cite{KrSt2012On-Automorphisms-o}). 

\begin{theorem}[\name{Belov-Kanel-Yu}]\label{Belov.thm}
Let $\phi$ be an automorphism of $\AutA{n}$ as an ind-group. Then $\phi$ is inner.
\end{theorem}
\begin{proof}
It is shown in \cite{KrSt2012On-Automorphisms-o} that for every automorphism $\phi$ of $\AutA{n}$ there is a $\g\in\AutA{n}$ such that $\g\cdot\phi(\h)\cdot\g^{-1}=\h$ for all tame automorphisms $\h$. Thus we can assume that $\phi$ is the identity on the tame automorphisms.

Next we claim that $\phi(\AutA{n}^{(d)}) = \AutA{n}^{(d)}$ for all $d\in\NN$. For a non-constant morphism $\mu\colon \kk \to X$ we have
$\mu^{*}(\mm_{\mu(0)}) = (t^{k})$ for some $k\geq 1$; we denote the exponent $k$ by $\sigma(\mu)$. For any $\g\in\AutA{n}$ consider the map $\lambda_{\g}\colon\Cst \to \AutA{n}$, $\lambda_{\g}(t):=(t^{-1}\id)\cdot\g\cdot (t\id)$. Then the subgroups $\AutA{n}^{(d)}$ for $d>0$ have the following description:
$$
\AutA{n}^{(d)} = \{\g\in\AutA{n} \mid  \lim_{t\to 0} \lambda_{\g}=\id \text{ and }\sigma(\lambda_{\g})\geq d\}
$$
Now the $\phi$-invariance of these groups follows. In fact, both conditions, $\lim_{t\to 0} \lambda_{\g}=\id$  and $\sigma(\lambda_{\g})\geq d$, are  invariant under $\phi$ since $\phi\circ\lambda_{\g}=\lambda_{\phi(\g)}$ and $\sigma(\lambda_{\phi(\g)})=\sigma(\lambda_{\g})$. 

To finish the proof we use Proposition~\ref{Anick.prop} which implies that for any $\g\in\AutA{n}$  we have $\g\equiv\phi(\g) \pmod {\mm^{d}}$ for all $d\in\NN$. Hence $\phi=\id$.
\end{proof}  

\ps
\subsection{The Lie algebra of \texorpdfstring{$\Aut(\An)$}{Aut(An)} and divergence} \label{AutAn-VF.subsec}

We will now determine the Lie algebras $\Lie \Aut(\An)$ and $\Lie \SAut(\An)$. Recall that $\SAut(\An)$ is the closed normal subgroup of $\Aut(\An)$ defined by 
$$
\SAut(\An):= \{\phi\in\Aut(\An) \mid \jac(\phi)=1\} = \Ker(\jac\colon \Aut(\An) \to \kst).
$$
The \itind{divergence} of a vector field $\delta = \sum_{i=1}^{n}p_{i}\frac{\partial}{\partial x_{i}}$ on $\A{n}$ is defined by $\Div \delta := \sum_{i=1}^{n}\frac{\partial p_{i}}{\partial x_{i}}$\idx{$\Div\delta$}. We have a canonical identification (see Remark~\ref{canonical-iso-VF.rem})
$$
\VEC (\AA^n) \simto \kk^{n}\otimes \kk[x_{1},\ldots,x_{n}] = \bigoplus_{m=0}^{\infty} \kk^{n}\otimes \kk[x_{1},\ldots,x_{n}]_{m}.
$$
given by $\sum_{i}f_{i}\ddxi \mapsto \sum_{i}e_{i}\otimes f_{i}$, and this isomorphism is equivariant with respect to the linear action of $\GL(n)$.
The divergence induces  surjective linear and $\GL(n)$-equivariant maps
\[ 
\Div_{m}\colon \kk^{n} \otimes \kk[x_{1},\ldots,x_{n}]_{m} \onto \kk[x_{1},\ldots,x_{n}]_{m-1}, 
\quad v\otimes f \mapsto \partial_{v}f=\sum_{i}a_{i}\frac{\partial f}{\partial x_{i}}
\]
where $v = (a_{1},\ldots,a_{n})\in \kk^{n}$.
Denoting its kernel by $M_m$, we get $\Ker \Div =\bigoplus_{m=0}^{\infty} M_{m}$, and we obtain
the following exact sequence of $\GLn$-modules:
\[
0 \to M_{m} \into  \kk^{n} \otimes \kk[x_{1},\ldots,x_{n}]_{m} \overset{\Div_{m}}{\longrightarrow}  \kk[x_{1},\ldots,x_{n}]_{m-1} \to 0.
\]
The next result is well known (see \name{Pieri}'s formula in \cite[Chap. 9, section 10.2]{Pr2007Lie-groups}).

\begin{lemma} 
\be
\item The $\SL_n$-modules $M_m$, $m \geq 0$, are simple and pairwise non-isomorphic.
\item The $\SL_n$-modules $\kk [x_1, \ldots,x_n]_{m}$, $m \geq 0$, are simple and pairwise non-isomorphic.
\ee
\end{lemma} \label{ker-Div.lem}
Now we can give the description of the Lie algebras $\Lie\Aut(\An)$ and $\Lie\SAut(\An)$. We use the canonical injection  $\xi\colon\Lie\Aut(X) \to \VEC(X)$, $A\mapsto \xi_{A}$, see Definition~\ref{xi_A.def} and Proposition~\ref{Liealg-VF.prop}.

\begin{proposition}  \label{LieAlgAutAn.prop}
The map $\xi$ induces the following anti-isomorphisms of Lie algebras:
\be
\item $\Lie\Aut(\An) \simto \VEC^{c}(\An):=\{\delta \in \VEC (\An) \mid \Div\delta \in \kk\}$;
\item $\Lie \SAut(\An) \simto \VEC^{0}(\An):=\{\delta \in \VEC (\An) \mid \Div \delta  = 0 \}$.
\ee
Moreover, $\Lie\overline{\Tame(\An)} = \Lie\Aut(\An)$.
\end{proposition}

\begin{proof}
(1)
By definition, an endomorphism $\f = (f_{1},\ldots,f_{n}) \colon \An \to \An$ is {\it \'etale}\idx{et@\'etale} if its \itind{jacobian determinant} $\jac(\f ):=\det\left(\frac{\partial f_{i}}{\partial x_{j}}\right)_{i,j}$ is a nonzero constant. Denote by $\Et(\An)$ the semigroup of \'etale endomorphisms. We have the following inclusions
$$
\Aut(\An) \underset{\text{\it \tiny closed}}{\subseteq} \Et(\An) \underset{\text{\it \tiny closed}}{\subseteq} \Dom(\An) 
\underset{\text{\it \tiny open}}{\subseteq} \End(\An).
$$
In fact, $\Dom(\An) = \jac^{-1}(k[x_{1},\ldots,x_{n}]\setminus \{0\})$ and $\Et(\An)=\jac^{-1}(\kk\setminus\{0\})$, and we have already seen that $\Aut(\An)$ is closed in $\Dom(\An)$ (Theorem~\ref{AutX-locally-closed-in-EndX.thm}). We get
$$
\Lie\Aut(\An) \subseteq T_{\id}\Et(\An) \subseteq T_{\id} \Dom(\An) = T_{\id} \End(\An) = \End(\An).
$$
The jacobian map $\jac\colon \End(\An) \to \kk[x_{1},\ldots,x_{n}]$ is an ind-morphism, and we have
$$
\jac(\id+\eps g) = 1 + \eps \sum_{i}\ab{g_{i}}{x_{i}}\mod \eps^{2}.
$$
This shows that the differential of $\jac$  is the divergence:
$$
d\jac_{\id}=\Div\colon T_{id}\End(\An) = \End(\An) \to T_{1}\kk[x_{1},\ldots,x_{n}] = \kk[x_{1},\ldots,x_{n}],
$$
hence $\xi(\Lie\Aut(\An)) \subseteq \VEC^{c}(\An)$ and $\xi(\Lie\SAut(\An))\subseteq \VEC^{0}(\An)$.

For the other inclusion we consider the Lie subalgebra $L \subseteq \Lie \Aut(\An)$ generated by the Lie algebras $\Lie\Aff(n)$ and $\Lie\JJJ(n)$ of the affine and the \name{de Jonqui\`eres} group and show that $\xi(L) = \VEC^{c}(\An)$. Since $L\subseteq \overline{\Lie\Tame(\An)}$ this proves also the last claim.

In fact, $\xi(L) \subseteq \VEC^{c}(\An)$ is stable under $\GLn$ and  contains nonzero elements in every degree.  Hence,  $\xi(L) \supseteq \Ker\Div=\bigoplus_{m\geq 0}M_{m}$ by the previous Lemma~\ref{ker-Div.lem}, and the claim follows, because $\VEC^{c}(\An)=\Ker\Div\oplus \kk(\sum_{i}x_{i}\frac{\partial}{\partial x_{i}})$, and 
$\kk(\sum_{i}x_{i}\frac{\partial}{\partial x_{i}}) = \Lie(\kst\id)$.
\ps
(2)
As in (1) let $L'$ be the Lie subalgebra of $\Lie \SAut(\An)$ generated by the Lie algebras of $\Aff(n)\cap\SAut(\An)$ and of $\JJJ(n) \cap\SAut(\An)$. Then $\xi(L')$ is stable under $\SLn$ and contains nonzero elements in every degree, hence one gets as above that 
$\xi(L')$ contains $\bigoplus_{m}M_{m}=\VEC^{0}(\An)$.
\end{proof}

\begin{remark}  \name{Kraft} and \name{Regeta} recently showed that the automorphism groups of the Lie algebras $\VEC(\An)$, $\VEC^{d}(\An)$ and $\VEC^{0}(\An)$ are all canonically isomorphic to $\Aut(\An)$, see \cite{KrRe2013A-Note-on-the-auto}. Using this result, one gets a different proof of Theorem~\ref{Belov.thm} above and a generalization to $\SAut(\An)$, see \cite{Kr2017Automorphism-group}. 
\end{remark}

\begin{corollary} 
If $\delta \in \VEC (\AA^n) $ is locally nilpotent, then $\Div(\delta)=0$, i.e. $\delta \in \VEC^{0}(\An)$.
\end{corollary}

\begin{proof}
Let $\mu\colon \kplus \to \Aut(\An)$ be the $\kplus$-action corresponding to $\delta$. Then 
the morphism $\jac\circ\mu\colon \kk^+ \to \kk^*$ is a character, hence trivial. Therefore,  $\mu\colon \kk^+ \to \Aut ( \AA^n)$, has values in $\SAut ( \AA^n)$, and so $\delta=\xi(d\mu_{0}(1)) \in \VEC^{0}(\An)$. 
\end{proof}

\begin{corollary}\label{tame-stable.cor}
If $\Aut(\An)$ acts on an ind-variety $\VVV$,
and if a closed subvariety $\WWW \subset \VVV$ is stable under $\Tame(\An)$, then $\WWW$ is stable under $\Aut(\An)$.
\end{corollary}

\begin{proof}
Since $\WWW$ is stable under $\Tame(\An)$ it is also stable under $\overline{\Tame(\An)}$, hence invariant under $\Lie\overline{\Tame(\An)}=\Lie\Aut(\An)$. Now the claim follows from Proposition~\ref{G-stable-is-LieG-invariant.prop}.
\end{proof}

\ps
\subsection{Families of locally finite automorphisms}
We start with a reformulation of some previous results  in term of families. 
Then, we address the particular case of families of automorphisms $(\Phi_{x})_{x\in X}$ such that all $\Phi_x$ are conjugate on a dense open set $ \subset X$. We then conclude the section showing in Corollary~\ref{ss-weakly-closed.cor} that the conjugacy class of a diagonalizable automorphism is weakly closed.

The first example shows  that for a family $(\Phi_{x})_{x\in X}$ of locally finite automorphisms of $\An$ one cannot conclude that $\Phi$, as an automorphism of $X\times \A{n}$, is locally finite.

\begin{example}
Consider the family $\Phi=([\begin{smallmatrix} t & 1\\ 0 & 1 \end{smallmatrix}])_{t \in  \Cst }$ of linear automorphisms of $\AA^{2}$. Clearly, every member is locally finite, but $\Phi$ as an automorphism of $\Cst \times \AA^{2}$ is not locally finite.
\end{example}

We have seen in Proposition~\ref{Glf-weakly-closed.prop} that the locally finite automorphisms form a weakly closed subset $\AutlfA{n} \subseteq \AutA{n}$, and the same holds for the unipotent automorphisms $\AutUA{n} \subseteq \AutA{n}$ (Lemma~\ref{unipotent-weakly-closed.lem}). In terms of families this means the following.

\begin{lemma} \label{re-interpretation-of-some-results-of-part-2.lem}
Let $\Phi=(\Phi_{x})_{x\in X}$ be a family of automorphisms of $\A{n}$.
If $\Phi_{x}$ is locally finite, resp. unipotent, on a dense open set $U \subseteq X$, then all $\Phi_{x}$ are locally finite, resp. unipotent.
\end{lemma}
Of course, we would like to know if ``locally finite (resp. unipotent) on a Zariski-dense set'' suffices to get the result, because this would imply that $\AutlfA{n}$ resp. $\AutU(\An)$ are closed in $\Aut(\An)$.  We can prove this only for $n=2$, see Theorem~\ref{constr.thm} below. 
\ps
If we assume in addition that all $\Phi_{x}$ are conjugate on a dense open set, then we can say more. 

\begin{proposition}  \label{fam.prop}
Let $\Phi=(\Phi_{x})_{x\in X}$ be a family of automorphisms of $\A{n}$. Assume that $\Phi_{x}$ is conjugate to a fixed $\g$ on
a dense open set $U\subseteq X$.
\be
\item
If $\g$ is locally finite, then $\Phi$ is a locally finite automorphism of $X\times\A{n}$.
\item 
If $\g$ is unipotent, then $\Phi$ is a unipotent automorphism of $X\times\A{n}$.
\item 
If $\g$ is semisimple, then  $\Phi$ is a semisimple automorphism of $X\times\A{n}$, and so all $\Phi_{x}$ are semisimple.
\ee
\end{proposition}

\begin{proof} By base change, we can assume that $\kk$ is uncountable. We can also assume that $X$ is irreducible and affine. 
\ps
(a) Denote by $\phi\colon X \to \AutA{n}$ the morphism $x\mapsto \Phi_{x}$, and let $\gamma\colon \AutA{n} \to \AutA{n}$ be the conjugating morphism $\h\mapsto \h^{-1}\circ\g\circ\h$. The image $\phi(U)$ is a constructible subset of $\AutA{n}$ which is, by assumption, contained in  $C(\g)=\gamma(\AutA{n})$. Therefore, $\phi(U)$ is covered by the images $\gamma(\AutA{n}_{k})$. Hence, by Lemma~\ref{constructible.lem}, there is a $m>0$ such that $\gamma(\AutA{n}_{m})\supseteq \phi(U)$. Denote by $Y$ the (reduced) fiber product
$$
\begin{CD}
Y & @>\mu>> & \AutA{n}_{m}\\
@V\gamma_{X}VV && @VV{\gamma}V \\
X & @>\phi>> & \AutA{n}
\end{CD}
$$
By construction, $\gamma_{X}(Y) \supseteq U$ and so $\gamma_{X}$ is dominant. The pull-back family $\gamma_{X}^{*}\Phi$ parametrized by $Y$ is conjugate to the constant family, i.e.
$$
(\gamma_{X}^{*}\Phi)_{y} = \phi_{\gamma_{X}(y)}= \mu_{y}^{-1} \circ \g \circ \mu_{y} \text{ for all } y\in Y.
$$
This implies that $\gamma_{X}^{*}\Phi$ is locally finite as an automorphism of $Y\times \A{n}$, and so $\Phi$ is locally finite, too, because $\OOO(X)\otimes\OOO(\A{n})$ is embedded into $\OOO(Y)\otimes\OOO(\A{n})$ as a $\Phi^{*}$-stable subspace. This proves (1).
\ps
(b) Now consider the Jordan decomposition $\Phi = \Phi_{s}\cdot\Phi_{u}$. It then follows from Lemma~\ref{JD.lem} that $\Phi_{s}=((\Phi_{s})_{x})_{x\in X}$,  $\Phi_{u}=((\Phi_{u})_{x})_{x\in X}$ are families of semisimple, resp. unipotent automorphisms and that   $\Phi_{x}=(\Phi_{s})_{x}\cdot (\Phi_{u})_{x}$ is the Jordan decomposition. This implies (2) and (3).
\end{proof}

Given a family $\Phi=(\Phi_{x})_{x\in X}$ of automorphisms of $\A{n}$ we denote by $(X\times\A{n})^{\Phi}$ the fixed point set of $\Phi$ considered  as an automorphism of $X\times\A{n}$. It is a closed subvariety, and it is smooth in case $\Phi$ is semisimple and $X$ is smooth (see \cite{Fo1973Fixed-point-scheme}). Here is a crucial result. We formulate it in the more general setting of families of reductive group actions, but we will need it only for semisimple automorphisms.

\begin{lemma}  \label{fixedpoints.lem}
Let $G$ be a reductive group and $\rho=(\rho_{x})_{x\in X}$ a family of $G$-actions on affine $n$-space $\A{n}$. If $X$ is smooth, then the induced morphism $p\colon (X\times \A{n})^{G} \to X$ is smooth.
\end{lemma}

\begin{proof}
If $Y$ is a smooth $G$-variety, then $Y^{G}$ is smooth, and for any $y\in Y^{G}$ we have $T_{y}(Y^{G}) = (T_{y}Y)^{G}$ (\cite{Fo1973Fixed-point-scheme}). In our situation this implies that $F:=(X\times \A{n})^{G}$ is smooth, as well as $F_{x}: = p^{-1}(x) = (\{x\}\times \A{n})^{G}$,  and for any $(x,a)\in(X\times \A{n})^{G}$ we get an exact sequence 
$$
\begin{CD}
0 @>>> T_{(x,a)}F_{x} @>\subseteq>> T_{(x,a)} F @>dp_{(x,a)}>> T_{x}X @>>> 0.
\end{CD}
$$
Since $X$, $F$ and $F_{x}$ are all smooth the claim follows.
\end{proof}

\begin{proposition}   \label{famdiag.prop}
Let $\Phi=(\Phi_{x})_{x\in X}$ be a family of automorphisms of $\A{n}$ where $X$ is smooth and connected. Assume that $\Phi_{x}$ is conjugate to a fixed diagonal automorphism $\d \in D(n)$ on a dense open set $U\subseteq X$.
Then all $\Phi_{x}$ are semisimple, the fixed point set $F:=(X\times \A{n})^{\Phi}$ is connected, and the differentials $d_{(x,a)}\Phi_{x}\in\GL_{n}$  are conjugate to $\d$ for $(x,a)\in F$.
\end{proposition}

\begin{proof}
We know from Proposition~\ref{fam.prop}(3) that all $\Phi_{x}$ are semisimple. Let $F_{0} \subseteq (X\times \A{n})^{G}$ be a connected component of the fixed point set $F$.  Then the differentials $d_{(x,a)}\Phi_{x}$ for $(x,a)\in F_{0}$ are conjugate in $\GL(n)$. In fact, $p(F_{0})$ is open and dense in $X$ by the Lemma above and so the claim holds on a dense open set of $F_{0}$. It follows that $TF_{0} \to F_{0}$ is a family of linear actions which are conjugate on a dense open set of $F_{0}$. Since semisimple conjugacy classes are closed in $\GL(n)$ the claim follows.

It remains to see that the fixed point set $F$ is connected. This is clear over $U$, by assumption. For every connected component $F_{i}$ of $F$ the image $p(F_{i})$ is open and dense in $X$ and thus meets $U$, hence the claim.
\end{proof}

\begin{corollary}  \label{ss-weakly-closed.cor}
Let $\g\in\AutA{n}$ be a diagonalizable automorphism. Then, its conjugacy class $C(\g)$ is weakly closed.
\end{corollary}

\begin{proof}
We can assume that $\g$ is diagonal. Let $\Phi=(\Phi_{x})_{x\in X}$ be a family of automorphisms of $\A{n}$ such that $\Phi_{x}$ is conjugate to $\g$ on a dense open set $U \subseteq X$. We have to show that $\Phi_{x}$ is conjugate to $\g$ for all $x\in X$. For this we can assume that $X$ is a smooth curve. The proposition above implies that all $\Phi_{x}$ are semisimple and that the differentials $d_{(x,a)}\Phi_{x}$ in the fixed points are conjugate to $\g$. Now Lemma~\ref{faithful-on-Tx.lem} implies that $\Phi_{x}$ is diagonalizable, hence conjugate to $\g$.
\end{proof}

\ps
\subsection{Semisimple automorphisms of \texorpdfstring{$\An$}{An}}  
\label{Semisimple-automorphisms-of-An.subsec}

Our work on semisimple automorphisms is motivated by the well-known problem asking whether a semisimple automorphism of $\AA^n$ is diagonalizable, i.e. conjugate to a diagonal automorphism $(a_1 x_1, \ldots, a_n x_n)$, $a_i \in \kk^*$. More generally, 
we address the following questions.

\begin{question}  \label{diagonalizable-and-semisimple.ques}
Let $\g \in \Aut (\AA ^n)$ be an automorphism. Are the four following assertions equivalent?

\be
\item \label{g-is-diagonalizable}
$\g$ is diagonalizable.
\item \label{g-is-semisimple}
$\g$ is semisimple.
\item \label{g-has-a-closed-conjugacy-class}
The conjugacy class $C (\g)$ is closed in $\Aut (\AA ^n)$.
\item \label{g-has-a-weakly-closed-conjugacy-class}
The conjugacy class $C (\g)$ is weakly closed in $\Aut (\AA ^n)$.
\ee
\end{question}

Let us summarize what we know about this question at the present time. The implications (\ref{g-is-diagonalizable})$\Rightarrow$(\ref{g-is-semisimple}) and (\ref{g-has-a-closed-conjugacy-class})$\Rightarrow$(\ref{g-has-a-weakly-closed-conjugacy-class}) are clear, and 
the implication  (\ref{g-is-diagonalizable})$\Rightarrow$(\ref{g-has-a-weakly-closed-conjugacy-class}) was proved in Corollary~\ref{ss-weakly-closed.cor}. 
We will see in Corollary~\ref{a-criterion-for-diagonalizability.cor} below that 
(\ref{g-is-semisimple}) and (\ref{g-has-a-weakly-closed-conjugacy-class}) imply (\ref{g-is-diagonalizable}). Finally, the known results about the most interesting implication (\ref{g-is-semisimple})$\Rightarrow$(\ref{g-is-diagonalizable}) are summarized in the next statement.

\begin{proposition}\label{diagonal.prop}
Let $\g \in \AutA{n}$ be a semisimple element. Then $\g$ is diagonalizable in the following cases.
\be
\item \label{n=2}
$n=2$;
\item \label{n=3-and-g-is-of-ingfinite-order}
$n=3$ and $\g$ is of infinite order;
\item \label{g-generates-a-large-subgroup}
$\dim \lgr \geq n-1$;
\item \label{g-is-triangularizable}
$\g$ is triangularizable.
\ee
\end{proposition}

\begin{proof} The group $\lgr$ is a diagonalizable group, hence reductive. Then (\ref{n=2}) follows from the amalgamated product structure of $\AutA{2}$ (see Section~\ref{amalgam.subsec}), and (\ref{n=3-and-g-is-of-ingfinite-order}) is proved in \cite{KrRu2014Families-of-group-}. As for (\ref{g-generates-a-large-subgroup}) it is known that a faithful  action of a commutative reductive group  of dimension $\geq n-1$ on $\A{n}$ is linearizable (see  \cite[Theorem VI.3.2(2)]{KrSc1992Reductive-group-ac}). (\ref{g-is-triangularizable}) is proved in \cite[\S1 Corollary~1]{KrKu1996Equivariant-affine}. It also follows from Proposition~\ref{triang.prop}(\ref{1-semisimple}).
\end{proof}

A necessary first step for proving that a semisimple automorphism is diagonalizable is the following fixed point result.

\begin{proposition}   \label{semisimple-auto-has-fixed-point.prop} 
Every semisimple automorphism of $\An$ has a fixed point. More generally, if $D$ is a diagonalizable algebraic group such that $D/D^{\circ}$ is cyclic, then any action of $D$ on $\An$ has fixed points.
\end{proposition}

\begin{proof}
If $\g \in \AutA{n}$ is semisimple, then the group $D:=\overline{\langle \g \rangle}$ is diagonalizable and $D/D^{\circ}$ is cyclic. Thus the first statement follows from the second.

The following Lemma~\ref{fixed-points-diag-groups.lem} shows that there is a $d\in D$ of finite order with the same fixed point set as $D$. But any finite-order automorphism of $\A{n}$ admits a fixed point, see \cite{PeRa1986Finite-order-algeb}. 
\end{proof}

\begin{lemma}\label{fixed-points-diag-groups.lem}
Let $D$ be a diagonalizable group acting on an affine variety $X$. Assume that $D/D^{\circ}$ is cyclic. Then there is an element $d \in D$ of finite order such that $X^{D} = X^{d}$.
\end{lemma}
\begin{proof} We can embed $X$ equivariantly into a $D$-module $V$. Hence it suffices to prove the claim for $X=V$.
We decompose $V$ in the form $V = V^{D}\oplus\bigoplus_{i=1}^{m}V_{\chi_{i}}$ where the $\chi_{i}$ are nontrivial characters of $D$, and $V_{\chi_{i}}:=\{v\in V \mid d v = \chi_{i}(d) \cdot v \text{ for all }d\in D\}$. Set $U_{i}:= D \setminus \Ker\chi_{i}$. Clearly, every element $d\in \bigcap_{i}U_{i}$ has the same fixed point set as $D$. Thus we have to show that $\bigcap_{i}U_{i} \neq \emptyset$, because every nonempty open set of $D$ contains elements of finite order. 

If $\bigcap_{i}U_{i} = \emptyset$, then $\bigcup_{i}\Ker\chi_{i} = D$. Hence, every irreducible component of $ D$ is contained in $\Ker\chi_{i}$ for some $i$. Since $D/D^{\circ}$ is cyclic, it follows that one of the irreducible components generates $D$ as a group. But this implies that $\Ker\chi_{i} = D$ for some $i$, contradicting the fact that the characters $\chi_{i}$ are nontrivial.
\end{proof}

\begin{remark}
Assume that $\g \in \AutA{n}$ admits a fixed point $a$. If $\g$ is semisimple, then the differential $d \g _a \colon T_a \A{n} \to T_a \A{n}$ is also semisimple. However, the converse does not hold. The automorphism $\g:=(x+y^2,y) \in \AutA{2}$ fixes the origin, $d \g_0= \id$ is semisimple, but $\g$ is not.
\end{remark}

The next result shows that if $\g$ has a fixed-point, then the weak closure $\wc{C(\g)}$ of its conjugacy class contains a linear automorphism. In fact, this already holds for the weak closure of its $\Aff(n)$-conjugacy class, where we recall that $\Aff(n) \subseteq \AutA{n}$ denotes the closed algebraic subgroup of affine transformations.

\begin{proposition}   \label{Anclosure.prop}
For any $\g\in \AutA{n}$ the $\Aff(n)$-conjugacy class 
$$
C_{\Aff(n)}(\g):=\{\h\cdot \g\cdot \h^{-1} \mid \h \in \Aff(n) \} \subseteq \AutA{n}
$$
is a (locally closed) algebraic subset. If $\g$ has a fixed point $a\in\A{n}$, then the closure $\overline{C_{\Aff(n)}(\g)}$ contains the differential $d_{a}\g\in\GL(n)$. In particular, it contains a diagonal automorphism.
\end{proposition}

\begin{proof}
All subset $\AutA{n}_{k}$ are stable under conjugation by $\Aff(n)$, and so the first statement is clear. For the second, using a conjugation with a suitable translation, we can assume that $\g(0) = 0$. Then, for $\lambda(t):=(tx_{1},\ldots,tx_{n})$, the limit  $\bar \g:= \lim_{t\to 0}\lambda(t)^{-1}\cdot  \g\cdot \lambda(t)$ exists and coincides with the differential $d_{0}\g\in\GL(n)$. Since the closure of every $\GL(n)$-conjugacy class in $\GL(n)$ contains a diagonal element, we are done.
\end{proof}

\begin{corollary}  \label{fpclosed.cor}
Assume that $\g$ has a fixed point. 
If $C_{\Aff(n)}(\g)$ is closed or if $C(\g)$ is weakly closed, then $\g$ is a diagonalizable element.
\end{corollary}

\begin{proof} The statement for $C_{\Aff(n) }(\g)$ is clear by the proposition above. If $C(\g)$ is weakly closed, then $C_{\Aff(n)}(\g)\subseteq C(\g)$ implies that  $\overline{C_{\Aff(n)}(\g)} \subseteq C(\g)$.
\end{proof}

Using Proposition~\ref{semisimple-auto-has-fixed-point.prop} and Corollary~\ref{ss-weakly-closed.cor}, we get the following result.

\begin{corollary}  \label{a-criterion-for-diagonalizability.cor}
Let $\g \in \Aut (\AA^n)$ be a semisimple automorphism. Then $\g$ is diagonalizable if and only if its conjugacy class $C(\g)$ is weakly closed.
\end{corollary}

\begin{remark}
The  $\Aff(n)$-conjugacy classes in $\Aff(n)$ are known and have been studied by \name{Blanc} in \cite{Bl2006Conjugacy-classes-}. It follows that the closed $\Aff(n)$-conjugacy classes are the classes of the semisimple elements, and that each such closed $\Aff(n)$-conjugacy class contains a diagonal element $(a_1x_1, \ldots, a_n x_n)$ which is uniquely determined up to a permutation of the scalars $a_i$.
\end{remark}

\begin{example}
Let $\g=(x,z,xz-y)\in \AutA{3} $. The fixed point set has two irreducible components, namely the two lines $\{y=z=0\}$ and $\{x=2,y=z\}$. An easy calculation shows that the differential $d_{p}\g$ in the fixed point $p=(a,0,0)$ has trace $a+1$ and is semisimple except for $a=\pm2$. Thus the weak closure $\wc{C(\g)}$ contains uncountably many conjugacy classes of diagonal elements.
\end{example}

It is well known that the closure of the $\GL(n)$-conjugacy class of a linear endomorphism $\g$ contains the semisimple part $\g_{s}$. So we might ask the following question.

\begin{question}\label{semisimple-part-in-the-weak-closure-of-the-conjugacy-class.ques} 
Let $\g\in\AutA{n}$ be a locally finite automorphism. Does the weak closure of $C(\g)$ contain $\g_{s}$? 
And what about the closure of $C(\g)$?
\end{question}

A positive answer to the first question would imply that if the conjugacy class of a locally finite automorphism $\g$ is weakly closed, then $\g$ is semisimple (see the implication (\ref{g-has-a-weakly-closed-conjugacy-class})$\Rightarrow$(\ref{g-is-semisimple}) from Question~\ref{diagonalizable-and-semisimple.ques}).

This holds for triangularizable automorphisms, by Proposition~\ref{triang.prop}(\ref{3-triangular}), and therefore in dimension $n=2$ since in that case an automorphism is locally finite if and only if it is triangularizable (see Lemma~\ref{lengthone.lem} in Section~\ref{Aut(A2).sec}). It is also true in the following case.

\begin{proposition}   \label{unipotent-part-is-translation.prop}
Let $\g \in \AutA{n}$ be a locally finite automorphism whose unipotent part $\g_u$ is conjugate to a translation. Then the weak closure of $C(\g)$ contains $\g_{s}$.
\end{proposition}

\begin{proof}
We may assume that 
$\u:=\g_u =(x_{1}+1,x_{2},\ldots,x_{n})$.
As $\s:=\g_s$ commutes with $\u$, it is of the form
\[ 
\s=(x_{1}+p(x_{2},\ldots,x_{n}), g_{2}(x_{2},\ldots,x_{n}),\ldots,g_{n}(x_{2},\ldots,x_{n}))\in\AutA{n}.
\]
Since $\s$ is semisimple, the group $G:= \overline { \langle \s \rangle }$ is reductive. The automorphism $\s$ induces an automorphism $\overline{ \s} := ( g_2(x_2, \ldots, x_n), \ldots, g_n( x_2, \ldots, x_n) )$ of $X:= \AA^{n-1}$. Therefore, the group $G$ also acts on $X$ and the second projection $\pr_2 \colon \AA^1 \times X \to X$ is $G$-equivariant. By the proposition below, there exist a $\pr_2$-automorphism $\phi$ of $\AA^1 \times X$ and a $\lambda \in \kk^*$ such that
\[
\s':=\phi\circ \s \circ \phi^{-1} = (\lambda x_{1},g_{2}(x_{2},\ldots,x_{n}),\ldots,g_{n}(x_{2},\ldots,x_{n})).
\] 
Since $\phi$ has the form $\phi=(\rho x_{1}+q(x_{2},\ldots,x_{n}),x_{2},\ldots,x_{n})$, it is clear that it commutes with $\u$. Hence, $\g':=\phi\cdot\g\cdot\phi^{-1}$ has the Jordan decomposition  $\g'=\u \cdot\s' = \s' \cdot\u$.
Now set $\lambda(t):=(tx_{1},x_{2},\ldots,x_{n})$, $t \in \kst$. An easy calculation shows that  
\[
\lim_{t \to 0}\lambda(t) \cdot \g'   \cdot \lambda(t)^{-1}= (x_{1}, g_{2},\ldots,g_{n})=\g'_{s} \in\AutA{n},
\]
and the claim follows.
\end{proof}

\begin{proposition}[\protect{\cite[Proposition~1]{KrKu1996Equivariant-affine}}]\label{lifting-red-group-actions.prop} Let $X$ be an affine variety with an action of a reductive group $G$. 
Assume that $G$ acts also on $\AA^1 \times X$ in such a way that the second projection $\pr_2 \colon \AA^1 \times X \to X$ is $G$-equivariant. Then the $G$-action on $\AA^1 \times X$ is equivalent to a diagonal action of $G$ of the form
\[ 
g (a,x) = ( \chi (g) \cdot a , gx) \quad \text{for } g \in G, \ a \in \AA^1,  \ x \in X, 
\]
where $\chi$ is a character of $G$.
\end{proposition}

\ps
\subsection{Unipotent automorphisms of \texorpdfstring{$\AA^n$}{An}}  \label{Unipotent-automorphisms-of-An.subsec}
\ps
\subsubsection*{\bfit{Generalized translations}} 
A unipotent element $\u \in \Aut(\An)$ will be called a \itind{generalized translation} if the corresponding $\kplus$-action $\mu$ has a section, i.e. there is a $\kplus$-equivariant morphism $\sigma\colon \An \to \kplus$ (see Section~\ref{local-sections.subsec}). It then follows that the morphism $\kplus\times Y \to \An$, $(s,y)\mapsto sy$, is an isomorphism where $Y := \sigma^{-1}(0)$.

Clearly, if $Y \simeq \AA^{n-1}$, then $\u$ is conjugate to a translation.  But it is an open problem if $A^{1}\times Y \simeq \AA^{n}$ always implies that $Y \simeq \AA^{n-1}$, i.e. if any generalized translation is conjugate to a translation (Cancellation Problem, see \cite{Kr1996Challenging-proble}). This is obvious for $n=2$, and it is also known for $n = 3$, due to the work of \name{Fujita}, \name{Miyanishi} and \name{Sugie}, see \cite{Fu1979On-Zariski-problem} and \cite{MiSu1980Affine-surfaces-co}.

\begin{remarks} 
\be
\item If the $\kplus$-action on $\An$ corresponds to the locally nilpotent vector field $\delta$, then $f \colon \An \to \kplus$ is a section if and only if  $\delta f = 1$ (Lemma~\ref{sections-via-vector-fields.lem}).
\item Let  $\u=(f_{1},\ldots,f_{n})\in\AutA{n}$ be a unipotent automorphism, and assume that there is a $j$ such that $f_{j}=x_{j}+c_{j}$ with a nonzero constant $c_{j}$. Then $\u$ is conjugate to a translation. 
\newline
(Conjugating $\u$ with a suitable $t\id$, $t \in\kst$, we can assume that $c_{j}=1$. For the corresponding $\kplus$-action $\u(s)$ we see that the $j$-coordinate is $x_{j}+s$. Therefore, the linear projection $p\colon \An 
\to \kk$  onto the $j$th coordinate is a section of this action, and since $p$ is linear we have $Y:=p^{-1}(0)\simeq \AA^{n-1}$.)
\ee
\end{remarks}  \label{unipotent.rems}
The next result about unipotent elements in the affine group $\Aff(\An)$ is due to \name{J\'er\'emy Blanc} \cite{Bl2006Conjugacy-classes-}.
\begin{proposition} \label{blanc.prop}
All fixed point free unipotent  $\u\in\Aff(\A{n})$ are conjugate in $\AutA{n}$.
\end{proposition}

\begin{proof}
By the Jordan normal form we see that $\u$ is $\GL(n)$-conjugate to an element of the form $(u_{1},\ldots,u_{n})$ where each $u_{j}$ is $x_{j}+c_{j}$ or  $x_{j}+x_{j+1}+c_{j}$. Conjugating with a suitable translation, we can assume that $c_{j}=0$ in the second case, and conjugation with a suitable diagonal element we finally end up with the three possibilities $u_{j}=x_{j}$, $u_{j}=x_{j}+1$, or $u_{j}=x_{j}+x_{j+1}$. If $\u$ has no fixed points, then the second case has to show up, and the claim follows from Remark~\ref{unipotent.rems}(2).
\end{proof}

\begin{question}\label{characterization-of-unipotent-automorphisms-conjugate-to-translations} 
Is there a characterization of those unipotent $\u\in\AutA{n}$  which are conjugate to  translations? Same question for  $\u\in\JJJ(n)$.
\end{question}

The following result is due to \name{Kaliman} \cite{Ka2004Free-C-actions-on-}.
\begin{proposition}\label{freeunipotent.prop}
A unipotent element $\u\in \AutA{3}$ is conjugate to a translation if and only if  $\u$ has no fixed points.
\end{proposition}

\begin{remark}\label{example of Winkelmann.rem}
The example of \name{Winkelmann} mentioned in Remark~\ref{ffunipotent.rem} gives a triangular unipotent and fixed point free automorphism of $\AA^{4}$ which is not even conjugate to a generalized translation.
\end{remark}

\begin{example}\label{Bass.exa}
In $\Aut(\Atwo)$  every unipotent automorphism is \itind{triangularizable}, i.e. conjugate to a triangular automorphism from $\Ju(2)$, see Lemma~\ref{lengthone.lem}.  \name{Bass} remarked in \cite{Ba1984A-nontriangular-ac} that this does not hold in dimension $\geq 3$. He starts with the locally nilpotent linear vector field $\delta:=-2y\ddx + z\ddy \in \VEC(\Athree)$. Then $\Delta:=xz+y^{2}$ belongs to the kernel of $\delta$, hence $\Delta\delta$ is again locally nilpotent. The corresponding unipotent automorphism $\n$ of $\Athree$ is, by construction, a modification of the triangular automorphism corresponding to $\delta$ (see Section~\ref{modification.subsec}). This unipotent automorphism $\n$  is the famous {\it\name{Nagata}-automorphism}\idx{Nag@\name{Nagata}-automorphism}. It has the following fixed point set:
$$
F:= (\Athree)^{\n} = \{\Delta=0\} \cup \{y=z=0\}
$$
In particular, $F$ has an isolated singularity in 0. On the other hand, the fixed point set of any unipotent triangular automorphism has the form $\Aone\times Z$ where $Z \subset \AA^{n-1}$, hence cannot have an isolated singularity. As a consequence, $\n$ is not conjugate to a triangular automorphism.
\end{example}

\ps
\subsubsection*{\bfit{Closures of unipotent conjugacy classes}}
The next result shows that conjugacy classes of triangular unipotent elements in $\AutA{n}$ behave in a very strange way.
\begin{proposition}\label{triangclosure.prop}
For every nontrivial $\u\in \Ju(n)$ the weak closure $\wc{C(\u)}$ of the conjugacy class $C(\u)$ contains $\Ju(n)$. In particular, all these conjugacy classes have the same weak closure and the same closure, and they are not locally closed.
\end{proposition}
\begin{proof}
The automorphism $\u$ is of the form $\u=(x_{1}+p_{1},x_{2}+ p_{2}, \ldots,x_{n}+p_{n})$ where $p_{i}\in\kk[x_{i+1},\ldots,x_{n}]$. 

(1) If $p_{n}\neq 0$, then $\u$ is conjugate to a translation (Remark~\ref{unipotent.rems}). Thus the conjugacy class of a translation contains an open dense set of $\Ju(n)$. This proves the first claim for translations.

(2) In general, conjugating  $\u$ with a generic translation, we can assume that the nonzero $p_{i}$ contain a nonzero constant term. Now let $j$ be maximal such that $p_{j}\neq 0$. If $j=n$, then the claim follows from (1). If $j<n$ we set $\d_{t} :=(x_{1},\ldots,x_{j},tx_{j+1},\ldots,tx_{n})$, $t\in\Cst$. Then
\begin{multline*}
\d_{t}^{-1}\cdot \u\cdot \d_{t} =\\ (x_{1}+p_{1}(x_{2},\dots,x_{j},tx_{j+1},\ldots, tx_{n}),\ldots,x_{j}+p_{j}(tx_{j+1},\ldots,tx_{n}), x_{j+1},\ldots,x_{n}),
\end{multline*}
and so
$$
\lim_{t\to 0} \d_{t}^{-1}\cdot \u\cdot \d_{t} = (x_{1}+p_{1}(x_{2},\ldots,x_{j},0,\ldots,0),\ldots,x_{j}+p_{j}(0),x_{j+1},\ldots,x_{n})
$$
which is conjugate to a translation, again by Remark~\ref{unipotent.rems}. This finishes the proof of the first claim. The remaining claims are immediate consequences.
\end{proof}

\ps
\subsection{Shifted linearization and simplicity}\idx{shifted linearization}
\label{shifted-lin.subsec}
We have already seen in Example~\ref{Bass.exa} above that the \name{Nagata}-automorphism
\idx{Nag@\name{Nagata}-automorphism} 
\[ 
\n:=(x-2y\Delta -z\Delta^{2},y+z\Delta,z), \quad \Delta:=xz+y^{2}
\]  
is a modification of the unipotent linear automorphism $(x-2y-z,y+z,z)$ which generates the $\kplus$-action $\mu(s)=(x-2sy-s^{2}z,y+sz,z)$. It was shown by \name{Shestakov-Umirbaev} that $\n$ is not tame, i.e. it does not belong to the subgroup $\TameA{3}$ generated by $\Aff(3)$ and $\JJJ(3)$ (\cite{ShUm2004The-tame-and-the-w,ShUm2004Poisson-brackets-a}). But it is unknown if a conjugate of $\n$ is tame. On the other hand,  \name{Maubach-Poloni} proved in \cite{MaPo2009The-Nagata-automor}  that ``twice'' the \name{Nagata}-automorphism, namely the composition $2\cdot\n:=(2\id)\cdot\n$, which is again not tame, is even linearizable. This is a special case of the following general result.

\begin{lemma}  \label{shifted.lem}
Let $\u$ and $\f$ be elements of an ind-group $\GGG$ satisfying the following conditions.
\begin{enumerate}
\item  \label{1-unipotent}
The element $\u$ is unipotent.
\item \label{3-normalizing}
The element $\f$ normalizes the closed subgroup $\overline {\langle \u \rangle }$.
\item \label{2-not-commute}
We have $\f\cdot\u \neq \u\cdot\f$.
\end{enumerate}
Then $\f\cdot\u$ and $\f$ are conjugate in $\GGG$. More precisely, there exists an element $\u_0 \in \overline {\langle \u \rangle }$ such that $\u_0\cdot (\f\cdot\u) \cdot\u_0^{-1} = \f$.
\end{lemma}

\begin{proof}
Denote by $\mu \colon \kplus \simto \overline {\langle \u \rangle }$ the unique isomorphism of algebraic groups such that $\mu (1) = \u$. By (\ref{3-normalizing}), there exists an element $s \in \kk$ such that 
$\f^{-1} \cdot\mu (1)\cdot \f= \mu (s_{0})$, and by (\ref{2-not-commute}) we have $s_{0} \neq 1$. It follows that
\[ 
\f^{-1} \cdot \mu(s) \cdot \f = \mu (s s_{0})  \text{ for all }s \in \kplus.
\]
Writing $u_0 = \mu (s)$, we find
$$
\u_0\cdot (\f\cdot\u) \cdot\u_0^{-1} = \f\cdot(\f^{-1}\cdot\u_{0}\cdot\f)\cdot\mu(1)\cdot\mu(-s)
= \f\cdot\mu(ss_{0}+1-s),
$$
hence $\u_0\cdot (\f\cdot\u) \cdot\u_0^{-1}=\f$  for $s:=\frac{1}{1-s_{0}}$.
\end{proof}

Since the \name{Nagata} automorphism $\n$ is a modification of a linear automorphism by a homogeneous invariant of degree 2, it is clear that the subgroup $\overline{\langle\n\rangle}$  is normalized by $\kst\id$.  In fact, we get
$$
(t^{-1}\id)\cdot\n\cdot(t\id) = (x-2t^{2}y\Delta - t^{4}z\Delta^{2},y+t^{2}z\Delta,z),
$$
hence $(t\id)$ and $\n$ only commute if $t^{2}=1$. Using the lemma above
this proves the result of \name{Maubach} and \name{Poloni}.

\begin{proposition}[Shifted linearization] \label{shifted-linearization.prop}
Let $\n \in \Aut (\AA^3)$ be the \name{Nagata}-automorphism. For each $t \neq \pm 1$, the automorphism $ t\cdot \n:=(t\id)\cdot \n$ is conjugate to the linear automorphism $ t \id$.
\end{proposition}

\begin{remark}\label{ss-and-triang-not-wclosed.rem}
This proposition also implies that the subset $\Aut^{\text{\it tr}}(\Athree)$ of triangularizable automorphisms and the subset $\Aut^{\text{\it ss}}(\Athree)$ of semisimple automorphisms are not weakly closed. In fact, in the family $\Phi:=(t\cdot\n)_{t\in\kst}$ the automorphisms are semisimple and triangularizable for $t\neq \pm1$, but unipotent and not triangularizable for $t=\pm1$ (Example~\ref{Bass.exa}).
\end{remark}

\begin{remark} \label{gen-Nagata.rem}
More generally,  let $\u\in\Aut(\An)$ be a unipotent automorphism with corresponding locally nilpotent vector field $\delta \in \LNV(\A{n})$. 
If $\delta$ is homogeneous of degree $d\geq 1$, i.e. $\delta(x_{i})$ are homogeneous polynomials of degree $d-1$, then, as above, 
$\kst\id$ normalizes $\overline{\langle\u\rangle}$, and $( t \id)^{-1} \cdot \u \cdot ( t \id) = \u$ if and only if $t^{d}=1$. Thus $t\cdot\u$ is conjugate to $t\id$ if  $t^{k}\neq 1$ by Lemma~\ref{shifted.lem}.
\end{remark} 

\ps
\subsubsection*{\bfit{Normal subgroups}}
It was shown by \name{Danilov} in \cite{Da1974Non-simplicity-of-} that $\SAutA{2}$ is not simple as an abstract group, cf. \cite{FuLa2010Normal-subgroup-ge}.
This should also be true in higher dimension. On the other hand, it might be true that $\SAut(\An)$ does not contain a {\it closed\/} normal subgroup.

\begin{proposition}  \label{simple.prop}
Let $\NNN \subseteq \AutA{n}$ be a nontrivial normal and weakly closed subgroup. Then
\be
\item \label{1-SLn} $\SL(n)\subseteq \NNN$ and $\Ju(n)\subseteq\NNN$. In particular, $\NNN$ contains the normal subgroup generated by the tame elements with jacobian determinant equal to 1.
\item \label{2-SAut(A2)}
If $n=2$, then $\NNN$ is equal to the preimage $\jac^{-1}(H)$ of a weakly closed subgroup $H \subseteq\kk^*$. 
\item \label{3-Nagata} If $n=3$, then $\NNN$ contains the \name{Nagata}-automorphism.
\ee
\end{proposition}

\begin{proof}
(\ref{1-SLn}) Let $\g\in\NNN$, $a\in  \AA^n$ and $b:=\g(a) \neq a$. If $\h$ is any automorphism which fixes $a$ and $b$, then the commutator 
$\c:=\g \h \g^{-1} \h ^{-1}$ fixes $b$, and the tangent representation in $T_{b} \An=\kk^{n}$ is given by $d_{b}\c = d_{a}\g \circ d_{a}\h \circ d_{b}\g^{-1}\circ d_{b}\h^{-1}$. We claim that there is an $\h$ such that $d_{b}\c$ is not a scalar multiple of the identity. For this we can assume that $a=(0,\ldots,0,0)$ and $b=(0,\ldots,0,1)$. Define
$$
\h := (x_{1}+x_{n}(x_{n}-1)f_{1}, \cdots, x_{n-1}+x_{n}(x_{n}-1)f_{n-1},x_{n})
$$
where $f_{1},\ldots,f_{n-1}\in\kk[x_{n}]$.
This automorphism fixes $a$ and $b$ and the differential is given by
$$
d_{a}\h = 
\begin{bmatrix} 1 & 0&\cdots &-f_{1}(0)\\
0 & 1 & \cdots &-f_{2}(0) \\
\vdots && \ddots & \vdots \\
0 &\cdots&\cdots &1
\end{bmatrix}
\qquad
d_{b}\h = 
\begin{bmatrix} 1 & 0&\cdots &f_{1}(1)\\
0 & 1 & \cdots &f_{2}(1) \\
\vdots && \ddots & \vdots \\
0 &\cdots&\cdots &1
\end{bmatrix}
$$
Now it is clear that for a suitable choice of the polynomials $f_{1},\ldots,f_{n-1}$ the composition 
$d_{b}\c=d_{a}\g \circ d_{a}\h \circ d_{b}\g^{-1}\circ d_{b}\h^{-1}$ is not a scalar multiple of the identity. 

By Proposition~\ref{Anclosure.prop} we have $d_{b}\c \in\NNN$, hence $\SL(n)\subseteq \NNN$, because $\SL(n)\cap\NNN$ is a closed normal subgroup of $\SL(n)$. Since the special affine group $\SA(n)$ is generated by the conjugates of $\SL_{n}$, we get $\SA(n) \subset \NNN$. In particular, $\NNN$ contains the translations, hence $\Ju(n)$ (Proposition~\ref{triangclosure.prop}). This proves the first claim.

\ps
(\ref{2-SAut(A2)}) 
The inclusion  $\SAutA{2} \subset \NNN$ follows from (\ref{1-SLn}), because $\Aut(\Atwo)$ is generated by $\Aff(n)$ and $\J{2}$ (see
Section~\ref{amalgam.subsec}), hence $\SAut(\Atwo)$ is generated by $\SL(2)$ and $\Ju(2)$. 
Setting $H:= \jac ( \NNN)$, we get $H= \{ h \in \kk^* \mid (hx,y) \in \NNN \}$, hence $H$ is weakly closed in $\kst$. 

\ps
(\ref{3-Nagata})
We have $\SL(3) \subset \NNN$ by (\ref{1-SLn}). Choosing for $t$ a primitive third root of unity and applying Proposition~\ref{shifted-linearization.prop}, it follows that the automorphism $t\cdot \n$ is conjugate to the linear automorphism $t \id \in \SL(3)$. Hence $\n \in \NNN$.
\end{proof}

Note that the proposition does not imply that   $\SAutA{2}$ is simple as an ind-group. However, there is the following result.

\begin{proposition}[\cite{Kr2017Automorphism-group}] \label{Kraft.prop}
Let $\GGG$ be an affine ind-group.
Every nontrivial homomorphism $\phi\colon \SAut(\An) \to \GGG$ is a closed immersion.
\end{proposition}
\begin{corollary}
A nontrivial action of $\SAut(\An)$ on a connected affine variety $X$ has no fixed points.
\end{corollary}

\begin{proof}
By the proposition above, the homomorphism $\SAut(\An) \to \Aut(X)$ is a closed immersion, and so the action of $\SAut(\An)$ is faithful.
Let $x\in X$ be a fixed point.
Then the representation of $\SAut(\An)$ on $T_{x}X$ is trivial, because there are no closed immersions $\SAut(\An) \into \GL(W)$ for a finite-dimensional $\kk$-vector space $W$. This implies that every reductive subgroup of $\SAut(\An)$ acts trivially on $T_{x}(X)$. This contradicts Lemma~\ref{faithful-on-Tx.lem} which shows that for a faithful action of a reductive group on a variety $X$ all tangent representations are also faithful. 
\end{proof}

\ps
\subsubsection*{\bfit{Embeddings into \texorpdfstring{$\Aut(\An)$}{Aut(An)}}}
It is an interesting question which automorphism groups can be embedded into $\Aut(\An)$. For example, setting $T_m:= (\kst)^{m}$, it is shown in \cite{DeKuWi1999Subvarieties-of-Cn} that the group $\Aut (T_2 )$ does not embed into $\Aut (\AA^n)$, for any  $n \geq 1$, even as an abstract group. 

We will use their idea to prove the following result about $\Aut(T_m)$ which implies that there is no injective homomorphism of ind-groups $\Aut(T_m ) \into \Aut(\An)$ for any $m\geq 2$. 

Recall that the automorphism group of $T_{m}$ is a semidirect product: $\Aut(T_{m}) = \GL_{m}(\ZZ) \ltimes T_{m}$ (see Example~\ref{aut-torus.exa}).

\begin{proposition} \label{faithful-action-has-fixed-point.prop} 
Assume that $\Aut(T_{m})$ acts faithfully on an irreducible variety $X$. If $m\geq 2$, then 
$T_{m} \subseteq \Aut(T_{m})$ has no fixed points in $X$.
\end{proposition}

\begin{proof} 
The natural action of $\GL_{m}(\ZZ)$ on $T_{m}$ induces an action on the character group $X(T_{m}) = \Hom(T_{m},\kst) =  \ZZ \chi_1 \oplus \cdots \oplus \ZZ \chi_m \simeq \ZZ^{m}$ which coincides with the standard representation of $\GL_{m}(\ZZ)$ on $\ZZ^{m}$. More precisely, if $\rho\colon T_{m} \to \GL(V)$ is a representation with character $\chi = \sum_{i}m_{i}\chi_{i}$ and if $\alpha\in\GL_{m}(\ZZ)$, then the representation $\rho\circ\alpha\colon T_{m} \to \GL(V)$ has character $\alpha(\chi)=\sum_{i}m_{i}\alpha (\chi_{i} )$.

Now assume that $X^{T_{m}}$ is not empty. For any fixed point $x$ denote by $\chi_{x}$ the character of the tangent representation of $T_{m}$ on $T_{x}X$ which is faithful by Lemma~\ref{faithful-on-Tx.lem}. Since $X^{T_m}$ is stable under $\GL_{m}(\ZZ)$ it follows that $\alpha(x)$ is also a fixed point for every $\alpha \in\GL_{m}(\ZZ)$, and that the representation of $T_{m}$ in $T_{\alpha(x)}X$ has character $\alpha(\chi_{x})$. This implies that there are infinitely many non-equivalent tangent representation in the fixed point set $X^{T_{m}}$, because $\GL_{m}(\ZZ)$ does not have a finite orbit in $\ZZ^{m}\setminus \{0\}$. This contradicts the following lemma and thus proves the proposition.
\end{proof}

\begin{lemma}
Consider a nontrivial action of a reductive group $G$ on an affine variety $X$. Then the number of equivalence classes of tangent representations of $G$ on $X^{G}$ is finite.
\end{lemma}

\begin{proof}
We can assume that $X$ is a $G$-stable closed subset of a $G$-module $V$. For any $v \in V^{G}$ the tangent representation is isomorphic to $V$, hence the tangent representation of $G$ on $T_{x}X$ is a submodule of $V$. Since $G$ is reductive, there are only finitely many equivalence classes of submodules of $V$.
\end{proof}

\begin{corollary}
If $m\geq 2$ there is no injective homomorphism of ind-groups $\Aut(T_m) \into \Aut(\An)$.
\end{corollary}

\begin{proof}
Such an injections $\Aut(T_m) \into \Aut(\An)$ defines an action of $\Aut(T_m )$ on $\An$. 
It follows from Proposition~\ref{semisimple-auto-has-fixed-point.prop} that $T_m \subset \Aut(T_m )$ has a fixed point which contradicts the proposition above.
\end{proof}

\ps
\subsubsection*{\bfit{Representations of \texorpdfstring{$\Aut(\An)$}{Aut(An)-modules}}}
We have seen that the linear action of $\Aut(\An)$ on $\OOO(\An) = \kk[x_{1},\ldots,x_{n}]$ is a representation (Proposition~\ref{locally-finite-action-on-Vec(X).prop}), i.e. $\OOO(\An)$ is an $\Aut(\An)$-module. It contains the trivial module $\kk$ as a submodule.
\begin{proposition}
For $n\geq 2$, the $\Aut(\An)$-module
$\OOO(\An) / \kk$ is simple. It is simple even as an $\SAut(\An)$-module.
\end{proposition}
\begin{proof}
As an $\SLn$-module we have the decomposition $\OOO(\An) = \bigoplus_{d}\OOO(\An)_{d}$ into homogeneous components $\OOO(\An)_{d}$ which are simple $\SLn$-modules and pairwise non-isomorphic
(Lemma~\ref{ker-Div.lem}). Therefore, an $\SAut(\An)$-submodule $M$ is a direct sum of homogeneous components. If $x_{1}^{d} \in M$, then $(x_{1}+1)^{d} \in M$ and so all homogeneous components of degree $\leq d$ belong to $M$. Since $(x_{1}+x_{2}^{m})^{d}\in M$ for all $m$ it follows that $M$ contains elements of arbitrary large degree.
\end{proof}

\begin{remark} The $\Aut(\An)$-module 
$\VEC( \An )$ contains the submodule $\VEC^{0}(\An)$. A similar argument as above, using again Lemma~\ref{ker-Div.lem}, shows that  $\VEC^{0}(\An)$ is simple even as an $\SAut(\An)$-module.
\end{remark}

\pmed
\section{Some Special Results about \texorpdfstring{$\Aut(\AA^2)$}{Aut(A2)}} 
\label{Aut(A2).sec}
\ps
\subsection{Amalgamated product structure of \texorpdfstring{$\AutA{2}$}{A2}} 
\label{amalgam.subsec}
The group $\AutA{2}$ is an \itind{amalgamated product} defined by the subgroups $\Aff:=\Aff(2)$ of affine automorphisms and $\JJJ:=\JJJ(2)$ of upper triangular automorphisms:
$$
\AutA{2} =  \Aff \star_{\BBB} \JJJ, \quad \BBB := \Aff\cap\JJJ.
$$
$\Aff$ are the automorphisms of degree 1 and $\JJJ$ are those of the form $(x,y)\mapsto (rx+p(y), sy+c)$ with $r,s,c\in\kk$,  $rs\neq 0$,  and $p\in\kk[y]$. This result goes back to \name{Jung} \cite{Ju1942Uber-ganze-biratio} and \name{van der Kulk} \cite{Ku1953On-polynomial-ring}.\idx{$\BBB$}

The {\it length of an element} $\g\in\AutA{2}$\idx{length}, denoted by $\ell( \g )$\idx{$\ell(\g)$}, is the minimal number $n$ such that $\g$ can be written as product of $n$ elements from $\Aff\cup\JJJ$. The elements of $\Aff\cup\JJJ$ are exactly those of length one, and every element $\g\notin\BBB$ has one of the following forms which are called {\it reduced expressions}\idx{reduced expression}:
\be
\item 
If $n=2m$ is even, then $\g=\a_{1}\cdot \t_{1}\cdot \a_{2}\cdot \t_{2}\cdots \a_{m}\cdot \t_{m}$ or $\g=\t_{1}\cdot \a_{1}\cdot \t_{2}\cdot \a_{2}\cdots \t_{m}\cdot \a_{m}$ 
where $\t_{i}\in\JJJ\setminus\BBB$ and $\a_{i}\in\Aff\setminus\BBB$.
\item 
If $n = 2m+1$ is odd, then $\g=\a_{1}\cdot \t_{1}\cdot \a_{2}\cdot \t_{2}\cdot \cdots \a_{m}\cdot \t_{m}\cdot \a_{m+1}$ or 
$\g=\t_{1}\cdot \a_{1}\cdot \t_{2}\cdot \a_{2}\cdots \t_{m}\cdot \a_{m}\cdot \t_{m+1}$
where $\t_{i}\in\JJJ\setminus\BBB$ and $\a_{i}\in\Aff\setminus\BBB$.
\ee
The reduced expressions are uniquely defined modulo the actions $\a\cdot \t\mapsto (\a \b^{-1}) \cdot (\b \t)$ and  $\t \cdot \a \mapsto (\t\b^{-1}) \cdot (\b \a)$ with $\b \in \BBB$. We will also use the fact that if $\f=\f_{1}\cdot \f_{2}\cdots \f_{k}$ is a reduced expression, then $\deg \f = \prod_{i}\deg \f_{i}$ (see \cite[Lemma 5.1.2]{Es2000Polynomial-automor}).

\begin{lemma}  \label{lem1}
Let $\g \in\AutA{2}$ be of even length,  and let $\g=\g_{1}\cdot \g_{2}\cdots \g_{n}$ be a reduced expression.
\be
\item If $\g'$ is conjugate to $\g$, then $\ell(\g')\geq \ell(\g)$.
\item If $\g'$ is conjugate to $\g$ and $\ell(\g') = \ell(\g)$, then $\g'$ is $\BBB$-conjugate to a cyclic permutation $\g_{i}\cdots \g_{n}\cdot \g_{1}\cdots \g_{i-1}$. Moreover, there is an $\f\in\AutA{2}$ such that $(\deg \f)^{2}\leq \deg \g$ and $\f^{-1}\cdot \g \cdot \f = \g'$.
\ee
\end{lemma}

\begin{proof}
(1) and (2):
Let $\g' = \f^{-1}\cdot \g \cdot \f$ and let $\f=\f_{1}\cdot \f_{2}\cdots \f_{k}$ be a reduced expression. Assume that $\g_{1}\in\JJJ\setminus\BBB$ and $\g_{n}\in\Aff\setminus\BBB$; the case where $\g_{1}\in\Aff\setminus\BBB$ and $\g_{n}\in\JJJ\setminus\BBB$ can be handled similarly by exchanging $\JJJ$ and $\Aff$. 

Assume that $\ell(\f_{1}^{-1}\cdot \g\cdot \f_{1}) > \ell(\g)$. If $\f_{1}\in\JJJ\setminus\BBB$, then $\g_{1}':=\f_{1}^{-1} \g_{1}\in\JJJ\setminus\BBB$, and so $\g^{(i)}:= (\f_{1}\cdots \f_{i})^{-1}\cdot \g\cdot (\f_{1}\cdots \f_{i}) = \f_{i}^{-1}\cdots \f_{2}^{-1}\cdot \g_{1}'\cdot \g_{2}\cdots \g_{n}\cdot \f_{1}\cdots \f_{i}$, and this  is a reduced expression for all $i$. Hence $\ell(\g^{(i)})= \ell(\g)+2i-1$ for $i=1,\ldots,k$. 
If $\f_{1}\in\Aff\setminus\BBB$, then $\g_{n}':=\g_{1}\cdot \f_{1}\in\Aff\setminus\BBB$, and so $\g^{(i)}= \f_{i}^{-1}\cdots \f_{1}^{-1}\cdot \g_{1}\cdots \g_{n-1}\cdot \g_{n}'\cdot \f_{2}\cdots \f_{i}$ is a reduced expression for all $i$. Hence again $\ell(\g^{(i)})= \ell(\g)+2i-1$ for $i=1,\ldots,k$. Thus we always get $\ell(\g')>\ell(\g)$.

Now assume that $\ell(\f_{1}^{-1}\cdot \g\cdot \f_{1}) \leq \ell(\g)$.  If $\f_{1}\in\JJJ\setminus\BBB$, then $\beta:=\f_{1}^{-1}\cdot \g_{1}\in \BBB$ and $(\beta\cdot \g_{2})\cdots \g_{n}\cdot \f_{1} = \beta \cdot (\g_{2}\cdots \g_{n}\cdot \g_{1})\cdot \beta^{-1}$, and this is a reduced expression of length $\ell(\g)$. 
Similarly,  if $\f_{1}\in\Aff\setminus\BBB$, then $\beta:=\g_{n}\cdot \f_{1}\in \BBB$ and $\f_{1}^{-1}\cdot \g_{1}\cdots \g_{n-1}\cdot (\g_{n}\cdot \f_{1}) = \beta^{-1}(\g_{n}\cdot \g_{1}\cdots \g_{n-1})\cdot \beta$, and this is a reduced expression of length $\ell(\g)$. 
Thus, by induction, we get (1) and the first claim of (2). As for the second, about degrees, we simply remark that 
\begin{eqnarray*}
\g_{i}\cdots \g_{n}\cdot \g_{1}\cdots \g_{i-1} &=& (\g_{1}\cdots \g_{i-1})^{-1}\cdot \g \cdot (\g_{1}\cdots \g_{i-1}) \\
&=& (\g_{i}\cdots \g_{n})\cdot \g \cdot (\g_{i}\cdots \g_{n})^{-1} 
\end{eqnarray*}
\end{proof}

\begin{remark}\label{Cmin.rem}
This lemma has some interesting consequences. Let $C\subseteq \AutA{2}$ be a conjugacy class which does not contain elements from $\JJJ\cup\Aff$. Then the elements in $C_{\text{min}}\subseteq C$ of minimal length consist of finitely many $\BBB$-conjugacy classes. All elements of $C_{\text{min}}$ have the same degree $d$ which is also minimal, and $C\cap\AutA{2}_{d} = C_{\text{min}}$.  Moreover, if $\g\in\AutA{2}$
has even length, then $\ell(\g^{n}) = |n|\ell(\g)$ and $\deg(\g^{n})= (\deg \g)^{|n|}$ for all $n\in\ZZ\setminus\{0\}$. In particular, the subgroup $\langle\g\rangle$ is \itind{discrete}, i.e. $\langle\g\rangle\cap\AutA{2}_{k}$ is finite for all $k$.
\end{remark}

\begin{lemma}\label{lem2}
For every $\g\in\AutA{2}$ there is an element $\f$ of degree $\leq \deg \g$ such that $\g':= \f^{-1}\cdot \g \cdot \f$ has either even length  or  length 1, and  $\deg \g' \leq \deg \g$. Moreover, if $\g$ is conjugate to an element of $\BBB$ we can arrange that $\g' \in\BBB$.
\end{lemma}

\begin{proof}
Let $\g$ be an element of odd length $n=2m+1\geq 3$ with reduced expression $\g=\g_{1}\cdots \g_{n}$. Then $\g_{1}$ and $\g_{n}$ both belong to $\JJJ \setminus \BBB$ or both to $\Aff\setminus \BBB$. Conjugating with $\g_{1}$ we get 
$$
\g_{1}^{-1}\cdot \g \cdot \g_{1} = \g_{2}\cdots \g_{n-1}\cdot (\g_{n}\cdot \g_{1}).
$$
If $\g_{n}\cdot \g_{1}\in\BBB$ the right hand side has odd length $n-2$ and $\deg(\g_{1}^{-1}\cdot \g \cdot \g_{1}) (\deg \g_{1})^{2}= \deg \g$, because $\deg \g_{1}=\deg \g_{n}$. Otherwise, it is a reduced expression of even length $n-1$. 

Now define $\g^{(i)}:= (\g_{1}\cdots \g_{i})^{-1}\cdot \g \cdot  (\g_{1}\cdots \g_{i})$ for $i\leq m$. Choose $k\leq m$ maximal such that $\g^{(1)},\ldots,\g^{(k)}$ all have odd length. Then $\g^{(k)}= \g_{k+1}\cdots \g_{n-k}\cdot \b$ where $\b\in\BBB$, and so $\ell(\g^{(k)})=\ell(\g) - 2k$ and $\deg \g^{(k)} (\deg \g_{1}\cdots \g_{k})^{2}= \deg \g$. If $k=m$ the claim follows with $\f := \g_{1}\cdots \g_{k}$.
If $k<m$, then $\g^{(k+1)} = \g_{k+2}\cdots (\g_{n-k}\cdot \b\cdot \g_{k})$ is of even length $\ell(\g)-2k-1\geq 2$, $\deg (\g_{1}\cdots \g_{k+1})\leq \deg \g$ and $\deg \g^{(k+1)}\leq \deg \g$. Putting $\f := \g_{1}\cdots \g_{k+1}$ the claim follows. 

Finally, assume that $\g$ is conjugate to an element of $\BBB$. Then $k=m$ and $\g^{(m)}= \g_{m+1}\cdot\b$. If $\g_{m+1}\in\JJJ\setminus\BBB$ and $\z^{-1}\cdot \g_{m+1}\cdot \z \in \BBB$ where $\z=\z_{1}\cdots \z_{j}$ is a reduced expression, then one easily sees that $\z = \z_{1}\in\JJJ\setminus\BBB$ and $\deg \z = \deg \g^{(k)}$. Similarly, we can handle the case $\g_{m+1}\in\Aff\setminus\BBB$. Thus the second claim follows with $\f:= \g_{1}\cdots \g_{k}\cdot \z$.
\end{proof}

\begin{remark} 
We have seen in the proof above that  if $\h\in \JJJ\cup\Aff$ and if $\g$ is conjugate to $\h$, then there is an automorphism $\f\in \AutA{2}$ such that $\g = \f^{-1}\cdot \h\cdot \f$ and $\deg \g = \deg \h \, (\deg \f)^{2}$. This is also shown in \cite[Lemma 4.1]{FuMa2007Locally-finite-pol}.  A similar formula cannot hold in general as seen from the example of an element $\g = \g_{1}\cdots \g_{n}$ of even length which is conjugate to $\g_{2}\cdots \g_{n}\cdot \g_{1}$ and not conjugate to any element of smaller length (see Lemma~\ref{lem1}).
\end{remark}

\begin{remark}\label{SAutA2.rem}
It is easy to see that the special automorphism group $\SAutA{2}$ is the amalgamated product of $\SAff(2)$ and $\SJ(2)$ over their intersection, with the obvious meaning  of $\SAff(2)$ and $\SJ(2)$. 
\end{remark}

\begin{lemma}[see \cite{FuMa2010A-characterization}]  \label{lengthone.lem}
The following statements for $\g \in \AutA{2}$ are equivalent.
\be
\item[(i)] $\g$ is triangularizable, i.e. conjugate to an element of $\JJJ$.
\item[(ii)] $\ell(\g^{2})\leq \ell(\g)$.
\item[(iii)] $\ell(\g^{n}) \leq \ell(\g)$ for all $n$.
\item[(iv)] $\deg(\g^{2})\leq \deg(\g)$.
\item[(v)] $\deg(\g^{n}) \leq \deg(\g)$ for all $n$.
\item[(vi)] $\g$ is of finite length.
\item[(vii)] $\g$ is locally finite.
\ee
\end{lemma}

\begin{proof} Let $\g=\g_{1}\cdots \g_{n}$ be a reduced expression. In the proof of Lemma~\ref{lem2} we have seen that $\g$ is conjugate to an element of length one if and only if $n=2k+1$ is odd and $(\g_{1}\cdots \g_{k})^{-1}\cdot \g\cdot (\g_{1}\cdots \g_{k})$ has length one, i.e. $\g_{k+2}\cdots \g_{n}\cdot \g_{1}\cdots \g_{k}=\b\in\BBB$. Since 
$$
\g^{2}= (\g_{1}\cdots \g_{k}) \cdot \g_{k+1}\cdot (\g_{k+2}\cdots \g_{n}\cdot \g_{1}\cdots \g_{k})\cdot \g_{k+1}\cdot (\g_{k+2}\cdots \g_{n})
$$
this is also equivalent to $\ell(\g^{2})\leq \ell(\g)$. From this, one immediately deduces the lemma.
\end{proof}

\begin{proposition}\label{prop1}
Let $\g,\h\in\AutA{2}$ be two conjugate elements. Then there is an $\f\in\AutA{2}$ of degree 
$\leq \max(\deg \g, \deg \h)^3$ such that $\g = \f^{-1}\cdot \h \cdot \f$.
\end{proposition}

\begin{proof}
By Lemma~\ref{lem2} we can find elements $\u,\vbold\in\AutA{2}$, $\deg \u\leq \deg \g$, $\deg \vbold \leq \deg \h$,  such that $\g':= \u^{-1}\cdot \g \cdot \u$ and  $\h':=\vbold^{-1}\cdot \h \cdot \vbold$ both have minimal length. If $\ell(\g')$ is even then, by Lemma~\ref{lem1}, $\ell(\h') = \ell(\g')$ and there is an $\z \in \AutA{2}$ of degree $\leq \min(\deg \g, \deg \h)$ such that $\z^{-1}\cdot \g'\cdot \z = \h'$. Putting $\f := \u \cdot \z \cdot \vbold^{-1}$ we get $\deg \f\leq \max(\deg \g, \deg \h)^3$ and $\f^{-1}\cdot \g \cdot \f = \h$.

If $\ell(\g')=1$, then $\g', \h'\in \JJJ\cup\Aff$. If $\g'$ is conjugate to an element of $\BBB$, then, by Lemma~\ref{lem2}, we can assume that $\g',\h'\in\BBB$. Then there is an element $\z\in\Aff$ such that $\z^{-1}\cdot \g' \cdot \z = \h'$, and the claim follows as above with $\f := \u \cdot \z \cdot \vbold^{-1}$. If $\g'\in\JJJ$ is not conjugate to an element of $\BBB$ and $\z^{-1}\cdot \g' \cdot \z = \h'$, then one easily sees by looking at a reduced expression of $\z$ that $\z\in\JJJ$ and $\deg \z \leq \max(\deg \g',\deg \h')$. Again, the claim follows as above. In a similar way we can handle the case where $\g'\in\Aff$.
\end{proof}

\begin{corollary}\label{ccwconst.cor}
For any $\g\in\AutA{2}$ the conjugacy class $C(\g)\subseteq\AutA{2}$  is weakly constructible. In particular, $C(\g)$ is closed in case $\g$ is semisimple.
\end{corollary}
\begin{proof} Consider the conjugating morphism $\phi\colon\AutA{2} \to \AutA{2}$, $\f \mapsto \f \cdot\g \cdot\f^{-1}$. The proposition above implies that for any $k\geq \deg\g$ we get $\phi(\AutA{2}_{k^{3}})\supseteq C(\g)\cap\AutA{2}_{k}$. Hence, the assumptions of Lemma~\ref{image-wconst.lem} are satisfied and thus the image $C(\g)$ is weakly constructible. If $\g$ is semisimple, then it is diagonalizable and so $C(\g)$ is weakly closed (Corollary~ \ref{ss-weakly-closed.cor}) and therefore closed (cf. point (4) of Remark~\ref{constr.rem}).
\end{proof}

\ps
\subsection{The group \texorpdfstring{$\AutA{2}$}{A2} is not linear}   
\label{Aut(A2)-is-not-linear.subsec}
The following result is due to \name{Yves de Cornulier} (\cite{Co2017Nonlinearity-of-so}).

\begin{proposition}[\name{Yves de Cornulier}, 2013]    \label{the-group-Aut(A2)-is-not-linear.prop}
The group $\Aut ( \AA^2)$ is not linear, i.e. it cannot be embedded into a linear group $\GL_d(K)$ where $K$ is a field.
\end{proposition}

\begin{proof}
Let us show that the subgroup $J^{u}$ of $\Aut (\AA^2)$ is not linear,
\[
J^{u}:= \{ (x+ p(y), y+c)\in \Aut(\Atwo) \mid p \in \kk [y], \, c \in \kk  \}. 
\]
Define the closed unipotent subgroups
$$
U_{k}:=\{(x+p(y),y+a)\in J \mid \deg p \leq k\} \supseteq U_{k}':= \{(x+p(y),y)\in J \mid \deg p \leq k\}.
$$
Then we have $(U_{k},U_{k}) = U_{k-1}'$ and $(U_{k},U_{\ell}') = U_{\ell-1}'$ for $k>\ell>0$, which implies that $U_{k}$ is nilpotent of class $k+1$. The next lemma shows that for any embedding $U_{k}\into \GL_{d}(K)$ where $K$ is a field we have $d\geq (k+1)^{2}$. Thus there is no embedding of $J$ into some $\GL_{d}(K)$.
\end{proof}

\begin{lemma} \label{nilpotency-class-unipotent-group.lem}
Let $U$ be a unipotent algebraic group of nilpotency class $n$. 
For any embedding $U \into \GL_{d}(K)$  where $K$ is a field we have  $n \leq d^2$.
\end{lemma}

\begin{proof}

We may assume that the field $K$ is algebraically closed. 
Let $U:=U_{0} \supset U_{1} \supset \cdots\supset U_{n}=\{e\}$ be the lower central series of $U$, i.e. $(U,U_{k}) = U_{k+1}$ for $k=0,\ldots, n-1$. It is well known that each $U_k$ is closed in $U$, so that $U_k$ is also a unipotent group (\cite[Proposition~17.2(a)]{Hu1975Linear-algebraic-g}).
Denote by $H:=\overline{U} \subset\GL_{d}(K)$ the Zariski-closure of $U$. Then the series
$$
H = \overline{U_{0}}\supseteq \overline{U_{1}}\supseteq\cdots\supseteq  \overline{U_{n}}=\{e\}
$$
has the property that $(H,\overline{U_{k}}) \subseteq \overline{U_{k+1}}$.
In fact, $(H,\overline{U_{k}})$ is generated by the image of the morphism $\gamma\colon H \times \overline{U_{k}} \to \GL_{d}(K)$, $(h,u)\mapsto huh^{-1}u^{-1}$. Since $\gamma(U \times U_{k}) \subseteq U_{k+1}$ it follows that $\gamma(H \times \overline{U_{k}}) \subseteq \overline{U_{k+1}}$, hence $(H,\overline{U_{k}}) \subseteq \overline{U_{k+1}}$.
As a consequence, $H$ is nilpotent of class $\leq n$. However, the nilpotency class of $H$ is $\geq n$, because $H \supseteq U$. In particular, we have $\overline{U_{k}}\neq \overline{U_{k+1}}$ for $k=0,\ldots,n-1$.

We claim that all $\overline{U_{k}}$ are connected. Then the lemma follows, because $\dim \overline{U_{k}} > \dim \overline{U_{k+1}}$ for $k=0,\ldots,n-1$, and so $n\leq \dim \overline{U} \leq d^{2}$.

In order to prove the claim we use that  the power map $u\mapsto u^{m}\colon U_{k} \to U_{k}$ is surjective for any integer $m\geq 1$, see Theorem~\ref{unipotent-groups.thm}(\ref{power-map}). Therefore, if $m :=[\overline{U_{k}}:\overline{U_{k}}^{\circ}]$, we get $U_{k} = \{u^{m}\mid u\in U_{k}\} \subseteq \overline{U_{k}}^{\circ}$, hence $\overline{U_{k}} = \overline{U_{k}}^{\circ}$.
\end{proof}

\ps
\subsection{No injection of \texorpdfstring{$\JJJ$}{J} into \texorpdfstring{$\GL_{\infty}$}{GLinfty}}   
\label{Jonq-does-not-embed-into-GLinfty.subsec}
We have already mentioned that the \name{de Jonqui\`eres} subgroup $\JJJ$ is nested (Proposition~\ref{jonq.prop}). However, there is no injective homomorphism of ind-groups
 into $\GL_{\infty}$ as the next result shows.

\begin{proposition}\label{no-inbedding-of-Jn-into-GLinfty.prop}
The \name{de Jonqui\`eres} subgroup $\JJJ\subseteq \Aut(\Atwo)$ does not admit an injective homomorphism of ind-groups
$\JJJ \into \GL_{\infty}$.
\end{proposition}
\begin{proof}
Denote by $\AAA \subseteq \JJJ(2)$ the group of automorphisms of $\Atwo$ of the form
$$
(x,y) \mapsto (x+yp(y) , \beta\cdot y)
$$ 
where $p(y) \in \kk [y] $ and $\beta \in \kst$. 
This is a closed nested ind-subgroup which is isomorphic to the semidirect product $\kst \rtimes (\kinfty)^{+}$ defined in Section~\ref{embeddings-into-GLinfty.subsec}. Since the latter has no such injection, by Proposition~\ref{no-injective-hom-of-A-into-GL_infty.prop}, the claim follows.
\end{proof}

\begin{question}\label{no-injection-unipotent.ques}
Does the proposition also hold for the unipotent part $\JJJ^{u}$?
\end{question}

\ps
\subsection{Unipotent automorphisms of \texorpdfstring{$\Atwo$}{A2} and \texorpdfstring{$\Cplus$}{C+}-actions}\label{unipotent.subsec}

\begin{proposition} \label{LND-are-closed.prop}
The set $\LNV(\AA^2)$ of locally nilpotent vector fields is closed in  $\VEC(\AA^2)$.
\end{proposition}

\begin{proof}
Let $D= a \, \partial _x + b \, \partial _b$ be a vector field of $\Atwo$, and put $m:= \max \{ \deg a, \deg b \}$. 
According to \cite[Theorem~1.3.52]{Es2000Polynomial-automor} (see also \cite[Proposition~8.4]{Fr2006Algebraic-theory-o}), $D$ is locally nilpotent if and only if $D^{m+2}x=D^{m+2}y=0$.
\end{proof}

The next result can be found in \cite{BhDu1997Kernel-of-locally-} (cf. \cite{BeEsMa2001Derivations-having}).

\begin{proposition}\label{Cplusfam.prop}
Let $\Phi=(\Phi_{z})_{z\in Z}$ be a family of $\kplus$-actions on $\A{2}$ where $Z$ is affine. Assume that $\Phi$ has no fixed points.
Then $\Phi$ is isomorphic to the trivial family $\Psi$ where $\Psi_{z}(s)=(x+s,y)$ for all $z\in Z$.
\end{proposition}
\begin{proof}
The locally nilpotent vector field $\delta$ corresponding to the $\kplus$-action on $Z \times \A{2}$ has the form $\delta=p\dx+q\dy$ where $p,q\in\OOO(Z)[x,y]$. The example above shows that $p=\ab{f}{y}$ and $q=-\ab{f}{x}$ for a suitable $f\in\OOO(Z)[x,y]$. Then $f$ is an invariant and so   $\phi=(\id,f)\colon Z \times \A{2} \to Z\times \kk$ is a $\kplus$-invariant morphism. We claim that $\phi$ is flat. Since every fiber of $\phi$ is reduced and isomorphic to $\kk$ this implies that $\phi$ is smooth. Hence $\phi$ is a $\kplus$-bundle which must be trivial, because $Z\times \kk$ is affine.

In order to see that $\phi$ is flat, we embed $Z \subseteq \kk^{m}$ and extend $f=\sum_{i j}f_{i j}x^{i}y^{j}$ to $\tilde f\in\OOO(\A{n})[x,y]$, $\tilde f = \sum_{i j}\tilde f_{i j}x^{i}y^{j}$. Then the fibers of $\tilde\phi:=(\id,\tilde f)\colon \kk^{m}\times\A{2} \to \kk^{m}\times \kk$ are one-dimensional (or empty) over the open subset $U\times\kk$ where $U$ is the complement of the zero set of the $\tilde f_{i j}$. Since $U$ is smooth this implies that $\tilde\phi\colon U \times \A{2} \to U\times\kk$ is flat by \cite[Corollary to Theorem~23.1]{Ma1989Commutative-ring-t} and the claim follows, because $U$ contains $Z$.
\end{proof}

Proposition~\ref{Cplusfam.prop} above has the following interesting interpretation in terms of the conjugacy class $C(\t)$ of a translation $\t\in\AutA{2}$. Let us first describe the situation of an action of an algebraic group $G$ on a variety $X$. In this case, the orbits $O_{x}:=G x \subset X$ are locally closed, hence varieties, and the orbit map $\mu_x\colon G \to O_{x}$ is a principal $G_{x}$-bundle, see Section~\ref{G-bundle.subsec}. Such bundles are locally trivial in the \'etale topology. This means that there is a surjective \'etale map $\eta \colon Z \to O_{x}$ such that the pull-back bundle $Z\times_{O_{x}}G$ is trivial, i.e. there is a section $\sigma\colon Z \to G$
\[
\begin{tikzcd}
Z \times_{O_x} G 
\arrow[r]
\arrow[d,xshift=-0.5ex]
& G
\arrow[d,"\mu_x"]
\\
Z
\arrow[r, twoheadrightarrow,"\eta"]
\arrow[ur, dashrightarrow, "\sigma"]
\arrow[u, xshift=0.5ex]
& O_x
\arrow[r,hook,"\subseteq"]
& X
\end{tikzcd}
\]
If $\eta\colon Z \to O_{x}$ is any morphism, then it is an interesting question whether there is a section. A typical example is the following.
If the stabilizer $G_{x}$ is a unipotent group, then we have a section for any morphism $Z \to O_{x}$ where $Z$ is affine, because principal bundles for unipotent groups are trivial over affine varieties. Another case is when $G_{x}$ is a connected solvable group. Then one has a section for any $Z \to O_{x}$ in case $Z$ is factorial. 

In case of ind-group $\GGG$ and an orbit map $\GGG \to O_{x}$ we can ask the same question. {\it For which morphisms $Z \to O_{x}$ do we have a section $\sigma\colon Z \to \GGG$?} The only case known to us is the following which is just another formulation of  Proposition~\ref{Cplusfam.prop} above.

\begin{proposition} Let $\t \in \Aut(\Atwo)$ be a translation, and let $C(\t) \subset \Aut(\Atwo)$ be its conjugacy class.
For every morphism $\phi\colon Z \to C(\t)\subseteq \Aut(\Atwo)$ where $Z$ is an affine variety, there is a section $\sigma\colon Z \to \AutA{2}$:
\[
\begin{tikzcd}
& \Aut(\Atwo)
\arrow[d,"\mu_{\t}"]
\\
Z
\arrow[r, "\phi"]
\arrow[ur, dashrightarrow, "\sigma"]
& C(\t)
\arrow[r,hook,"\subseteq"]
& \Aut(\Atwo)
\end{tikzcd}
\]
\end{proposition}

\ps
\subsection{Classification of \texorpdfstring{$\kplus$}{kplus}-actions on \texorpdfstring{$\Atwo$}{Atwo}}
We finish this section with the following result due to \name{Rentschler} \cite{Re1968Operations-du-grou} describing the $\kplus$-actions on $\Atwo$.

\begin{proposition}\label{rentschler.prop}
Every $\kplus$-action on $\Atwo$ is conjugate to a modification of the translation action, i.e. to an action of the form $s\colon (x,y) \mapsto (x,y +  sh(x))$ for some polynomial $h \in \kk[x]$.
\end{proposition}

\begin{proof}
Let $\rho\colon \kplus \to \Aut(\Atwo)$ be a nontrivial $\kplus$-action on $\Atwo$, and let $\delta = p\dx + q\dy$ be the corresponding  locally nilpotent vector field. Since $\Div(\delta) = \ab{p}{x}+\ab{q}{y}=0$ by the corollary above, there exists an $r\in\kk[x,y]$ such that $p=-\ab{r}{y}$ and $q=\ab{r}{x}$. In particular, $\delta r=0$, hence $r$ is an invariant.

It is clear that the affine quotient $\Atwo\quot\kplus$ is isomorphic to $\Aone$. In fact, it is one-dimensional, normal, rational and without non-constant invertible functions. It follows that $\kk[x,y]^{\kplus}=\kk[f]$ where $f\colon \Atwo \to \Aone$ is the quotient map. The generic fibers of $f$ are  
$\kplus$-orbits, and so the famous Embedding Theorem of \name{Abhyankar-Moh-Suzuki} (see \cite{AbMo1975Embeddings-of-the-}, \cite{Su1974Proprietes-topolog}) implies that $f$ is a variable. Since $r = h(f)$ for some polynomial $h \in \kk[t]$ we get  
$$
p=-\ab{r}{y}=-h'(f)\ab{f}{y} \text{ and } q=\ab{r}{x}=h'(f)\ab{f}{x}.
$$ 
As a consequence, $\rho$ is a modification of the action $\tilde\rho$ corresponding to the vector field $\tilde\delta:=-\ab{f}{y}\dx + \ab{f}{x}\dy$. Clearly, this action $\tilde\rho$ is free, because $\tilde\delta$ has no zeroes. 

Let  $g$ be another variable such that $\kk[x,y]=\kk[f,g]$ and that $\jac(f,g)=1$. Then the automorphism $\phi:=(f,g)$ of  $\Atwo$ induces a linear automorphism of the vector fields (Section~\ref{LF-Jordan-decomposition.subsec}). We claim that the image of $\tilde\delta$ under $\phi$ is $\dy$: $\phi(\tilde\delta) = \dy$. 
In fact, let $\phi^{-1}=(f_{0},g_{0})$. Then, by Example~\ref{image-of-ddx.exa}, we get
$$
\phi^{-1}(\dy) = \ab{f_{0}}{y}(\phi(x))\dx + \ab{g_{0}}{y}(\phi(x))\dy.
$$
Using the identities $f_{0}(f,g)=x$ and $g_{0}(f,g)=y$ together with $\jac(f,g)=1$ one finds $\ab{f_{0}}{y}(\phi(x))=-\ab{f}{y}$ and 
$\ab{g_{0}}{y}(\phi(x)) = \ab{f}{x}$, hence the claim.

It follows that $\tilde\rho$ is conjugate under $\phi$ to the action $(x,y)\mapsto (x,y+s)$, and so $\rho$ is conjugate to a modification of this action. 
\end{proof}

\ps
\subsection{A characterization of semisimple elements of \texorpdfstring{$\Aut (\AA^2)$}{Aut(A2}} \label{main.subsec}

Recall that every locally finite element of $\AutA{2}$ is triangularizable (see Lemma~\ref{lengthone.lem} for a more general statement). Statements (\ref{Locally-finite-elements-are-closed}), (\ref{Unipotent-elements-are-closed}) and (\ref{An-element-is-semisimple-if-and-only-if-its-conjugacy-class-is-closed}) of the next theorem can be found in  \cite{FuMa2010A-characterization}.

\begin{theorem} 
\be
\item \label{Locally-finite-elements-are-closed} 
The subset $\AutlfA{2} \subseteq \AutA{2}$ of locally finite elements is closed.
\item \label{Unipotent-elements-are-closed} 
The unipotent elements  $\AutUA{2} \subseteq \AutA{2}$ form a closed subset.
\item \label{Conjugacy-classes-are-weakly-constructible} 
Conjugacy classes in $\AutA{2}$ are weakly constructible.
\item \label{The-semisimple-part-of-a-locally-finite-element-belongs-to-the-closure-of-its-conjugacy-class} 
If $\g \in \Aut ( \AA^2) $ is locally finite, then  $\g_{s} \in \wc{C(\g)}$.
\item  \label{An-element-is-semisimple-if-and-only-if-its-conjugacy-class-is-closed} 
If $\g \in \Aut ( \AA^2) $, then its conjugacy class $C(\g)$ is closed if and only if $\g$ is semisimple.
\item \label{The-closure-of-the-conjugacy-class-of-a-unipotent-element}  
If $\u$ is unipotent and nontrivial, then $\wc{C(\u)} = \overline{C(\u)}=\AutUA{2}$.
\item \label{Weakly-closed-normal-subgroups-of-Aut(A2)} 
Any weakly closed (resp. closed) nontrivial normal subgroup of $\AutA{2}$ is equal to the preimage of a weakly closed (resp. closed) subgroup  of $\kk^*$ by the jacobian map $\jac \colon \Aut ( \AA^2 ) \to \kk^*$.
\ee
\end{theorem} \label{constr.thm}

The proof of the theorem needs some preparation. It will be given at the end of this section.

For completeness we collect without proofs some results for non-locally finite elements of $\AutA{2}$ in the following theorem. A lot of material about these maps can be found in the papers \cite{FrMi1989Dynamical-properti} by \name{Friedland-Milnor} and \cite{La2001Lalternative-de-Ti} by \name{Lamy}. In the proof of the theorem above we will only need part (\ref{a-hyperbolic-element-has-fixed-points}), the existence of fixed points.

\begin{theorem}\label{henon.thm}
Assume that $\f \in\AutA{2}$ is not locally finite.
\be
\item \label{a-hyperbolic-element-is-conjugate-to-a-product-of-Henon-maps}
\cite[Theorem 2.6]{FrMi1989Dynamical-properti} $\f$ is conjugate to a product of elements of the form $(y,-\delta x + p(y))$ 
where $\delta\in\Cst$ and $\deg p\geq 2$. Moreover, one can assume that the highest coefficient of $p$ is $1$ and the next highest coefficient is zero.
\item \label{a-hyperbolic-element-has-fixed-points}
\cite[Theorem 3.1]{FrMi1989Dynamical-properti} $\f$ has fixed points. More precisely, the fixed points form a finite scheme of length $d:=\min\{\deg \g \mid \g\in C(\f)\}\geq 2$.
\item  \label{a-hyperbolic-element-generates-a-discrete-subgroup}
The subgroup $\langle \f \rangle$ is discrete, i.e. $\langle \f \rangle\cap\AutA{2}_{k}$ is finite for all $k$ (Remark~\ref{Cmin.rem}).
\item \label{centralizer-of-a-hyperbolic-element}
\cite[Proposition~4.8]{La2001Lalternative-de-Ti} The centralizer $\AutA{2}_{\f}:=\{\g\in\Aut(\Atwo)\mid \g\cdot \f = \f \cdot \g\}$ is isomorphic to $\ZZ \rtimes \ZZ/d$.
\ee
\end{theorem} 

\begin{proof}[Proof of Theorem~\ref{constr.thm}]
(\ref{Locally-finite-elements-are-closed}) If $\g$ is locally finite of degree $\leq k$, then $\deg \g^{m}\leq k$ for all $m\in\ZZ$ (Lemma~\ref{lengthone.lem}), and so $(\g^{*})^{m}x, (\g^{*})^{m}y \in \kk[x,y]_{\leq k}$ for all $m\in\ZZ$. This implies that 
$$
\AutlfA{2}\cap \AutA{2}_{k} = \{\g\in\AutA{2}_{k}\mid
\dim\Span\{(\g^{*})^{m}x,(\g^{*})^{m}y\mid m\in\ZZ\}\leq \textstyle{\binom{k+1}{2}}\}.
$$
Hence, $\AutlfA{2}\cap \AutA{2}_{k}$ is closed in $\AutA{2}_{k}$ for all $k$. This proves (\ref{Locally-finite-elements-are-closed}). 

(\ref{Unipotent-elements-are-closed}) If $\u$ is unipotent of degree $\leq k$, then a similar argument as in (\ref{Locally-finite-elements-are-closed}) shows that $(\u^{*}-\id)^{N}x = (\u^{*}-\id)^{N}y=0$ for $N\geq \binom{k+1}{2}$. Hence, $\AutUA{2} \cap \AutA{2}_{k}$ is closed in $\AutA{2}_{k}$, proving (2).

(\ref{Conjugacy-classes-are-weakly-constructible}) follows from Corollary~\ref{ccwconst.cor}, and (\ref{The-semisimple-part-of-a-locally-finite-element-belongs-to-the-closure-of-its-conjugacy-class}) from Proposition~\ref{triang.prop}(d) since every locally finite element in $\AutA{2}$ is triangularizable (Lemma~\ref{lengthone.lem}).

(\ref{An-element-is-semisimple-if-and-only-if-its-conjugacy-class-is-closed}) is a consequence of (\ref{The-semisimple-part-of-a-locally-finite-element-belongs-to-the-closure-of-its-conjugacy-class}) and Corollary~\ref{ccwconst.cor}  for locally finite elements. For the general case we use Theorem~\ref{henon.thm}(2) which shows that the conjugacy class of a non-locally finite element is not closed.

(\ref{The-closure-of-the-conjugacy-class-of-a-unipotent-element}) is a consequence of (\ref{Unipotent-elements-are-closed}) and Proposition~\ref{triangclosure.prop}, because every unipotent element is conjugate to an element of $\Ju(2)$.

(\ref{Weakly-closed-normal-subgroups-of-Aut(A2)}) is Proposition~\ref{simple.prop}.
\end{proof}

\ps
\subsection{The closure of \texorpdfstring{$\Aut(\Atwo)$ in $\End ( \Atwo)$ }{Aut(A2) in End(A2) }}
We have seen above that $\Aut(X)$ is closed in the dominant endomorphisms $\Dom(X)$ and that $\Dom(X)$ is open in all endomorphisms $\End(X)$. So an interesting problem is to determine the closure of $\Aut(X)$. In case of affine 2-space $\Atwo$ we have the following  answer. Recall that a regular function $v\in\OOO(\An)$ is called a \itind{variable} if there exist $v_{2},\ldots,v_{n} \in\OOO(\An)$ such that $\OOO(\An)=\CC[v,v_{2},\ldots,v_{n}]$.

\begin{definition}
Denote by $E \subseteq \End(\Atwo)$ the subset of endomorphisms $f=(f_{1},f_{2})$ satisfying the following two conditions:
\be
\item The image $f(\Atwo) \subseteq \Atwo$ is contained in a line, or, equivalently, $\CC[f_{1},f_{2}]$ is a polynomial ring in one variable.
\item There exist polynomials $p_{1},p_{2}\in\CC[t]$ and a variable $v\in\OOO(\Atwo)$ such that $f_{i}=p_{i}(v)$.
\ee
\end{definition}

\begin{theorem}  \label{closure-of-Aut(Atwo).thm}
We have  $\overline{\Aut(\Atwo)}=\Aut(\Atwo)\cup E$. More precisely, 
$$
\overline{\Aut(\Atwo)_{k}} = \Aut(\Atwo)_{k}\cup E_{k} \text{ \ for all } k
$$
where $E_{k}:=E\cap\End(\Atwo)_{k}$.
\end{theorem}

\begin{remark}  \label{E-is-closed.rem}
Since $\Aut(\Atwo)$ is open in its closure (Theorem~\ref{AutX-locally-closed-in-EndX.thm}), the theorem implies that $E$ is closed in $\End(\Atwo)$. 

The second assertion says that the closure $\overline{\Aut(\Atwo)}$ coincides with the \itind{weak closure}, see Section~\ref{top.subsec}.
\end{remark}

For the proof of the theorem above we will use some consequences of the famous Embedding Theorem of \name{Abhyankar-Moh-Suzuki} (see \cite{AbMo1975Embeddings-of-the-}, \cite{Su1974Proprietes-topolog}) and of the amalgamated product structure of $\Aut(\Atwo)$. They are given in the following proposition.

\begin{proposition}  \label{AMS.prop}
Let $p,q \in \CC[t]$ such that $\CC[t]=\CC[p,q]$, and set  $k:= \deg(p,q):=\max\{\deg p,\deg q\}$. Then
\be
\item $\deg p$ divides $\deg q$ or $\deg q$ divides $\deg p$.
\item If $k>1$, then there is an automorphism $g$ of degree $d>1$ such that $d|k$ and $\deg (g\circ(p,q)) = \frac{k}{d}$.
\item There is a variable $v\in\CC[x,y]$ of degree $\leq k$ such that $v(p,q)=0$.
\item For every variable $v\in\CC[x,y]$ of degree $\geq 2$ there is a variable $w\in\CC[x,y]$ such that $\CC[x,y]=\CC[v,w]$ and $\deg w<\deg v$.
\ee
\end{proposition}

\begin{proof}
(1) This is one form of the main theorem of \name{Abhyankar-Moh-Suzuki}, see \cite[1.1~Main Theorem]{AbMo1975Embeddings-of-the-}. An equivalent form is given in \cite[Th\'eor\`eme~5]{Su1974Proprietes-topolog}.
\par\smallskip
(2) Assume first that $k=\deg q > \deg p$. Then $p\neq 0$, because $\CC[p,q]=\CC[t]$, and by (1), $\deg p$ divides $k$. It follows that  $\deg (q - \alpha p^{k'})<\deg q$ for $k':=\frac{k}{\deg p}$ and a suitable $\alpha \in \Cst$. Clearly, $(p,q'):= (p,q - \alpha p^{k'}) = (x, y - \alpha x^{k'})\circ(p,q)$ still satisfies the assumptions. Thus we can proceed in the same way to find a polynomial $r(x)$ of degree $k'$ such that $(x,y-r(x))\circ(p,q) = (p,q - r(p))$ has degree equal to $\deg p$. This proves the claim for $\deg p\neq \deg q$. If $\deg p = \deg q$ we can first apply a linear automorphism to get a pair $(p,q)$ such that $\deg p \neq \deg q$.
\par\smallskip
(3) It follows from (2), by induction, that there is an automorphism $(v,w)$ of degree $\leq k$ such that $(v,w)(p,q) = (0,b)$ with a linear $b = \alpha x +\beta y +\gamma$. In particular, $v(p,q)=0$.
\par\smallskip
(4) This follows immediately from the structure of $\Aut(\Atwo)$ as an amalgamated product together with the fact that the degree of a reduced expression is equal to the product of the degrees of the factors. 
\end{proof}

\begin{proof}[Proof of Theorem~\ref{closure-of-Aut(Atwo).thm}]
(a) We first show that $E_{k}$ is contained in the closure of $\Aut(\Atwo)_{k}$. This clearly holds for $k=1$. If  $f=(p_{1}(v),p_{2}(v))\in E$ has degree $k>1$, then, by Proposition~\ref{AMS.prop}(2), there is an automorphism $g$ of degree $d>1$ such that $\deg(g\circ (p(v),q(v))=k'$ where $k' :=\frac{k}{d}<k$. By induction, we get $g\circ(p,q) \in \overline{\Aut(\Atwo)_{k'}}$ which implies that  $(p,q) \in g^{-1}\circ\overline{\Aut(\Atwo)_{k'}} \subseteq \overline{\Aut(\Atwo)_{k}}$.
\par\smallskip
(b) Next we show that $\overline{\Aut(\Atwo)_{k}} \subseteq \Aut(\Atwo)_{k}\cup E_{k}$. Choose an $f\in \overline{\Aut(\Atwo)_{k}}$ which is not an automorphism. If $f$ is a constant map, we are done. Otherwise, $\dim f(\Atwo)=1$. Since $\Aut(\Atwo)$ is open in $\overline{\Aut(\Atwo)}$ we can find an irreducible affine and factorial curve $C$, a point $c_{0}\in C$ and a morphism $\phi\colon C \to \overline{\Aut(\Atwo)}$ such that $\phi(c_{0})=f$ and $\phi(c) \in \Aut(\Atwo)$ for all $c\neq c_{0}$. Thus we get a family of endomorphisms $\phi=(\phi_{c})_{c\in C}$ parametrized by the curve $C$. Now we use   a famous result of \name{Sathaye}'s to conclude that if $\dim f(\Atwo) =1$, then $ f(\Atwo)\subseteq \Atwo$ is a line, see \cite[Remark~2.1]{Sa1983Polynomial-ring-in}. Now we claim that $f =(p_{1}(v),p_{2}(v)$ with polynomials $p_{1},p_{2}\in \CC[t]$ and a variable $v$. 

If $\deg f=1$, then the claim is obvious. So let us assume that $\deg f = k \geq 2$ and that $\deg f_{1}=k \geq \deg f_{2}$. If $f_{2}\neq 0$, then, by \cite[Theorem~2]{Sa1983Polynomial-ring-in}, we see that $\deg f_{1}$ is a multiple of $\deg f_{2}$, $k = d\cdot \deg f_{2}$,  and so $\tilde f_{1}:=f_{1} - \alpha f_{2}^{d}$ has lower degree than $f_{1}$ for a suitable $\alpha\in\Cst$. Replacing the family $(\phi)_{c\in C}$ by $\tilde\phi:=(g\circ\phi_{c})_{c\in C}$ where $g = (x-\alpha y^{d},y)\in\Aut(\Atwo)$ we still have that $\tilde\phi_{c}$ is an automorphism for $c\neq c_{0}$ and that $\tilde\phi_{c_{0}}=(\tilde f_{1},f_{2})$ has a one-dimensional image. If $\tilde\phi_{c_{0}}$ belongs to $E$, then the same holds for $\phi_{c_{0}}$. In fact, if $\tilde f_{1}=p_{1}(v)$ and $f_{2}=p_{2}(v)$, then $f_{1}= \tilde f_{1}+\alpha f_{2}^{d}=p_{1}(v) +\alpha p_{2}(v)^{d}$. Thus we can proceed to reach a situation where $\deg \tilde\phi_{c_{0}}$ is strictly less than $\deg\phi_{c_{0}}$ and then conclude by induction, except in case $\deg f_{1}=\deg f_{2}$ and $\tilde f_{1}=0$. 

In this case, we use the assumption that $C$ is a factorial curve which implies that the maximal ideal $\mm_{c_{0}}$ is principal, $\mm_{c}=(s) \subseteq \OOO(C)$. The family $\tilde\phi$ has the form $\tilde\phi = (\tilde F_{1},F_{2})$ where $\tilde F_{1}, F_{2}\in\OOO(C)[x,y]$, and so $\tilde f_{1}=0$ implies that $\tilde F_{1}$ is divisible by $s$. We can write $\tilde F_{1}=s^{\ell} \cdot G_{1}$ such that $g_{1}:=\tilde G_{1}|_{c=c_{0}} \neq 0$, and we replace the family $\tilde\phi$ by the family $\psi := (G_{1},F_{2})$.

Now there are two possibilities. If $\psi_{c_{0}}$ is an automorphism, then $f_{2}$ is a variable, and we are done. If not, then the image of $\psi_{c_{0}}$ is a line and we can proceed as above. Again we note that if $\psi_{c_{0}} \in E$, then $\tilde\phi_{c_{0}}\in E$, because $f_{2}=p_{2}(v)$, hence $\phi_{c_{0}} \in E$ as we have seen above.
\end{proof}

Denote by $C \subseteq \CC[x,y]$ the set of variables. The following result can be found in \cite[Theorem~4]{Fu2002On-the-length-of-p}.

\begin{corollary}\label{closure-of-variables.cor}
The closure of the set of variables $C \subset \CC[x,y]$ is given by
$$
\overline{C}=\{p(v)\mid p\in\CC[t], v\in C\}.
$$
\end{corollary}

\begin{proof}
Define $W:= \{p(v)\mid p\in\CC[t], v\in C\}$, and set $W_{k}:=W \cap \CC[x,y]_{\leq k}$. Similarly, we put  $C_{k}:=C \cap \CC[x,y]_{\leq k}$.
\ps
(a) We first note that $W_{k}\subseteq \overline{V_{k}}$. In fact, if $p(v)\in W_{k}$, then there is a variable $w$ of degree $\leq \deg v$ such that $(v,w) \in\Aut(\Atwo)$. Hence, for all $t\neq 0$, $t w + p(v)$ is a variable of degree $\leq k$, and the claim follows.

\ps
(b) It remains to show that $W$ is closed in $\CC[x,y]$. The set $E \subseteq \End(\Atwo)=\CC[x,y]^{2}$ defined above is closed (Remark~\ref{E-is-closed.rem}). Define $E_{k}:=E \cap \End(\Atwo)_{k}$. Since the first projection $\pr_{1}\colon \CC[x,y]_{\leq k}\oplus \CC[x,y]_{\leq k} \to  \CC[x,y]_{\leq k}$ is the affine quotient under the action of $\Cst$ by scalar multiplication on the second summand, and since $E_{k}$ is closed and stable under this action we see that $\pr_{1}(E_{k}) \subseteq \CC[x,y]_{\leq k}$ is closed. By construction, $\pr_{1}(E_{k})=W_{k}$. In fact, for any $p(v)\in W_{k}$, we have $(p(v),0)\in E_{k}$. Thus, $W=\bigcup_{k}W_{k}$ is closed in $\CC[x,y]$, and the claim follows.
\end{proof}

\begin{remark}
Here is a short proof of Corollary~\ref{closure-of-variables.cor}. As above we define  $W:= \{ p(v) \mid p \in \CC [t], v \in V \}$. The inclusion $W \subseteq \overline{V}$ is clear, so let us prove that $W$ is closed. By \cite[Corollary 4.7]{Fr2006Algebraic-theory-o}, we know that a polynomial $w \in \CC [x,y]$ belongs to $W$ if and only if the \itind{jacobian derivation} $q \mapsto \jac (w,q)$ of $\CC [x,y]$ is locally nilpotent. This is a consequence of a result of \name{Rentschler}  (\cite{Re1968Operations-du-grou}) asserting that any locally nilpotent derivation of $\CC[x,y]$ is conjugate (by an automorphism of $\CC[x,y]$) to a triangular derivation $p(x) \ddy$, see Proposition~\ref{rentschler.prop}. Let $\phi  \colon \CC [x,y]  \to \VEC (\AA^2)$ be the morphism sending $p \in \CC [x,y]$ to the derivation $q \mapsto \jac (p,q)$. We have $W= \phi ^{-1}( \LNV (\AA^2) )$.  By Proposition~\ref{LND-are-closed.prop}, $\LNV (\AA^2)$ is closed in $\VEC (\AA^2)$, and so $W$ is closed in $\CC [x,y]$.
\end{remark}

\pmed
\section{Some Special Results about \texorpdfstring{$\Aut(\AA^3)$}{Aut(An)}} 
\label{section: Some special results on Aut(A^3)}
\subsection{A braid group action on \texorpdfstring{$\AA^{3}$}{A3}}\label{braid.subsec}
The following construction was indicated to us by \name{Daan Krammer}. 
Denote by $F_2:=\langle a,b \rangle$ the free group in two generators $a,b$. Then the automorphism group $\Aut(F_2)$ acts algebraically on $\Hom(F_2,\SLtwo) \simto \SLtwo\times \SLtwo$. This action is faithful, because there exist injective homomorphisms $F_{2}\to \SLtwo$. Since the action commutes with the action of $\SLtwo$ by conjugation it follows that $\Aut (F_2)$ acts on the affine quotient 
\[ 
(\SLtwo\times\SLtwo)\quot\SLtwo \simto \A{3}
\]
where the isomorphism is given by $(A,B)\mapsto (\tr A, \tr B, \tr AB)$. We want to describe the image of $\Aut(F_{2})$ in $\Aut(\AA^{3})$ which 
we denote by $\bG$.

Let us identify an automorphism $\phi \in \Aut (F_2)$ with the pair $(\phi(a),\phi(b))$. The group $\Aut (F_2)$ is generated by the following three automorphisms  \cite{Ne1933Die-Automorphismen}:
\[ 
\alpha=(ab,b),\quad\beta=(a^{-1},b),\quad\gamma=(b,a).
\]
An easy calculation shows that the corresponding automorphisms of $\A{3}$ are given by
\[ 
\balpha =(z,y,y z - x),\quad \bbeta =(x,y,x y - z),\quad \bgamma =(y,x,z),
\]
and that the function $I:=x^{2}+y^{2}+z^{2}-x y z$ is an invariant. Note that the image of $(a^{-1},ab)$ is $(x,z,y)$ and of 
$(a b,b^{-1})$ is $(z,y,x)$ which shows that $S_{3}\subseteq \GL_{3}(\ZZ)$ is included in $\bG$. 

Clearly, the inner automorphisms of $F_{2}$ act trivially on the affine quotient, so that we get an action of the outer automorphisms group $\Out (F_2)$ on $\Athree$. The abelianization map $F_2 \to \ZZ^2$ induces a homomorphism $\Aut (F_2) \to \GL _2(\ZZ)$ which is trivial on inner automorphisms, hence factors through $\Out (F_2)$, and  the induced map is an isomorphism $\Out(F_{2})\simto\GL_{2}(\ZZ)$ (\name{Nielsen} \cite{Ni1924Die-Gruppe-der-dre}, cf. \cite[Corollary N4, p. 169]{MaKaSo1976Combinatorial-grou}). 

The automorphism $(a^{-1},b^{-1}) \in \Aut (F_2)$ acts trivially on $\A{3}$ and its image in $\GL_2(\ZZ)$ is $\left[ \begin{smallmatrix} -1& \phantom{-}0 \\ \phantom{-}0& -1 \end{smallmatrix} \right]$. Thus we finally get an action of $\PGL_2(\ZZ)$ on $\A{3}$.

\begin{proposition} \label{a faithful action of PGL2 on A3.prop}
The action of $\PGL_2(\ZZ)$ on $\A{3}$ is faithful. It has two fixed points, $(0,0,0)$ and $(2,2,2)$, with non-equivalent tangent representations, and the ring of invariants is generated by $I:=x^{2}+y^{2}+z^{2}-x y z$. 
\end{proposition}

\begin{proof}
(a) The faithfulness is proved in the following proposition where we even show that $\PGL_2(\ZZ)$ acts faithfully on all the hypersurfaces $H_{c}:=\VVV(I-c)\subseteq \AA^{3}$, $c\in\kk$.

\par\smallskip
(b) Consider the following quadratic involutions from $\bG$:
\begin{align*}
\ii_{1} &:= (x,y,x y - z) = \bbeta,\\
\ii_{2} &:= (x,x z - y,z) = (x,z,y)\cdot \bbeta\cdot (x,z,y), \\  
\ii_{3} &:= (y z - x,y,z) = (z,y,x)\cdot \bbeta\cdot (z,y,x). 
\end{align*}
Then $\ii_{1}\cdot\ii_{2} = (x,x z - y,(x^{2}-1)z-x y)$ is linear over $\kk(x)$ with matrix $\left[\begin{smallmatrix} -1 & x\\ -x & x^{2}-1 \end{smallmatrix}\right]$ which is a semisimple automorphism where both eigenvalues are transcendental over $\kk$. Let $\Cb{x}$ be the algebraic closure of $\kk(x)$ in $\Cb{x,y,z}$. It follows that all orbits of the group $\bG_{1}:=\langle \ii_{1}\cdot \ii_{2}\rangle$ acting on $\Cb{x}^{2}$ are infinite, except $\{0\}$. Since $I - c$ is irreducible for almost all $c\in \Cb{x}$, this implies that for almost all $p\in \Cb{x}^{2}$ the orbit $\bG_{1}p$ is dense in the zero set $\VVV(I - I(p))\subseteq \Cb{x}^{2}$. Hence $\Cb{x}[y,z]^{\bG_{1}}=\Cb{x}[I]$.

The same holds for the group $\bG_{2}$ generated by $\ii_{2}\cdot\ii_{3} = (y z-x, (z^{2}-1)y - x z, z)$ which is linear over $\kk(z)$, and we get $\Cb{z}[x,y]^{\bG_{2}} = \Cb{z}[I]$. We have $\Cb{x}\cap \Cb{z} = \kk$. In fact, if $f \in \Cb{x,y,z}\setminus \kk$ is algebraic over $\kk(x)$ and over $\kk(z)$, then $x,z$ are both algebraic over $\kk(f)$ which is a contradiction. Thus we finally get 
\[
\kk[x,y,z]^\bG \subseteq \kk[x,y,z]^{\bG_{1}} \cap \kk[x,y,z]^{\bG_{2}} 
\subseteq \Cb{x}[I] \cap \Cb{z}[I] = (\Cb{x}\cap\Cb{z})[I] = \kk[I],
\]
hence $\kk[x,y,z]^\bG = \kk[I]$.

\ps
(c) Looking at the generators $\balpha,\bbeta,\bgamma$ one easily sees that $(\AA^{3})^{\bG} = \{(0,0,0), (2,2,2)\}$. Moreover, a short calculation shows that the image of $\bG$ in $\GL(T_{(0,0,0)}\AA^{3})=\GL_{3}(\kk)$ is the finite group $T_{\ZZ}\rtimes S_{3}$, where $T_{\ZZ}\simto (\ZZ/2)^{3}$ are the diagonal matrices of $\GL_{3}(\ZZ)$. On the other hand, the image of $\balpha$ in $\GL(T_{(2,2,2)}\AA^{3})$ is nontrivial and has trace equal to $3$, hence has infinite order.
\end{proof}

\ps
\subsection{Surfaces with a discrete automorphism group}
The following results have been pointed out to us by \name{Serge Cantat} and \name{St\'ephane Lamy}. 
For $c\in \kk$ define the cubic hypersurface $H_{c}:=\VVV(I - c) \subseteq \A{3}$ which is smooth for $c\neq 0,4$, and is invariant under $\bG$.
Consider the \name{Klein} group
\[
V_4 := \Bigl\{(x,y,z), \ (x,-y,-z), \ (-x,y,-z),\ (-x,-y,z)\Bigr\} \subseteq \Aut(\AA^{3}).
\]
It stabilizes the invariant $I$, and is normalized by $\bG$. This is clear for $\bgamma=(y,x,z)$. For $\balpha$ and $\bbeta$ we get
$\balpha\cdot(-x,-y,z)\cdot\balpha^{-1}= (x,-y,-z)$, and  
$\bbeta\cdot(-x,-y,z)\cdot\bbeta^{-1}= (-x,-y,z)$,
and similarly for the other elements from $V_{4}$. It follows that the group $\tbG:=\langle \bG, V_{4} \rangle$ stabilizes $I$. Define $\Aut_{I}(\AA^{3}):=\{\g\in\Aut(\AA^{3})\mid \g^{*}(I)=I\}$ and $\Aut_{H_{c}}(\AA^{3}) = \{\g\in\Aut(\AA^{3})\mid \g(H_{c})=H_{c}\}$. 

\begin{proposition}\label{Aut-Hc.prop}
We have $\Aut_{I}(\AA^{3})=\tbG = \iii \rtimes S_{4}$, and the canonical maps 
$$
\begin{CD}
\tbG = \Aut_{I}(\AA^{3}) @>p_{c}>\simeq> \Aut_{H_{c}}(\AA^{3}) @>\res_{c}>\simeq> \Aut(H_{c})
\end{CD}
$$
are isomorphisms. In particular, $\Aut(H_{c})$ is discrete. Moreover,  $\iii$ is a free product of the subgroups $\langle \ii_{k}\rangle \simeq \ZZ/2$, $\bG=\iii\rtimes S_{3}$ and so $\tbG = V_{4}\rtimes \bG$. Finally, the action of $\PGL_{2}(\ZZ)$ on $\AA^{3}$ is faithful.
\end{proposition}

\begin{proof}
(1) We have $\bG = \langle \ii_{1},\ii_{2},\ii_{3}, S_{3}\rangle$. In fact,  $\langle \ii_{1},\ii_{2},\ii_{3}, S_{3}\rangle \subseteq \bG=\langle\balpha,\bbeta,\bgamma\rangle$, by definition, and to get the other inclusion we note that $\bbeta = \ii_{1}$, $\bgamma \in S_{3}$, and $\balpha=(z,y,x)\cdot\ii_{3}$. 
\par\smallskip
(2) By \cite[Theorem 2]{Ho1975Induced-automorphi}, we have $\Aut_{I}(\AA^{3}) = \tbG$. 
\par\smallskip
(3) Next we show that $p_{c}$ is surjective. Let $\g \in\Aut_{H_{c}}(\AA^{3})$, so that $\g^{*}(I-c) = \lambda(I-c)$ for some $\lambda\in\Cst$. Then $\g(H_{c+d}) = H_{c+\lambda d}$. But $H_{e}$ is singular only for $e=0,4$, and $H_{0}$ has a unique singular point in the origin whereas $H_{4}$ has 4 singular points. Thus the affine map $c+d \mapsto c+\lambda d$ has two fixed points, hence is the identity, and so $\lambda = 1$.
\par\smallskip
(4) Now we claim that $S_{4}\cap\iii = (\id)$. Since $S_{4}=V_{4}\rtimes S_{3}$ this implies that $\tbG = \iii \rtimes S_{4}=(\iii\times V_{4})\rtimes S_{3}= V_{4}\rtimes \bG$, and $\bG = \iii\rtimes S_{3}$. Since the image of $\iii$ in $\GL(T_{0}\AA^{3}) = \GL_{3}(\kk)$ is the group $T_{\ZZ}$ of diagonal matrices of $\GL_{3}(\ZZ)$ we get $S_{3}\cap\iii = (\id)$, and since no element from  $S_{4}\setminus S_{3}$ fixes $(2,2,2)$ we have $(S_{4}\setminus S_{3})\cap\iii=\emptyset$. The claim follows.
\par\smallskip
(5) By \cite[Theorem 2]{El1974Cubic-surfaces-of-}, $\Aut (H_c)$ is generated by  the image of $\langle \ii_1,\ii_2,\ii_3 \rangle$ and the group of affine transformations of $\A{3}$ preserving $H_c$.  It is easy to see that the latter is the linear group $S_4$. Hence the composition $\res_{c}\circ p_{c}\colon \tbG \to \Aut(H_{c})$ is surjective.
\par\smallskip
(6) Now we claim that this map is also injective. It is clear that $S_{4}\to\Aut(H_{c})$ is injective, because $S_{4}$ acts freely on the complement of a finite union of planes. Now let $\f = \jj \cdot s$ where $\jj \in \iii\setminus\{\id\}$, and $s\in S_{4}$, and assume that $\f|_{H_{c}}=\id$. This implies that the components $f_{i}$ of $\f$ have the form
$f_{1}= x+h_{1}(I-c), f_{2}= y+h_{2}(I-c), f_{3}= z+ h_{3}(I-c)$ with some polynomials $h_{1},h_{2},h_{3}$, and so  the leading terms of the $f_{i}$ are given by $\lt(f_{i}) = \lt(h_{i})x y z$. But these leading terms are the same as the leading terms of the components of $\jj$, because $S_{4}$ is acting by permuting the variables and changing some signs. Thus we end up with a contradiction to Lemma~\ref{free-product.lem} below, and the claim follows. This lemma also shows that $\iii$ is a free product of the subgroups $\langle\ii_{k}\rangle\simeq\ZZ/2$. 
\par\smallskip
(7) It remains to prove that the canonical map $\phi\colon \PGL_{2}(\ZZ) \to \bG$ is an isomorphism. The element $\gamma=(b,a)\in\Aut(F_{2})$ has image 
$(y,x,z)$ in $\bG$ and image $P_{1}:=\bsm 0&1\\ 1&0 \esm$ in $\PGL_{2}(\ZZ)$, hence $\phi(P_{1})=(y,x,z)$. 
Similarly, $(ab,b^{-1})$ has image $(z,y,x)$ in $\bG$ and image $P_{2}:=\bsm 1&\phantom{-}0\\ 1&-1 \esm$ in $\PGL_{2}(\ZZ)$, hence $\phi(P_{2})=(z,y,x)$. It is easy to see that the subgroup $S:=\langle P_{1},P_{2}\rangle \subseteq \PGL_{2}(\ZZ)$ is isomorphic to $S_{3}$, and so $\phi$ induces an isomorphism  $S\simto S_{3}$. In a similar way we find the elements  $I_{1}:=\bsm -1&0\\\phantom{-}0 &1\esm, I_{2}:=\bsm -1&2\\ \phantom{-}0 &1\esm$, $I_{3}:=\bsm 1&\phantom{-}0\\ 2 & -1\esm \in \PGL_{2}(\ZZ)$ with the property that $\phi(I_{k}) = \ii_{k}$ and $I_{k}^{2}=E$.  This implies that the subgroup $\langle I_{1},I_{2},I_{3}\rangle \subseteq \PGL_{2}(\ZZ)$ is mapped isomorphically onto $\iii \subseteq \bG$. Now one easily checks that $S$ permutes  
$\{I_{1},I_{2},I_{3}\}$, hence normalizes $\langle I_{1},I_{2},I_{3}\rangle$, and that $\PGL_{2}(\ZZ) = S \langle I_{1},I_{2},I_{3}\rangle$, because $\PGL_{2}(\ZZ)$  is generated by  $\bsm -1 & 0\\\phantom{-}0 & 1\esm = I_{1}$,  $\bsm 0 & 1\\ 1&0 \esm=P_{1}$, and $\bsm 1&1\\0&1\esm = P_{3}I_{1}$. The claim follows as well as the fact that $\PGL_{2}(\ZZ) = \langle I_{1},I_{2},I_{3}\rangle \rtimes S$.
\end{proof}

\begin{remark}
The subgroup $\langle I_{1},I_{2},I_{3}\rangle \subseteq \PGL_{2}(\ZZ)$ constructed in the proof above is the congruence subgroup
$$
\Gamma_{2}:= \{M \in \PGL_{2}(\ZZ) \mid M\equiv \bsm 1 & 0 \\ 0 & 1 \esm \mod 2 \},
$$
and the semi-direct product structure corresponds to the split exact sequence
$$
1 \to \Gamma_{2} \to \PGL_{2}(\ZZ) \to \SL_{2}(\FF_{2}) \to 1
$$
\end{remark}

\begin{lemma}\label{free-product.lem}
The group $\iii \subseteq \Aut(\AA^{3})$ is the free product $\langle\ii_{1}\rangle * \langle\ii_{2}\rangle * \langle\ii_{3}\rangle$. 
Moreover, for $(f_{1},f_{2},f_{3}) \in \iii\setminus\{\id\}$ the leading terms $\lt(f_{i})$ are monomials in only two of the variables $x,y,z$, and they satisfy $\lt(f_{i})=\lt(f_{j})\lt(f_{k})$ for some $i\in\{1,2,3\}$ where $\{i,j,k\} = \{1,2,3\}$.
\end{lemma}

\begin{proof}
Let $\jj=\jj_{m}\cdot\jj_{m-1}\cdots\jj_{1}$, $\jj_{r}\in\{\ii_{1},\ii_{2},\ii_{3}\}$, be a reduced expression, and let $\jj=(f_{1},f_{2},f_{3})$. Now we claim the following.
\par\smallskip
{\it If $\jj_{m}=\ii_{i}$, then the leading term of $f_{i}$ is equal to the product of the leading terms of the two other $f_{j},f_{k}$. Moreover, these leading terms are  monomials in the same two variables as the leading term of $\jj_{1}$.}
\par\smallskip 
This follows by induction. By symmetry, we can assume that $\jj_{k} = \ii_{1}$, so that $\lt(f_{1})=\lt(f_{2})\lt(f_{3})$. Then $\ii_{2}\cdot\jj = (f_{1},f_{1}f_{3}-f_{2},f_{3})$ and $\lt(f_{1}f_{3}-f_{2}) = \lt(f_{1}f_{3}) = \lt (f_{1})\lt(f_{3})$. The case $\ii_{3}\cdot\jj$ is similar.
\end{proof}

There is a well-known homomorphism from the \itind{braid group} $B_3$ to $\Aut(F_2)$ (see e.g. \cite[Theorem N6, p. 173]{MaKaSo1976Combinatorial-grou}). In fact, the two automorphisms
\[ 
\sigma_{1}:=(a,ab) \text{ \ and \ } \sigma_{2}:=(ab^{-1},b)
\]
satisfy the braid relation $\sigma_{1}\sigma_{2}\sigma_{1} = \sigma_{2}\sigma_{1}\sigma_{2}$. 
The image of the induced homomorphism $\psi\colon B_3 \to \PGL_2(\ZZ)$ is $\PSL(2, \ZZ)$, and its kernel is the center $\Cent(B_3)\simeq \ZZ$, generated by  $(\sigma_{1}\sigma_{2})^3= (\sigma_{1} \sigma_{2} \sigma_{1})^{2}$ (\cite[Theorem 1.24, p. 22, and Theorem A.2, p. 312]{KaTu2008Braid-groups}). Hence we get the following corollary (cf. \cite{BeBrHi2011Cluster-cylic-quiv}).

\begin{corollary}
There is an algebraic action of $B_3$ on $\A{3}$ with kernel the center of $B_3$. It has two fixed points, $(0,0,0)$ and $(2,2,2)$, with non-isomorphic tangent  representations. The invariant functions are generated by $I:=x^{2}+y^{2}+z^{2}-x y z$.
\end{corollary}

\ps
\subsection{A closed subgroup with the same Lie algebra}  \label{tame-is-closed.subsec}
In this section we construct a strict closed subgroup of a connected  affine ind-group which has the same Lie algebra.
\subsubsection{The construction}\label{construction.subsec}
In the paper \cite{Sh1981On-some-infinite-d} \name{Shafarevich} claims that $\SAutA{n}$ is simple in the sense that it does not contain a nontrivial closed normal subgroup. The proof is based on the fact that $\Lie \SAutA{n}$ is simple together with a theorem which states that a closed subgroup of a connected ind-group $\GGG$ with the same Lie algebra is equal to $\GGG$. In this section we construct a counterexample to this last claim.

Consider the closed subgroup $\GGG \subseteq \AutA{3}$ of those automorphisms of $\A{3}$ which leave the projection $p_{3}\colon \A{3} \to \A{1}$ invariant,
\[ 
\GGG:=\{\f=(f_{1},f_{2},f_{3})\in \AutA{3} \mid f_{3}=z\},
\]
and let $\tG:= \GGG \cap \TameA{3} \subseteq \GGG$ be the subgroup consisting of tame elements.

\begin{theorem}
\be
\item $\tG \subseteq \GGG$ is a closed subgroup and $\tG \neq \GGG$.
\item $\GGG$ is connected.
\item $\Lie \tG = \Lie \GGG$.
\ee
\end{theorem}\label{closed-subgroup-with-same-Liealgebra.thm}

The proof of (2) follows immediately from Proposition~\ref{k*-connected.prop}. In fact, it is easy to see that the morphism $\phi\colon \Cst \to \AutA{3}$ connecting $\f,\g\in\GGG$ constructed in the proof has its image in $\GGG$, because $\GGG\cap\GL_{3}$ is connected.

For (3) we first remark that the anti-isomorphism $\Lie\Aut(\Athree) \simto \VEC^{c}(\Athree)$ from Proposition~\ref{LieAlgAutAn.prop} identifies $\Lie\GGG$ with the following Lie subalgebra:
\[
\{\delta = f\dx + g\dy \mid f,g \in \kk[x,y,z], \ \Div(\delta)\in\kk\} \subseteq \VEC(\A{3}).
\]
From this description one easily deduces that $\Lie \GGG$ is generated by the Lie algebras of the two subgroups $\GGG \cap \GL_{3}$ and $\GGG \cap \J{3}$. Since both belong to  $\tG$ the claim follows.

Finally, the famous Nagata automorphism $\n:=(x-2y \Delta -z \Delta^2,y+z \Delta, z)$ where $\Delta :=x z+y^2$ (Section~\ref{shifted-lin.subsec}) belongs to $\GGG$, but it is not tame as shown by the fundamental work of \name{Shestakov-Umirbaev} (see \cite[Corollary 9]{ShUm2004The-tame-and-the-w}). So it remains to show that $\tG$ is closed. 

\begin{remark}
It was recently shown by \name{Edo} and \name{Poloni} that $\Tame(\Athree)$ is not even weakly closed in $\AutA{3}$, see \cite{EdPo2015On-the-closure-of-}. 
\end{remark}

\begin{remark}\label{Lie-H-cap-G.rem}
Let $\GGG$ be an ind-group and $H,K \subseteq \GGG$ closed algebraic subgroups. It is easy to see that if $H$ is connected and $\Lie H \subseteq \Lie K$, then $H\subseteq K$, because we have $\Lie (H\cap G) = \Lie H \cap \Lie G$, see Proposition~\ref{Lie-H-cap-G.prop}. However, the above example shows that this does not hold if $H \subseteq \GGG$ is a closed ind-subgroup. In fact, the closed subgroup $K  \simeq \kplus$ defined by the locally nilpotent Nagata-derivation does not belong to $\tG$, but $\Lie K \subseteq \Lie \tG$.
\end{remark}

The basic idea for the proof is the following: Consider the set $\Ft \subseteq\kk[x,y,z]$ of first components of the elements of $\tG$ and show that this set is closed in the set $F$ of first components of $\GGG$. This will implies the claim, because an automorphism from $\GGG$ with first component in $\Ft$ belongs to $\tG$ (see Lemma~\ref{FirstComp.lem} below). To prove that $\Ft$ is closed in $F$ we define a length function on $\Ft$ using the existence and uniqueness of a predecessor. The details are given in the following Sections~\ref{FirstRed.subsec} and \ref{LengthFunction.subsec}  

\ps
\subsubsection{First reductions} \label{FirstRed.subsec}  
Define
$$
\Gz:=\{(f_{1},f_{2},z)\in\GGG \mid f_{1}(0,0,z)=f_{2}(0,0,z)=0\} \subseteq \GGG,
$$ 
the closed subgroup fixing pointwise the $z$-axis,  and put $\ZZZ:=\{(x+p(z),y+q(z),z)\mid p,q\in\kk[z]\}\subseteq \tG$.
\begin{lemma}
The morphism $\mu\colon \ZZZ \times \Gz \to \GGG$, $(\z,\f)\mapsto \z\circ\f$, is an isomorphism of ind-varieties.
\end{lemma}
\begin{proof}
This is clear, because $(x+p(z),y+q(z),z)\circ (f_{1},f_{2},z) = (f_{1}+p, f_{2}+q, z)$.
\end{proof}
Putting $\tGz:=\tG \cap \Gz$ the lemma implies that $\tG$ is closed in $\GGG$ if and only if $\tGz$ is closed in $\Gz$.

Denote by $\pi_{i}\colon \AutA{3} \to \kk[x,y,z]$, $i=1,2,3$, the projections onto the coordinates, and put $F:=\pi_{1}(\Gz)$, $\Ft:=\pi_{1}(\tGz)$.

\begin{lemma}\label{FirstComp.lem}
We have $\pi_{1}^{-1}(\Ft)\cap \Gz=\tGz$.
\end{lemma}

\begin{proof}
Let $f\in\Ft$ and choose $\f=(f,h,z)\in\tGz$. If $\tilde\f=(f,\tilde h,z)\in \Gz$, then $\tilde\f\circ\f^{-1} = (x,q,z)$ for some $q\in\kk[x,y,z]$. It follows that $q = \alpha y + p(x,z)$. Thus $\tilde\f\circ\f^{-1} \in \tGz$, and so $\tilde\f  \in \tGz$.
\end{proof}
This lemma shows that $\tGz$ is closed in $\Gz$ if $\Ft$ is closed in $F$. Thus the theorem follows from the next result.
\begin{main lemma}\label{MainLemma.lem}
The closure of $\Ft$ in $\kk[x,y,z]$ is given by
$$
\overline{\Ft} = \{p(f,z)\mid f\in \Ft \text{ and } p\in t\,\kk[t,z]\} \subseteq \kk[x,y,z].
$$
As a consequence, $\Ft$ is closed in $F$.
\end{main lemma}
The proof will be given at the end of Section~\ref{LengthFunction.subsec}. We only indicate here the easy parts of the proof.

The inclusion $\supseteq$ is clear. In fact, if $(f,h,z)\in\tGz$, $p\in t\,\kk[t,z]$ and $\alpha\in\Cst$, then 
$(\alpha y + p(x,z),x,z)$ is tame and belongs to $\tGz$, and so $(\alpha y + p(x,z),x,z)\circ(f,h,z) = (\alpha h + p(f,z),f,z) \in \tGz$. It follows that $\alpha h + p(f,z) \in \Ft$ for all $\alpha\neq 0$, hence $p(f,z)\in \overline{\Ft}$.

From this description of the closure $\overline{\Ft}$ we immediately get the last claim.  We have $\jac (p(f,z),h,z) = \frac{\partial f}{\partial t}(f,z) \, \jac (f,h,z)$. If $\frac{\partial f}{\partial t}\in\Cst$, then $p(f,z)= \alpha f\neq 0$, and thus $p(f,z)\in\Ft$. Otherwise, $p(f,z)$ is not the first coordinate of an automorphism, and thus $p(f,z)\notin F$. 

It remains to prove that $\overline{\Ft} \subseteq \{p(f,z)\mid f\in \Ft \text{ and } p\in t\,\kk[t,z]\}$.

\ps
\subsubsection{A length function} \label{LengthFunction.subsec}
If $f,g \in \Ft$, we say that {\it $g$ is a  predecessor of $f$} if $(f,g,z) \in \GGG$ and $\deg f > \deg g$. It then follows that $(f,g,z)\in\tGz$, by Lemma~\ref{FirstComp.lem}.\idx{predecessor}
The next result is crucial.

\begin{lemma}\label{predecessor.lem}
If $f \in \Ft$ has degree $\geq 2$, then $f$ admits a predecessor $g$ which is unique up to a scalar multiple.
\end{lemma}

\begin{proof}
(1) Let $g \in \Ft \subseteq \kk [x,y,z]$ be an element of minimal degree such that $(f,g,z) \in \GGG$. By \cite[Corollary 8]{ShUm2004The-tame-and-the-w}, the automorphism $\f:=(f,g,z)$ admits an elementary reduction. This means that for one of the components of $\f$ the degree can be lowered by subtracting a polynomial in the two other components. Clearly, the degrees of $z$ cannot be lowered. Moreover, since $f(0,0,z)=g(0,0,z)=0$, we can assume that the result of such an elementary reduction is again an element from $\Gz$. Since the degree of $g$ is minimal, we can only ``reduce'' the first component $f$, i.e. there exists a polynomial $p(s,t)$ such that $\deg(f - p(g,z) ) < \deg f$. 

For the homogeneous terms of maximal degree we get $\bar f = \overline{p(g,z)}$. Since $\bar f \notin\kk[z]$, we have a relation of the form $\bar f=\sum_{i,j} \alpha_{i j}{\bar g}^{i}z^{j}$ where the right hand side contains a nonzero term 
$\alpha_{i j} {\bar g}^{i}z^{j}$ of degree $\deg f$ with $i\neq 0$, and where $\alpha_{0,j}=0$, because $f(0,0,z)=0$.  Hence, $\deg  f \geq \deg g$, and $\deg f = \deg g$ implies that $\bar f = \alpha \bar g$. Thus $g-\alpha^{-1}f$ has lower degree, contradicting the minimality of $g$. This shows that a predecessor of $f$ exists.

(2) Now let $g_{1},g_{2}\in\Ft$ be two predecessors. Then $(f,g_{1},z) = (x,h,z)\circ (f,g_{2},z)$ where $(x,h,z)\in \tGz$. Thus $h(x,y,z) = \alpha y + x q(x,z)$, hence $g_{1}=\alpha g_{2}+ f q(x,f)$, and so $q=0$, because $\deg g_{i}<\deg f$.
\end{proof}

\begin{definition}\label{LengthFunction.def}
The \itind{length function} $\lgt\colon F^{t}\to \NN$ is inductively defined by 
$\lgt (f) =0$ if $\deg f= 1$, and $\lgt (f) = \lgt (g) + 1$ if $\deg f \geq 2$ and $g$ is a predecessor of $f$.
Note that $\deg f > \lgt(f)$.
\end{definition}
Using the length function we define  the subset  $\Ft_{k}:=\{f\in\Ft\mid \lgt(f)\leq k\} \subseteq \Ft$ for any $k\geq 0$, and $\Ft_{-1}:=\{0\}$. 
If $f\in\Ft_{k}$, then $(f,g,z)\in \tGz$ for some $g\in \Ft$, and if $p\in t\kk[t,z]$, then 
$(f,\alpha g + p(f,z),z) \in \tGz$ for every  $\alpha\in\Cst$. If $\deg p\geq 2$, then $\deg p(f,z)>\deg f$ and so $f$ is a predecessor of
$\alpha g + p(f,z)$. This implies that $\lgt(\alpha g + p(f,z))=\lgt(f) + 1 \leq k + 1$ and $p(f,z) \in \overline{\Ft_{k + 1}}$. The latter clearly holds also for 
$\deg p\leq 1$.  

Putting $R(S):= \{p(f,z) \mid p\in t\kk[t,z], f\in S\}$ for a subset $S \subseteq \kk[x,y,z]$ this shows that 
$$
R(\Ft_{k}) \subseteq \overline{\Ft_{k + 1}}.
$$ 

\begin{lemma}\label{closure.lem}
For all $k\geq 0$ we have $\overline{\Ft_{k}}=\Ft_{k}\cup R(\Ft_{k-1})$.
\end{lemma}
\begin{proof}
We have to show that $C_{k}:=\Ft_{k}\cup R(\Ft_{k-1})$ is closed for all $k$. By definition, we have $\Ft_{0}=\{f\in\Ft\mid \deg f = 1\}$ which is equal to the set $\{\alpha x + \beta y\mid (\alpha,\beta)\neq (0,0)\}$, and so $\overline{\Ft_{0}} = \Ft_{0}\cup\{0\} = C_{0}$.

Put $P:=\{f\in\kk[x,y,z]\mid f(0,0,z)=0\}=x\kk[x,y,z]+y\kk[x,y,z]$, $P_{d}:=\{f\in P \mid \deg f \leq d\}$, ${\dot P}:=P\setminus\{0\}$, ${\dot P}_{d}:=P_{d}\setminus\{0\}$ and  $\dot C_{k} = C_{k} \setminus \{ 0 \}$.
Clearly, $P = \bigcup_{d}P_{d}$ and ${\dot P} = \bigcup_{d}{\dot P}_{d}$ are both ind-varieties, and both are stable under scalar multiplication (with nonzero elements).
We have to show that $C_{k+1}\cap P_{d}\subseteq P_{d}$ is closed for all $d$.

Define the closed subset
\[
J:=\{(f,g)\in P \times {\dot P} \mid \jac (f,g,z)\in \kk\} \subseteq P \times {\dot P}.
\]
Clearly, $J$ is stable under both scalar multiplications, on $P$ and on $\dot P$. Since the projection $\pr_{1}\colon \PP(P_{d})\times \PP(P_{d}) \to \PP(P_{d})$ is closed we see that for every closed subset $C \subseteq P_{d} \times {\dot P}_{d}$ which is stable under both scalar multiplications, the image $\pr_{1}(C)\subseteq P_{d}$  is closed. Therefore, the lemma follows, by induction, from the following equality:
\[\tag{$*$}
C_{k+1} \cap P_d= \pr_{1}(J_{d} \cap \pr_{2}^{-1}(\dot C_{k})) \text{ where }J_{d}=J \cap (P_d \times {\dot P}_d).
\]
For the proof we distinguish several cases.
\par\noindent
(1) Let $f\in C_{k+1}\cap P_{d}$.
\be
\item[(1a)]
If $f\in\Ft_{k+1}$ and if $g\in \Ft_{k}$ is a predecessor, then $(f,g,z)\in\tG$ and $\deg g = d-1$, and so 
$f\in \pr_{1}(J_{d}\cap\pr_{2}^{-1}(g))$.
\item[(1b)] If $f\in R(\Ft_{k})$, $f=p(g,z)$ for some $g\in\Ft_{k}$, then $\jac (f,g,z)=0$ and $\deg g \leq d$, and so $f\in \pr_{1}(J_{d}\cap\pr_{2}^{-1}(g))$. 
\ee
This proves the  inclusion ``$\subseteq$''  of $(*)$.

\noindent
(2) Assume now that $(f,g)\in J_{d}$ where $g\in \dot C_{k}$. 
\be
\item[(2a)] Let $g\in\Ft_{k}$. If $\jac (f,g,z)\in\Cst$, then $f\in\Ft_{k+1}$ and so $f\in C_{k+1}\cap P_{d}$. Otherwise, $\jac (f,g,z)=0$ and then $f\in\kk[g,z]$. In fact, there is an automorphism of the form $(g,h,z)$, and so the subalgebra $\kk[g,z]$ is algebraically closed in $\kk[x,y,z]$. Since $f(0,0,z)=0$ we see that $f\in g\,\kk[g,z]\subseteq R(\Ft_{k})$, hence $f\in C_{k+1}\cap P_{d}$. 
\item[(2b)] If $g \in R(\Ft_{k-1})$, then $g = p(h,z)$ for some $h\in\Ft_{k-1}\setminus\{0\}$. Since $\jac (f, g,z) = \frac{\partial p}{\partial t}(h,z)\jac (f,h,z) \in \kk$ we get $\jac (f,h,z) \in \kk$, and so $f\in C_{k}\cap P_{d}$, by the previous case (2a).
\ee
This proves the other inclusion of $(*)$.
\end{proof}

\ps
\subsubsection{The final step}
It remains to prove the Main Lemma~\ref{MainLemma.lem}.

\begin{proof}
As we have seen at the end of Section~\ref{FirstRed.subsec}, we have  to show that $\overline{\Ft} \subseteq \{p(f,z)\mid f\in \Ft \text{ and } p\in t\,\kk[t,z]\} = R(\Ft)$. For this it suffices to prove that $R(\Ft)$ is closed, i.e. that $R(\Ft)\cap\kk[x,y,z]_{\leq d}$ is closed in $\kk[x,y,x]_{\leq d}$ for all $d$. We have $R(\Ft) = \bigcup_{k} R(\Ft_{k}) = \bigcup_{k} C_{k}$ where  $C_{k}:=\Ft_{k}\cup R(\Ft_{k-1})\subseteq \kk[x,y,z]$ is closed, by Lemma~\ref{closure.lem}. We claim that for every $d$ we have
$$
R(\Ft)\cap \kk[x,y,z]_{\leq d} = (\bigcup_{k} C_{k}) \cap \kk[x,y,z]_{\leq d} = C_{d}\cap\kk[x,y,z]_{\leq d}
$$
which shows that $R(\Ft)\cap \kk[x,y,z]_{\leq d}$ is closed. For the claim we simply remark that for $f\in F^{t}$ we have $\deg f > \ell(f)$, by the definition of the length function (Definition~\ref{LengthFunction.def}), and that for any nonzero $p\in t\kk[t,z]$ we get $\deg p(f,z) \geq \deg f$. This implies that $\deg f>d$ for $f \in C_{d+1}\setminus C_{d}$.
\end{proof}

\begin{remark}\label{group theoretical interpretation.rem}
Lemma \ref{predecessor.lem} which relies on the work of \cite{ShUm2004The-tame-and-the-w} has the following  group theoretical interpretation. Set
\begin{eqnarray*}
A & := & \{ (a x+b y + p(z), c x+d y +q(z),z)\mid a,b,c,d \in \kk,  a d - b c \neq 0,  p,q \in \kk [z] \}, \\
B & := & \{ (a x + p(y,z), b y +q(z), z), \mid a,b \in \kk^*,  p \in \kk [y,z],  q \in \kk [z] \}.
\end{eqnarray*}
Then $\tG$ is the amalgamated product of $A$ and $B$ along their intersection
\[ 
\tG = A *_{A \cap B}B.  
\]
Note that $A \simeq \GL(2) \ltimes \ZZZ$ and that $B=\GGG \cap \J{3} = \tG \cap \J{3}$.
\end{remark}

\newpage
\section{Problems and Questions}
For the convenience of the reader we collect here the questions from previous sections. 

\subsection{Special morphisms and homomorphisms}
The following questions came up in the study of bijective morphisms and homomorphisms.

\Ques{towards-a-criterion-for-a-bijective-morphism-to-be-an-isomorphism.ques}
{Is it true that a bijective ind-morphism $\phi\colon \VVV \to \WWW$
is an isomorphism if the differential $d\phi_{v}$ is an isomorphism in every point $v\in\VVV$? Maybe one has to assume in addition that $\VVV$ is connected or even curve-connected.}

\Ques{Is-the-Kernel-of-the-differential-the-Lie-algebra-of-the-Kernel.ques}
{If $\phi \colon \GGG \to G$ is a surjective homomorphism of ind-groups where 
$\GGG$ is an affine ind-group and
$G$ is an algebraic group, do we have that $\Ker d \phi_{e} = \Lie\Ker\phi$?} 

\Ques{bijective+isomorphism-on-Lie-algebras-imply-isomorphism.ques} 
{Let $\phi\colon\GGG \to \HHH$ be a bijective homomorphism of affine
ind-groups, and assume that $\Lie\phi\colon \Lie\GGG \to \Lie\HHH$ is an isomorphism. Does this imply that $\phi$ is an isomorphism?}

\ps
\subsection{Endomorphisms and automorphisms}\strut{}

\Ques{does-Z-open-affine-in-Y-affine-imply-that-Mor(X,Z)-is-locally-closed-in-Mor(X,Y).ques}
{Let $X, Y$ be affine varieties, and let $U \subset Y$ be an open affine dense subset. Is it true that $\Mor(X,U) \subset \Mor(X,Y)$ is locally closed?}

\Ques{AutR-locally-closed.que}
{Is $\Aut(R) \subseteq \End(R)$  locally closed when $R$ is an associative $\kk$-algebra?}

\ps
\subsection{Representations} Not much is known about the representation theory of ind-groups or, more specially, of automorphism groups $\Aut(X)$ of affine varieties. The second class always admits at least  one faithful representation, namely the standard representation on the coordinate ring $\OOO(X)$.

\Ques{faithful-representation.ques} 
{Does any affine ind-group admit a faithful (or at least a nontrivial) representation on a $\kk$-vector space of countable dimension?}

\Ques{Aut(X)-acts-faithfully-on-vector-fields.ques}
{Is it true that $\Aut (X)$ acts faithfully on $\VEC (X)$?}

\Ques{locally-nilpotent.ques}
{Assume that $\rho\colon \Aut(X) \to \GL(V)$ is a representation such that
$d\rho\colon \Lie \Aut (X) \to \End(V)$ is injective
(where $V$ is a vector space of countable dimension).
Is it true that if $d\rho(N) \in \End(V)$ is locally nilpotent, then $N \in \Lie\Aut(X)$ is locally nilpotent?

Assume that the adjoint representation $\ad\colon \Lie\Aut(X) \to \End(\Lie\Aut(X))$ is faithful. Is it true that $N\in\Lie\Aut(X)$ is locally nilpotent if and only if $\ad N$ is locally nilpotent?}

\Ques{no-injection-unipotent.ques}
{Does Proposition~\ref{no-inbedding-of-Jn-into-GLinfty.prop} also hold for the unipotent part $\JJJ^{u}$?
Or is there an injective homomorphism of $\JJJ^{u} \into \GL_{\infty}$?}

\subsection{Special elements}\strut{}

\Ques{locally-finite-elements-and-generation.ques} 
{If $\GGG$ is a connected ind-group, do we have $ \GGG = \langle  \GGG^{\text{\it lf}}   \rangle$, 
or at least $ \GGG = \overline{ \langle  \GGG^{\text{\it lf}}\rangle}$?}

\Ques{Closedness-of-the-set-of-unipotent-automorphisms.ques} 
{Do the unipotent elements $\AutU(X)$ form a closed subset of $\Aut(X)$? Is $\LNV(X)$ closed in $\VEC(X)$?}

\Ques{Closedness-of-the-set-of-unipotent-endomorphisms.ques}
{More generally, is $\End^{\text{\it ln}} (X)$ closed in  $\End(X)$?}

\subsection{Elements of finite order and locally finite elements}\strut{}

\Ques{additive-or-multiplicative-group-in-Aut(X).ques} 
{Is it true that a nondiscrete automorphism group $\Aut(X)$ of an affine variety $X$ always contains a copy of the additive group $\kplus$ or a copy of the multiplicative group $\kst$?}

\Ques{all-elements-of-finite-order-implies-discrete.ques}
{Is it true that every ind-group $\GGG$ consisting of elements of finite order is discrete?
More generally, is it true that a subgroup  $F \subseteq \GGG$ consisting of elements of finite order is countable?}

\Ques{locally-finite-implies-nested.question}
{If $\GGG$ is a connected ind-group whose elements are all locally finite, is it true that $\GGG$ is nested?
}

\Ques{generation-by-locally-finite-derivations.ques}
{Is $\Lie \Aut ( \R)$ generated by the locally finite derivations for any finitely generated general algebra $ \R$?}

\ps
\subsection{Vector fields}\strut{}

\Ques{Is-the-tangent-space-at-the-identity-of-End(X)-a-Lie-algebra.question}
{Is it true that the image of $T_{\id}\End(X)$ in $\VEC(X)$ is a Lie subalgebra?}

\Ques{Is-T(End(R))-a-Lie-algebra.question}
{Is  $T_{\id}\End (R)$  a Lie subalgebra of  $\Der_{\kk} (R)$
for a finitely generated associative algebra $ \R$?}

\Ques{Is-T(End(RR))-a-Lie-algebra.question}
{Is  $T_{\id}\End (\R)$  a Lie subalgebra of  $\Der_{\kk} \R$ for a finitely generated general algebra $\R$?}

\ps
\subsection{Homomorphisms}\strut{}

\Ques{Delta.ques}
{Let $G,L$ be linear algebraic groups, and let $H,K \subset G$ be closed subgroups which generate $G$.
Is it true that the canonical map 
$$
\Delta\colon\Hom(G,L) \to \Hom(H,L) \times \Hom(K,L)
$$
is a closed immersion of ind-varieties?}

\ps
\subsection{The group \texorpdfstring{$\Aut(\An$}{Aut(An)})}\strut{}

\Ques{SAut(An).ques} 
{Let $\UUU \subseteq \AutA{n}$ be the subgroup generated by the modifications of the translations
and let $\SAutA{n}= \{\phi\in\Aut(\An) \mid \jac(\phi)=1\}$ be the special automorphism group.
Do we have $\UUU = \SAutA{n}$ for all $n\geq 2$, or at least $\overline{\UUU} = \SAutA{n}$?}

\Ques{diagonalizable-and-semisimple.ques}
{Let $\g \in \Aut (\AA ^n)$ be an automorphism. Are the four following assertions equivalent?
\be
\item
$\g$ is diagonalizable.
\item
$\g$ is semisimple.
\item
The conjugacy class $C (\g)$ is closed in $\Aut (\AA ^n)$.
\item
The conjugacy class $C (\g)$ is weakly closed in $\Aut (\AA ^n)$.
\ee
}

\Ques{semisimple-part-in-the-weak-closure-of-the-conjugacy-class.ques} 
{Let $\g\in\AutA{n}$ be a locally finite automorphism. Does the weak closure of $C(\g)$ contain
the semisimple part $\g_{s}$ of $\g$?
}

\Ques{characterization-of-unipotent-automorphisms-conjugate-to-translations} 
{Is there a characterization of those unipotent automorphisms
$\u\in\AutA{n}$  which are conjugate to  translations? 
Same question for the upper-triangular automorphisms
$\u\in\JJJ(n)$.}

\ps
\subsection{Large subgroups}

Let $G$ be a linear algebraic group and $X$ an affine variety. If $\rho\colon G \to \Aut(X)$ is homomorphism of ind-groups we have defined in Section~\ref{large-subgroups.subsec} a homomorphism of ind-groups $\tilde\rho\colon G(\OOO(X)^{G}) \to \Aut(X)$ which extends $\rho$. We have also set
\begin{multline*}
\Aut_{\rho}(X):= \\ \{\phi\in\Aut(X)\mid \phi(Gx) = Gx \text{ for all } x \in X, \,\phi|_{Gx}=\rho(g_{x}) \text{ for some }g_{x}\in G\}.
\end{multline*}

\Ques{image-of-rho.ques}
{Is $\Aut_{\rho}(X)$ the image of $\tilde\rho$? And is the image of $\tilde\rho$ closed in $\Aut(X)$?}

\Ques{rho-closed-immersion.ques}
{If $\tilde\rho$ is injective, is it a closed immersion?}

\ps
\subsection{Miscellaneous}
Let $\GGG$ be an affine ind-group. In Section~\ref{exponential-for-affine-ind-groups.subsec}, we have defined two maps 
$$ 
\eps_{\GGG}\colon \Hom(\kplus,\GGG) \to \GGG \text{ \ and \ }\nu_{\GGG}\colon \Hom(\kplus,\GGG) \to \gg
$$ 
by the formulas $\eps_{\GGG}(\lambda):=\lambda(1)$ and $\nu_{\GGG}(\lambda):=d\lambda_{0}(1)$. We have also defined the logarithm $\log_{\GGG} \colon  \GGG^{u} \to \gg$.
\ps
\Ques{eps-and-nu.ques}
{\be
\item Is $\eps_{\GGG}$ a closed immersion?
\item Is $\nu_{\GGG}$ a closed immersion?
\item Is it true that for any morphism $\phi \colon Y \to \GGG$ with image in $\GGG^{u}$ the composition $\log_{\GGG} \circ \phi$ is a morphism?
\ee}

\printindex  

\par\bigskip
\renewcommand{\MR}[1]{}
\bibliography{HP-Bib-AutAn}

\newcommand{\etalchar}[1]{$^{#1}$}
\providecommand{\bysame}{\leavevmode\hbox to3em{\hrulefill}\thinspace}
\providecommand{\MR}{\relax\ifhmode\unskip\space\fi MR }
\providecommand{\MRhref}[2]{%
  \href{http://www.ams.org/mathscinet-getitem?mr=#1}{#2}
}
\providecommand{\href}[2]{#2}
\begin{thebibliography}{BvdEM01}

\bibitem[AFK{\etalchar{+}}13]{ArFlKa2013Flexible-varieties}
I.~Arzhantsev, H.~Flenner, S.~Kaliman, F.~Kutzschebauch, and M.~Zaidenberg,
  \emph{Flexible varieties and automorphism groups}, Duke Math. J. \textbf{162}
  (2013), no.~4, 767--823. \MR{3039680}

\bibitem[AK14]{AnKr2014Varieties-characte}
Rafael~B. Andrist and Hanspeter Kraft, \emph{Varieties characterized by their
  endomorphisms}, Math. Res. Lett. \textbf{21} (2014), no.~2, 225--233, {\tt
  arXiv:1309.0030 [math.AG]}. \MR{3247051}

\bibitem[AM75]{AbMo1975Embeddings-of-the-}
Shreeram~S. Abhyankar and Tzuong~Tsieng Moh, \emph{Embeddings of the line in
  the plane}, J. Reine Angew. Math. \textbf{276} (1975), 148--166. \MR{0379502
  (52 \#407)}

\bibitem[Ani83]{An1983Limits-of-tame-aut}
David~J. Anick, \emph{Limits of tame automorphisms of {$k[x_{1},\cdots
  ,x_{N}]$}}, J. Algebra \textbf{82} (1983), no.~2, 459--468. \MR{704764
  (85d:13005)}

\bibitem[Art69]{Ar1969On-Azumaya-algebra}
Michael Artin, \emph{On {A}zumaya algebras and finite dimensional
  representations of rings}, J. Algebra \textbf{11} (1969), 532--563.
  \MR{0242890}

\bibitem[Bas84]{Ba1984A-nontriangular-ac}
Hyman Bass, \emph{A nontriangular action of {${\bf G}_{a}$}\ on {${\bf
  A}^{3}$}}, J. Pure Appl. Algebra \textbf{33} (1984), no.~1, 1--5. \MR{750225}

\bibitem[BBH11]{BeBrHi2011Cluster-cylic-quiv}
Andre Beineke, Thomas Br{{\"u}}stle, and Lutz Hille, \emph{Cluster-cylic
  quivers with three vertices and the {M}arkov equation}, Algebr. Represent.
  Theory \textbf{14} (2011), no.~1, 97--112, With an appendix by Otto Kerner.
  \MR{2763295 (2012a:16028)}

\bibitem[BCW82]{BaCoWr1982The-Jacobian-conje}
Hyman Bass, Edwin~H. Connell, and David Wright, \emph{The {J}acobian
  conjecture: reduction of degree and formal expansion of the inverse}, Bull.
  Amer. Math. Soc. (N.S.) \textbf{7} (1982), no.~2, 287--330. \MR{663785
  (83k:14028)}

\bibitem[BD97]{BhDu1997Kernel-of-locally-}
S.~M. Bhatwadekar and Amartya~K. Dutta, \emph{Kernel of locally nilpotent
  {$R$}-derivations of {$R[X,Y]$}}, Trans. Amer. Math. Soc. \textbf{349}
  (1997), no.~8, 3303--3319. \MR{1422595 (97m:13009)}

\bibitem[Ber03]{Be2003Lifting-of-morphis}
Florian Berchtold, \emph{Lifting of morphisms to quotient presentations},
  Manuscripta Math. \textbf{110} (2003), no.~1, 33--44. \MR{1951798}

\bibitem[BF13]{BlFu2013Topologies-and-str}
J{\'e}r{\'e}my Blanc and Jean-Philippe Furter, \emph{Topologies and structures
  of the {C}remona groups}, Ann. of Math. (2) \textbf{178} (2013), no.~3,
  1173--1198. \MR{3092478}

\bibitem[BKY12]{BeYu2012Lifting-of-the-aut}
Alexei Belov-Kanel and Jie-Tai Yu, \emph{Lifting of the automorphism group of
  polynomial algebras}, arXiv:1207.2045v1 (2012).

\bibitem[Bla06]{Bl2006Conjugacy-classes-}
J{{\'e}}r{{\'e}}my Blanc, \emph{Conjugacy classes of affine automorphisms of
  {$\mathbb K^n$} and linear automorphisms of {$\mathbb P^n$} in the {C}remona
  groups}, Manuscripta Math. \textbf{119} (2006), no.~2, 225--241.
  \MR{MR2215969 (2006m:14015)}

\bibitem[Bor91]{Bo1991Linear-algebraic-g}
Armand Borel, \emph{Linear algebraic groups}, second ed., Graduate Texts in
  Mathematics, vol. 126, Springer-Verlag, New York, 1991.

\bibitem[Bri15]{Br2015On-extensions-of-a}
Michel Brion, \emph{On extensions of algebraic groups with finite quotient},
  Pacific J. Math. \textbf{279} (2015), no.~1-2, 135--153. \MR{3437773}

\bibitem[BS64]{BoSe1964Theoremes-de-finit}
A.~Borel and J.-P. Serre, \emph{Th{\'e}or{\`e}mes de finitude en cohomologie
  galoisienne}, Comment. Math. Helv. \textbf{39} (1964), 111--164. \MR{0181643}

\bibitem[BvdEM01]{BeEsMa2001Derivations-having}
Joost Berson, Arno van~den Essen, and Stefan Maubach, \emph{Derivations having
  divergence zero on {$R[X,Y]$}}, Israel J. Math. \textbf{124} (2001),
  115--124. \MR{1856507 (2002f:13057)}

\bibitem[BW00]{BeWi2000Automorphisms-and-}
Yuri Berest and George Wilson, \emph{Automorphisms and ideals of the {W}eyl
  algebra}, Math. Ann. \textbf{318} (2000), no.~1, 127--147. \MR{1785579}

\bibitem[CD03]{CoDr2003From-Lie-algebras-}
Arjeh~M. Cohen and Jan Draisma, \emph{From {L}ie algebras of vector fields to
  algebraic group actions}, Transform. Groups \textbf{8} (2003), no.~1, 51--68.
  \MR{1959763 (2004a:17025)}

\bibitem[CML08]{CrMa2008An-algebraic-proof}
Anthony~J. Crachiola and Leonid~G. Makar-Limanov, \emph{An algebraic proof of a
  cancellation theorem for surfaces}, J. Algebra \textbf{320} (2008), no.~8,
  3113--3119. \MR{MR2450715 (2009h:14105)}

\bibitem[Coh85]{Co1985Free-rings-and-the}
P.~M. Cohn, \emph{Free rings and their relations}, second ed., London
  Mathematical Society Monographs, vol.~19, Academic Press, Inc. [Harcourt
  Brace Jovanovich, Publishers], London, 1985. \MR{800091 (87e:16006)}

\bibitem[Cor17]{Co2017Nonlinearity-of-so}
Yves Cornulier, \emph{Nonlinearity of some subgroups of the planar cremona
  group}, arXiv:1701.00275 [math.AG], 2017.

\bibitem[Cze71]{Cz1971Automorphisms-of-a}
Anastasia~J. Czerniakiewicz, \emph{Automorphisms of a free associative algebra
  of rank {$2$}}, Bull. Amer. Math. Soc. \textbf{77} (1971), 992--994.
  \MR{0285568 (44 \#2786)}

\bibitem[Cze72]{Cz1972Automorphisms-of-a}
\bysame, \emph{Automorphisms of a free associative algebra of rank {$2$}.
  {II}}, Trans. Amer. Math. Soc. \textbf{171} (1972), 309--315. \MR{0310021 (46
  \#9124)}

\bibitem[Dan74]{Da1974Non-simplicity-of-}
V.~I. Danilov, \emph{Non-simplicity of the group of unimodular automorphisms of
  an affine plane}, Mat. Zametki \textbf{15} (1974), 289--293. \MR{0357626 (50
  \#10094)}

\bibitem[Deb01]{De2001Higher-dimensional}
Olivier Debarre, \emph{Higher-dimensional algebraic geometry}, Universitext,
  Springer-Verlag, New York, 2001. \MR{1841091 (2002g:14001)}

\bibitem[Dem82]{De1982Combinatorial-inva}
A.~S. Demushkin, \emph{Combinatorial invariance of toric singularities},
  Vestnik Moskov. Univ. Ser. I Mat. Mekh. (1982), no.~2, 80--87, 117.
  \MR{655409}

\bibitem[DG70]{DeGa1970Groupes-algebrique}
Michel Demazure and Pierre Gabriel, \emph{Groupes alg{\'e}briques. {T}ome {I}:
  {G}{\'e}om{\'e}trie alg{\'e}brique, g{\'e}n{\'e}ralit{\'e}s, groupes
  commutatifs}, Masson \& Cie, {\'E}diteur, Paris, 1970, avec un appendice {\it
  Corps de classes local} par Michiel Hazewinkel. \MR{0302656 (46 \#1800)}

\bibitem[DK97]{DeKr1997Constructive-invar}
Harm Derksen and Hanspeter Kraft, \emph{Constructive invariant theory},
  Alg{\`e}bre non commutative, groupes quantiques et invariants ({R}eims,
  1995), S{\'e}min. Congr., vol.~2, Soc. Math. France, Paris, 1997,
  pp.~221--244. \MR{MR1601147 (98m:13010)}

\bibitem[DK02]{DeKe2002Computational-inva}
Harm Derksen and Gregor Kemper, \emph{Computational invariant theory},
  Invariant Theory and Algebraic Transformation Groups, I, Springer-Verlag,
  Berlin, 2002, Encyclopaedia of Mathematical Sciences, 130. \MR{MR1918599
  (2003g:13004)}

\bibitem[DKW99]{DeKuWi1999Subvarieties-of-Cn}
Harm Derksen, Frank Kutzschebauch, and J{\"o}rg Winkelmann, \emph{Subvarieties
  of {${\bf C}^n$} with non-extendable automorphisms}, J. Reine Angew. Math.
  \textbf{508} (1999), 213--235. \MR{1676877 (2000c:32072)}

\bibitem[Dub90]{Dube1990The-structure-of-p}
Thomas~W. Dub{\'e}, \emph{The structure of polynomial ideals and {G}r\"obner
  bases}, SIAM J. Comput. \textbf{19} (1990), no.~4, 750--775. \MR{1053942
  (91h:13021)}

\bibitem[{\`E}H74]{El1974Cubic-surfaces-of-}
M.~H. {{\`E}}l Huti, \emph{Cubic surfaces of {M}arkov type}, Mat. Sb. (N.S.)
  \textbf{93(135)} (1974), 331--346, 487. \MR{0342518 (49 \#7264)}

\bibitem[EP15]{EdPo2015On-the-closure-of-}
{\'E}ric Edo and Pierre-Marie Poloni, \emph{On the closure of the tame
  automorphism group of affine three-space}, Int. Math. Res. Not. IMRN
  \textbf{19} (2015), 9736--9750. \MR{3431609}

\bibitem[FL10]{FuLa2010Normal-subgroup-ge}
Jean-Philippe Furter and St{{\'e}}phane Lamy, \emph{Normal subgroup generated
  by a plane polynomial automorphism}, Transform. Groups \textbf{15} (2010),
  no.~3, 577--610. \MR{2718938}

\bibitem[FM89]{FrMi1989Dynamical-properti}
Shmuel Friedland and John Milnor, \emph{Dynamical properties of plane
  polynomial automorphisms}, Ergodic Theory Dynam. Systems \textbf{9} (1989),
  no.~1, 67--99. \MR{991490 (90f:58163)}

\bibitem[FM07]{FuMa2007Locally-finite-pol}
Jean-Philippe Furter and Stefan Maubach, \emph{Locally finite polynomial
  endomorphisms}, J. Pure Appl. Algebra \textbf{211} (2007), no.~2, 445--458.
  \MR{2340462 (2008e:14084)}

\bibitem[FM10]{FuMa2010A-characterization}
\bysame, \emph{A characterization of semisimple plane polynomial
  automorphisms}, J. Pure Appl. Algebra \textbf{214} (2010), no.~5, 574--583.
  \MR{2577663}

\bibitem[Fog73]{Fo1973Fixed-point-scheme}
John Fogarty, \emph{Fixed point schemes}, Amer. J. Math. \textbf{95} (1973),
  35--51. \MR{0332805 (48 \#11130)}

\bibitem[Fre06]{Fr2006Algebraic-theory-o}
Gene Freudenburg, \emph{Algebraic theory of locally nilpotent derivations},
  Encyclopaedia of Mathematical Sciences, vol. 136, Springer-Verlag, Berlin,
  2006, Invariant Theory and Algebraic Transformation Groups, VII. \MR{2259515
  (2008f:13049)}

\bibitem[FSR05]{FeRi2005Actions-and-invari}
Walter Ferrer~Santos and Alvaro Rittatore, \emph{Actions and invariants of
  algebraic groups}, Pure and Applied Mathematics (Boca Raton), vol. 269,
  Chapman \& Hall/CRC, Boca Raton, FL, 2005. \MR{2138858}

\bibitem[Fuj79]{Fu1979On-Zariski-problem}
Takao Fujita, \emph{On {Z}ariski problem}, Proc. Japan Acad. Ser. A Math. Sci.
  \textbf{55} (1979), no.~3, 106--110. \MR{531454 (80j:14029)}

\bibitem[Fur02]{Fu2002On-the-length-of-p}
Jean-Philippe Furter, \emph{On the length of polynomial automorphisms of the
  affine plane}, Math. Ann. \textbf{322} (2002), no.~2, 401--411. \MR{1893923
  (2003a:14091)}

\bibitem[Fur08]{Fu2008Sur-les-automorphi}
\bysame, \emph{Sur les automorphismes polynomiaux de l'espace affine}, Th\`ese
  d'{H}abilitation \`a diriger des recherche, Universit\'e de La Rochelle,
  d\'ecembre 2008.

\bibitem[GP03]{GP2003Automorphism-groups}
Nikolai~L. Gordeev and Vladimir~L. Popov, \emph{Automorphism groups of finite
  dimensional simple algebras}, Ann. of Math. (2) \textbf{158} (2003), no.~3,
  1041--1065. \MR{2031860 (2005b:20086)}

\bibitem[Gro95]{Gr1995Techniques-de-cons}
Alexander Grothendieck, \emph{Techniques de construction et th\'eor\`emes
  d'existence en g\'eom\'etrie alg\'ebrique. {IV}. {L}es sch\'emas de
  {H}ilbert}, S\'eminaire {B}ourbaki, {V}ol.\ 6, Soc. Math. France, Paris,
  1995, pp.~Exp.\ No.\ 221, 249--276. \MR{1611822}

\bibitem[Her26]{He1926Die-Frage-der-endl}
Grete Hermann, \emph{Die {F}rage der endlich vielen {S}chritte in der {T}heorie
  der {P}olynomideale}, Math. Ann. \textbf{95} (1926), no.~1, 736--788.
  \MR{1512302}

\bibitem[Hor75]{Ho1975Induced-automorphi}
Robert~D. Horowitz, \emph{Induced automorphisms on {F}ricke characters of free
  groups}, Trans. Amer. Math. Soc. \textbf{208} (1975), 41--50. \MR{0369540 (51
  \#5773)}

\bibitem[Hum75]{Hu1975Linear-algebraic-g}
James~E. Humphreys, \emph{Linear algebraic groups}, Springer-Verlag, New
  York-Heidelberg, 1975, Graduate Texts in Mathematics, No. 21. \MR{0396773 (53
  \#633)}

\bibitem[Hum78]{Hu1978Introduction-to-Li}
\bysame, \emph{Introduction to {L}ie algebras and representation theory},
  Graduate Texts in Mathematics, vol.~9, Springer-Verlag, New York, 1978,
  Second printing, revised. \MR{499562 (81b:17007)}

\bibitem[Igu73]{Ig1973Geometry-of-absolu}
Jun-ichi Igusa, \emph{Geometry of absolutely admissible representations},
  Number theory, algebraic geometry and commutative algebra, in honor of
  {Y}asuo {A}kizuki, Kinokuniya, Tokyo, 1973, pp.~373--452. \MR{0367077 (51
  \#3319)}

\bibitem[Iva92]{Iv1992On-the-Burnside-pr}
Sergei~V. Ivanov, \emph{On the {B}urnside problem on periodic groups}, Bull.
  Amer. Math. Soc. (N.S.) \textbf{27} (1992), no.~2, 257--260. \MR{1149874}

\bibitem[Jun42]{Ju1942Uber-ganze-biratio}
Heinrich W.~E. Jung, \emph{\"{U}ber ganze birationale {T}ransformationen der
  {E}bene}, J. Reine Angew. Math. \textbf{184} (1942), 161--174. \MR{0008915
  (5,74f)}

\bibitem[Kal02]{Ka2002Polynomials-with-g}
Shulim Kaliman, \emph{Polynomials with general {$\mathbb{C}^2$}-fibers are
  variables}, Pacific J. Math. \textbf{203} (2002), no.~1, 161--190.
  \MR{1895930 (2003d:14078)}

\bibitem[Kal04]{Ka2004Free-C-actions-on-}
\bysame, \emph{Free {$\mathbb{C}_+$}-actions on {$\mathbb{C}^3$} are
  translations}, Invent. Math. \textbf{156} (2004), no.~1, 163--173.
  \MR{MR2047660 (2005b:14102)}

\bibitem[Kam96]{Ka1996Pro-affine-algebra}
Tatsuji Kambayashi, \emph{Pro-affine algebras, {I}nd-affine groups and the
  {J}acobian problem}, J. Algebra \textbf{185} (1996), no.~2, 481--501.
  \MR{1417382 (97m:13047)}

\bibitem[Kam03]{Ka2003Some-basic-results}
\bysame, \emph{Some basic results on pro-affine algebras and ind-affine
  schemes}, Osaka J. Math. \textbf{40} (2003), no.~3, 621--638. \MR{2003740
  (2004j:14002)}

\bibitem[Kap71]{Ka1971Lie-algebras-and-l}
Irving Kaplansky, \emph{Lie algebras and locally compact groups}, The
  University of Chicago Press, Chicago, Ill.-London, 1971. \MR{0276398 (43
  \#2145)}

\bibitem[KD15]{KrDu2015Invariants-and-Sep}
Hanspeter Kraft and Emilie Dufresne, \emph{Invariants and separating morphisms
  for algebraic group actions}, Math. Z. \textbf{280} (2015), 231--255, {\tt
  arXiv:1402.4299 [math.AC]}.

\bibitem[KK96]{KrKu1996Equivariant-affine}
Hanspeter Kraft and Frank Kutzschebauch, \emph{Equivariant affine line bundles
  and linearization}, Math. Res. Lett. \textbf{3} (1996), no.~5, 619--627.
  \MR{MR1418576 (97h:14065)}

\bibitem[KKV89]{KnKrVu1989The-Picard-group-o}
Friedrich Knop, Hanspeter Kraft, and Thierry Vust, \emph{The {P}icard group of
  a {$G$}-variety}, Algebraische {T}ransformationsgruppen und
  {I}nvariantentheorie, DMV Sem., vol.~13, Birkh{\"a}user, Basel, 1989,
  pp.~77--87. \MR{MR1044586}

\bibitem[KM05]{KaMi2005On-two-recent-view}
Tatsuji Kambayashi and Masayoshi Miyanishi, \emph{On two recent views of the
  {J}acobian conjecture}, Affine algebraic geometry, Contemp. Math., vol. 369,
  Amer. Math. Soc., Providence, RI, 2005, pp.~113--138. \MR{2126658
  (2006i:14070)}

\bibitem[Kos63]{Ko1963Lie-group-represen}
Bertram Kostant, \emph{Lie group representations on polynomial rings}, Amer. J.
  Math. \textbf{85} (1963), 327--404.

\bibitem[KPZ16]{KoPeZa2016On-Automorphism-Gr}
Sergei Kovalenko, Alexander Perepechko, and Mikhail Zaidenberg, \emph{On
  automorphism groups of affine surfaces}, preprint, arXiv: 1511.09051.

\bibitem[KR08]{KrRo2008Computational-comm}
Martin Kreuzer and Lorenzo Robbiano, \emph{Computational commutative algebra
  1}, Springer-Verlag, Berlin, 2008, Corrected reprint of the 2000 original.
  \MR{2723052 (2011h:13041)}

\bibitem[KR14a]{KrRe2013A-Note-on-the-auto}
Hanspeter Kraft and Andriy Regeta, \emph{Automorphisms of the {L}ie algebra of
  vector fields}, Preprint, 2014.

\bibitem[KR14b]{KrRu2014Families-of-group-}
Hanspeter Kraft and Peter Russell, \emph{Families of group actions, generic
  isotriviality, and linearization}, Transform. Groups \textbf{19} (2014),
  no.~3, 779--792. \MR{3233525}

\bibitem[Kra84]{Kr1984Geometrische-Metho}
Hanspeter Kraft, \emph{Geometrische {M}ethoden in der {I}nvariantentheorie},
  Aspects of Mathematics, D1, Friedr. Vieweg \& Sohn, Braunschweig, 1984.
  \MR{MR768181 (86j:14006)}

\bibitem[Kra96]{Kr1996Challenging-proble}
\bysame, \emph{Challenging problems on affine {$n$}-space}, Ast{\'e}risque
  (1996), no.~237, Exp.\ No.\ 802, 5, 295--317, S{{\'e}}minaire Bourbaki, Vol.
  1994/95. \MR{MR1423629 (97m:14042)}

\bibitem[Kra17]{Kr2017Automorphism-group}
\bysame, \emph{Automorphism groups of affine varieties and a characterization
  of affine {$n$}-space}, Trans. Moscow Math. Soc. \textbf{78} (2017),
  171--186. \MR{3738084}

\bibitem[KS92]{KrSc1992Reductive-group-ac}
Hanspeter Kraft and Gerald~W. Schwarz, \emph{Reductive group actions with
  one-dimensional quotient}, Inst. Hautes {\'E}tudes Sci. Publ. Math.
  \textbf{76} (1992), 1--97. \MR{MR1215592 (94e:14065)}

\bibitem[KS13]{KrSt2012On-Automorphisms-o}
Hanspeter Kraft and Immanuel Stampfli, \emph{On automorphisms of the affine
  {C}remona group}, Ann. Inst. Fourier (Grenoble) \textbf{63} (2013), no.~3,
  1137--1148. \MR{3137481}

\bibitem[KT08]{KaTu2008Braid-groups}
Christian Kassel and Vladimir Turaev, \emph{Braid groups}, Graduate Texts in
  Mathematics, vol. 247, Springer, New York, 2008, With the graphical
  assistance of Olivier Dodane. \MR{2435235 (2009e:20082)}

\bibitem[Kum02]{Ku2002Kac-Moody-groups-t}
Shrawan Kumar, \emph{Kac-{M}oody groups, their flag varieties and
  representation theory}, Progress in Mathematics, vol. 204, Birkh{\"a}user
  Boston Inc., Boston, MA, 2002. \MR{1923198 (2003k:22022)}

\bibitem[KZ99]{KaZa1999Affine-modificatio}
Shulim Kaliman and Mikhail Zaidenberg, \emph{Affine modifications and affine
  hypersurfaces with a very transitive automorphism group}, Transform. Groups
  \textbf{4} (1999), no.~1, 53--95. \MR{1669174 (2000f:14099)}

\bibitem[Lam01]{La2001Lalternative-de-Ti}
St{{\'e}}phane Lamy, \emph{L'alternative de {T}its pour {${\rm Aut}[{\mathbb
  C}^2]$}}, J. Algebra \textbf{239} (2001), no.~2, 413--437. \MR{1832900
  (2002d:14102)}

\bibitem[LRU18]{LiReUrCharacterization-o}
Alvaro Liendo, Andriy Regeta, and Christian Urech, \emph{Characterization of
  affine toric varieties by their automorphism groups}, Preprint (2018), 1--14,
  {\tt arXiv:1805.03991 [math.AG]}.

\bibitem[Lun73]{Lu1973Slices-etales}
Domingo Luna, \emph{Slices {\'e}tales}, Sur les groupes alg{\'e}briques, Soc.
  Math. France, Paris, 1973, pp.~81--105. Bull. Soc. Math. France, Paris,
  M{\'e}moire 33. \MR{MR0342523 (49 \#7269)}

\bibitem[Mag78]{Ma1978Separately-algebra}
Andy~R. Magid, \emph{Separately algebraic group laws}, Amer. J. Math.
  \textbf{100} (1978), no.~2, 407--409. \MR{489964 (80d:14024)}

\bibitem[Mat89]{Ma1989Commutative-ring-t}
Hideyuki Matsumura, \emph{Commutative ring theory}, second ed., Cambridge
  Studies in Advanced Mathematics, vol.~8, Cambridge University Press,
  Cambridge, 1989, Translated from the Japanese by M. Reid. \MR{MR1011461
  (90i:13001)}

\bibitem[Mau03]{Ma2003The-commuting-deri}
Stefan Maubach, \emph{The commuting derivations conjecture}, J. Pure Appl.
  Algebra \textbf{179} (2003), no.~1-2, 159--168. \MR{1958381 (2004e:13014)}

\bibitem[MKS76]{MaKaSo1976Combinatorial-grou}
Wilhelm Magnus, Abraham Karrass, and Donald Solitar, \emph{Combinatorial group
  theory}, revised ed., Dover Publications Inc., New York, 1976, Presentations
  of groups in terms of generators and relations. \MR{0422434 (54 \#10423)}

\bibitem[ML70]{Ma1970The-automorphisms-}
L.~G. Makar-Limanov, \emph{The automorphisms of the free algebra with two
  generators}, Funkcional. Anal. i Prilo\v zen. \textbf{4} (1970), no.~3,
  107--108. \MR{0271161 (42 \#6044)}

\bibitem[MM82]{MaMe1982The-complexity-of-}
Ernst~W. Mayr and Albert~R. Meyer, \emph{The complexity of the word problems
  for commutative semigroups and polynomial ideals}, Adv. in Math. \textbf{46}
  (1982), no.~3, 305--329. \MR{683204}

\bibitem[MM84]{MoMo1984Upper-and-lower-bo}
H.~Michael M{\"o}ller and Ferdinando Mora, \emph{Upper and lower bounds for the
  degree of {G}roebner bases}, E{UROSAM} 84 ({C}ambridge, 1984), Lecture Notes
  in Comput. Sci., vol. 174, Springer, Berlin, 1984, pp.~172--183. \MR{779124
  (86k:13008)}

\bibitem[MP09]{MaPo2009The-Nagata-automor}
Stefan Maubach and Pierre-Marie Poloni, \emph{The {N}agata automorphism is
  shifted linearizable}, J. Algebra \textbf{321} (2009), no.~3, 879--889.
  \MR{2488557 (2009k:14120)}

\bibitem[MS80]{MiSu1980Affine-surfaces-co}
Masayoshi Miyanishi and Tohru Sugie, \emph{Affine surfaces containing
  cylinderlike open sets}, J. Math. Kyoto Univ. \textbf{20} (1980), no.~1,
  11--42. \MR{564667 (81h:14020)}

\bibitem[Mum99]{Mu1999The-red-book-of-va}
David Mumford, \emph{The red book of varieties and schemes}, expanded ed.,
  Lecture Notes in Mathematics, vol. 1358, Springer-Verlag, Berlin, 1999,
  Includes the Michigan lectures (1974) on curves and their Jacobians, With
  contributions by Enrico Arbarello. \MR{1748380 (2001b:14001)}

\bibitem[Mum08]{Mu2008Abelian-varieties}
\bysame, \emph{Abelian varieties}, Tata Institute of Fundamental Research
  Studies in Mathematics, vol.~5, Published for the Tata Institute of
  Fundamental Research, Bombay, 2008, With appendices by C. P. Ramanujam and
  Yuri Manin, Corrected reprint of the second (1974) edition. \MR{2514037
  (2010e:14040)}

\bibitem[Neu33]{Ne1933Die-Automorphismen}
Bernhard Neumann, \emph{Die {A}utomorphismengruppe der freien {G}ruppen}, Math.
  Ann. \textbf{107} (1933), no.~1, 367--386. \MR{1512806}

\bibitem[Nie24]{Ni1924Die-Gruppe-der-dre}
Jakob Nielsen, \emph{Die gruppe der dreidimensionalen gittertransformationen},
  Danske Vid. Selsk. Mat.-Fys. Medd. \textbf{V} (1924), no.~12, 1--29.

\bibitem[PR86]{PeRa1986Finite-order-algeb}
Ted Petrie and John~D. Randall, \emph{Finite-order algebraic automorphisms of
  affine varieties}, Comment. Math. Helv. \textbf{61} (1986), no.~2, 203--221.
  \MR{856087 (88a:57073)}

\bibitem[Pro07]{Pr2007Lie-groups}
Claudio Procesi, \emph{Lie groups}, Universitext, Springer, New York, 2007, An
  approach through invariants and representations. \MR{MR2265844 (2007j:22016)}

\bibitem[Ren68]{Re1968Operations-du-grou}
Rudolf Rentschler, \emph{Op\'erations du groupe additif sur le plan affine}, C.
  R. Acad. Sci. Paris S\'er. A-B \textbf{267} (1968), A384--A387.

\bibitem[Ros57]{Ro1957Some-rationality-q}
Maxwell Rosenlicht, \emph{Some rationality questions on algebraic groups}, Ann.
  Mat. Pura Appl. (4) \textbf{43} (1957), 25--50. \MR{0090101}

\bibitem[Sat83]{Sa1983Polynomial-ring-in}
A.~Sathaye, \emph{Polynomial ring in two variables over a {DVR}: a criterion},
  Invent. Math. \textbf{74} (1983), no.~1, 159--168. \MR{722731 (85j:14098)}

\bibitem[Sei67]{Se1967Differential-ideal}
A.~Seidenberg, \emph{Differential ideals in rings of finitely generated type},
  Amer. J. Math. \textbf{89} (1967), 22--42. \MR{0212027}

\bibitem[Sei74]{Se1974Constructions-in-a}
\bysame, \emph{Constructions in algebra}, Trans. Amer. Math. Soc. \textbf{197}
  (1974), 273--313. \MR{0349648}

\bibitem[Ser58]{Se1958Espaces-fibres-alg}
Jean-Pierre Serre, \emph{Espaces fibr\'es alg\'ebriques}, S\'eminaire Chevalley
  (1958), expos\'e n$^\circ$ 5.

\bibitem[Sha66]{Sh1966On-some-infinite-d}
I.~R. Shafarevich, \emph{On some infinite-dimensional groups}, Rend. Mat. e
  Appl. (5) \textbf{25} (1966), no.~1-2, 208--212. \MR{0485898 (58 \#5697)}

\bibitem[Sha81]{Sh1981On-some-infinite-d}
\bysame, \emph{On some infinite-dimensional groups. {II}}, Izv. Akad. Nauk SSSR
  Ser. Mat. \textbf{45} (1981), no.~1, 214--226, 240. \MR{607583 (84a:14021)}

\bibitem[Sha95]{Sh1995Letter-to-the-edit}
\bysame, \emph{Letter to the editors: ``{O}n some infinite-dimensional groups.
  {II}'' [{I}zv.\ {A}kad.\ {N}auk {SSSR} {S}er.\ {M}at.\ {\bf 45} (1981), no.\
  1, 214--226, 240; {MR}0607583 (84a:14021)]}, Izv. Ross. Akad. Nauk Ser. Mat.
  \textbf{59} (1995), no.~3, 224. \MR{1347084 (96e:14054)}

\bibitem[Sha04]{Sh2004On-the-group-rm-GL}
\bysame, \emph{On the group {${\rm GL}(2,K[t])$}}, Tr. Mat. Inst. Steklova
  \textbf{246} (2004), no.~Algebr. Geom. Metody, Svyazi i Prilozh., 321--327.
  \MR{2101302 (2005h:14111)}

\bibitem[Sie96]{Si1996Lie-algebras-of-de}
Thomas Siebert, \emph{Lie algebras of derivations and affine algebraic geometry
  over fields of characteristic {$0$}}, Math. Ann. \textbf{305} (1996), no.~2,
  271--286. \MR{1391215 (97e:17033)}

\bibitem[Sno89]{Sn1989Unipotent-actions-}
Dennis~M. Snow, \emph{Unipotent actions on affine space}, Topological methods
  in algebraic transformation groups ({N}ew {B}runswick, {NJ}, 1988), Progr.
  Math., vol.~80, Birkh{\"a}user Boston, Boston, MA, 1989, pp.~165--176.
  \MR{1040863 (91f:14048)}

\bibitem[Spr89]{Sp1989Aktionen-reduktive}
Tonny~A. Springer, \emph{Aktionen reduktiver {G}ruppen auf {V}ariet{\"a}ten},
  Algebraische {T}ransformationsgruppen und {I}nvariantentheorie, DMV Sem.,
  vol.~13, Birkh{\"a}user, Basel, 1989, pp.~3--39. \MR{1044583}

\bibitem[Spr98]{Sp1998Linear-algebraic-g}
T.~A. Springer, \emph{Linear algebraic groups}, second ed., Progress in
  Mathematics, vol.~9, Birkh{\"a}user Boston Inc., Boston, MA, 1998.
  \MR{MR1642713 (99h:20075)}

\bibitem[Sri91]{Sr1991On-the-embedding-d}
V.~Srinivas, \emph{On the embedding dimension of an affine variety}, Math. Ann.
  \textbf{289} (1991), no.~1, 125--132. \MR{1087241 (92a:14013)}

\bibitem[SU03]{ShUm2003The-Nagata-automor}
Ivan~P. Shestakov and Ualbai~U. Umirbaev, \emph{The {N}agata automorphism is
  wild}, Proc. Natl. Acad. Sci. USA \textbf{100} (2003), no.~22, 12561--12563
  (electronic). \MR{2017754 (2004j:13036)}

\bibitem[SU04a]{ShUm2004Poisson-brackets-a}
\bysame, \emph{Poisson brackets and two-generated subalgebras of rings of
  polynomials}, J. Amer. Math. Soc. \textbf{17} (2004), no.~1, 181--196
  (electronic). \MR{2015333 (2004k:13036)}

\bibitem[SU04b]{ShUm2004The-tame-and-the-w}
\bysame, \emph{The tame and the wild automorphisms of polynomial rings in three
  variables}, J. Amer. Math. Soc. \textbf{17} (2004), no.~1, 197--227
  (electronic). \MR{2015334 (2004h:13022)}

\bibitem[Suz74]{Su1974Proprietes-topolog}
Masakazu Suzuki, \emph{Propri{\'e}t{\'e}s topologiques des polyn\^omes de deux
  variables complexes, et automorphismes alg{\'e}briques de l'espace {${\bf
  C}^{2}$}}, J. Math. Soc. Japan \textbf{26} (1974), 241--257. \MR{0338423 (49
  \#3188)}

\bibitem[Tit72]{Ti1972Free-subgroups-in-}
Jacques Tits, \emph{Free subgroups in linear groups}, J. Algebra \textbf{20}
  (1972), 250--270. \MR{0286898}

\bibitem[vdE00]{Es2000Polynomial-automor}
Arno van~den Essen, \emph{Polynomial automorphisms and the {J}acobian
  conjecture}, Progress in Mathematics, vol. 190, Birkh{\"a}user Verlag, Basel,
  2000. \MR{1790619 (2001j:14082)}

\bibitem[vdK53]{Ku1953On-polynomial-ring}
W.~van~der Kulk, \emph{On polynomial rings in two variables}, Nieuw Arch.
  Wiskunde (3) \textbf{1} (1953), 33--41. \MR{0054574 (14,941f)}

\bibitem[Win90]{Wi1990On-free-holomorphi}
J{{\"o}}rg Winkelmann, \emph{On free holomorphic {$\bf C$}-actions on {${\bf
  C}^n$} and homogeneous {S}tein manifolds}, Math. Ann. \textbf{286} (1990),
  no.~1-3, 593--612. \MR{1032948 (90k:32094)}

\end{thebibliography}
\bibliographystyle{amsalpha}

\end{document}